\documentclass[11pt]{article}
\usepackage[usenames]{color} 
\usepackage{amssymb} 
\usepackage{amsmath} 
\usepackage[utf8]{inputenc} 
\usepackage[utf8]{inputenc}
\usepackage[T1]{fontenc}
\usepackage[toc,page]{appendix}
\usepackage{geometry}
\usepackage{fancyhdr}
\usepackage{setspace}
\usepackage{amsmath}
\usepackage{amsthm}
\usepackage{latexsym}
\usepackage{amssymb}
\usepackage{amsfonts}
\usepackage{pifont}
\usepackage{epstopdf}
\usepackage{graphicx}
\usepackage{tensor}
\usepackage{tikz}
\usepackage{titlesec}
\usepackage{url}
\usepackage{hyperref}
\usepackage{color}
\usepackage{relsize}
\usepackage{tabularx}
\usepackage{multirow}
\usepackage{array}
\usepackage{subfig}
\usepackage{float}
\usepackage[colorinlistoftodos]{todonotes}
\usepackage[nottoc,notlot,notlof]{tocbibind}
\usepackage{mathrsfs}
\usepackage{accents}
\usepackage[symbols,nogroupskip,nonumberlist]{glossaries-extra}
\usepackage{empheq}
\usepackage{imakeidx}
\usepackage{dsfont}

\newrobustcmd{\SymbolPrint}[1]{\ensuremath{#1}\index[Symbols]{\ensuremath{#1}}}%

\newrobustcmd{\SymbolPrinta}[1]{\ensuremath{#1}\index[Symbols]{\ensuremath{#1@#1}}}

\makeindex[intoc=true,name=Symbols, columns=2, title={List of Symbols}]

\renewcommand{\arraystretch}{1.5}
\usepackage[english]{nomencl}
\makenomenclature

\let\headruleORIG\headrule
\renewcommand{\headrule}{\color{black} \headruleORIG}

\usepackage{colortbl}
\arrayrulecolor{black}

\fancypagestyle{plain}{
  \fancyhead{}
  \fancyfoot[C]{\thepage}
  
}
\numberwithin{equation}{section}

\titleformat{\paragraph}
{\normalfont\normalsize\bfseries}{\theparagraph}{1em}{}
\titlespacing*{\paragraph}
{0pt}{3.25ex plus 1ex minus .2ex}{1.5ex plus .2ex}

\title{Stationary axisymmetric Einstein-Vlasov bifurcations of the Kerr spacetime}
\author{Fatima Ezzahra Jabiri \thanks{Sorbonne Universit\'{e}, CNRS, Universit\'{e} de Paris, Laboratoire Jacques-Louis Lions (LJLL), F-75005 Paris, France, jabiri@ljll.math.upmc.fr}  \thanks{University College London, Department of Mathematics, London, UK, f.jabiri@ucl.ac.uk}}

\theoremstyle{remark}
\newtheorem{remark}{Remark}
\theoremstyle{plain}
\newtheorem{theoreme}{Theorem}
\theoremstyle{definition}
\newtheorem{definition}{Definition}
\theoremstyle{plain}
\newtheorem{lemma}{Lemma}
\theoremstyle{plain}
\newtheorem{Propo}{Proposition}
\theoremstyle{plain}

\theoremstyle{plain}

\newcommand{\g}[2]{\tensor{g}{_{#1#2}}}
\newcommand{\ginv}[2]{\tensor{g}{^{#1#2}}}
\newcommand{\T}[2]{\tensor{{T}}{_{#1#2}}}

\newcommand{\Chris}[3]{\tensor{\Gamma}{^{#1}_{#2#3}}}

\newcommand{\Horizon}{\mathscr{H}}
\newcommand{\Axis}{\mathscr{A}}
\newcommand{\Bbarre}{\overline{\mathscr{B}}}
\newcommand{\BB}{\mathscr{B}}
\newcommand{\spacetime}{\mathcal{M}}

\newcommand{\BAbarre}{\overline{\mathscr{B}_A}}
\newcommand{\BHbarre}{\overline{\mathscr{B}_H}}
\newcommand{\Bnbarre}{\overline{\mathscr{B}_N}}
\newcommand{\Bsbarre}{\overline{\mathscr{B}_S}}
\newcommand{\Adm}{ \mathcal A^{admissible}}

\newcommand{\Chi}{\mbox{\Large$\chi$}}

\newcommand{\Ltheta}{\mathcal L_{\Theta}}
\newcommand{\LX}{\mathcal L_{X}}
\newcommand{\LY}{\mathcal L_{Y}}
\newcommand{\Lsigma}{\mathcal L_{\sigma}}
\newcommand{\LB}{\mathcal L_{B}}
\newcommand{\Llambda}{\mathcal L_{\lambda}}

\newcommand{\Nsigma}{\mathcal N_{\sigma}}
\newcommand{\NB}{\mathcal N_{B}}
\newcommand{\NX}{\mathcal N_{X}}
\newcommand{\NY}{\mathcal N_{Y}}
\newcommand{\Ntheta}{\mathcal N_{\Theta}}

\newcommand{\sigmazero}{\overset{\circ}{\sigma}}
\newcommand{\Xzero}{\overset{\circ}{X}}
\newcommand{\Yzero}{\overset{\circ}{Y}}
\newcommand{\thetazero}{\overset{\circ}{\Theta}}
\newcommand{\lambdazero}{\overset{\circ}{\lambda}}

\newcommand{\Abound}{\mathcal{A}_{bound}}
\newcommand{\Bbound}{\mathcal{B}_{bound}}
\newcommand{\Ascattered}{\mathcal{A}_{scattered}}

\newcommand{\rcc}{r_{c}}
\newcommand{\ve}{\varepsilon}

\newcommand{\Fpert}{F^{(\ve, \ell_z), \rho}}
\newcommand{\rz}{(\rho, z)}

\DeclareMathOperator\supp{supp}

\usepackage{etoolbox}
\renewcommand\nomgroup[1]{%
  \item[\bfseries
  \ifstrequal{#1}{P}{Physics Constants}{%
  \ifstrequal{#1}{N}{Number Sets}{%
  \ifstrequal{#1}{O}{Other Symbols}{}}}%
]}

\begin{document}


\nomenclature{$N$}{le nombre d'éléments}

\pagestyle{fancy}

\maketitle

\begin{abstract}
We construct a one-parameter family of stationary axisymmetric and asymptotically flat spacetimes solutions to the Einstein-Vlasov system bifurcating from the Kerr spacetime. The constructed solutions have the property that the spatial support of the matter is a finite, axisymmetric shell located away from the black hole. Our proof is mostly based on the analysis of the set of trapped timelike geodesics for stationary axisymmetric spacetimes close to Kerr, where the geodesic flow is not necessarily integrable. Moreover, the analysis of the Einstein field equations relies on the modified Carter Robinson theory developed by Chodosh and Shlapentokh-Rothman. This provides the first construction of black hole solutions to the Einstein-Vlasov system in the axisymmetric case and generalises the construction already done in the spherically symmetric case. 
\end{abstract}

\tableofcontents					
\section{Introduction}

\subsection{Relativistic Kinetic theory}
The center of most galaxies, such as our galaxy, is typically modelled as a supermassive black hole, the galaxy consisting  of gas, plasmas and stars orbiting around it. Kinetic theory then plays  an important role in the description of these matter fields. In the geometric context of general relativity, the formulation of a relativistic kinetic theory was developed
by Synge \cite{synge1934energy}, who in particular introduced the world lines of the gas particles,  Tauber and Weinberg \cite{tauber1961internal}, who developed a covariant form of phase space and the corresponding Liouville theorem, and  
 Israel \cite{israel1963relativistic} who derived conservation laws based on the fully covariant Boltzmann equation. 
\\
\\When the number of particles is large,  mathematical models of particle systems are often described by kinetic or fluid equations. The choice of a good model may depend on the physical properties of interests, the existence of good numerical schemes or of a well developed theory.  A characteristic feature of kinetic theory is that its models are statistical and the particle systems are described by distribution functions  defined on  phase space. A distribution function represents the number of particles with given spacetime position and velocity. It contains a wealth of information and macroscopic quantities are easily calculated from it, such as energy densities, mass density and moments.  
\\
\\ In this work, we are concerned with one specific model of kinetic theory:  the so-called collisionless or Vlasov matter model.  It is used to describe galaxies or globular galaxies where the stars play the role of gas particles and  collisions between them are sufficiently rare to be neglected, so that the only interaction taken into account is gravitation. The distribution is then transported along the trajectories of free falling particles, resulting in the Vlasov equation. The latter is coupled to the equations for the gravitational field, where the source terms are computed from the distribution function. In the non-relativistic setting, i.e.~the Newtonian framework, the resulting nonlinear system of partial differential equations is the Vlasov-Poisson (VP) system, while its general relativistic counterpart forms the Einstein-Vlasov (EV) system. Collisionless matter possesses several attractive features from a partial differential equations viewpoint. On any fixed background, it avoids pathologies such as shock formation, contrary to more traditional fluid models. Moreover, one has global classical solutions of the VP system in three dimensions for general initial data \cite{pfaffelmoser1992global}, \cite{lions1991propagation}. 
\\
\\ The local well-posedness of the Cauchy problem for the EV system was first investigated in \cite{choquet1971probleme} by Choquet-Bruhat. Concerning the nonlinear stability of the Minkowski spacetime as the trivial solution of the EV system, it was proven in the case of spherically symmetric initial data by Rendall and Rein \cite{rein1992global} in the massive case and by Dafermos \cite{dafermos2006note} for the massless case.  The general case was recently shown  by Fajman, Joudioux and Smulevici \cite{fajman2017stability} and independently by  Lindblad-Taylor \cite{lindblad2017global} for the massive case, and by Taylor \cite{taylor2017global} for the massless case, see also \cite{bigorgne2020asymptotic} for an alternative proof without the compact support assumption. Nonlinear stability results have been given by Fajman \cite{fajman2017nonvacuum}  and Ringtröm \cite{ringstrom2013topology} in the case of cosmological spacetimes. See also \cite{smulevici2011area}, \cite{andreasson2005existence}, \cite{dafermos2006strong}, \cite{weaver2004area}, \cite{smulevici2008strong}, \cite{dafermos2016strong} for several results on cosmological spacetimes with symmetries.  

\subsection{Steady states of the EV system}
\noindent While self-gravitating Vlasov systems have proven to be useful models in astrophysics and general relativity, there are still many open questions concerning the space of stationary solutions. In particular, the problem of finding steady states is challenging without strong symmetry assumptions. More precisely, these models are well-studied under the restriction of spherical symmetry and they can  be obtained by assuming that the distribution function has the following form 
\begin{equation*}
f(x,v) = \Phi(E, \ell),
\end{equation*}
where $E$ and $\ell$ are interpreted as the energy and the total angular momentum of particles respectively. In fact, in the Newtonian setting, the distribution function associated to a stationary and spherically symmetric solution to the VP system is necessarily described by a function depending only on $E$ and $\ell$.  Such statement is referred to as Jean's theorem \cite{jeans2017problems}, \cite{jeans1915theory}, \cite{batt1986stationary}. However, it has been shown that its generalisation to general relativity is false in general \cite{schaeffer1999class}.  
\\ A particular choice of $\Phi$, called the polytropic ansatz, which is commonly used to construct static and spherically symmetric states for both VP and EV system is given by 
\begin{equation} 
\label{polytropes}
\Phi(E, \ell):=
\left\{ 
\begin{aligned}
(E_0 - E)^\mu \ell^k, \quad &E<E_0, \\
0 ,\quad &E\geq E_0, 
\end{aligned}
\right. 
\end{equation}
where $E_0>0$, $\mu>-1$ and $k>-1$.
In \cite{rein1993smooth}, Rein and Rendall obtained the first class of asymptotically flat, static, spherically symmetric solutions to EV system with finite mass and finite support such that $\Phi$ depends only on the energy of particles with $\displaystyle \mu\in[0, \frac{7}{2}[$. In \cite{rein1994static}, Rein extended the above result for distribution functions depending also on $\ell$, where $\Phi$ is similar to the polytropic ansatz \eqref{polytropes}: $\displaystyle \Phi(E, \ell) = (-E)_+^\mu(\ell)_+^k$, $\mu\geq 0$, $k>-\frac{1}{2}$ and $\displaystyle \mu<3k + \frac{7}{2}$. Among these, there are singularity-free solutions with a regular center, and also solutions with a Schwarzschild-like black hole.  Based on perturbations arguments,  Andréasson-Fajman-Thaller in \cite{andreasson2015static}  proved the existence of static spherically symmetric solutions of the Einstein-Vlasov-Maxwell system with non-vanishing  cosmological constant and among these, there are solutions which contain black holes. Recently, we have obtained an alternative approach to the construction of black hole spherically symmetric steady states \cite{jabiri2020static}. The construction is  based on the analysis of the set of trapped timelike geodesics and of the effective potential energy for static spacetimes close to Schwarzschild, see Section \ref{recall::result}. The goal of this paper is to extend the strategy of \cite{jabiri2020static} in the axisymmetric case 
\\ 
\\Beyond spherical symmetry however, the equations become much more complicated and thus, few mathematical or numerical results have been established so far. More precisely,  only two mathematical constructions had been obtained in the case of axisymmetry for the axisymmetric EV system: static and axisymmetric solutions were constructed by Andréasson-Kunze-Rein in \cite{andreasson2011existence} and then extended to establish the existence of rotating stationary and axisymmetric solutions to the EV system  in \cite{andreasson2014rotating}. The  constructed solutions are obtained as bifurcations of a spherically symmetric Newtonian steady state and they do not contain black holes \cite{andreasson2014rotating}, see also Remark \ref{Hakan:result}. Moreover, the steady states obtained in \cite{andreasson2011existence} are non rotating and the ones obtained in \cite{andreasson2014rotating} are slowly rotating with a possible presence of an ergoregion. We note that  the strategy of the proof based on  bifurcation argument  was initially introduced by  Lichtenstein \cite{lichtenstein2013gleichgewichtsfiguren} who proved the existence of non-relativistic, stationary, axisymmetric self-gravitating fluid balls. Furthermore, numerical constructions have been provided by Ames-Andréasson-Logg \cite{ames2016axisymmetric}. The constructed solutions are not necessarily slowly rotating. Interestingly,  the resulting spacetimes contain an ergoregion but no black holes for a certain class of the profile $\Phi$.  

\subsection{The main result}
In this paper, we  generalise the approach presented in  \cite{jabiri2020static}  in order to construct stationary axisymmetric bifurcations from the Kerr spacetime of  the EV system. The solutions have the property that the matter shell  is located in the exterior region of a Kerr-like black hole. Our construction is based on the study of trapped timelike geodesics of spacetimes close to Kerr. In particular, as in the spherically symmetric case, we show (and exploit) that for some values of energy and total angular momentum $(\ve, \ell_z)$, the effective potential associated to a particle moving in a perturbed Kerr spacetime and that of a particle with same $(\ve, \ell_z)$ moving in Kerr are similar. Our distribution function will then be supported on the set of trapped timelike geodesics, and this will lead to the finiteness of the mass and its compact support.
\\ For any stationary and axisymmetric spacetime, one can define an energy $\ve$ and angular momentum $\ell_z$ associated to a timelike geodesic. An open set of trapped geodesic can then be identified based only on $(\ve, \ell_z)$ and the initial position.

\noindent Our main result is the following
\begin{theoreme}
\label{thm::2:bis}
Let $(\spacetime , g^K_{a, M})$ be the exterior of a sub-extremal  kerr spacetime, i.e. $0<|a|<M$. For any appropriate profile $\Phi$, there exists a $1-$parameter family of stationary, axisymmetric asymptotically flat black holes spacetimes $(\spacetime, g_\delta)_{[0, \delta_0[}$ and distribution functions $f^\delta: \Gamma_1\to \mathbb R_+$ solving the Einstein-Vlasov system, such that 
\begin{enumerate}
\item when $\delta = 0$, $f^0$ identically vanishes and $g_0$ coincides with the sub-extremal Kerr metric $g^K_{a, M}$, 
\item $\forall \delta\geq 0$, $f^\delta$ verifies 
\begin{equation}
\forall(x, v)\in\Gamma_1,\; f^\delta(x,v) = \Phi(E^\delta, \ell_z; \delta)\Psi((\rho, z), (E^\delta, \ell_z^\delta), g_\delta). 
\end{equation}
where $\Phi(\cdot, \cdot;\delta)$ is supported on a compact set $\Bbound$ of the set of parameters $(\ve, \ell_z)$ corresponding to trapped timelike trajectories, $\ve$ is the energy of the particle and $\ell_z$ its azimutal angular momentum, $\Psi$ is a positive cut-off function which selects the trapped geodesics with parameters $(\ve, \ell_z)\in \Bbound$, $\Gamma_1$ is the mass shell of particles with rest mass $m=1$, and $E^\delta$ is the local energy with respect to the metric $g_\delta$. 
\item $f$ is compactly supported in the exterior region and its supports does not depend on $\delta$. 
\item The boundary of  $(\spacetime , g_\delta)$ corresponds  to a non degenerate bifurcate Killing event horizon on which the metric has a $C^{2, \alpha}$ extension, for all $\alpha\in[0, 1[$.
\end{enumerate}
\end{theoreme}
\noindent We refer to Section \ref{main::result:two:ref} for a more detailed version of the main result. 
\begin{remark}
\label{remark:2:Intro}
The support of $\Phi(\ve, \ell_z; \delta)$ as a function on the mass shell has two connected components: one corresponds to geodesics which reach the horizon in a finite proper time, and the other one corresponds to trapped geodesics. $\Psi$ is introduced so that it is equal to $0$ outside $\Bbound$ and equal to a cut-off function depending on the $r$ variable,  $\Chi$ on $\Bbound$. The latter is equal to $0$ on the first connected component of the support of $\Phi(\ve, \ell_z; \delta)$ and to $1$ on the second component. This allows to eliminate the undesired trajectories. The underlying reason behind all of this is that  $(\ve, \ell_z)$ are not sufficient to characterise the geodesic motion. More precisely, if we consider the motion of a particle in a Kerr spacetime, then the type of trajectory (trapped, unbounded, plunging) depends on $(\ve, \ell_z)$, possibly also on Carter constant,  and on the initial radial position: for a fixed $(\ve, \ell_z)$, the particle can have different trajectories depending on where it initially starts.
\end{remark}
\begin{remark}
We refer to Section \ref{Ansatz:for:the:distribution:function} for the precise assumptions on the profile $\Phi$.  Roughly speaking, we assume $C^1$ regularity of $\Phi$ with respect to each variable. 
\end{remark}
\begin{remark}
\noindent The analysis of the reduced Einstein equations, necessary for the proof of Theorem 1  is based  the work of Chodosh and Shlapentokh-Rothman on time-periodic Einstein-Klein-Gordon  bifurcations of Kerr \cite{chodosh2017time}, \cite{chodosh2015stationary} . They relied on Carter-Robinson theory \cite{carter2009republication}, \cite{carter2010republication} and the approach of Weinstein \cite{weinstein1990rotating}, \cite{weinstein1992stationary} concerning rotating black holes in equilibrium.
\\ Our work  is based, on one hand,  on their modified Carter-Robinson theory, and on the other hand, on a generalisation of the arguments concerning the analysis of trapped timelike geodesics from the spherically symmetric case\cite{jabiri2020static}, see Section \ref{recall::result} below.  More precisely, we use these arguments for the analysis of the Vlasov matter terms, and this is the main contribution of our work.
\end{remark}
\begin{remark}
\label{Hakan:result}
 Another influential work  is the one of Andréasson-Kunze-Rein  on the construction of rotating, general relativistic and asymptotically flat non-vacuum spacetimes \cite{andreasson2014rotating}. The authors provided the first mathematical construction of stationary axisymmetric asymptotically flat solutions to the EV system which are geodesically complete and with non-zero total angular momentum.  Their method was based on an implicit function theorem and a bifurcation argument from spherically symmetric steady states of the VP system.  The ansatz for the distribution function was given by 
\begin{equation*}
f^{\mu_1, \mu_2}(x, v) = \phi\left(E - \frac{1}{\mu_1}\right)\psi\left(\mu_2, \ell_z\right),
\end{equation*}
where $\mu_1$ turns on general relativity and the second parameter turns on the dependence on $\ell_z$. Moreover, the reduction of the Einstein tensor associated to a stationary and axisymmetric metric followed the work of Bardeen \cite{bardeen1973rapidly} and the energy-momentum tensor components were computed via a reparametrisation of the mass shell. 
\\
\\ Our work also uses a similar ansatz for the distribution function\footnote{We still have a cut-off depending on $\rho$ as in the spherically symmetric case.} and this leads to similar reductions of the components of the energy-momentum tensor. On the other hand, we use the implicit function theorem to bifurcate from a possibly  rapidly rotating Kerr spacetime. Moreover, our reduction of the EV system follows the work of Chodosh and Shlapentokh-Rothman \cite{chodosh2017time}. \end{remark}
\section{Key ideas of  the proof}
 \noindent We provide an overview of the proof and we present the main ideas for the construction and the difficulties. The proof is based on the generalisation of the arguments in the spherically symmetric case. We thus start by presenting the  key arguments in \cite{jabiri2020static}
\subsection{Spherically symmetric matter shells orbiting a Schwarzschild-like black hole}
 \label{recall::result}
\subsubsection{Geodesic motion in Schwarzschild spacetime and the set of trapped geodesics}
The study of the geodesic motion in Schwarzschild spacetime is included in the classical books of general relativity. See for example \cite[Chapter 3]{Chandrasekhar:1985kt} or \cite[Chapter 33]{misner2017gravitation}. 
\\ The geodesic motion in Schwarzschild forms an integrable Hamiltonian system. The problem of solving the geodesic equation is then reduced to a one dimensional problem in the radial direction, and the classification of timelike geodesics  is therefore  based on the roots of the equation
            \begin{equation}
            \label{roots}
                E_{\ell}^{Sch}(r) = E^2, 
            \end{equation}
where $E_{\ell}^{Sch}$ is the effective potential energy associated to a timelike geodesic. $E_{\ell}^{Sch}$ is a polynomial of degree 3 of $\displaystyle \frac{1}{r}$ and thus it admits at most three roots in the region $]2M, \infty[$. In particular, trapped timelike geodesics occur when \eqref{roots} admits three distinct roots: $r_0^{Sch}(E, \ell) < r_1^{Sch}(E, \ell) < r_2^{Sch}(E, \ell)$. We denote by $\Abound^{Sch}$ the set of of parameters $(E, \ell)$ for which the latter occur. 
\\Now, given $(E, \ell)\in \Abound^{Sch}$, the allowed region for a timelike geodesics, defined to be the subset in $]2M, \infty[$ such that $$E_{\ell}^{Sch}(r) \leq E^2, $$ has two connected components: $]2M, r_0^{Sch}(E, \ell)]$ and $[r_1^{Sch}(E, \ell), r_2^{Sch}(E, \ell)]$. Therefore, 
\begin{itemize}
    \item either the geodesic starts from $r_0^{Sch}(E, \ell)$ and reaches the horizon in a finite time,
\item or the geodesic is periodic. It oscillates between an aphelion $r_2^{Sch}(E, \ell)$ and a perihelion $r_1^{Sch}(E, \ell)$.  
\end{itemize}
As a consequence, a cut-off function was used in the ansatz for the distribution function in order to select only trapped orbits, see Remark \ref{remark:2:Intro}.
\\
\\ Given a Schwarzschild spacetime, one can classify the geodesics based on the integrals of motion only. Then, we show that trapped geodesics remain stable under spherically symmetric perturbations. The stability result is key for controlling the matter terms for the Vlasov field. Moreover, it leads to the compact support of matter. 
\subsubsection{Stability of trapped geodesics of Schwarzschild}
We denote by $g_{Sch}$ the Schwarzschild metric written in the spherical coordinates and by $ B(g_{Sch}, \delta_0)$ the ball of radius $\delta_0$ centred around $g_{Sch}$ in some functional space adapted to the problem.  The construction in \cite{jabiri2020static} is mostly based the  following result.
\begin{Propo}
\label{matter::shell}
Let $ 0< \tilde\delta_0 < \delta_{max}$. Then, there exists $\displaystyle\delta_0\in]0, \tilde\delta_0]$ such that $\displaystyle \forall g\in B(g_{Sch}, \delta_0)$, $\displaystyle \forall(E, \ell)\in B_{bound}$, there exist unique $\displaystyle r_i(g, (E, \ell))\in B(r^{Sch}_i(E, \ell), \delta_0), i\in\left\{0, 1, 2\right\}$ such that $r_i(g, (E, \ell))$ solve the equation
\begin{equation*}
E_{\ell}(g, r) = E^2,
\end{equation*}  
where $E_{\ell}(g, \cdot)$ is the effective potential energy associated to timelike geodesics moving in a static and spherically symmetric spacetime with metric $g$. 
\\Moreover, there are no other roots for the above equation outside the balls $B(r_i^{Sch}(E, \ell), \delta_0)$.
\end{Propo}
\subsubsection{Solving the reduced system}
For static and spherically symmetric spacetimes, the EV system is reduced to a system of ODEs with respect to the radial variable in the metric coefficients $(\lambda, \mu)$ . In particular, one can derive an independent equation for one of the metric component $\mu$ . Once  $\mu$ is solved, one can easily construct the remaining component.  
\\In order to solve for $\mu$,
\begin{itemize}
    \item one chooses $\rho>0$ sufficiently small and $R>0$ sufficiently large so that the matter is supported \footnote{The matter shell is actually located outside the photon sphere $r = 3M$.} in the region $]2M+\rho, R[$. 
    \item By Birkhoff theorem, the solution is given by Schwarzschild with parameter $M$ in the region $]2M, 2M+\rho[$.     \item An ODE  is solved in the region $]2M+\rho, R[$ using an implicit function theorem and the analysis of Vlasov matter. The solution obtained is then close to Schwarzschild. 
    \item Again, by Birkhoff theorem, the solution is also given by Schwarzschild with new parameter $M + m(\delta)$ in the region $]R, \infty[$, where $m(\delta)$ is the total mass of the Vlasov field. 
\end{itemize}

\subsection{Strategy for the axisymmetric case}
The strategy of the construction in the axisymmetric case consists of generalising the stability result for trapped timelike geodesics and using the modified Carter Robinson theory for the analysis of the field equations. In this section, we emphasise the main difficulties and differences compared to the spherically symmetric case. 
\subsubsection{General framework and geometric setting}
\noindent In order to motivate the main ideas of the construction and the difficulties, we present briefly the geometric setting which will be detailed in Section \ref{Stat:axis:coord}.
\\ We assume that we are looking for stationary, axially symmetric and asymptotically flat spacetimes with stationary and axially symmetric matter fields. In this context, we use the Weyl coordinates $(t, \phi, \rho, z)$ defined on 
        \begin{equation*}
            \mathcal O := \mathbb R_t\times]0, 2\pi[_\phi\times\BB_{(\rho, z)}\quad\text{where}\quad \mathscr{B}:= \left\{ \rho > 0 \;,\; z\in\mathbb R\right\}
        \end{equation*}
which are suitable for axially symmetric problems. The horizon is defined by
\begin{equation*}
\Horizon := \left\{(\rho, z)\in\Bbarre\;;\; \rho = 0 \quad\text{and}\quad z\in]-\gamma, \gamma[\right\},
\end{equation*} 
the axis of symmetry is defined by
                \begin{equation*}
\Axis := \left\{(\rho, z)\in\Bbarre\;;\; \rho = 0 \quad\text{and}\quad z\in]-\infty, -\gamma[\cup]\gamma, +\infty[\right\}, 
\end{equation*} 
and the poles are defined by  $p_{N, S} = (0, \pm\gamma)$ where $\gamma := \sqrt{M^2 -a^2}\;$ such that $0<|a|<M$. Here, $\Bbarre$ is the "extension of" $\BB$ obtained by gluing $\BB$ with its boundary.  
\\ We assume the following metric ansatz
\begin{equation}
\label{metric:ansatz}
g := -V dt^2 +2  W dtd\phi + X d\phi^2 + e^{2 \lambda}\left(d\rho^2 +dz^2\right)
\end{equation}
where the metric components are functions defined on $\BB$ and the following ansatz for the distribution function: 
        \begin{equation*}
            f(x,v) = \Phi(\ve, \ell_z)\Psi((\rho, z), (\ve, \ell_z), g). 
        \end{equation*}
\subsubsection{Main novel  difficulties and key ideas}
\label{main::difficulties}
In the following, we discuss some of the difficulties we encountered while proving Theorem \ref{thm::2:bis}: 

\begin{enumerate}
    \item \textbf{Non-integrability of the geodesic system in stationary and axisymmetric spacetimes}:  In general stationary and axisymmetric spacetimes, there are a priori only three integrals of motion: the Hamiltonian $H$, the energy measured at infinity $\ve$, and the azimutal angular momentum $\ell_z$. Therefore, the problem of solving the geodesic equations which consist of integrating a system of 8 ordinary differential equations is reduced to solving a problem with two degrees of freedom defined on a four dimensional submanifold of the tangent bundle parametrized by $(\ve, \ell_z)$.  In Kerr spacetime,  there exists a fourth integral of motion, due to Carter \cite{carter2009republication}. We note that  the study of the geodesic motion in Kerr spacetime is included in the classical books of general relativity. See for example \cite[Chapter 6]{Chandrasekhar:1985kt} for a classification of orbits with constant radial motion and of orbits confined in the equatorial plane, and  \cite[Chapter 4]{o2014geometry} for a full classification of timelike geodesics based on the Carter constant. However, in order to construct $\Abound^K$, we need to reparametrize  the trapped geodesics based only on the integrals $(\ve, \ell_z)$. Indeed, the Carter constant does not exist for arbitrary perturbations. A key idea  is to identify a set of trapped timelike geodesics moving in Kerr independently of this fourth integral and based only on $(\ve, \ell_z)$.  In this context, we recall the classification of timelike geodesics in Section \ref{Kerr:geo:aux} and we revisit its proof. Then, we reparametrise the timelike geodesics based only on $(\ve, \ell_z)$. This leads to the generalisation  of $\Abound^{Sch}$ in Section \ref{classif:BL:geodesics}. 

\item \textbf{Stability of the set of trapped timelike geodesics: }In Kerr spacetime, if we do not make use of the Carter constant, the geodesic motion is reduced to two-dimensional motion in the $\BB$ plane. Therefore,  we  define a two dimensional potential $E^K_{\ell_z}$ on $\BB$ associated to a timelike geodesic. As in the spherically symmetric case, the classification is based on the solutions of the equation 
        \begin{equation}
        \label{seven}
            E^K_{\ell_z}(\rho, z) = \ve
        \end{equation}
        which are no longer points, but they are curves in the  $\BB$ plane and their shape determine the nature of the orbit. Therefore, the turning points can be generalised in the following definition
        \begin{definition}
\label{ZVC:St:Axis}
Let $\gamma: I\to\spacetime$ be a timelike future directed geodesic with constants of motion $(\ve, \ell_z)$. We define the zero velocity curve (ZVC) associated to $\gamma$ denoted by $Z(\ve, \ell_z)$  to be the curve in $\BB$ defined by 
\begin{equation*}
{Z}^K(\ve, \ell_z) := \left\{(\rho, z)\in\BB\;:\; E^K_{\ell_z}(\rho, z) = 0 \right\}. 
\end{equation*}
\end{definition}
As already discussed, there exists a fourth constant of motion, $Q$, due to Carter \cite{carter2009republication}, so that any geodesic is characterised by the set $(H, \ve, \ell_z, Q)$ and the initial spacetime position. In our picture, one can use $Q$ as a parameter along the curve that solves \eqref{seven}. 
\\Eventually, for $(\ve, \ell_z)\in\Abound^K$,  we obtain a curve with possibly several connected components. In particular, trapped geodesics occur when the solution curve has a compact connected component, $Z^{K, trapped}(\ve, \ell_z)$.  
In stationary and axisymmetric spacetimes with metric $g$, we analogously  define the effective potential energy for a timelike particle with angular momentum $\ell_z$, $E_{\ell_z}(g,\cdot, \cdot)$. Again, the classification of timelike geodesics is based on the solutions of the equation 
    \begin{equation}
            E_{\ell_z}(g, \rho, z) = \ve
        \end{equation}
        The idea is to prove that $Z^{K, trapped}(\ve, \ell_z)$  is stable against stationary and axisymmetric perturbations. We state a rough version of the perturbation result (see Section \ref{perturbed:Kgeo})
         \begin{Propo}
\label{matter:shell}
Let $ 0< \tilde\delta_0 < \delta_{max}$. Then, there exists $\displaystyle\delta_0\in]0, \tilde\delta_0]$ such that $\displaystyle \forall g\in B(g_K, \delta_0)$, $\displaystyle \forall(\ve, \ell)\in \Bbound$, there exists a unique
smooth curve $Z^{pert}(g, (\ve, \ell_z))$ diffeomorphic to $\mathbb S^1$ in $\BB$ "close to" $Z^K(\ve,  \ell_z)$ such that $\forall (\rho, z)\in Z^{pert}(g, (\ve, \ell_z))$, $(\rho, z)$ solves the equation
 \begin{equation*}
 E_{\ell_z}(g, \rho, z) = \ve. 
 \end{equation*} 
\end{Propo}
 \noindent One of the technical difficulties that we encountered in the proof is to give a definition of "perturbed closed curves" , see Definition \ref{Z:K::perturb}, and to chose a functional space for the metric $g$ which is compatible with the PDE problem for the metric coefficients. It turned out that the theory of Shlapentokh-Rothman and Chodosh \cite{chodosh2017time} is sufficient to solve the second problem.    
 \\   
%
%
\\ The analysis  of Yakov and Otis is itself based on the Carter Robinson theory. We thus start by reviewing the original Carter Robinson theory developed in the context of Kerr conjecture. 
    \item \textbf{Carter-Robinson theory and the analysis of the reduced Einstein equations:}   We recall that the uniqueness conjecture of the Kerr family, is known to be true if the spacetime is assumed to be axisymmetric. The problem was reduced to solving  a harmonic map system with boundary conditions at infinity, the horizon, the axis of symmetry and their intersection. Indeed,   the twist one-form $\theta$ associated to the Killing field generating the axial symmetry is closed on $\BB$, which is simply connected. This allowed Carter and Robinson to define an Ernst potential, $Y$ which vanishes at infinity such that
    \begin{equation*}
    dY = \theta
    \end{equation*}
    and  which forms, together with the metric coefficient $X$ a harmonic map system which decouples from the remaining equations for the other metric components: 
    \begin{equation*}
        \left\{
        \begin{aligned}
            \rho^{-1}\partial_{\rho}(\rho\partial_{\rho}X) + \rho^{-1}\partial_{z}(\rho\partial_{z}X) &= \frac{(\partial_{\rho}X)^2 + (\partial_{z}X)^2 - (\partial_{\rho}Y)^2 - (\partial_{z}Y)^2}{X}, \\
            \rho^{-1}\partial_{\rho}(\rho\partial_{\rho}Y) + \rho^{-1}\partial_{z}(\rho\partial_{z}Y) &= \frac{2(\partial_{\rho}Y)(\partial_{\rho}X) + 2(\partial_{z}Y)(\partial_{z}X)}{X}.
        \end{aligned}
        \right.
        \end{equation*}
        
    In fact, if $(\rho, z)$ are considered with $\phi\in(0, 2\pi)$ as being the cylindrical coordinates in $\mathbb R^3$,  $X$ and $Y$ can then be seen  as  axisymmetric functions on $\mathbb R^3$. Therefore,  $(X,Y)$  forms a harmonic map from $\mathbb R^3$ to hyperbolic space $\mathbb H^2$. The requirements of asymptotic flatness and regular extensions to the axis and event horizon lead to natural boundary conditions for $X$ and $Y$.  This determines uniquely $(X, Y)$. \footnote{ The uniqueness uses a divergence identity,  generalised to the so-called  Mazur identity, see \cite[Chapter 10]{heusler1996black}. }
    \\Given a particular solution $(X, Y)$ to the harmonic map system, the remaining of the metric coefficients are then uniquely determined. First, it is shown that 
    
    \begin{equation*}
    \Delta_{\mathbb R^2}\sigma = 0 \quad\text{where}\quad \sigma := \sqrt{XV + W^2}. 
    \end{equation*}
    The boundary conditions for $\sigma$ imply that $\sigma = \rho$. Next, the definition of the twist leads to the following equation on $W$
    \begin{equation*}
\partial_{\rho}(X^{-1}W)d\rho + \partial_{z}(X^{-1}W)dz = \frac{\rho}{X^2}((\partial_{\rho}Y)dz-(\partial_{z}Y)d\rho).    
\end{equation*}
   Again, the boundary conditions are used to determine uniquely $W$ in terms of $(X, Y)$. As for $\lambda$, it  satisfies the equation 
 \begin{equation*}
            \left\{
            \begin{aligned}
                \partial_{\rho}\lambda &= \frac{1}{4}\rho{X}^{-2} ((\partial_{\rho}X)^2-(\partial_{z}X)^2+(\partial_{\rho}Y)^2-(\partial_{z}Y)^2) - \frac{1}{2}\partial_{\rho}\log{X}, \\
                \partial_{z}\lambda &= \frac{1}{4}\rho{X}^{-2} ((\partial_{\rho}X)(\partial_{z}X)+(\partial_{\rho}Y)(\partial_{z}Y)) - \frac{1}{2}\partial_{z}\log{X}.
            \end{aligned}
            \right.
            \end{equation*}
This determines uniquely $\lambda$ in terms of $(X, Y)$. One can then conclude the uniqueness of Kerr once $(X_K, Y_K)$ associated to the Kerr metric are checked to verify the harmonic map system. 
\\
    \\ In the static spherically symmetric case, we recall that the problem was solved only  in a bounded region of the $r$-domain and we  Birkhoff theorem was used in order to extend the solution on the whole exterior region. 
    \\Unfortunately, the gluing argument cannot be applied in the axisymmetric case because the above Carter-Robinson theory cannot be applied to the Einstein Vacuum equations with boundary conditions different from those at infinity. Indeed, we cannot assume that the metric is given by Kerr between the horizon and the inner boundary of the matter shell. Therefore, we have to solve the system on the whole exterior region \footnote{ Interestingly, the solutions that we construct are vacuum near the horizon and thus they emphasise the need for a global analysis to address the uniqueness conjecture of Kerr. } with  suitable boundary conditions on the horizon, the axis of symmetry, the infinity and the intersection between the horizon and the axis of symmetry. This leads to equations with singular coefficients and this is where we rely on the work \cite{chodosh2017time} to overcome these difficulties, see Section \ref{solving::unknowns}. 
\\ One extra difficulty is  that in the presence of matter the twist-one form,  $\theta$ is no longer closed. However, in \cite{chodosh2017time}, the authors managed to introduce an Ernst-like potential and another one-form $B$ such that 
\begin{equation*}
 dY = \theta - B. 
\end{equation*}
From this decomposition, they obtained a harmonic map system but it is now coupled with the remaining equations. We will  adapt this argument in our proof in order to reduce the EV system to a system of semi-linear elliptic equations coupled to first order PDEs in the metric coefficients only. For the matter terms, it will remain to estimate all the components of $\T{\alpha}{\beta}$ at the same regularity as the metric, see Section \ref{regularity}.  
\\Eventually, to close the whole argument, we will apply a fixed point theorem, see Section \ref{solving::unknowns}. 
\end{enumerate}

\subsubsection{Overview of the poof}
In this section, we give an overview of the proof of Theorem \ref{main::result}. 
\begin{enumerate}
\item First of all, we will present in Section \ref{Preliminaries:Kerr} basic background material on the axisymmetric  Einstein-Vlasov system as well as some properties of sub-extremal Kerr exteriors. In Section \ref{Kerr:geo:aux}, we will study the geodesic motion of timelike particles moving in Kerr exteriors in BL coordinates and in Weyl coordinates. We will also construct the set $\Abound^K$ and define the zero velocity curves.
\item Then, we will compute in Section \ref{EV:reduced:system} the components of the energy momentum tensor and reduce the EV system to a system of integro-partial differential equations in the metric data only. We will also introduce the required functional spaces for the analysis. 
\item In Section \ref{main::result:two:ref},  we will give a detailed formulation of  Theorem \ref{main::result}. 
\item In Section \ref{perturbed:Kgeo}, we will prove the stability result for trapped geodesics. To this end, we control quantitatively the effective potential and the resulting trapped timelike geodesics for stationary axisymmetric spacetimes close to Kerr in every region where ZVC can be written as the graph of a function. Then, using the compactness of $\Bbound$, we will show that trapped geodesics moving in the perturbed spacetimes lie in a compact region of $\BB$ which is uniform in $(\ve, \ell_z)$. This will allow us to obtain a distribution function which is compactly supported in $\BB$. Consequently, all the matter terms $F_i(\thetazero, \Xzero, \sigmazero)$ will compactly supported in $\BB$ and vanish in a neighbourhood of the horizon, the axis of symmetry and  the poles. 
\item Then,  we will use  two fixed point lemmas to solve the nonlinear aspects of the problem, which will be introduced in Section \ref{Fixed::point::lemmaa}. We will start with the study of a toy model which illustrates the application of these lemmas. In the general case, we will have to deal with the difficulty related to the nonlinear coupling of the equations. 
\item At this stage, we will introduce a bifurcation parameter $\delta\geq 0 $ in the ansatz for the distribution function which turns on  the presence of Vlasov matter. This will allow us to transform the problem of finding solutions to the reduced EV system for the renormalised quantities into that of finding a one-parameter family of solutions which depends on $\delta$, by applying a fixed point lemma, considered as a zero of a well-defined operator.  
\item In Section \ref{solving::unknowns}, we will solve the reduced Einstein-Vlasov system. We note that we will solve each equation separately and the order in which we solve them matters. See Remark \ref{order::solving}. More precisely:
\begin{itemize}
\item We will  begin by solving the equation for $\sigmazero$ in terms of the remaining quantities and the bifurcation parameter $\delta$. The regularity for the matter terms will allow us to have a $C^1$ dependence of $\sigmazero$ in $(\Xzero, \thetazero, \lambdazero)$ and a continuous dependance with respect to $\delta$. To this end, we will apply a fixed point lemma.
\item Then, we will solve the equations for $B$ in terms of $(\sigmazero, \Xzero, \thetazero)$. Note that $\sigmazero$ depends on the other renormalised quantities and $\delta$. Therefore, after the application of the fixed point theorem, we will obtain a one parameter family of solutions $(B, \sigmazero)$ which depend in $C^1$ manner of $(\Xzero, \thetazero, \lambdazero)$ and continuously on $\delta$. 
\item We iterate the solving process in order to solve the equations for $(\Xzero, \Yzero)$ in terms of $(\thetazero, \lambdazero; \delta)$, then $\thetazero$ in terms of $(\lambdazero; \delta)$ and finally the equations for $\lambdazero$ in terms of $\delta$ only. 
\item Consequently, we will obtain a one-parameter family of solutions $(\sigmazero, B, \Xzero, \Yzero, \thetazero, \lambdazero)$ which depends continuously on $\delta$. 
\end{itemize}
\item Finally, we will extend the solutions to a larger spacetime which boundary consists of an event horizon, see Section \ref{Final::proof}. 
\end{enumerate}

\subsection*{Acknowledgements}
I would like to thank my PhD advisor Jacques Smulevici for suggesting this problem to me, for the interesting discussions and crucial suggestions, and for reading this work.  The majority of this work was supported by the ERC grant 714408 GEOWAKI, under the European Union’s Horizon 2020 research and innovation program. Its completion was done in University College London and supported by the EPSRC Early Career Fellowship EP/S02218X/1.

\section{Preliminaries and basic background material}
\label{Preliminaries:Kerr}
In this section, we introduce basic material necessary for the rest of this work.
\subsection{The Einstein-Vlasov system}
In this work, we study the Einstein field equations for a smooth, time oriented, strongly causal   Lorentzian manifold  $(\spacetime, g)$ in the presence of matter
\begin{equation}
\label{EFE}
    Ric(g)-\frac{1}{2}gR(g) = 8\pi T(g),
\end{equation}
where $Ric$ denotes the \textit{Ricci curvature tensor} of $g$, $R$ denotes the \textit{scalar curvature} and $T$ denotes the \textit{energy-momentum tensor}  which must be specified by the matter theory. The model considered here is  the Vlasov matter. It is assumed that the latter  is represented by a scalar positive function $f:T\spacetime\to \mathbb R_+$  called the \textit{distribution function}. The condition that $f$ represents the distribution of a collection of particles moving freely in the given spacetime is that it should be constant along the \textit{geodesic  flow}, that is
\begin{equation}
\label{Liouville::}
L[f] = 0, 
\end{equation} 
where $L$ denotes the \textit{Liouville vector field}. The latter equation is called \textit{the Vlasov equation}.  In a local coordinate chart $(x^\alpha, v^\beta)$ on $T\spacetime$, where $(v^\beta)$ are the components of the four-momentum corresponding to $x^\alpha$, the Liouville vector field  $L$ reads  
\begin{equation}
\label{Liouville:VF}
L = v^\mu\frac{\partial }{\partial x^\mu}-\Chris{\mu}{\alpha}{\beta}(g)v^\alpha v^\beta\frac{\partial }{\partial v^\mu}
\end{equation}
and the corresponding integral curves satisfy the geodesic equations of motion
\begin{equation}
\label{eq::motion1}
\left\{
\begin{aligned}
&\frac{dx^\mu}{d\tau}(\tau) = v^\mu, \\
&\frac{dv^\mu}{d\tau}(\tau) = -\Chris{\mu}{\alpha}{\beta}v^\alpha v^\beta,
\end{aligned}
\right. 
\end{equation}
where $\displaystyle \Chris{\mu}{\alpha}{\beta}$ are the Christoffel symbols given in the local coordinates $x^\alpha$ by 
\begin{equation*}
\Chris{\mu}{\alpha}{\beta} = \frac{1}{2}\ginv{\mu}{\nu}\left( \frac{\partial\g{\beta}{\nu}}{\partial x^\alpha}+\frac{\partial\g{\alpha}{\nu}}{\partial x^\beta} - \frac{\partial \g{\alpha}{\beta}}{\partial x^\nu}   \right)
\end{equation*}
and where $\tau$ is an affine parameter which corresponds to the proper time in the case of timelike geodesics. The trajectory of a particle in $T\spacetime$ is an element of the geodesic flow generated by $L$ and its projection onto the spacetime manifold $\spacetime$ corresponds to a geodesic of the spacetime. In this work, we assume that all particles have the same rest mass and it is normalised to $1$. 
\\It is easy to see that the quantity $\displaystyle \mathcal L(x,v) := \frac{1}{2}v^\alpha v^\beta\g{\alpha}{\beta}$ is conserved along solutions of \eqref{eq::motion1} \footnote{We note that $\mathcal L$ is the Lagrangian of a free-particle.}. In the case of timelike geodesics, we rescale the affine parameter $\tau$ so that: 
\begin{equation}
\label{L:conservation}
\displaystyle \mathcal L(x,v) = -\frac{1}{2}. 
\end{equation}
For physical reasons, we require that all particles move on future directed timelike geodesics. Therefore, the distribution function is supported on the seven dimensional manifold of $T\spacetime$ \footnote{See \cite{sarbach2014geometry}  Lemma.7}, called the \textit{the mass shell}, denoted by $\Gamma$ and defined by 
\begin{equation}
\label{mass::shell}
\SymbolPrint \Gamma := \left\{ (x,v)\in T\spacetime : g_x(v, v)= -1,  \quad\text{and}\quad v^\alpha  \text{ is future pointing}\right\}. 
\end{equation}
We note that by construction $\Gamma$ is invariant under the geodesic flow.  
\\ 
\\  We assume that there exist local coordinates on $\spacetime$, denoted by $(x^\alpha)_{\alpha=0\cdots 3}$ defined on some open subset $U\subset\spacetime$ such that
\begin{equation*} 
\left\{\left.\frac{\partial}{\partial x^0}\right|_x, \left.\frac{\partial}{\partial x^1}\right|_x, \left.\frac{\partial}{\partial x^2}\right|_x, \left.\frac{\partial}{\partial x^3}\right|_x  \right\}, \quad\quad x\in U, 
\end{equation*}
is a basis of $T_x\spacetime$, with the property that for each $x\in U$, $\displaystyle \left.\frac{\partial}{\partial x^0}\right|_x$ is timelike and all the vectors of the form $v^i\left.\frac{\partial}{\partial x^i}\right|_x$  are spacelike. Now let, $(x^\alpha, v^\alpha)$ be a coordinate system on $T\spacetime$. Then the mass shell condition 
\begin{equation*}
\g{\alpha}{\beta}v^\alpha v^\beta = - 1 \quad\quad \text{where}\; v^\alpha \text{ is future directed }
\end{equation*}
allows to write $v^0$ in terms of $(x^\alpha, v^a)$. It is given by 
\begin{equation*}
v^0 = -(\g{0}{0})^{-1}\left(\g{0}{j}v^j +\sqrt{(\g{0}{j}v^j)^2-\g{0}{0}(1 + \g{i}{j}v^iv^j)}\right).
\end{equation*}
Therefore, $\Gamma$ can be parametrised by $(x^0, x^a, v^a)$.  Hence, the distribution function can be written as a function of $(x^0, x^a, v^a)$ and the Vlasov equation has the form
\begin{equation}
\label{Vlasov2}
\displaystyle \frac{\partial f}{\partial x^0} + \frac{v^a}{v^0}\frac{\partial f}{\partial x^a} - \Chris{a}{\alpha}{\beta}\frac{v^\alpha v^\beta}{v^0}\frac{\partial f}{\partial v^a} = 0.
\end{equation}
In order to define the energy-momentum tensor which couples the Vlasov equation to the Einstein field equations, we introduce the natural  volume element on the fibre $$\Gamma_{x}:= \left\{ v^\alpha\in T_x\spacetime \;:\;  \ginv{\alpha}{\beta}v_\alpha v_\beta = - 1, \; v^0>0 \right\}$$ of $\Gamma$ at a point $x\in\spacetime$ given in  the adapted local coordinates  $(x^0, x^a, v^a)$ by
\begin{equation}
\label{vol::form}
d\text{vol}_x(v) := \frac{\sqrt{-\det{(\g{\alpha}{\beta})}}}{-v_0}dv^1dv^2dv^3.  
\end{equation}
The energy momentum tensor is now defined by
\begin{equation}
\label{EM_tensor}
\forall x\in\spacetime\;\quad\T{\alpha}{\beta}(x):= \int_{\Gamma_{x}} v_\alpha v_\beta f(x,v)\;d\text{vol}_x(v),
\end{equation}
where $f = f(x^0,x^a, v^a)$ and $d\text{vol}_x(v) = d\text{vol}_x(v^a)$\footnote{The latin indices run from 1...3.}. In order for \eqref{EM_tensor} to be well defined, $f$ has to have certain regularity and decay properties. One sufficient requirement would be to demand that $f$ has compact support on  $\Gamma_x$, $\forall x\in \spacetime$ and it is integrable with respect to $v$. Finally, we refer to \eqref{EFE} and \eqref{Liouville::} with $T$ given by \eqref{EM_tensor} as the Einstein-Vlasov system.

\subsection{Stationary and axisymmetric black holes with matter}
\label{Stat:axis:coord}
We recall from \cite{chodosh2015stationary} the geometric framework for the construction of non-vacuum black holes whose metrics are stationary and axisymmetric. We refer to \cite{chrusciel2012stationary} for general definitions on  exterior and black holes regions, the event horizon and its properties in the axisymmetric case.  
\subsubsection{Metric ansatz}
\label{s:metric:ansatz}
Let $\SymbolPrint \spacetime := \left\{ (t, \phi, \rho, z)\in\mathbb R\times(0, 2\pi)\times \BB \right\}$, where 
\begin{equation}
\label{B::outer}
\SymbolPrint{\BB} := \left\{ \rho > 0 \;,\; z\in\mathbb R\right\}.
\end{equation}
We will assume that the exterior regions of our spacetimes, minus the axis of symmetry, are given by $(\spacetime, g)$ where the Lorentzian metrics g take the form
\begin{equation}
\label{metric:ansatz}
g := -\SymbolPrint{V}dt^2 +2 \SymbolPrint W dtd\phi + \SymbolPrint Xd\phi^2 + e^{2\SymbolPrint \lambda}\left(d\rho^2 +dz^2\right)
\end{equation}
for suitable functions $V, W, X, \lambda :  \BB\to\mathbb R$. Observe that the vector fields $\displaystyle\overline \Phi := \frac{\partial}{\partial\phi} $ and $\displaystyle T := \frac{\partial }{\partial t}$ are both Killing. We will always assume that $X > 0$ (otherwise there would exist closed causal curves) and that $\displaystyle XV +W^2 >0$, which is equivalent to $g$ being a Lorentzian metric. We do not assume that $V > 0$. Thus, we allow for the presence of an ergoregion.
\\ In the following,  we replace the metric components $V, W, X, \lambda$ by a different collection of data $X, W, \theta, \sigma, \lambda $, which reduces under symmetries in a nice manner, where
\begin{itemize}
\item $\theta$ denotes the twist one-form associated to $\Phi$:
\begin{equation}
\label{twist:phi}
\SymbolPrint \theta := 2\iota_\Phi(\ast\nabla\Phi_\flat). 
\end{equation}
\item $\sigma$ denotes the square root of the negative of the area of the parallelogram in $T\spacetime$ spanned by $T$ and $\overline\Phi$:
\begin{equation}
\label{def:sigma}
\SymbolPrint \sigma := \sqrt{XV + W^2}.
\end{equation}
\end{itemize}
We will refer to the quantities $X, W, \theta, \sigma, \lambda$ as the "metric data". 
\subsubsection{The conformal manifold with corners $\Bbarre$}
\label{bbarre}
In this section, we recall the construction of the conformal manifold with corners $\Bbarre$ made in \cite{chodosh2015stationary} on which the metric components will be extended. 
Let $\beta>0$ and let $\displaystyle 0 < e < c < a < b$ and be sufficiently large \footnote{Theses constants will be fixed below, see Section \ref{choice:of:cons}}. 
\begin{itemize}
    \item First we define four submanifolds of $\mathscr{B}$,    
    \begin{equation*}
    \begin{aligned}
               \SymbolPrint{\mathscr{B}_{A}^{(\beta)}}&:= \left\{(\rho,z)\in\mathscr{B}, \rho^2+(z\pm\beta)^2>\frac{\beta}{a}, |z|+|\rho|>\left( 1 - \frac{1}{b}\right)\beta   \right\}, \\
         \SymbolPrint{\mathscr{B}_{H}^{(\beta)}}&:=  \left\{ (\rho,z)\in\mathscr{B}, \rho^2+(z\pm\beta)^2>\frac{\beta}{a}, |z|+|\rho|<\left( 1 + \frac{1}{b}\right)\beta   \right\}, \\       
           \SymbolPrint{\mathscr{B}_{N}^{(\beta)}}& :=  \left\{ (\rho,z)\in\mathscr{B}, \rho^2+(z-\beta)^2<\frac{\beta}{c} \right\},  \\
      \SymbolPrint {\mathscr{B}_{S}^{(\beta)}}&:=  \left\{ (\rho,z)\in\mathscr{B}, \rho^2+(z+\beta)^2<\frac{\beta}{c} \right\},
    \end{aligned}
    \end{equation*} 
        so that they cover the domain of outer communications $\mathscr B$. 
        \item Then, we  glue the points  
        \begin{equation*}
            \left\{ (0,z) / |z\pm\beta|>  \sqrt{\frac{\beta}{a}} \right\}.
        \end{equation*}
        to $\mathscr{B}_H^{(\beta)}$ in order to get
        \begin{equation*}
            \SymbolPrint{\overline{\mathscr{B}_{H}}^{(\beta)}}:= \left\{(\rho,z)\in\overline{\mathscr{B}}, \rho^2+(z\pm\beta)^2>\frac{\beta}{a}, |z|+|\rho|<\left( 1 + \frac{1}{b}\right)\beta   \right\}.
        \end{equation*}
        Similarly, we glue the points 
        \begin{equation*}
            \left\{ (0,z) / |z\pm \beta |>\sqrt{\frac{\beta}{a}} \right\}.
        \end{equation*}
        to $\mathscr{B}_A^{(\beta)}$ so that we define 
        \begin{equation*}
            \SymbolPrint{\overline{\mathscr{B}_{A}}^{(\beta)}}:= \left\{(\rho,z)\in\overline{\mathscr{B}}, \rho^2+(z\pm\beta)^2>\frac{\beta}{a}, |z|+|\rho|>\left( 1 - \frac{1}{b}\right)\beta   \right\}.
        \end{equation*}

        \item We extend in the same way $\mathscr{B}_N^{(\beta)}$ to $\Bbarre_N^{(\beta)}$ and $\mathscr{B}_S^{(\beta)}$ to $\Bbarre_S^{(\beta)}$ by gluing the points 
        \begin{equation*}
            \left\{ (0,z) / |z-\beta|< \sqrt{\frac{\beta}{c}} \right\}.
        \end{equation*}
        to $\mathscr{B}_N^{(\beta)}$ and the points 
        \begin{equation*}
            \left\{ (0,z) / |z+\beta|< \sqrt{\frac{\beta}{c}} \right\}.
        \end{equation*}
        to $\mathscr{B}_S^{(\beta)}$.
\end{itemize}        
However, at the end points $p_N = (0, \beta)$ and $p_S = (0, -\beta)$,  the Kerr solution expressed in isothermal coordinates is not smooth. In order to overcome this difficulty, one can introduce a regularisation which consists of a new system of coordinates $(s, \chi)$ with respect to which the Kerr solution is $C^{\infty}$. 
\begin{itemize}
         \item Now we define the change of coordinates on  $\mathscr{B}_N^{(\beta)}\backslash\left\{ (0,\beta)\right\}$
        \begin{equation}\label{regularisation}
            \rho=s\chi, \quad\quad z=\frac{1}{2}(\chi^2-s^2)+\beta.
        \end{equation}
 	\item We glue the points $\displaystyle \left\{ (s,\chi) \quad / 0\le s,\chi<\left(\frac{\beta}{e}\right)^{\frac{1}{4}} \right\}$ to $\mathscr{B}_N^{(\beta)}$ so that we add the north pole. Hence we obtain
        \begin{equation*}
            \SymbolPrint{\overline{\mathscr{B}_{N}}^{(\beta)}} :=  \left\{ (\rho,z)\in\Bbarre,\; z\neq \beta,\;  \rho^2+(z-\beta)^2<\frac{\beta}{c} \right\}\cup \left\{ (s,\chi)\in\Bbarre \quad / 0\le s,\chi<\left(\frac{\beta}{e}\right)^{\frac{1}{4}} \right\} .
        \end{equation*}
      
        \item Similarly, we introduce the regularisation
        on  $\mathscr{B}_S^{(\beta)}\backslash\left\{ (0,-\beta)\right\}$
        \begin{equation*}
            \rho=s'\chi', \quad\quad z=\frac{1}{2}((\chi')^2-(s')^2)-\beta.
        \end{equation*}
        so that we construct $\overline{\mathscr{B}_{S}}^{(\beta)}$
        \begin{equation*}
            \SymbolPrint{\overline{\mathscr{B}_{S}}^{(\beta)}} :=  \left\{ (\rho,z)\in\Bbarre, \; z\neq-\beta, \;  \rho^2+(z+\beta)^2<\frac{\beta}{c} \right\}\cup \left\{ (s',(\chi)')\in\Bbarre  \quad / 0\le s',(\chi)'<\left(\frac{\beta}{e}\right)^{\frac{1}{4}} \right\} .
        \end{equation*}
        \item Finally, define $\overline{\mathscr{B}}^{(\beta)}$ to be
        \begin{equation*}
            \overline{\mathscr{B}}^{(\beta)} :=  \overline{\mathscr{B}_{\mathcal{A}}}^{(\beta)}\cup\overline{\mathscr{B}_{\mathcal{H}}}^{(\beta)}\cup\overline{\mathscr{B}_{N}}^{(\beta)}\cup\overline{\mathscr{B}_{S}}^{(\beta)}.
        \end{equation*}
\end{itemize}
$\beta$ will be fixed in Remark \ref{fix:kappa}. Finally, we define the following regions of $\partial\Bbarre$:
\begin{enumerate}
\item The axis $\Axis$ is defined to be the region 
\begin{equation}
\SymbolPrint{\Axis} := \left\{(\rho, z)\in\Bbarre\;;\; \rho = 0 \quad\text{and}\quad z\in]-\infty, -\beta[\cup]\beta, +\infty[\right\}. 
\end{equation} 
\item The horizon $\Horizon$ is defined to be the region 
\begin{equation}
\SymbolPrint{\Horizon} := \left\{(\rho, z)\in\Bbarre\;;\; \rho = 0 \quad\text{and}\quad z\in]-\beta, \beta[\right\}
\end{equation} 
\end{enumerate} 
and we refer to Figure \ref{Bbarre} for a diagram of $\Bbarre$. 

\begin{figure}
    
    \begin{center}
\begin{tikzpicture}
\draw[->] (-2,0) -- (5,0);
\draw (5,0) node[right] {$\rho$};
\draw [->] (0,-5) -- (0,5);
\draw (0,5) node[above] {$z$};
\draw[red] (0,-2.5) node[left] {$-\beta$} arc (-90:90:0.5)  ;
\draw[red] (0,1.5) node[left] {$\beta$} arc (-90:90:0.5)  ;
\draw[green] (0,-2.75)  arc (-90:90:0.75)  ;
\draw[green] (0,1.25)  arc (-90:90:0.75)  ;
\foreach \Point in {(0,-2), (0,2)}{
    \node at \Point {\textbullet};
}


\draw[blue] (2.25,0)--(0.5,1.8);
\draw[blue,dashed] (0.5,1.8)--(0,2.25);

\draw[blue,dashed] (0,-2.25)--(0.5,-1.8);
\draw[blue] (0.5,-1.8)--(2.25,0);

\draw[blue] (0,-1.5)--(0,1.5);
\draw[blue] (0,-1.5)  arc (90:20:0.5)  ;
\draw[blue] (0,1.5)  arc (-90:-25:0.5)  ;

\draw[magenta] (1.75,0)--(0.25,1.55);
\draw[magenta,dashed] (0.25,1.55)--(0,1.75);

\draw[magenta,dashed] (0,-1.75)--(0.25,-1.55);
\draw[magenta] (0.25,-1.55)--(1.75,0);

\draw[magenta] (0,-5)--(0,-2.5) ;
\draw[magenta] (0,2.5)--(0,5);

\draw[magenta] (0,-2.5) arc (-90:50:0.5) ;
\draw[magenta] (0,2.5) arc (90:-60:0.5) ;
\end{tikzpicture}    
\end{center}

    \caption{Boundaries of $\BAbarre$ (magenta), of $\BHbarre$ (blue), $\Bnbarre$ and $\Bsbarre$ (green).}
    \label{Bbarre}
\end{figure}
\noindent Now, we define a partition of unity subordinate to $\displaystyle (\BAbarre\cup\BHbarre, \Bnbarre, \Bsbarre)$. To this end, we first give a definition of smooth functions on $\Bbarre$
\begin{definition}
Let $f:\Bbarre\mapsto\mathbb R$. $f$ is said to be smooth on $\Bbarre$ if and only if 
\begin{enumerate}
\item $f_{\left|\Bnbarre \right.}$ is smooth on $\text{Int}\left(\Bnbarre\right)$ and extends smoothly to $\partial\left(\Bnbarre\right)$, 
\item $f_{\left|\Bsbarre \right.}$ is smooth on $\text{Int}\left(\Bsbarre\right)$ and extends smoothly to $\partial\left(\Bsbarre\right)$, 
\item $f_{\left|(\BAbarre\cup\BHbarre) \right.}$ is smooth on $\text{Int}\left(\BAbarre\cup\BHbarre\right)$ and extends smoothly to $\partial\left(\BAbarre\cup\BHbarre\right)$.
\end{enumerate}
\end{definition}
\noindent Moreover, we give the following definition
\begin{definition}
\label{partition:unity}
We define $\xi_N$ and $\xi_S$ in the following way:  $\xi_N,\xi_S:\Bbarre\mapsto[0, 1]$ such that they are smooth and they verify 
\begin{itemize}
\item $\partial_{\rho}\xi_N = \partial_{\rho}\xi_S = 0$ for $\rho$ small,
\item   $supp(\xi_N)\subset\Bnbarre$, $supp(\xi_S)\subset\Bsbarre$ , and $supp(1-\xi_N-\xi_S) \subset\BHbarre\cup\BAbarre$. 
\end{itemize}
\end{definition}
Note that $\xi_N = 1$ and $\xi_S = 1$ in a neighbourhood of $p_N$ and $p_S$ so that their support does not lie in the region $\BAbarre\cup\BHbarre$. Now, from the above definition,  we claim that 
\begin{lemma}
$(\xi_N, \xi_S, 1 - \xi_N - \xi_S)$ is a smooth partition of unity subordinate to $\displaystyle (\BAbarre\cup\BHbarre, \Bnbarre, \Bsbarre)$. 
\end{lemma}
\noindent In the remaining of this work, we will use the following notations 
\begin{itemize}
            \item the gradient with respect to $(\rho,z)$ coordinates, 
            
            \begin{equation*}
                \partial f = (\partial_{\rho} f, \partial_{z} f),    
            \end{equation*}
            
            \item the gradient with respect to $(s,\chi)$ coordinates, 
            
            \begin{equation*}
                \underline{\partial} f = (\partial_{s} f, \partial_{\chi} f),
            \end{equation*}
            
            \item a renormalised gradient norm 
            \begin{equation}
            \label{hat:deri}
            |\hat{\partial}f|:= |\nabla_{\mathbb{R}^4}(\xi_Nf)_{\mathbb{R}^4}|+|\nabla_{\mathbb{R}^4}(\xi_Sf)_{\mathbb{R}^4}|+|\nabla_{\mathbb{R}^3}((1-\xi_N-\xi_S)f)_{\mathbb{R}^3}|.    
            \end{equation}
\end{itemize}
Finally, direct computations imply 
\begin{equation*}
\label{grad1}
\forall (s, \chi)\neq (0, 0) \quad;\quad   \partial_{\rho} = \frac{\chi}{\chi^2+s^2}\partial_s + \frac{s}{\chi^2 + s^2}\partial_\chi, \quad\quad  \partial_z = \frac{-s}{\chi^2+s^2}\partial_s + \frac{\chi}{\chi^2+s^2}\partial_\chi
\end{equation*}
and
\begin{equation*}\label{grad2}
\partial_s = \chi\partial_{\rho}-s\partial_z, \quad\quad \partial_{\chi} = s\partial_{\rho} + \chi\partial_z.
\end{equation*}
Moreover, we have 
\begin{equation*}
\left|\partial f \right|^2 = \frac{1}{s^2 + \chi^2}\left|\underline \partial f \right|^2. 
\end{equation*}

\subsubsection{Boundary conditions and extendibility}
We are interested in finding asymptotically flat spacetimes with an event horizon. Therefore, in addition to the equations, the metric coefficients have to satisfy appropriate asymptotic behaviours as $\rho\to 0$ and at infinity so that  the spacetime $(\spacetime, g)$ extends to a larger Lorentzian manifold with boundary $(\tilde M, \tilde g)$ which is asymptotically flat and has a boundary consisting of a non-degenerate bifurcate Killing event horizon. 
\noindent Now, we state definitions of extendibility of $(\spacetime, g)$ to a larger manifold as well as asymptotic flatness in the current context. We refer to \cite{chodosh2015stationary} and \cite{chrusciel2012stationary} for more details.

Let $\alpha\in(0, 1)$ and let $k\in \mathbb N$. 
\label{extendibility}
\begin{definition}[Extendibility around the axis]
\label{ext:A}
Let $\tilde\Axis\subset\BAbarre$ be an open set around the axis $\Axis$. We say that a stationary axisymmetric spacetime $(\mathcal{M}, g)$ is extendable (resp. $C^{k, \alpha}-$ extendable) along $\partial\BAbarre\cap\Axis$ 
if \begin{enumerate}
    \item there exists a smooth function $V_{\Axis}(\rho,z):\tilde{\Axis}\to\mathbb{R}$ such that $V_{\tilde\Axis}(0,z)>0$ and
    \[
    \left. V(\rho,z)\right|_{\tilde\Axis} = V_{\Axis}(\rho^2,z),
    \]
    \item there exists a smooth function $W_{\Axis}(\rho,z):\tilde{\Axis}\to\mathbb{R}$ such that 
    \[
    \left. W(\rho,z)\right|_{\tilde{\Axis}} = \rho^2W_{\Axis}(\rho^2,z),
    \]
    \item there exists a smooth function $X_{\Axis}(\rho,z):\tilde{\Axis}\to\mathbb{R}$ such that $X_{\Axis}(0,z)>0$ and
    \[
    \left. X(\rho,z)\right|_{\tilde{\Axis}} = \rho^2X_{\Axis}(\rho^2,z),
    \]
    \item there exists a smooth function $\Sigma_{\Axis}(\rho,z):\tilde{\Axis}\to\mathbb{R}$ such that 
    \[
    \left. e^{2\lambda(\rho,z)}\right|_{\tilde{\Axis}} = X_{\Axis}(\rho^2,z) + \rho^2\Sigma_{\Axis}(\rho^2,z),
    \]
\end{enumerate}    
\end{definition}

\begin{definition}[Extendibility around the horizon]
\label{ext:H}
Let $\tilde{\Horizon}\subset\BHbarre$ be an open set around the horizon $\Horizon$.

We say that a stationary axisymmetric spacetime $(\mathcal{M}, g)$ is extendable (resp. $C^{k, \alpha}-$ extendable) along $\partial\BHbarre\cap\Horizon$ if there exists $\Omega\in\mathbb{R}$, $\kappa>0$ such that 
\begin{enumerate}
    \item there exists a smooth function $V_{\Horizon}(\rho,z):\tilde{\Horizon}\to\mathbb{R}$ such that $V_{\Horizon}(0,z)>0$ and
    \[
    \left.(V(\rho,z)-2\Omega W(\rho,z) - \Omega^2X(\rho,z))\right|_{\tilde{\Horizon}} = \rho^2V_{\Horizon}(\rho^2,z),
    \]
    \item there exists a smooth function $W_{\Horizon}(\rho,z):\tilde{\Horizon}\to\mathbb{R}$ such that 
    \[
    \left.(W(\rho,z)+\Omega X(\rho,z))\right|_{\tilde{\Horizon}} = \rho^2W_{\Horizon}(\rho^2,z),
    \]
    \item there exists a smooth function $X_{\Horizon}(\rho,z):\tilde{\Horizon}\to\mathbb{R}$ such that $X_{\Horizon}(0,z)>0$ and
    \[
    \left. X(\rho,z)\right|_{\tilde{\Horizon}} = X_{\Horizon}(\rho^2,z),
    \]
    \item there exists a smooth function $\Sigma_{\Horizon}(\rho,z):\tilde{\Horizon}\to\mathbb{R}$ such that 
    \[
    \left. e^{2\lambda(\rho,z)}\right|_{\tilde{\Horizon}} = \kappa^{-2} V_{\Horizon}(\rho^2,z) + \rho^2\Sigma_{\Horizon}(\rho^2,z).
    \]
\end{enumerate}    
\end{definition}

\begin{definition}[Extendibility around $p_N$]
\label{ext:N}
We say that a stationary axisymmetric spacetime $(\mathcal{M}, g)$ is extendable (resp. $C^{k, \alpha}-$ extendable) along $\partial\Bnbarre$ if there exists $\Omega\in\mathbb{R}$, $\kappa>0$ such that 
\begin{enumerate}
    \item there exists a smooth function $V_{N}(s,\chi):\Bnbarre\to\mathbb{R}$ such that $V_{N}(0,z)>0$ and
    \[
    \left.(V(s,\chi)-2\Omega W(s,\chi) - \Omega^2X(s,\chi))\right|_{\Bnbarre} = \chi^2V_{N}(s^2,\chi^2),
    \]
    \item there exists a smooth function $W_{N}(s,\chi):\Bnbarre\to\mathbb{R}$ such that
    \[
    \left.(W(s,\chi)+\Omega X(s,\chi))\right|_{\Bnbarre} = s^2\chi^2W_{N}(s^2,\chi^2),
    \]
    \item there exists a smooth function $X_{N}(s,\chi):\Bnbarre\to\mathbb{R}$ such that $X_{N}(0,\chi)>0$, $X_{N}(s,0)>0$, and
    \[
    \left.X(s,\chi)\right|_{\Bnbarre}= s^2X_{N}(s^2,\chi^2),
    \] 
    \item there exists a smooth function $\Sigma_{N}^{(1)}(s,\chi), \Sigma_{(2)}(s,\chi):\Bnbarre\to\mathbb{R}$ such that 
    \begin{equation*}
    \begin{aligned}
    \left.e^{2\lambda(s,\chi)}\right|_{\Bnbarre} &= X_N(s^2,\chi^2) + s^2\Sigma_{N}^{(1)}(s^2,\chi^2), \\
    \left.e^{2\lambda(s,\chi)}\right|_{\Bnbarre} &= \kappa^{-2}V_N(s^2,\chi^2) + \chi^2\Sigma_{N}^{(2)}(s^2,\chi^2), 
    \end{aligned}
    \end{equation*}
\end{enumerate}    
\end{definition}

\begin{definition}[Extendibility around $p_S$]
\label{ext:S}
We say that a stationary axisymmetric spacetime $(\mathcal{M}, g)$ is extendable (resp. $C^{k, \alpha}-$ extendable) along $\partial\Bsbarre$ if there exists $\Omega\in\mathbb{R}$, $\kappa>0$ such that 
\begin{enumerate}
    \item there exists a smooth function $V_{S}(s',\chi'):\Bsbarre\to\mathbb{R}$ such that $V_{s}(0,z)>0$ and
    \[
    \left.(V(s',\chi')-2\Omega W(s',\chi') - \Omega^2X(s',\chi'))\right|_{\Bsbarre} = \chi^2V_{S}((s')^2, (\chi')^2),
    \]
    \item there exists a smooth function $W_{S}(s,\chi):\Bsbarre\to\mathbb{R}$ such that
    \[
    \left.(W(s',\chi')+\Omega X(s',\chi'))\right|_{\Bsbarre} = (s')^2(\chi')^2W_{S}(s^2,\chi^2),
    \]
    \item there exists a smooth function $X_{S}(s',\chi'):\Bsbarre\to\mathbb{R}$ such that $X_{S}(0,\chi')>0$, $X_{S}(s',0)>0$, and
    \[
    \left.X(s',\chi')\right|_{\Bsbarre}= (s')^2X_{S}((s')^2,(\chi')^2),
    \] 
    \item there exists a smooth function $\Sigma_{S}(s',\chi'):\Bsbarre\to\mathbb{R}$ such that 
    \begin{equation*}
    \begin{aligned}
    \left.e^{2\lambda(s',\chi')}\right|_{\Bsbarre} &= X_S((s')^2,(\chi')^2) + s^2\Sigma_{S}^{(1)}((s')^2,(\chi')^2) \\
    \left.e^{2\lambda(s',\chi')}\right|_{\Bsbarre} &= \kappa^{-2}V_S((s')^2,(\chi')^2) + (\chi')^2\Sigma_{S}^{(2)}((s')^2,(\chi')^2).
    \end{aligned}
    \end{equation*}
\end{enumerate}       
\end{definition}

\begin{Propo}
\label{extendable}
Let $(\spacetime, g)$ be a stationary and axisymmetric spacetime which is extendable along $\partial \BAbarre\cap\Axis$, $\partial \BAbarre\cap\Axis$, $\partial \Bnbarre$ and $\partial \Bsbarre$.   Then $(\spacetime, g)$  is extendable to a Lorentzian manifold with corners $(\tilde\spacetime, \tilde g)$ which is stationary and axisymmetric, and whose boundary corresponds to a bifurcate Killing event horizon.
\end{Propo}
We refer to \cite{chodosh2015stationary} for a proof.

\begin{definition}[Extendability]
\label{extendable:1}
Let $\alpha\in(0, 1)$ and let $k\in\mathbb N$. We say that a stationary and axisymmetric spacetime $(\spacetime, g)$ is “extendable to a regular black hole spacetime” if $(\spacetime, g)$ satisfies the assumptions of Proposition \ref{extendable}.
\end{definition}

\noindent Now, we recall from \cite{chodosh2015stationary} a definition for asymptotic flatness, convenient to our work we refer to the Appendix A therein for more details. 
\begin{definition}[Asymptotic flatness]
\label{asymptotic::flatness}
We say that a stationary and axisymmetric spacetime $(\spacetime , g)$ is asymptotically flat if in the $(t, x, y, z)$ coordinates defined in Appendix \ref{coord:for:AF}  the metric $\tilde g$ verifies
\begin{equation*}
\begin{aligned}
\tilde g &= \left(1 + O\left(\frac{1}{r}\right) \right)(-dt^2 + dx^2 + dy^2 + dz^2) + O\left(\frac{1}{r^2}\right)\left(dtdx + dt dy + dx dy \right), \\
\partial \tilde g &= \left(1 + O\left(\frac{1}{r^2}\right) \right)(-dt^2 + dx^2 + dy^2 + dz^2) + O\left(\frac{1}{r^3}\right)\left(dtdx + dt dy + dx dy \right), \\
\partial^2 \tilde g &= \left(1 + O\left(\frac{1}{r^3}\right) \right)(-dt^2 + dx^2 + dy^2 + dz^2) + O\left(\frac{1}{r^4}\right)\left(dtdx + dt dy + dx dy \right), \\
\end{aligned}
\end{equation*}
where $r = \sqrt{1 + x^2 + y^2 + z^2}$. 
\end{definition}

\subsubsection{Vlasov field on stationary and axisymmetric spacetimes }
The distribution function $f$ is conserved along the geodesic flow. Hence, any function of the integrals of motion will satisfy the Vlasov equation. In this context, we  look for integrals of motion for the geodesic equation \eqref{eq::motion1} on a stationary and axially symmetric background. By symmetry assumptions, the vector fields $T$ and $\overline \Phi$ are Killing. Hence, the quantities 
\begin{equation}
\label{E:energy}
\SymbolPrint\ve := -v_t  = -\g{t}{\alpha}v^\alpha,
\end{equation}
and
\begin{equation}
\label{L:momentum}
\SymbolPrint{\ell_z} := v_\phi = \g{\phi}{\alpha}v^\alpha
\end{equation}
are conserved.  Note that $\ve$ and $\ell_z$ are interpreted as the energy relative to infinity per unit mass and the azimutal angular momentum per unit mass respectively. 
\\ We assume that the distribution function is supported on the set of trapped geodesics in the exterior region in order to obtain a shell of matter with finite mass and  located away from the horizon, see Section \ref{Ansatz:for:the:distribution:function}. Therefore, if $f$ is a function of $(\ve,\ell_z)$, there exists $\Phi $ such that $f(x, v) = \Phi(\ve, \ell_z)$.
For our construction, we will require that $\Phi$  is supported on a subset of the set of parameters $(\ve, \ell_z)$ leading to trapped geodesics.  However, as in the spherically symmetric case, for a given $(\ve, \ell_z)$,  the support of $\Phi(\ve, \ell_z)$ in the $(\rho, z)$ variables has two connected components: one corresponds to trapped and the the other one corresponds to orbits that reach the horizon in a finite proper time. Hence, the above ansatz will be modified in order to obtain a shell of matter with compact support, see \eqref{ansatz:for:f} for the exact ansatz for $f$. 

\subsubsection{Effective potential energy in stationary and axisymmetric spacetimes}
\label{2:DOF}
We are interested in future directed timelike geodesics moving in stationary, axisymmetric and asymptotically flat spacetimes described by a metric of the form \eqref{metric:ansatz}. 
\\ Since we allowed the presence of an ergoregion (recall that $V$ is not assumed to be positive on the whole exterior region), $T$ is a priori not timelike everywhere. In order to fix the time orientation, we introduce the following vector field $\Omega$ defined by 
\begin{equation*}
\SymbolPrint \Omega := \frac{\partial}{\partial t} + \omega\frac{\partial}{\partial \phi} \quad\quad \omega := -\frac{W}{X}. 
\end{equation*}
First, note that $\Omega$ is timelike on $\BB$. In fact, 
\begin{equation*}
g(\Omega, \Omega) = - \frac{XV + W^2}{X} = -\frac{\sigma^2}{X} < 0. 
\end{equation*}
Hence, we choose $\Omega$ for the time orientation. Now, let $(v^t, v^\phi, v^\rho, v^z)\in\mathbb R^4$ be the conjugate coordinates to the spacetime coordinates $(t, \phi, \rho, z)$ and let $\displaystyle v = v^\alpha\frac{\partial}{\partial x^\alpha}\in T_x\spacetime$ such that $\displaystyle g(v, v) = -1$. We compute 
\begin{align*}
g(\Omega, v) &= g(\frac{\partial}{\partial t} + \omega\frac{\partial}{\partial \phi}, v^\alpha\frac{\partial}{\partial x^\alpha}) \\
&=  v^tg\left(\frac{\partial}{\partial t}, \frac{\partial}{\partial t}\right) + v^\phi g\left(\frac{\partial}{\partial t}, \frac{\partial}{\partial \phi}\right) + v^t\omega g\left(\frac{\partial}{\partial \phi}, \frac{\partial}{\partial t}\right) + v^\phi\omega g\left(\frac{\partial}{\partial \phi}, \frac{\partial}{\partial \phi}\right) \\
&= -v^t V + v^\phi W + v^t\omega W+ v^\phi\omega X \\
&= -v^t\frac{XV + W^2}{X} = -v^t\frac{\sigma^2}{X}. 
\end{align*}
Therefore, the requirement   
\begin{equation*}
g(\Omega, v)  < 0
\end{equation*}
is equivalent to $v^t>0$.
\\ We define the mass shell by 
\begin{equation*}
\Gamma =  \left\{ (x,v)\in T\spacetime : g_x(v, v)= - 1,  \quad\text{and}\quad  g(\Omega, v)  < 0\right\}.
\end{equation*} 
We henceforth consider only future directed timelike particles. 
\\ In the presence of two Killing vector fields for the spacetime, the problem of solving the geodesic equation, which consists of integrating a system of 8 ordinary differential equations, is reduced to a problem with two-degree of freedom defined on a four dimensional submanifold of the tangent bundle. The remaining of this section is devoted to the reduction of the geodesics equation and to the introduction of a two dimensional effective potential that will play a key role for the classification of the orbits.


Let $(t, \phi, \rho, z)\in\spacetime$ and  let $ (v^t, v^\phi, v^\rho, v^z)$ be the conjugate coordinates to the spacetime coordinates. We recall that the quantity (free particle Lagrangian)
\begin{equation*}
\mathcal L (x, v):= \frac{1}{2}\g{\alpha}{\beta}v^\alpha v^\beta. 
\end{equation*}
is conserved along the geodesic flow and normalised to $\displaystyle -\frac{1}{2}$. Hence, we obtain
\begin{equation}
\label{m:shell} 
e^{2\lambda}\left( (v^\rho)^2 + (v^z)^2 \right) = -1  +  \frac{X}{\sigma^2}\ve^2 + \frac{2W}{\sigma^2}\ve\ell_z - \frac{V}{\sigma^2}\ell_z^2. 
\end{equation}
and 
\begin{equation*}
v^t >0. 
\end{equation*}
We set 
\begin{equation}
\label{potential:1:bis}
\SymbolPrint J(\rho, z, \varepsilon, \ell_z) :=  -1 +  \frac{X}{\sigma^2}\varepsilon^2 + \frac{2W}{\sigma^2}\varepsilon\ell_z - \frac{V}{\sigma^2}\ell_z^2.
\end{equation}
In this work, we will have to distinguish particles which co-rotate with the black holes and particles which counter-rotate with the black hole. To this end, we state the following definitions
\begin{definition}
\label{co:counter}
Let $\gamma: I \to\mathbb \spacetime$ be a timelike geodesic with angular momentum $\ell_z$. $\gamma$ is said to be direct (or co-rotating) if $\displaystyle -W\ell_z >0$ and retrograde (or counter-rotating) if $\displaystyle -W\ell_z <0$. 
\end{definition}
Now,  let $\gamma:I\to\mathcal M$ be a timelike future directed geodesic  in the spacetime, defined on some interval $I\subset\mathbb R$. In the adapted coordinates for the tangent bundle, we have 
\begin{equation*}
\gamma(\tau) = (t(\tau), \phi(\tau), \rho(\tau), z(\tau))\;\; \text{and}\;\; \gamma'(\tau) = (v^t(\tau), v^\phi(\tau), v^\rho(\tau), v^z(\tau)).
\end{equation*}
Besides, 
\begin{equation}
\label{m:shell:bis}
e^{2\lambda}\left( (v^\rho)^2 + (v^z)^2 \right) = J(\rho, z, \ve, \ell_z, d) 
\end{equation}
so that 
\begin{equation}
\label{allowed::region}
J(\rho, z, \varepsilon, \ell_z) \geq 0. 
\end{equation}
along $\gamma$. 
Moreover, $\gamma$ satisfies the geodesics equation
\begin{equation}
\label{geodesic:equation}
\left\{
\begin{aligned}
\frac{dx^\mu}{d\tau} &= v^\mu \\
\frac{dv^\mu}{d\tau} &= - \Chris{\mu}{\alpha}{\beta}(x) v^\alpha v^\beta. 
\end{aligned}
\right.
\end{equation}
We have 
\begin{equation}
\label{v:t:v:phi}
v^t = \frac{X}{\sigma^2}\ve + \frac{W}{\sigma^2}\ell_z, \quad\quad v^\phi = -\frac{W}{\sigma^2}\ve + \frac{V}{\sigma^2}\ell_z. 
\end{equation}
Recall that $v^t>0$ for future directed orbits. Now, we claim that the problem of solving the geodesic equations is equivalent to solving the following reduced system
\begin{equation}
\label{geodesic::system}
\left\{
\begin{aligned}
\frac{d\rho}{d\tau} &= v^\rho ,\\
\frac{dz}{d\tau} &= v^z , \\
\frac{dv^\rho}{d\tau} &= -\frac{1}{2}e^{-2\lambda}\partial_\rho J(\rho, z, \varepsilon, \ell_z) - \Chris{\rho}{i}{j}v^iv^j , \quad i,j\in\left\{\rho, z \right\} \\ 
\frac{dv^z}{d\tau} &= -\frac{1}{2}e^{-2\lambda}\partial_z J(\rho, z, \varepsilon, \ell_z) - \Chris{z}{i}{j}v^iv^j.
\end{aligned}
\right.
\end{equation}
In fact, $\forall\, \gamma:I\to\spacetime$ satisfying \eqref{geodesic:equation}, we have 
for $k\in\left\{i, j\right\}$
\begin{equation*}
\begin{aligned}
-\frac{dv^k}{d\tau} &= \Chris{k}{\alpha}{\beta}(\gamma(\tau)) v^\alpha v^\beta = \Chris{k}{a}{b}(\gamma(\tau)) v^a v^b +  \Chris{k}{i}{j}(x) v^i v^j \quad\text{for}\quad a, b\in\left\{t, \phi\right\}, i, j\in\left\{\rho, z \right\}  \\
&= - \frac{1}{2}\ginv{k}{k}\frac{\partial(\g{a}{b}(\gamma(\tau))v^a v^b)}{\partial x^k} +  \Chris{k}{i}{j}(\gamma(\tau)) v^i v^j \\
&= - \frac{1}{2}e^{-2\lambda}\frac{\partial(\g{a}{b}(\gamma(\tau))v^a v^b)}{\partial x^k} +  \Chris{k}{i}{j}(\gamma(\tau)) v^i v^j. 
\end{aligned}
\end{equation*} 
We differentiate \eqref{potential:1:bis} with respect to $(\rho, z)$ at the point $(\rho(\tau), z(\tau))$ to obtain 
\begin{equation*}
\nabla_{(\rho, z)}J(\rho, z, \ve, \ell_z) = -\nabla_{(\rho, z)}(\g{a}{b}(\gamma(\tau))v^a v^b). 
\end{equation*} 
Hence, $\tau \mapsto (\rho, z)(\tau)$ solves \eqref{geodesic::system}. Now, assume that $\gamma: I\to\spacetime$ is a curve such that $\tau \mapsto (\rho, z)(\tau)$ solves \eqref{geodesic::system} and such that 
\begin{equation*}
\begin{aligned}
&\frac{d\ve}{d\tau} = -\frac{dv_t}{d\tau} = 0 \\
&\frac{d\ell_z}{d\tau} = -\frac{dv_\phi}{d\tau} = 0.
\end{aligned}
\end{equation*}
We claim that $\gamma$ solves \eqref{geodesic:equation}. It suffices to show that 
\begin{equation*}
\begin{aligned}
\frac{dv^t}{d\tau} &= - \Chris{t}{\alpha}{\beta}(\gamma(\tau)) v^\alpha v^\beta, \\
\frac{dv^\phi}{d\tau} &= - \Chris{\phi}{\alpha}{\beta}(\gamma(\tau)) v^\alpha v^\beta, \\
\end{aligned}
\end{equation*}
We compute, 
\begin{align*}
&\frac{dv^t}{d\tau} =  \frac{d}{d\tau}\left(\frac{X}{\sigma^2}\ve + \frac{W}{\sigma^2}\ell_z\right) \\
&= \ve\nabla_{(\rho, z)}\left(\frac{X}{\sigma^2}\right)\cdot (v^\rho(\tau)\;v^z(\tau)) + \ell_z\nabla_{(\rho, z)}\left(\frac{W}{\sigma^2}\right)\cdot (v^\rho(\tau) \;v^z(\tau)) \\
&= v^\rho\left(\ve\frac{W^2\partial_\rho X - (X\partial_\rho V + 2W\partial_\rho W)X}{\sigma^4} + \ell_z\frac{(XV - W^2)\partial_\rho W -  (V_K\partial_\rho X + X\partial_\rho V)W}{\sigma^4} \right) \\
&+v^z\left(\ve\frac{W^2\partial_z X -  (X\partial_z V + 2W\partial_\rho W)X}{\sigma^4} + \ell_z\frac{(XV - W^2)\partial_z W -  (V\partial_z X + X\partial_z V)W}{\sigma^4} \right) \\
\end{align*}
Now, note that 
\begin{equation*}
\Chris{a}{b}{i}(\gamma(\tau)) = \frac{1}{2}\ginv{a}{c}\frac{\partial \g{b}{c}}{\partial x^i} \quad a,b,c\in\left\{t, \phi \right\}, i\in\left\{\rho, z\right\}
\end{equation*}
and 
\begin{equation*}
\Chris{a}{b}{c}(\gamma(\tau)) = \Chris{a}{i}{j}(\gamma(\tau)) = 0 \quad a,b,c\in\left\{t, \phi \right\}, i, j\in\left\{\rho, z\right\}.
\end{equation*}
Hence, 
\begin{align*}
\Chris{t}{\alpha}{\beta}(\gamma(\tau))v^\alpha v^\beta &= \Chris{a}{b}{i}(\gamma(\tau))v^b v^i = v^\rho\left(\Chris{t}{t}{\rho} v^t + \Chris{t}{\phi}{\rho} v^\phi\right) + v^z\left(\Chris{t}{t}{z} v^t + \Chris{t}{\phi}{z} v^\phi\right) \\
&= v^\rho\left(\left(\frac{X}{\sigma^2}\partial_\rho V + \frac{W}{\sigma^2}\partial_\rho W\right)v^t + \left(-\frac{X}{\sigma^2}\partial_\rho W + \frac{W}{\sigma^2}\partial_\rho X\right)v^\phi \right) \\
&+ v^z\left(\left(\frac{X}{\sigma^2}\partial_z V + \frac{W}{\sigma^2}\partial_z W\right)v^t + \left(-\frac{X}{\sigma^2}\partial_z W + \frac{W}{\sigma^2}\partial_z X\right)v^\phi \right) \\
\end{align*}
We simplify 
\begin{align*}
&\ve\frac{W^2\partial_\rho X - (X\partial_\rho V + 2W\partial_\rho W)X}{\sigma^4} + \ell_z\frac{(XV - W^2)\partial_\rho W - (V\partial_\rho X + X\partial_\rho V)W}{\sigma^4} = \\
& \left(-\frac{X\partial_\rho V + W\partial_\rho W}{\sigma^2} \right)\left(\frac{X}{\sigma^2}\ve + \frac{W}{\sigma^2}\ell_z\right) + \left(\frac{X\partial_\rho W - W\partial_\rho X}{\sigma^2} \right)\left(\frac{-W}{\sigma^2}\ve + \frac{V}{\sigma^2}\ell_z\right) = \\
& \left(-\frac{X\partial_\rho V + W\partial_\rho W}{\sigma^2} \right) v^t  -\left(-\frac{X\partial_\rho W + W\partial_\rho X}{\sigma^2} \right)\ v^\phi
\end{align*}
Therefore, 
\begin{equation*}
\frac{dv^t}{d\tau} = - \Chris{t}{\alpha}{\beta}(\gamma(\tau))v^\alpha v^\beta. 
\end{equation*}
Similarly, we show that
\begin{equation*}
\frac{dv^\phi}{d\tau} = - \Chris{\phi}{\alpha}{\beta}(\gamma(\tau))v^\alpha v^\beta.
\end{equation*}
Consider again the equation
\begin{equation*}
e^{2\lambda}\left( (v^\rho)^2 + (v^z)^2 \right) = -1  +  \frac{X}{\sigma^2}\ve^2 + \frac{2W}{\sigma^2}\ve\ell_z - \frac{V}{\sigma^2}\ell_z^2.
\end{equation*}
Because of the term $\displaystyle \frac{2W}{\sigma^2}\varepsilon\ell_z$, the dependence of $J$ on $\varepsilon$ cannot be separated. In this case, we cannot write the total energy as a sum of a kinetic term and a potential term depending only on the angular momentum as in the spherically symmetric case. However, we are only interested in the turning points,  because they can be used to determine the nature of orbits. To this end, we introduce the following definition
\begin{definition}
\label{turning:point:def}
Let $\gamma$ be a timelike future directed geodesic with constants of motion $(\ve, \ell_z)$. A point $(\rho_0, z_0)(\ve, \ell_z)\in \BB$ is called  a turning point associated to $\gamma$  if it is solution to the equation  $\displaystyle  J(\rho, z,\varepsilon, \ell_z) = 0$.
\end{definition}
Now, since $J$ is quadratic in $\ve$, it is easy to write the latter quantity in terms of the remaining quantities: 
\begin{equation*}
\varepsilon = \frac{-W}{X}\ell_z \pm \frac{\sigma}{X}\sqrt{\ell_z^2 + X}. 
\end{equation*}
Since we are interested in future directed timelike orbits, $v^t>0$. Therefore, 
\begin{equation*}
\ve + \frac{W}{X}\ell_z > 0
\end{equation*}
Hence, 
\begin{equation*}
\varepsilon = \frac{-W}{X}\ell_z + \frac{\sigma}{X}\sqrt{\ell_z^2 + X}. 
\end{equation*}
Now, we define the effective potential energy to be the function $\SymbolPrint{E_{\ell_z}}$ $: \BB\to\mathbb R$ by 
\begin{equation}
\label{eff::potential}
E_{\ell_z}(\rho, z) := \frac{-W(\rho, z)}{X(\rho, z)}\ell_z + \frac{\sigma}{X(\rho, z)}\sqrt{\ell_z^2 + X(\rho, z)}.
\end{equation}
Hence, for a fixed $(\ve, \ell_z)$,  a turning point in $\BB$, say $(\rho_0, z_0)(\ve, \ell_z)$ is characterised by: 
\begin{equation*}
\varepsilon = E_{\ell_z}((\rho_0, z_0)(\ve, \ell_z))
\end{equation*}
It is convenient to make the dependence of $E_{\ell_z}$ on the metric components explicit. Therefore, we adapt the definition of $E_{\ell_z}$ and we define $E_{\ell_z}:\mathcal X\times\BB\to\mathbb R$ to be 
\begin{equation}
\label{eff::potential:g}
E_{\ell_z}(W, X, \sigma, \rho, z) := \frac{-W(\rho, z)}{X(\rho, z)}\ell_z + \frac{\sigma}{X(\rho, z)}\sqrt{\ell_z^2 + X(\rho, z)}
\end{equation}
where $\mathcal X$ is a product functional space where the metric components will live that will be defined later in this work (see Section \ref{function::spaces:bis}). 
\\ In order to determine the nature of timelike orbits, we will need to study the stationary solutions of the reduced system \eqref{geodesic::system}. Let  $(\ve, \ell_z)\in\mathbb R\times\mathbb R$.  Recall that $(\rho, z, v^\rho, v^z)(\ve, \ell_z, \cdot): I\to \mathbb \BB\times\mathbb R^2$ is a timelike future-directed stationary solution of \eqref{geodesic::system} if 
\begin{itemize}
\item $I = \mathbb R$, 
\item there exists $(\rho_s, z_s, v^\rho_s, v^z_s)(\ve, \ell_z)\in \BB\times \mathbb R^2$ such that $\forall \tau\in\mathbb R$, we have $\displaystyle (\rho, z, v^\rho, v^z)(\ve, \ell_z, \tau) = (\rho_s, z_s, v^\rho_s, v^z_s)(\ve, \ell_z)$, 
\item $(\rho_s, z_s, v^\rho_s, v^z_s)(\ve, \ell_z)$ verifies \eqref{m:shell:bis}. 
\end{itemize}
Therefore, it is easy to obtain the following lemma
\begin{lemma}
\label{general:st:sol}
Let $(\ve, \ell_z)\in\mathbb R\times \mathbb R^\ast$ and  let $(\rho_s, z_s, v^\rho_s, v^z_s)(\ve, \ell_z)$ be a timelike future-directed stationary solution  of \eqref{geodesic::system}. Then, 
\begin{equation*}
J(\rho_s, z_s, \ve, \ell_z) = 0 \quad,\quad \nabla_{(\rho, z)}\,J(\rho_s, z_s, \ve, \ell_z) = 0. 
\end{equation*}
Moreover, we have 
\begin{equation*}
v^\rho_s =  v^z_s = 0
\end{equation*}
\end{lemma}
\begin{proof}
If $(\rho_s, z_s, v^\rho_s, v^z_s)(\ve, \ell_z)$ is a stationary solution, then by the first two equations of \eqref{geodesic::system}, it verifies $v_s^\rho = v_s^z = 0$. Moreover, by the last equations we obtain 
\begin{equation*}
\frac{1}{2}e^{-2\lambda(\rho_s, z_s)}\nabla_{(\rho, z)}\,J(\rho_s, z_s, \ve, \ell_z) =  -\frac{dv^\rho_s}{d\tau} - \Chris{\rho}{i}{j}v^iv^j = 0.  
\end{equation*}
Moreover, $(\rho_s, z_s, v^\rho_s, v^z_s)$ verifies 
\begin{equation*}
e^{\lambda(\rho_s, z_s)}((v^\rho_s)^2 + (v^z_s)^2) = J(\rho_s, z_s, \ve, \ell_z).
\end{equation*}
Hence, 
\begin{equation*}
J(\rho_s, z_s, \ve, \ell_z) = 0. 
\end{equation*}
\end{proof}
Now, we make the link between the stationary points of \eqref{geodesic::system} and the critical points of $E_{\ell_z}(W, X, \sigma, \cdot)$. We state the following lemma
\begin{lemma}
\label{link:E:J}
Let $\ell_z\in\mathbb R^\ast$ and let $(\rho_c, z_c)(\ell_z)\footnote{$(\rho_c, z_c)$ depends also on the metric coefficients. We omitted the dependence in order to lighten the expressions.}$ be a critical point of $E_{\ell_z}(W, X, \sigma, \cdot)$. Then, $(\rho_c, z_c, 0, 0)$ is a timelike future-directed stationary solution of the system \eqref{geodesic::system} with parameters $(\ve_c:= E_{\ell_z}(\rho_c, z_c), \ell_z)$. 
\\ Reciprocally, let $(\ve, \ell_z)\in]0, \infty[\times\mathbb R^\ast$ and $(\rho_s, z_s, v_s^\rho, v_s^z)(\ve, \ell_z)$ be a timelike future-directed stationary solution of \eqref{geodesic::system}. Then, $(\rho_s, z_s)(\ve, \ell_z)$ is critical point of $E_{\ell_z}(W, X, \sigma, \cdot)$ and $\ve = E_{\ell_z}(W, X, \sigma, \cdot)$. 
\end{lemma}
\begin{proof}
\begin{itemize}
\item If $(\rho_c, z_c)(\ell_z)$ is a critical point of $E_{\ell_z}(W, X, \sigma, \cdot)$. Then, 
\begin{equation*}
\nabla_{(\rho, z)} E_{\ell_z}(W, X, \ve_c, z_c) = 0. 
\end{equation*}
Set $\ve_c:= E_{\ell_z}(W, X, \sigma, \rho_c, z_c)$. Then, $J(\rho_c, z_c, \ve_c, \ell_z) = 0$. 
\\Now, recall that along timelike future directed solutions (in particular stationary solutions), we have 
\begin{equation*}
J( \rho, z, \ve, \ell_z) = 0 \quad\text{if and only if}\quad \ve = E_{\ell_z}(W, X, \sigma, \rho, z). 
\end{equation*}
Moreover
\begin{equation}
\label{eq::n}
\nabla_{(\rho, z)}E_{\ell_z}(\rho, z) = -\frac{\nabla_{(\rho, z)} J(\rho, z, E_{\ell_z}(\rho, z), \ell_z)}{\frac{\partial J(\rho, z, E_{\ell_z}(\rho, z), \ell_z)}{\partial\ve}}.
\end{equation}
Evaluating the latter at $(\rho_c, z_c)$, we obtain
\begin{equation*}
\nabla_{(\rho, z)}J(\rho_c, z_c, \ve_c, \ell_z) = 0. 
\end{equation*}
Therefore, $(\rho_c, z_c, 0, 0)$ is a stationary solution of \eqref{geodesic::system}. 
\item Now, if $(\rho_s, z_s, v^\rho_s, v^\rho_s)$ is a timelike future-directed stationary solution of \eqref{geodesic::system}, then, by Lemma \ref{general:st:sol}, we have 
\begin{equation*}
J( \rho, z, \ve, \ell_z) = 0 \quad\text{if and only if}\quad \ve = E_{\ell_z}(W, X, \sigma, \rho, z). 
\end{equation*}
The first equation is equivalent to
\begin{equation*}
\ve = E_{\ell_z}(W, X, \sigma, \rho_s, z_s). 
\end{equation*}
Moreover, by \eqref{eq::n}, we obtain
\begin{equation*}
\nabla_{(\rho, z)}E_{\ell_z}(W, X, \sigma, \rho_s, z_s) = 0.  
\end{equation*}
\end{itemize}
Therefore, $(\rho_s, z_s)(\ell_z)$ is a critical point of $E_{\ell_z}(W, X, \sigma, \cdot)$
\end{proof}
\noindent In Section \ref{G::motion}, the above lemmas will be applied to compute stationary solutions for the reduced system in the case of Kerr. 
\\In spherical symmetry, the intersection of $\ve$ and the effective potential $E_{\ell_z}$, depending only on the radial direction, gave us isolated turning points, which allowed us to determine the nature of the orbits. In the axisymmetric case, the effective potential defined by \eqref{eff::potential} is two-dimensional. Therefore, the intersection of an energy level $\ve$ and $\ell_z$ will lead to a curve in $\BB$. In this context, we will define introduce the {\it zero zelocity curve $Z(\ve, \ell_z)$} (ZVC) associated to a timelike future directed geodesic which generalises the set of isolating turning points.  We state the following definition
\begin{definition}
\label{ZVC:St:Axis}
Let $\gamma: I\to\spacetime$ be a timelike future directed geodesic with constants of motion $(\ve, \ell_z)$. We define the zero velocity curve (ZVC) associated to $\gamma$ denoted by $Z(\ve, \ell_z)$  to be the curve in $\BB$ defined by 
\begin{equation*}
\SymbolPrint {Z}(\ve, \ell_z) := \left\{(\rho, z)\in\BB\;:\; J(\rho, z, \ve, \ell_z) = 0 \right\}. 
\end{equation*}
We define the allowed region for $\gamma$ to be the subset $A(\ve, \ell_z)\subset\BB$ defined  by 
\begin{equation*}
\SymbolPrint A(\ve, \ell_z):= \left\{(\rho, z)\in\BB\;:\; J(\rho, z, \ve, \ell_z) \geq 0 \right\}. 
\end{equation*}
\end{definition} 
Note that the above definitions are equivalent\footnote{The equivalence is valid only for timelike future-directed geodesics.} to 
\begin{equation*}
Z(\ve, \ell_z) = \left\{(\rho, z)\in\BB\;:\; E_{\ell_z}(\rho, z) = \ve \right\}. 
\end{equation*}
\begin{equation*}
A(\ve, \ell_z)= \left\{(\rho, z)\in\BB\;:\; E_{\ell_z}(\rho, z) \leq \ve \right\}.
\end{equation*}
Finally, we note that the ZVC associated to a timelike future directed geodesic is the set of its turning points. 

\begin{remark}
By an abuse of notations, we will identify $Z(\ve, \ell_z)$,  a subset in $\BB$ with the corresponding subset $Z(\ve, \ell_z)$  in $(r, \theta)$ coordinates. 
\end{remark}

\subsection{The Kerr spacetime}
\label{The::Kerr}
The Kerr family of spacetimes $(\spacetime^K, g_{a, M}^K)$ is a two-parameter family of stationary, axisymmetric, asymptotically flat Lorentzian manifolds which are solutions to the Einstein-vacuum equations. The Kerr solution is called \textit{sub-extremal} if the parameters $a$ and $M$ verify $0\leq |a|<M$; $M$ denotes the mass and $a$ denotes the specific angular momentum. 
The  domain of outer communication of a sub-extremal Kerr spacetime  can be represented  in Boyer-Lindquist (BL) coordinates by 
$\mathcal O^{K} $ defined by 
\begin{equation}
\label{domain::Kerr}
\begin{aligned}
\mathcal O^{K} := &\mathbb R_t\times]r_+(a, M), \infty[_r\times \mathbb S^2_{(\theta, \phi)}
\end{aligned}
\end{equation} 
where
\begin{equation*}
 r_\pm(a, M) := M \pm \sqrt{M^2 - a^2},
\end{equation*}
and a metric which takes the form
\begin{equation*}
g_{Kerr} = -\left(1-\frac{2Mr}{\Sigma^2} \right)dt^2 - \frac{4aMr\sin^2\theta}{\Sigma^2}dt d\phi + \frac{\Pi}{\Sigma^2}\sin^2\theta d\phi^2+ \frac{\Sigma^2}{\Delta}dr^2+ \Sigma^2d\theta^2, 
\end{equation*}
where
\begin{equation*}
\SymbolPrint{\Delta} = r^2-2Mr+a^2,\quad \SymbolPrint \Sigma^2 = r^2+a^2\cos^2\theta \quad\text{and}\quad \SymbolPrint \Pi = (r^2+a^2)^2-a^2\sin^2\theta\Delta. 
\end{equation*}
\noindent The metric is degenerate in the limit  $r \to r _\pm(a, M)$ at $\displaystyle (r, \theta) = \left(0, \frac{\pi}{2}\right)$. However, $r = r_\pm(a, M)$ is a coordinate singularity. Indeed, we introduce the following change of coordinates $(t^\ast, r, \theta, \phi^\ast)$ in the region  $\mathcal O^K$ defined by: 
\begin{equation*}
\begin{aligned}
t^\ast &= t + \int^r\,(r^2 + a^2)\Delta^{-1} , \\
\phi^\ast &= \phi + \int^r\,a\Delta^{-1}.
\end{aligned}
\end{equation*}
Then, the metric $g_{Kerr}$ takes the following form in the above system of coordinates
\begin{equation*}
\begin{aligned}
g_{Kerr}^\ast &= \Sigma^2d\theta^2 - 2a\sin^2\theta drd\phi^\ast + 2drdt^\ast + \Sigma^{-2}\left((r^2 + a^2)^2 - \Delta a^2\sin^2\theta\right)\sin^2\theta d(\phi^\ast)^2  \\
&- 4a\Sigma^{-2}mr\sin^2\theta d\phi^\ast dt^\ast - \left(1 - 2mr\Sigma^{-2} \right)d(t^\ast)^2. 
\end{aligned}
\end{equation*}
The above expression  $g_{Kerr}^\ast$ is formally regular at $r = r_{\pm}(a, M)$ and it is defined on the region 
\begin{equation*}
\spacetime^{Kerr}  := \mathbb R^2_{(t^\ast, r)}\times\mathbb S^2_{(\theta, \phi^\ast)} \backslash \left\{(t^\ast, r, \theta, \phi^\ast) \;:\; (r, \theta) = \left(0, \frac{\pi}{2} \right) \right\}. 
\end{equation*}
Moreover, we have an isometric embedding 
\begin{equation*}
\begin{aligned}
&\left(\mathcal O^{K}, g_{Kerr}^\ast\right) \to \left(\mathcal M^{Kerr}, g^{\ast}_{Kerr}\right) \\
&(t, r, \theta, \phi)\mapsto(t^\ast, r, \theta, \phi^\ast). 
\end{aligned}
\end{equation*}
Now, we define the following subset of $\mathcal M^{Kerr}$
\begin{equation*}
\mathcal B^{Kerr}_{out}:= \mathbb R_{t^\ast}\times[0, r_+(a, M)]_r\times \mathbb S^2_{(\theta, \phi^\ast)} \backslash \left\{(t^\ast, r, \theta, \phi^\ast) \;:\; (r, \theta) = \left(0, \frac{\pi}{2} \right) \right\}.
\end{equation*}
The (outer) event  horizon  $\Horizon^{Sch}$ is the hypersurface 
\begin{equation*}
    \Horizon^{K}:= \partial\mathcal B^{Kerr}_{out}= \left\{(t^\ast, r, \theta, \phi^\ast)\in\spacetime^{Kerr}\;,\; r = r_{+}(a, M) \right\}. 
\end{equation*}
\begin{remark}
One can also define the (inner) event horizon to be the boundary of the region
\begin{equation*}
\mathcal B^{Kerr}_{in}:= \mathbb R_{t^\ast}\times[0, r_-(a, M)]_r\times \mathbb S^2_{(\theta, \phi^\ast)} \backslash \left\{(t^\ast, r, \theta, \phi^\ast) \;:\; (r, \theta) = \left(0, \frac{\pi}{2} \right) \right\}.
\end{equation*}
However, for the purposes of this work, we shall be interested only in the region $\mathcal O^K$ and its boundary $\Horizon^{K}$. 
\end{remark}
\noindent It is easy to see that the vector fields 

\begin{equation*}
T := \frac{\partial}{\partial t} \quad\text{and}\quad \Phi :=  \frac{\partial}{\partial \phi}
\end{equation*}
are Killing vector fields. 
\\
\\An important feature of the Kerr spacetime  is that the norm of the Killing vector field $T$ is not timelike everywhere in the exterior region: when $a\neq 0$, there exists a non-empty set of points in $\mathcal O^K$, called the \textit{ergosphere} such that the norm of $T$ vanishes:
\begin{equation*}
g(T, T) = -\g{t}{t} = 1 - \frac{2Mr}{\Sigma^2(r, \theta)} = 0. 
\end{equation*}
It is defined by 
\begin{equation}
\label{ergo}
\SymbolPrint{S^K} := \left\{ (r, \theta)\in(r_+(a, M), \infty)\times(0, \pi)\:\; r =  M + \sqrt{M^2 - a^2\cos^2\theta}\right\}.
\end{equation}
Therefore, $T$ becomes spacelike in the so-called \textit{ergoregion}, the region in $\mathcal O^K$ bounded by  $\SymbolPrint{S^K}$ and $\Horizon^K$ 
\begin{equation*}
\mathscr E= \left\{ (r, \theta)\in(r_+(a, M), \infty)\times(0, \pi)\:\;;\;  r <  M + \sqrt{M^2 - a^2\cos^2\theta}\right\}.
\end{equation*}
Test particles which are located in the ergoregion may extract energy from the black hole. This phenomenon is often called the Penrose process \cite{wald2010general}, \cite{chandrasekhar1998mathematical}. When $a=0$, the ergosphere coincides with the event horizon and there is no ergoregion. Thus, the extraction of energy does not occur in the Schwarzschild spacetime.

\subsection{Properties of the Kerr spacetime}
\label{Kerr:in:Weyl}
In this section, we recall the main properties of sub-extremal Kerr exteriors and we express the metic, the exterior region, the horizon and the axis of  symmetry in terms of Weyl coordinates. 
\\ We recall from Section \ref{The::Kerr}  that the exterior region of a sub-extremal Kerr spacetime with parameters  $(a, M)$  is represented in BL coordinates by $\mathcal O$: 

\begin{equation}
\label{Kerr:ext}
\SymbolPrint{\mathcal O} := \left\{ \mathbb R\times (r_+(a, M), \infty)\times(0, \pi)\times(0,2\pi)\right\},
\end{equation} 
where $r_+(a, M)$ is defined by 
\begin{equation}
r_+(a, M) := M + \sqrt{M^2 - a^2},
\end{equation}
and a metric which takes the form 
\begin{equation}
\label{Kerr:BL}
g_{Kerr} = -\left(1-\frac{2Mr}{\Sigma^2} \right)dt^2 - \frac{4aMr\sin^2\theta}{\Sigma^2}dt d\phi + \frac{\Pi}{\Sigma^2}\sin^2\theta d\phi^2+ \frac{\Sigma^2}{\Delta}dr^2+ \Sigma^2d\theta^2, 
\end{equation}
where
\begin{equation*}
\SymbolPrint{\Delta} = r^2-2Mr+a^2,\quad \SymbolPrint \Sigma^2 = r^2+a^2\cos^2\theta \quad\text{and}\quad \SymbolPrint \Pi = (r^2+a^2)^2-a^2\sin^2\theta\Delta. 
\end{equation*}
Its inverse is given by 
\begin{equation*}
g_{Kerr}^{-1} = -\frac{\Pi}{\Delta\Sigma^2}\frac{\partial}{\partial t}\otimes\frac{\partial}{\partial t} - \frac{4aMr}{\Delta\Sigma^2}\frac{\partial}{\partial t}\otimes\frac{\partial}{\partial\phi} + \frac{\Delta-a^2\sin^2\theta}{\Delta\Sigma^2\sin^2\theta}\frac{\partial}{\partial \phi}\otimes \frac{\partial}{\partial \phi} + \frac{\Delta}{\Sigma^2}\frac{\partial}{\partial r}\otimes \frac{\partial}{\partial r} + \frac{1}{\Sigma^2}\frac{\partial}{\partial \theta}\otimes\frac{\partial}{\partial \theta}. 
\end{equation*}
\noindent It is clear to see that 
\begin{equation*}
T := \frac{\partial }{\partial t} 
\end{equation*}
that generates stationarity and 
\begin{equation*}
\Phi := \frac{\partial }{\partial \phi} 
\end{equation*}
that generates axial symmetry are Killing. 

\begin{remark}
\label{fix:kappa}
From now on, we  fix $\beta$ as introduced in Section \ref{bbarre},  to be $\SymbolPrint\beta = \sqrt{M^2 - a^2}$. We henceforth omit the dependance of the different regions on $\beta$.
\end{remark}
\noindent The event horizon is the hypersurface 
\begin{equation*}
\Horizon =  \left\{ (t, r, \theta, \phi)\;:\; r = r_+(a, M))\right\}. 
\end{equation*}
The axis of symmetry is the set of points 
\begin{equation*}
\Axis =  \left\{ (t, r, \theta, \phi)\;:\; \theta\in\left\{0, \pi\right\}\right\}. 
\end{equation*}
\noindent  The cylindrical coordinates introduced in Section \ref{Stat:axis:coord} are adapted for the stationary and axisymmetric Einstein-Vlasov system. Therefore, we put the Kerr metric \eqref{Kerr:BL} and the ones we will construct in the form considered in \eqref{metric:ansatz}. To this end, we introduce the following change of coordinates: the set of points in $\mathcal O$ with a fixed $(t, \phi)\in\mathbb R\times(0, \pi)$ defines a $2-$surface on which we introduce the following functions:
\begin{equation}
\label{change:Weyl}
\begin{aligned}
\rho(r,\theta) &:=  \sqrt{\Delta}\sin\theta, \\
z(r, \theta)&:= (r-M)\cos\theta. 
\end{aligned}
\end{equation} 
Such $(\rho, z)$ provides a coordinate system on any surface with constant $(t, \phi)$ with range $\BB$. Moreover, we have 
\begin{lemma}
The mapping 
\begin{equation*}
\begin{aligned}
]r_+(a, M), \infty[\times]0, \pi[ &\to \BB \\
 \quad (\quad r\quad, \quad\theta\quad)&\mapsto (\rho(r, \theta), z(r, \theta))
\end{aligned}
\end{equation*}
is a $C^{\infty}-$diffeomorphism. Its inverse is given by 
\begin{equation*}
\begin{aligned}
 \BB&\to ]r_+(a, M), \infty[\times]0, \pi[ \\
 (\rho, z)&\mapsto (r(\rho, z), \theta(\rho, z))
\end{aligned}
\end{equation*}
where 
\begin{align}
\label{r:rho:z}
r(\rho, z) &= M+\frac{1}{\sqrt{2}}\sqrt{(\beta^2+\rho^2+z^2)+\sqrt{(\rho^4+2\rho^2(z^2+\beta^2)+(z^2-\beta^2)^2)}}, \\
\sin\theta(\rho, z) &= \frac{\rho }{\sqrt{\Delta(\rho, z)}}, 
\end{align}
with
\begin{equation}
\label{beta:def}
\beta = \sqrt{M^2 - a^2}. 
\end{equation}
\end{lemma}

\begin{proof}
It is clear  that the mapping is well-defined and smooth. We show that it is bijective from $]r_+(a, M), \infty[\times]0, \pi[ $ to $\BB$: Let $(\rho_0, z_0)\in \BB$. We claim that there exists a unique $(r_0, \theta_0)$ such that: 
\begin{equation}
\label{eq::inverse:}
\left(\rho(r_0, \theta_0), z(r_0, \theta_0)\right) = (\rho_0, z_0). 
\end{equation}
By the latter, we have  
\begin{equation*}
    \frac{\rho_0^2}{\Delta(r_0)}+\frac{z_0^2}{(r_0-M)^2} = 1,
\end{equation*}
which is equivalent to 
\begin{equation*}
    \rho_0^2(r_0-M)^2+(r_0-M)^2+\Delta(r_0) z_0^2 -\beta^2z_0^2 = (r_0-M)^4 -\beta^2(r_0-M)^2.
\end{equation*}
By setting $x:=(r_0-M)^2$, $x$ satisfies the following quadratic equation
\begin{equation*}
    x^2-(\rho_0^2+z_0^2+\beta^2)x+\beta^2z_0^2 = 0. 
\end{equation*}
Therefore, 
\begin{equation*}
x = \frac{(\rho_0^2+z_0^2+\beta^2)\pm\sqrt{(\rho_0^4+2\rho_0^2(z_0^2+\beta^2)+(z_0^2-\beta^2)^2)}}{2}. 
\end{equation*} 
Since $x$ is positive, we obtain
\begin{equation*}
r_0(\rho_0, z_0) = M+\frac{1}{\sqrt{2}}\sqrt{(\beta^2+\rho_0^2+z_0^2)+\sqrt{(\rho_0^4+2\rho_0^2(z_0^2+\beta^2)+(z_0^2-\beta^2)^2)}} .
\end{equation*}
Moreover, $\theta_0$ satisfies: 
\begin{equation*}
\sin\theta_0(\rho_0, z_0) = \frac{\rho_0}{\sqrt{\Delta(r_0)}}. 
\end{equation*}
Hence, $(r_0, \theta_0)$ exists and it is unique. Finally, it is straightforward that the inverse is also smooth on $\BB$. This ends the proof.
\end{proof}
We compute the Jacobian of the above change of coordinates: 
\begin{equation*}
    J_{iso} = \left(
  \begin{array}{ c  c}
     \frac{r-M}{\sqrt{\Delta}}\sin\theta & \sqrt{\Delta}\cos\theta \\
     \cos\theta & -(r-M)\sin\theta
  \end{array} \right).
\end{equation*}
\\ Now, we compute $dr$ and $d\theta$ in terms of $d\rho$ and $dz$:
\begin{equation*}
    \left(
  \begin{array}{ c}
     d{r}  \\
     d\theta
  \end{array} \right)= J_{iso}^{-1}\left(
  \begin{array}{ c  }
     d\rho  \\
     dz 
  \end{array} \right).
\end{equation*}
Here $J_{iso}^{-1}$ is given by 
\begin{equation*}
    J_{iso}^{-1} = \frac{\sqrt{\Delta}^{-1}}{\left(\cos^2\theta+\frac{(r-M)^2}{\Delta}\sin^2\theta \right)}\left(
  \begin{array}{ c  c}
     (r-M)\sin\theta & \sqrt{\Delta}\cos\theta \\
     \cos\theta & -\frac{r-M}{\sqrt{\Delta}}\sin\theta
  \end{array} \right).
\end{equation*}
\\Then, we compute $dr^2$ and $d\theta^2$:
\begin{equation*}
    dr^2 = \left(\frac{\sqrt{\Delta}^{-1}}{\left(\cos^2\theta+\frac{(r-M)^2}{\Delta}\sin^2\theta \right)}\right)^2\left((r-M)^2\sin^2\theta d\rho^2 + 2\sqrt{\Delta}(r-M)\sin\theta\cos\theta d\rho dz + \Delta dz^2 \right),
\end{equation*}
and 
\begin{equation*}
    dr^2 = \left(\frac{\sqrt{\Delta}^{-1}}{\left(\cos^2\theta+\frac{(r-M)^2}{\Delta}\sin^2\theta \right)}\right)^2\left( d\rho^2 - 2\frac{(r-M)\sin\theta\cos\theta}{\sqrt{\Delta}} d\rho dz + \frac{(r-M)^2\sin^2\theta}{\Delta} dz^2 \right).
\end{equation*}
Hence
\begin{equation*}
    \Sigma^2\left(\frac{1}{\Delta}dr^2+d\theta^2\right) = \frac{\Sigma^2}{\Delta}\left(\cos^2\theta+\frac{(r-M)^2}{\Delta}\sin^2\theta \right)^{-1}(d\rho^2+dz^2). 
\end{equation*}
Now we set 
\begin{equation*}
    e^{2\lambda_K} := \Sigma^2\left(\frac{1}{\Delta}dr^2+d\theta^2\right) = \frac{\Sigma^2}{\Delta}\left(\cos^2\theta+\frac{(r-M)^2}{\Delta}\sin^2\theta \right)^{-1},
\end{equation*}
Therefore, the Kerr metric, written in the $(t, \phi, \rho, z)$ coordinates takes the form: 
\begin{equation*}
        g_{a,M} = -V_Kdt^2 + 2W_Kdtd\phi + X_Kd\phi^2 + e^{2\lambda_K}\left( d\rho^2 + dz^2\right)
\end{equation*}
where
\begin{equation*}
V_K = (1 - \frac{2Mr}{\Sigma^2}), \; W_K = - \frac{2Mar\sin^2\theta}{\Sigma^2}, \; X_K = \sin^2\theta\frac{\Pi}{\Sigma^2} 
\end{equation*}
and
\begin{equation*}
 e^{2\lambda_K} = \Sigma^2\Delta^{-1}\left( \frac{(r - M)^2}{\Delta}\sin^2\theta + \cos^2\theta \right)^{-1}.
\end{equation*}
The event horizon as well as the axis of symmetry are henceforth given by 
\begin{equation*}
\Horizon  = \left\{(\rho, z)\\:\; \rho = 0\;,\; |z|<\beta \right\}
\end{equation*}
and
\begin{equation*}
\Axis = \left\{(\rho, z)\\:\; \rho = 0\;,\; |z|>\beta \right\}. 
\end{equation*}
The intersection of the horizon with the axis of symmetry is given by  the points $\displaystyle p_N:= (0, \beta)$ and $\displaystyle p_S:= (0, -\beta)$. 

\begin{remark}
Note that the induced metric on each $(t, \phi) = cste$ surfaces is conformally equivalent to the Euclidean metric with conformal factor given by $e^{2\lambda_K}$. 
 \end{remark}
\begin{remark}
The coordinates $(\rho, z)$ fail to be regular a the points $p_N$ and $p_S$. 
\end{remark}

\begin{remark}
The subscript $^K$ will always refer to the Kerr metric. For instance, any ZVC associated to a Kerr geodesic is denoted by $Z^K$ instead of just $Z$. 
\end{remark}
\noindent Finally, we state the following result 
\begin{Propo}[Extendibility of Kerr exterior - \cite{chodosh2015stationary}]
\label{extendibility}
The Kerr exterior is extendable to a regular black hole spacetime in the sense of Definition \ref{extendable:1}. 
\end{Propo}

\section{Timelike future directed geodesics in Kerr spacetime}
\label{Kerr:geo:aux}
\subsection{Study of the geodesic motion in BL coordinates}
In this work, we are interested in future directed particles moving in the exterior region of a Kerr spacetime. The aim of this section is to classify their orbits based on their constants of motion. 
\\  An important property of the Kerr spacetime, as shown in the work of Carter  \cite{carter2009republication}, is the fact that the geodesic equations form an integrable Hamiltonian system so that one has a complete set of explicit integrals of motion. Therefore, one can determine the nature of timelike orbits based on the possible values of these integrals of motion.
\\ In this section, we first present the geodesic equations and the four constants of motion. Then, we determine  for which constants of motion the orbit is circular and confined in the equatorial plane. Similarly, we determine for which values the orbits have a constant radius $r$. These special classes of orbits are the key for the general classification which is obtained towards the end of this section.
\\ The classification is based on the following definition
\begin{definition}
\label{classify:categorie}
Let $\gamma: I\ni 0\to\mathcal O$ be a timelike future-directed geodesic\footnote{Here, $(\gamma, I)$ is a maximal solution  in $\mathcal O$ of the geodesic equation. } parametrised by its proper time such that $\gamma(\tau) = \left(t(\tau), \phi(\tau), r(\tau), \theta(\tau)\right)$. $\gamma$ is said to be 
\begin{enumerate}
\item {\it spherical} if $I = \mathbb R$ and there exists $r_s\in]r_+(a, M), \infty[$ such that $\forall\tau\in I\;:\; r(\tau) = r_s$.  
\item {\it circular} if $\gamma$ is spherical and $\forall \tau\in I\;:\; \theta(\tau) = \frac{\pi}{2}$.
\item {\it scattered at infinity} if $I = \mathbb R$ and there exists a set of the form\footnote{By a small abuse of notation, we will make the confusion between $\mathcal O$ and the set $]r_+(a, M), \infty[\times(0, \pi)$ in this section.} $[r_{sc}, \infty[\times [\theta_{sc}, \pi - \theta_{sc}] \subset\mathcal O$,  with $\theta_{sc}\in]0, \pi[$, such that $\displaystyle\forall \tau\in I\;:\;\; \left(r(\tau), \theta(\tau)\right)\in [r_{sc}, \infty[\times [\theta_{sc}, \pi - \theta_{sc}]$. 
\item {\it trapped non-spherical} if $I = \mathbb R$ and there exit a compact set $K\subset\subset \mathcal O$ such that $\displaystyle\forall \tau\in I\;:\;\; \gamma(\tau)\in K$. 
\item {\it plunging} if $I = ]a, b[$, where $-\infty<a<0<b<+\infty$. In this case, $\gamma$ reaches the horizon in a finite proper time.
\item {\it plunging from infinity} if $I= ]-\infty, a[$ where $0<a<+\infty$ such that $\gamma$ reaches the horizon in a finite proper time, given by $a$. 
\item {\it emanating  from the white hole to infinity} if $I= ]a, +\infty[$ where $0<a<+\infty$ and $r(\tau)\to\infty$ when $\tau\to+\infty$. 
\end{enumerate}
\end{definition}
 \begin{definition}
 An orbit $(\gamma, I)$ is said to be \textit{confined in the equatorial plane} if $\displaystyle \forall\tau\in I\,,\, \theta(\tau) =\frac{\pi}{2}$.
 \end{definition}
 \begin{remark}
 We will refer to spherical, circular, scattered and trapped orbits as classical since they possess  Newtonian analogs. 
 \end{remark}
 \subsubsection{Geodesic equations in the BL coordinates}
\label{G::motion}
Consider a sub-extremal Kerr exterior with parameters $(a, M)$. In  BL coordinates, the metric is given by \eqref{Kerr:BL} and the metric components are defined on the domain \eqref{Kerr:ext}. 
Let $\gamma: I \to \mathcal O$ be a timelike future directed geodesic parametrised by its proper time $\tau$ and let $v = \displaystyle \frac{d \gamma}{d\tau}$ be its four-velocity vector. We recall that $(\ve, \ell_z)$ defined respectively by \eqref{E:energy} and \eqref{L:momentum} are conserved along the geodesic flow. 
For particles moving in the equatorial plane, these quantities  together with the conservation of the Hamiltonian are sufficient to classify their trajectories. In the general case, the geodesic motion in Kerr forms an integrable system thanks to the existence of a fourth integral of motion  $\SymbolPrint q$, called the Carter constant \cite{carter2009republication} given by:
\begin{equation}
\label{Q:Carter}
\forall (x, v)\in T\spacetime \; :\; q(x, v) := v_\theta^2 + \cos\theta^2\left( a^2(1 - \ve^2) + \frac{\ell_z^2}{\sin^2\theta} \right). 
\end{equation}
\noindent In BL coordinates, $\gamma(\tau) = (t(\tau), \phi(\tau), r(\tau), \theta(\tau))$ and $\dot\gamma(\tau) = (\dot t(\tau), \dot\phi(\tau), \dot r(\tau), \dot\theta(\tau))$. By \eqref{L:conservation}, we have 
\begin{equation*}
- 1 = -\left(1 - \frac{2Mr}{\Sigma^2} \right)\dot t^2 - \frac{4aMr\sin^2\theta}{\Sigma^2}\dot t\dot\phi + \frac{\Pi\sin^2\theta}{\Sigma^2}\dot\phi^2 + \frac{\Sigma^2}{\Delta}\dot r^2 + \Sigma^2\dot\theta^2. 
\end{equation*}
Combining the latter with  \eqref{E:energy}, \eqref{L:momentum} and \eqref{Q:Carter}, one can separate the motion in the $(t, \dot t)$, $(\phi, \dot\phi)$, $(r, \dot r)$ and $(\theta, \dot\theta)$ in the following way:
\begin{align}
\label{r:plane}
(r, \dot r)&\,:\;\Sigma^4{\dot r}^2 =  \left( (r^2+a^2)\ve - a\ell_z\right)^2- \Delta(r^2+ (\ell_z-a\ve)^2 + q), \\
\label{theta:plane}
(\theta, \dot\theta)&\,:\; \Sigma^4{\dot\theta}^2 = q - \cos^2\theta\left(a^2(1 - \ve^2) + \frac{\ell_z^2}{\sin^2\theta} \right) \\
\label{phi:plane}
(\phi, \dot \phi)&\,:\; \Sigma^2{\dot\phi} = -\left(a\ve - \frac{\ell_z}{\sin^2\theta}\right) + \frac{a\left( \ve(r^2 + a^2) - \ell_za\right)}{\Delta} , \\
\label{t:plane}
 (t, \dot t)&\,:\; \Sigma^2\dot t = -a\left(a\ve\sin^2\theta - \ell_z\right) + \left( r^2 + a^2\right)\frac{\ve(r^2 + a^2) - \ell_za}{\Delta}
\end{align}
\noindent Note that the motions in $r$ and in $\theta$ are still coupled. In order to separate the $r-$motion and the $\theta-$motion, one can use a new time parameter $\lambda$, called Mino time \cite{mino2003perturbative}, defined in the following way: 
\begin{equation*}
\begin{aligned}
\lambda &: I\to\mathbb ]0, \infty[ \\
&\tau\mapsto \lambda(\tau):= \int_{\tau_0}^\tau\frac{1}{\Sigma^2(r(s), \theta(s))}\,ds
\end{aligned}
\end{equation*}
where $\tau_0\in I$. Since $\Sigma^2$ is positive, $\lambda$ is well-defined and we have: 
\begin{equation}
\label{mino:mino}
d\tau = \Sigma^2 d\lambda. 
\end{equation}
Therefore, 
\begin{equation*}
\frac{d}{d\lambda} = \Sigma^2 \frac{d}{d\tau}. 
\end{equation*}
In terms of $\lambda$, \eqref{r:plane} and \eqref{theta:plane} become: 
\begin{align}
\label{r:plane:mino}
\left(\frac{dr}{d\lambda}\right)^2 &=  \left( (r^2+a^2)\ve - a\ell_z\right)^2- \Delta(r^2+ (\ell_z-a\ve)^2 + q), \\
\label{theta:plane:mino}
\left(\frac{d\cos\theta}{d\lambda}\right)^2 &= q - (q+ a^2(1 - \ve^2) + \ell_z^2)(\cos\theta)^2 + a^2(1 - \ve^2)(\cos\theta)^4. 
\end{align}
Hence $r(\lambda)$ and $\theta(\lambda)$ become independent from each other and we can solve the equations of motion by solving  \eqref{r:plane:mino} to determine  $r(\lambda)$ and solves \eqref{theta:plane:mino} to determine $\theta(\lambda)$, then, by plugging the latter in \eqref{mino:mino}, we compute $\tau$. It remains to integrate \eqref{t:plane} and \eqref{phi:plane} in order to derive  $\phi$ and $t$. We will comeback later to the resolution of these equations in order to classify the solutions (see Section \ref{classif:BL:geodesics}). 
\\Now, we introduce the fourth order polynomials
\begin{equation}
\label{T::4}
\SymbolPrint T(Y, \ve, \ell_z, q, a):= q - (q+ a^2(1 - \ve^2) + \ell_z^2)Y^2 + a^2(1 - \ve^2)Y^4
\end{equation}
and 
\begin{equation}
\label{R::4}
\SymbolPrint R(X, \ve, \ell_z, q, a):= (\ve(X^2+a^2)-a\ell_z)^2 - (X^2-2XM + a^2)(X^2+ (a\ve - \ell_z)^2 + q). 
\end{equation}
Hence, 
\begin{align}
\label{r::plane}
(r, \dot r)&\,:\;\Sigma^4{\dot r}^2 = R(r, \ve, \ell_z, q, a), \\
\label{theta::plane}
(\theta, \dot\theta)&\,:\; \Sigma^4\sin^2\theta{\dot\theta}^2 = T(\cos\theta, \ve, \ell_z, q, a)
\end{align}
We also introduce the following dimensionless quantities:

\begin{equation}
\tilde r \index[Symbols]{r @$\tilde r$} := \frac{r}{M}  \;,\; \SymbolPrint{d} := \frac{a}{M}\;,\;\tilde\ell_z \index[Symbols]{\ell_z @$\tilde \ell_z$}:= \frac{\ell_z}{M}\;,\; \tilde{q} \index[Symbols]{q @$\tilde q$}:= \frac{q}{M^2}
\end{equation}
and 
\begin{equation}
\label{r:H:d}
\SymbolPrint r_H (d) := 1+\sqrt{1-d^2} = \frac{r_+(a, M)}{M}. 
\end{equation}
Therefore,
 \begin{equation}
\label{T::4}
\frac{T}{M^2}\left(Y, \ve, \tilde \ell_z, \tilde q, d\right)= \tilde q - (\tilde q+ d^2(1 - \varepsilon^2) + \tilde \ell_z^2)Y^2 + d^2(1 - \varepsilon^2)Y^4
\end{equation}
and 
\begin{equation}
\label{R::4}
\frac{R}{M^4}\left(X, \ve, \tilde \ell_z, \tilde q, d\right) = \left(\ve\left( \left(\frac{X}{M}\right)^2+d^2\right)-d\tilde \ell_z\right)^2 - \left(\left(\frac{X}{M}\right)^2 - 2\frac{X}{M} + d^2\right)\left(\left(\frac{X}{M}\right)^2+ (d\varepsilon - \tilde \ell_z)^2 + \tilde q\right). 
\end{equation}
\begin{remark}
We shall henceforth take $\displaystyle M = 1$ and identify the above different quantities and their normalisations.
\end{remark}
\begin{remark}
From now on, the dependence on $(a, M)$ and thus on $d$ will not be written in order to lighten the equations. 
\end{remark}
\noindent The equations \eqref{r:plane} and \eqref{theta:plane} do not form a regular system of ODEs. Consequently, we derive in the following the equivalent Hamiltonian form of the geodesic equations of motion. The latter will form a smooth system of ODEs, even at turning points (roots of $R$ and $T$). This will allow us to compute stationary solutions of the geodesic system, which will be used later in this work. 
First of all, we introduce the Hamiltonian of a free-falling timelike particle of mass $1$,  defined by: 
\begin{equation*}
H(x^\alpha, v_\alpha) := \frac{1}{2}\ginv{\alpha}{\beta}v_\alpha v_\beta = -\frac{1}{2}. 
\end{equation*}
\noindent In BL coordinates, $v_\alpha$ are given by: 
\begin{equation*}
\begin{aligned}
v_t &= -\left(1 - \frac{2r}{\Sigma^2} \right)v^t - \frac{2dr\sin^2\theta}{\Sigma^2}v^\phi,  \\ 
v_{\phi}&= \sin^2\theta\left(r^2 + d^2 + \frac{2d^2r\sin^2\theta}{\Sigma^2} \right)v^\phi - \frac{2dr\sin^2\theta}{\Sigma^2}v^t, \\
v_r &= \frac{\Sigma^2}{\Delta}v^r, \\
v_{\theta}&= \Sigma^2v^\theta.  
\end{aligned}
\end{equation*}
Since $\displaystyle v = \frac{d\gamma}{d\tau}$, we obtain 
\begin{equation*}
\begin{aligned}
v_t &= -\left(1 - \frac{2r}{\Sigma^2} \right)\dot t - \frac{2dr\sin^2\theta}{\Sigma^2} \dot \phi,  \\ 
v_{\phi}&= \sin^2\theta\left(r^2 + d^2 + \frac{2d^2r\sin^2\theta}{\Sigma^2} \right)\dot\phi - \frac{2dr\sin^2\theta}{\Sigma^2}\dot t, \\
v_r &= \frac{\Sigma^2}{\Delta}\dot r, \\
v_{\theta}&= \Sigma^2\dot \theta.  
\end{aligned}
\end{equation*}
$v_\alpha$ are also related to the constants of motion by 
\begin{equation*}
\ve = -v_t \quad,\quad \ell_z = v_\phi \quad\text{and}\quad   q =  v_\theta^2 + \cos\theta^2\left( d^2(1 - \ve^2) + \frac{\ell_z^2}{\sin^2\theta} \right). 
\end{equation*}
Now, let $(x^\alpha, v^\alpha)$ be a solution to the equations of motion. 
\\ We have
\begin{equation*}
2H(x, v) + 1 = \left( \ginv{t}{t}\ve^2 - 2\ginv{t}{\phi}\ve\ell_z + \ginv{\phi}{\phi}\ell_z^2\right) + 1  + \left(  \frac{\Delta}{\Sigma^2}v_r^2 + \frac{1}{\Sigma^2}v_\theta^2\right).
\end{equation*}
Set 
\begin{equation}
\begin{aligned}
\label{J:::tilde}
-\tilde J^K(r, \theta, \ve, \ell_z)&:= \ginv{t}{t}\ve^2 + 2\ginv{t}{\phi}\ve\ell_z + \ginv{\phi}{\phi}\ell_z^2 + 1, \\
& = 1 -  \frac{X_K(r, \theta)}{\Delta\sin^2\theta}\varepsilon^2 - \frac{2W_K(r, \theta)}{\Delta\sin^2\theta}\varepsilon\ell_z + \frac{V_K(r, \theta)}{\Delta\sin^2\theta}\ell_z^2.
\end{aligned}
\end{equation}
Recall that solutions to the equations of motion also satisfy \eqref{r::plane} and \eqref{theta::plane}.
Therefore, we obtain by multiplying the latter equations by $\Sigma^{-2}$: 
\begin{equation*}
\begin{aligned}
&\frac{\Delta}{\Sigma^2}v_r^2 - \frac{1}{\Sigma^2}\frac{R(r, \ve, \ell_z, q)}{\Delta} = 0, \\
&\frac{1}{\Sigma^2}v_\theta^2 - \frac{1}{\Sigma^2\sin^2\theta}T(\cos\theta, \ve, \ell_z, q) = 0. 
\end{aligned}
\end{equation*}
Hence, 
\begin{equation*}
\frac{\Delta}{\Sigma^2}v_r^2 + \frac{1}{\Sigma^2}v_\theta^2 - \frac{1}{\Sigma^2}\left( \frac{R(r, \ve, \ell_z, q)}{\Delta} + \frac{T(\cos\theta, \ve, \ell_z, q)}{\sin^2\theta}\right) - 1 = -1.  
\end{equation*}
By the conservation of the Hamiltonian, we have 
\begin{equation*}
- 1 = 2H(x, v) = -\tilde J^K(r, \theta, \ve, \ell_z) +  \left(\frac{\Delta}{\Sigma^2}v_r^2 + \frac{1}{\Sigma^2}v_\theta^2\right) - 1. 
\end{equation*}
Hence, 
\begin{equation}
\label{J::sepa}
-\tilde J^K(r, \theta, \ve, \ell_z) = -\frac{1}{\Sigma^2}\left( \frac{R(r, \ve, \ell_z, q)}{\Delta} + \frac{T(\cos\theta, \ve, \ell_z, q)}{\sin^2\theta}\right). 
\end{equation}
and 
\begin{equation}
\label{eq:hamilton}
H(x, v) = \frac{1}{2\Sigma^2}\left( \Delta v_r^2 + v_\theta^2 -\left( \frac{R(r, \ve, \ell_z, q)}{\Delta} + \frac{T(\cos\theta, \ve, \ell_z, q)}{\sin^2\theta}\right)\right) - \frac{1}{2}
\end{equation}
Now, we evaluate the equations of motion 
\begin{equation*}
\begin{aligned}
\frac{dx^\alpha}{d\tau} &= \frac{\partial H}{\partial v_\alpha}, \\
\frac{dv_\alpha}{d\tau} &= -\frac{\partial H}{\partial x^\alpha},
\end{aligned}
\end{equation*}
for the Hamiltonian \eqref{eq:hamilton}. We compute 
\begin{align*}
&\frac{dv_r}{d\tau} = -\frac{\partial H(x, v)}{\partial r}  \\
&= -\frac{1}{2\Sigma^2}\left(\Delta'(r)v_r^2 - \frac{\partial }{\partial r}\left( \frac{R(r, \ve, \ell_z, q)}{\Delta} + \frac{T(\cos\theta, \ve, \ell_z, q)}{\sin^2\theta}\right) \right)  \\
&+ \frac{\partial_r\Sigma^2}{2\Sigma^4}\left( \Delta v_r^2 + v_\theta^2 -\left( \frac{R(r, \ve, \ell_z, q)}{\Delta} + \frac{T(\cos\theta, \ve, \ell_z, q)}{\sin^2\theta}\right)\right) , \\
&= \frac{1}{2\Sigma^2}\left(-\Delta'(r)v_r^2 + \frac{-\Delta'(r)R(r, \ve, \ell_z, q) + \Delta(r)\partial_r R(r, \ve, \ell_z, q)}{\Delta(r)^2} + \partial_r\Sigma^2\left(2H(x, v) + 1\right)\right). 
\end{align*}

\begin{align*}
&\frac{dv_\theta}{d\tau} = -\frac{\partial H(x, v)}{\partial \theta} = \frac{1}{2\Sigma^2}\left( \partial_{\theta}\left( \frac{T(\cos\theta, \ve, \ell_z, q)}{\sin^2\theta}\right) + \partial_\theta\Sigma^2\left(2H(x, v) + 1\right) \right).  
\end{align*}
Here, we used the independence of $R$ and $T$ on $\theta$ and $r$ respectively.  Finally,  Hamilton equations are written under the form: 
\begin{empheq}[left={\empheqlbrace}]{alignat=2}
\label{r:evol}
&\frac{dr}{d\tau} &&= \frac{\Delta}{\Sigma^2}v_r ,\\
\label{pr:evol}
&\frac{dv_r}{d\tau}&&= \frac{1}{2\Sigma^2}\left(-\Delta'(r)v_r^2 + \frac{-\Delta'(r)R(r, \ve, \ell_z, q) + \Delta(r)\partial_r R(r, \ve, \ell_z, q)}{\Delta(r)^2}  \right. \\
& &&\left. + \partial_r\Sigma^2\left(2H(x, v) + 1\right)\right), \\
\label{theta:evol}
&\frac{d\theta}{d\tau} &&= \frac{1}{\Sigma^2}v_\theta, \\
\label{ptheta:evol}
&\frac{dv_\theta}{d\tau}&&= \frac{1}{2\Sigma^2}\left( \partial_{\theta}\left( \frac{T(\cos\theta, \ve, \ell_z, q)}{\sin^2\theta}\right) + \partial_\theta\Sigma^2\left(2H(x, v) + 1\right) \right),  \\
\label{phi:evol}
&\frac{d\phi}{d\tau}&&= - \frac{1}{2\Sigma^2}\frac{\partial}{\partial v_\phi}\left( \frac{R(r, \ve, \ell_z, q)}{\Delta} + \frac{T(\cos\theta, \ve, \ell_z, q)}{\sin^2\theta}\right)\\
&\frac{dv_\phi}{d\tau}&& = 0, \\
\label{t:evol}
&\frac{dt}{d\tau}&&= \frac{1}{2\Sigma^2}\frac{\partial}{\partial v_t}\left( \frac{R(r, \ve, \ell_z, q)}{\Delta} + \frac{T(\cos\theta, \ve, \ell_z, q)}{\sin^2\theta}\right) \\
&\frac{dv_t}{d\tau}&& = 0. 
\end{empheq} 
\noindent The terms $2H(x, v) + 1$ above all vanish along any timelike orbit so that \eqref{r:evol}-\eqref{ptheta:evol} become
\begin{equation}
\label{reduced:reduced}
\left\{
\begin{aligned}
\frac{dr}{d\tau} &= \frac{\Delta}{\Sigma^2}v_r ,\\
\frac{dv_r}{d\tau}&= \frac{1}{2\Sigma^2}\left(-\Delta'(r)v_r^2 + \frac{\Delta'(r)R(r, \ve, \ell_z, q) - \Delta(r)\partial_r R(r, \ve, \ell_z, q)}{\Delta(r)^2}\right), \\
\frac{d\theta}{d\tau} &= \frac{1}{\Sigma^2}v_\theta, \\
\frac{dv_\theta}{d\tau}&= \frac{1}{2\Sigma^2}\partial_{\theta}\left( \frac{T(\cos\theta, \ve, \ell_z, q)}{\sin^2\theta}\right). 
\end{aligned}
\right. 
\end{equation}
Solutions to this system such that the conserved hamiltonian verifies 
\begin{equation*}
H(x, v)= -\frac{1}{2}
\end{equation*}
will be called \textit{future-directed timelike geodesics}. 
Note that any term on the right hand side which contains the Carter constant in $R$, $T$ or their derivatives cancel out so that the equations are independent of $q$. 
We recall that  we can separate the motion in the $t-$direction, $\phi-$direction and in the $(r, \theta)-$plane from each other. In order to solve the above system with a given initial conditions $(\gamma(0), \dot\gamma(0))$, we solve the Cauchy problem for its projection in the $(r, \theta, v_r, v_\theta)$ with initial conditions $(r(0), \theta(0), v_r(0), v_\theta(0))$ and with parameters $(\ve, \ell_z, q)$, which are computed using the initial conditions. Therefore, we obtain $(r(\tau), \theta(\tau), v_r(\tau), v_\theta(\tau))$. Then, we plug the latter solutions into the remaining equations  and we integrate \eqref{t:evol} and \eqref{phi:evol} in order to obtain $t(\tau)$ and $\phi(\tau)$. More precisely, we state the following lemma
\begin{lemma}
\label{Lemma:::12:bis}
Let $\gamma:I\ni 0\to\mathcal O$ be the timelike future-directed geodesic with initial conditions $(\gamma(0), \dot\gamma(0))$. One can compute uniquely $(\ve, \ell_z, q)$ which are, together with the signs of $v_r(0)$ and $v_\theta(0)$ and $(r(0), \theta(0))$, sufficient to solve the reduced system. 
\end{lemma}

\begin{proof}
Let $\gamma:I\ni 0\to\mathcal O$ be the timelike future-directed geodesic with initial conditions $(\gamma(0), \dot\gamma(0)) =  (t(0), \phi(0), r(0), \theta(0), v^t(0), v^{\phi}(0), v^r(0), v^{\theta}(0))$. 
\begin{itemize}
\item First, we compute $(\ve, \ell_z)$ from $(r(0), \theta(0), v^t(0), v^{\phi}(0))$: 
\begin{equation*}
\ve = V_K(r(0), \theta(0))v^t(0) - W_K(r(0), \theta(0))v^\phi(0),
\end{equation*}
\begin{equation*}
\ell_z = W_K(r(0), \theta(0))v^t(0) + X_K(r(0), \theta(0))v^\phi(0),
\end{equation*}
then $q$ from $(\theta(0), v_{\theta}(0), \ve, \ell_z)$ 
\begin{equation*}
q = v_\theta(0)^2 + \cos^2\theta(0)\left(d^2(1-\ve^2) + \frac{\ell_z^2}{\sin^2\theta(0)} \right). 
\end{equation*}  
\item Now, we consider the reduced system \eqref{reduced:reduced} with parameters $(\ve, \ell_z, q)$. 
\item The motion in $(r, \theta)$ is determined uniquely by $(r(0), \theta(0))(\ve, \ell_z, q)$ and $(v_r, v_\theta)(0)$. In fact, $(v_r, v_\theta)(0)$ is determined using \eqref{r::plane} and \eqref{theta::plane}. Imposing the sign on the latter allows one to determine uniquely $(r, \theta)$.
\end{itemize}
\end{proof}
\begin{remark}
It will be sufficient to study the reduced system in order to determine the nature of $\gamma$ according to Definition \ref{classify:categorie}.  
\end{remark}
\noindent Therefore, we will study the reduced system \eqref{reduced:reduced}.   
\\ We first determine the admissible values for $(\ve, \ell_z)$
\begin{lemma}
 Let $(\gamma, I)$ be a timelike future directed geodesic moving in the exterior region with constants of motion $(\ve, \ell_z, q)$. Then 
 \begin{equation*}
 (\ve, \ell_z)\in\mathcal A^{admissible, +}\cup \mathcal A^{admissible, -} =: \Adm
 \end{equation*}
 where 
 \begin{equation*}
 \mathcal A^{admissible, +} := \left\{(\ve, \ell_z)\in \mathbb R^2 \;:\; \ve > 0 \;\text{and}\; d\ell_z >0\right\} 
 \end{equation*}
 and 
  \begin{equation*}
 \mathcal A^{admissible, -} := \left\{(\ve, \ell_z)\in \mathbb R^2 \;:\; \ve > \frac{d\ell_z}{2r_H}  \;\text{and}\; d\ell_z <0\right\} 
 \end{equation*}
 \end{lemma}
 
\noindent Before classifying the solutions of the reduced system, we begin by determining necessary and sufficient conditions for the existence of  stationary solutions.
\begin{lemma}
\label{stationary::point}
Let $(\ve, \ell_z)\in\mathbb R\times\mathbb R$ and let $(r_s, \theta_s, v_{r, s}, v_{\theta, s})$ be a timelike future-directed stationary solution of \eqref{reduced:reduced}. Then, $r_s$ is a double root of the fourth order polynomial $R(\cdot, \ve, \ell_z,  q_s)$ and $\cos\theta_s$ is a double root of the  polynomial $T(\cdot, \ve, \ell_z,  q_s)$. 
\\ Reciprocally, if $r_s$ is a double root of $R(\cdot, \ve, \ell_z, q_s)$ and $\cos\theta_s$ is a double root of the  polynomial $T(\cdot, \ve, \ell_z,  q_s)$, then we have a stationary solution of \eqref{reduced:reduced}.
\end{lemma}
\begin{proof}
 Let $(\ve, \ell_z)\in\mathbb R\times\mathbb R$ and let $(r_s, \theta_s, v_{r, s}, v_{\theta, s})$ be a stationary timelike future-directed solution of \eqref{reduced:reduced}. Then $(r_s, \theta_s, v_{r, s}, v_{\theta, s})$ verifies
\begin{equation*}  
\begin{aligned}
&v_{r, s} = v_{\theta, s} = 0 , \\
 &\frac{\Delta'(r_s)R(r_s, \ve, \ell_z, q_s, d) - \Delta(r_s)\partial_r R(r_s, \ve, \ell_z, q_s, d)}{\Delta(r_s)^2} = 0 , \\
 &  \partial_{\theta}\left( \frac{T(\cos\theta_s, \ve, \ell_z, q_s, d)}{\sin^2\theta_s}\right) = 0. 
\end{aligned}
\end{equation*}
where $q_s$ is the Carter constant given by \eqref{Q:Carter}. Moreover, by \eqref{r::plane} and \eqref{theta::plane}, we have 
 \begin{equation*}
 R(r_s, \ve, \ell_z, q_s, d) = T(\cos\theta_s, \ve, \ell_z, q_s, d) = 0.
 \end{equation*}
 Therefore, 
 \begin{equation*}
 \partial_r R(r_s, \ve, \ell_z, q_s, d) = \partial_\theta T(\cos\theta_s, \ve, \ell_z, q_s, d) = 0. 
 \end{equation*}
 \noindent Reciprocally, let $r_s$ be a double root of the four polynomial $R(\cdot, \ve, \ell_z,  q_s)$ and $\cos\theta_s$ be a double root of the  polynomial $T(\cdot, \ve, \ell_z,  q_s)$ and let $(r, \theta, v_r, v_\theta): I\ni 0\to \mathcal O$ be a solution to \eqref{reduced:reduced} such that $(r, \theta)(0, 0) = (r_s, \theta_s)$. Then, $(r, \theta): I\to\mathcal O$ satisfies \eqref{r::plane} and \eqref{theta::plane}. Therefore, 
 \begin{equation*}
 v_r(0) = 0 \quad\text{and}\quad v_\theta(0) = 0. 
 \end{equation*}
 Moreover, the point $(r_s, \cos\theta_s, 0, 0)$ is a critical point for the system \eqref{reduced:reduced}. This yields the result. 
 \end{proof}
\noindent Now, we derive sufficient conditions for stationary solutions to the reduced system \eqref{reduced:reduced}. 
\\
\\ \noindent $R$ can be seen as an effective potential governing the motion in the radial direction $r$ and $T$ as an effective potential governing the motion in the angular direction $\theta$. Therefore, we can characterise the geodesic motion by studying the number of turning points of the radial motion and the number of turning points of the angular motion. 
\\ We recall from Section \ref{2:DOF} that the allowed region for a timelike future directed geodesic $\gamma: I\to \mathcal O$ with constants of motion $(\ve, \ell_z)$ is given  by 
\begin{equation}
\label{allowed:for:kerr}
A^K(\ve, \ell_z) = \left\{(r, \theta)\in ]r_H, \infty[\times]0, \pi[\;:\;  \tilde J^K(r, \theta, \ve, \ell_z)\geq 0\right\},
\end{equation} 
where $\tilde J^K$ is defined by \eqref{J:::tilde} and that the associated ZVC, the set of turning points,  is given by 
\begin{equation*}
Z^K(\ve, \ell_z) = \left\{(r, \theta)\in ]r_H, \infty[\times]0, \pi[\;:\;  \tilde J^K(r, \theta, \ve, \ell_z) = 0\right\}. 
\end{equation*} 
In BL coordinates, we characterise $Z^K$ by the following lemma 
\begin{lemma}
\begin{equation*}
\label{charac:Zc}
\SymbolPrint{Z^K}(\ve, \ell_z) = \left\{(r, \theta)\in ]r_H, \infty[\times]0, \pi[\;:\;  R(r, \ve, \ell_z, q) = 0 \;\text{where}\; q = \cos^2\theta\left(d^2(1-\ve^2) + \frac{\ell_z^2}{\sin^2\theta}\right)\right\}. 
\end{equation*} 
\end{lemma}
\begin{proof}
Let $(r,\theta)\in Z^K(\ve, \ell_z)$. Then,
\begin{equation*}
\tilde J^K(r, \theta, \ve, \ell_z) = 1 -  \frac{X_K(r, \theta)}{\Delta\sin^2\theta}\varepsilon^2 - \frac{2W_K(r, \theta)}{\Delta\sin^2\theta}\varepsilon\ell_z + \frac{V_K(r, \theta)}{\Delta\sin^2\theta}\ell_z^2 = 0. 
\end{equation*}
This implies: 
\begin{equation*}
1 - \frac{\Pi}{\Delta\Sigma^2}\ve^2 + \frac{4dr}{\Delta\Sigma^2}\ve\ell_z + \left(1- \frac{2r}{\Sigma^2}\right)\frac{\Sigma^2}{\sin^2\theta}\ell_z^2= 0.
\end{equation*}
Therefore, 
\begin{equation*}
(r^2 +d^2)^2\ve^2 - \ve^2d^2\sin^2\theta\Delta - \Delta r^2 - \Delta d^2\cos^2\theta - 4dr\ve\ell_z - (r^2 + d^2 - d^2\sin^2\theta - 2r)\frac{\ell_z^2}{\sin^2\theta} = 0. 
\end{equation*}
Hence,
\begin{equation*}
((r^2 +d^2)^2\ve - d\ell_z)^2 - \Delta \left(r^2 + \ve^2d^2\sin^2\theta  - 2d\ve\ell_z - (\cos^2\theta + \sin^2\theta)\frac{\ell_z^2}{\sin^2\theta}\right)= 0. 
\end{equation*}
Finally, we set $q$ to be: 
\begin{equation*}
q = \cos^2\theta\left(d^2(1-\ve^2) + \frac{\ell_z^2}{\sin^2\theta}\right).
\end{equation*}
The latter expression becomes
\begin{equation*}
((r^2 +d^2)^2\ve - d\ell_z)^2 - \Delta \left(r^2 + (d\ve - \ell_z)^2 + q\right) = 0.
\end{equation*}
Hence, 
\begin{equation*}
R(r, \ve, \ell_z, q) = 0. 
\end{equation*}
Reciprocally, if $\displaystyle (r, \theta)\in \left\{(r, \theta)\in ]r_H, \infty[\times]0, \pi[\;:\;  R(r, \ve, \ell_z, q) = 0 \;\text{where}\; q = \cos^2\theta\left(d^2(1-\ve^2) + \frac{\ell_z^2}{\sin^2\theta}\right)\right\}$, then
 \begin{equation*}
 R(r, \ve, \ell_z, q(\theta, \ve, \ell_z)) = 0 \quad\text{and}\quad q(\theta, \ve, \ell_z) = \cos^2\theta\left(d^2(1-\ve^2) + \frac{\ell_z^2}{\sin^2\theta}\right). 
 \end{equation*}
 We plug the expression of $q$ in the first equation and we inverse the above steps to obtain
 \begin{equation*}
 \tilde J^K(r, \theta, \ve, \ell_z) = 0. 
 \end{equation*}
\end{proof}
\noindent Therefore, a turning point $(r_0, \theta_0)(\ve, \ell_z)$ is such that $r_0$ is a root of the fourth order polynomial $\displaystyle R(\cdot, \ve, \ell_z, q)$ with $\displaystyle q = \cos^2\theta_0\left(d^2(1-\ve^2) + \frac{\ell_z^2}{\sin^2\theta_0}\right)$.
\noindent From the above definition, we obtain conditions on the number of roots for occurence of the above orbits, given by the following lemma 
\begin{lemma}
\label{classifi}
Let $\gamma: I\to\mathcal O$ be a timelike future-directed geodesic with integrals of motion $(\ve, \ell_z, q)$. 
\begin{itemize}
\item Spherical orbits occur only  when $R(\cdot, \ve, \ell_z, q)$ has a double root. 
\item Trapped non-spherical orbits occur when $R(\cdot, \ve, \ell_z, q)$ has three distinct roots. 
\item Plunging orbits occur in all cases. 
\item Scattered orbits occur only when $R(\cdot, \ve, \ell_z, q)$ has two distinct roots. 
\end{itemize}
\end{lemma}
\begin{proof}
\begin{enumerate}
\item Suppose that $\gamma$ has a constant radius $r_s$. Then, $\forall \gamma\in I$, we have 
\begin{equation*}
r(\tau) = r_s \quad\text{and}\quad \dot r(\tau) = 0. 
\end{equation*}
This also implies that $\displaystyle \ddot r(\tau) = 0$. By \eqref{r::plane}, we obtain that $r_s$ is a root for $R(\cdot, \ve, \ell_z, q)$. Moreover, by \eqref{reduced:reduced}, we get: $\displaystyle \partial_rR(r_s, \ve, \ell_z, q) = 0$. Therefore, $r_s$ is a double root for $R(\cdot, \ve, \ell_z, q)$. 
\item If $R(\cdot, \ve, \ell_z, q)$ has three distinct roots, then by \eqref{r::plane}, $\forall r\in I$, $r\in]r_H, r^K_0(\ve, \ell_z, q)]$ or $r\in[r^K_1(\ve, \ell_z, q), r^K_2(\ve, \ell_z, q)]$, where $r^K_i$ are the roots of $R(\cdot, \ve, \ell_z, q)$. If $r$ lies in the compact region, then it is strapped.  
\item If $R(\cdot, \ve, \ell_z, q)$ has two distinct roots, then $\forall r\in I$, $r\in]r_H, r^K_0(\ve, \ell_z, q)]$ or $r\in[r^K_1(\ve, \ell_z, q), \infty[$. If $r$ lies in the unbounded region, then it is scattered.  
\end{enumerate}
\end{proof}
\noindent In order to determine the allowed region for a particle, $A^K(\ve, \ell_z)$, we first determine $\partial A^K(\ve, \ell_z) = Z^K(\ve, \ell_z)$ in the region $]r_H, \infty[\times]0, \pi[$ based on the possible values of $(\ve, \ell_z)$.  To this end, we study the roots or the  polynomial $R(\cdot, \ve, \ell_z)$ in the region $(r_H (d), \infty)$ whose existence restrict the range of $q$ and thus  that of $\theta$. In fact, the ZVCs are smooth curves with eventually different connected components.
\\Now, we claim that 
\begin{lemma}
\label{roots:of:RR}
Let $(\ve, \ell_z, q)\in\mathbb R^3$. 
\begin{itemize}
\item if $\ve^2< 1$,  $R(\cdot, \ve, \ell_z, q)$ has either one real root or three real  roots counted with their multiplicity in the region $]r_H, \infty[$.
\item Otherwise, $R(\cdot, \ve, \ell_z, q)$ has either zero real roots or two real roots counted with their multiplicity  in the region $]r_H(d), \infty[$.
\item $R(\cdot, \ve, \ell_z, q)$ cannot have four roots  in the region $]r_H, \infty[$. 
\end{itemize}
\end{lemma}
\begin{proof}
We have the following asymptotics: 
\begin{equation}
\label{asympto::R}
R(r_H) = (2r_H\ve - d\ell_z)^2\geq 0 \quad\text{and}\quad \lim_{r\to\infty} R(r) = \lim_{r\to\infty} (\ve^2 - 1)r^4.
\end{equation}
Therefore, 
\begin{equation*}
 \lim_{r\to\infty} (\ve^2 - 1)r^4 = 
 \left\{ 
 \begin{aligned}
 &-\infty \quad\text{if}\quad \ve^2<1, \\
 &+\infty \quad\text{if}\quad \ve^2>1, 
 \end{aligned}
 \right.
\end{equation*}
If $\ve^2 = 1$, we look at the sign of the third degree term $-2r^3$. Hence, $\lim_{r\to\infty} R(r) = -\infty$. Now since $R$ is a fourth degree polynomial (cubic when $\ve^2 = 1$), the number of roots counted with their multiplicity is at most $4$ ($3$ when $\ve^2 = 1$). This yields the result.
\\ Finally, we reexpress $R(r, \ve, \ell_z, q)$ in terms of $x := r-1$: 
\begin{equation*}
\begin{aligned}
\tilde R (x)&:=  = a_4x^4 + a_3x^3 + a_2x^2 + a_1x + a_0,
\end{aligned}
\end{equation*}
where $a_4 := \ve^2 - 1$ and $a_3 := 4\ve^2 - 2$. In the region, $r>r_h$, we have $x>0$.  If $\ve^2> 1$, then the latter terms are positive and we have at most two variations of sign and therefore at most two roots. If $\ve^2<1$ then, $R(\ve, \ell_z, q)$ cannot admit four roots according to the first point. 
\end{proof}
\noindent  We now  state a necessary condition for non existence of classical orbits. 
\begin{lemma}
\label{positive::q}
Let $\gamma: I\to \mathcal O$ be a timelike future-directed geodesic with constants of motion $(\ve, \ell_z, q)$. If $q< 0$ then, $\ve^2>1$ and $R(\cdot, \ve, \ell_z, q)$ has no roots in the region $(r_H, \infty)$. Therefore, the geodesic starts from infinity and reaches the horizon in a finite proper time. Consequently, if $\ve^2<1$, then we necessarily have $q\geq0$. 
\end{lemma}
\begin{proof}
Suppose that $q<0$ and let $r$ be a root of $R(\cdot, \ve, \ell_z, q)$. Then, by \eqref{Q:Carter} we have 
\begin{equation*}
d^2(1 - \ve^2) + \frac{\ell_z^2}{\sin^2\theta} < 0. 
\end{equation*}
Thus, 
\begin{equation}
\label{eq::3333}
d^2(\ve^2 - 1) >\frac{\ell_z^2}{\sin^2\theta} \geq 0. 
\end{equation}
We recall that $\displaystyle r_H = 1 + \sqrt{1 - d^2}$.  Therefore, the roots of $R(\cdot, \ve, \ell_z, q)$ in the exterior region will satisfy $r - 1\geq \sqrt{1 - d^2} \geq 0 $. Now, we introduce $x$ such that 
\begin{equation*}
x := r - 1
\end{equation*}
and we reexpress $R$ in terms of $x$:
\begin{equation*}
\begin{aligned}
\tilde R (x)&:= (\ve^2 - 1)x^4 + (4\ve^2 - 2)x^3 + ((6 + d^2)\ve^2 - d^2 - \ell_z^2 - q)x^2 + ((4 + 2d^2)(\ve^2 - 1) + 6 + 2d^2\ve^2 - 4d\ell_z\ve)x \\
&+( (2d\ve - \ell_z)^2 + (\ve^2 + 1 + q)(1-d^2)) \\
&= a_4x^4 + a_3x^3 + a_2x^2 + a_1x + a_0. 
\end{aligned}
\end{equation*}
We are interested only in the positive  roots of $\tilde R$. Now, by Lemma \ref{roots:of:RR}, $\tilde R$ admits either two positive roots or no positive roots. We look at the variations of signs in $\tilde R$: 
\begin{itemize}
\item The sign of the factors $a_4$ and $a_3$ are positive since $\ve^2>1$. 
\item The sign of $a_2$ is positive by \eqref{eq::3333} and by $q<0$.  
\end{itemize}
Depending on the sign of $a_1$, either we have zero roots if $a_1>0$ or one root if $a_1<0$,  by Descartes's rule of sign.This ends the proof. 
\end{proof}
Since we are interested in trapped geodesics, we will henceforth study the case where  $q$ is  non-negative.  The case of  vanishing $q$ is of particular interest. More precisely, 
\begin{lemma}
\label{vanishing:q}
Let $(\gamma, I)$ be a timelike future-directed geodesic moving in the exterior region  with constants of motion $(\ve, \ell_z, q)$. Then, $q = 0$ is a necessary and sufficient condition for a motion initially in the equatorial plane to remain in the equatorial plane for all time.
\end{lemma}
\begin{proof}
If $\gamma$ is confined to the equatorial plane, then $\forall\tau\in I$
\begin{equation*}
\theta(\tau) = \frac{\pi}{2} \quad\text{and}\quad v^\theta(\tau) = \dot\theta(\tau) = 0.  
\end{equation*}
Therefore, by \eqref{Q:Carter}, we have $q = 0$. 
Now, assume that $q = 0$ and $\theta(0) = \frac{\pi}{2}$.  Then, there exists a unique $\gamma: I\to\mathcal O$ solution to \eqref{reduced:reduced} with initial conditions $\left( r(0), \frac{\pi}{2}, v_r(0), 0\right)$. Therefore, by \eqref{theta::plane}, we obtain 
\begin{equation*}
\forall\tau\in I\;,\; v_\theta(\tau) = 0 .
\end{equation*}
\end{proof}
\noindent Now, we study the roots of $T(\cdot, \ve, \ell_z, q)$ in the region $]-1, 1[$ for $(\ve, \ell_z, q)\in\mathbb R\times \mathbb R\times[0, \infty[$. We state the following lemma
\begin{lemma}
\label{T:double:root}
Assume that $(\ve, \ell_z, q)\in\mathbb R\times\mathbb R\times [0, \infty[$ and consider the equation in $Y$
\begin{equation}
\label{T::4:0}
T(Y, \ve, \ell_z, q) = 0. 
\end{equation}
on $]-1, 1[$. Then,  $T$ has 
\begin{itemize}
\item  two roots counted with their multiplicity in the region $]-1, 1[$ if and only if $d^2(\ve^2 - 1)\leq \ell_z^2$. They are given by 
\begin{equation*}
Y = \pm\sqrt{y_+},
\end{equation*}
\item  two simple roots  in the region $]-1, 1[$ if and only if $d^2(\ve^2 - 1)> \ell_z^2$ and $q>0$. They are given by 
\begin{equation*}
Y = \pm\sqrt{y_-},
\end{equation*}
\item  one double root given by $0$ and two simple roots  in the region $]-1, 1[$ if and only if $d^2(\ve^2 - 1)> \ell_z^2$ and $q=0$. The simple roots are given by 
\begin{equation*}
Y = \pm\sqrt{y_-},
\end{equation*}
\end{itemize}
where 
\begin{equation}
\label{theta:sol}
y = \left\{
\begin{aligned}
&y_+ = \frac{(\ell_z^2+d^2(1-\varepsilon^2) + q + \sqrt{(\ell_z^2+d^2(1-\varepsilon^2) + q)^2 - 4qd^2(1-\varepsilon^2)})}{2d^2(1-\varepsilon^2)} \quad\text{if}\; \ve^2 < 1\\
&y_- = \frac{(\ell_z^2+d^2(1-\varepsilon^2) + q - \sqrt{(\ell_z^2+d^2(1-\varepsilon^2) + q)^2 - 4qd^2(1-\varepsilon^2)})}{2d^2(1-\varepsilon^2)} \quad\text{if}\; \ve^2 > 1\\
&\frac{q}{\ell_z^2 + q} \quad\text{if} \;  \ve^2 = 1. \\
\end{aligned}
\right. 
\end{equation}
Moreover, if $Y$ is a double root of $T$ in $]-1, 1[$, then $\displaystyle Y = 0$ and $q = 0$. 
\end{lemma}
\begin{proof}
Let $(\ve, \ell_z, q)\in]0, \infty[\times\mathbb R^\ast\times [0, \infty[$ and let $Y\in]-1, 1[$ be a solution of \eqref{T::4:0}. Then, $Y$ verifies 
\begin{equation*}
q = F(Y, \ve, \ell_z) \quad\text{where}\quad  F(Y, \ve, \ell_z) := Y^2\left(-d^2(\ve^2 - 1) + \frac{\ell_z^2}{1 - Y^2} \right). 
\end{equation*}
We have $\forall (\ve, \ell_z)\in\mathbb R\times\mathbb R^\ast$ 
\begin{itemize}
\item \begin{equation*}
\lim_{|Y|\to 1}\, F(Y, \ve, \ell_z) = + \infty,
\end{equation*}
\item $$\displaystyle F(0, \ve, \ell_z) = 0,$$
\item
\begin{equation*}
\begin{aligned}
\frac{\partial F}{\partial Y} &= \frac{2Y}{(1 - Y^2)^2}\left(  \ell_z^2 - d^2(\ve^2 - 1)(1 - Y^2)\right).  
\end{aligned}
\end{equation*}
Therefore, $Y = 0$ is always a critical point for $F(\cdot, \ve, \ell_z)$ and 
\begin{itemize}
\item if $d^2(\ve^2 - 1)\leq \ell_z^2$, then  $0$ is the unique critical point.
\item Otherwise, there are two more critical points given by 
\begin{equation*}
y_c = \pm \left(1 - \sqrt{\frac{\ell_z^2}{d^2(\ve^2 - 1)}} \right),
\end{equation*}
and they verify 
\begin{equation*}
-1 < y_c^-<0<y_c^+<1. 
\end{equation*}
\end{itemize}
\end{itemize}
Therefore, if $q\geq 0$, then the equation $q = F(Y, \ve, \ell_z)$ admits
\begin{itemize}
\item  if $\displaystyle d^2(\ve^2 - 1)\leq \ell_z^2$,  two roots in the region $]-1, 1[$ symmetric about $0$. These roots coincide if and only if $q = 0$ and they are given by $0$. 
\item Otherwise, 
\begin{itemize}
\item four roots if and only if $q = 0$: one double root given by $0$ and two simple roots symmetric about $0$,
\item two simple roots otherwise. 
\end{itemize}
\end{itemize}
\noindent In order to write $Y$ in terms of $(\ve, \ell_z, q)$,  we make the following change of variables $Y = \sqrt{y}$ and we consider $\overline T(y):= T(\sqrt{y})$ on $[0, 1[$.  Since $\overline T$ is quadratic and its discriminant is always positive, its roots are given by 
\begin{equation*}
\left\{
\begin{aligned}
y &= y_\pm :=  \frac{\ell_z^2+d^2(1-\varepsilon^2) + q \pm \sqrt{(\ell_z^2+d^2(1-\varepsilon^2) + q)^2 - 4qd^2(1-\varepsilon^2)}}{2d^2(1-\varepsilon^2)} \quad\ve \neq 1. \\
&=  \frac{q}{\ell_z^2 + q} \quad\;  \ve^2 = 1.
\end{aligned}
\right.
\end{equation*}
\begin{itemize}
\item If $d^2(\ve^2 - 1)\leq \ell_z^2$, then $y_+>0$ and $y_-<0$.  Therefore, 
 \begin{equation*}
 y = y_+. 
 \end{equation*}
\item Otherwise,
\begin{equation*}
y = y_-
\end{equation*}
\end{itemize}
\end{proof}
Finally, we state necessary and sufficient conditions for the  occurence of spherical orbits.
\begin{lemma}
\label{cond:circular}
Let $\gamma: I\to\mathcal O$ be a timelike future-directed geodesic with constants of motion $(\ve, \ell_z, q)$. Then $\gamma$ is spherical of radius $r_c$ and confined to the equatorial plane if and only if 
\begin{itemize}
\item $\gamma$ starts at some point $\displaystyle \left(t, \phi, r_c, \frac{\pi}{2}\right)$. 
\item $q = 0$ and $r_c$ is a double root of $R(\cdot, \ve, \ell_z, q)$. 
\end{itemize}
\end{lemma}

\begin{proof}
If $\gamma$ spherical of radius $r_c$ then by Lemma \ref{classifi}, $r_c$ is a double root of $R(\cdot, \ve, \ell_z, q)$. If $\gamma$ is confined in the equatorial plane, then by Lemma \ref{vanishing:q},  $q = 0$. 
\\Reciprocally, if $\gamma$ satisfies the conditions of Lemma \ref{cond:circular}, then by Lemma  \ref{vanishing:q}, $\gamma$ is confined to the equatorial plane. Now, if $r_c$ is a double root of $R(\cdot, \ve, \ell_z, q)$ and starts at $r(0) = r_c$, then the point $\displaystyle (r_c, \frac{\pi}{2}, 0, 0)$ is a critical point for the reduced system \eqref{reduced:reduced}. Therefore, $\gamma$ has a constant radius.  
\end{proof}

\begin{lemma}
\label{spherical}
Let $\gamma: I\to\mathcal O$ be a timelike future-directed geodesic with constants of motion $(\ve, \ell_z, q)$. Then $\gamma$ is spherical of radius $r_c$  if and only if 
\begin{itemize}
\item $\gamma$ starts at some point $(t, \phi, r_s, \theta)\in \mathcal O$. 
\item  $r_s$ is a double root of $R(\cdot, \ve, \ell_z, q)$. 
\end{itemize}
\end{lemma}

\begin{proof}
The proof is similar to the previous lemma. 
\end{proof}

\subsubsection{Circular orbits confined in the equatorial plane $\displaystyle \theta = \frac{\pi}{2}$}
\label{circ::eq} 
We present here a detailed study of  circular geodesic motion in the equatorial plane $\theta$. We note that the study of circular orbits is included in classical books of general relativity. See for example \cite[Chapter 6]{chandrasekhar1998mathematical}. We state the main result of this section
\begin{Propo}
\label{classif:circ:eq}
Let $\gamma: I\to\mathcal O$ be a timelike future directed geodesic with constants of motion $(\ve, \ell_z, q)\in\mathbb R\times\mathbb R\times \mathbb R$. Assume that $\gamma$ is a circular orbit of radius $r_c$ confined in the equatorial plane. Then, 
\begin{equation*}
(\ve, \ell_z, q)\in \mathcal A_{circ}^{+, \leq 1} \cup \mathcal A_{circ}^{+, \geq 1} \cup \mathcal A_{circ}^{-, \leq 1} \cup \mathcal A_{circ}^{-, \geq 1} =: \SymbolPrint{\mathcal A_{circ}}
\end{equation*}
where
\begin{equation}
\begin{aligned}
&\SymbolPrint{\mathcal A_{circ}^{+, \leq 1}} :=  \\
&\left\{(\ve, \ell_z, 0)\;,\; \ve\in]\ve^+_{min}, 1[\;,\; \ell_z = \ell_{lb}^+(\ve)   \right\}\sqcup \left\{(\ve, \ell_z, 0)\;,\; \ve\in]\ve^+_{min}, 1[\;,\; \ell_z = \ell_{ub}^+(\ve) \right\} \sqcup \left\{(\ve^+_{min}, \ell^+_{min}, 0) \right\}, 
\end{aligned}
\end{equation}
\begin{equation}
\begin{aligned}
&\SymbolPrint{\mathcal A_{circ}^{-, \leq 1}} :=  \\
&\left\{(\ve, \ell_z, 0)\;,\; \ve\in]\ve^-_{min}, 1[\;,\; \ell_z = \ell_{lb}^-(\ve)   \right\}\sqcup \left\{(\ve, \ell_z, 0)\;,\; \ve\in]\ve^-_{min}, 1[\;,\; \ell_z = \ell_{ub}^-(\ve) \right\} \sqcup \left\{(\ve^-_{min}, \ell^-_{min}, 0) \right\}, 
\end{aligned}
\end{equation}
\begin{equation}
\SymbolPrint{\mathcal A_{circ}^{+, \geq 1}} := \left\{(\ve, \ell_z, 0)\;,\; \ve\in[1, \infty[\;,\; \ell_z = \ell_{lb}^+(\ve)   \right\}, 
\end{equation}
and 
\begin{equation}
\SymbolPrint{\mathcal A_{circ}^{-, \geq 1}} :=  \left\{(\ve, \ell_z, 0)\;,\; \ve\in[1, \infty[\;,\; \ell_z = \ell_{lb}^-(\ve)   \right\}
\end{equation}
and where $\ve^\pm_{min}$, $\ell_{min}^\pm$, $\ell_{ub}^\pm(\ve)$ and $\ell_{lb}^\pm(\ve)$ are given by \eqref{e:ms}, \eqref{l:ms} and Definition \ref{lb:ub} respectively. 
\\Moreover, $r_c$ is given by 
\begin{enumerate}
\item if $(\ve, \ell_z, q)\in \mathcal A_{circ}^{+, \leq 1}$, then 
\begin{equation*}
r_c\in\left\{r_{ms}^+, \tilde r_{max}^+(\ve), r_{min}^+(\ve) \right\}, 
\end{equation*}
\item if $(\ve, \ell_z, q)\in \mathcal A_{circ}^{-, \leq 1}$, then 
\begin{equation*}
r_c\in\left\{r_{ms}^-, \tilde r_{max}^-(\ve), r_{min}^-(\ve) \right\}, 
\end{equation*}
\item if $(\ve, \ell_z, q)\in \mathcal A_{circ}^{+, \geq 1}$, then 
\begin{equation*}
r_c = r_{max}^+(\ve),
\end{equation*}
\item if $(\ve, \ell_z, q)\in \mathcal A_{circ}^{-, \geq 1}$, then 
\begin{equation*}
r_c = r_{max}^-(\ve),
\end{equation*}
\end{enumerate}
where $r^\pm_{ms}$ is defined by  \eqref{r::ms}, $\tilde r^\pm_{max}$ and $r^\pm_{min}$ are defined in Lemma \ref{lemma:circular}. 
\end{Propo}
\noindent The remaining of this section is devoted to the proof of Proposition \ref{classif:circ:eq}. We start with the following lemma 
\begin{lemma}
\label{lemma::22}
Let $\gamma: I\to \mathcal O\,,\, \tau\mapsto (t(\tau), \phi(\tau), r(\tau), \theta(\tau))$ be a circular orbit of radius $\rcc>r_H$ confined in the equatorial plane and let $(\ve, \ell_z, q)\in \mathbb R\times\mathbb R\times\mathbb R$ be its associated integrals of motion. Then 
 \begin{equation*}
 q = 0 
 \end{equation*}
 and $(r_c, \ve, \ell_z)$ satisfies the following system of equations 
 \begin{align*}
3\rcc^4 - \frac{\rcc^2}{\varepsilon^2}(3\rcc^2-4\rcc+d^2) + \rcc^2d^2  &= \rcc^2\ell_z^2 , \\
\rcc^4 - \frac{\rcc^3}{\varepsilon^2}(\rcc-1) - d^2\rcc  &= \rcc(\ell^2 - 2d\ell_z). 
\end{align*}
\end{lemma}
\begin{proof}
$\gamma$ is confined to the equatorial plane.  Since $\gamma$ is circular, by Lemma \ref{cond:circular}, $r_c$ is a double root of $R(\cdot, \ve, \ell_z, q)$ and $q = 0$. Therefore,  $(\rcc, \ve, \ell_z, q= 0)$ satisfies 
\begin{equation}
\label{R:dot:R}
\begin{aligned}
R(\rcc, \ve, \ell_z, q= 0) &=  0, \\
\frac{\partial R}{\partial r}(\rcc, \ve, \ell_z, q= 0) &= 0,  
\end{aligned} 
\end{equation}
which is equivalent to\footnote{The system of equations is obtained by a linear combination of \eqref{R:dot:R}. We refer to the computations performed in Section \ref{critical:orbit} for the general case (non-vanishing q).}
\begin{equation}
\label{r:s:bis:}
 \begin{aligned}
3\rcc^4 - \frac{\rcc^2}{\varepsilon^2}(3\rcc^2-4\rcc+d^2) + \rcc^2d^2  &= \rcc^2\ell_z^2 , \\
\rcc^4 - \frac{\rcc^3}{\varepsilon^2}(\rcc-1) - d^2\rcc  &= \rcc(\ell^2 - 2d\ell_z). 
\end{aligned}
\end{equation}
\end{proof}
\noindent We use the equations \eqref{r:s:bis:} (See \cite[Chapter 6 ]{chandrasekhar1998mathematical}, \cite{bardeen1973timelike}, \cite{bardeen1972rotating}) to express $\ell_z$ and $\varepsilon$ in terms of $\rcc$:
\begin{align}
\label{phi::pm}
\varepsilon &= \frac{\rcc^{\frac{3}{2}} - 2\rcc^{\frac{1}{2}} \pm d}{\rcc^{\frac{3}{4}}\sqrt{\rcc^{\frac{3}{2}} - 3\rcc^{\frac{1}{2}} \pm 2d}} =: \SymbolPrint{\Phi_\pm}(\rcc) ,\\
\label{psi::pm}
\ell_z &= \pm\frac{\rcc^2 \mp 2d\rcc^{\frac{1}{2}} + d^2}{\rcc^{\frac{3}{4}}\sqrt{\rcc^{\frac{3}{2}} - 3\rcc^{\frac{1}{2}} \pm 2d}} =: \SymbolPrint{\Psi_\pm}(\rcc).\\
\end{align}
\begin{figure}[h!]
\includegraphics[width=\linewidth]{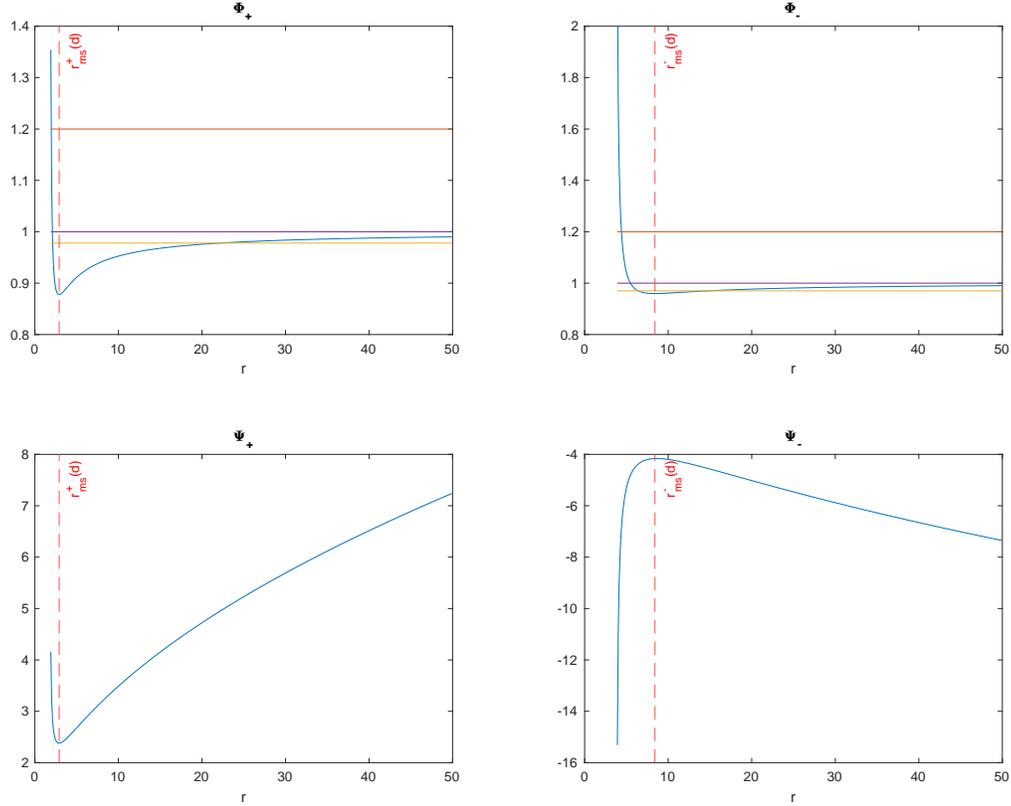}
\caption{{\it Shape of the the energy and the angular momentum in terms of the radius of circular motion for retrograde and direct orbit at $d = 0.8$. The horizontal lines represent the energy levels at $\varepsilon = 0.97$, $\varepsilon = 1$ and $\varepsilon = 1.2$}.}
\label{circ::orbit}
\end{figure}
Here, the upper sign refers to direct orbits $(d\ell_z>0)$ and the lower sign refers to retrograde orbits $(d\ell_z<0)$.  In order to determine the admissible values of $(\ve, \ell_z, \rcc)$ for circular orbits, we study the mappings $\Phi_\pm(\cdot, d)$ and $\Psi_\pm(\cdot, d)$ on the region $]r_H, \infty[$. First, we study their properties and we state the following lemma
\begin{lemma}
\begin{enumerate}
\item The mappings $\Phi_+$ and $\Psi_+$ given respectively by \eqref{phi::pm} and \eqref{psi::pm} (see also Figure \ref{circ::orbit}) are well-defined on $]r_{ph}^+, \infty[$ and they are smooth,   
\item the mappings $\Phi_-(\cdot, d)$ and $\Psi_-(\cdot, d)$ given respectively by \eqref{phi::pm} and \eqref{psi::pm} (see also Figure \ref{circ::orbit}) are well-defined on $]r_{ph}^-, \infty[$ and they are smooth. 
\end{enumerate}
Here $\SymbolPrint{r^{\pm}_{ph}}$ are defined by 
\begin{equation*}
r^{\pm}_{ph}(d) := 2\left(1 + \cos\left(\frac{2}{3}\cos^{-1}(\mp d) \right)\right). 
\end{equation*}
Moreover, we have the following asymptotics: 
\begin{equation*}
\begin{aligned}
&\lim_{r\to\infty} \Phi_\pm(r) = 1 \quad\text{and}\quad  \lim_{r\to\infty} \Psi_\pm(r) = \pm\infty, \\
&\lim_{r\to r^{\pm}_{ph}} \Phi_\pm(r) = +\infty \quad\text{and}\quad  \lim_{r\to r^{\pm}_{ph}} \Psi_\pm(r) = \pm\infty, \\
\end{aligned}
\end{equation*} 
\end{lemma} 
\begin{proof}
The mappings $\Phi_\pm$ and $\Psi_\pm$ are well-defined if and only if 
\begin{equation*}
r^{\frac{3}{2}} - 3r^{\frac{1}{2}} \pm 2d > 0.
\end{equation*}
Now, we consider the equations 
\begin{equation*}
r^{\frac{3}{2}} - 3r^{\frac{1}{2}} \pm 2d = 0
\end{equation*}
on $]r_H, \infty[$. After the change of variable $u = \sqrt{r}$, we have 
\begin{equation*}
u^3 - 3u \pm 2d = 0. 
\end{equation*}
We use Cardano's formula to express the unique real root in $]r_H, \infty[$ of the above equation: 
\begin{equation*}
u^2 = r^{\pm}_{ph}(d) := 2\left(1 + \cos\left(\frac{2}{3}\cos^{-1}(\mp d) \right)\right). 
\end{equation*}
Therefore $\Phi_+$, $\Psi_+$ are well-defined if and only if 
\begin{equation*}
r > r^{+}_{ph}(d)
\end{equation*}
and $\Phi_-$, $\Psi_-$ are well-defined if and only if 
\begin{equation*}
r > r^{-}_{ph}(d).
\end{equation*}
As for the asymptotics, it is clear that 
\begin{equation*}
\lim_{r\to r^{\pm}_{ph}} \Phi_\pm(r) = +\infty \quad\text{and}\quad  \lim_{r\to r^{\pm}_{ph}} \Psi_\pm(r) = \pm\infty.
\end{equation*}
In a neighbourhood of $+\infty$, we have 
\begin{equation*}
\Phi_\pm(r) \sim 1 \quad\text{and}\quad \Psi_\pm(r) \sim \pm\sqrt{r}. 
\end{equation*}
\end{proof}
\begin{remark}
$r^{\pm}_{ph} = r^{\pm}_{ph}(d)$ is a photon orbit. There are no circular direct orbits of radius $\displaystyle r <r^{+}_{ph}$ and there are no circular retrograde orbits of radius $\displaystyle r <r^{-}_{ph}$. When $d = 0$, $\displaystyle r^+_{ph} = r^-_{ph} = 3$, which the location of the photon sphere with $M = 1$. 
\end{remark}
\noindent Now, we claim that  
\begin{lemma}
\label{r:min:stable}
$\displaystyle \forall d\in[0, 1]\quad:\quad \exists! r^\pm_{ms}(d)\in]r^\pm_{ph}, +\infty[$ such that $\Phi_\pm'(r^\pm_{ms}(d))= \Psi_\pm'(r^\pm_{ms}(d)) = 0$. 
It is given by 
\begin{equation}
\label{r::ms}
\SymbolPrint{r_{ms}^\pm}(d) = 3 + Z_2(d) \mp \sqrt{(3 - Z_1)(3 + Z_1(d) + 2Z_2(d))},
\end{equation}
where 
\begin{equation*}
Z_1(d) = 1 + (1 - d^2)^{\frac{1}{3}}((1+d)^\frac{1}{3} + (1-d)^\frac{1}{3})\;, \; Z_2(d) = \sqrt{3d^2 + Z_1^2}. 
\end{equation*}
\end{lemma}

\begin{proof}
\begin{itemize}
We will prove the result for $\Phi_+$ and $\Psi_+$. The other case is treated in the same manner. Along the proof, we omit $+$ from the expression of the latter functions in order to lighten the expressions. 
\item First, we claim that $\Phi$ and $\Psi$ have the same critical points. In fact, let $r\in]r_{ph}, \infty[$ then $(r, \Phi(r), \Psi(r), 0)$ verifies: 
\begin{equation*}
R(r, \Phi(r), \Psi(r), 0) = 0 \quad\text{and}\quad \partial _rR(r, \Phi(r), \Psi(r), 0) = 0. 
\end{equation*}
We differentiate the first equation with respect to $r$ to obtain:
\begin{equation}
\label{CNS:CP}
\Phi'(r)\frac{\partial R}{\partial \ve}(r, \Phi(r), \Psi(r), 0) + \Psi'(r)\frac{\partial R}{\partial \ell_z}(r, \Phi(r), \Psi(r), 0) = 0. 
\end{equation}
Now, we show that $\displaystyle \frac{\partial R}{\partial \ve}(r, \Phi(r), \Psi(r), 0)$ and $\displaystyle\frac{\partial R}{\partial \ell_z}(r, \Phi(r), \Psi(r), 0)$ do not vanish on $]r_H, \infty[$:
\begin{itemize}
\item We compute
\begin{equation*}
\begin{aligned}
\frac{\partial R}{\partial \ell_z}(r, \Phi(r), \Psi(r), 0) &= -2r(2d\Phi(r) + (r-2)\Psi(r)) \\
&= -2r\frac{2d(r^{\frac{3}{2}} - 2r^{\frac{1}{2}} + d) + ( r- 2)(r^2 - 2dr^{\frac{1}{2}} + d^2)}{r^{\frac{3}{4}}\sqrt{r^{\frac{3}{2}} - 3r^{\frac{1}{2}} + 2d}} \\
&= -2r^2\frac{\Delta}{r^{\frac{3}{4}}\sqrt{r^{\frac{3}{2}} - 3r^{\frac{1}{2}} + 2d}}. 
\end{aligned}
\end{equation*}
The latter can not vanish since $r> r_H$. 
\item Now we compute
\begin{equation*}
\begin{aligned}
\frac{\partial R}{\partial \ve}(r, \Phi(r), \Psi(r), 0) &= 2r(-2d\Psi(r) + \Phi(r)(r^3 + d^2(r + 2))) \\
&= 2r\frac{d^3 + dr(r-2) + d^2r^{\frac{3}{2}} + (r-2)r^{\frac{5}{2}}}{r^{\frac{3}{4}}\sqrt{r^{\frac{3}{2}} - 3r^{\frac{1}{2}} + 2d}} \\
&= 2r^{\frac{5}{4}}\frac{\Delta(d + r^{\frac{3}{2}})}{\sqrt{r^{\frac{3}{2}} - 3r^{\frac{1}{2}} + 2d}}. 
\end{aligned}
\end{equation*}
The latter can not vanish since $r>r_H$. 
\end{itemize}
Therefore, by \eqref{CNS:CP}, $\Phi$ and $\Psi$ have the same critical points.  
\item  If $r_c$ is a critical point for $\Phi$, then $r_c$ satisfies: 
\begin{equation*}
\frac{\partial^2R}{\partial r^2}(r_c, \Phi(r_c), \Psi(r_c), 0) = 0. 
\end{equation*} 
Indeed, we have $\forall r>r_{ph}\,,\, $ $$ \frac{\partial}{\partial r}R(r, \Phi(r), \Psi(r), 0) = 0,$$
We differentiate this expression with respect to $r$ in order to obtain
\begin{equation*}
\frac{\partial^2 R}{\partial r^2}(r, \Phi(r), \Psi(r), 0) + \Phi'(r)\frac{\partial^2 R}{\partial \ve\partial r}(r, \Phi(r), \Psi(r), 0) + \Psi'(r)\frac{\partial^2 R}{\partial \ell_z\partial r}(r, \Phi(r), \Psi(r), 0) = 0. 
\end{equation*} 
If $r = r_c$, then it follows
\begin{equation*}
\frac{\partial^2R}{\partial r^2}(r_c, \Phi(r_c), \Psi(r_c), 0) = 0.
\end{equation*} 
\item Now we compute
\begin{equation*}
\begin{aligned}
\frac{\partial^2R}{\partial r^2}(r_c, \Phi(r_c), \Psi(r_c), 0)  &= 2(d^2(\Phi^2(r) -1) - \Psi(r)^2 + 6r + 6r^2(\Phi^2(r) - 1))  \\
&= -2\sqrt{r}\frac{-3d^2 + 8d\sqrt{r} + r(r - 6)}{2d + \sqrt{r}(r - 3)}. 
\end{aligned}
\end{equation*}
\item It remains to solve the equation 
\begin{equation*}
-3d^2 + 8d\sqrt{r} + r(r - 6) = 0
\end{equation*}
on $]r_{ph}, \infty[$. Again, we introduce the change of variables $u = \sqrt{r}$ and we obtain the quartic polynomial:
\begin{equation*}
u^4 - 6u^2 + 8du -3d^2 = 0. 
\end{equation*}
The latter is a depressed quartic equation that can be solved explicitly using Ferrari's method. The condition that $r>r_{ph}$ yields a unique root given by $r_{ms}$. 
\end{itemize}
\end{proof}
\noindent In the following, it is useful to introduce the quantity
\begin{equation}
\label{rho::ms}
\SymbolPrint{\rho^{\pm}_{ms}}(d):= \sqrt{\Delta(r^\pm_{ms}(d))}.  
\end{equation} 
$\displaystyle r_{ms}^\pm(d)$ is a minimiser for $\displaystyle\Phi_\pm, \displaystyle \Psi_+$ and $-\Psi_-$ and their minima are respectively given by  
\begin{equation}
\label{e:ms}
\SymbolPrint{\ve_{min}^\pm}(d):= \Phi_\pm(r^\pm_{ms}(d)) 
\end{equation}
and 
\begin{equation}
\label{l:ms}
\SymbolPrint{\ell_{min}^+}(d):= \Psi_+(r^+_{ms}) \quad\text{and}\quad \SymbolPrint{\ell_{min}^-}(d):= -\Psi_-(r^-_{ms})
\end{equation}
Moreover, we have 
\begin{equation}
\ve_{min}^+ \leq \ve_{min}^-
\end{equation}
and
\begin{equation}
\ve_{min}^+ = \ve_{min}^-
\end{equation}
 if and only if $d = 0$. In this case, $\displaystyle \ve_{min}^\pm(0) = \sqrt{\frac{8}{9}}$. From now on, we do not write the dependence of the above quantities on $d$.  
From the previous lemma, we obtain the following monotonicity properties for $\Phi_\pm$ and $\Psi_\pm$: 
\begin{lemma}
\label{monotonicite:::}
\begin{enumerate}
\item $\Phi_\pm$ is monotonically decreasing on $]r^{\pm}_{ph}, r^\pm_{ms}[$ and monotonically increasing on $]r^\pm_{ms}(d), \infty[$, 
\item $\Psi_+$ is monotonically decreasing on $]r^{+}_{ph}, r^+_{ms}[$ and monotonically increasing on $]r^+_{ms}, \infty[$, 
\item $\Psi_-$ is monotonically increasing on $]r^{-}_{ph}, r^-_{ms}[$ and monotonically decreasing on $]r^-_{ms}, \infty[$. 
\end{enumerate}
\end{lemma}
Now, we go back to the equations \eqref{phi::pm} and \eqref{psi::pm}: let $\ve\in]0, \infty[$ and consider the equations
\begin{equation}
\label{e:phi:p}
\ve = \Phi_+(r, d) \quad\text{on}\quad ]r^+_{ph}, \infty[ 
\end{equation}
and 
\begin{equation}
\label{e:phi:m}
\ve = \Phi_-(r, d) \quad\text{on}\quad ]r^-_{ph}, \infty[ 
\end{equation}
By the monotonicity properties of $\Phi_\pm$, we obtain the following lemma 
\begin{lemma}
\label{lemma:circular}
Let $\varepsilon\in\mathbb R$ and\footnote{The case of $d = 0$ corresponds to the Schwarzschild case and was already tackled. See Proposition \ref{Propo1}. } $d\in]0, 1]$, we have the following cases:
\begin{enumerate}
\item If $\displaystyle \varepsilon < \ve_{min}^+ $, then the equations \eqref{e:phi:p} and \eqref{e:phi:m} do not have solutions.  
\item If $\displaystyle \varepsilon = \ve_{min}^+$ , then \eqref{e:phi:p} admits a unique solution given by $r = r_{ms}^+$ and \eqref{e:phi:m} does not have solutions.  
\item If $\displaystyle \ve_{min}^+<\varepsilon<\ve_{min}^-$,  then \eqref{e:phi:p} admits two solutions $\displaystyle r^+_{max}(\ve)$ and $\displaystyle r^+_{min}(\ve)$ which satisfy 
\begin{equation*}
\displaystyle r^+_{max}(\varepsilon) < r^+_{ms} <  r^+_{min}(\varepsilon).   
\end{equation*} 
 and the second one does not have solutions.
\item If $\displaystyle \varepsilon = \ve_{min}^- $, then \eqref{e:phi:p} admit two solutions and \eqref{e:phi:m} admits one solution given by $r= r_{ms}^-$.   
\item If $\ve_{min}^- < \displaystyle \varepsilon < 1 $, then both equations \eqref{e:phi:p}  and \eqref{e:phi:m} admit two distinct solutions $\displaystyle \SymbolPrint{r^\pm_{max}(\ve)}$ and $\displaystyle \SymbolPrint{r^\pm_{min}}(\ve)$ which satisfy 
\begin{equation*}
\displaystyle r^\pm_{max}(\ve) < r^\pm_{ms} <  r^\pm_{min}(\ve).   
\end{equation*} 
\item If $ \displaystyle \varepsilon \geq 1 $, then both equations \eqref{e:phi:p} and \eqref{e:phi:m} admit a unique solution  $\displaystyle r^\pm_{max}(\ve, d)$  satisfying 
\begin{equation*}
\displaystyle r^\pm_{max}(\ve) < r^\pm_{ms}.   
\end{equation*} 
\end{enumerate}
\end{lemma}
We  note that there exists exactly two circular orbits with energy $\ve = 1$ which we denote  by $\displaystyle r_{mb}^\pm(d)$. They correspond to the radii of the marginally bound circular direct and retrograde orbits. 
\\ In order to compute $r_{mb}^\pm$, we solve the equation
\begin{equation*}
\Phi_\pm(r_{mb}^\pm) = 1. 
\end{equation*}
\\They are given by: 
\begin{equation}
\label{r:mb}
\SymbolPrint{r^{\pm}_{mb}}(d):= 2 \mp  d +  2\sqrt{1\mp d}.
\end{equation}
Therefore, by Lemma \ref{monotonicite:::}, 
\begin{itemize}
\item If $\ve>1$, then there exist exactly one direct circular orbit $\displaystyle r<r_{mb}^+(d)$ and one retrograde circular orbit $\displaystyle r>r_{mb}^-(d)$. 
\item Otherwise, there exists at most four circular orbits and we refer to Lemma \ref{lemma:circular} for details. 
\end{itemize}
It is useful to rewrite the previous lemma in terms of the implicit functions: 
\begin{lemma}
\label{r(e, d)}
\begin{enumerate}
\item There exists a unique smooth function $\displaystyle r^{+}_{max}: [\ve_{min}^+, \infty[\to]r^+_{ph}(d), r^+_{ms}[$ such that $r^{+}_{max}(\ve)$ solves \eqref{e:phi:p},
\item there exists a unique smooth function $\displaystyle r^{-}_{max}: [\ve_{min}^-, \infty[\to]r^-_{ph}(d), r^+_{ms}[$ such that $r^{-}_{max}(\ve)$ solves \eqref{e:phi:m},
\item there exists a unique smooth function $\displaystyle r^{+}_{min}: [\ve_{min}^+, 1[\to]r^+_{ms}, \infty[$ such that $r^{+}_{min}(\ve)$ solves \eqref{e:phi:p},
\item there exists a unique smooth function $\displaystyle r^{-}_{min}: [\ve_{min}^-, 1[\to ]r^-_{ms}, \infty[$ such that $r^{-}_{min}(\ve)$ solves \eqref{e:phi:m}. 
\end{enumerate}
\end{lemma}
Now, we state the monotonicity properties of  the functions defined in the previous lemma
\begin{lemma}
\label{monot:1}
\begin{enumerate}
\item $\displaystyle r^{+}_{max}$ is monotonically decreasing on $]\ve_{min}^+, \infty[$, 
\item $\displaystyle r^{-}_{max}$ is monotonically decreasing on $]\ve_{min}^-, \infty[$, 
\item $\displaystyle r^{+}_{min}$ is monotonically increasing on $]\ve_{min}^+, 1[$, 
\item $\displaystyle r^{-}_{min}$ is monotonically increasing on $]\ve_{min}^-, 1[$. 
\end{enumerate}
\end{lemma}
\begin{proof}
The proof is straightforward using Lemma \ref{monotonicite:::}. First, we have 
\begin{equation*}
\ve = \Phi_\pm(r^\pm_{max}(\ve)) = \Phi_\pm(r^\pm_{min}(\ve)). 
\end{equation*}
We differentiate with respect to $\ve$: 
\begin{equation*}
1 = \frac{\partial r^\pm_{max}(\ve)}{\partial \ve}\frac{\Phi_\pm(r^\pm_{max}(\ve))}{\partial r} = \frac{\partial r^\pm_{min}(\ve)}{\partial \ve}\frac{\partial \Phi_\pm(r^\pm_{min}(\ve))}{\partial r}. 
\end{equation*}
By monotonicity properties of $\Phi_\pm$, we obtain the desired result.
\end{proof}
\noindent It remains to find conditions on $\ell_z$ for the circular motion. We recall that a circular orbit of radius $\rcc$, and constants of motion $(\ve, \ell_z, q = 0)$ must satisfy \eqref{phi::pm} and \eqref{psi::pm}. Therefore, 
\begin{enumerate}
\item If $d\ell_z > 0$, then $\displaystyle \ve\in[\ve^{+}_{min}(d), \infty[$ and $r$ is given by Lemma \ref{r(e, d)}. Moreover $\ell_z$ must satisfy
\begin{equation}
\ell_z = \Psi_+(r). 
\end{equation} 
If 
\begin{enumerate}
\item $\displaystyle \ve\in[1, \infty[$, then $r = r^+_{max}(\ve)$ and 
\begin{equation}
\ell_z = \Psi_{+}(r^+_{max}(\ve)). 
\end{equation}
\item Otherwise, $\ell_z$ verifies 
\begin{equation}
\ell_z = \Psi_{+}(\tilde r^+_{max}(\ve)) \quad\text{or}\quad \ell_z = \Psi_{+}(r^+_{min}(\ve))
\end{equation}
where $\tilde r^+_{max} \index[Symbols]{$r^+_{max}$ @$\tilde r^+_{max}$}$ denotes the restriction of $r^+_{max}$ on $[\ve^+_{min}(d), 1[$.
\end{enumerate}
\item If $d\ell_z < 0$, then $\ve\in]\ve^{-}_{min}, \infty[$ and $r$ is given by Lemma \ref{r(e, d)}. Moreover $\ell_z$ must satisfy
\begin{equation}
\ell_z := \Psi_-(r). 
\end{equation} 
If 
\begin{enumerate}
\item $\displaystyle \ve\in[1, \infty[$, then $r = r^-_{max}(\ve)$ and 
\begin{equation}
\ell_z = \Psi_{-}(r^-_{max}(\ve)). 
\end{equation}
\item Otherwise, $\ell_z$ verifies 
\begin{equation}
\ell_z = \Psi_{-}(\tilde r^-_{max}(\ve)) \quad\text{or}\quad \ell_z = \Psi_{-}(r^-_{min}(\ve))
\end{equation}
\end{enumerate}
where $\tilde r^-_{max}\index[Symbols]{$r^+_{max}$ @$\tilde r^+_{max}$}$ denotes the restriction of $r^-_{max}$ on $[\ve^{-}_{min}, 1[$
\end{enumerate}
Moreover, note that when $\ve = \ve^{\pm}_{min}$, one has $r = r^\pm_{ms}$ and $\ell_z = \ell^\pm_{min}$.
Now,  we introduce the following functions
\begin{definition}
\label{lb:ub}
We define the functions $\ell_{lb}^\pm:[\ve^{\pm}_{min}(d), \infty[\times]0, 1]\to \mathbb R $ and $\ell_{ub}^\pm:[\ve^{\pm}_{min}(d), 1[\times]0, 1]\to\mathbb R$ by 
\begin{equation}
\label{critical::values}
\SymbolPrint{\ell_{lb}^{\pm}}(\varepsilon, d) := \Psi_\pm(r_{max}^\pm(\varepsilon, d))\quad , \quad \SymbolPrint{\ell_{ub}^{\pm}}(\varepsilon, d) := \Psi_\pm(r_{min}^\pm(\varepsilon, d)). 
\end{equation}
\end{definition}
The proof of Proposition \ref{classif:circ:eq} now follows from lemmas \ref{lemma::22}, \ref{r:min:stable}, \ref{r(e, d)} and Definition \ref{lb:ub}. 
\begin{lemma}
\label{monotonicity:ell}
\begin{enumerate}
\item $\ell_{ub}^+$ is monotonically increasing from $\ell^+_{min}$ to $\infty$ when $\ve$ goes from $\ve_{min}^+$ to $1$.   
\item  $\ell_{lb}^+$ is monotonically increasing from $\ell^+_{min}$ to $\infty$ when $\varepsilon$ goes from $\ve_{min}^+$ to $\infty$.
\item $\ell_{ub}^-$ is monotonically decreasing from $\ell^-_{min}$ to $-\infty$ when $\ve$ goes from $\ve_{min}^-$ to $1$.   
\item  $\ell_{lb}^-$ is monotonically decreasing from $\ell^-_{min}$ to $-\infty$ when $\varepsilon$ goes from $\ve_{min}^-$ to $\infty$.
\end{enumerate}
\end{lemma}
\begin{proof}
The proof is a straightforward application of  Lemma \ref{monot:1} and the monotonicity properties of $\Psi_\pm$. 
\end{proof}
\noindent From now on, the dependence of the above quantities in $d$ will not be written. 
\\ It will be useful in the remaining of our work (see Section \ref{proper:EKlz}) to write $\ve$ in terms of $\ell_z$. Following Lemma \eqref{monotonicity:ell}, we introduce the functions 
\begin{equation}
\label{lower:E}
\varepsilon_s^{\pm}(\ell_z, d):= {\ell_{lb}^{\pm}}(\cdot, d)^{-1}(\ell_z) \quad\text{on}\quad[\ell_{min}^\pm, \infty[ 
\end{equation}
and 
\begin{equation}
\label{upper:E}
\varepsilon_{m}^{\pm}(\ell_z, d):=  {\ell_{ub}^{\pm}}(\cdot, d)^{-1}(\ell_z) \quad\text{on}\quad[\ell_{min}^\pm, \infty[ 
\end{equation}
and we state the following lemma
\begin{lemma}
\label{monotonicity::e}
\begin{itemize}
\item $\displaystyle \varepsilon^+_{s}$ increases monotonically from $\varepsilon^+_{min}$ to $\infty$ when $\ell_z$ grows from $\ell_{min}^+$ to $\infty$. 
\item $\displaystyle \varepsilon^-_{s}$ increases monotonically from $\varepsilon^-_{min}$ to $\infty$ when $\ell_z$ grows from $\ell_{min}^+$ to $\infty$.
\item $\displaystyle \varepsilon^+_{m}$ increases monotonically from $\varepsilon^+_{min}$ to $1$ when $\ell_z$ grows from $\ell_{min}^+$ to $\infty$. 
\item $\displaystyle \varepsilon^-_{m}$ increases monotonically from $\varepsilon^-_{min}$ to $1$ when $\ell_z$ grows from $\ell_{min}^+$ to $\infty$.
\end{itemize}
\end{lemma}  
\begin{remark}
We recover the values found in \cite{jabiri2020static} when $d=0$. In particular:
\begin{equation}
\varepsilon_{min}^{\pm}(0) =  \sqrt{\frac{8}{9}} \quad\text{and}\quad \ell_{min}^\pm(0) = \sqrt{12}. 
\end{equation}
\end{remark}
The necessary conditions of Proposition \ref{classif:circ:eq}  are also sufficient in the following sense
\begin{Propo}
\label{circ::clasifi}
Let $\gamma: I\to\mathcal O$ be a timelike future directed geodesic with constants of motion $(\ve, \ell_z, q)\in\mathbb R\times\mathbb R\times \mathbb R$.  Assume that $(\ve, \ell_z, q)\in \mathcal A_{circ}$. If $\displaystyle (r, \theta)(0) = \left(r_c, \frac{\pi}{2}\right)$ where $r_c$ is given by one of the above cases, then $\gamma$ is circular of radius $r_c$ and confined in the equatorial plane. 
\end{Propo}

\subsubsection{Orbits with constant radial motion}
\label{critical:orbit}
In this section, we are interested in spherical orbits, given by  Definition \ref{classify:categorie}. We recall that spherical orbits are circular if they are confined in the equatorial plane and their classification was obtained in the previous section. In the general case, by Lemma \ref{spherical}, spherical orbits of radius $r_s>r_H$ occur if and only if 
\begin{itemize}
\item $\gamma$ starts at some point $(t, \phi, r_s, \theta)$,
\item  $r_s$ is a double root of $R(\cdot, \ve, \ell_z, q)$,
\end{itemize}
We note that the study of spherical orbits was is included in classical books of general relativity. See for example \cite[Chapter 6]{chandrasekhar1998mathematical} and \cite[Chapter 4]{o2014geometry}. 
\noindent Now, we state the main result of this section
\begin{Propo}
\label{spherical:orbit}
Let $\gamma:\tau\ni I\to \mathcal O$ be a timelike future directed  geodesic and let  $\displaystyle (\ve, \ell_z, q)$ be its associated integrals of motion. Assume that that $\gamma$ is spherical of radius $r_s\in]r_H, \infty[$. Then, 
\begin{equation}
(\ve, \ell_z, q)\in \mathcal A_{spherical}
\end{equation}
defined by \eqref{A:::sph}. 
\end{Propo}

\noindent The remaining of this section is devoted to the proof of Proposition \ref{spherical:orbit}.  We start with the following lemma 

\begin{lemma}
Let $\gamma: I\to \mathcal O\,,\, \tau\mapsto (t(\tau), \phi(\tau), r(\tau), \theta(\tau))$ be a spherical orbit of radius $r_s>r_H$. Let $(\ve, \ell_z, q)\in \mathbb R\times\mathbb R\times\mathbb R$ be its associated integrals of motion. Then,  the quadruplet $(r_s, \ve, \ell_z, q)$ satisfies
\begin{equation}
\begin{aligned}
\ell &= \ell_c^\pm(r_s, \ve^2), \\
d^2(r_s - 1)\eta &= d^2(r_s - 1)\eta_c^\pm(r_s, \ve^2), 
\end{aligned}
\end{equation}
where 
\begin{equation}
\label{l:circular}
\SymbolPrint{\ell} := \frac{\ell_z}{\varepsilon} \quad\text{and}\quad \SymbolPrint{\eta} := \frac{q }{\varepsilon^2},
\end{equation}
\begin{equation}
\label{l:circular}
\ell^\pm_c(r_s, \ve^2) := \frac{1}{d(r_s-1)}\left((r_s^2 - d^2) \pm r_s(r_s^2-2r_s+d^2)\sqrt{1- \frac{1}{\varepsilon^2}(1-\frac{1}{r_s})} \right) 
\end{equation}
and 
\begin{equation}
\label{q:circular}
\begin{aligned}
\eta^\pm_c(r_s, \ve^2) &:= \frac{r_s^3}{r_s-1}\left(4d^2-r_s(r_s-3)^2 \right) + \frac{r_s^2}{\varepsilon^2}\left( r_s(r_s-2)^2 - d^2\right) \\
&- \frac{2r_s^3}{r_s-1}(r_s^2-2r_s+d^2)\left(1 \pm \sqrt{1-\frac{1}{\varepsilon^2}\left(1 - \frac{1}{r_s}\right)} \right).
\end{aligned}
\end{equation}
\end{lemma}

\begin{proof}
By Lemma \ref{spherical}, the quadruplet $(r_s, \ve, \ell_z, q)$ verifies
\begin{align}
\label{root:simple}  
R(r, \varepsilon, \ell_z, q, d) &= 0 \\
\label{root:double}
 \frac{\partial R}{\partial r}(r, \varepsilon, \ell_z, q, d) &= 0. 
\end{align}
Moreover, by Lemma \ref{positive::q}, we require 
\begin{equation}
\label{q:is:pos}
q\geq 0. 
\end{equation}
Solving the equations \eqref{root:simple} and \eqref{root:double} simultaneously eliminates two of the four unknowns $r_s, \ve, \ell_z$ and $q$.  We introduce the following change of variables 
\begin{equation}
\label{l:circular}
\SymbolPrint{\ell} = \frac{\ell_z}{\varepsilon} \quad\text{and}\quad \SymbolPrint{\eta} = \frac{q }{\varepsilon^2} 
\end{equation}
After developing the above equations, we obtain
\begin{align*}
r_s^4 - \frac{r_s^2}{\varepsilon^2}(r_s^2-2r_s+d^2)-\eta(r_s^2-2r_s+d^2)+(d^2-\ell^2)r_s^2+2(d^2+\ell^2-2d\ell)r_s &= 0 ,\\
 4r_s^3 - \frac{2r_s}{\varepsilon^2}(r_s^2-2r_s+d^2)-\frac{2r_s^2}{\varepsilon^2}(r_s-1)-2\eta(r_s-1)+2r_s(d^2-\ell^2)+2(d^2+\ell^2-2d\ell) &= 0.
\end{align*}
The linear combinations $\displaystyle \frac{1}{\varepsilon^2}\left(r_s\frac{\partial R}{\partial r}(r_s) - R(r_s)\right)$ and $\displaystyle \frac{1}{\varepsilon^2}\left(\frac{r_s}{2}\frac{\partial R}{\partial r}(r_s) - R(r_s)\right)$ yield
\begin{align}
\label{equ1}
3r_s^4 - \frac{r_s^2}{\varepsilon^2}(3r_s^2-4r_s+d^2) + r_s^2d^2 - \eta(r_s^2 - d^2) &= r_s^2\ell^2 , \\
\label{equ2}
r_s^4 - \frac{r_s^3}{\varepsilon^2}(r_s-1) - d^2 r_s + \eta(d^2 - r_s) &= r_s(\ell^2 - 2d\ell). 
\end{align}
We solve for $\ell$ first by eliminating $\eta$. Hence, we multiply the first equation by $(d^2- r_s)$ and the second one by $(r_s^2 - d^2)$ and we obtain a second order polynomial in $\ell$. Straightforward computations yield the following solutions for $\ell$:
\begin{equation}
\label{l:circular}
\ell = \frac{1}{d(r_s-1)}\left((r_s^2 - d^2) \pm r_s(r_s^2-2r_s+d^2)\sqrt{1- \frac{1}{\varepsilon^2}(1-\frac{1}{r_s})} \right) =: \ell^\pm_c(r_s, \ve^2). 
\end{equation}
Now, we replace $\ell$ with its expression in \eqref{equ1} in order to obtain 
\begin{equation}
\label{q:circular}
\begin{aligned}
d^2(r_s-1)\eta &= \frac{r_s^3}{r_s-1}\left(4d^2-r_s(r_s-3)^2 \right) + \frac{r_s^2}{\varepsilon^2}\left( r_s(r_s-2)^2 - d^2\right) - \frac{2r_s^3}{r_s-1}(r_s^2-2r_s+d^2) \\
&\left(1 \pm \sqrt{1-\frac{1}{\varepsilon^2}\left(1 - \frac{1}{r_s}\right)} \right) \\
&=: d^2(r_s-1)\eta^\pm_c(r_s, \ve^2). 
\end{aligned}
\end{equation}
\end{proof}
\noindent We have shown so far that \eqref{root:simple} and \eqref{root:double} imply \eqref{l:circular} and \eqref{q:circular}. Tedious but straightforward computations imply that if $(\ell, \eta) = (\ell^{\pm}_c(r_s, \ve^2), \eta^{\pm}_c(r_s, \ve^2))$, then the quadruplets $(r_s, \ve^2, \ell^{\pm}_c(r_s, \ve^2),$  $\eta^{\pm}_c(r_s, \ve^2))$ \eqref{root:simple} and \eqref{root:double} solve the equations \eqref{root:simple} and \eqref{root:double}. 
\\ In the following, we will determine the set of admissible parameters $(r_s, \ve^2)$ so that solutions of \eqref{root:simple}-\eqref{root:double} are given by   a two parameter family  indexed by $(r_s, \ve^2)$. 
\begin{lemma}
\label{properties:eta}
\begin{enumerate}
\item $\eta^\pm_c$ and $\ell^\pm_c$ are defined on the domain
\begin{equation*}
\begin{aligned}
D_s&:= \left\{ (r, \ve)\in]r_H, \infty[\times[1, \infty[ \right\}\sqcup \left\{ (r, \ve)\in]r_H, \infty[\times]0, 1[ \;:\; r> (1-\varepsilon^2)^{-1}\right\} \\
&:= D_s^{\geq 1}\sqcup D_s^{\leq 1}. 
\end{aligned}
\end{equation*}
\item We have the following asymptotics for $\eta^\pm_c(\cdot ,\ve^2)$: 
\begin{itemize}
\item $ \forall \ve^2>0$
\begin{equation}
\label{rh:limit}
\lim_{r\to r_H} \eta(r, \varepsilon^2) < 0
\end{equation}
\item and 
\begin{equation}
\label{infty:limit}
\begin{aligned}
\lim_{r\to(1-\varepsilon^2)^{-1}} \eta(r, \varepsilon^2) <0  &\quad\text{if}\quad \ve^2<1, \\
 \lim_{r\to\infty} \eta(r, \varepsilon^2) <0 &\quad\text{if}\quad \ve^2\geq1, \\
\end{aligned}
\end{equation}
\end{itemize}
\item $\eta_c^+$ is negative on $D_s$.  
\end{enumerate}
\end{lemma}

\begin{proof}
\begin{enumerate}
\item The first point is straightforward. 
\item For the second point, we compute
\begin{align*}
\eta_c(r_H, \ve^2) &= \frac{r_H^3}{r_H-1}\left(4d^2-r_H(r_H - 3)^2 \right) + \frac{r_H^2}{\varepsilon^2}\left( r_H(r_H - 2)^2 - d^2\right) \\
&= \frac{r_H^2}{r_H - 1}\left(r_H\left(4d^2-r_H(r_H - 3)^2 \right) + \frac{(r_H-1)}{\varepsilon^2}\left( r_H(r_H - 2)^2 - d^2\right) \right) \\
&= \frac{r_H^2}{r_H - 1} \left( I + II \right),
\end{align*}
where 
\begin{align*}
I &:=  r_H\left(4d^2-r_H(r_H - 3)^2 \right) \\
&= (1 + \sqrt{1 - d^2})(4d^2 - (1 + \sqrt{1-d^2})(\sqrt{1-d^2} - 2)^2) \\
&= (1 + \sqrt{1 - d^2})^2(d^2 - 1) < 0
\end{align*}
and 
\begin{align*}
II &:= \frac{(r_H-1)}{\varepsilon^2}\left( r_H(r_H - 2)^2 - d^2\right) \\
&= \frac{\sqrt{1-d^2}}{\ve^2}\left( (1 + \sqrt{1 - d^2})(\sqrt{1 -d^2} - 1)^2 -d^2)\right) \\
&= -\frac{d^2}{\ve^2}(1-d^2). 
\end{align*}
Therefore, 
\begin{equation*}
\eta_c(r_H, \ve^2) < 0. 
\end{equation*}
Now, we compute $\forall \ve^2<1\,,\,\forall d\in]0, 1[$ 
\begin{equation*}
\eta_c((1-\ve^2)^{-1}, \ve^2) = \left(\left((1-\ve^2)^{-1} - 1 \right)^{-1}\right) \left(a(\ve^2, d) + b(\ve^2, d) + c(\ve^2, d)\right)
\end{equation*}
where 
\begin{align*}
a(\ve^2, d)&:= \frac{1}{(1-\ve^2)^3}\left(4d^2 - \frac{1}{1 - \ve^2}\left(\frac{1}{1-\ve^2} - 3\right)^2 \right)  \\
&=   \frac{1}{(1 - \ve^2)^6}(4d^2(1-\ve^2)^3 - (2-3\ve^2)^2), \\
b(\ve^2, d)&:= \frac{1}{\ve^2}\frac{1}{(1-\ve^2)^2}\left(\frac{1}{1-\ve^2} - 1\right)\left(\frac{1}{1 - \ve^2}\left(\frac{1}{1 - \ve^2} - 2\right)^2 - d^2 \right)  \\
&= \frac{1}{(1 - \ve^2)^6}\left((2\ve^2 - 1)^2 - d^2(1 - \ve^2)^3 \right), \\
c(\ve^2, d)&:= -\frac{2}{(1-\ve^2)^3}\left(\frac{1}{(1-\ve^2)^2} - \frac{2}{1 - \ve^2} + d^2\right)  \\
&= -\frac{2}{(1-\ve^2)^5}\left((2\ve^2 - 1) + d^2(1 - \ve^2)^2 \right). 
\end{align*}
Straightforward computations imply
\begin{equation*}
\eta_c((1-\ve^2)^{-1}, \ve^2) = \frac{1}{(1 - \ve^2)^4}\left( d^2(1 - \ve^2)- 1\right) < 0. 
\end{equation*}
\item Let $(r, \ve)\in D_s$. Then, $\forall r \in ]r_H, \infty[$, $\ve$ must verify 
\begin{equation*}
\ve ^2 \geq \left(1 - \frac{1}{r} \right). 
\end{equation*} 
Now let $r \in ]r_H, \infty[$ and consider the function $\tilde \eta_c^+(r, \cdot)$ defined on $\left]\sqrt{\left(1 - \frac{1}{r} \right)} , \infty\right[$ by 
\begin{equation*}
\tilde \eta_c^+(r, \ve) := \frac{r-1}{r^2}\eta_c^+(r, \ve). 
\end{equation*}
It is easy to see that the terms $r(r-2)^2 - d^2$ and $r(r^2 - 2r + d^2)$ are positive since $r>1 + \sqrt{1 - d^2}$ and $0<d^2<1$. Therefore $\tilde \eta_c^+(r, \cdot)$ is monotonically decreasing on $\left]\sqrt{\left(1 - \frac{1}{r} \right)} , \infty\right[$. Hence 
\begin{equation*}
\forall \ve\geq\sqrt{\left(1 - \frac{1}{r} \right)} \;:\; \eta_c^+(r, \ve) <  \eta_c^+(r, \sqrt{\left(1 - \frac{1}{r} \right)}). 
\end{equation*}
Straightforward computations imply 
\begin{equation*}
\eta_c^+(r, \sqrt{\left(1 - \frac{1}{r} \right)}) = -r(r - d^2) < 0.
\end{equation*}
Therefore, $\tilde \eta_c^+(r, \cdot)$ is negative and so is $\eta_c^+$. 
\end{enumerate}
\end{proof}
\noindent Since $q\geq 0$, $\ell$ and $\eta$ must equal  $\ell^-_c$ and $\eta^-_c$. From now on, we omit the sign from the latter quantities so that they are simply denoted by  $\ell_c$ and $\eta_c$. 
\begin{remark}
Note that if  $(r, \ve)\in D_s^{\leq 1}$ then we have a lower bound on $\ve$. Indeed, $\forall r>r_H$, $\ve^2$ satisfies 
\begin{equation*}
\ve ^2 \geq \left(1 - \frac{1}{r} \right) > \left(1 - \frac{1}{r_H} \right) = \frac{\sqrt{1-d^2}}{1 + \sqrt{1-d^2}}. 
\end{equation*} 
\end{remark}
By the second point of the previous lemma and the positivity of $q$, not all values of $r$ are allowed. In order to determine the allowed region for $r$, we first look at the equation
\begin{equation}
\label{eta:vanish}
\eta_c(r, \ve^2) = 0. 
\end{equation}
The latter is equivalent to a vanishing Carter constant. This case corresponds to circular orbits confined in the equatorial plane. Hence, by Proposition \ref{classif:circ:eq}, $(\ve, \ell_z, q)\in \mathcal A_{circ}$ and $r_s$ is given by one of the cases in the latter proposition. More precisely, one of the following cases is possible:
\begin{enumerate}
\item If $\ve< \ve^+_{min}$, then the equation \eqref{eta:vanish} does not admit solutions. 
\item If $\ve = \ve^+_{min}$, then the equation \eqref{eta:vanish} admits a unique solution given by $r_{ms}^+$. 
\item If $ \ve^+_{min} < \ve < \ve^-_{min}$, then the equation \eqref{eta:vanish} admits two solutions $\tilde r_{max}^+(\ve)$ and $r_{min}^+(\ve)$. 
\item If  $\ve = \ve^-_{min}$, then the equation \eqref{eta:vanish} admits three solutions $\tilde r_{max}^+(\ve)$, $r_{min}^+(\ve)$ and $r_{ms}^-$. 
\item If  $\ve^-_{min} < \ve < 1$, then the equation \eqref{eta:vanish} admits four solutions $\tilde r_{max}^+(\ve)$, $r_{min}^+(\ve)$, $\tilde r_{max}^-(\ve)$ and $r_{min}^-(\ve)$.
\item If $\ve\geq 1$,  then the equation \eqref{eta:vanish} admits two solutions $r_{max}^+(\ve)$ and $r_{max}^-(\ve)$. 
\end{enumerate}
Moreover, by the asymptotics of $\eta_c$ given by the second point of Lemma \ref{properties:eta}, the allowed regions for $r$ so that $\eta_c$ is positive are given by the following lemma
\begin{lemma}
Let $(r, \ve)\in D_s$. Then $\eta_c$ is positive ($\geq 0$) if and only if 
\begin{equation}
(r, \ve)\in \overline D_s:= \overline D_s^{\leq 1} \sqcup \overline D_s^{\geq 1}
\end{equation}
where 
\begin{equation}
\overline D_s^{\geq 1} := \left\{(r, \ve)\in D_s^{\geq 1} \;:\; r\in [r_{max}^+(\ve), r_{max}^-(\ve)] \right\}. 
\end{equation}
and 
\begin{equation}
\begin{aligned}
\overline D_s^{\leq 1} &:=  \left\{(r_{ms}^+, \ve^+_{min})\right\}\sqcup \left\{(r, \ve)\in D_s^{\leq 1}(d)  \;:\; \ve \in]\ve^+_{min}, \ve^-_{min}]\;\text{and}\; r\in [\tilde r_{max}^+(\ve), \tilde r_{min}^+(\ve)]\right\} \\ 
&\sqcup \left\{(r, \ve)\in D_s^{\geq 1}(d) \;:\; \ve\in]\ve^-_{min}, 1[\;\text{and}\; r\in [\tilde r_{max}^+(\ve), \tilde r_{max}^-(\ve)]\sqcup [r_{min}^-(\ve), r_{min}^+(\ve)] \right\}
\end{aligned}
\end{equation}

\end{lemma}  

\begin{proof}
The proof is straightforward using the above six cases and the asymptotics of $\eta_c(\cdot, \ve^2)$. 
\end{proof}
\begin{lemma}[Critical points of $\eta_c(\cdot, \ve^2)$ when $\ve^2<1$]
\label{critical:eta:c}
Let $\ve^+_{min}< \ve < 1$.Then,  the critical points of $\eta_c(\cdot, \ve^2)$ are:
\begin{enumerate}
\item If $\displaystyle \ve<\sqrt{\frac{8}{9}}$,  the critical points of $\eta_c(\cdot, \ve^2)$ are the critical points of $\ell_c(\cdot, \ve^2)$.
\item Otherwise,
\begin{itemize}
\item the critical points of $\ell_c(\cdot, \ve^2)$,  
\item the points $\overset{\circ}{r}_{\pm}(\ve)$: 
\begin{equation*}
\overset{\circ}{r}_{\pm}(\ve) := \frac{-4+3\ve^2\pm\ve\sqrt{-8 + 9\ve^2}}{2(\ve^2-1)}. 
\end{equation*}
\end{itemize}
\end{enumerate}
\end{lemma}
\begin{proof}
Let $(r,\ve)\in \overline D_s$ such that $r$  is a critical point for $\eta_c(\cdot, \ve^2)$. We have 
\begin{equation*}
R(r, \ve, \ve\ell_c(r, \ve^2), \ve\eta_c(r, \ve^2)) = 0 \quad\text{and}\quad \frac{\partial R}{\partial r}(r, \ve, \ve\ell_c(r, \ve^2), \ve\eta_c(r, \ve^2)) = 0.
\end{equation*} 
We differentiate the first equation with respect to $r$ to obtain: 
\begin{equation*}
\ve\partial_r\ell_c(r, \ve^2)\frac{\partial R}{\partial \ell_z}(r, \ve, \ve\ell_c(r, \ve^2), \ve\eta_c(r, \ve^2))  + \ve^2\partial_r\eta_c(r, \ve^2)\frac{\partial R}{\partial q}(r, \ve, \ve\ell_c(r, \ve^2), \ve\eta_c(r, \ve^2)) = 0. 
\end{equation*} 
Therefore, 
\begin{equation*}
\partial_r\ell_c(r, \ve^2) = 0\quad\text{or}\quad \frac{\partial R}{\partial \ell_z}(r, \ve, \ve\ell_c(r, \ve^2), \ve\eta_c(r, \ve^2)). 
\end{equation*}
We use the equations \eqref{root:simple}-\eqref{root:double} in order to obtain
\begin{equation*}
\frac{\partial R}{\partial \ell_z}(r, \ve, \ve\ell_c(r, \ve^2), \ve\eta_c(r, \ve^2)) = -2(2d\ve + \ve\ell_c(r, \ve^2)(r-2))r. 
\end{equation*}
Hence $r$ is a solution of the equation
\begin{equation*}
\ell_c(r, \ve^2)(r-2) = -2d. 
\end{equation*}
Now we plug the expression of $\ell_c$ in the latter equation to obtain
\begin{equation*}
\ve((r^2- d^2)(r - 2) + 2d^2(r - 1)) - \sqrt{r}\Delta\sqrt{r(\ve^2 - 1) + 1}(r - 2) = 0,
\end{equation*}
which is equivalent to 
\begin{equation*}
r\ve\Delta - \sqrt{r}(r - 2)\Delta\sqrt{r(\ve^2 - 1) + 1} = 0. 
\end{equation*}
Therefore, $r$ satisfies the equation
\begin{equation*}
\sqrt r\ve = (r - 2)\sqrt{r(\ve^2 - 1) + 1}. 
\end{equation*}
We compute 
\begin{align*}
((r - 2)\sqrt{r(\ve^2 - 1) + 1})^2 - (\sqrt r\ve)^2&=  4 - 8r + 4r\ve^2 + 5 r^2 - 4 \ve^2 r^2 - r^3 + \ve^2 r^3 - r\ve^2\\
&= (r - 1)((\ve^2 - 1)r^2 + (4 - 3\ve^2)r - 4). 
\end{align*}
Since $r>r_H>1$, $r$ is solution of the second degree polynomial
\begin{equation*}
((\ve^2 - 1)r^2 + (4 - 3\ve^2)r - 4).
\end{equation*}
Straightforward computations imply that $r$ is given by $\overset{\circ}{r}_{\pm}(\ve)$. 
\end{proof}
Now, we study the function $\ell_c$ on the domain $\overline D_s$. We recall that $\ell_c$ is defined by 
\begin{equation}
\ell_c(r, \ve^2) = \frac{1}{d(r-1)}\left((r^2 - d^2) - r(r^2-2r+d^2)\sqrt{1- \frac{1}{\varepsilon^2}\left(1-\frac{1}{r}\right)} \right). 
\end{equation}
\noindent In the following, it is useful to introduce the following function $\xi: \overline D_s\to\mathbb R$ defined by 
\begin{equation}
\label{xi:def}
\xi(r, \ve^2):= \eta_c(r, \ve^2) + (\ell_c(r, \ve^2) - d)^2. 
\end{equation}
We state the following lemma on the critical points of $\ell_c(\cdot, \ve^2)$ and $\xi(\cdot, \ve^2)$ when $\ve^+_{min}<\ve^2<1$.
\begin{lemma}
\label{critical:l:c}
Let $\ve^+_{min} < \ve < 1$. Then, the critical points of $\xi(\cdot, \ve^2)$ are the critical points of $\ell_c(\cdot, \ve^2)$. Moreover, there exists a unique $r_m(\ve)\in]\tilde r_{max}^+(\ve), r_{min}^+(\ve)[$ such that 
\begin{equation}
\partial_r\ell_c(r_m(\ve), \ve^2) = \partial_r\xi(r_m(\ve), \ve^2) = 0.
\end{equation}
Moreover, $\ell_c(r_m(\ve), \ve^2)$ is a global minimum for $\ell_c(\cdot, \ve^2)$ and $\xi(r_m(\ve), \ve^2)$ is a global maximum for $\xi(\cdot, \ve^2)$. 
\end{lemma}
\begin{proof}
Let $\displaystyle \ve^+_{min}<\ve<1$. 
\begin{enumerate}
\item  Let $r$ be a critical point for $\ell_c(\cdot, \ve^2)$. Then, $\partial_r\ell_c(r, \ve^2) = 0$. Moreover, by Lemma \ref{critical:eta:c}, $r$ is also a critical point for $\eta_c(\cdot, \ve^2)$. Therefore,
\begin{equation*}
\partial_r\xi(r, \ve^2) = \partial_r\eta_c(r, \ve^2) + 2(\ell_c(r, \ve^2) -d)\partial_r\ell_c(r, \ve^2)  = 0. 
\end{equation*}
\noindent Now, let $r$ be a critical point for $\xi(\cdot, \ve^2)$. Then,  $(r, \ve^2)$ satisfies
\begin{equation}
\label{double:root:critic}
R(r, \ve, \ve\ell_c(r, \ve^2), \ve^2(\xi(r, \ve^2) - (d-\ell_c(r, \ve^2))^2)) = 0.
\end{equation}
We differentiate the latter with respect to $r$ and we use the fact that $r$ is double root for $R$ to obtain 
\begin{equation*}
\begin{aligned}
&\partial_r\ell_c(r, \ve^2)\frac{\partial R}{\partial\ell_z}(r, \ve, \ve\ell_c(r, \ve^2), \ve^2(\xi(r, \ve^2) - (d-\ell_c(r, \ve^2))^2)) + \\
&(\partial_r\xi(r, \ve^2) - 2(\ell_c(r, \ve^2) -d)\partial_r\ell_c(r, \ve^2))\frac{\partial R}{\partial q}(r, \ve, \ve\ell_c(r, \ve^2), \ve^2(\xi(r, \ve^2) - (d-\ell_c(r, \ve^2))^2)) = 0. 
\end{aligned}
\end{equation*}
Therefore, 
\begin{equation*}
\partial_r\ell_c(r, \ve^2)\frac{\partial R}{\partial\ell_z}(r, \ve, \ve\ell_c(r, \ve^2), \ve^2(\xi(r, \ve^2) - (d-\ell_c(r, \ve^2))^2)) +2(\ell_c(r, \ve^2) -d)\partial_r\ell_c(r, \ve^2)\Delta(r)= 0. 
\end{equation*}
If $\partial_r\ell_c(r, \ve^2) \neq 0$, then 
\begin{equation*}
\frac{\partial R}{\partial\ell_z}(r, \ve, \ve\ell_c(r, \ve^2), \ve^2(\xi(r, \ve^2) - (d-\ell_c(r, \ve^2))^2)) +2(\ell_c(r, \ve^2) -d)\Delta(r)= 0. 
\end{equation*}
We plug the expression of the $\ell_z$-derivative of $R$ in the above equation in order to obtain
\begin{equation*}
\ell_c(r, \ve^2)r(r-2) + 2rd = \Delta(r)\ell_c(r, \ve^2) -d\Delta(r). 
\end{equation*}
Therefore, $\ell_c(r, \ve^2)$ verifies 
\begin{equation*}
\ell_c(r, \ve^2) = \frac{r^2}{d} + d. 
\end{equation*}
Furthermore, $(r, \ve)$ must satisfy \eqref{double:root:critic}. We plug the expression of $\ell_c(r, \ve^2)$ in the latter equation to obtain
\begin{equation*}
-\Delta(r)(r^2 + \xi(r, \ve^2)) = 0. 
\end{equation*}
Therefore, 
\begin{equation*}
\xi(r, \ve^2) = -r^2 <0, 
\end{equation*}
which is not possible since $\xi$ is positive. Therefore, $\partial_c\ell_c(r, \ve^2) = 0$.  
\item For the uniqueness, we refer to Appendix II of \cite{fayos2008geometrical} for a detailed proof. 
\end{enumerate}
\end{proof}
\begin{lemma}
\label{critical:ellc:1}
Let $\ve^2>1$. Then $\ell_c(\cdot, \ve^2)$ does not have critical points and $\eta_c(\cdot, \ve^2)$ admits a unique extremum (maximum) at $r^{\geq 1}_{max}(\ve)\in]r_{max}^+(\ve), r_{max}^-(\ve)[$.  
\end{lemma}
\begin{proof}
\begin{enumerate}
\item Let $(r,\ve)\in \overline D_s^{\geq 1}$ such that $r = r^c(\ve^2)$  is a critical point for $\ell_c(\cdot, \ve^2)$. We have 
\begin{equation*}
R(r, \ve, \ve\ell_c(r, \ve^2), \ve\eta_c(r, \ve^2)) = 0 \quad\text{and}\quad \frac{\partial R}{\partial r}(r, \ve, \ve\ell_c(r, \ve^2), \ve\eta_c(r, \ve^2)) = 0.
\end{equation*} 
We differentiate the first equation with respect to $r$ to obtain: 
\begin{equation*}
\ve\partial_r\ell_c(r, \ve^2)\frac{\partial R}{\partial \ell_z}(r, \ve, \ve\ell_c(r, \ve^2), \ve\eta_c(r, \ve^2))  + \ve^2\partial_r\eta_c(r, \ve^2)\frac{\partial R}{\partial q}(r, \ve, \ve\ell_c(r, \ve^2), \ve\eta_c(r, \ve^2)) = 0. 
\end{equation*} 
Now we differentiate the second equation with respect to $r$ to obtain:
\begin{equation*}
\begin{aligned}
&\frac{\partial^2R}{\partial r^2}(r, \ve, \ve\ell_c(r, \ve^2), \ve\eta_c(r, \ve^2)) + \ve\partial_r\ell_c(r, \ve^2)\frac{\partial^2 R}{\partial r\partial \ell_z}(r, \ve, \ve\ell_c(r, \ve^2), \ve\eta_c(r, \ve^2))   \\
&+ \ve^2\partial_r\eta_c(r, \ve^2)\frac{\partial^2 R}{\partial r\partial q}(r, \ve, \ve\ell_c(r, \ve^2), \ve\eta_c(r, \ve^2)) = 0. 
\end{aligned}
\end{equation*}
Therefore, $r$ is a critical point of $\eta_c(\cdot, \ve^2)$ since the derivative of $R$ with respect to $q$ does not vanish. Hence, $r$ verifies: 
 \begin{equation}
 \label{second:der}
 \frac{\partial^2R}{\partial r^2}(r, \ve, \ve\ell_c(r, \ve^2), \ve\eta_c(r, \ve^2)) = 0. 
 \end{equation}
$r\in[r^+_{max}(\ve), r^+_{min}(\ve)]$.  Therefore, $r^c(\ve)$ is a triple root of the polynomial $R$ with parameters $(\ve, \ve\ell_{c}(\ve), \ve^2\eta_c(r^c(\ve), \ve^2))$. This cannot happen because of Lemma \ref{roots:of:RR}. 
\item Since $\ell_c(\cdot, \ve^2)$ does not have critical points, $\eta_c(\cdot, \ve^2)$ admits two critical points, given by $\overset{\circ}{r}_{\pm}(\ve)$. A study of the function $\overset{\circ}{r}_{-}$ on $[1, \infty[$ shows that this function is always negative. Since we are looking for critical points in the region $]r_{max}^+(\ve), r_{max}^-(\ve)[\subset]0, \infty[$, $\ell_c(\cdot, \ve^2)$ admits a unique critical point given by $\overset{\circ}{r}_{+}(\ve)\in]0, \infty[$. It remains to show that it lies in the region $]r_{max}^+(\ve), r_{max}^-(\ve)[$. This is straightforward by the mean value theorem. Indeed, $\eta_c(\cdot,\ve^2)$ vanishes at $r_{max}^+(\ve)$ and $r_{max}^-(\ve)$. 
\end{enumerate}
\end{proof}
\noindent Finally, we introduce the following notations: 
\begin{equation}
\label{tilde:ell:min}
\tilde \ell_{min} \index[Symbols]{\ell_{min} @ $\ell_{min}$}(\ve):= \ve \ell_c(r_{m}(\ve) , \ve^2) \quad\text{for}\quad{\ve\in]\ve^+_{min}, 1[}
\end{equation}
and 
\begin{equation}
\label{tilde:ell:min}
\SymbolPrint{q_{max}}(\ve):= \ve^2 \eta_c(r^{\geq 1}_{max}(\ve) , \ve^2) \quad\text{for}\quad \ve^2>1. 
\end{equation}

\begin{lemma}[Properties of $\tilde \ell_{min}$ and $r_m(\ve)$]
\label{l:min:r:m}
Let $\ve\in[\ve_{min}^+, 1[$. 
Then, 
\begin{itemize}
\item $r_m$ is monotonically increasing from $r_{ms}^+$ to $+\infty$ when $\ve$ grows monotonically from $\ve_{min}^+$ to $1$. 
\item $\displaystyle \tilde\ell_{min}$ is monotonically decreasing from $\ell_{min}^+$ to $-\infty$ when $\ve$ grows monotonically from $\ve_{min}^+$ to $1$.
\end{itemize}
In particular, if $\ve = \ve^{-}_{min}$, then 
\begin{equation*}
r_m(\ve) = r_{ms}^- \quad\text{and}\quad \tilde\ell_{min}(\ve) = \ell_{min}^-.
\end{equation*}
Moreover, if $\ve\in]\ve_{min}^-, 1[$, then $\displaystyle r_m(\ve)\in]r_{max}^-(\ve), r_{min}^-(\ve)[$.  
\end{lemma}
\begin{proof}
This follows from the definitions and the implicit function theorem. 
\end{proof}
\noindent Based on the above lemmas, we derive the monotonicity properties of $\ell_c(\cdot, \ve^2)$. 
\begin{lemma}
\label{monotonicity:lc}
\begin{enumerate}
\item If $\displaystyle (r, \ve) = (r_{ms}^+, \ve^+_{min})$, then 
\begin{equation}
\ell_c(r, \ve^2) = \frac{\ell_{min}^+}{\ve^+_{min}}. 
\end{equation}
\item If $\ve_{min}^+<\ve\leq\ve_{min}^-$, then the restriction of $\ell_c(\cdot, \ve^2)$ on the domain $[\tilde r^+_{max}(\ve), r^+_{min}(\ve)]$ has the following properties: $\ell_c(\cdot, \ve^2)$ is monotonically decreasing  on $]\tilde r^+_{max}(\ve), r_{m}(\ve)[$ from $\displaystyle \frac{\ell_{lb}^+(\ve)}{\ve}$ to $\displaystyle \frac{\tilde\ell_{min}(\ve)}{\ve}<\frac{\ell_{min}^+}{\ve}$ and monotonically increasing on $]r_{m}(\ve), r^+_{min}(\ve)[$  from $\displaystyle \frac{\tilde \ell_{min}(\ve)}{\ve}$ to $\displaystyle \frac{\ell_{ub}^+(\ve)}{\ve}$. Therefore, there exists a unique $\SymbolPrint{r_{lb}}(\ve)\in]r_{m}(\ve), r^+_{min}(\ve)[$ such that $\displaystyle \ell_c(r_{lb}(\ve), \ve^2)= \frac{\ell_{lb}^+(\ve)}{\ve}$.
\item If $\ve_{min}^-<\ve< 1$,  then the restriction of $\ell_c(\cdot, \ve^2)$ on the domain $[\tilde r^+_{max}(\ve), \tilde r^-_{max}(\ve)]\sqcup [r^-_{min}(\ve), r^-_{min}(\ve)]$ has the following properties:
\begin{itemize}
\item $\ell_c(\cdot, \ve^2)$ is monotonically decreasing  on $]\tilde r^+_{max}(\ve), \tilde r^-_{max}(\ve)[$ from  $\displaystyle \frac{\ell_{lb}^+(\ve)}{\ve}$ to $\displaystyle \frac{\ell_{lb}^-(\ve)}{\ve}$.
\item $\ell_c(\cdot, \ve^2)$ is monotonically increasing on $]r^-_{min}(\ve), r^+_{min}(\ve)[$ from $\displaystyle \frac{\ell_{ub}^-(\ve)}{\ve}$ to  $\displaystyle \frac{\ell_{ub}^+(\ve)}{\ve}$
\end{itemize}
\item If $\displaystyle (r, \ve)\in \overline D_s^{\geq 1}$, then $\forall \ve\geq 1$, the restriction of $\ell_c(\cdot, \ve^2)$ on the domain $[r^+_{max}(\ve), r^-_{max}(\ve)]$ satisfies 
\begin{itemize}
\item $\ell_c(\cdot, \ve^2)$ is monotonically decreasing  on $]r^+_{max}(\ve), r^-_{max}(\ve)[$ from  $\displaystyle \frac{\ell_{lb}^+(\ve)}{\ve}$ to $\displaystyle \frac{\ell_{lb}^-(\ve)}{\ve}$.
\end{itemize}
\end{enumerate}
\end{lemma}
\begin{proof}
 \begin{enumerate}
\item If $\ve^+_{min}\leq \ve<1$,   then we use Lemmas \ref{critical:l:c} and \ref{l:min:r:m} to obtain the result. 
\item If $\ve^2\geq 1$, then we use Lemma \ref{critical:ellc:1} to obtain the result. 
\end{enumerate}
\end{proof}
\noindent Now, instead of writing the solutions of the system of equations \eqref{root:simple}-\eqref{root:double} as a two parameter family $(\ell_c, \eta_c)(r, \ve^2)$ indexed by $(r, \ve^2)\in \overline D_s$, we will write the solutions of  \eqref{root:simple}-\eqref{root:double} as a two-parameter family $(r, \eta_c)(\ve^2, \ell_z)$ indexed by $(\ve^2, \ell_z)\in \overline D_{sph}$ where $\overline D_{sph}$ is the set of admissible values of $(\ve^2, \ell_z)$ which will be determined. More precisely, we will fix $\ve$ so that  $\ell_z$ becomes a function of  $r$ only, which will be inverted. This will allow us to  write $r$ in terms of $\ell_c$.  We state the following lemma 
\begin{lemma}
\label{lemma:29}
Let $d\in]0, 1]$ and let $(\ve, \ell_z)\in ]0, \infty[\times \mathbb R$. The system of equations $\eqref{root:simple}-\eqref{root:double}$ admits solutions $(r, q)\in]r_H, \infty[\times[0, \infty[ $ if and only if one of the following cases occur. 
\begin{enumerate}
\item $\ve = \ve^+_{min}$ and  $\ell_z = \ell_{min}^+$. In this case, 
\begin{equation}
(r, q) = (r_{ms}^+, 0). 
\end{equation}
\item $\ve^+_{min}<\ve\leq\ve^-_{min}$ and $\ell_z\in[\tilde \ell_{min}(\ve), \ell_{ub}^+(\ve)]$. 
\begin{itemize}
\item If $\ell_z\in[\tilde \ell_{min}(\ve), \ell_{lb}^+(\ve)]$, there exist two solutions $(r_s^1, q^1_s)(\ve, \ell_z)$ and $(r_s^2, q^2_s)(\ve, \ell_z)$ such that 
\begin{itemize}
\item $\SymbolPrint{r^1_s}(\ve, \ell_z)$ lies in the region $[r_{max}^+(\ve), \SymbolPrint{r_{m}}(\ve)]$ and is given by 
\begin{equation}
r^1_s(\ve, \ell_z) = \ell_c^{-1}\left(\frac{\ell_z}{\ve}, \ve^2\right),
\end{equation}
where $\ell_c^{-1}$ is the inverse of the restriction of $\ell_c(\cdot, \ve^2)$ on $]r_{max}^+(\ve), r_{m}(\ve)[$.
\item $\SymbolPrint{r^2_s}(\ve, \ell_z)$ lies in the region $[r_{m}(\ve), r_{lb}(\ve)]$ and is given by 
\begin{equation}
r^2_s(\ve, \ell_z) = \ell_c^{-1}\left(\frac{\ell_z}{\ve}, \ve^2\right),
\end{equation}
where $\ell_c^{-1}$ is the inverse of the restriction of $\ell_c(\cdot, \ve^2)$ on $]r_{m}(\ve), r_{lb}(\ve)[$.
\item $\SymbolPrint{q^i_s}(\ve, \ell_z)$ are given by 
\begin{equation*}
q^i_s(\ve, \ell_z) = \ve^2\eta_c(r^i_s(\ve, \ell_z), \ve^2).
\end{equation*}
\end{itemize}

\item If $\ell_z\in[\ell_{lb}^+(\ve), \ell_{ub}^+(\ve)]$, there exists a unique $(\tilde r^+ \index[Symbols]{$r^+$ @$\tilde r^+$}, \tilde q^+ \index[Symbols]{$q^+$ @$\tilde q^+$})(\ve, \ell_z)$ such that 
\begin{itemize}
\item $\tilde r^+(\ve, \ell_z)$ lies in the region $[r_{lb}(\ve, d), r_{min}^+(\ve)]$ and  is given by 
\begin{equation}
\tilde r^+(\ve, \ell_z) = \ell_c^{-1}\left(\frac{\ell_z}{\ve}, \ve^2\right),
\end{equation}
where $\ell_c^{-1}$ is the inverse of the restriction of $\ell_c(\cdot, \ve^2)$ on $]r_{lb}(\ve), r_{min}^+(\ve)[$.
\item $\tilde q^+(\ve, \ell_z)$ is given by 
\begin{equation*}
\tilde q^+(\ve, \ell_z) = \ve^2\eta_c(\tilde r^+(\ve, \ell_z), \ve^2), 
\end{equation*}
\end{itemize}
\end{itemize}
\item $\ve^-_{min}<\ve<1$ and $\ell_z\in[\ell_{ub}^-(\ve), \ell_{ub}^+(\ve)]$. 
\begin{itemize}
\item If $\ell_z = \ell_z^+\geq \ell_{lb}^+(\ve)$ or  $\ell_z = \ell_z^-\leq \ell_{lb}^-(\ve)$, then there exists a unique $(\tilde r^\pm, \tilde q^\pm)(\ve, \ell_z)$ such that 
\begin{itemize}
\item $\tilde r^\pm(\ve, \ell_z)$ lie in the region  $[r^-_{min}(\varepsilon), r^+_{min}(\varepsilon)]$ and are given by 
\begin{equation}
\tilde r^\pm(\ve, \ell_z) = \ell_c^{-1}\left(\frac{\ell_z^\pm}{\ve}, \ve^2\right),
\end{equation}
where $\ell_c^{-1}$ is the inverse of the restriction of $\ell_c(\cdot, \ve^2)$ on $]r^-_{min}(\varepsilon), r^+_{min}(\varepsilon)[$ and
\item $\tilde q^\pm(\ve, \ell_z)$ are given by 
\begin{equation}
\label{qi::}
\tilde q^\pm(\ve, \ell_z) := \ve^2\eta_c(\tilde r^\pm(\ve, \ell_z, d), \ve^2). 
\end{equation} 
\end{itemize}
\item If  $\ell_z\in[\ell_{lb}^-(\ve), \ell_{lb}^+(\ve)]$, there exist two solutions  $(\tilde r_1 \index[Symbols]{$r_1$ @$\tilde r_1$}, \tilde q_1 \index[Symbols]{$q_1$ @$\tilde q_1$})(\ve, \ell_z)$ and $(\tilde r_2\index[Symbols]{$r_2$ @$\tilde r_2$} , \tilde q_2 \index[Symbols]{$r_2$ @$\tilde r_2$})(\ve, \ell_z)$ such that 
\begin{itemize}
\item $\tilde r_1(\ve, \ell_z)$ lies in the region  $[r^+_{max}(\varepsilon), r^-_{max}(\varepsilon)]$ and is given by 
\begin{equation}
\tilde r_1(\ve, \ell_z) = \ell_c^{-1}\left(\frac{\ell_z}{\ve}, \ve^2\right),
\end{equation}
where $\ell_c^{-1}$ is the inverse of the restriction of $\ell_c(\cdot, \ve^2)$ on $]r^+_{max}(\ve), r^-_{max}(\ve)[$. 
\item $\tilde r_2(\ve, \ell_z)$ lies in the region  $[r^-_{min}(\ve), r^+_{min}(\ve)]$ and is given by 
\begin{equation}
\tilde r_2(\ve, \ell_z) = \ell_c^{-1}\left(\frac{\ell_z}{\ve}, \ve^2\right),
\end{equation}
where $\ell_c^{-1}$ is the inverse of the restriction of $\ell_c(\cdot, \ve^2)$ on $]r^-_{min}(\ve), r^+_{min}(\ve)[$. 
\item $\tilde q_i(\ve, \ell_z)$ are given by 
\begin{equation}
\label{qi::}
\tilde q_i(\ve, \ell_z) := \ve^2\eta_c(\tilde r_i(\ve, \ell_z, d), \ve^2). 
\end{equation} 
\end{itemize}

\end{itemize}

\item $\ve\geq 1$ and $\ell_z\in \left]\ell^{-}_{lb}(\ve, d), \ell^{+}_{lb}(\ve) \right[$. In this case, there a unique $(\SymbolPrint{\overline r}, \SymbolPrint{\overline q})(\ve, \ell_z)$ such that 
\begin{itemize}
\item $\overline r(\ve, \ell_z)$ lies in the region $]r^+_{max}(\ve), r^-_{max}(\ve)[$ and is given by 
\begin{equation}
\overline r(\ve, \ell_z) = \ell_c^{-1}\left(\frac{\ell_z}{\ve}, \ve^2\right),
\end{equation}
where $\ell_c^{-1}$ is the inverse of the restriction of $\ell_c(\cdot, \ve^2)$ on $]r^+_{max}(\ve), r^-_{max}(\ve)[$.
\item $\overline q(\ve, \ell_z)$ is given by 
\begin{equation}
\overline q(\ve, \ell_z) = \ve^2\eta_c(\overline r(\ve, \ell_z), \ve^2)
\end{equation}
and satisfies 
\begin{equation*}
0 \leq q \leq q_{max}(\ve^2). 
\end{equation*}
\end{itemize}
\end{enumerate}
\end{lemma}

\begin{proof}
The proof is straightforward based on the monotonicity properties of $\ell_c(\cdot, \ve^2)$ on the allowed regions for $r$ given in  Lemma \ref{monotonicity:lc}.
\end{proof}
The previous lemma allows us to introduce the following subsets of $\mathbb R\times\mathbb R\times[0, \infty[$:
\begin{equation}
\mathcal A^{\geq 1}_{spherical} := \left\{ (\ve, \ell_z, q)\in]0, \infty[\times\mathbb R\times[0, \infty[\;:\; \ve\geq 1\;,\; \ell_z\in[\ell_{lb}^-(\ve), \ell_{lb}^+(\ve)\;:\; q = \overline q(\ve, \ell_z)]\right\},  
\end{equation}
\begin{equation}
\mathcal A^{\leq 1}_{spherical} := \left\{(\ve_{min}^+, \ell_{min}^+, 0) \right\}\sqcup \mathcal A^{\leq 1}_{+} \sqcup \mathcal A^{\leq 1}_{-}\sqcup \mathcal A^{\leq 1}_{1}\sqcup \mathcal A^{\leq 1}_{2},
\end{equation}
where
\begin{equation}
\begin{aligned}
\mathcal A^{\leq 1}_{1} &:= \left\{ \ve\in]\ve^+_{min}, \ve^-_{min}]\;, \, \ell_z\in[\tilde\ell_{min}(\ve), \ell_{lb}^+(\ve)]\;,\; q = q_s^1(\ve, \ell_z)\right\} \\&\sqcup \left\{ \ve\in]\ve^+_{min}, \ve^-_{min}]\;, \, \ell_z\in[\tilde\ell_{min}(\ve), \ell_{lb}^+(\ve)]\;,\; q = q_s^2(\ve, \ell_z)\right\},
\end{aligned}
\end{equation}


\begin{equation}
\begin{aligned}
\mathcal A^{\leq 1}_{2} &:= \left\{ \ve^-_{min}<\ve<1\;, \, \ell_z\in[\ell_{lb}^-(\ve), \ell_{lb}^+(\ve)]\;,\; q = \tilde q^1(\ve, \ell_z)\right\} \\ 
&\sqcup \left\{ \ve^-_{min}<\ve<1\;, \, \ell_z\in[\ell_{lb}^-(\ve), \ell_{lb}^+(\ve)]\;,\; q = \tilde q^2(\ve, \ell_z)\right\},
\end{aligned}
\end{equation}

\begin{equation}
\mathcal A^{\leq 1}_{+} := \left\{ \ve\in]\ve^+_{min}, 1[\;, \, \ell_z\in[\ell_{lb}^+(\ve), \ell_{ub}^+(\ve)]\;,\; q = \tilde q^+(\ve, \ell_z)\right\}
\end{equation}
and
\begin{equation}
\mathcal A^{\leq 1}_{-} := \left\{ \ve\in]\ve^-_{min}, 1[\;, \, \ell_z\in[\ell_{ub}^-(\ve), \ell_{lb}^-(\ve)]\;,\; q = \tilde q^-(\ve, \ell_z)\right\}. 
\end{equation}

Finally, we introduce the set 
\begin{equation}
\label{A:::sph}
\SymbolPrint{\mathcal A_{spherical}} := \mathcal A^{\geq 1}_{spherical}\sqcup \mathcal A^{\leq 1}_{spherical}. 
\end{equation}
We conclude that if $\gamma$ is a timelike future directed spherical orbit with constants of motion $(\ve, \ell_z, q)$, then 
\begin{equation*}
(\ve, \ell_z, q)\in\mathcal A_{spherical}. 
\end{equation*}
\noindent This end the proof of Proposition \ref{spherical:orbit}. Reciprocally, we have 
\begin{Propo}
Let $\gamma:\tau\ni I\to \mathcal O$ be a timelike future directed  geodesic with constants of motion $\displaystyle (\ve, \ell_z, q)$. If $(\ve, \ell_z, q)\in \mathcal A_{spherical}$ and $\gamma$ starts at $(t, \phi, r_s, \theta)$ where $r_s$ is determined by one of the cases of Lemma \ref{monotonicity:lc}. Then, $\gamma$ is spherical. 
\end{Propo}
\subsubsection{Roots of the fourth order polynomial $R$}
In this section, we will determine the number of solutions in $r$ of  the equation 
\begin{equation}
\label{R::4:0}
R(r, \ve, \ell_z, q) = 0
\end{equation}
 in the region $]r_H, \infty[$ based on the possible values of $(\ve, \ell_z, q)$. This will allow us to compute the ZVCs associated to timelike future-directed geodesics. First, we recall that $R$ is defined by 
\begin{equation*}
R(r, \ve, \ell_z, q)= (\varepsilon(r^2+d^2)-d\ell_z)^2 - (r^2 - 2r + d^2)(r^2+ (d\varepsilon - \ell_z)^2 + q). 
\end{equation*}
Let $(\ve, \ell_z, q)\in\mathbb R\times\mathbb R\times \mathbb R$.  We will study the existence of roots of the polynomial $R(\cdot, \ve, \ell_z, q)$ in the region $(r_H, \infty)$. First of all, we note that:
\begin{enumerate}
\item By Lemma \ref{positive::q}, if $q<0$, then $\ve>1$ and  the equation $\displaystyle R(r, \ve, \ell_z, q) = 0$ does not have roots in the region $(r_H, \infty)$. 
\item $\forall (r, \ve, \ell_z)\in(r_H, \infty)\times\mathbb R\times\mathbb R, $ the functions $R(r, \ve, \ell_z, \cdot)$ and $\frac{\partial R}{\partial r}(r, \ve, \ell_z, \cdot, d)$ are monotonically decreasing on $\mathbb R$. 
\item Let $(\ve, \ell_z, q)\in\mathbb R\times\mathbb R\times\mathbb R_+$ and let $r(\ve, \ell_z, q)\in]r_H, \infty[$ be a root of $R(\cdot, \ve, \ell_z, q)$. Then, if  $r$ is a simple root, then it defines a smooth function of $(\ve, \ell_z, q)$. Indeed, we have 
\begin{equation*}
\frac{\partial R}{\partial r}(\ve, \ell_z, q) \neq 0. 
\end{equation*}
\noindent By the implicit function theorem, $r$ is locally a smooth function of $(\ve, \ell_z, q)$. 
\item Finally, we note that in general, the roots of $R(\cdot, \ve, \ell_z, q)$ can be parametrised by functions of $(\ve, \ell_z, q)$ which are merely continuous when double or triple roots occur. 
\end{enumerate}
\noindent Based on the asymptotics of $R(\cdot, \ve, \ell_z, q)$ \eqref{asympto::R},  we will separate the cases $\ve<1$ and $\ve>1$.  First of all, we classify the roots of $R$ when $q = 0$. We state
\begin{Propo}[Roots of $R$ when $q = 0$]
\label{roots:eq}
Let $(\ve^2, \ell_z)\in[0, \infty[\times\mathbb R$, then the roots of $R(\cdot, \ve, \ell_z, 0)$ are summarised in Tables \ref{table3} and \ref{table4}. 

\begin{table}[ht]
\begin{center}
\begin{tabular}{ |p{1.5cm}|p{1.5cm}|p{2cm}|p{3cm}|p{3cm}|p{1.5cm}|p{3cm}|} 
\hline
 \multicolumn{2}{|c|}{$\ve = \ve_{min}^+$}   & \multicolumn{2}{|c|}{$\ve_{min}^+<\ve< \ve_{min}^-$} & \multicolumn{3}{|c|}{$\ve = \ve_{min}^-$} \tabularnewline 
\hline
$\ell_z\in\mathbb R$  & $\ell_z = \ell_{min}^+$ & $\ell_z\in\mathbb R\backslash \ell_{min}^+$  & $\ell_{lb}^+(\ve)\leq \ell_z\leq \ell_{ub}^+(\ve)$ & $\ell_z<\ell_{lb}^+(\ve)$ and $\ell_z\neq \ell_{min}^-$ & $\ell_z = \ell_{min}^-$ & $\ell_{lb}^+(\ve)\leq \ell_z\leq \ell_{ub}^+(\ve)$   \tabularnewline  
\hline
One root & One tripe root $r^+_{ms}$ & One root & Three roots & One root & Three roots & One triple root $r^-_{ms}$ \tabularnewline
\hline 
\end{tabular}
\end{center}
\caption{Possible roots of $R(\cdot, \ve, \ell_z, 0)$}
\label{table4}
\end{table}

\begin{table}[ht]
\begin{center}
\begin{tabular}{ |p{1.5cm}|p{3cm}|c|p{3cm}|p{4cm}|} 
\hline
\multicolumn{1}{|c|}{$\ve<\ve^+_{min}$} & \multicolumn{2}{|c|}{$\ve^2\geq 1$} &  \multicolumn{2}{|c|}{$\ve_{min}^-<\ve<1$}   \tabularnewline 
\hline
$\ell_z\in\mathbb R$  & $\ell_{lb}^-(\ve)<\ell_z< \ell_{lb}^+(\ve)$  & $\quad\ell_z\geq \ell_{lb}^+(\ve)$ or $\ell_z\leq \ell_{lb}^-(\ve)$ & $\ell_{lb}^-(\ve)<\ell_z< \ell_{lb}^+(\ve)$  & $\ell_{lb}^+(\ve)\leq \ell_z\leq \ell_{ub}^+(\ve)$ or $\ell_{ub}^-(\ve)\leq \ell_z\leq \ell_{lb}^-(\ve)$ \tabularnewline  
\hline
One root & No roots & Two roots & One root & Three roots  \tabularnewline
\hline 
\end{tabular}
\end{center}
\caption{Possible roots of $R(\cdot, \ve, \ell_z, 0)$}
\label{table3}
\end{table}
\end{Propo}
\begin{proof}
\begin{enumerate}
\item We will write details for the case $\ve^2>1$. The other cases follow using the same method. 
\item Assume that $\ve^2>1$. We recall that if $\ell_z = \ell_{lb}^\pm(\ve)$, then $R(\cdot, \ve, \ell_z, 0)$  admits a unique double root in the region $]r_h(d), \infty[$ given by  $r = r_{max}^\pm(\ve)$. 

\item Now, in view of Lemma \ref{roots:of:RR}, it suffices to prove that if $\ell_z\in]-\infty, \ell_{lb}^-(\ve)]\cup[\ell_{lb}^+(\ve), +\infty[$, then there exists $r\in]r_H, \infty[$ such that $\displaystyle R(r, \ve, \ell_z, 0) \leq 0$. Otherwise, $\displaystyle \forall r\in]r_H, \infty[\;,\; R(r, \ve, \ell_z, 0)>0$.

\noindent To this end, we claim that $\displaystyle\forall \ve^2>1$, the function $R(r^+_{max}(\ve), \ve, \cdot, 0)$ is monotonically decreasing on $[\ell_{lb}^+(\ve), +\infty[$ and the function $R(r^+_{max}(\ve), \ve, \cdot, 0)$ is monotonically increasing on $]-\infty, \ell_{lb}^-(\ve)]$, 
\begin{itemize}
\item Indeed, we note the following relation between $\ve,  \ell_{lb}^\pm(\ve)$ and $r_{max}^\pm(\ve)$:
\begin{equation*}
\ell_{lb}^\pm(\ve) - 2\frac{\ell_{lb}^\pm(\ve) - d\ve}{r_{max}^\pm(\ve)} = \mp\frac{\Delta\left(\frac{1}{r_{max}^\pm(\ve)}\right)}{\sqrt{\frac{Q_\mp\left( r_{max}^\pm(\ve)\right)}{r_{max}^\pm(\ve)}}} \quad\text{where}\quad Q_\mp(u):= 1 - 3u \mp 2du^{\frac{3}{2}}. 
\end{equation*}
\item Moreover, we recall that $\forall \ve>1, \forall d\in[0, 1]\; ,\; r^-_{max}(\ve) = r^-_{max}(\ve, d)> 2$ and there exists $0<d_0<1$ such that $\forall \ve>1, \forall d\in[d_0, 1]\; ,\; r^+_{max}(\ve) = r^+_{max}(\ve, d)< 2$
\item Now, we compute
\begin{equation*}
\begin{aligned}
\forall \ell_z\in\mathbb R\;,\; \frac{\partial R}{\partial \ell_z}(r^\pm_{max}(\ve), \ve, \ell_z) &= -2r^\pm_{max}(\ve)(2d\ve + (r^\pm_{max}(\ve)-2)\ell_z) \\
&= -2(r^\pm_{max})^2(\ve)\left( \ell_z - 2\frac{\ell_z - d\ve}{r_{max}^\pm(\ve)}\right)
\end{aligned} 
\end{equation*}
Therefore, 
\begin{enumerate}
\item If $\ell_z<\ell_{lb}^-(\ve)< 0$, then
\begin{equation*}
\frac{\partial R}{\partial \ell_z}(r^-_{max}(\ve), \ve, \ell_z) > 2(r^-_{max})^2 \frac{\Delta\left(\frac{1}{r_{max}^-(\ve)}\right)}{\sqrt{\frac{Q_+\left( r_{max}^-(\ve)\right)}{r_{max}^-(\ve)}}} > 0. 
\end{equation*}
\item If  $\ell_z>\ell_{lb}^+(\ve)> 0$, then 
\begin{itemize}
\item if $r_{max}^+(\ve)\geq 2$, then $\displaystyle \frac{\partial R}{\partial \ell_z}(r^+_{max}(\ve), \ve, \ell_z)  < 0$.
\item Otherwise, we have 
\begin{equation*}
\frac{\partial R}{\partial \ell_z}(r^\pm_{max}(\ve), \ve, \ell_z) < -2(r^+_{max})^2 \frac{\Delta\left(\frac{1}{r_{max}^+(\ve)}\right)}{\sqrt{\frac{Q_-\left( r_{max}^+(\ve)\right)}{r_{max}^-(\ve)}}} < 0.
\end{equation*}
\end{itemize}
\end{enumerate}
\end{itemize}
Hence, 
\begin{itemize}
\item For $\ell_z>\ell_{lb}^+(\ve)$, we have $\displaystyle R(r^+_{max}(\ve), \ve, \ell_z, 0) < R(r^+_{max}(\ve), \ve, \cdot, \ell_{lb}^+(\ve)) = 0$. 
\item For $\ell_z<\ell_{lb}^-(\ve)$, we have $\displaystyle R(r^-_{max}(\ve), \ve, \ell_z, 0) < R(r^-_{max}(\ve), \ve, \cdot, \ell_{lb}^-(\ve)) = 0$. 
\end{itemize}
\end{enumerate}
This ends the proof.

\end{proof}
\noindent  Now, we discuss the general case. 
\paragraph{Case $\ve^2\geq 1$}
\begin{Propo}
\label{roots:unbounded}
Let $\ve^2 > 1$. The possible number of roots of $R(\cdot, \ve, \ell_z, q)$ are summarised in Table \ref{table2}.
\begin{table}[ht]
\begin{center}
\begin{tabular}{ |c|c|c|c|} 
\hline
 \multicolumn{4}{|c|}{$\ve^2\geq 1$}  \tabularnewline 
\hline
\multicolumn{2}{|c|}{$\ell_{lb}^-(\ve)<\ell_z< \ell_{lb}^+(\ve)$}  & \multicolumn{2}{|c|}{$\quad\ell_z\geq \ell_{lb}^+(\ve)$ or $\ell_z\leq \ell_{lb}^-(\ve)$}  \tabularnewline  
\hline
$q\geq \overline q(\ve, \ell_z)$ & $q< \overline q(\ve, \ell_z)$ & $q\geq 0$ & $q<0$   \tabularnewline
\hline
Two roots & No roots & Two roots & No roots \tabularnewline
\hline 
\end{tabular}
\end{center}
\caption{Possible roots or $R$ in the unbounded case}
\label{table2}
\end{table}
\end{Propo}

\begin{proof}
Let  $(\ve, \ell_z, q)\in[1, \infty[\times \mathbb R\times \mathbb R$ and consider the equation in $r$
\begin{equation}
R(r, \ve, \ell_z, q) = 0. 
\end{equation}
By Lemma \ref{roots:of:RR}, $R(\cdot, \ve, \ell_z, q)$ admits either zero roots or two roots in the region $]r_H, \infty[$. Moreover, by Lemma \ref{positive::q}, if $q<0$, then $R(\cdot, \ve, \ell_z)$ has no roots. 
Now, assume that $q\geq 0$. Since $R(r, \ve, \ell_z, \cdot)$ is monotonically decreasing on $\mathbb R$, we have: if $r_0(\ve, \ell_z, 0)$ is a root of $R(\cdot, \ve, \ell_z, 0)$, then 
\begin{equation*}
\forall q>0 \;:\; R(r_0(\ve, \ell_z, 0), \ve, \ell_z, q) < R(r_0(\ve, \ell_z, 0), \ve, \ell_z, 0) = 0. 
\end{equation*}
Hence, in view of the asymptotics, $R(\cdot, \ve, \ell_z, 0)$ admits two roots. 
\\Now, by lemma \ref{roots:eq} $R(\cdot, \ve, \ell_z, 0)$ admits two roots if and only if $\ell_z\geq \ell_{lb}^+(\ve)$ or $\ell_z\leq \ell_{lb}^-(\ve)$. Hence, if $\ell_z>\ell_{lb}^+(\ve)$ or $\ell_z<\ell_{lb}^-(\ve)$, $R(\cdot, \ve, \ell_z, q)$ admits two simple roots for all $q\geq 0$. Moreover, if $q = 0$ and $\ell_z = \ell_{lb}^\pm(\ve)$, then the roots coincide. 
\\Now, assume that $\ell_z\in]\ell_{lb}^-(\ve), \ell_{lb}^+(\ve)[$. Then, by Proposition \ref{spherical:orbit}, $R(\cdot, \ve, \ell_z)$ admits a double root if and only if $q = \overline q(\ve, \ell_z, d)$ and it is given by $\overline r(\ve, \ell_z)$. By the monotonicity properties of $R(r, \ve, \ell_z, \cdot)$, we have 
\begin{equation*}
\forall q>\overline q(\ve, \ell_z)\;, \; R(\overline r(\ve, \ell_z), \ve, \ell_z, q) < R(\overline r(\ve, \ell_z), \ve, \ell_z, \overline q(\ve, \ell_z)) = 0
\end{equation*}
Now assume that $q<\overline q(\ve, \ell_z)$. Then, 
\begin{equation*}
\forall r> r_H\;:\; R(r, \ve, \ell_z, q) > R(r, \ve, \ell_z, \overline q(\ve, \ell_z)) \geq 0. 
\end{equation*}
Therefore, $R$ has no roots. This ends the proof. 
\end{proof}
 
\paragraph{Case $\ve^2<1$}
\begin{Propo}
\label{roots:bounded}
Assume that $(\ve, \ell_z, q)\in ]-1, 1[\times\mathbb R\times [0, \infty[$. The possible number of roots of $R(\cdot, \ve, \ell_z, q)$ are summarised in Table \ref{table1} and Table \ref{table1:bis}.   
\begin{table}[ht]
\begin{center}
\begin{tabular}{|p{2cm}|p{2cm}|p{2.5cm}|p{3cm}|p{2cm}|p{1.5cm}|p{1.5cm}|} 
\hline
$0<\ve<\ve^+_{min}$ & \multicolumn{6}{|c|}{$\ve^+_{min}\leq\ve\leq\ve^-_{min}$}  \tabularnewline
\hline
$\ell_z\in\mathbb R$ & \multicolumn{1}{|c|}{$\ell_z < \tilde\ell_{min}(\ve)$} &\multicolumn{3}{|c|}{$\tilde\ell_{min}(\ve)\leq \ell_z\leq\ell_{lb}^+(\ve)$} &\multicolumn{2}{|c|}{$\ell_{lb}^+(\ve)\leq \ell_z\leq\ell_{ub}^+(\ve)$} \tabularnewline
\hline 
$q\geq 0$ & $q\geq 0$ & $0\leq q < q_s^1(\ve, \ell_z)$ & $q_s^1(\ve, \ell_z)\leq q \leq q_s^2(\ve, \ell_z)$ & $q>q_s^2(\ve, \ell_z)$ & $0\leq q\leq \tilde q^+(\ve, \ell_z) $ & $q>\tilde q^+(\ve, \ell_z) $ \tabularnewline
\hline 
One root & One root &One root & Three roots & One root & Three roots &  One root  \tabularnewline
\hline
$r_{abs}^K (\ve, \ell_z, q)$ & $r_{abs}^K (\ve, \ell_z, q)$ & $r_{abs}^1 (\ve, \ell_z, q)$ & $r_{abs}^i(\ve, \ell_z, q)\;,i\in\left\{1, 2, 3\right\}$ & $r_{abs}^3 (\ve, \ell_z, q)$ & $r_i^K(\ve, \ell_z, q)\;,\; i\in\left\{0, 1, 2\right\}$ &  $r_{abs}^0 (\ve, \ell_z, q)$\tabularnewline
\hline
\end{tabular}
\end{center} 
\caption{Possible roots or $R$ in the bounded case}
\label{table1}
\end{table}

\begin{table}[ht]
\begin{center}

\begin{tabular}{|p{3cm}|p{3cm}|p{2.5cm}|p{3cm}|p{4cm}|} 
\hline
\multicolumn{5}{|c|}{$\ve^-_{min}<\ve<1$ }  \tabularnewline 
\hline
\multicolumn{3}{|c|}{$\ell_{lb}^-(\ve)\leq\ell_z\leq\ell_{lb}^+(\ve)$} & \multicolumn{2}{|c|}{$\ell_{lb}^+(\ve)\leq\ell_z\leq \ell_{ub}^+(\ve)$ or $\ell_{ub}^-(\ve)\leq\ell_z\leq \ell_{lb}^-(\ve)$} \tabularnewline  
\hline
$0\leq q < \tilde q^1(\ve, \ell_z)$ & $\tilde q^1(\ve, \ell_z)$ $\leq q\leq \tilde q^2(\ve, \ell_z)$ & $q>\tilde q^2(\ve, \ell_z)$ & $0\leq q\leq \tilde q^\pm(\ve, \ell_z)$ & $q>\tilde q^\pm(\ve, \ell_z)$ \tabularnewline
\hline
$r_0^K (\ve, \ell_z, q)$ & $r_i^K(\ve, \ell_z, q)\;,i\in\left\{0, 1, 2\right\}$ & $r_0^K (\ve, \ell_z, q)$ & $r_i^K(\ve, \ell_z, q)\;,i\in\left\{0, 1, 2\right\}$ & $r_0^K (\ve, \ell_z, q)$  \tabularnewline
\hline 

\end{tabular}
\end{center}
\caption{Possible roots or $R$ in the bounded case}
\label{table1:bis}
\end{table}

\end{Propo}
\begin{proof}
Let $(\ve, \ell_z, q)\in]0, 1[\times\mathbb R\times[0, \infty[$. By Lemma \ref{roots:of:RR}, $R$ admits either one root or three roots in the region $]r_H, \infty[$. Note also that $\partial_r R$ has either no roots or two roots in the region$]r_H, \infty[$.  
\\ By Lemma \ref{roots:eq}, $R(\cdot, \ve, \ell_z, 0)$ admits three roots if and only if $\ve_{min}^{+}<\ve<1$ and $\ell_z\in[\ell_{lb}^+(\ve),\ell_{ub}^+(\ve)]$ or  $\ve_{min}^{-}(d)<\ve<1$ and $\ell_z\in[\ell_{ub}^-(\ve),\ell_{lb}^-(\ve)]$.  Hence, 
\begin{enumerate}
\item If $\ve< \ve_{min}^{+}$, then $\forall\ell_z\in \mathbb R$, $R(\cdot, \ve, \ell_z, 0)$ admits a unique root and $\displaystyle \frac{\partial R}{\partial r}(\cdot, \ve, \ell_z, 0)<0$. Therefore, 
\begin{equation*}
\forall q\geq 0\;:\; \frac{\partial R}{\partial r}(\cdot, \ve, \ell_z, q) \leq \frac{\partial R}{\partial r}(\cdot, \ve, \ell_z, 0)< 0. 
\end{equation*}
Hence, $R(\cdot, \ve, \ell_z, q)$ admits a unique root. 
\item If $\ve_{min}^{+}(d)\leq\ve<\ve_{min}^{-}(d)$
\begin{enumerate}
\item If $\ell_z\in [\ell_{lb}^+(\ve),\ell_{ub}^+(\ve)]$. 
Then, $R(\cdot, \ell_z, 0)$ admits three roots and $\displaystyle \frac{\partial R}{\partial r}(\cdot, \ve, \ell_z, 0)$ admits two roots, denoted by $\underline r_i(\ve, \ell_z)$.   
\begin{itemize}
\item If $q = \tilde q^+(\ve, \ell_z)$, then $R(\cdot, \ve, \ell_z, q)$ admits a double root given by $\tilde r^+(\ve, \ell_z)$. 
\item By monotonicity properties of $\partial_r R(r, \ve, \ell_z, \cdot)$, we have 
\begin{equation*}
\forall q<\tilde q^+(\ve, \ell_z)\;,\; \frac{\partial R}{\partial r}(\tilde r^+(\ve, \ell_z), \ve, \ell_z, q) > \frac{\partial R}{\partial r}(\tilde r^+(\ve, \ell_z), \ve, \ell_z, \tilde q^+(\ve, \ell_z)) = 0. 
\end{equation*}
Thus,   $\displaystyle \frac{\partial R}{\partial r}(\cdot, \ve, \ell_z, q)$ admits two roots. This yields to three roots for $R(\cdot, \ell_z, q)$. 
\item We have, 
\begin{equation*}
\forall q>\tilde q^+(\ve, \ell_z)\;,\;\forall r>r_H\;,\;  \frac{\partial R}{\partial r}(\tilde r^+(, \ve, \ell_z, q) < \frac{\partial R}{\partial r}(r, \ve, \ell_z, \tilde q^+(\ve, \ell_z)) \leq 0.
\end{equation*}
The latter inequality is due to the fact that $\tilde r^+(\ve, \ell_z)$ is a global maximum for $\displaystyle \frac{\partial R}{\partial r}(\cdot, \ve, \ell_z, \tilde q^+(\ve, \ell_z))$. Hence, $\displaystyle \frac{\partial R}{\partial r}(\cdot, \ve, \ell_z, q)$ does not change sign. In this case, $R(\cdot, \ve, \ell_z, q)$ admits only one root. 
\end{itemize}
\item The remaining cases  $\ell_z\in]0, \tilde\ell^+_{min}(\ve)[$ and  $\ell_z\in[\tilde\ell^+_{min}(\ve), \ell^+_{lb}(\ve)[$ follow using similar arguments. 
\end{enumerate}
\item We use similar arguments for the remaining case $\ve^-_{min}\leq\ve<1$
\end{enumerate}
\end{proof}
\subsubsection{Roots of the fourth order polynomial $T$}
\label{T::four::Roots}
In this section, we recall the solutions of 
\begin{equation}
T(Y, \ve, \ell_z, q) = 0
\end{equation}
in the region $]-1, 1[$ at a given $(\ve, \ell_z, q)\in]0, \infty[\times\mathbb R\times\mathbb R$. 
\begin{enumerate}
\item If $q\geq 0$, then by Lemma \ref{T:double:root}, $T$ admits two roots given by \eqref{theta:sol} which coincide if and only if $q = 0$.
\item Otherwise, $\ve^2>1$. Then $T$ admits two distinct roots $\SymbolPrint{\mu_\pm}(\ve, \ell_z, q)$ in the region $]-1, 1[$ which satisfy  
\begin{equation*}
-1 < \mu_-< 0 < \mu_+ < 1
\end{equation*}
and are given by 
\begin{equation*}
\cos\mu_{\pm}(\ve, \ell_z, q) = \pm\frac{(\ell_z^2+d^2(1-\varepsilon^2) + q + \sqrt{(\ell_z^2+d^2(1-\varepsilon^2) + q)^2 - 4qd^2(1-\varepsilon^2)})}{2d^2(1-\varepsilon^2)}
\end{equation*}
\end{enumerate}
We end this section by defining the following angles:  
\begin{itemize}
\item $\SymbolPrint{\theta_{1}}(\ve, \ell_z)\in\left(0, \frac{\pi}{2} \right]$ defined on the domain $\displaystyle \left\{ (\ve, \ell_z) \;:\; \ve_{min}^+<\ve<\ve_{min}^-\:,\; \ell_z\in]\ell_{min}^+, \ell_{lb}^+(\ve)[ \right\}$  by the following expression: 
\begin{equation}
\label{theta:i}
\cos\theta_{1}(\ve, \ell_z) := \frac{(\ell_z^2+d^2(1-\varepsilon^2) + q^1_s(\ve, \ell_z) - \sqrt{(\ell_z^2+d^2(1-\varepsilon^2) + q^1_s(\ve, \ell_z))^2 - 4 q^1_s(\ve, \ell_z)d^2(1-\varepsilon^2)})}{2d^2(1-\varepsilon^2)},
\end{equation}
\item $\tilde \theta_{1} \index[Symbols]{$\theta_1$ @$\tilde\theta_1 $}(\ve, \ell_z) \in\left(0, \frac{\pi}{2} \right]$ defined on the domain  $\displaystyle  \left\{ (\ve, \ell_z) \;:\; \ve_{min}^-<\ve<1\:,\; \ell_z\in]\ell_{lb}^-(\ve), \ell_{lb}^+(\ve)[ \right\} $  by the following expression: 
\begin{equation}
\label{tilde:theta:i}
\cos\tilde \theta_{1}(\ve, \ell_z) := \frac{(\ell_z^2+d^2(1-\varepsilon^2) + \tilde q^1(\ve, \ell_z) - \sqrt{(\ell_z^2+d^2(1-\varepsilon^2) + \tilde q^1(\ve, \ell_z))^2 - 4\tilde q^1(\ve, \ell_z)d^2(1-\varepsilon^2)})}{2d^2(1-\varepsilon^2)},
\end{equation}
\item $\SymbolPrint{\overline \theta_{max}^{<1}}(\ve, \ell_z)\in\left(0, \frac{\pi}{2} \right)$ defined on the domain $\displaystyle \left\{ (\ve, \ell_z) \;:\; \ve_{min}^+ <  \ve\leq \ve_{min}^-\:,\; \ell_z\in]\ell_{lb}^+(\ve), \ell_{ub}^+(\ve)[ \right\}$ by the following expression: 
\begin{equation}
\label{theta::max:p}
\cos\overline \theta_{max}^{<1}(\ve, \ell_z) := \frac{(\ell_z^2+d^2(1-\varepsilon^2) + \tilde q^+(\ve, \ell_z) - \sqrt{(\ell_z^2+d^2(1-\varepsilon^2) + \tilde q^+(\ve, \ell_z))^2 - 4\tilde q^+(\ve, \ell_z)d^2(1-\varepsilon^2)})}{2d^2(1-\varepsilon^2)},
\end{equation}
\item $\SymbolPrint{\overline \theta_{max}^{\geq 1}}(\ve, \ell_z)\in\left(0, \frac{\pi}{2} \right)$ defined on the domain $\displaystyle \left\{ (\ve, \ell_z) \;:\; \ve>1\:,\; \ell_z\in]\ell_{lb}^-(\ve), \ell_{lb}^+(\ve)[ \right\}$ by the following expression: 
\begin{equation}
\label{theta:bu}
\cos\overline \theta_{max}^{\geq 1}(\ve, \ell_z) := \frac{(\ell_z^2+d^2(1-\varepsilon^2) + \overline q(\ve, \ell_z) +\sqrt{(\ell_z^2+d^2(1-\varepsilon^2) + \overline q(\ve, \ell_z))^2 - 4\overline q(\ve, \ell_z)d^2(1-\varepsilon^2)})}{2d^2(1-\varepsilon^2)},
\end{equation}
\end{itemize}

\begin{remark}
A timelike geodesic with negative $q$ must either in the region $]r_H, \infty[\times]\arccos{\mu_-(\ve, \ell_z)}, \pi[$ or in the region $]r_H, \infty[\times]0, \arccos{\mu_+(\ve, \ell_z)}, [$
\end{remark}

\subsubsection{Stationary solutions of the geodesic equation}
In this section, we compute stationary solutions corresponding to the free particle (future directed) Hamiltonian moving in the exterior region of a Kerr spacetime. This will allow us to analyze the geodesic motion in Weyl coordinates. See Section \ref{Weyl::geo::motion}.
\\ We state the main result of this section 
\begin{lemma}
\label{stationary::points}
Let $\gamma: I\to \mathcal O$ be a solution of \eqref{eq::motion1} with constants of motion $(\ve, \ell_z, q)$ such that $\forall \tau\in I\;,\; (\gamma, \dot\gamma)(\tau)\in \Gamma$. 
\begin{itemize}
\item The system \eqref{eq::motion1} admits direct stationary solutions  $(\gamma_s, \dot\gamma_s) = (x_s, v_s)$ if and only if $\ve \geq \ve^+_{min}$, $d\ell_z>0$ and $q = 0$. In this case, 
\begin{enumerate}
\item $(r_s, \theta_s)$ is given by 
\begin{equation*}
\theta_s = \frac{\pi}{2}
\end{equation*}
and 
\begin{itemize}
\item if $\ve^2\geq1$, $r_s = r^+_{max}(\ve)$, 
\item if $\ve^2<1$, $r_s\in\left\{r^+_{min}(\ve), r^+_{max}(\ve)\right\}$,
\end{itemize}
where $r^+_{max}(\ve)$ and $r^+_{min}(\ve)$ are given by Lemma \ref{r(e, d)}. 
\item Moreover, $\ell_z$ is given by 
\begin{itemize}
\item if $\ve^2\geq1$, $\ell_z = \ell^+_{lb}(\ve)$
\item if $\ve^2<1$, $\ell_z = \ell^+_{lb}(\ve)$ if $r_s = r^+_{max}(\ve)$ and  $\ell_z = \ell^+_{ub}(\ve)$ if $r_s = r^+_{min}(\ve),$
\end{itemize}
where $\ell_{lb}^+(\ve)$ and $\ell_{ub}^+(\ve)$ are defined in  \eqref{critical::values}. Furthermore, we have a lower bound on $\ell_z$: 
\begin{equation*}
\ell_z\geq \ell^+_{min}. 
\end{equation*}
\end{enumerate}
\item The system \eqref{eq::motion1} admits retrograde stationary solutions  $(\gamma_s, \dot\gamma_s) = (x_s, v_s)$ if and only if $\ve \geq \ve^-_{min}$, $d\ell_z<0$ and $q = 0$. In this case, 
\begin{enumerate}
\item $(r_s, \theta_s)$ is given by 
\begin{equation*}
\theta_s = \frac{\pi}{2}
\end{equation*}
and 
\begin{itemize}
\item if $\ve^2\geq1$, $r_s = r^-_{max}(\ve)$, 
\item if $\ve^2<1$, $r_s\in\left\{r^-_{min}(\ve), r^-_{max}(\ve)\right\}$.
\end{itemize}
\item Moreover, $\ell_z$ is given by 
\begin{itemize}
\item if $\ve^2\geq1$, $\ell_z = \ell^-_{lb}(\ve)$
\item if $\ve^2<1$, $\ell_z = \ell^-_{lb}(\ve)$ if $r_s = r^-_{max}(\ve)$ and  $\ell_z = \ell^-_{ub}(\ve)$ if $r_s = r^-_{min}(\ve)$. 
\end{itemize}
Furthermore, we have an upper bound on $\ell_z$: 
\begin{equation*}
\ell_z\leq \ell^-_{min}. 
\end{equation*}
\end{enumerate}

\end{itemize}
\end{lemma}
\begin{proof}
The proof follows from the previous results:
\begin{enumerate}
\item  By Lemma \ref{stationary::point},  stationary solutions $(x_s, v_s)$ verify $v_s^r = v^\theta_s = 0$, $r_s$ is a double root of $R(\cdot, \ve, \ell_z, q)$ and $\cos\theta_s$ is a double root of $T(\cdot, \ve, \ell_z, q)$. 
\item  Now, by Lemma \ref{T:double:root}, $q = 0$ and $\theta_s = \frac{\pi}{2}$. 
\item Therfore, stationary solutions of  \eqref{eq::motion1} are the circular orbits confined in the equatorial plane. 
\item We conclude using Proposition \ref{classif:circ:eq} and Proposition \ref{circ::clasifi}. 
\end{enumerate}
\end{proof}

\subsubsection{Classification of timelike future directed geodesics}
\label{classif:BL:geodesics}
In this section,  we classify the timelike future directed geodesics according to the roots of $R(\cdot, \ve, \ell_z, q)$. First of all, we determine $Z^K(\ve, \ell_z)$, which is a curve in the $(r, \theta)-$plane with possibly different connected components, associated to a timelike future directed geodesic with constants of motion $(\ve, \ell_z)$. Moreover, the condition \eqref{allowed:for:kerr} yields restrictions on the initial position $(r(0), \theta(0))$. The classification is based on the possible values of $(\ve, \ell_z, q)$ and $(r(0), \theta(0))$. 
\\ Let $\gamma:I\ni 0\to\mathcal O$ be a timelike future-directed geodesic with constants of motion $(\ve, \ell_z)\in\mathbb R\times\mathbb R$ and $\tilde \gamma(\tau):= (r(\tau), \theta(\tau))$ be its projection in the $(r, \theta)$-plane.
We recall that by separability of the geodesics equation, the radial motion decouples from the motion in $\theta$ direction. Moreover, $\tilde \gamma$ satisfies the system of equations
\begin{equation*}
\begin{aligned}
\Sigma^4\dot r^2 &= R(r, \ve, \ell_z, q) \\
\Sigma^4\sin^2\theta\dot \theta^2 &= T(\cos\theta, \ve, \ell_z, q) \\
\end{aligned}
\end{equation*}
where $q\in\mathbb R$ is the Carter constant, the fourth integral of motion associated to $\gamma$.  In order to determine $Z^K(\ve, \ell_z)$, we will use $q$ as a parameter for the curve $Z^K(\ve, \ell_z)$. 
We state the first result of this section concerning the classification of $Z^K(\ve, \ell_z)$ associated to $\gamma$. 
\begin{Propo}{(Shape of Zero-velocity curves)}
\label{ZVC:topology}
Let $d\in]0, 1[$ and let $\gamma: I\to\mathcal O$ be a timelike future directed geodesic with constants of motion $(\ve, \ell_z)\in\mathbb R\times\mathbb R$. 
\begin{enumerate}

\item If $0<\ve\leq \ve^+_{min}$, 
\begin{enumerate}
\item If $\ve = \ve^+_{min}$ and $\ell_z = \ell_{min}^+$, $Z^K(\ve, \ell_z)$ is diffeomorphic to $\mathbb R$ with a singular point at $\displaystyle\left(r^+_{ms}, \frac{\pi}{2} \right)$. (See Figure \ref{Zcritic}) 
\item Otherwise, $Z^K(\ve, \ell_z)$ is a smooth curve diffeomorphic to $\mathbb R$ with the following properties: it is symmetric with respect to the equatorial plane and intersects it at a unique point $r_0^K(\ve, \ell_z, 0)$, the unique root of the polynomial $\displaystyle R(\cdot, \ve, \ell_z, 0)$.
\end{enumerate}
\item If $\ve^+_{min} < \ve \leq \ve^-_{min}$, 
\begin{enumerate}
\item If $\ve = \ve^-_{min}$ and $\ell_z = \ell_{min}^-$, $Z^K(\ve, \ell_z)$ is diffeomorphic to $\mathbb R$ with a singular point at $\displaystyle\left(r^-_{ms}, \frac{\pi}{2} \right)$. 
\item Otherwise, 
\begin{enumerate}
\item If $\ell_z<\ell_{lb}^+(\ve)$, then $Z^K(\ve, \ell_z)$ is a smooth curve diffeomorphic to $\mathbb R$ with the following properties: it is symmetric with respect to the equatorial plane and intersects it at a unique point $r_2^K(\ve, \ell_z, 0)$, the unique root of the polynomial $\displaystyle R(\cdot, \ve, \ell_z, 0)$. (See Figure \ref{Zabsb})
\item If $\ell_z = \ell_{lb}^+(\ve)$, then $Z^K(\ve, \ell_z)$ is a self-intersecting smooth curve  with the following properties: it is symmetric with respect to the equatorial plane and intersects it at two points $r_0^K(\ve, \ell_z, 0)< r_2^K(\ve, \ell_z, 0)$, the roots of the polynomial $\displaystyle R(\cdot, \ve, \ell_z, 0)$. Moreover, $Z^K(\ve, \ell_z)$ admits one singular point $\displaystyle \left(r^K_0(\ve, \ell_z, 0), \frac{\pi}{2}\right)$ where it self-intersects.  (See Figure \ref{Zcircb})
\item If $\ell_{lb}^+(\ve)<\ell_z<\ell_{ub}^+(\ve)$, then $Z^K(\ve, \ell_z)$ is a smooth curve with two connected components: $\SymbolPrint{Z^{K, trapped}}(\ve, \ell_z)$ is diffeomorphic to $\mathbb S^1$ and $\SymbolPrint{Z^{K, abs}}(\ve, \ell_z)$ is diffeomorphic to $\mathbb R$. The latter curves are symmetric with respect to the equatorial plane. $Z^{K, trapped}(\ve, \ell_z)$ intersects it at $r_1^K(\ve, \ell_z, 0)$ and $r_2^K(\ve, \ell_z, 0)$ and $Z^{K, abs}(\ve, \ell_z)$ intersects it at $r^K_0(\ve, \ell_z, 0)$. 
\item $\ell_z = \ell_{ub}^+(\ve)$, then $Z^K(\ve, \ell_z)$ is a smooth curve with two connected components: the point $\displaystyle \left(r_2^K(\ve, \ell_z, 0), \frac{\pi}{2}\right)$ and $Z^{K, abs}(\ve, \ell_z)$ is diffeomorphic to $\mathbb R$. The latter curve is symmetric with respect to the equatorial plane and  intersects it at $r_0^K(\ve, \ell_z, 0)$. (See Figure \ref{Ztrapped})
\end{enumerate}
\end{enumerate}
\item If $\ve^-_{min}<\ve<1$,
\begin{enumerate}
\item If $\ell_{lb}^-(\ve)<\ell_z<\ell_{lb}^+(\ve)$, then $Z^K(\ve, \ell_z)$ is a smooth curve diffeomorphic to $\mathbb R$ with the following properties: it is symmetric with respect to the equatorial plane and intersects it at a unique point $r_2^K(\ve, \ell_z, 0)$. 
\item If $\ell_z = \ell_{lb}^+(\ve)$ or $\ell_z = \ell_{lb}^-(\ve)$, then $Z^K(\ve, \ell_z)$ is a smooth self-intersection curve with the following properties: it is symmetric with respect to the equatorial plane and intersects it at two points $r_0^K(\ve, \ell_z, 0)< r_2^K(\ve, \ell_z, 0)$. Moreover, $Z^K(\ve, \ell_z)$ admits one singular point  $\displaystyle \left(r^K_0(\ve, \ell_z, 0), \frac{\pi}{2}\right)$ where it self-intersects. 
\item If $\ell_{lb}^+(\ve)<\ell_z<\ell_{ub}^+(\ve)$ or $\ell_{ub}^-(\ve)<\ell_z<\ell_{lb}^-(\ve)$, then $Z^K(\ve, \ell_z)$ is a smooth curve with two connected components: $Z^{K, trapped}(\ve, \ell_z)$ is diffeomorphic to $S^1$ and $Z^{K, abs}(\ve, \ell_z)$ is diffeomorphic to $\mathbb R$. The latter curves are symmetric with respect to the equatorial plane. $Z^{K, trapped}(\ve, \ell_z)$ intersects it at $r_1^K(\ve, \ell_z, 0)$ and $r_2^K(\ve, \ell_z, 0)$ and $Z^{K, abs}(\ve, \ell_z)$ intersects it at $r^K_0(\ve, \ell_z, 0)$.
\item If $\ell_z = \ell_{ub}^+(\ve)$ or $\ell_z = \ell_{ub}^-(\ve)$, then $Z^K(\ve, \ell_z)$ is a smooth curve with two connected components: the point $\displaystyle \left(r_2^K(\ve, \ell_z, 0), \frac{\pi}{2}\right)$ and $Z^{K, abs}(\ve, \ell_z)$ is diffeomorphic to $\mathbb R$. The latter curve is symmetric with respect to the equatorial plane and  intersects it at $r_0^K(\ve, \ell_z, 0)$. 
 
\end{enumerate}

\item If $\ve>1$

\begin{enumerate}
\item If $\ell_{lb}^-(\ve)<\ell_z<\ell_{lb}^+(\ve)$, then $Z^K(\ve, \ell_z)$ is a smooth curve with two connected components $\SymbolPrint{Z^{K, z>0}}(\ve, \ell_z)$ and $\SymbolPrint{Z^{K, z<0}}(\ve, \ell_z)$ which are diffeomorphic to $\mathbb R$ and which do not intersect the equatorial plane. $Z^K(\ve, \ell_z)$  is symmetric with respect to the equatorial plane. (See Figure \ref{Zabsub})
\item If $\ell_z = \ell_{lb}^+(\ve)$ or $\ell_z = \ell_{lb}^-(\ve)$, then $Z^K(\ve, \ell_z)$ has the following properties: It consists of the union of two connected curves diffeomorphic to $\mathbb R$ which intersect at the point $\displaystyle\left(r_0^K(\ve, \ell_z, 0), \frac{\pi}{2}\right)$. They are symmetric with respect to the equatorial plane and intersect at the point $r_0^K(\ve, \ell_z, 0)$ . (See Figure \ref{Zcircub})
\item If $\ell_z>\ell_{lb}^+(\ve)$ or $\ell_z<\ell_{lb}^-(\ve)$, then $Z^K(\ve, \ell_z)$ is a smooth curve with two connected components: $\SymbolPrint{Z^{K, scat}}(\ve, \ell_z)$ is diffeomorphic to $\mathbb R$ and $Z^{K, abs}(\ve, \ell_z)$ is diffeomorphic to $\mathbb R$. The latter curves are symmetric with respect to the equatorial plane. $Z^{K, scat}(\ve, \ell_z)$ intersects it at $r_1^K(\ve, \ell_z, 0)$ and $Z^{K, abs}(\ve, \ell_z)$ intersects it at $r^K_0(\ve, \ell_z, 0)$. (See Figure \ref{Zscat})
\end{enumerate}

\end{enumerate}

\end{Propo}

\begin{proof}
Let $\gamma: I \to \mathcal O$ be a timelike future directed geodesic with $(\ve, \ell_z)\in\Adm$. Let $(r, \theta)\in Z^K(\ve, \ell_z)$. By Lemma \ref{charac:Zc}, we have 
\begin{equation}
\label{root:::R}
R(r, \ve, \ell_z, q(\theta, \ve, \ell_z)) = 0,
\end{equation}
where
\begin{equation}
\label{Carter:turning:bis}
q(\theta, \ve, \ell_z) = \cos^2\theta\left(d^2(1 - \ve^2) + \frac{\ell_z^2}{\sin^2\theta}\right). 
\end{equation}
$q(\cdot, \ve, \ell_z)$ is monotonically decreasing on $\displaystyle\left(0, \frac{\pi}{2}\right)$ from $\infty$ to $0$ if $\ell_z\neq 0$ and from $d^2(1-\ve^2)$ to $0$ if $\ell_z = 0$ and  monotonically increasing on $\displaystyle\left(\frac{\pi}{2}, \pi\right)$ from $0$ to $\infty$. 

\noindent By the definition of $R$ and \eqref{R::4}, it is easy to express $q$ in terms of the remaining variables: 
\begin{equation}
\label{q::r}
q = \overline{q}(r) := \frac{((r^2+d^2)\varepsilon-d\ell_z)^2}{\Delta}-(r^2 + (\ell_z-d\varepsilon)^2). 
\end{equation}
Moreover, 
and $\forall r\in]r_H, \infty[$
\begin{equation*}
\frac{\partial q}{\partial r}(r, \ve, \ell_z) = \Delta^{-1}(r)\frac{\partial R}{\partial r}(r, \ve, \ell_z, \overline q(r, \ve, \ell_z)).
\end{equation*}
In the following, we will use \eqref{Carter:turning:bis} and  \eqref{q::r} in order to eliminate $q$ from the equations. As a consequence, either $\theta$ will be seen as a function  of  $r$ or $r$ as a function of $\theta$. This will determine $Z^K(\ve, \ell_z)$. 
\begin{enumerate}
\item If $\displaystyle 0<\ve\leq \ve_{min}^+$ 
\begin{enumerate}
\item If $\ve = \ve_{min}^+$ and $\ell_z = \ell_{min}^+$, then 
\begin{itemize}
\item if $q = 0$, there exists a unique triple  root $r_{abs}(\ve, \ell_z, 0) = r^+_{ms}$ which solves \eqref{root:::R}. See 
\item if $q > 0$, there exists a unique simple root $r_{abs}(\ve, \ell_z, q)$ which solves \eqref{root:::R}. 
\end{itemize}
\noindent The application $r_{abs}(\ve, \ell_z, \cdot)$ is well defined continuous on $[0, \infty[$ and is monotonically decreasing on $]0, \infty[$ from $r^+_{ms}$ to $r_H$. Moreover, $r_{abs}(\ve, \ell_z, \cdot)$ is smooth on $]0, \infty[$ and we have
\begin{equation*}
\forall q>0\;,\; \frac{\partial r_{abs}}{\partial q}(\ve, \ell_z, q) = -\frac{\frac{\partial R}{\partial q}(r_{abs}(\ve, \ell_z, q), \ve, \ell_z, q)}{\frac{\partial R}{\partial r}(r_{abs}(\ve, \ell_z, q), \ve, \ell_z, q)}. 
\end{equation*}
\noindent Now, by \eqref{Carter:turning:bis} and the above properties, the function $\theta\ni]0, \pi[\to r_{abs}(\ve, \ell_z, q(\theta)) $  is smooth on $]0, \pi[\backslash \frac{\pi}{2}$ with 
\begin{equation*}
\left. \frac{\partial}{\partial\theta}r_{abs}(\ve, \ell_z, q(\theta))\right|_{\frac{\pi}{2}^-} = -\infty \quad\text{and }\quad  \left.\frac{\partial}{\partial\theta}r_{abs}(\ve, \ell_z, q(\theta))\right|_{\frac{\pi}{2}^+} = +\infty. 
\end{equation*} 
Therefore,  
\begin{equation*}
Z^K(\ve, \ell_z) = Graph(r_{abs}(\ve, \ell_z, q(\cdot))),
\end{equation*}
with a singular point at $(r^+_{ms}, \frac{\pi}{2})$ and which is diffeomorphic to $\mathbb R$. 
\item Otherwise,  $\forall q\geq 0$, there exists a unique simple root $r_{abs}(\ve, \ell_z, q)$ which solves \eqref{root:::R}. Moreover, the application $r_{abs}(\ve, \ell_z, \cdot)$ is well defined, smooth and monotonically decreasing on $[0, \infty[$  from $r^K_0(\ve, \ell_z, 0)$ to $r_H$. 
\\ Hence 
\begin{equation*}
Z^K(\ve, \ell_z) = Graph(r_{abs}(\ve, \ell_z, q(\cdot))),
\end{equation*}
which is diffeomorphic to $\mathbb R$. 
\end{enumerate}
\item If $\displaystyle \ve_{min}^+<\ve\leq \ve_{min}^-$, 
\begin{enumerate}
\item $\ve = \ve_{min}^-$ and $\ell_z = \ell_{min}^-$, then we use similar arguments to those of Case 1.a. 
\item Otherwise, 
\begin{enumerate}
\item If $\ell_z< \tilde \ell_{min}(\ve)$, then $\forall q\geq 0$, there exists a unique simple root $r_{abs}(\ve, \ell_z, q, d)$ which solves \eqref{root:::R}. Moreover, the application $r_{abs}(\ve, \ell_z, \cdot, d)$ is well defined on $[0, \infty[$ and is monotonically decreasing from $r^K_0(\ve, \ell_z, 0)$ to $r_H$ and we refer to  Case 1.b for conclusions.  
\item If $\tilde \ell_{min}(\ve)< \ell_z< \ell_{lb}^+(\ve)$, then $\forall q\geq 0$,  the solutions in $r$ of \eqref{root:::R} are given by
\begin{equation*}
\left\{
\begin{aligned}
&r^1_{abs}(\ve, \ell_z, q) \quad\text{if}\quad 0 \leq q< q_s^1(\ve, \ell_z), \\
&r^1_{abs}(\ve, \ell_z, q) , r^2_{abs}(\ve, \ell_z, q), r^3_{abs}(\ve, \ell_z, q)\quad\text{if}\quad q_s^1(\ve, \ell_z) \leq q\leq  q_s^2(\ve, \ell_z), \\
&r^3_{abs}(\ve, \ell_z, q)  \quad\text{if}\quad q > q_s^2(\ve, \ell_z)
\end{aligned}
\right. 
\end{equation*}
These solutions satisfy 
\begin{itemize}
\item $r^2_{abs}(\ve, \ell_z, q) =  r^3_{abs}(\ve, \ell_z, q) = r_s^1(\ve, \ell_z)$   if and only if $q = q_s^1(\ve, \ell_z)$, 
\item  $r^1_{abs}(\ve, \ell_z, q) =  r^2_{abs}(\ve, \ell_z, q) = r_s^2(\ve, \ell_z)$  if and only if $q = q_s^2(\ve, \ell_z)$, 
\item $\displaystyle r^3_{abs}(\ve, \ell_z, q) \leq r_s^1(\ve, \ell_z) \leq r^2_{abs}(\ve, \ell_z, q) \leq  r_s^2(\ve, \ell_z)\leq r^1_{abs}(\ve, \ell_z, q)$,
\end{itemize}
where  $q_s^i(\ve, \ell_z)$ are defined in Lemma \ref{lemma:29}.  Moreover, seen as functions of $q$, $r^i_{abs}(\ve, \ell_z, \cdot)$ have the following monotonicity properties
\begin{itemize}
\item $r^1_{abs}(\ve, \ell_z, \cdot)$ is monotonically decreasing on $]0, q_s^2(\ve, \ell_z)[$ from $r_0^K(\ve, \ell_z)$ to $r_s^2(\ve, \ell_z)$,
\item  $r^2_{abs}(\ve, \ell_z, \cdot)$ is monotonically decreasing on $]q_s^1(\ve, \ell_z), q_s^2(\ve, \ell_z)[$ from $r_s^2(\ve, \ell_z)$ to $r_s^1(\ve, \ell_z)$,
\item  $r^3_{abs}(\ve, \ell_z, \cdot)$ is monotonically decreasing on $]q_s^1(\ve, \ell_z), \infty[$ from $r_s^1(\ve, \ell_z)$ to $r_H$, 
\end{itemize}
Now, we construct an atlas for $Z^K(\ve, \ell_z)$: 
\begin{itemize}
\item We recall the angles $\theta_i(\ve, \ell_z)\in\left]0, \frac{\pi}{2} \right[$ defined by \eqref{theta:i}. 
\item By monotonicity properties of $q$ as a function of $\theta$, we can define the following functions
\begin{enumerate}
\item $r^1_{abs}(\ve, \ell_z, q(\cdot)): \left]\theta_2(\ve, \ell_z), \pi - \theta_2(\ve, \ell_z)\right[ \to [r_0^K(\ve, \ell_z), r^2_s(\ve, \ell_z)[$ defined by 
\begin{equation*}
r^1_{abs}(\ve, \ell_z, q(\theta)).
\end{equation*}
Here, $r^1_{abs}(\ve, \ell_z, q(\cdot))$ is symmetric with respect to the equatorial plane and has a maximum at $\frac{\pi}{2}$ given by $r_0^K(\ve, \ell_z)$. 
\item $r^{2, a}_{abs}(\ve, \ell_z, q(\cdot)): \left]\theta_1(\ve, \ell_z), \theta_2(\ve, \ell_z)\right[ \to ]r_s^1(\ve, \ell_z), r^2_s(\ve, \ell_z)[$ defined by 
\begin{equation*}
r^{2}_{abs}(\ve, \ell_z, q(\theta)).
\end{equation*}
\item $r^{2, b}_{abs}(\ve, \ell_z, q(\cdot)): \left]\pi - \theta_2(\ve, \ell_z), \pi - \theta_1(\ve, \ell_z)\right[ \to ]r_s^1(\ve, \ell_z), r^2_s(\ve, \ell_z)[$ defined by 
\begin{equation*}
r^{2}_{abs}(\ve, \ell_z, q(\theta)).
\end{equation*}
\item $r^{3, a}_{abs}(\ve, \ell_z, q(\cdot)): \left]0, \theta_2(\ve, \ell_z)\right[ \to ]r_H, r^1_s(\ve, \ell_z)[$ defined by 
\begin{equation*}
r^3_{abs}(\ve, \ell_z, q(\theta)), 
\end{equation*}
\item $r^{3, b}_{abs}(\ve, \ell_z, q(\cdot)): \left] \pi - \theta_2(\ve, \ell_z), \pi\right[ \to ]r_H, r^1_s(\ve, \ell_z)[$ defined by 
\begin{equation*}
r^3_{abs}(\ve, \ell_z, q(\theta)), 
\end{equation*}
\end{enumerate}
\item It remains to cover the points $\displaystyle \left(r^i_s(\ve, \ell_z), \theta^i(\ve, \ell_z)\right)$ and $\displaystyle \left(r^i_s(\ve, \ell_z), \pi - \theta^i(\ve, \ell_z)\right)$. To this end, we introduce the following functions 
\begin{enumerate}
\item $\theta^{a}_{abs}(\ve, \ell_z, q(\cdot)): ]r_s^1(\ve, \ell_z), r^2_s(\ve, \ell_z)[\to \left]\theta_1(\ve, \ell_z), \theta_2(\ve, \ell_z)\right[ $ defined by 
\begin{equation*}
\theta^{a}_{abs}(\ve, \ell_z, q(r)) := \left(\left.q\right|_{\left(0, \frac{\pi}{2} \right)}(\cdot, \ve, \ell_z, d)\right)^{-1}(\overline q(r))
\end{equation*}
where $\overline q(r)$ is given by \eqref{q::r} and $\displaystyle \left(\left.q\right|_{\left(0, \frac{\pi}{2} \right)}(\cdot, \ve, \ell_z, d)\right)^{-1}$ is the inverse of the restriction of $q(\cdot, \ve, \ell_z)$ on $\displaystyle \left]0, \frac{\pi}{2}\right[$. 
\item $\theta^{b}_{abs}(\ve, \ell_z, q(\cdot)): ]r_H, r^2_s(\ve, \ell_z)[\to \left]0, \frac{\pi}{2}\right[ $ defined by 
\begin{equation*}
\theta^{b}_{abs}(\ve, \ell_z, q(r)) := \left(\left.q\right|_{\left(0, \frac{\pi}{2} \right)}(\cdot, \ve, \ell_z)\right)^{-1}(\overline q(r))
\end{equation*}
where $\overline q(r)$ is given by \eqref{q::r} and $\displaystyle \left(\left.q\right|_{\left(0, \frac{\pi}{2} \right)}(\cdot, \ve, \ell_z, d)\right)^{-1}$ is the inverse of the restriction of $q(\cdot, \ve, \ell_z)$ on $\displaystyle \left]0, \frac{\pi}{2}\right[$.  In particular, when $r = r_s^1(\ve, \ell_z)$ we have $\displaystyle \theta^{b}_{abs}(\ve, \ell_z, q(r)) = \theta_1(\ve, \ell_z)$. 
\item We also introduce $\displaystyle \pi - \theta^{a}_{abs}(\ve, \ell_z, q(\cdot))$ and $\displaystyle \pi - \theta^{b}_{abs}(\ve, \ell_z, q(\cdot))$ whose images lie in $\displaystyle \left]\frac{\pi}{2}, \pi\right[$. 
\end{enumerate}
\item Therefore, $Z^K(\ve, \ell_z)$ is given by 
\begin{equation*}
\begin{aligned}
Z^K(\ve, \ell_z) &= Graph(r^1_{abs}(\ve, \ell_z, q(\cdot)))\cup Graph(r^{2, a}_{abs}(\ve, \ell_z, q(\cdot)))\cup Graph(r^{2, b}_{abs}(\ve, \ell_z, q(\cdot))) \\
&\cup Graph(r^{3, a}_{abs}(\ve, \ell_z, q(\cdot)))  \cup Graph(r^{3, b}_{abs}(\ve, \ell_z, q(\cdot))) \cup Graph(\theta^{a}_{abs}(\ve, \ell_z, q(\cdot)))  \\
&\cup Graph(\theta^{b}_{abs}(\ve, \ell_z, q(\cdot))) \cup Graph( \pi - \theta^{a}_{abs}(\ve, \ell_z, q(\cdot)))  \\
&\cup  Graph(\pi - \theta^{b}_{abs}(\ve, \ell_z, q(\cdot))). 
\end{aligned}
\end{equation*}
\end{itemize}
\item If $\ell_z = \tilde \ell^+_{min}(\ve, d)$, the above analysis remain valid and we obtain the same result. The only difference between the two cases is that $r_s^1(\ve, \ell_z) = r_s^2(\ve, \ell_z)$ and $q_s^1(\ve, \ell_z) = q_s^2(\ve, \ell_z) = q_m(\ve, \ell_z)$, where, 
\begin{equation*}
 q_m(\ve, \ell_z)  = \ve^2\eta_c(r_m(\ve), \ve^2) \quad\text{and $r_m(\ve)$ is defined in Lemma \ref{critical:l:c}.}
\end{equation*} 
Therefore, $r^i_{abs}$ coincide and we can describe $Z^K(\ve, \ell_z)$ as the graph of one function  $r^1_{abs}(\ve, \ell_z, q(\cdot))$ defined on $(0, \pi)$  which satisfies
\begin{equation*}
\frac{\partial r^1_{abs}(\ve, \ell_z, q(\cdot))}{\partial \theta}(\theta_m(\ve)) = 0 \quad\text{and}\quad \frac{\partial^2 r^1_{abs}(\ve, \ell_z, q(\cdot))}{\partial \theta^2}(\theta_m(\ve)) = 0, 
\end{equation*}
where $\theta_m(\ve)$ is the angle in $\displaystyle \left(0, \frac{\pi}{2}\right)$ which satisfies 
\begin{equation*}
q(\theta) =  q_m(\ve, \ell_z). 
\end{equation*}
\item If $\ell_{lb}^+(\ve)< \ell_z< \ell_{ub}^+(\ve)$. We proceed as in Case 2.b.ii: we have $\forall q\geq 0$,  the solutions in $r$ of \eqref{root:::R} are given by
\begin{equation*}
\left\{
\begin{aligned}
&r^0_{abs}(\ve, \ell_z, q) \quad\text{if} q > \tilde q^+(\ve, \ell_z), \\
&r^0_{abs}(\ve, \ell_z, q) , r^1_{tr}(\ve, \ell_z, q), r^2_{tr}(\ve, \ell_z, q)\quad\text{if}\quad 0 \leq q\leq  \tilde q^+(\ve, \ell_z), \\
\end{aligned}
\right. 
\end{equation*}
These solutions satisfy 
\begin{itemize}
\item $r^1_{tr}(\ve, \ell_z, q) =  r^2_{tr}(\ve, \ell_z, q) = \tilde r^+(\ve, \ell_z)$   if and only if $q =  \tilde q^+(\ve, \ell_z)$, 
\item $\displaystyle r^0_{abs}(\ve, \ell_z, q) < r_{tr}^1(\ve, \ell_z) \leq  \tilde r^+(\ve, \ell_z) \leq r^2_{tr}(\ve, \ell_z, q)$,
\end{itemize}
where  $ \tilde q^+(\ve, \ell_z)$ and $ \tilde r^+(\ve, \ell_z)$ are defined in Lemma \ref{lemma:29}.  Moreover, seen as functions of $q$, $r^0_{abs}(\ve, \ell_z, \cdot)$  and $r^i_{tr}(\ve, \ell_z, q)$ have the following monotonicity properties
\begin{itemize}
\item $r^0_{abs}(\ve, \ell_z, \cdot)$ is monotonically decreasing on $]0, \infty[$ from $r_0^K(\ve, \ell_z)$ to $r_H$,
\item  $r^1_{tr}(\ve, \ell_z, \cdot)$ is monotonically increasing on $]0,  \tilde q^+(\ve, \ell_z) [$ from $r_1^K(\ve, \ell_z)$ to $\tilde r^+(\ve, \ell_z)$,
\item  $r^2_{tr}(\ve, \ell_z, \cdot)$ is monotonically decreasing on $]0,  \tilde q^+(\ve, \ell_z) [$ from $r_2^K(\ve, \ell_z)$ to $\tilde r^+(\ve, \ell_z)$, 
\end{itemize}
Now, we construct an atlas for $Z^K(\ve, \ell_z)$. The properties that $Z^K(\ve, \ell_z)$ has two connected components such that one is diffeomorphic to $\mathbb R$ and the other is diffeomorphic to $\mathbb S^1$  will follow from the construction. To this end, 
\begin{itemize}
\item We recall the angle $\displaystyle \theta\in \left]0, \frac{\pi}{2} \right[$ defined by \eqref{theta::max:p}.
\item We define the following mappings 
\begin{enumerate}
\item 
\begin{equation}
\label{r0:abs:theta}
\begin{aligned}
\SymbolPrint{r^0_{abs}(\ve, \ell_z, q(\cdot))}: &\left]0, \frac{\pi}{2}\right[ \to ]r_H, r_0^K(\ve, \ell_z)] \\
&\theta\mapsto r^0_{abs}(\ve, \ell_z, q(\theta)).
\end{aligned}
\end{equation}
\item
\begin{equation}
\label{r1:tr:theta}
\begin{aligned}
\SymbolPrint{r^1_{tr}(\ve, \ell_z, q(\cdot))}: \left]\overline \theta_{max}^{<1}(\ve, \ell_z), \pi - \overline \theta_{max}^{<1}(\ve, \ell_z) \right[&\to [r_1^K(\ve, \ell_z), \tilde r^+(\ve, \ell_z)[ \\
\theta &\mapsto r^1_{tr}(\ve, \ell_z, q(\theta)), 
\end{aligned}
\end{equation}
\item $\SymbolPrint{r^2_{tr}(\ve, \ell_z, q(\cdot))}: \left]\overline \theta_{max}^{<1}(\ve, \ell_z), \pi - \overline \theta_{max}^{<1}(\ve, \ell_z) \right[\to ]\tilde r^+(\ve, \ell_z), r_2^K(\ve, \ell_z)]$ defined by 
\begin{equation}
\label{r2:tr:theta}
r^2_{tr}(\ve, \ell_z, q(\theta)).  
\end{equation}
\end{enumerate}

\item In order to cover the points  $\displaystyle \left( \tilde r^+(\ve, \ell_z), \overline \theta_{max}^{<1}(\ve, \ell_z)  \right)$ and $\displaystyle \left(\tilde r^+(\ve, \ell_z), \pi - \overline \theta_{max}^{<1}(\ve, \ell_z)  \right)$, we introduce the mapping $\displaystyle \theta_{tr}(\ve, \ell_z, q(\cdot)): ]r_1^K(\ve, \ell_z), r_2^K(\ve, \ell_z)[\to \left[\overline \theta_{max}^{<1}(\ve, \ell_z), \frac{\pi}{2}\right[$ defined by 
\begin{equation}
\label{theta::tr:r}
\theta_{tr}(\ve, \ell_z, q(r)) := \left(\left.q\right|_{\left(0, \frac{\pi}{2} \right)}(\cdot, \ve, \ell_z, d)\right)^{-1}(q(r))
\end{equation}
and its symmetric with respect to the equatorial plane $\pi -  \theta_{tr}(\ve, \ell_z, q(\cdot))$. In particular, when $r = \tilde r^+(\ve, \ell_z)$, we have $\displaystyle \theta_{tr}(\ve, \ell_z, q(r)) =  \overline \theta_{max}^{<1}(\ve, \ell_z) $. 

\item Now, we set 
\begin{equation*}
Z^{K, abs}(\ve, \ell_z) = Graph (r^0_{abs}(\ve, \ell_z, \cdot))
\end{equation*}
and 
\begin{equation*}
\begin{aligned}
Z^{K, trapped}(\ve, \ell_z) &= Graph\left(r^1_{tr}(\ve, \ell_z, q(\cdot))\right) \cup Graph\left(r^2_{tr}(\ve, \ell_z, q(\cdot))\right) \\
&\cup Graph\left(\theta_{tr}(\ve, \ell_z, q(\cdot))\right) \cup Graph\left( \pi -  \theta_{tr}(\ve, \ell_z, q(\cdot))\right). 
\end{aligned}
\end{equation*}
\item In order to show that $\displaystyle Z^{K, abs}(\ve, \ell_z)$ and $\displaystyle Z^{K, trapped}(\ve, \ell_z)$ are disjoint, we use the monotonicity properties of $r^0_{abs}(\ve, \ell_z, \cdot)$ and $r^1_{tr}(\ve, \ell_z, q(\cdot))$ and the fact that 
\begin{equation*}
r_0^K(\ve, \ell_z)< r_1^K(\ve, \ell_z). 
\end{equation*}
\item Moreover, it is easy to see that $Z^{K, trapped}(\ve, \ell_z)$ is closed. 
\end{itemize}

\item If $\ell_z = \ell_{lb}^+(\ve)$, we use the same arguments as in the previous case. The only difference is that when $\ell_z = \ell_{lb}^+(\ve)$ we have 
\begin{equation*}
r_0^K(\ve, \ell_z)=  r_1^K(\ve, \ell_z). 
\end{equation*}
More precisely, when $q = 0$, the equation \eqref{root:::R} admits one double root given by $ r_0^K(\ve, \ell_z)=  r_1^K(\ve, \ell_z) = r_{max}^+(\ve)$ and one simple root given by $r_2^K(\ve, \ell_z)$. 
\\ Therefore, $Z^{K, trapped}(\ve, \ell_z)$ and $Z^{K, abs}(\ve, \ell_z)$ intersect at the point $\left(r_{max}^+(\ve), \frac{\pi}{2}\right)$. 
\end{enumerate}
\end{enumerate}
\item The remain cases follow in the same manner
\begin{enumerate}
\item At a given $(\ve, \ell_z)$, we use Proposition \ref{roots:bounded} and Proposition \ref{roots:unbounded} to determine the  roots to the equation \eqref{root:::R}. These roots are functions of $q$, which can also be seen as a function of $\theta$. 
\item We construct an atlas that covers the set of solutions, i.e.~$Z^K(\ve, \ell_z)$. Locally, it is either the graph of $r$ depending on $\theta$ or of $\theta$ depending on $r$ using similar arguments as above. 
\end{enumerate} 
\end{enumerate}

\end{proof}
\noindent Following the above proposition, we introduce the following sets that will be used along this work. 
\begin{definition}
\label{A:bound:def}
\begin{equation}
\label{Ap::bound}
\begin{aligned}
\SymbolPrint{\mathcal  A_{bound}^+} :=  &\left\{ (\varepsilon, \ell_z)\in\left( \ve^+_{min}, 1\right)\times ({\ell}_{min}^+, +\infty) \;:\; \quad  {\ell}_{lb}^{+}(\varepsilon ) <\ell_z< {\ell}_{ub}^{+}(\ve)\;\right\}, 
\end{aligned}
\end{equation}

\begin{equation}
\label{Am::bound}
\begin{aligned}
\SymbolPrint{\mathcal  A_{bound}^-} :=  &\left\{ (\varepsilon, \ell_z)\in\left( \ve^-_{min}, 1\right)\times (-\infty, {\ell}_{min}^-) \;:\; \quad  {\ell}_{ub}^{-}(\varepsilon) <\ell_z< {\ell}_{lb}^{-}(\varepsilon)\;\right\}, 
\end{aligned}
\end{equation}

\begin{equation}
\label{A:bound}
\SymbolPrint{\Abound} := \Abound^+\cup\Abound^-, 
\end{equation}

\begin{equation}
\label{Ap::scattered}
\begin{aligned}
\SymbolPrint{\mathcal  A^+_{scattered}}:=  &\left\{ (\varepsilon, \ell_z)\in\left[1, \infty \right)\times ({\ell}_{min}^+, +\infty) \;:\; \quad  \ell_z > {\ell}_{lb}^{+}(\varepsilon)\right\}, 
\end{aligned}
\end{equation}

\begin{equation}
\label{Am::scattered}
\begin{aligned}
\SymbolPrint{\mathcal  A^-_{scattered}}:=  &\left\{ (\varepsilon, \ell_z)\in\left[1, \infty \right)\times (-\infty, {\ell}_{min}^-) \;:\; \quad  \ell_z < {\ell}_{lb}^{-}(\varepsilon)\right\}, 
\end{aligned}
\end{equation}

\begin{equation}
\label{A:scat}
\SymbolPrint{\mathcal  A_{scattered}} := \mathcal  A_{scattered}^+\cup\mathcal  A_{scattered}^-, 
\end{equation}

\begin{equation}
\label{Ap::circ}
\begin{aligned}
\mathcal  A^+_{circ}:=  &\left\{\left(\ve^+_{min}, \ell^+_{min}\right) \right\}\cup \left\{ (\varepsilon, \ell_z)\in\left]\ve_{min}^+, \infty \right[\times ({\ell}_{min}^+, +\infty) \;:\; \quad  \ell_z = {\ell}_{lb}^{+}(\varepsilon )\right\}  \\
&\cup \left\{ (\varepsilon, \ell_z)\in\left]\ve_{min}^+, \infty \right[\times ({\ell}_{min}^+, +\infty) \;:\; \quad  \ell_z = {\ell}_{ub}^{+}(\ve) \quad,\quad \ve < 1\right\}, 
\end{aligned}
\end{equation}

\begin{equation}
\label{Am::circ}
\begin{aligned}
\mathcal  A^-_{circ}:=  &\left\{\left(\ve^-_{min}, \ell^-_{min}\right) \right\}\cup \left\{ (\varepsilon, \ell_z)\in\left]\ve_{min}^-, \infty \right[\times ({\ell}_{min}^-, +\infty) \;:\; \quad  \ell_z = {\ell}_{lb}^{-}(\varepsilon)\right\} \\
&\cup \left\{ (\varepsilon, \ell_z)\in\left]\ve_{min}^-, \infty \right[\times (-\infty, {\ell}_{min}^-) \;:\; \quad  \ell_z = {\ell}_{ub}^{-}(\varepsilon) \quad,\quad \ve < 1\right\}, 
\end{aligned}
\end{equation}

\begin{equation}
\label{A:circ}
\SymbolPrint{\mathcal  A_{circ}} := \mathcal  A_{circ}^+\cup\mathcal  A_{circ}^-, 
\end{equation}
We also define the remaining subsets of $]0, \infty[\times\mathbb R$: 

\begin{equation}
\label{A:abs:b}
\SymbolPrint{\mathcal A^{< 1}_{abs}} :=  \left\{(\ve, \ell_z)\in \Adm \;:\; \ve <1 \quad\text{and}\quad (\ve, \ell_z)\notin\left( \Abound \cup\mathcal  A_{circ}\right)  \right\} , 
\end{equation}

\begin{equation}
\label{A:abs:un}
\SymbolPrint{\mathcal A^{\geq 1}_{abs}} :=  \left\{(\ve, \ell_z)\in \Adm \;:\; \ve \geq 1 \quad\text{and}\quad (\ve, \ell_z)\notin\left( \Abound \cup\mathcal  A_{circ}\right)  \right\}, 
\end{equation}

\begin{equation}
\label{A:abs}
\SymbolPrint{\mathcal A_{abs}} := \mathcal A^{< 1}_{abs} \cup \mathcal A^{\geq 1}_{abs}. 
\end{equation}

\end{definition}
\noindent It is easy to see that 
\begin{equation*}
\Adm = \mathcal A_{abs}\sqcup \mathcal  A_{circ} \sqcup  \mathcal  A_{scattered} \sqcup \Abound.
\end{equation*}
\noindent Only $(\ve, \ell_z)$ are required to characterise the zero velocity curves. This allows one to determine the allowed regions for a timelike future-directed geodesic with constants of motion $(\ve, \ell_z)$. However, these constants of motion are not enough to determine the nature of orbits. An orbit with same $(\ve, \ell_z)$ but starting at a two different positions in the $(r, \theta)-$plane can have two different behaviours, as we shall see in the remaining of this section. 
In this context, we note that the above definitions of the different subsets can be misleading. For example, assuming that  $(\ve, \ell_z)\in\Abound$ does not necessarily imply that the orbit is trapped. One needs more assumptions to obtain boundedness. Moreover, there exist trapped geodesics such that $(\ve, \ell_z)\notin\Abound$. 
\\ Our definition of $\Abound$ is such that the associated allowed region has a compact connected component. 
\\A full classification of timelike future directed orbits is achieved by finding a foliation of the mass shell $\Gamma$ indexed by $(\ve, \ell_z, q)$ and by studying the initial position $(t(0), \phi(0), r(0), \theta(0))$(See Lemma \ref{Lemma:::12:bis}).
\begin{figure}[h!]
\includegraphics[width=\linewidth]{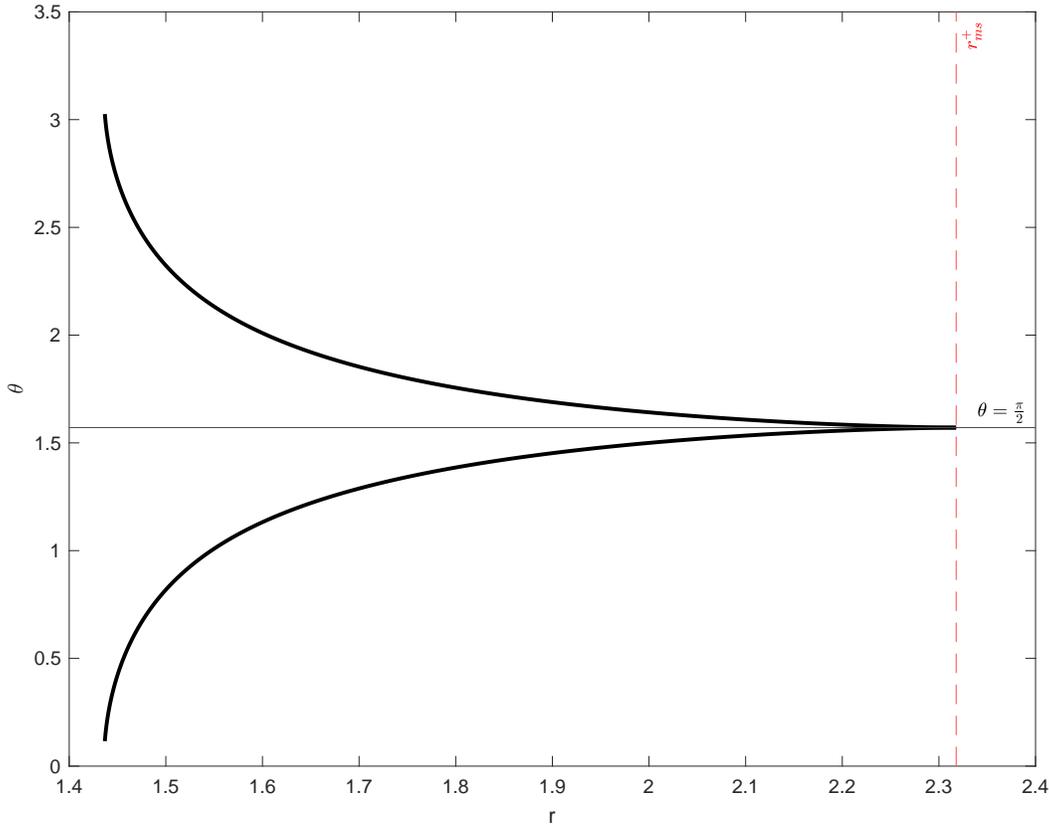}
\caption{{\it Shape of the zero velocity curve associated to a direct orbit with $(\ve, \ell_z) = (\ve_{min}^+, \ell_{min}^+)$  when $d = 0.9$}}
\label{Zcritic}
\end{figure}

\begin{figure}[h!]
\includegraphics[width=\linewidth]{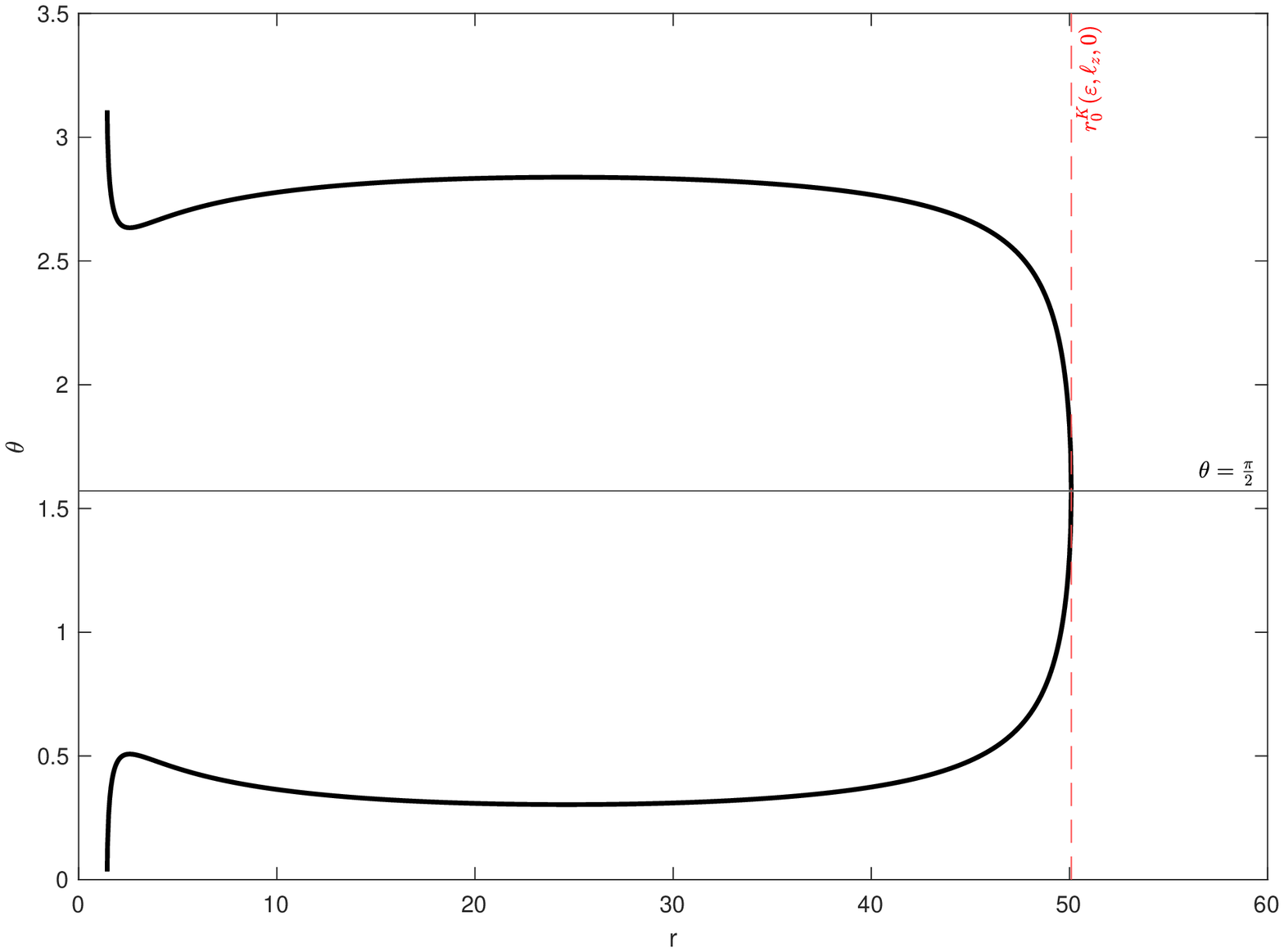}
\caption{{\it Shape of the zero velocity curve associated to a direct orbit with $(\ve, \ell_z) \in \mathcal A_{abs}^{<1}$  when $d = 0.9$}}
\label{Zabsb}
\end{figure}
 
\begin{figure}[h!]
\includegraphics[width=\linewidth]{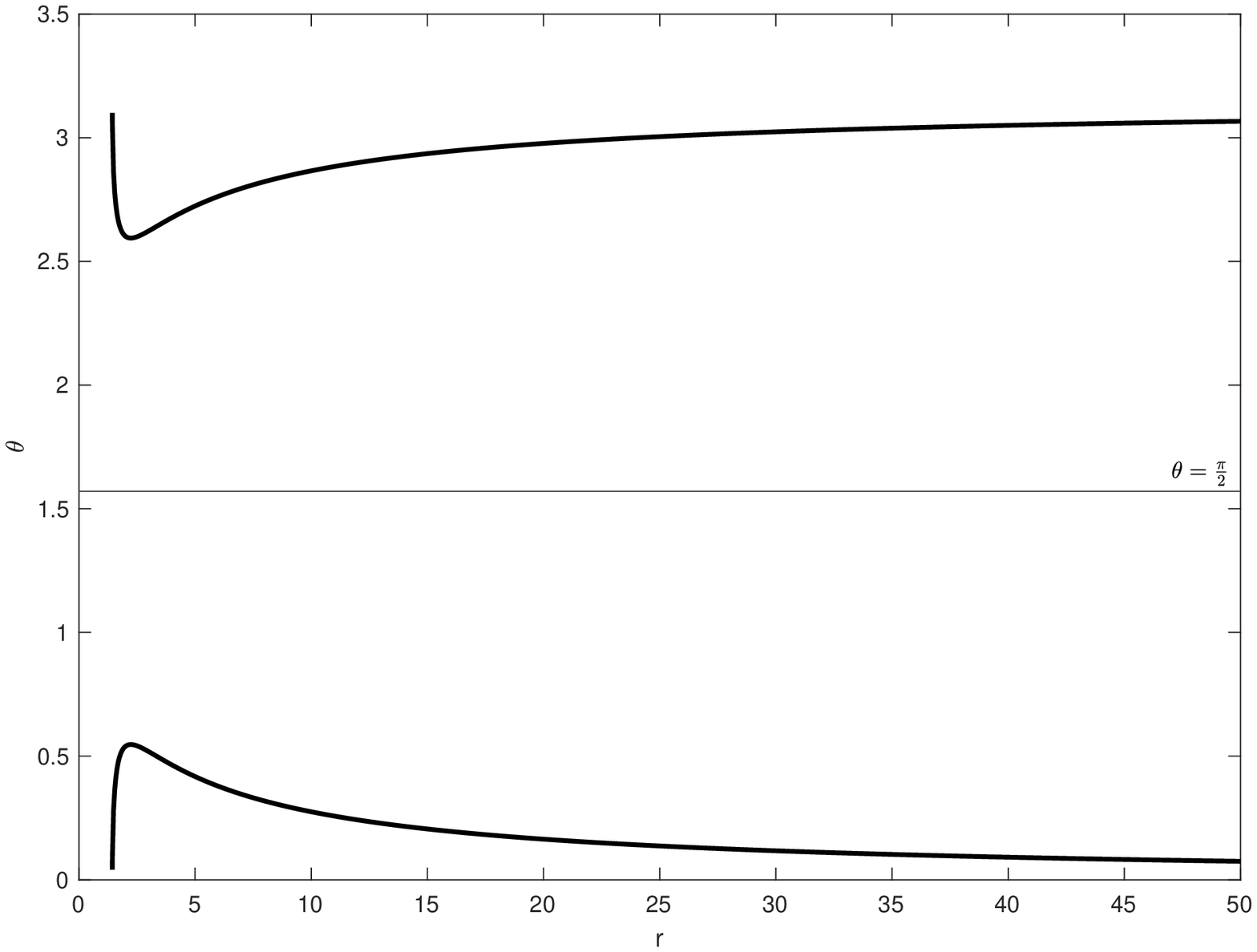}
\caption{{\it Shape of the zero velocity curve associated to a direct orbit with $(\ve, \ell_z) \in \mathcal A_{abs}^{\geq1}$  when $d = 0.9$}}
\label{Zabsub}
\end{figure} 

\begin{figure}[h!]
\includegraphics[width=\linewidth]{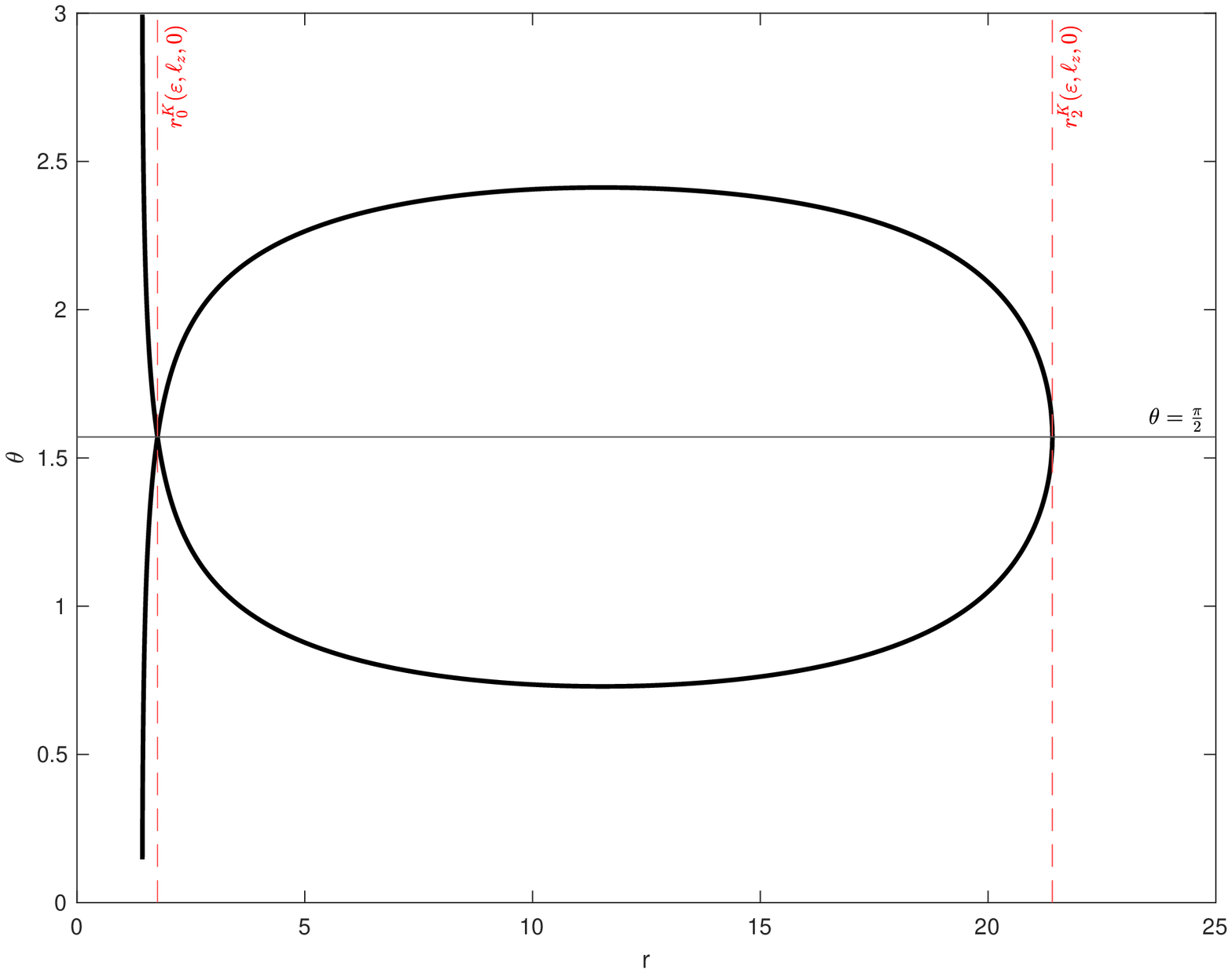}
\caption{{\it Shape of the zero velocity curve associated to a direct orbit with $(\ve, \ell_z) \in \mathcal A_{circ}$ and $\ve< 1$  when $d = 0.9$}}
\label{Zcircb}
\end{figure}

\begin{figure}[h!]
\includegraphics[width=\linewidth]{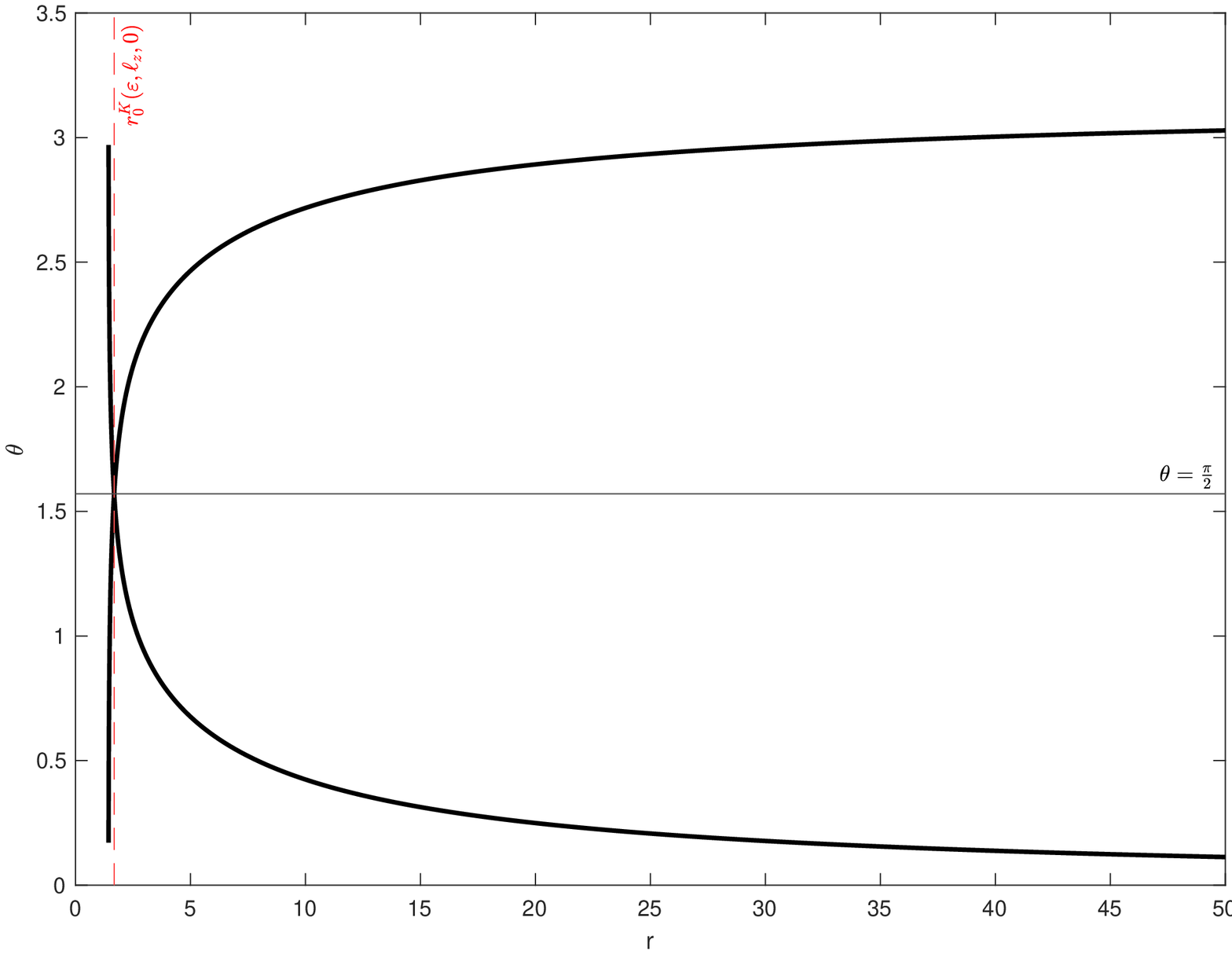}
\caption{{\it Shape of the zero velocity curve associated to a direct orbit with $(\ve, \ell_z) \in \mathcal A_{circ}$ and $\ve\geq 1$  when $d = 0.9$}}
\label{Zcircub}
\end{figure}

\begin{figure}[h!]
\includegraphics[width=\linewidth]{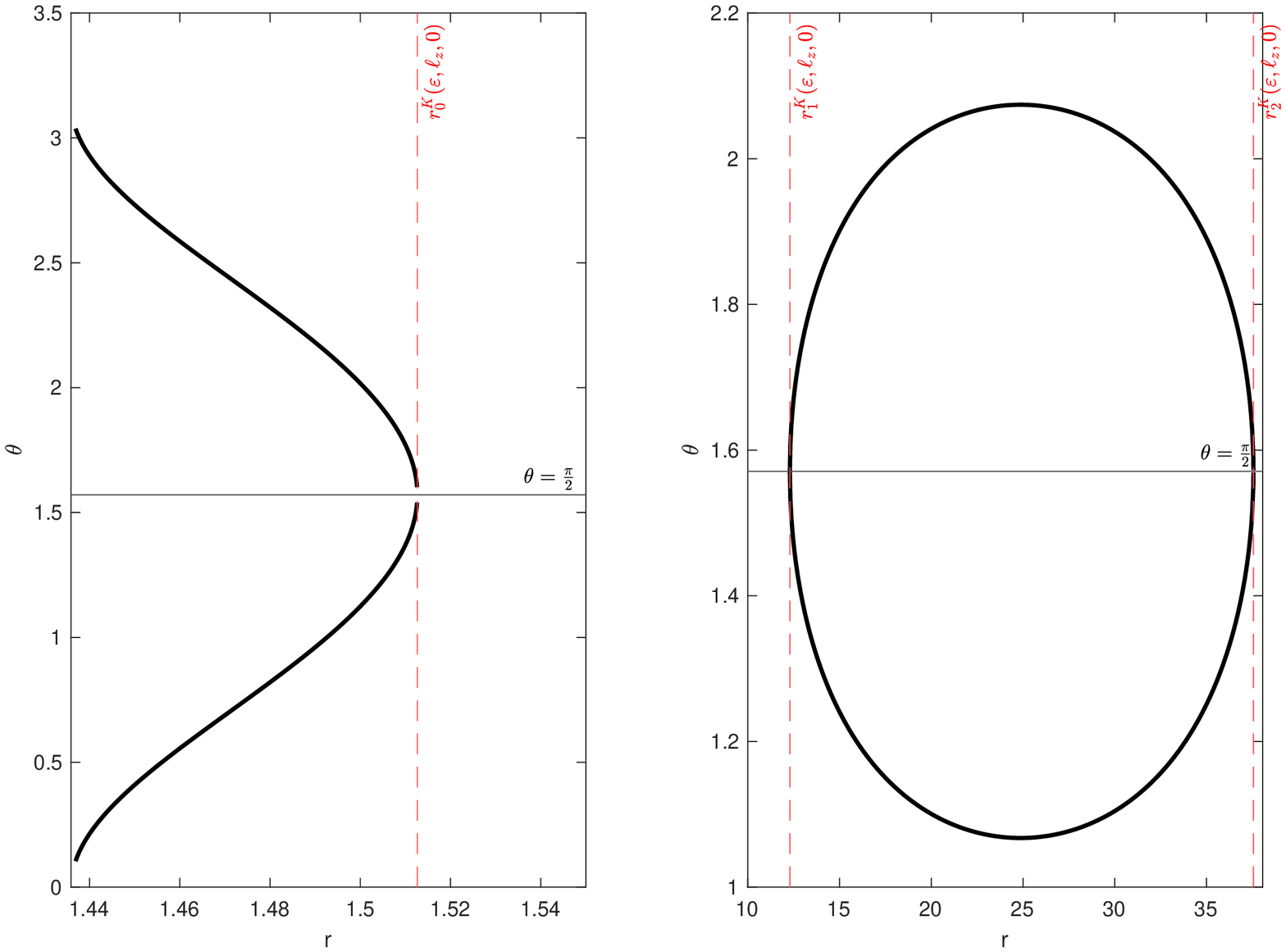}
\caption{{\it Shape of the zero velocity curve associated to a direct orbit with $(\ve, \ell_z) \in \Abound$   when $d = 0.9$}}
\label{Ztrapped}
\end{figure}

\begin{figure}[h!]
\includegraphics[width=\linewidth]{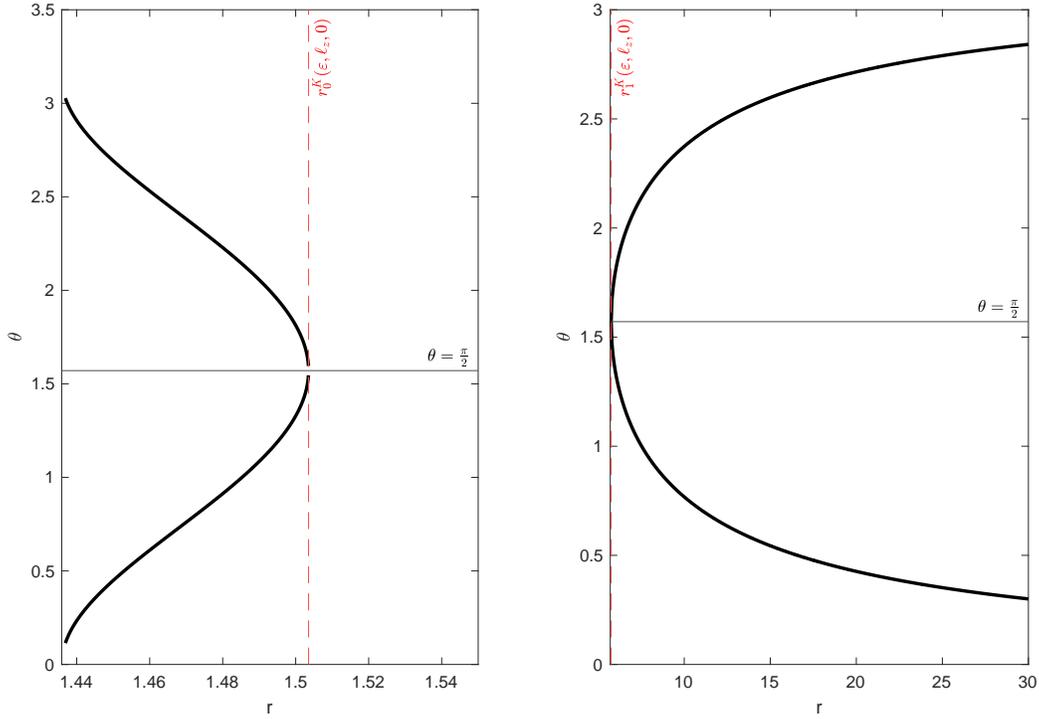}
\caption{{\it Shape of the zero velocity curve associated to a direct orbit with $(\ve, \ell_z) \in \mathcal A_{scattered}$   when $d = 0.9$}}
\label{Zscat}
\end{figure}

\noindent Now, we determine the allowed region $A^K(\ve, \ell_z)$ associated to $\gamma$. By Lemma \ref{ZVC:topology}, we obtain
\begin{lemma} 
\label{denote:A}
Let $(\gamma, I\ni 0)$ be a timelike future-directed geodesic with constants of motion $(\ve, \ell_z)$ and let $A^K(\ve, \ell_z)$ be the allowed region associated to $\gamma$:  
\begin{itemize}
\item If $(\ve, \ell_z)\in \Abound$, $A^K(\ve, \ell_z)$ consists of two  connected components. The component which frontier in $]r_H, \infty[\times]0, \pi[$ is $Z^{K, trapped}(\ve, \ell_z)$ is compact and it will be denoted by $\SymbolPrint{A^{K, trapped}}(\ve, \ell_z)$ and the component which frontier in $]r_H, \infty[\times]0, \pi[$ is $Z^{K, abs}(\ve, \ell_z)$ will be denoted by $\SymbolPrint{A^{K, abs}}(\ve, \ell_z)$
\item If $(\ve, \ell_z)\in \Ascattered$, $A^K(\ve, \ell_z)$ consists of two connected components. The component which frontier in $]r_H, \infty[\times]0, \pi[$ is $Z^{K, scat}(\ve, \ell_z)$ is closed and unbounded and it will be denoted by $\SymbolPrint{A^{K, scat}}(\ve, \ell_z)$ and the component which frontier in $]r_H, \infty[\times]0, \pi[$ is $Z^{K, abs}(\ve, \ell_z)$ will be denoted by $A^{K, abs}(\ve, \ell_z)$. 
\item If $(\ve, \ell_z)\in \mathcal A_{circ}$, then 
\begin{itemize}
\item if $(\ve, \ell_z)\in \mathcal A_{circ}\cap\left\{ (\ve, \ell_z)\;:\; \ve<1 \right\} $, then  $A^K(\ve, \ell_z)$ is the union  of two connected components which intersect at the point $\displaystyle \left(r^\pm_{s}(\ve, \ell_z), \frac{\pi}{2} \right)$. The first component is bounded by $ Z^{K, abs}(\ve, \ell_z)$ and the second component is compact and bounded by $Z^{K, trapped}(\ve, \ell_z)$.
\item if $(\ve, \ell_z)\in \mathcal A_{circ}\cap\left\{ (\ve, \ell_z)\;:\; \ve\geq 1 \right\} $, then  $A^K(\ve, \ell_z)$ consists of two connected components which intersect at the point $\displaystyle \left(r^\pm_{s}(\ve, \ell_z), \frac{\pi}{2} \right)$. The first component is bounded by $ Z^{K, abs}(\ve, \ell_z)$ and the second component is bounded by $Z^{K, scat}(\ve, \ell_z)$.
\end{itemize}
\item If $(\ve, \ell_z)\in \mathcal A^{\leq 1}_{abs}$, $A^K(\ve, \ell_z)$ consists of one connected component which frontier is  given by $Z^K(\ve, \ell_z)$. $A^K(\ve, \ell_z)$ will be denoted by $A^{K, abs}(\ve, \ell_z)$. 
\item Otherwise, $A^K(\ve, \ell_z)$ is $]r_H, \infty[\times]0, \pi[$. In this case, we split $A^K(\ve, \ell_z)$ in three regions
\begin{itemize}
\item $\SymbolPrint{A^{K, z>0}}(\ve, \ell_z)$, the region in $\BB$ located between the first component of the axis of symmetry and $Z^{K, z>0}(\ve, \ell_z)$, 
\item $\SymbolPrint{A^{K, z>0}}(\ve, \ell_z)$, the region in $\BB$ located between the second component of the axis of symmetry and $Z^{K, z<0}(\ve, \ell_z)$
\item $\tilde A^{K, abs} \index[Symbols]{$A^{K, abs}$ @$\tilde A^{K, abs}$}(\ve, \ell_z)$, the remaining region which contains the equatorial plane and which frontier in $\BB$ is given by $Z^{K, z>0}(\ve, \ell_z)\sqcup Z^{K, z<0}(\ve, \ell_z)$
\end{itemize}
\end{itemize}
\end{lemma}

Now, we announce the second main result of this section
\begin{Propo}{(Classification of timelike future-directed geodesics)}
\label{Classification:}
Let $\gamma:I\ni 0\to\mathcal O$ be a timelike future-directed geodesic with constants of motion $(\ve, \ell_z, q)\in\Adm\times \mathbb R$ and let $\tilde\gamma$ be its projection in the $(r, \theta)-$plane. 

\begin{enumerate}
\item If $0<\ve\leq \ve^+_{min}$
\begin{enumerate}
\item if $\ve = \ve^+_{min}$, $\ell_z = \ell^+_{min}$ and $\tilde\gamma$ starts at $\displaystyle \left(r_{ms}^+, \frac{\pi}{2}\right)$, then $\tilde\gamma$ is a circle confined in the equatorial plane. 
\item Otherwise, the orbit $\tilde\gamma$ starts at some point, $(r_0, \theta_0)$ in the region $A^{K, abs}(\ve, \ell_z)\subset $ $]r_H, r_0^K(\ve, \ell_z)]\times]0, \pi[$ and reaches the horizon $r = r_H$ in a finite proper time while oscillating around the equatorial plane between $\theta_0$ and $\pi - \theta_{0}$, where  $r_0^K(\ve, \ell_z)$ is the unique root of the equation \eqref{R::4:0}.
\end{enumerate}
\item If $\ve^+_{min} < \ve\leq \ve^-_{min}$, 
\begin{enumerate}
\item if $\ve = \ve^-_{min}$, $\ell_z = \ell^-_{min}$ and $\tilde\gamma$ starts at $\displaystyle \left(r_{ms}^-, \frac{\pi}{2}\right)$, then $\tilde\gamma$ is a circle confined in the equatorial plane. 
\item Otherwise, 
\begin{enumerate}
\item if $\ell_z<\ell_{lb}^+(\ve)$
\begin{itemize}
\item If $\ell_z=  \tilde \ell_{min}^+(\ve)$ and $\tilde\gamma$ starts at the point $\displaystyle \left(r_{m}(\ve), \theta_m(\ve)\right)$, then $\tilde \gamma$ is spherical of radius $r_{m}(\ve)$. 
\item If $\ell_z<  \tilde \ell_{min}^+(\ve)$ or  $\ell_z=  \tilde \ell_{min}^+(\ve)$  and starts at some point different from $\displaystyle \left(r_{m}(\ve), \theta_m(\ve)\right)$,  then $\tilde\gamma$ starts at some point in  $A^{K, abs}(\ve, \ell_z)$ and reaches the horizon in a finite proper time. See Case 1.b. 
\item Otherwise, i.e~ $\tilde\ell_{min}(\ve)<\ell_z<\ell_{lb}^+(\ve)$, 
\begin{itemize}
\item if  $\tilde\gamma$ starts somewhere in the region $\displaystyle \left\{r^1_s(\ve, \ell_z)\right\}\times]\theta_1(\ve, \ell_z), \pi - \theta_1(\ve, \ell_z)[$ or in the region $\displaystyle \left\{ r^2_s(\ve, \ell_z)\right\}\times]\theta_2(\ve, \ell_z), \pi - \theta_2(\ve, \ell_z)[$, then $\tilde \gamma$ has a constant radius given by $r^i_s(\ve, \ell_z)$  and oscillates between $\theta_i(\ve, \ell_z)$ and $\pi - \theta_i(\ve, \ell_z)$ in the $\theta$ direction, where $r^i_s(\ve, \ell_z)$ are defined in Lemma  \ref{lemma:29} and $\theta_i(\ve, \ell_z)$ are given by \eqref{theta:i}.  
\item If  $\tilde\gamma$ starts somewhere in the region $\displaystyle A^{K, abs}(\ve, \ell_z)\cap(]r_H,  r^1_s(\ve, \ell_z)[\times]0, \pi[)$ or in the region $\displaystyle A^{K, abs}(\ve, \ell_z)\cap(]r^1_s(\ve, \ell_z), r^K_0(\ve, \ell_z)[\times]\theta_1(\ve, \ell_z), \pi - \theta_1(\ve, \ell_z)[)$, then it reaches the horizon in a finite proper time. 
\item Otherwise, $\tilde\gamma$ is trapped. 
\end{itemize}
\end{itemize}
\item if $\ell_z = \ell_{lb}^+(\ve)$, 
\begin{itemize}
\item If  $\tilde\gamma$ starts at $\displaystyle \left(r^K_0(\ve, \ell_z\right), \frac{\pi}{2})$, then $\tilde\gamma$ is a circle confined in the equatorial plane. 
\item Otherwise, 
\begin{itemize}
\item if $\tilde \gamma$ starts at some point in $A^{K, abs}(\ve, \ell_z)$, 
\begin{itemize}
\item if $\tilde \gamma$ is confined to the equatorial plane and starts with positive radial velocity, then  it approaches the circle of radius $r^K_0(\ve, \ell_z)$  in an infinite proper time.
\item Otherwise, $\tilde \gamma$ reaches the horizon in a finite proper time, 
\end{itemize}
\item If, $\tilde \gamma$ starts at some point in $A^{K, trapped}(\ve, \ell_z)$, then it is either trapped or spherical of radius $\tilde r^+(\ve, \ell_z)$. 
\end{itemize}
\end{itemize}

\item if $\ell_{lb}^+(\ve)<\ell_z<\ell_{ub}^+(\ve)$
\begin{itemize}
\item If $\tilde\gamma$ starts at some point in $A^{K, trapped}(\ve, \ell_z)$, then $\tilde\gamma$ is either trapped or spherical with radius $\tilde r^+(\ve, \ell_z)$ given by Lemma \ref{lemma:29}. 
\item Otherwise, $\tilde\gamma$ starts at some point in  $A^{K, abs}(\ve, \ell_z)$ and reaches the horizon in a finite proper time.  
\end{itemize}
\item if $\ell_z = \ell_{ub}^+(\ve)$, 
\begin{itemize}
\item If   $\tilde\gamma$ starts at $\displaystyle \left(r^K_2(\ve, \ell_z), \frac{\pi}{2}\right)$, then $\tilde\gamma$ is a circle confined in the equatorial plane. 
\item Otherwise, $\tilde\gamma$ starts at some point in  $A^{K, abs}(\ve, \ell_z)$ and reaches the horizon in a finite proper time.  
\end{itemize}
\end{enumerate}

\end{enumerate}

\item If $\ve^-_{min} < \ve < 1$, 
\begin{enumerate}
\item if $\ell_{lb}^-(\ve)<\ell_z<\ell_{lb}^+(\ve)$
\begin{itemize}
\item if  $\tilde\gamma$ starts somewhere in the region $\displaystyle \left\{\tilde r^1(\ve, \ell_z)\right\}\times[\tilde \theta_1(\ve, \ell_z), \pi - \tilde \theta_1(\ve, \ell_z)]$ or in the region $\displaystyle \left\{\tilde r^2(\ve, \ell_z)\right\}\times[\tilde \theta_2(\ve, \ell_z), \pi - \tilde \theta_2(\ve, \ell_z)]$, then $\tilde \gamma$ has a constant radius given by $\tilde r^i(\ve, \ell_z)$  and oscillates between $\tilde \theta_i(\ve, \ell_z)$ and $\pi - \tilde\theta_i(\ve, \ell_z)$ in the $\theta$ direction, where $\tilde r^i(\ve, \ell_z)$ is defined in Lemma \ref{lemma:29} and $\tilde\theta_i(\ve, \ell_z)$  is defined by \eqref{tilde:theta:i}. 
\item If  $\tilde\gamma$ starts somewhere in the region $\displaystyle A^{K, abs}(\ve, \ell_z)\cap(]r_H, \tilde r^1(\ve, \ell_z)[\times]0, \pi[)$ or in the region $\displaystyle A^{K, abs}(\ve, \ell_z)\cap(]\tilde r^1(\ve, \ell_z), r_0^K(\ve, \ell_z)[\times]\tilde\theta_1(\ve, \ell_z), \pi - \tilde\theta_1(\ve, \ell_z)[)$, then it reaches the horizon in a finite proper time. 
\item Otherwise, $\tilde\gamma$ is trapped. 
\end{itemize}

\item if $\ell_z = \ell_{lb}^+(\ve)$ or $\ell_z = \ell_{lb}^-(\ve)$
\begin{itemize}
\item If  $\tilde\gamma$ starts at $\displaystyle \left(r^K_0(\ve, \ell_z\right), \frac{\pi}{2})$, then $\tilde\gamma$ is a circle confined in the equatorial plane. 
\item Otherwise, 
\begin{itemize}
\item if $\tilde \gamma$ starts at some point in $A^{K, abs}(\ve, \ell_z)$, 
\begin{itemize}
\item if $\tilde \gamma$ is confined to the equatorial plane and starts with positive radial velocity, then  it approaches the circle of radius $r^K_0(\ve, \ell_z)$  in an infinite proper time.
\item Otherwise, $\tilde \gamma$ reaches the horizon in a finite proper time, 
\end{itemize}
\item If, $\tilde \gamma$ starts at some point in $A^{K, trapped}(\ve, \ell_z)$, then it is either trapped or spherical of radius $r^2_s(\ve, \ell_z)$. 
\end{itemize}
\end{itemize}

\item if $\ell_{lb}^+(\ve)<\ell_z<\ell_{ub}^+(\ve)$ or $\ell_{ub}^-(\ve)<\ell_z<\ell_{lb}^-(\ve)$
\begin{itemize}
\item If $\tilde\gamma$ starts at some point in $A^{K, trapped}(\ve, \ell_z)$, then $\tilde\gamma$ is either trapped or spherical with radius $\tilde r^\pm(\ve, \ell_z)$, given by Lemma \ref{lemma:29}. 
\item Otherwise, $\tilde\gamma$ starts at some point in  $A^{K, abs}(\ve, \ell_z)$ and reaches the horizon in a finite proper time.  
\end{itemize}
\item if $\ell_z = \ell_{ub}^+(\ve)$ or $\ell_z = \ell_{ub}^-(\ve)$, 
\begin{itemize}
\item If   $\tilde\gamma$ starts at $\displaystyle \left(r^K_2(\ve, \ell_z), \frac{\pi}{2}\right)$, then $\tilde\gamma$ is a circle confined in the equatorial plane. 
\item Otherwise, $\tilde\gamma$ starts at some point in  $A^{K, abs}(\ve, \ell_z)$ and reaches the horizon in a finite proper time.  
\end{itemize}

\item If $\ve^2>1$

\begin{enumerate}
\item If $\ell_{lb}^-(\ve)<\ell_z<\ell_{lb}^+(\ve)$, then  
\begin{itemize}
\item If $\tilde\gamma$ starts at the point $\displaystyle  \left\{ \overline r(\ve, \ell_z) \right\}\times[\overline \theta_{max}^{\geq 1}(\ve, \ell_z), \pi - \overline \theta_{max}^{\geq 1}(\ve, \ell_z)]$, then $\tilde\gamma$ has a constant radius and oscillates between $\overline \theta_{max}^{\geq 1}(\ve, \ell_z)$ and $\pi - \overline \theta_{max}^{\geq 1}(\ve, \ell_z)$ in the $\theta$ direction, where $\displaystyle \overline \theta_{max}^{\geq 1}(\ve, \ell_z)$ is given by \eqref{theta:bu}
\item If $\tilde\gamma$ starts at some point in the region $]r_H, \infty[\times[\theta_{min}(\ve, \ell_z), \pi - \theta_{min}(\ve, \ell_z)]\backslash \left\{ \overline r(\ve, \ell_z) \right\}\times[\overline \theta_{max}^{\geq 1}(\ve, \ell_z), \pi - \overline \theta_{max}^{\geq 1}(\ve, \ell_z)]$, then $\tilde\gamma$ remains in the region $A^{K, q\geq 0}$. 
\begin{itemize}
\item if $\tilde\gamma$ starts with a negative radial velocity, then it reaches the horizon in a finite proper time while oscillating around the equatorial plane.
\item if $\tilde\gamma$ starts with a positive radial velocity, then it goes to infinity while oscillating around the equatorial plane.
\end{itemize}
\item If $\tilde\gamma$ starts in the region $A^{K, z>0}(\ve, \ell_z)$ or  in the region $A^{K, z<0}(\ve, \ell_z)$.  Then, 
\begin{itemize}
\item if $\tilde\gamma$ starts with a negative radial velocity, then it reaches the horizon in a finite proper time while staying inside $A^{K, z>0}(\ve, \ell_z)$ or $A^{K, z<0}(\ve, \ell_z)$. 
\item if $\tilde\gamma$ starts with a positive radial velocity, then it goes to infinity while staying in the region $A^{K, z>0}(\ve, \ell_z)$ or $A^{K, z<0}(\ve, \ell_z)$. 
\end{itemize}
\end{itemize}
\item If $\ell_z = \ell_{lb}^+(\ve)$ or $\ell_z = \ell_{lb}^-(\ve)$, then 
\begin{itemize}
\item If  $\tilde\gamma$ starts at $\displaystyle \left(r^K_0(\ve, \ell_z), \frac{\pi}{2}\right)$, then $\tilde\gamma$ is a circle confined in the equatorial plane. 
\item Otherwise, 
\begin{itemize}
\item If $\tilde \gamma$ starts at some point in $A^{K, abs}(\ve, \ell_z)$ 
\begin{itemize}
\item If $\gamma$ is confined in the equatorial plane and starts with a positive radial velocity, then it  approaches the circle of radius $r^K_0(\ve, \ell_z)$ in an infinite proper time, 
\item otherwise, it reaches the horizon in a finite proper time. 
\end{itemize}
\item otherwise, i.e~ if $\tilde \gamma$ starts at some point in $A^{K, scattered}(\ve, \ell_z)$, then 
\begin{itemize}
\item If $\gamma$ is confined in the equatorial plane and starts with a negative radial velocity, then it  approaches the circle of radius $r^K_0(\ve, \ell_z)$ in an infinite proper time.
\item Otherwise, the orbit (with negative initial radial velocity) hits a potential barrier and goes back infinity while oscillating around the equatorial plane or (with positive initial radial velocity) goes to infinity while trapped in the $\theta-direction$ between $\theta(0)$ and $\pi - \theta(0)$.
\end{itemize}

approaches the circle of radius $r^K_0(\ve, \ell_z, 0)$ and confined in the equatorial plane in an infinite proper time. 
\end{itemize}
\end{itemize}
\item If $\ell_z>\ell_{lb}^+(\ve)$ or $\ell_z<\ell_{lb}^-(\ve)$, then 
\begin{itemize}
\item If $\tilde\gamma$ starts at some point in the region $A^{K, scattered}(\ve, \ell_z)$, then the orbit (with negative initial radial velocity) hits a potential barrier and goes back infinity while oscillating around the equatorial plane between...  or (with positive initial radial velocity) goes to infinity while trapped in the $\theta-direction$. 
\item Otherwise, $\tilde\gamma$ starts at some point in the region $A^{K, abs}(\ve, \ell_z)$. It reaches the horizon in a finite proper time while oscillating around the equatorial plane.  
\end{itemize}

\end{enumerate}

\end{enumerate}

\end{enumerate}
\end{Propo}

\begin{proof}
Let $\gamma: I\to\mathcal O$ be a timelike future-directed geodesic with initial conditions $(\gamma(0), \dot\gamma(0))$. 
First of all, recall that by Lemma \ref{Lemma:::12:bis}, the nature of  $\gamma$ is determined by the set $(\ve, \ell_z, q)$ and $\gamma(0) = (t(0), \phi(0), r(0), \theta(0))$. Moreover, due to spacetime symmetries, it suffices to determine the nature of  its projection in the $(r, \theta)-$plane. We compute $(\ve, \ell_z, q)$ from the initial condition by: 
\begin{equation*}
\ve = V_K(r(0), \theta(0))v^t(0) - W_K(r(0), \theta(0))v^\phi(0),
\end{equation*}
\begin{equation*}
\ell_z = W_K(r(0), \theta(0))v^t(0) + X_K(r(0), \theta(0))v^\phi(0),
\end{equation*}
and 
\begin{equation*}
q = v_\theta(0)^2 + \cos^2\theta(0)\left(d^2(1-\ve^2) + \frac{\ell_z^2}{\sin^2\theta(0)} \right),\footnote{constraints on $q$ can be seen as restriction on the angular direction.} 
\end{equation*} 
Now, we study  the possible maximal solutions of the reduced system \eqref{reduced:reduced} for all   $(\ve, \ell_z)\in\Adm$. 
\\ Let $(\tilde r, \tilde\theta)\in]r_H, \infty[\times]0, \pi[$ and let $(\tilde v_r, \tilde v_\theta)\in \mathbb R^2$ and consider the  following Cauchy problem 
\begin{equation}
\label{Cauchy:r:theta}
\left\{ 
\begin{aligned}
\frac{dr}{d\tau} &= \frac{\Delta}{\Sigma^2}v_r ,\\
\frac{dv_r}{d\tau}&= \frac{1}{2\Sigma^2}\left(-\Delta'(r)v_r^2 + \frac{\Delta'(r)R(r, \ve, \ell_z, q) - \Delta(r)\partial_r R(r, \ve, \ell_z, q)}{\Delta(r)^2} \right), \\
\frac{d\theta}{d\tau} &= \frac{1}{\Sigma^2}v_\theta, \\
\frac{dv_\theta}{d\tau}&= \frac{1}{2\Sigma^2}\left( \partial_{\theta}\left( \frac{T(\cos\theta, \ve, \ell_z, q)}{\sin^2\theta}\right) \right)
\end{aligned}
\right. 
\end{equation}
with initial conditions 
\begin{equation*}
\begin{aligned}
\tilde\gamma(0) &= (\overline r, \overline\theta) , \\
\dot{\tilde\gamma}(0) &= (\overline v_r, \overline v_\theta). 
\end{aligned}
\end{equation*}
By Cauchy-Lipschitz theorem, there exists a unique maximal solution for the above system given by $(\tilde\gamma, \dot{\tilde\gamma}, I\ni 0)$. We decouple the equations for $(r, v_r)$ and for $(\theta, v_\theta)$, we use Mino time, $\Lambda$, defined by \eqref{mino:mino}.  
We have 
\begin{equation*}
\forall\tau\in I\;:\; \Lambda(\tau) = \int_{0}^\tau\frac{1}{\Sigma^2(r(s), \theta(s))}\,ds. 
\end{equation*}
Now set $J := \Lambda(I)$. Then, the above system becomes: $\forall \lambda\in J$
\begin{equation*}
\label{reduced::system::bis}
\left\{ 
\begin{aligned}
\frac{dr}{d\lambda} &= \Delta(r(\Lambda^{-1}(\lambda)))v_r(\Lambda^{-1}(\lambda)) ,\\
\frac{dv_r}{d\lambda}&= \frac{1}{2}\left(-\Delta'(r(\Lambda^{-1}(\lambda)))v_r^2(\Lambda^{-1}(\lambda)) \right. \\
&\left. + \frac{\Delta'(r(\Lambda^{-1}(\lambda)))R(r(\Lambda^{-1}(\lambda)), \ve, \ell_z, q) - \Delta(r(\Lambda^{-1}(\lambda)))\partial_r R(r(\Lambda^{-1}(\lambda)), \ve, \ell_z, q)}{\Delta(r(\Lambda^{-1}(\lambda)))^2} \right), \\
\frac{d\theta}{d\lambda} &= v_\theta(\Lambda^{-1}(\lambda)), \\
\frac{dv_\theta}{d\lambda}&= \frac{1}{2}\left( \partial_{\theta}\left( \frac{T(\cos(\theta(\Lambda^{-1}(\lambda))), \ve, \ell_z, q)}{\sin^2(\theta(\Lambda^{-1}(\lambda)))}\right) \right)
\end{aligned}
\right. 
\end{equation*}
In the following, we identify $r\circ\Lambda^{-1}$, $\theta\circ\Lambda^{-1}$, $v_r\circ\Lambda^{-1}$ and $v_\theta\circ\Lambda^{-1}$with $r$, $\theta$, $v_r$ and $v_\theta$ respectively. 
We consider the two Cauchy problems: 
\begin{equation}
\label{Cauchy::r}
\left\{ 
\begin{aligned}
\frac{dr}{d\lambda} &= \Delta(r(\lambda))v_r(\lambda) ,\\
\frac{dv_r}{d\lambda}&= \frac{1}{2}\left(-\Delta'(r(\lambda))v_r^2(\lambda) + \frac{\Delta'(r(\lambda))R(r(\lambda), \ve, \ell_z, q) - \Delta(r(\lambda))\partial_r R(r(\lambda), \ve, \ell_z, q)}{\Delta(r(\lambda))^2} \right), \\
r(0) &= \overline r \;\;,\;\; v_r(0) = \overline v_r
\end{aligned}
\right. 
\end{equation}
and 
\begin{equation}
\label{Cauchy::theta}
\left\{ 
\begin{aligned}
\frac{d\theta}{d\lambda} &= v_\theta(\lambda), \\
\frac{dv_\theta}{d\lambda}&= \frac{1}{2}\left( \partial_{\theta}\left( \frac{T(\cos(\theta(\lambda)), \ve, \ell_z, q)}{\sin^2(\theta(\lambda))}\right) \right) \\
\theta(0) &= \overline \theta \;\;,\;\; v_\theta(0) = \overline v_\theta. 
\end{aligned}
\right. 
\end{equation}
\noindent It is easy to see that at a given $(\ve, \ell_z, q)$,  if $\left(r, \theta, v_r, v_\theta, I\right)$ is the maximal solution of \eqref{Cauchy:r:theta}, then $(r, v_r, J)$ and $(\theta, v_\theta, J )$ are the maximal solutions of \eqref{Cauchy::r} and \eqref{Cauchy::theta} respectively. Reciprocally, if $(r, v_r, J_r)$ and $(\theta, v_\theta, J_\theta )$ are the maximal solutions of \eqref{Cauchy::r} and \eqref{Cauchy::theta} respectively, then $(r, \theta, v_r, v_\theta, I = \Lambda^{-1}(J_r\cap J_\theta))$. Consequently, we classify the maximal solutions of  \eqref{Cauchy::r} and \eqref{Cauchy::theta} in order to obtain the general classification. 
\\ In the following, we tackle in details the case when $(\ve, \ell_z)\in\Abound$. The remaining cases are similar. 
\begin{enumerate}
\item If $(\ve, \ell_z)\in\Abound$,  then $\forall \lambda\in J := J_r\cap J_\theta,\; (r, \theta)(\lambda)\in A^{K, abs}(\ve, \ell_z) \sqcup A^{K, trapped}(\ve, \ell_z)$ and more precisely, $\forall \lambda\in J$
\begin{equation*}
r(\lambda)\in ]r_H, r_0^K(\ve, \ell_z, q)]\sqcup [r_1^K(\ve, \ell_z, q), r_2^K(\ve, \ell_z, q)] \quad\text{and}\quad \theta(\lambda)\in [\theta_{min}(\ve, \ell_z, q, d), \pi - \theta_{min}(\ve, \ell_z, q, d)]
\end{equation*}
where $r_0^K (\ve, \ell_z, q), r_1^K (\ve, \ell_z, q)$ and $r_2^K(\ve, \ell_z, q)$ are the  roots of the equation \eqref{R::4:0} and $\displaystyle \theta_{min}(\ve, \ell_z, q, d)$ is the unique angle in $\left]0, \frac{\pi}{2}\right]$ that solves \eqref{T::4:0}.  
\begin{enumerate}
\item Let $(r, v_r, J_r)$ be the maximal solution of \eqref{Cauchy::r}. Then, $\forall \lambda\in J_r,$ we have 
\begin{equation*}
2\frac{dv_r}{d\lambda}= -\Delta'(r(\lambda))v_r^2(\lambda) + \frac{\Delta'(r(\lambda))R(r(\lambda), \ve, \ell_z, q) - \Delta(r(\lambda))\partial_r R(r(\lambda), \ve, \ell_z, q)}{\Delta(r(\lambda))^2}.
 \end{equation*}
 We multiply the latter by $v_r$ to obtain 
\begin{equation*}
\frac{d(v_r)^2}{d\lambda}= -\Delta'(r(\lambda))v_r^3(\lambda) + \frac{\Delta'(r(\lambda))R(r(\lambda), \ve, \ell_z, q) - \Delta(r(\lambda))\partial_r R(r(\lambda), \ve, \ell_z, q)}{\Delta(r(\lambda))^2}v_r.
 \end{equation*}
 By the first equation of \eqref{Cauchy::r}, we have 
 \begin{equation*}
\frac{d(v_r)^2}{d\lambda}= -\Delta'(r(\lambda))v_r^2(\lambda)\dot r\frac{1}{\Delta} + \frac{\Delta'(r(\lambda))R(r(\lambda), \ve, \ell_z, q) - \Delta(r(\lambda))\partial_r R(r(\lambda), \ve, \ell_z, q)}{\Delta(r(\lambda))^2}\frac{\dot r}{\Delta}. 
 \end{equation*}
 Therefore, 
 \begin{equation*}
\frac{d}{d\lambda}(\Delta v_r^2) = \frac{d}{d\lambda}\left(\frac{R(r, \ve, \ell_z, q)}{\Delta} \right). 
 \end{equation*}
 Finally we integrate between $0$ and some $\lambda\in J_r$ to obtain 
 \begin{equation*}
 \Delta^2 v_r^2 = R(r, \ve, \ell_z, \ve, \ell_z) + (\Delta^2(\overline r)\overline v_r - R(\overline r, \ve, \ell_z, \ve, q)) 
 \end{equation*}
 Now, we have 
 \begin{equation*}
 (\Delta^2(\overline r)\overline v_r - R(\overline r, \ve, \ell_z, \ve, q))  = 0. 
 \end{equation*}
 Therefore, 
 \begin{equation*}
  \Delta^2 v_r^2 = R(r, \ve, \ell_z, q). 
 \end{equation*}
\item 
\end{enumerate}

\begin{itemize}

\item If $(\tilde r, \tilde \theta)\in A^{K, abs}(\ve, \ell_z)$. Then, $\forall \lambda\in J_r\cap J_\theta$
\begin{equation*}
r(\lambda)\in]r_H, r_0^K(\ve, \ell_z, q)] \quad \text{and} \quad \theta(\lambda)\in[\theta_{min}(\ve, \ell_z, q), \pi - \theta_{min}(\ve, \ell_z, q)].
\end{equation*} 
Let $J_r = ]T_{min}, T^{max}[$. We claim that 
\begin{equation*}
-\infty<T_{min} \quad,\quad T_{max}<+\infty
\end{equation*}
and 
\begin{equation*}
\lim_{\lambda\to T_{max}}\, r(\lambda) =  \lim_{\lambda\to T_{min}}\, r(\lambda) = r_H. 
\end{equation*}

\begin{enumerate}
\item If $(\overline r, \overline\theta)\in\partial A^{K, abs}(\ve, \ell_z)$, then $(\overline v_r, \overline v_\theta) = (0, 0)$ and 
\begin{equation*}
\begin{aligned}
\overline r &= r^K_0(\ve, \ell_z, q) , \\
\overline\theta &\in\left\{\theta_{min}(\ve, \ell_z, q), \pi - \theta_{min}(\ve, \ell_z, q)\right\}.
\end{aligned}
\end{equation*}
We claim that $\forall \lambda\in J_r\backslash\left\{0\right\}$, $$ R(r(\lambda))>0. $$ In fact, suppose that there exists $ \tilde\lambda\in]T^{min}, 0[\cup]0, T^{max}[$ such that $\displaystyle R(r(\tilde\lambda)) = 0. $ Then, $$ (\Delta^{-1}(r(\cdot))\dot r)(\tilde\lambda) = 0. $$
Besides, we have $$ (\Delta^{-1}(r(\cdot))\dot r)(0) = 0. $$
By Rolle's theorem, there exists $\lambda_0\in]0, \tilde\lambda [$ such that  $$ \frac{d}{d\lambda}(\Delta^{-1}(r(\cdot))\dot r)(\lambda_0) = 0. $$ Therefore, by the second equation of \eqref{Cauchy::r}, $$ \partial_r R(r(\lambda_0), \ve, \ell_z, q) = 0$$. Contradiction since the $r-$derivative of $R$ is negative on $]r_H, r^K_0(\ve, \ell_z, q)]$. 
\\ This implies
\begin{equation*}
\forall \lambda\in J_r\backslash\left\{0\right\} \;,\; \frac{dr}{d\lambda} = \pm\sqrt{R(r(\lambda), \ve, \ell_z, q)}. 
\end{equation*}
More precisely, 
\begin{equation*}
\begin{aligned}
&\forall \lambda\in]0, T^{max}[\;,\;  \dot r = -\sqrt{R(r(\lambda), \ve, \ell_z, q)} ,\; 
&\forall \lambda\in]T^{min}, 0[\;,\;  \dot r = \sqrt{R(r(\lambda), \ve, \ell_z, q)}. \; 
\end{aligned}
\end{equation*}
Now, we claim that 
\begin{equation*}
T^{max} = \int_{r_H}^{r_0^K(\ve, \ell_z, q)}\;\frac{1}{\sqrt{R(r(s), \ve, \ell_z, q)}}\,ds
\end{equation*}
and 
\begin{equation*}
T^{min} = -\int_{r_H}^{r_0^K(\ve, \ell_z, q)}\;\frac{1}{\sqrt{R(r(s), \ve, \ell_z, q)}}\,ds. 
\end{equation*}
Let $\lambda\in]0, T^{max}[$. We have 
\begin{equation*}
1 = -\frac{\dot r(\lambda)}{{\sqrt{R(r(\lambda), \ve, \ell_z, q)}}}
\end{equation*}
We integrate between $0$ and $\lambda$ to obtain 
\begin{equation*}
\lambda = -\int_0^{\lambda}\;\frac{\dot r(s)}{{\sqrt{R(r(s), \ve, \ell_z, q)}}}\,ds. 
\end{equation*}
Now we make the change of variable $\displaystyle u = r(s)$ in the right hand side. We obtain 
\begin{equation*}
\lambda = \int_{r_H}^{r_0^K(\ve, \ell_z, q)}\;\frac{du}{{\sqrt{R(u, \ve, \ell_z, q)}}}. 
\end{equation*}
We introduce the function $G(\cdot, \ve, \ell_z, q)$ defined on $]r_H, r^K_0(\ve, \ell_z, q)]$ by  
\begin{equation*}
G(s, \ve, \ell_z, q):= \int_s^{r^K_0(\ve, \ell_z, q)}\;\frac{du}{{\sqrt{R(u, \ve, \ell_z, q)}}}.
\end{equation*}
Since $r^K_0(\ve, \ell_z, q)$ is a simple root of $R$, $G$ is well defined. Moreover $G(\cdot, \ve, \ell_z, q)$ is monotonically decreasing on $]r_H, r^K_0(\ve, \ell_z, q)]$. Therefore, it is bijective from $]r_H, r^K_0(\ve, \ell_z, q)]$ to $[0, T^{max}[$ where 
\begin{equation*}
T^{max} := \lim_{r\to r_H} G(s) = \int_{r_H}^{r_0^K(\ve, \ell_z, q)}\;\frac{1}{\sqrt{R(r(s), \ve, \ell_z, q)}}\,ds <+\infty. 
\end{equation*}
Hence, 
\begin{equation*}
\forall\lambda\in[0, T^{max}[\;,\; r(\lambda) = G^{-1}(\lambda). 
\end{equation*}
In the same manner, we obtain 
\begin{equation*}
T^{min} := \lim_{r\to r_H} H(s) = -\int_{r_H}^{r_0^K(\ve, \ell_z, q)}\;\frac{1}{\sqrt{R(r(s), \ve, \ell_z, q)}}\,ds >-\infty 
\end{equation*}
and 
\begin{equation*}
\forall\lambda\in]T^{min}, 0]\;,\; r(\lambda) = H^{-1}(\lambda)
\end{equation*}
where $H(\cdot, \ve, \ell_z, q)$ is the function defined on $]r_H, r^{K}_0(\ve, \ell_z, q)]$ by 
\begin{equation*}
H(s, \ve, \ell_z, q):= -\int_s^{r^K_0(\ve, \ell_z, q)}\;\frac{du}{{\sqrt{R(u, \ve, \ell_z, q)}}}.
\end{equation*}
It remains to analyse the motion in the $\theta$-direction in order to determine $J_\theta$:  $\forall \lambda\in J_{\theta}\;,\; \theta(\lambda)\in \tilde\theta\in[\theta_{min}(\ve, \ell_z, q), \pi - \theta_{min}(\ve, \ell_z, q)]$. By compactness, $J_\theta = \mathbb  R$. Moreover, $\lambda\to \theta(\lambda)$ is periodic with period
\begin{equation}
\label{T::theta}
T_\theta := 2\int_{\theta_{min}(\ve, \ell_z, q)}^{\pi - \theta_{min}(\ve, \ell_z, q)}\; \frac{\sin\sigma}{\sqrt{T(\cos\sigma, \ve, \ell_z, q)}}\,d\sigma. 
\end{equation}
In fact, let $(\theta, v_\theta, J_\theta)$ be the maximal solution of \eqref{Cauchy::theta}. Introduce the function $\tilde G_1(\cdot, \ve, \ell_z, q)$ defined on $[\theta_{min}(\ve, \ell_z, q), \pi - \theta_{min}(\ve, \ell_z, q)]$ by 
\begin{equation*}
\tilde G_1 (s, \ve, \ell_z, q) := \int_{\theta_{min}(\ve, \ell_z, q)}^s\frac{\sin\sigma}{\sqrt{T(\cos\sigma, \ve, \ell_z, q)}}\,d\sigma
\end{equation*} 
Since $\cos(\theta_{min}(\ve, \ell_z, q))$ is a simple root of $T$, $\tilde G(\cdot, \ve, \ell_z, q)$ is well-defined. Moreover, it is monotonically increasing on its domain so that it defines a bijection from $[\theta_{min}(\ve, \ell_z, q), \pi - \theta_{min}(\ve, \ell_z, q)]$ to $\left[0, \frac{T_\theta}{2}\right]$. Now we denote its inverse by $\tilde G^{-1}_1(\cdot, \ve, \ell_z, q)$. In the same way, we define the bijective function $\tilde G_2(\cdot, \ve, \ell_z, q)$ defined from $[\theta_{min}(\ve, \ell_z, q), \pi - \theta_{min}(\ve, \ell_z, q)]$ to $\left[\frac{T_\theta}{2},  T_\theta\right]$ by 
\begin{equation*}
\tilde G_2 (s, \ve, \ell_z, q) := \frac{T_\theta}{2} + \int^{\pi - \theta_{min}(\ve, \ell_z, q)}_s\frac{\sin\sigma}{\sqrt{T(\cos\sigma, \ve, \ell_z, q)}}\,d\sigma
\end{equation*}
and we denote its inverse by $\tilde G_2^{-1} (\cdot, \ve, \ell_z, q)$. Now we define $\tilde\theta$ on $[0, T_\theta]$ by 
\begin{equation*}
\tilde\theta(\lambda) := \left\{
\begin{aligned}
& \tilde G^{-1}_1(\lambda, \ve, \ell_z, q) \;\text{if}\;  \lambda\in\left[0, \frac{T_\theta}{2}\right]\\
& \tilde G^{-1}_2(\lambda, \ve, \ell_z, q) \;\text{if}\;  \lambda\in\left[\frac{T_\theta}{2}, T_\theta\right]. 
\end{aligned}
\right. 
\end{equation*}
and $\tilde v_\theta$ on $[0, T_\theta]$ by 
\begin{equation*}
\tilde v_\theta := \theta'(\lambda). 
\end{equation*}
Now, we can extend $(\tilde \theta, \tilde v_\theta)$ to a periodic solution defined on $\mathbb R$.  Moreover, It easy to see that $(\tilde\theta, \tilde v_\theta)$ satisfies\eqref{Cauchy::theta}. By uniqueness, $\theta$ is periodic with period $T_\theta$. 
\item If $(\overline r, \overline\theta)\in A^{K, abs}(\ve, \ell_z)\backslash(\partial A^{K, abs}(\ve, \ell_z))$. Then, we proceed as above to obtain the same result.
\end{enumerate}

\item If $(\tilde r, \tilde \theta)\in A^{K, trapped}(\ve, \ell_z)$.  Suppose that $(\tilde r, \tilde \theta)\in \partial A^{K, trapped}(\ve, \ell_z)$. The other cases are studied in the same manner. 
\begin{equation*}
\tilde r\in[r_1^K(\ve, \ell_z, q), r_2^K(\ve, \ell_z, q)] \quad \text{and} \quad \tilde\theta\in[\theta_{min}(\ve, \ell_z, q), \pi - \theta_{min}(\ve, \ell_z, q)]
\end{equation*} 
Therefore, $\displaystyle \forall \lambda\in J\;,\; (r, \theta)(\lambda)\in [r_1^K(\ve, \ell_z, q), r_2^K(\ve, \ell_z, q)]\times [\theta_{min}(\ve, \ell_z, q, d), \pi - \theta_{min}(\ve, \ell_z, q, d)]$ for all $\tau\in I$. By compactness, $J = \mathbb R$ and $\tilde\gamma$ is trapped. Moreover, $r$ is periodic with period: 
\begin{equation}
\label{T::r}
T_r := \int_{r_1^K(\ve, \ell_z, q)}^{r_2^K(\ve, \ell_z, q)}\; \frac{dr}{\sqrt{R(r, \ve, \ell_z, q)}}. 
\end{equation}
and $\theta$ is periodic with period $T_\theta$ defined by \eqref{T::theta}. For the periodicity of $\theta$, it has already been tackled in the first case. As for the periodicity of $r$, we proceed in the same way: let $(r, v_r,  J_r)$ be the maximal solution of \eqref{Cauchy::r}. Introduce the function $\overline G_1(\cdot, \ve, \ell_z, q)$ defined on $[r^K_1(\ve, \ell_z, q), r^K_2(\ve, \ell_z, q)]$ by 
\begin{equation*}
\overline G_1 (s, \ve, \ell_z, q) := \int_{r^K_1(\ve, \ell_z, q)}^s\frac{d\sigma}{\sqrt{R(\sigma, \ve, \ell_z, q)}}\,d\sigma
\end{equation*} 
Since $r^K_1(\ve, \ell_z, q)$ and  $r^K_2(\ve, \ell_z, q)$ are simple roots of $R$, $\overline G(\cdot, \ve, \ell_z, q)$ is well-defined. Moreover, it is monotonically increasing on its domain so that it defines a bijection from $[r^K_1(\ve, \ell_z, q), r^K_2(\ve, \ell_z, q)]$ to $\left[0, \frac{T_r}{2}\right]$. Now,  we denote its inverse by $\overline G^{-1}_1(\cdot, \ve, \ell_z, q)$. In the same way, we define the bijective function $\overline G_2(\cdot, \ve, \ell_z, q)$ defined from $[r^K_1(\ve, \ell_z, q), r^K_2(\ve, \ell_z, q)]$ to $\left[\frac{T_r}{2},  T_r\right]$ by 
\begin{equation*}
\overline G_2 (s, \ve, \ell_z, q) := \frac{T_r}{2} + \int^{r^K_2(\ve, \ell_z, q)}_s\frac{d\sigma}{\sqrt{R(\sigma, \ve, \ell_z, q)}}\,d\sigma
\end{equation*}
and we denote its inverse by $\overline G_2^{-1} (\cdot, \ve, \ell_z, q)$. Now we define $\tilde r$ on $[0, T_r]$ by 
\begin{equation*}
\tilde r(\lambda) := \left\{
\begin{aligned}
& \overline G^{-1}_1(\lambda, \ve, \ell_z, q) \;\text{if}\;  \lambda\in\left[0, \frac{T_r}{2}\right]\\
& \overline G^{-1}_2(\lambda, \ve, \ell_z, q) \;\text{if}\;  \lambda\in\left[\frac{T_r}{2}, T_r\right]. 
\end{aligned}
\right. 
\end{equation*}
and $\tilde v_r$ on $[0, T_r]$ by 
\begin{equation*}
\tilde v_r := \Delta^{-1}(r(\lambda))r'(\lambda). 
\end{equation*}
Now, we can extend $(\tilde r, \tilde v_r)$ to a periodic solution defined on $\mathbb R$.  Moreover, It easy to see that $(\tilde r, \tilde v_r)$ satisfies\eqref{Cauchy::r}. By uniqueness, $r$ is periodic with period $T_r$.
 
\end{itemize}
\item The remaining cases follow using similar arguments.
%
\end{enumerate}

\end{proof}

\subsection{Study of the geodesic motion  in Weyl coordinates}
\label{Weyl::geo::motion}
The aim of this section is to analyse the geodesic motion in the Weyl coordinates defined in Section \ref{Kerr:in:Weyl}.  We will focus on {\it trapped non-spherical geodesics}. 
\noindent Weyl coordinates are well adapted to the axisymmetric problem especially when it comes to the resolution of the reduced Einstein Vlasov system. 
In a Kerr exterior, we have already seen that the geodesic motion forms an integrable system in BL coordinates. In particular, the $r-$motion decouples from the $\theta-$motion. In a general stationary and axisymmetric spacetime, this is not necessarily true. Therefore, it is useful to study the Kerr geodesic motion in Weyl coordinates without relying on the decoupling of the $r-$motion and the $\theta$-motion, i.e without relying on the existence of $q$. 
\\ Let $x = (t, \phi, \rho, z)\in\spacetime$, let $v = (v^t, v^\phi, v^\rho, v^z)$ be the conjugate coordinates  to the spacetime coordinates. In view of Section \ref{2:DOF}, the  geodesics equation reduces to the following two degree of freedom problem: 
\begin{equation}
\label{geodesic::system:kerr}
\left\{
\begin{aligned}
\frac{d\rho}{d\tau} &= v^\rho ,\\
\frac{dz}{d\tau} &= v^z , \\
\frac{dv^\rho}{d\tau} &= -\frac{1}{2}e^{-2\lambda}\partial_\rho J^K(\rho, z, \varepsilon, \ell_z, d) - \Chris{\rho}{i}{j}v^iv^j , \quad i,j\in\left\{\rho, z \right\} \\ 
\frac{dv^z}{d\tau} &= -\frac{1}{2}e^{-2\lambda}\partial_z J^K(\rho, z, \varepsilon, \ell_z, d) - \Chris{z}{i}{j}v^iv^j.
\end{aligned}
\right.
\end{equation}

where $J^K: \BB\times\mathbb R\times\mathbb R$ is defined by
 \begin{equation}
\label{potential:1}
J^K(\rho, z, \varepsilon, \ell_z) :=  -1 +  \frac{X_K}{\sigma_K^2}\varepsilon^2 + \frac{2W_K}{\sigma_K^2}\varepsilon\ell_z - \frac{V_K}{\sigma_K^2}\ell_z^2
\end{equation}
We introduce the effective potential energy $E_{\ell_z}^K:\BB\to\mathbb R$ relative to a timelike future directed geodesic, $(\gamma, I),$  with angular momentum $\ell_z$ and energy $\ve$:
\begin{equation}
\label{eff::potential}
E^K_{\ell_z}(\rho, z) := \frac{-W_K(\rho, z)}{X_K(\rho, z)}\ell_z + \frac{\sigma_K}{X_K(\rho, z)}\sqrt{\ell_z^2 + X_K(\rho, z)}.
\end{equation}
We refer to Figure \ref{Shape:E:l} for the shape of $E^K_{\ell_z}$ and we recall that the allowed region for $\gamma$ is given by 
\begin{equation*}
\begin{aligned}
A^K(\ve, \ell_z) &= \left\{(\rho, z)\in\BB \;:\;  J^K(\rho, z, \varepsilon, \ell_z) \geq  0\right\} \\
&= \left\{(\rho, z)\in\BB \;:\;  E_{\ell_z}^K(\rho, z) \leq \ve\right\} \\
&= \left\{(r, \theta)\in ]r_H(d), \infty[\times]0, \pi[\;:\; R(r, \ve, \ell_z, q) \geq 0 \quad\text{and}\quad T(\cos\theta, \ve, \ell_z, q)\geq 0  \right\} \\
\end{aligned}
\end{equation*}
We also recall that the boundary of $A^K(\ve, \ell_z)$ is the zero velocity curve $Z^K(\ve, \ell_z)$ given by Definition \ref{ZVC:St:Axis}. 
\begin{remark}
$Z^K(\ve, \ell_z)$ can also be seen as the level sets of the effective potential energy $E^K_{\ell_z}$ at $\varepsilon$. 
\end{remark}
\subsubsection{Properties of the effective potential energy $E^K_{\ell_z}$}
\label{proper:EKlz}
The classification of timelike geodesics is based on the topology of the $Z^K(\ve, \ell_z)$ curves whose shapes (depending on $(\ve, \ell_z)$) were already determined in Section \ref{classif:BL:geodesics}, Proposition \ref{ZVC:topology}. In this section, we will rewrite the latter proposition in terms of the level sets of $E_{\ell_z}^K$. From this perspective, the 
shape of Zero velocity curves associated to timelike trapped future directed geodesics will not depend on the Carter constant. 
This is key to the identification trapped geodesics in stationary and axisymmetric spacetimes close to Kerr. 
\\ Since we are interested in the level sets of $E_{\ell_z}^K$, we will first study its critical points. We make the difference between {\it direct critical points} and {\it retrograde critical points} defined by 
\begin{definition}
Let $\ell_z\in\mathbb R$. A point $(\rho^+_c, z^+_c)$ is a {\it direct critical point} of $E_{\ell_z}$ if 
\begin{equation*}
\nabla_{(\rho, z)}E_{\ell_z}(\rho^+_c, z^+_c) = 0 \quad\text{with}\quad -W\ell_z>0.
\end{equation*} 
A point $(\rho^-_c, z^-_c)$ is a {\it retrograde critical point} of $E_{\ell_z}$ if 
\begin{equation*}
\nabla_{(\rho, z)}E_{\ell_z}(\rho^-_c, z^-_c) = 0 \quad\text{with}\quad -W\ell_z<0.
\end{equation*} 
\end{definition}
\noindent We begin the analysis of $E_{\ell_z}$ with the study of critical lemma 
\begin{Propo}[Existence of critical points for $E^K_{\ell_z}$]
Let $\ell_z\in\mathbb R$. Then, 
\begin{itemize}
\item $E_{\ell_z}^K$ admits direct critical point if and only if $\ell_z\in]\-\infty, \ell_{min}^-]$,
\item $E_{\ell_z}^K$ admits retrograde critical point if and only if $\ell_z\in[\ell_{min}^+, \infty[$,
\end{itemize}
where $\ell_{min}^\pm$ is given by 
\begin{equation*}
\ell_{min}^\pm = \Psi_{\pm}(r_{ms}^\pm), 
\end{equation*}
and $r_{ms}^\pm$\footnote{See Lemma \ref{r:min:stable}.} is given  by 
\begin{equation*}
\label{r::ms}
{r_{ms}^\pm}(d) = 3 + Z_2(d) \mp \sqrt{(3 - Z_1)(3 + Z_1(d) + 2Z_2(d))}
\end{equation*}
where 
\begin{equation*}
Z_1(d) = 1 + (1 - d^2)^{\frac{1}{3}}((1+d)^\frac{1}{3} + (1-d)^\frac{1}{3})\;, \; Z_2(d) = \sqrt{3d^2 + Z_1^2}. 
\end{equation*}
Moreover, 
\begin{itemize}
\item if $\displaystyle \ell_z\geq\ell_{min}^+$,  then the critical points are given by 
\begin{equation}
\label{rho:c:p}
(\rho^+_s, z^+_s) := (\sqrt{\Delta(r_{max}^+(\ell_z))}, 0) \quad \text{and} \quad (\rho^+_{min}, z^+_{min}) := (\Delta(\sqrt{r_{min}^+(\ell_z)}), 0). 
\end{equation}
\item if $\displaystyle \ell_z\leq\ell_{min}^-$,  then the critical points are given by 
\begin{equation}
\label{rho:c:m}
(\rho^-_s, z^-_s) := (\sqrt{\Delta(r_{max}^-(\ell_z)}), 0) \quad \text{and} \quad (\rho^-_{min}, z^-_{min}) := (\Delta(\sqrt{r_{min}^-(\ell_z)}), 0),
\end{equation}
\end{itemize}
where $r_{max}^\pm(\ell_z)$ and $r_{min}^\pm(\ell_z)$ are given by
\begin{equation*}
r^\pm_{max}(\ell_z) :=  (\Psi^1)^{-1}_\pm(\ell_z) \quad \text{and}\quad r^\pm_{max}(\ell_z) :=  (\Psi^2)^{-1}_\pm(\ell_z),
\end{equation*}
where $(\Psi^1)_\pm^{-1}$ and $(\Psi^2)_\pm^{-1}$ are the inverse of the restriction of $(\Psi)_\pm$ on $]r_{ph}, r_{ms}[$ and $]r_{ms}, \infty[$ respectively. 
\end{Propo}

\begin{proof}
We only consider the case of direct critical points. The remaining case follows in the same manner.  We henceforth drop the $\pm$ symbol from all the quantities and assume that $\ell_z\in[0, \infty[$. 
\begin{enumerate}
\item By Lemma \ref{link:E:J}, $(\rho_c, z_c)$ is a critical point of $E^K_{\ell_z}$ if and only $(\rho_c, z_c, 0, 0)$ is a stationary solution of the reduced system  
\begin{equation}
\label{geodesic::system::kerr}
\left\{
\begin{aligned}
\frac{d\rho}{d\tau} &= v^\rho ,\\
\frac{dz}{d\tau} &= v^z , \\
\frac{dv^\rho}{d\tau} &= -\frac{1}{2}e^{-2\lambda}\partial_\rho J^K(\rho, z, \varepsilon, \ell_z, d) - \Chris{\rho}{i}{j}v^iv^j , \quad i,j\in\left\{\rho, z \right\} \\ 
\frac{dv^z}{d\tau} &= -\frac{1}{2}e^{-2\lambda}\partial_z J^K(\rho, z, \varepsilon, \ell_z, d) - \Chris{z}{i}{j}v^iv^j.
\end{aligned}
\right.
\end{equation}
with parameters $(\ve _c, \ell_z) = (E_{\ell_z}^K(\rho_c, z_c), \ell_z)$.
\item Now, we claim $(\rho_c, z_c, 0, 0)$ is a stationary solution of if and only if $(r_c, \theta_c, 0, 0)$ is a stationary solution of the reduced system 
\begin{equation}
\label{reduced:reduced:kerr}
\left\{
\begin{aligned}
\frac{dr}{d\tau} &= \frac{\Delta}{\Sigma^2}v_r ,\\
\frac{dv_r}{d\tau}&= \frac{1}{2\Sigma^2}\left(-\Delta'(r)v_r^2 + \frac{\Delta'(r)R(r, \ve, \ell_z, q) - \Delta(r)\partial_r R(r, \ve, \ell_z, q)}{\Delta(r)^2}\right), \\
\frac{d\theta}{d\tau} &= \frac{1}{\Sigma^2}v_\theta, \\
\frac{dv_\theta}{d\tau}&= \frac{1}{2\Sigma^2}\partial_{\theta}\left( \frac{T(\cos\theta, \ve, \ell_z, q)}{\sin^2\theta}\right). 
\end{aligned}
\right. 
\end{equation}
with parameters $(\ve_c, \ell_z)$. To prove the latter, it suffices to note that the system \eqref{reduced:reduced:kerr} is the system \eqref{geodesic::system::kerr} written in the isothermal coordinates $(\rho, z)$. 
\item By Lemma \ref{stationary::points},  stationary solutions of \eqref{reduced:reduced:kerr} exist if and only if $\ell_z\in[\ell_{min}, \infty[$, where $\ell_{min}$ is defined by \eqref{l:ms}. They are given by
\begin{equation*}
\left(r_{max}(\ell_z), \frac{\pi}{2}\right) \quad\text{and}\quad \left(r_{min}(\ell_z), \frac{\pi}{2}\right). 
\end{equation*}
\item Therefore $(\rho_c, z_c)$ are given by 
\begin{equation*}
\left(\rho_{s}^\pm(\ell_z), 0\right) \quad\text{and}\quad (\rho_{min}^\pm(\ell_z), 0),
\end{equation*}
where $\displaystyle \rho_{s}^\pm(\ell_z)$ and $\displaystyle \rho_{min}^\pm(\ell_z)$ are defined by \eqref{rho:c:p} and \eqref{rho:c:m}. 
Moreover, the corresponding $\ve_c$ satisfy
\begin{equation*}
\ve_c^s(\ell_z) := E^K_{\ell_z}(\rho_s(\ell_z), 0) = \ve_s(\ell_z)
\end{equation*}
and 
\begin{equation*}
\ve_c^m(\ell_z) := E^K_{\ell_z}(\rho_{min}(\ell_z), 0) = \ve_m(\ell_z)
\end{equation*}
where $\ve_s(\ell_z)$ and $\ve_m(\ell_z)$ are given by \eqref{lower:E} and \eqref{upper:E} respectively. 
\end{enumerate}
\end{proof}

\begin{Propo}[Study of critical points]
\label{critical::Elz:points}
Let $\ell_z\geq \ell^+_{min}$ or $\ell_z\leq \ell^-_{min}$. Then, 
\begin{itemize}
\item $\displaystyle (\rho_{s}^\pm(\ell_z), 0)$ corresponds to a saddle point.
\item  $\displaystyle (\rho_{min}^\pm(\ell_z), 0)$ corresponds to a local minimum. 
\end{itemize}
\end{Propo}

\begin{proof}
Let $\ell_z\in]-\infty, \ell^-_{min}]\cup[\ell_{min}^+, \infty[.$ We will study the critical points of $E_{\ell_z}^K$ in the BL coordinates. We recall that
\begin{equation*}
\tilde E^K_{\ell_z}(r, \theta):= E^K_{\ell_z}(\rho(r, \theta), z(r, \theta)). 
\end{equation*}
We have, $\forall (r, \theta)\in]r_H, \infty[\times]0, \pi[$,
\begin{equation} 
\label{exp::aux::}
\tilde J^K(r, \theta, \tilde E^K_{\ell_z}(r, \theta), \ell_z) = 0, 
\end{equation}
where $\tilde J^K$ is defined by \eqref{J:::tilde}. Now, we differentiate twice the expression \eqref{exp::aux::} in order to obtain
\begin{equation*}
\begin{aligned}
&\partial_\ve^2\tilde J^K(r, \theta, \tilde E^K_{\ell_z}(r, \theta), \ell_z)\nabla_{(r, \theta)}\tilde E^K_{\ell_z}(r, \theta)\left(\nabla_{(r, \theta)}\tilde E^K_{\ell_z}(r, \theta)\right)^t + \partial_\ve\tilde J^K(r, \theta, \tilde E^K_{\ell_z}(r, \theta), \ell_z)\nabla^2_{(r, \theta)}\tilde E^K_{\ell_z}(r, \theta)  \\
&+ \nabla^2_{(r, \theta)}\tilde J^K(r, \theta, \tilde E^K_{\ell_z}(r, \theta), \ell_z) + \nabla_{(r, \theta)}\tilde E_{\ell_z}(r, \theta)\left(\nabla_{(r, \theta)} \partial_\ve \tilde J^K(r, \theta, \tilde E^K_{\ell_z}(r, \theta), \ell_z)\right)^t  \\
&= 0. 
\end{aligned}
\end{equation*}
\noindent In particular, if $(r, \theta)$ is a critical point, then the latter expression reduces to 
\begin{equation*}
\begin{aligned}
\partial_\ve\tilde J^K(r_c, \theta_c, \tilde E^K_{\ell_z}(r_c, \theta_c), \ell_z)\nabla^2_{(r, \theta)}\tilde E^K_{\ell_z}(r_c, \theta_c) + \nabla^2_{(r, \theta)}\tilde J^K(r_c, \theta_c, \tilde E^K_{\ell_z}(r_c, \theta_c), \ell_z)  = 0. 
\end{aligned}
\end{equation*}
Now, we recall from Section \ref{2:DOF}, that the term $\displaystyle \partial_\ve\tilde J^K(r_c, \theta_c, \tilde E^K_{\ell_z}(r_c, \theta_c), \ell_z)$ does not vanish. Moreover, it is positive on its domain. Therefore, 
\begin{equation*}
\displaystyle \nabla^2_{(r, \theta)}\tilde E^K_{\ell_z}(r_c, \theta_c) = - \frac{\nabla^2_{(r, \theta)}\tilde J^K(r_c, \theta_c, \tilde E^K_{\ell_z}(r_c, \theta_c), \ell_z)}{\partial_\ve\tilde J^K(r_c, \theta_c, \tilde E^K_{\ell_z}(r_c, \theta_c), \ell_z)}. 
\end{equation*}
Now, we compute the Hessian of $\tilde J^K$ with respect to $(r, \theta)$ at the points $(\ve_c(\ell_z) = \tilde E^K_{\ell_z}(r_c, \theta_c), \ell_z, r_c(\ell_z), \theta_c(\ell_z))$. We find that 
\begin{equation*}
\begin{aligned}
\partial_{rr}\tilde J^K\left(\ve_c(\ell_z), \ell_z, r_c(\ell_z), \frac{\pi}{2}\right) &= \Sigma^{-2}\left(r_c(\ell_z), \frac{\pi}{2}\right)\partial_{rr}R(r_c(\ell_z), \ve_c(\ell_z), \ell_z, 0), \\
\partial_{\theta\theta}\tilde J^K\left(\ve_c(\ell_z), \ell_z, r_c(\ell_z), \frac{\pi}{2}\right) &= \Sigma^{-2}\left(r_c(\ell_z), \frac{\pi}{2}\right)\partial_{\theta\theta}{T(0, \ve_c(\ell_z), \ell_z, 0)}, \\
\partial_{r\theta}\tilde J^K\left(\ve_c(\ell_z), \ell_z, r_c(\ell_z), \frac{\pi}{2}\right) &= 0. 
\end{aligned}
\end{equation*}
Here, we used that 
\begin{itemize}
\item $\displaystyle \cos\theta_c(\ell_z) = 0$ is a double root of the polynomial $T(\cdot,\ve_c(\ell_z), \ell_z, q)$ so that $q = 0$, see Lemma \ref{T:double:root}, 
\item $r_c(\ell_z)$ is a double root  of the polynomial $R(\cdot,\ve_c(\ell_z), \ell_z, 0)$. 
\end{itemize}
Moreover, 
\begin{equation*}
\partial_{\theta\theta}T(0, \ve_c(\ell_z), \ell_z, 0) > 0
\end{equation*}
and 
\begin{equation*}
\partial_{rr}R(r_c(\ell_z), \ve_c(\ell_z), \ell_z, 0) = \left\{ 
\begin{aligned}
&> 0 \quad\text{if}\quad r_c(\ell_z) = r^\pm_{min}(\ell_z) ,\\
&< 0 \quad\text{if}\quad r_c(\ell_z) = r^\pm_{max}(\ell_z). \\
\end{aligned}
\right. 
\end{equation*}
Hence, 
\begin{itemize}
\item $\displaystyle \left(r_{max}^\pm(\ell_z), \frac{\pi}{2}\right)$ corresponds to a saddle point for $\tilde E^K_{\ell_z}$.
\item  $\displaystyle \left(r_{min}^\pm(\ell_z), \frac{\pi}{2}\right)$ corresponds to a local minimum $\tilde E^K_{\ell_z}$. 
\end{itemize}
\end{proof}

\noindent Now, we study the basic properties of $E_{\ell_z}^K$:
\begin{lemma}[Properties of $E^K_{\ell_z}$]
We have
\begin{enumerate}
\item  Let $\tilde\Axis$ be a neighbourhood of the axis. Then 
$$\displaystyle \lim_{||(\rho, z)||\to\infty, (\rho, z)\in (\BAbarre\cup\BHbarre)\backslash \tilde\Axis}\; E^K_{\ell_z}(\rho, z) = 1.$$ 
\item  $\forall(\rho, z)\in\BB,$
\begin{equation*}
E^K_{\ell_z}(\rho, z)  = E^K_{\ell_z}(\rho, -z).  
\end{equation*}
\item For $d\ell_z< 0$, $E^K_{\ell_z}$ is negative in a neighbourhood of the horizon \footnote{In this region, particles with positive energies in the local observer's frame can have negative energy with respect to infinity.}.  
\end{enumerate}
\end{lemma}
\begin{proof}
\begin{enumerate}
\item  We have $\forall (\rho, z)\in\BB\,,\,\forall \ell_z\in\mathbb R$
\begin{equation*}
-\frac{W_K(\rho, z)}{X_K(\rho, z)}\ell_z = \frac{2dr(\rho, z)}{\tilde\Pi(\rho, z)}\ell_z, 
\end{equation*} 
where $r(\rho, z)$ is given \eqref{r:rho:z} and 
\begin{equation*}
\tilde\Pi(\rho, z) := (r^2(\rho, z) + d^2)^2 - \rho^2d^2. 
\end{equation*}
Therefore, 
\begin{equation*}
\lim_{||(\rho, z)||\to\infty}\; -\frac{W_K(\rho, z)}{X_K(\rho, z)}\ell_z = \lim_{||(\rho, z)||\to\infty}\frac{2dr(\rho, z)}{r^4(\rho, z)}\ell_z = 0. 
\end{equation*}
Now, we compute  $\displaystyle \lim_{||(\rho, z)||\to\infty}\; \frac{\rho}{X_K(\rho, z))}\sqrt{\ell_z^2 + X_K(\rho, z)}.$ We have  $\forall (\rho, z)\in\BB\,,\,\forall \ell_z\in\mathbb R$
\begin{equation*}
\frac{\rho}{X_K(\rho, z)} = \frac{1}{\rho}\frac{\Delta(r(\rho, z))\tilde\Sigma^2(\rho, z)}{\tilde\Pi(\rho, z)}
\end{equation*}
where 
\begin{equation*}
\tilde\Sigma^2(\rho, z) := r^2(\rho, z) + \frac{d^2z^2}{(r(\rho, z) - 1)^2}
\end{equation*}
We have 
\begin{equation*}
\lim_{||(\rho, z)||\to\infty} \frac{\Delta(r(\rho, z))\tilde\Sigma^2(\rho, z)}{\tilde\Pi(\rho, z)} = 1. 
\end{equation*}
\item By \eqref{r:rho:z}, $\forall(\rho, z)\in\BB$, 
\begin{equation*}
r(\rho, z) = r(\rho, -z). 
\end{equation*}
Therefore, $\forall \rho>0$, $W_K(\rho,\cdot)$ and $X_K(\rho,\cdot)$ are even.
\item Let $(\rho, z)\in \BB$. Then, 
\begin{equation*}
\begin{aligned}
E^K_{\ell_z}(\rho, z)  &=  -\frac{W_K(\rho, z)}{X_K(\rho, z)}\ell_z +  \frac{\rho}{X_K(\rho, z))}\sqrt{\ell_z^2 + X_K(\rho, z)} \\
&= \frac{2d\ell_zr(\rho, z)}{\tilde\Pi(\rho, z)} + \frac{\rho}{X_K(\rho, z))}\sqrt{\ell_z^2 + X_K(\rho, z)}. 
\end{aligned}
\end{equation*}
We show that there exists a neighbourhood of $\Horizon$, say  $\tilde\Horizon\subset \BHbarre$, such that $\forall d\ell_z\leq 0$, $E_{\ell_z}^K$ is negative on $\tilde\Horizon$. By Lemma \ref{extendibility},  we have the following asymptotics for $X_K$  near the horizon $X_{\Horizon}:\tilde{\Horizon}\to\mathbb{R}$ with  $X_{\Horizon}(0,z)>0$ such that 
    \[
    \left. X(\rho,z)\right|_{\tilde{\Horizon}} = X_{\Horizon}(\rho^2,z). 
    \]
Moreover the function $\displaystyle (\rho, z)\to \frac{r(\rho, z)}{\tilde\Pi(\rho, z)}$ is smooth and positive on $\BHbarre$. Therefore, $E^K_{\ell_z}$ extends smoothly to $\BHbarre$. Now we make a first order Taylor expansion for $E_{\ell_z}^K$ around $(0, z)$ with $|z|<\gamma$: $\forall\rho\geq 0$ small, 
\begin{equation*}
E^K_{\ell_z}(\rho, z) = E^K_{\ell_z}(0, z) + \rho\partial_\rho E^K_{\ell_z}(0, z) + O(\rho^2). 
\end{equation*}
We compute: 
\begin{equation*}
E^K_{\ell_z}(0, z) = 2d\ell_z\frac{r_H}{\tilde\Pi(0, z)} = \frac{d\ell_z}{2r_H}
\end{equation*}
and 
\begin{equation*}
\partial_\rho E^K_{\ell_z}(0, z) = \frac{1}{X_{\Horizon}(0, z))}\sqrt{\ell_z^2 + X_{\Horizon}(0, z)} + 2d\ell_z\frac{\partial_\rho r(0, z)\tilde\Pi(0, z) - r_H\partial_\rho\tilde\Pi(0, z)}{16r_H^4}.
\end{equation*}
We have 
\begin{equation*}
\partial_\rho r(0, z) = 0 \quad\text{and}\quad \partial_\rho\tilde\Pi(0, z) = 8r_H^2. 
\end{equation*}
Hence 
\begin{equation*}
\partial_\rho r(0, z)\tilde\Pi(0, z) - r_H\partial_\rho\tilde\Pi(0, z) = -8r_H^2
\end{equation*}
and 
\begin{equation*}
\partial_\rho E^K_{\ell_z}(0, z) > 0. 
\end{equation*}
Now, we choose $\rho>0$ such that 
\begin{equation*}
\frac{d\ell_z}{2r_H} + \rho\partial_\rho E^K_{\ell_z}(0, z) < 0. 
\end{equation*} 
\end{enumerate}

\end{proof}

\begin{figure}[h!]
\includegraphics[width=\linewidth]{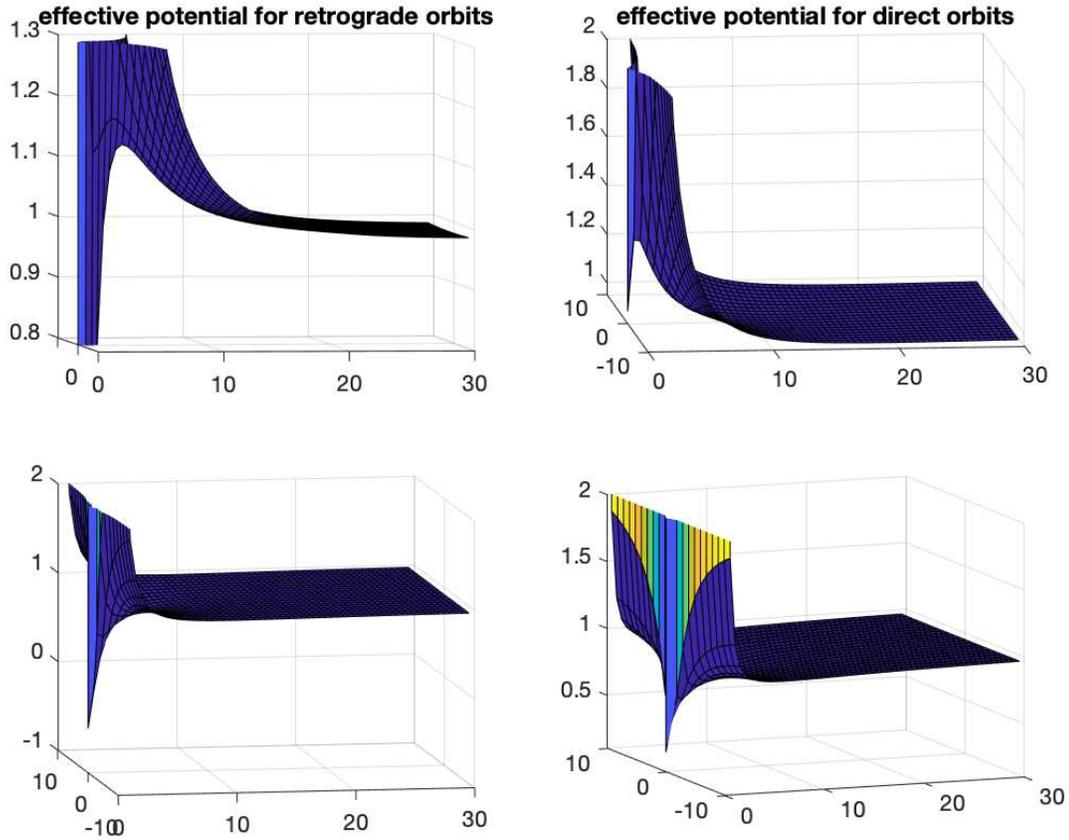}
\caption{Shape of $E^K_{\ell_z}$ with $d= 0.8$. On the left: $\ell_z = 1$. On the right: $\ell_z = 4$}
\label{Shape:E:l}
\end{figure}

\subsubsection{Trapped timelike future-directed geodesics}
\label{classification::geodesics}
A full classification of timelike future-directed geodesics in Kerr can be achieved using the integrability of the geodesics equation in BL coordinates (see Proposition \ref{Classification:}). A priori,  the nature of the orbits then depends also on the Carter constant $Q$. In stationary and axisymmetric spacetimes, the geodesic motion does not necessarily form an integrable Hamiltonian system since there is no generalisation of $Q$ and there are only three constants of motion $E$, $L_z$ and $m$. 
\\ However, since only an open set of trapped non-spherical orbits is relevant to our work, we will determine sufficient conditions on $(\ve, \ell_z)$ which are independent of $Q$ and on the precise  initial position $(\rho(0), z(0))$ for a trajectory to be trapped. This will allow us to construct an open subset of parameters $(\ve, \ell_z)$ on which the distribution function $f$ will be supported. 
\\ We begin by recalling from Proposition \ref{Classification:} the necessary and sufficient conditions for a timelike future-directed orbit to be trapped and non-spherical. 
\\Let $\gamma: I\ni 0\to \mathcal O$ be a timelike future-directed geodesic with constants of motion $(\ve, \ell_z)$ and let $\tilde \gamma $ its projection in $\BB$. Then $\gamma$ is trapped if and only if one of the following cases occur: 
\begin{enumerate}
\item $\displaystyle \ve\in]\ve^+_{min}, \ve^-_{min}]$, $\displaystyle \ell_z\in]\tilde\ell_{min}(\ve), \ell^+_{lb}(\ve)[$ and 
\begin{equation*}
\begin{aligned}
\tilde\gamma (0)&\in \left( A^{K, abs}(\ve, \ell_z)\cap(]\tilde \rho^+(\ve, \ell_z), \rho^K_0(\ve, \ell_z)[\times]\gamma -\tilde z^+(\ve, \ell_z), \gamma[)\right) \\
&\cup \left( A^{K, abs}(\ve, \ell_z)\cap(]\tilde \rho^+(\ve, \ell_z), \rho^K_0(\ve, \ell_z)[\times]-\gamma, \gamma+\tilde z^+(\ve, \ell_z)[)\right),
\end{aligned}
\end{equation*}
\item $\displaystyle \ve\in]\ve^+_{min}, \ve^-_{min}]$, $\displaystyle \ell_z\in]\ell^+_{lb}(\ve), \ell^+_{ub}(\ve)[$ and $\displaystyle \tilde\gamma (0)\in A^{K, trapped}(\ve, \ell_z)$,
\item $\displaystyle \ve\in]\ve^-_{min}, 1[$, $\displaystyle \ell_z\in]\ell^-_{ub}(\ve), \ell^-_{lb}(\ve)\cup]\ell^+_{lb}(\ve), \ell^+_{ub}(\ve)[$ and $\displaystyle \tilde\gamma (0)\in A^{K, trapped}(\ve, \ell_z)$,
\item $\displaystyle \ve\in]\ve^-_{min}, 1[$, $\displaystyle \ell_z\in]\ell^-_{lb}(\ve), \ell^+_{lb}(\ve)[$ and 
\begin{equation*}
\begin{aligned}
\tilde\gamma (0)&\in \left( A^{K, abs}(\ve, \ell_z)\cap(]\tilde \rho^1(\ve, \ell_z), \rho^K_0(\ve, \ell_z)[\times]\gamma - \tilde z^1(\ve, \ell_z), \gamma[)\right) \\
&\cup \left( A^{K, abs}(\ve, \ell_z)\cap(]\tilde \rho^1(\ve, \ell_z), \rho^K_0(\ve, \ell_z)[\times]-\gamma, -\gamma + \tilde z^1(\ve, \ell_z), \pi[)\right), 
\end{aligned}
\end{equation*}
\end{enumerate}
where 
\begin{equation*}
(\tilde \rho^+(\ve, \ell_z), \tilde z^+(\ve, \ell_z) ) := (\sqrt{\Delta(\tilde r^+(\ve, \ell_z))}\sin(\theta_1(\ve, \ell_z)), (\tilde r^+(\ve, \ell_z) - 1)\cos(\theta_1(\ve, \ell_z))) 
\end{equation*}
and 
\begin{equation*}
(\tilde \rho^1(\ve, \ell_z), \tilde z^1(\ve, \ell_z) ) := (\sqrt{\Delta(\tilde r^1(\ve, \ell_z))}\sin(\tilde \theta_1(\ve, \ell_z)), (\tilde r^1(\ve, \ell_z) - 1)\cos(\tilde\theta_1(\ve, \ell_z))) 
\end{equation*}
\noindent In particular, if $\gamma$ is a timelike future-directed geodesic with constants of motion $(\ve, \ell_z)\in\Abound$,  where $\Abound $ is defined by \eqref{A:bound} (see Figure \ref{Abound:fig}), then $\tilde \gamma$ is either trapped or plunging. Moreover, the zero velocity curve associated to $\gamma$, $Z^K(\ve, \ell_z)$, has two connected components. Indeed, by Lemma \ref{ZVC:topology}, we have
\begin{equation*}
Z^K(\ve, \ell_z) = Z^{K, abs}(\ve, \ell_z)\cup Z^{K, trapped}(\ve, \ell_z)
\end{equation*}
where $Z^{K, abs}(\ve, \ell_z)$ is diffeomorphic to $\mathbb R$ and $Z^{K, trapped}(\ve, \ell_z)$ is diffeomorphic to $\mathbb S^1$. 
\\   
\\ In the following section, we will reparameterize\footnote{For the sole purpose of the main theorem, we could have constructed an atlas only for $Z^{K, trapped}(\ve, \ell_z)$. We also do the analysis of $Z^{K, abs}(\ve, \ell_z)$} $Z^K(\ve, \ell_z)$ for $(\ve, \ell_z)\in\Abound$.  
\subsubsection{Reparameterization of the zero velocity curves associated to trapped geodesics}
\label{reparam}
Let $\gamma$ be a timelike future-directed geodesic with constants of motion $(\ve, \ell_z)\in\Abound$ and let $\tilde \gamma$ be its projection in $\BB$. We recall that the allowed region for $\tilde \gamma$  has two connected components $A^{K, abs}(\ve, \ell_z)$ and $A^{K, trapped}(\ve, \ell_z)$ bounded respectively by $Z^{K, abs}(\ve, \ell_z)$ and   $Z^{K, trapped}(\ve, \ell_z)$. 
\\ Since $Z^{K, abs}(\ve, \ell_z)$ and   $Z^{K, trapped}(\ve, \ell_z)$ are one-dimensional manifolds, we will construct atlases  $(\Phi^{K, abs}_{(\ve, \ell_z)}, I_{(\ve, \ell_z)}\subset \mathbb R)$ and $(\Phi^{K, trapped}_{(\ve, \ell_z), i}, I_{(\ve, \ell_z), i})_{i=1\cdots i_0}$ for $Z^{K, abs}(\ve, \ell_z)$ and   $Z^{K, trapped}(\ve, \ell_z)$ respectively. This will allow us to see locally the solutions of the equation
\begin{equation}
\label{eq::motion:}
\varepsilon = E^K_{\ell_z}(\rho, z). 
\end{equation}
as the graph of smooth functions. As a consequence, the problem of finding solutions to Equation \eqref{eq::motion:} on $\BB$ is equivalent to writing locally $\rho$ as a function of $z$ or $z$ as a function of $\rho$. This representation of the solutions will help us to formulate the stability result for trapped non-spherical timelike future-directed geodesics in the following section. 
 
\begin{figure}[h!]
\includegraphics[width=\linewidth]{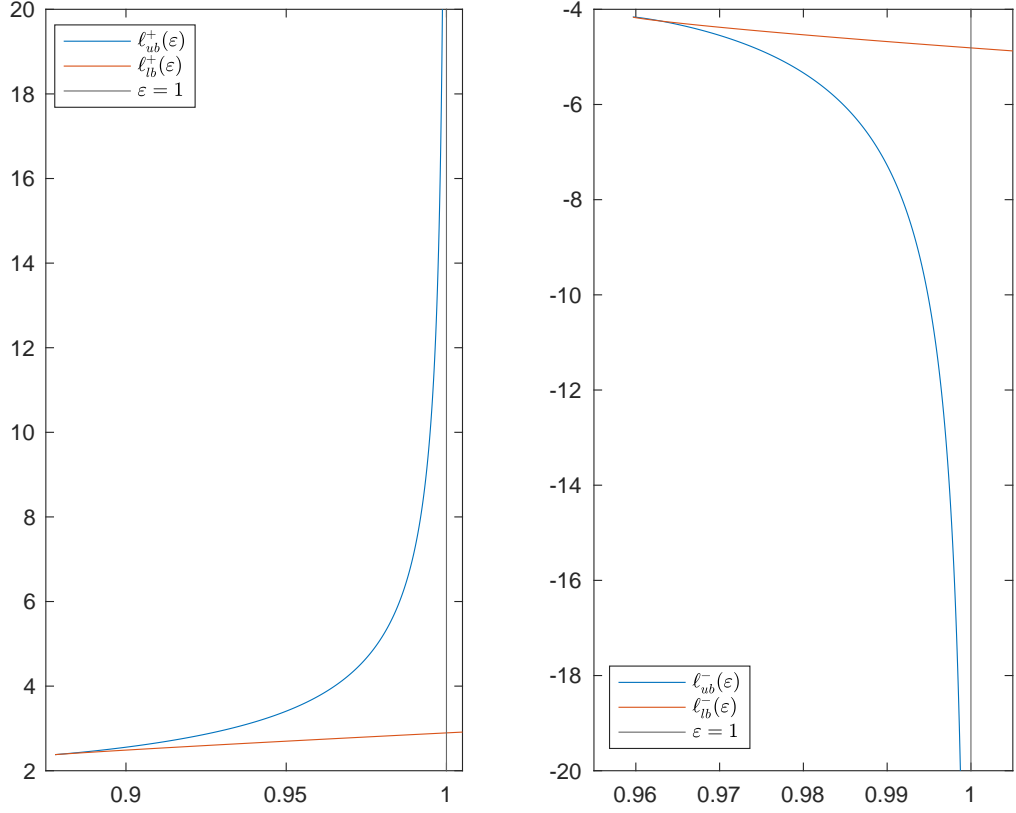}
\caption{The set $\Abound$ with $d= 0.8$ as the union of the two set bounded by the red curve, the blue curve and the black curve. On the left: $\Abound^+$. On the right: $\Abound^-$. }
\label{Abound:fig}
\end{figure}
\noindent To this end, we begin with the following lemma
\begin{lemma}
\label{z::max:def}
$\forall (\ve, \ell_z)\in\Abound$, we have 
\begin{equation*}
Z^{K, trapped}(\ve, \ell_z) \subset[\rho_1^K(\ve, \ell_z), \rho_2^K(\ve, \ell_z)]\times[-z_{max}^K(\ve, \ell_z), z_{max}^K(\ve, \ell_z)]
\end{equation*}
where $\rho_i^K(\ve, \ell_z)$ are the two largest roots of the equation \eqref{eq::motion:} with $z=0$ and $z_{max}^K(\ve, \ell_z)$ is defined by 
\begin{equation}
\label{z:K:max:def}
\SymbolPrint{z_{max}^K}(\ve, \ell_z) := \left(\tilde r(\ve, \ell_z) - 1\right)\cos\overline\theta^{<1}_{max}(\ve, \ell_z)
\end{equation}
where $\tilde r(\ve, \ell_z)$ is defined in Lemma \ref{lemma:29} and $\cos\overline\theta^{<1}_{max}(\ve, \ell_z)$ is given by \eqref{theta::max:p}. 
\\Moreover, $\forall (\ve, \ell_z)\in\Abound$, the equation \eqref{eq::motion:} with $z = \pm z_{max}^K(\ve, \ell_z)$ admits a unique solution $\rho_{max}(\ve, \ell_z)$ in the region $[\rho_1^K(\ve, \ell_z), \rho_2^K(\ve, \ell_z)]$, given by 
\begin{equation}
\label{rho::m::def}
\SymbolPrint{\rho_{max}}(\ve, \ell_z) = \sqrt{\Delta(\tilde r(\ve, \ell_z))}\sin\overline\theta^{<1}_{max}(\ve, \ell_z). 
\end{equation}
\end{lemma} 
\begin{proof}
We recall from Proposition \ref{ZVC:topology} that 
\begin{equation*}
\begin{aligned}
Z^{K, trapped}(\ve, \ell_z) &= Graph\left(r^1_{tr}(\ve, \ell_z, q(\cdot))\right) \cup Graph\left(r^2_{tr}(\ve, \ell_z, q(\cdot))\right) \\
&\cup Graph\left(\theta_{tr}(\ve, \ell_z, q(\cdot))\right) \cup Graph\left( \pi -  \theta_{tr}(\ve, \ell_z, q(\cdot))\right),
\end{aligned}
\end{equation*}
where $r^1_{tr}$, $r^2_{tr}$ and $\theta_{tr}$ are given by \eqref{r1:tr:theta}, \eqref{r2:tr:theta} and \eqref{theta::tr:r}  respectively.  Moreover, the latter functions verify 
\begin{itemize}
\item $r^1_{tr}(\ve, \ell_z, q(\cdot))$ has a global minimum  at $r_1^K(\ve, \ell_z)$ and 
\begin{equation*}
r^1_{tr}(\ve, \ell_z, q(\overline \theta_{max}^{<1}(\ve, \ell_z))) = r^1_{tr}(\ve, \ell_z, q(\pi - \overline \theta_{max}^{<1}(\ve, \ell_z)))  = \tilde r(\ve, \ell_z). 
\end{equation*}
Therefore, 
\begin{equation*}
Graph\left(r^1_{tr}(\ve, \ell_z, q(\cdot))\right) \subset \left[r_1^K(\ve, \ell_z),  \tilde r^+(\ve, \ell_z)\right]\times \left[\overline \theta_{max}^{<1}(\ve, \ell_z), \pi - \overline \theta_{max}^{<1}(\ve, \ell_z)\right]. 
\end{equation*}
\item $r^2_{tr}(\ve, \ell_z, q(\cdot))$ has a global maximum  at $r_2^K(\ve, \ell_z)$ and 
\begin{equation*}
r^2_{tr}(\ve, \ell_z, q(\overline \theta_{max}^{<1}(\ve, \ell_z))) = r^2_{tr}(\ve, \ell_z, q(\pi - \overline \theta_{max}^{<1}(\ve, \ell_z)))  = \tilde r(\ve, \ell_z), 
\end{equation*}
Thus, 
\begin{equation*}
Graph\left(r^2_{tr}(\ve, \ell_z, q(\cdot))\right) \subset \left[\tilde r(\ve, \ell_z), r_2^K(\ve, \ell_z)\right]\times \left[\overline \theta_{max}^{<1}(\ve, \ell_z), \pi - \overline \theta_{max}^{<1}(\ve, \ell_z)\right]. 
\end{equation*}
\item $\displaystyle \theta_{tr}(\ve, \ell_z, q(\cdot))$ has a global maximum at $\displaystyle \overline \theta_{max}^{<1}(\ve, \ell_z)$ and 
\begin{equation*}
Graph\left(\theta_{tr}(\ve, \ell_z, q(\cdot))\right) \subset \left[r_1^K(\ve, \ell_z), r_2^K(\ve, \ell_z)\right]\times \left[0,  \overline \theta_{max}^{<1}(\ve, \ell_z)\right].
\end{equation*}
\item $\displaystyle \pi - \theta_{tr}(\ve, \ell_z, q(\cdot))$ has a global minimum at $\displaystyle \pi - \overline \theta_{max}^{<1}(\ve, \ell_z)$ and 
\begin{equation*}
Graph\left(\theta_{tr}(\ve, \ell_z, q(\cdot))\right) \subset \left[r_1^K(\ve, \ell_z), r_2^K(\ve, \ell_z)\right]\times \left[0, \pi - \overline \theta_{max}^{<1}(\ve, \ell_z)\right].
\end{equation*}
\end{itemize}
Therefore, 
\begin{equation*}
\begin{aligned}
Z^{K, trapped}(\ve, \ell_z) &\subset   [r_1^K(\ve, \ell_z), r_2^K(\ve, \ell_z)]\times[\overline \theta_{max}^{<1}(\ve, \ell_z), \pi - \overline \theta_{max}^{<1}(\ve, \ell_z)] ,
\end{aligned}
\end{equation*}
or if we use the coordinates $(\rho, z)$, we have 
\begin{equation*}
 Z^{K, trapped}(\ve, \ell_z) \subset  [\rho_1^K(\ve, \ell_z), \rho_2^K(\ve, \ell_z)]\times[-z_{max}^K(\ve, \ell_z), z_{max}^K(\ve, \ell_z)]. 
\end{equation*}
\end{proof}
\noindent Now, let $\Bbound \subset \subset \Abound$. We claim that $\Bbound$ can be included in a finite union of products of closed intervals. More precisely, we have 
\begin{lemma}
\label{compact:Bbound}
Let $\Bbound \subset \subset \Abound$. Then, there exists a finite number $N$ of products of closed intervals $[\ve^{i, \pm}_1, \ve^{i, \pm}_2]\times[\ell^{i, \pm}_1, \ell^{i, \pm}_2]$ such that 
\begin{equation*}
[\ve^{i, \pm}_1, \ve^{i, \pm}_2]\times[\ell^{i, \pm}_1, \ell^{i, \pm}_2]\subset\Abound^\pm,
\end{equation*}
\begin{equation*}
\Bbound^\pm \subset\bigcup_{i= 1}^N [\ve^{i, \pm}_1, \ve^{i, \pm}_2]\times[\ell^{i, \pm}_1, \ell^{i, \pm}_2]
\end{equation*}
and $\displaystyle \ve^{i, \pm}_j$ and $\displaystyle \ell^{i, \pm}_j$ satisfy
\begin{equation}
\label{bounds:}
\ve^\pm_{min}<\ve^{i, \pm}_1<\ve^{i, \pm}_2<1 \quad\text{,}\quad \ell_{lb}(\ve^+_2)<\ell^{i, +}_1<\ell^{i, +}_2<\ell_{ub}(\ve^+_1) \quad\text{and}\quad \ell_{ub}(\ve^-_1)<\ell^{i, -}_1<\ell^{i, -}_2<\ell_{lb}(\ve^-_2). 
\end{equation} 
\end{lemma}
\begin{proof}
By compactness of $\Bbound$,  there exists a finite number $N$ of product closed intervals $[\ve^{i, \pm}_1, \ve^{i, \pm}_2]\times[\ell^{i, \pm}_1, \ell^{i, \pm}_2]$ such that 
\begin{equation}
\label{inclusion:Abound}
[\ve^{i, \pm}_1, \ve^{i, \pm}_2]\times[\ell^{i, \pm}_1, \ell^{i, \pm}_2]\subset\Abound^\pm,
\end{equation}
\begin{equation*}
\Bbound^\pm\subset\bigcup_{i= 1}^N [\ve^{i, \pm}_1, \ve^{i, \pm}_2]\times[\ell^{i, \pm}_1, \ell^{i, \pm}_2]. 
\end{equation*}
We need to check that $\displaystyle \ve^{i, \pm}_j$ and $\displaystyle \ell^{i, \pm}_j$ satisfy \eqref{bounds:}. Let $i\in\left\{1, \dots N\right\}$. Then, 
\begin{equation*}
[\ve^{i, \pm}_1, \ve^{i, \pm}_2] \subset ]\ve^+_{min}, 1[
\end{equation*}
and 
\begin{equation*}
[\ell^{i, \pm}_1, \ell^{i, \pm}_2] \subset ]\ell_{lb}^+(\ve), \ell_{ub}^+(\ve)[ \quad \text{for all}\;  \ve \in [\ve^{i, \pm}_1, \ve^{i, \pm}_2].
\end{equation*}
By monotonicity properties of $\ell_{ub}^+$ and $\ell_{lb}^+$ (see Lemma \ref{monotonicity:ell}), we have 
\begin{equation*}
\begin{aligned}
\ell_{lb}^+(\ve^{i, +}_1) <\ell_{lb}^+(\ve) < \ell_{ub}^+(\ve^{i, +}_2) \\
\ell_{ub}^+(\ve^{i, +}_1) <\ell_{ub}^+(\ve) < \ell_{ub}^+(\ve^{i, +}_2). 
\end{aligned}
\end{equation*}
Therefore, $\ell_1^{i, +}$ and $\ell_2^{i, +}$ verify
\begin{equation*}
 \ell_{lb}^+(\ve^{i, +}_1)  < \ell_{lb}^+(\ve_2) < \ell_1^{i, +} < \ell_2^{i, +} < \ell_{lb}^+(\ve) <  \ell_{ub}^+(\ve^{i, +}_1). 
\end{equation*}
\end{proof}
\begin{remark}
In the remaining of our work, we will suppose that $\Bbound$ is included in one of the product intervals $\displaystyle [\ve^{i, \pm}_1, \ve^{i, \pm}_2]\times[\ell^{i, \pm}_1, \ell^{i, \pm}_2]$. The general case can be dealt by a partition of unity argument. 
\end{remark}
\noindent We shall henceforth assume that $\Bbound$ has the following form
\begin{equation}
\label{BB:bound}
\Bbound := \Bbound^-\cup \Bbound^+
\end{equation}
where $\Bbound^-$ and $\Bbound^+$ are defined by 
\begin{equation*}
\Bbound^\pm := [\ve_1^\pm, \ve_2^\pm]\times[\ell_1^\pm, \ell_2^\pm] 
\end{equation*}
where $\displaystyle \ve^{i, \pm}_j$ and $\displaystyle \ell^{i, \pm}_j$ satisfy 
\begin{equation}
\label{bounds:}
\ve^\pm_{min}(d)<\ve^\pm_1<\ve^\pm_2<1 \quad\text{,}\quad \ell_{lb}(\ve^+_2)<\ell^+_1<\ell^+_2<\ell_{ub}(\ve^+_1) \quad\text{and}\quad \ell_{ub}(\ve^-_1)<\ell^-_1<\ell^-_2<\ell_{lb}(\ve^-_2).
\end{equation} 
\noindent In the remaining of this section, we omit the $\pm$ in order to lighten the expressions. We state the following lemma  
\begin{lemma}
\label{monotonicity:rho:i}
Let $(\ve, \ell_z)\in \Bbound$ and let $\rho^K_i(\ve, \ell_z), i\in\left\{0, 1, 2\right\}$ be the solutions of the equation 
\begin{equation*}
E^K_{\ell_z}(\rho, 0) = \ve
\end{equation*} 
such that 
\begin{equation*}
\rho^K_0(\ve, \ell_z) < \rho^K_1(\ve, \ell_z) < \rho^K_2(\ve, \ell_z). 
\end{equation*}
Then, 
\begin{enumerate}
\item $\forall \ve\in[\ve_1, \ve_2]$, 
\begin{itemize}
\item $\displaystyle \rho^K_0(\ve, \cdot)$ decreases monotonically on $[\ell_1, \ell_2]$. 
\item $\displaystyle \rho^K_1(\ve, \cdot)$ increases monotonically on $[\ell_1, \ell_2]$. 
\item $\displaystyle \rho^K_2(\ve, \cdot)$ decreases monotonically on $[\ell_1, \ell_2]$.
 \end{itemize}
\item $\forall \ell_z\in[\ell_1, \ell_2]$, 
\begin{itemize}
\item $\displaystyle \rho^K_0(\cdot, \ell_z)$ decreases monotonically on $[\ve_1, \ve_2]$. 
\item $\displaystyle \rho^K_1(\cdot, \ell_z)$ increases monotonically on $[\ve_1, \ve_2]$. 
\item $\displaystyle \rho^K_2(\cdot, \ell_z)$ decreases monotonically on $[\ve_1, \ve_2]$.
 \end{itemize}
\end{enumerate}
\end{lemma}
\begin{proof}
First of all, we recall from Proposition \ref{critical::Elz:points} that $E^K_{\ell_z}(\cdot, 0)$ admits two critical points: a maximum at $\rho_s(\ell_z)$ and a minimum at $\rho_s(\ell_z)$. It follows that 
\begin{itemize}
\item $E^K_{\ell_z}(\cdot, 0)$ is monotonically increasing on $]0, \rho_s(\ell_z)[$,
\item $E^K_{\ell_z}(\cdot, 0)$ is monotonically decreasing on $]\rho_s(\ell_z), \rho_m(\ell_z)[$,
\item $E^K_{\ell_z}(\cdot, 0)$ is monotonically increasing on $]\rho_m(\ell_z), \infty[$. 
\end{itemize}
Moreover, $\forall \rho >0$,  $E^K_{\cdot}(\rho, 0)$ is monotonically increasing on $]\ell_1^+, \ell_2^+[$. Furthermore, we have 
\begin{equation*}
\rho^K_0(\ve, \cdot) < \rho_s(\ell_z) < \rho^K_1(\ve, \ell_z)  <  \rho_m(\ell_z) <\rho^K_2(\ve, \ell_z)
\end{equation*}
and 
\begin{equation*}
\frac{\partial \rho^K_i}{\partial\ve}(\ve, \ell_z) = \frac{1}{\frac{\partial E^K_{\ell_z}}{\partial \rho}(\rho^K_i(\ve, \ell_z), 0)} \quad\text{and}\quad \frac{\partial \rho^K_i}{\partial\ell_z}(\ve, \ell_z) = -\frac{\frac{\partial E^K_{\ell_z}}{\partial \ell_z}(\rho^K_i(\ve, \ell_z), 0)}{\frac{\partial E^K_{\ell_z}}{\partial \rho}(\rho^K_i(\ve, \ell_z), 0)}.
\end{equation*}
This yields the result. 
\end{proof}
\noindent Now, we state the following lemma 
\begin{lemma}
\label{dist:kerr}
There exists $\eta>0$ such that $\forall (\ve, \ell_z)\in B_{bound}$, we have
\begin{equation*}
\rho_1^K(\varepsilon, \ell_z) - \rho_0^K(\varepsilon, \ell_z)> 2\eta. 
\end{equation*}
\end{lemma}

\begin{proof}
Let $(\varepsilon, \ell_z)\in \Bbound$. Then, $\varepsilon\in[\varepsilon_1, \varepsilon_2]$ and $\ell_z\in[\ell_{1}(\varepsilon), \ell_{2}(\varepsilon)]$ where $\varepsilon_{min}(d) < \varepsilon_1<\varepsilon_2<1$ and $\ell_{lb}(\varepsilon)< \ell_{1}(\varepsilon) < \ell_{2}(\varepsilon) < \ell_{ub}(\varepsilon)$.  By monotonicity properties of $\rho_i^K$, we have $\forall\varepsilon\in[\varepsilon_1, \varepsilon_2], \forall\ell_z\in[\ell_{1}(\varepsilon), \ell_{2}(\varepsilon)]$
 \begin{equation*}
 \rho_1^K(\varepsilon, \ell_z) - \rho_0^K(\varepsilon, \ell_z) \geq \rho_1^K(\varepsilon, \ell_{1}(\varepsilon)) - \rho_0^K(\varepsilon, \ell_{1}(\varepsilon)) \geq \min_{[\varepsilon_1, \varepsilon_2]} \rho_1^K(\varepsilon, \ell_{1}(\varepsilon)) - \rho_0^K(\varepsilon, \ell_{1}(\varepsilon)) =: 2\eta.  
 \end{equation*}
 By compactness and regularity of $\rho_1^K$ and $\rho_0^K$, 
 \begin{equation*}
 \eta := \frac{1}{2}\left(\rho_1^K(\varepsilon_0, \ell_{1}(\varepsilon_0)) - \rho_0^K(\varepsilon_0, \ell_{1}(\varepsilon_0))\right) \quad\text{for some $\varepsilon_0\in[\varepsilon_1, \varepsilon_1]$}. 
 \end{equation*}
 In order to prove that $\eta>0$, it suffices to note that $\displaystyle \ell_1(\varepsilon_0)>\ell_{lb}(\varepsilon_0)$ so that $\rho^K_0$ and $\rho^K_1$ do not coincides (recall that this case only occurs when $\displaystyle \ell_1(\varepsilon_0) = \ell_{lb}(\varepsilon_0)$). 
In particular, we have: 
\begin{equation*}
\rho_1^K(\varepsilon_0, \ell_{1}(\varepsilon_0)) > \rho_0^K(\varepsilon_0, \ell_{1}(\varepsilon_0)).
\end{equation*}
\end{proof}
\noindent Consequently, the quantities 
\begin{equation}
\rho^K_{0, max} := \max_{\Bbound}\rho^K_0(\varepsilon, \ell_z)\quad,\quad \rho^K_{1, min} := \min_{\Bbound}\rho^K_1(\varepsilon, \ell_z)\quad\text{and}\quad \rho^K_{2, max} := \max_{\Bbound}\rho^K_2(\varepsilon, \ell_z)
\end{equation}
are well-defined. and we have 
\begin{equation}
\label{dist:eta}
\rho^K_{1, min} - \rho^K_{0, max} > 2\eta. 
\end{equation}
Now, we claim that 
\begin{lemma}
\label{lemma:35}
$\forall z\in[0, z_{max}^K(\ve, \ell_z)]$, there exist $\rho_1(z, \ve, \ell_z), \rho_2(z, \ve, \ell_z)\in [\rho_1^K(\ve, \ell_z), \rho_2^K(\ve, \ell_z)]$ solutions of the equation 
\begin{equation*}
E_{\ell_z}^K(\rho, z) = \ve
\end{equation*}
which satisfy 
\begin{equation*}
\rho_1(z, \ve, \ell_z) \leq \rho_{max}(\ve, \ell_z) \leq  \rho_2(z, \ve, \ell_z)
\end{equation*}
\end{lemma}
Let $(\ve, \ell_z)\in\Bbound$. 
\begin{itemize}
\item If $z_{max}(\ve, \ell_z) = 0$, then by Lemma \ref{z::max:def}, $Z^K(\ve, \ell_z)$ is confined in the equatorial plane: 
\begin{equation*}
\begin{aligned}
&Z^{K, abs} =  \left\{ \rho_0^K(\ve, \ell_z)\right\}, \\
&Z^{K, trapped}(\ve, \ell_z) = \left\{\rho_1^K(\ve, \ell_z), \rho_2^K(\ve, \ell_z) \right\}. 
\end{aligned}
\end{equation*}
\item Otherwise, we can choose $\overline z_{max}(\ve, \ell_z)$ and $\tilde z_{max}(\ve, \ell_z)$ such that
\begin{equation*}
0 <\overline z_{max}(\ve, \ell_z)<\tilde z_{max}(\ve, \ell_z)< z_{max}(\ve, \ell_z). 
\end{equation*}
By Lemma \ref{lemma:35}, there exist $\overline \rho_i(\ve, \ell_z), \tilde \rho_i(\ve, \ell_z) \in ]\rho_1^K(\ve, \ell_z), \rho_2^K(\ve, \ell_z)[, i\in\left\{ 1, 2\right\}$ which solve the equations 
\begin{equation*}
E_{\ell_z}^K(\rho, \overline z_{max}(\ve, \ell_z)) = \ve \quad\text{and}\quad E_{\ell_z}^K(\rho, \tilde z_{max}(\ve, \ell_z)) = \ve
\end{equation*}
respectively and which satisfy 
\begin{equation*}
\rho_1^K(\ve, \ell_z) <\overline \rho_1(z, \ve, \ell_z)<\tilde\rho_1(z, \ve, \ell_z) < \rho_{max}(\ve, \ell_z) <  \tilde \rho_2(z, \ve, \ell_z) < \overline \rho_2(z, \ve, \ell_z) < \rho_2^K(\ve, \ell_z). 
\end{equation*}
We introduce the following open subsets of $\BB$: 
\begin{align}
\BB^{abs} &:= \left]0,  \rho^K_{0, max} + \frac{3\eta}{2}\right[\times \mathbb R, \\
\BB_1 &:= \left]\rho^K_{0, max} + \frac{\eta}{2}, \infty\right[\times \left]\overline z_{max}(\ve, \ell_z) , +\infty\right[, \\
\BB_2 &:= \left]\rho^K_{0, max} + \frac{\eta}{2}, \infty\right[\times \left]-\infty,  - \overline z_{max}(\ve, \ell_z)\right[, \\
\BB_3 &:= \left]\rho^K_{0, max} + \frac{\eta}{2}, \tilde\rho_1(\ve, \ell_z)\right[\times \mathbb R, \\
\BB_4 &:= \left]\tilde\rho_2(\ve, \ell_z), \infty\right[\times \mathbb R. 
\end{align}
In order to  define $\BB_5$, we choose $z_r(\ve, \ell_z)\in]\overline z_{max}(\ve, \ell_z), \tilde z_{max}(\ve, \ell_z)[$, $\rho^1_r(\ve, \ell_z)\in]\overline\rho_1(z, \ve, \ell_z), \tilde\rho_1(z, \ve, \ell_z)[$ and $\rho^2_r(\ve, \ell_z)\in]\overline\rho_2(z, \ve, \ell_z), \tilde\rho_2(z, \ve, \ell_z)[$ and we set 
\begin{equation}
\BB_5 := ]\rho_r^1(\ve, \ell_z), \rho_r^2(\ve, \ell_z)[\times]-z_r(\ve, \ell_z), z_r(\ve, \ell_z)[. 
\end{equation}

\end{itemize}

\begin{remark}
If $z_{max}(\ve, \ell_z) = 0$, then the geodesic is confined in the equatorial plane and $Z^K(\ve, \ell_z)$ is reduced to the set of points $\rho_i^K(\ve, \ell_z)$. In this case, there is no work to be done here. Indeed, the classification of the equatorial orbits is the same as that of Schwarzschild geodesics. 
\\In our work, we are interested in the general case (not necessarily equatorial orbits.). Therefore, we will assume that $\displaystyle z_{max}(\ve, \ell_z) > 0$.
\end{remark}
\noindent  By construction, the above subsets cover $\BB$ (See Figure \ref{decomposition}):
\begin{lemma}
\label{decomp:B}
We have 
\begin{equation*}
\BB = \BB^{abs}\cup\left(\cup_{i= 1\cdots 5}\BB_i \right). 
\end{equation*}
\end{lemma}
\noindent Now, we  introduce the following functions 
\begin{equation}
\label{rho:z:abs}
\begin{aligned}
\rho_{abs}^0(\ve, \ell_z, \theta)&:= \sqrt{\Delta(r_{abs}^0(\ve, \ell_z, \theta))}\sin\theta, \\
z_{abs}^0(\ve, \ell_z, \theta)&:= (r_{abs}^0(\ve, \ell_z, \theta) - 1)\cos\theta, \\
\end{aligned}
\end{equation}
\begin{equation}
\label{rho:z:abs}
\begin{aligned}
\rho_{tr}^1(\ve, \ell_z, \theta)&:= \sqrt{\Delta(r_{tr}^1(\ve, \ell_z, \theta))}\sin\theta, \\
z_{tr}^1(\ve, \ell_z, \theta)&:= (r_{tr}^1(\ve, \ell_z, \theta) - 1)\cos\theta, \\
\end{aligned}
\end{equation}

\begin{equation}
\label{rh:z:2tr}
\begin{aligned}
\rho_{tr}^2(\ve, \ell_z, \theta)&:= \sqrt{\Delta(r_{tr}^2(\ve, \ell_z, \theta))}\sin\theta, \\
z_{tr}^2(\ve, \ell_z, \theta)&:= (r_{tr}^2(\ve, \ell_z, \theta) - 1)\cos\theta, \\
\end{aligned}
\end{equation}

\begin{equation}
\label{rh:z:tr}
\begin{aligned}
\rho_{tr}(\ve, \ell_z, r)&:= \sqrt{\Delta(r)}\sin\theta_{tr}(\ve, \ell_z, r), \\
z_{tr}(\ve, \ell_z, r)&:= (r - 1)\cos\theta_{tr}(\ve, \ell_z, r), \\
\end{aligned}
\end{equation}
where $r_{abs}^0$, $r_{abs}^1$, $r_{abs}^2$ and \eqref{rh:z:tr} are defined by \eqref{r0:abs:theta}, \eqref{r1:tr:theta}, \eqref{r2:tr:theta}, \eqref{theta::tr:r} respectively.
We state the following lemma 
\begin{lemma}
\label{lemme::49}
$\forall (\ve, \ell_z)\in \Abound$, we have 
\begin{itemize}
\item $z_{abs}^0(\ve, \ell_z, \cdot)$,  $z_{tr}^1(\ve, \ell_z, \cdot)$ and $z_{tr}^2(\ve, \ell_z, \cdot)$ are monotonically decreasing on $]0, \pi[$,
\item $\rho_{tr}(\ve, \ell_z, \cdot)$ is monotonically increasing on $]r^K_1(\ve, \ell_z), r^K_2(\ve, \ell_z)[$. 
\end{itemize}
\end{lemma}
\begin{proof}
\begin{enumerate}
\item First of all, it is easy to see that $\displaystyle z_{abs}^0(\ve, \ell_z, \cdot)$ is smooth on $]0, \pi[$ and $\forall \theta\in]0, \pi[,$ we have 
\begin{equation*}
\frac{\partial z_{abs}^0}{\partial\theta}(\ve, \ell_z, \theta) =\frac{\partial r_{abs}^0}{\partial\theta}(\ve, \ell_z, \theta)\cos\theta  - (r_{abs}^0 - 1)\sin\theta. 
\end{equation*}
By the monotonicity properties of $r_{abs}^0(\ve, \ell_z,\cdot)$ (decreasing  on $\displaystyle \left]0, \frac{\pi}{2}\right[$ and increasing on $\displaystyle \left]\frac{\pi}{2}, \pi\right[$),  the first term of the right hand side is always negative. Moreover, the second term is always negative on $]0, \pi[$. Therefore, $z_{abs}^0(\ve, \ell_z, \cdot)$  is monotonically decreasing on $]0, \pi[$ and we can write write $\theta$ in terms of $z$. The same analysis can be made for $z_{tr}^1(\ve, \ell_z, \cdot)$ and $z_{tr}^2(\ve, \ell_z, \cdot)$. 
\item Since $\theta_{tr}(\ve, \ell_z, \cdot)$ is smooth,  $\rho_{tr}(\ve, \ell_z, \cdot)$ is smooth on $]r^K_1(\ve, \ell_z), r^K_2(\ve, \ell_z)[$ and we have 
\begin{equation*}
\partial_r\rho_{tr}(\ve, \ell_z, r) = \frac{\Delta'(r)}{2\sqrt{\Delta(r)}}\sin\theta_{tr}(\ve, \ell_z, r) + \sqrt{\Delta(r)}\partial_r\theta_{tr}(\ve, \ell_z, r)\cos\theta_{tr}(\ve, \ell_z, r). 
\end{equation*}
Since $\theta_{tr}(\ve, \ell_z, \cdot)$ is monotonically increasing on $]r^K_1(\ve, \ell_z), \tilde r(\ve, \ell_z)[$ and $\theta_{tr}(\ve, \ell_z, r)\in \left]\overline \theta_{max}^{<1}(\ve, \ell_z), \frac{\pi}{2}\right[$, the second term is always positive. Therefore, $\rho_{tr}(\ve, \ell_z, \cdot)$ is monotonically increasing on $]r^K_1(\ve, \ell_z), r^K_2(\ve, \ell_z)[$ and we can write $r$ in terms of $\rho$. 
\end{enumerate}
\end{proof}
\noindent Now, we introduce the following functions
\begin{definition}
\label{def::atlas}
Let $(\ve, \ell_z)\in \Abound$. 
\begin{enumerate}
\item $\displaystyle \Phi^{K, abs}_{(\ve, \ell_z)}: ]-\gamma, \gamma[\mapsto ]0, \rho_0^K(\ve, \ell_z)]$ is defined by 
\begin{equation}
\Phi^{K, abs}_{(\ve, \ell_z)}(z):=  \rho^0_{abs}(\ve, \ell_z, (z_{abs}^0(\ve, \ell_z, \cdot))^{-1}(z)) ,
\end{equation}
\item $\displaystyle \Phi^{K, i}_{(\ve, \ell_z)}$ are defined in the following way 
\begin{equation}
\label{I:i:def}
\begin{aligned}
\Phi^{K, 1}_{(\ve, \ell_z)}&:I^1_{(\ve, \ell_z)}:= ]\overline \rho_1(\ve, \ell_z), \overline \rho_2(\ve, \ell_z)[ \mapsto ]\overline z_{max}(\ve, \ell_z), z_{max}(\ve, \ell_z)] , \\
\Phi^{K, 2}_{(\ve, \ell_z)}&:I^2_{(\ve, \ell_z)}:= ]\overline \rho_1(\ve, \ell_z), \overline \rho_2(\ve, \ell_z)[ \mapsto [-z_{max}(\ve, \ell_z), -\overline z_{max}(\ve, \ell_z)[ \\
\Phi^{K, 3}_{(\ve, \ell_z)}&:I^3_{(\ve, \ell_z)}:= ]-\tilde z_{max}(\ve, \ell_z), \tilde z_{max}(\ve, \ell_z)[ \mapsto ]\tilde \rho_1(\ve, \ell_z), \rho_1^K(\ve, \ell_z)] , \\
\Phi^{K, 4}_{(\ve, \ell_z)}&:I^4_{(\ve, \ell_z)}:= ]-\tilde z_{max}(\ve, \ell_z), \tilde z_{max}(\ve, \ell_z)[ \mapsto ]\tilde \rho_2(\ve, \ell_z), \rho_2^K(\ve, \ell_z)],\\
\end{aligned}
\end{equation}
where 
\begin{equation}
\begin{aligned}
&\Phi^{K, 1}_{(\ve, \ell_z)}(\rho):=  z_{tr}(\ve, \ell_z, (\rho_{tr}(\ve, \ell_z, \cdot))^{-1}(\rho)) , \\
&\Phi^{K, 2}_{(\ve, \ell_z)}(\rho):= - \Phi^{K, 1}_{(\ve, \ell_z)}(\rho) , \\
&\Phi^{K, 3}_{(\ve, \ell_z)}(z):=  \rho^1_{tr}(\ve, \ell_z, (z_{tr}^1(\ve, \ell_z, \cdot))^{-1}(z)) , \\
&\Phi^{K, 4}_{(\ve, \ell_z)}(z):=  \rho^2_{tr}(\ve, \ell_z, (z_{tr}^2(\ve, \ell_z, \cdot))^{-1}(z)) . 
\end{aligned}
\end{equation}

\end{enumerate}
\end{definition}
\noindent Following Lemma \ref{lemme::49}, we obtain 
\begin{lemma}
\label{lemma::50}
$\forall (\ve, \ell_z)\in\Abound$, the functions $\Phi^{K, abs}_{\ve, \ell_z}$ and $\Phi^{K, i}_{\ve, \ell_z}$  are well-defined and smooth on $]-\gamma, \gamma[$ and $I^i_{(\ve, \ell_z)}$ respectively.  Moreover, 
\begin{enumerate}
\item $\Phi^{K, abs}_{(\ve, \ell_z)}$ admits a unique critical point (a global maximum) on $]-\gamma, \gamma[$ given by $\rho^0_{abs}(\ve, \ell_z)$, reached at the point $z = 0$.   
\item $\Phi^{K, 1}_{(\ve, \ell_z)}$ admits a unique critical point (a global maximum) on $I^1_{(\ve, \ell_z)}$ given by $z_{max}(\ve, \ell_z)$, reached at the point $\rho = \rho_{max}(\ve, \ell_z)$.
\item $\Phi^{K, 2}_{(\ve, \ell_z)}$ admits a unique critical point (a global minimum) on $I^2_{(\ve, \ell_z)}$ given by $-z_{max}(\ve, \ell_z)$, reached at the point $\rho = \rho_{max}(\ve, \ell_z)$.
\item $\Phi^{K, 3}_{(\ve, \ell_z)}$ admits a unique critical point (a global minimum) on $I^3_{(\ve, \ell_z)}$ given by $\rho^1_{abs}(\ve, \ell_z)$, reached at the point $z = 0$.   
\item $\Phi^{K, 4}_{(\ve, \ell_z)}$ admits a unique critical point (a global maximum) on $I^4_{(\ve, \ell_z)}$ given by $\rho^2_{abs}(\ve, \ell_z)$, reached at the point $z = 0$.    
\end{enumerate}
\end{lemma}
\noindent Now, we use the previous results to obtain 
\begin{Propo}
\label{repara:Z:K}
$\forall (\ve, \ell_z)\in \Bbound$ ,
\begin{enumerate}
\item there exists $\displaystyle \Phi^{K, abs}_{(\ve, \ell_z)}: ]-\gamma, \gamma[\mapsto ]0, \rho_0^K(\ve, \ell_z)]$ such that 
\begin{equation*}
Z^{K, abs}(\ve, \ell_z) = \text{Gr} \left(\Phi^{K, abs}_{(\ve, \ell_z)}\right),
\end{equation*} 
\item there exist $\displaystyle \Phi^{K, i}_{(\ve, \ell_z)}: I^i_{(\ve, \ell_z)}\mapsto \mathbb R,\; i=1\cdots 4$ such that 
\begin{equation*}
Z^{K, trapped}(\ve, \ell_z) = \bigcup_{i=1\cdots 4}\text{Gr} \left(\Phi^{K, i}_{(\ve, \ell_z)}\right),
\end{equation*} 
\end{enumerate}
where $\displaystyle \Phi^{K, abs}_{(\ve, \ell_z)}$ and $\displaystyle \Phi^{K, i}_{(\ve, \ell_z)}$ are given by Definition \ref{def::atlas}. 
\end{Propo}
\noindent To summarise, instead of solving the equation on $\BB$, we solve it on $\BB_i$ and $\BB^{abs}$ and on each of these region, the solutions are either functions of $\rho$ or $z$. 
\\ Therefore, the problem of finding solutions on $\BB$ to the equation 
\begin{equation}
\label{roots:rho}
E_{\ell_z}^K(\rho, z) = \ve
\end{equation}
is equivalent to the problem of finding solutions on $\BB_i$ and on $\BB^{abs}$. Given $(\ve, \ell_z)\in\Bbound$, the latter is equivalent to the problem of finding a function defined on $I^i_{(\ve, \ell_z)}$ or $I^{abs}_{(\ve, \ell_z)}$. Hence, by the above proposition, the solutions are given by $\Phi^{K, abs}_{(\ve, \ell_z)}$ in $\BB^{abs}$ and by  $\Phi^{K, i}_{(\ve, \ell_z)}$ in $\BB_{i}$.

\begin{figure}[h!]
\subfloat[First component $Z^{K, abs}(\ve, \ell_z)$]{
\label{fig::cover1}
\includegraphics[width=0.5\textwidth]{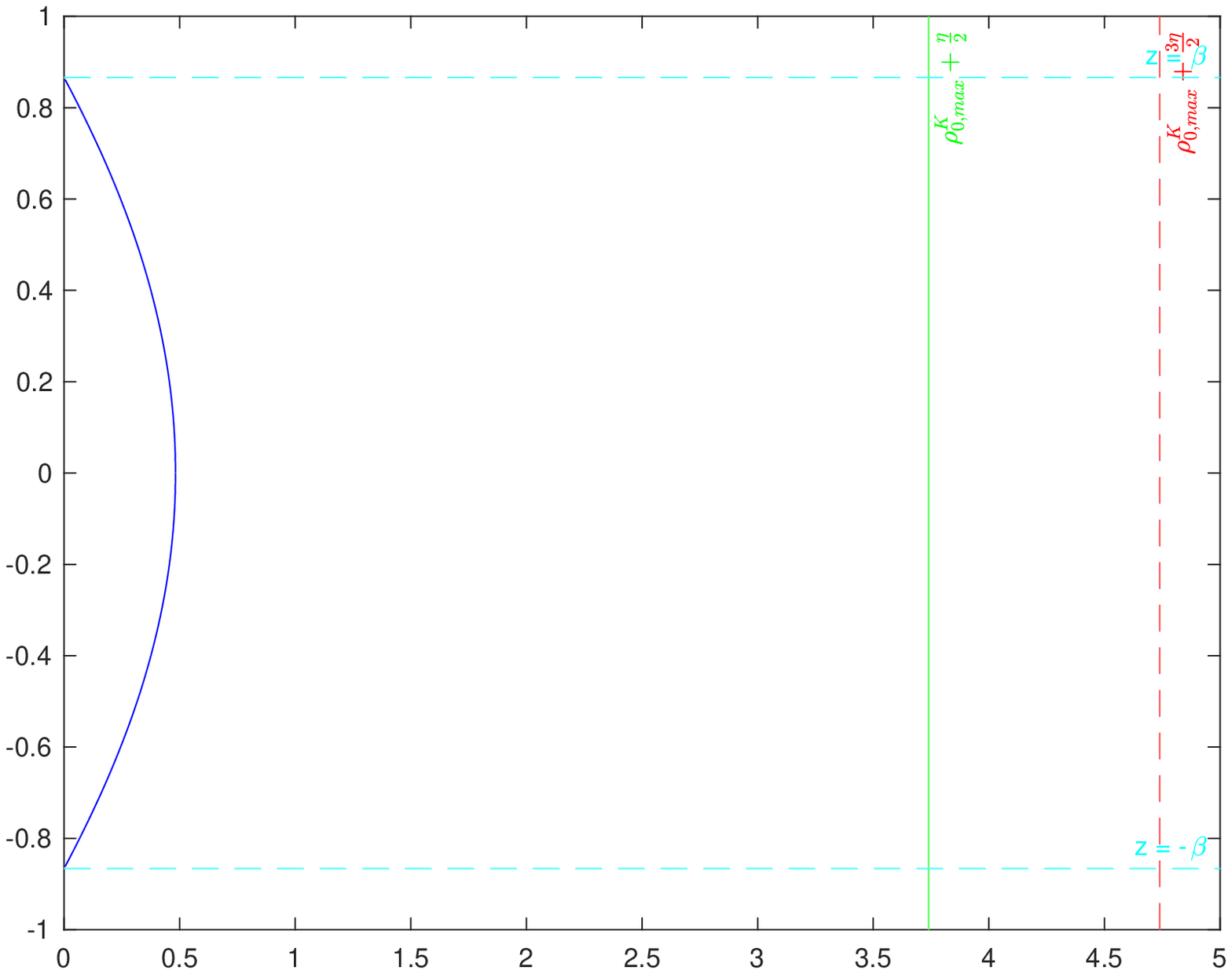}
}\hfill
\subfloat[Second component $Z^{K, trapped}(\ve, \ell_z)$]{
\label{fig::cover2}
\includegraphics[width=0.5\textwidth]{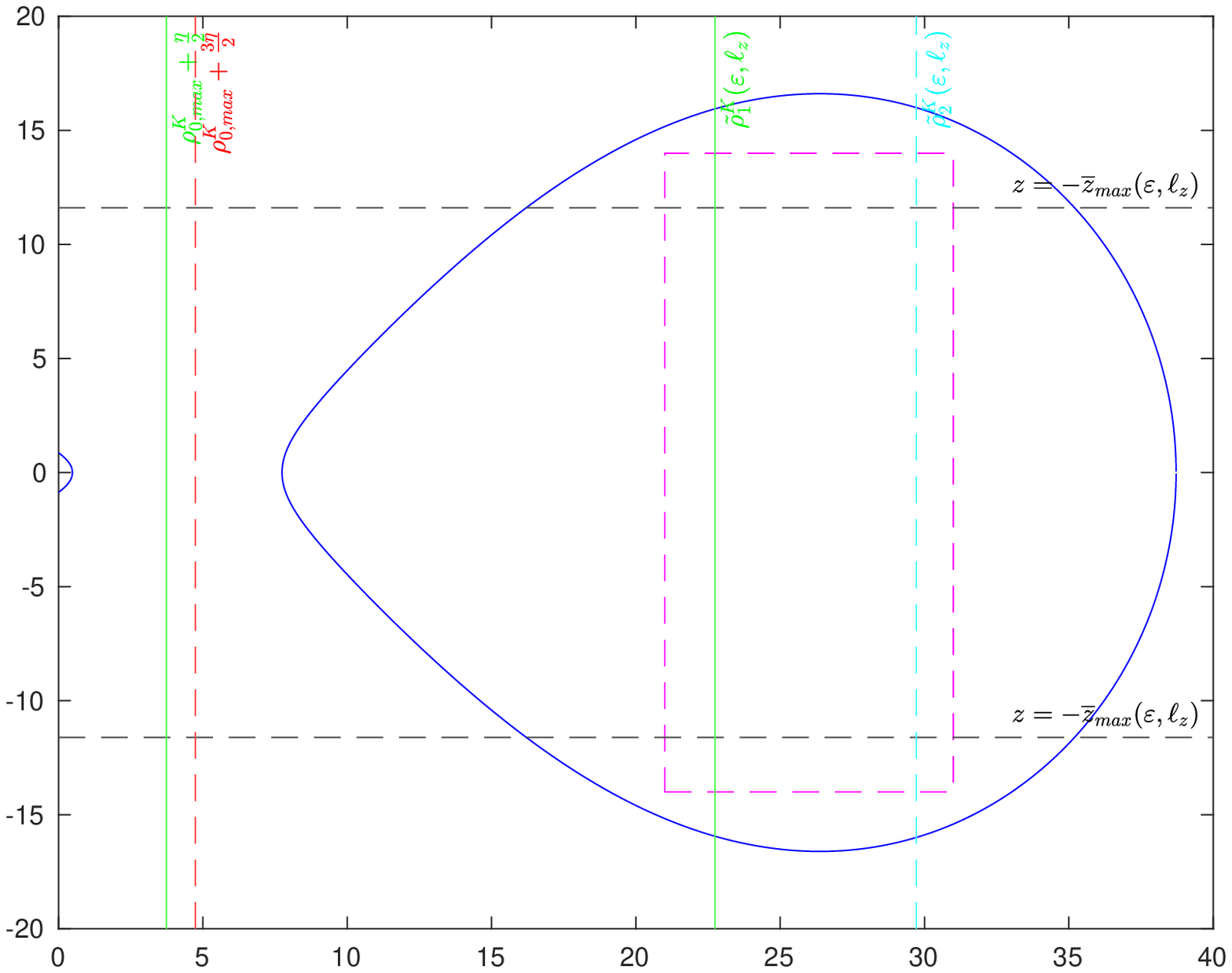}
}\hfill
\caption{\it The zero velocity curve (in blue)  with $d= 0.5$ and $(\ve, \ell_z) = (0.98, 4)\in\Abound$. The dashed lines separate $\BB$ into six regions: $\BB_0$ is the region that contains $Z^{K, abs}$ and it is delimited by the red line (Figure (a)). The region $\BB_5$ is the region bounded by the cyan rectangle in Figure (b). $\BB_3$ is the region delimited by the two green lines in Figure (b). $\BB_4$ is the region delimited by the cyan vertical line  from the left. The horizontal black line above the equatorial plane delimits the region $\BB_1$ from below and the horizontal black line below the equatorial plane delimits the region $\BB_2$ from above.}
\label{decomposition}
\end{figure}
\subsubsection{Domain of trapped geodesics in a sub-extremal Kerr spacetime}
The aim of this section is to prove that  the region $\BB^{trapped}\subset \Bbarre$ where geodesic motion occurs stay away from $\partial\Bbarre$. From the previous section, we recall that trapped timelike future-directed orbits occur when the associated zero velocity curve has a compact connected component, $Z^{K, trapped}\subset \BB$. In this case, the orbit is confined in the region bounded by the latter curve. 
We note that the zero velocity curves are only defined in $\BB$. More precisely, the effective potential energy is only defined in the exterior region of the spacetime minus the axis of symmetry. However, since we are interested in metrics which are  at least $C^2-$extendable to the horizon and the axis of symmetry, the boundaries of $\BB$, we need to consider the extension of $Z^K(\ve, \ell_z)$ in $\Bbarre$ associated to timelike future-directed geodesics in order to determine $\BB^{trapped}\subset \Bbarre$. Given $(\ve, \ell_z)\in \Abound$, $Z^K(\ve, \ell_z)$ may have accumulation points on the boundary, namely on the poles. However, in the case of the second connected component of $Z^K(\ve, \ell_z)$, that is  $Z^{K, trapped}(\ve, \ell_z)$, the accumulation points are distincts from $\partial \Bbarre$ and thus from the poles. 
\\Before we state the precise result of this section, we recall the definition of $\Bbarre$
\begin{equation*}
            \overline{\mathscr{B}} :=  \overline{\mathscr{B}_{\mathcal{A}}}\cup\overline{\mathscr{B}_{\mathcal{H}}}\cup\overline{\mathscr{B}_{N}}\cup\overline{\mathscr{B}_{S}}.
        \end{equation*}
where
        \begin{equation*}
            {\overline{\mathscr{B}_{H}}}:= \left\{(\rho,z)\in\overline{\mathscr{B}}, \rho^2+(z\pm\beta)^2>\frac{\beta}{a}, |z|+|\rho|<\left( 1 + \frac{1}{b}\right)\beta   \right\}, 
        \end{equation*}

        \begin{equation*}
            {\overline{\mathscr{B}_{A}}}:= \left\{(\rho,z)\in\overline{\mathscr{B}}, \rho^2+(z\pm\beta)^2>\frac{\beta}{a}, |z|+|\rho|>\left( 1 - \frac{1}{b}\right)\beta   \right\}.
        \end{equation*}
        
                \begin{equation*}
            {\overline{\mathscr{B}_{N}}} :=  \left\{ (\rho,z)\in\Bbarre,\; z\neq \beta,\;  \rho^2+(z-\beta)^2<\frac{\beta}{c} \right\}\cup \left\{ (s,\chi)\in\Bbarre \quad / 0\le s,\chi<\left(\frac{\beta}{e}\right)^{\frac{1}{4}} \right\} .
        \end{equation*}
      
                \begin{equation*}
            {\overline{\mathscr{B}_{N}}} :=  \left\{ (\rho,z)\in\Bbarre,\; z\neq -\beta,\;  \rho^2+(z+\beta)^2<\frac{\beta}{c} \right\}\cup \left\{ (s',\chi')\in\Bbarre \quad / 0\le s',\chi'<\left(\frac{\beta}{e}\right)^{\frac{1}{4}} \right\} .
        \end{equation*}
  for $0<e<c<a<b$. 

\noindent We state the following result 
\begin{lemma}
\label{choice:of:cons}
Let $\Bbound\subset\subset \Abound$. Then, we can choose $0<e<c$ uniform in $(\ve, \ell_z)$ such that $\forall (\ve, \ell_z)\in\Bbound$, 
\begin{equation}
\label{no::trapped}
\Bnbarre, \Bsbarre \cap Z^{K, trapped}(\ve, \ell_z) = \emptyset.  
\end{equation}
Moreover, $\forall (\ve, \ell_z)\in \Abound$, the accumulation points for $Z^{K, abs}(\ve, \ell_z)$ are $p_N:= (0, \beta)$ and $p_S:= (0, -\beta)$. 
\end{lemma}
\begin{proof}

\end{proof}
\noindent Now, we define the domain of  trapped timelike geodesics, $\BB^{\pm, trapped}(a, M)$, which depends only on $(a, M)$ by the following proposition
\begin{Propo}[Domain of trapped geodesics in a sub-extremal Kerr spacetime]
\label{domain:trapped}
Let  $(a, M)$ be such that $0<|a|<M$. Then, there exists $\rho^{mb, \pm}(a, M)>0$
\begin{equation*}
\forall (\ve, \ell_z)\in\Abound\quad\;\quad \rho^{K, \pm}_1(\ve, \ell_z, a, M)>  \rho^{mb, \pm}(a, M)>0, 
\end{equation*}
where $\rho_1^{K, \pm}(\ve, \ell_z, a, M)$  is the second largest root of the equation 
\begin{equation*}
E_{\ell_z}(\rho, 0) = \ve
\end{equation*}
and 
\begin{equation*}
\rho^{mb, \pm}(a, M) :=\sqrt{\Delta\left(r^\pm_{mb}(a, M)\right)}, \quad r^\pm_{mb}(a, M):=  2M \mp  a +  2\sqrt{M^2\mp aM}. 
\end{equation*}
Hence, we set 
\begin{equation*}
\BB^{\pm, trapped}(a, M) := \left]\rho^{mb, \pm}(a, M), \infty\right[\times\mathbb R
\end{equation*}
\end{Propo}
\begin{remark}
We recall that $r^\pm_{mb}(a, M)$ is defined to be the unique solution of the equation 
\begin{equation*}
\Phi_\pm(r) = 1,
\end{equation*}
where $\Phi_\pm$ is defined by \eqref{phi::pm}. 
\end{remark}
\begin{proof}
Let $(\ve, \ell_z)\in\Abound$ and recall that $\displaystyle d = \frac{a}{M}$. Then, by monotonicity properties  of $\tilde\rho_1^{K, \pm}$\footnote{Recall that $\tilde\rho$ is defined by $ \rho = M\tilde\rho$.}, proved in Lemma \ref{monotonicity:rho:i}, we have 
\begin{equation*}
\tilde\rho_1^{K, \pm}(\ve, \ell_z, d) > \tilde\rho_1^{K, \pm}(\ve, \ell_{ub}(\ve, d), d) = \tilde\rho^{K, \pm}_s(\ell_{ub}^\pm(\ve, d), d) > \tilde\rho^{K, \pm}_s(\ell_{ub}^\pm(1, d), d) = \sqrt{\Delta\left(\tilde r^\pm_{mb}(d)\right)}. 
\end{equation*}
It remains to show that $\displaystyle \rho^{mb, \pm}(a, M)>0$ for all $\displaystyle (a, M)$ such that $\displaystyle 0<|a|<M$.
\\ By the definition of $\displaystyle \rho^{mb, \pm}(a, M)$, it is easy to see that $\displaystyle \rho^{mb, -}(a, M)$ never vanishes and $\displaystyle \rho^{mb, +}(a, M)$ vanishes if and only if $|a| = M$. Indeed, 
\begin{enumerate}
\item $\tilde \rho^{mb, +}:[0, 1[\to \mathbb R_+$ is monotonically decreasing on $]0, 1[$ and we have 
\begin{equation*}
\lim_{d\to 1}\,\tilde \rho^{mb, +}(d) = 0.
\end{equation*}
\item $\tilde \rho^{mb, -}:[0, 1[\to \mathbb R_+$ is monotonically increasing on $]0, 1[$ and we have 
\begin{equation*}
\lim_{d\to 1}\,\tilde \rho^{mb, -}(d) \approx 4.83. 
\end{equation*}
\end{enumerate}
\end{proof}
\noindent For later purposes, in particular the study of  regularity of the matter terms, It will be convenient to adjust the definition of $\BHbarre$ and $\BAbarre$ so that the support of Vlasov matter remain in the region $\BAbarre$. More precisely, we will show the following 
\begin{lemma}
\label{support:matter}
There exists $b>0$ such that $\forall (\ve, \ell_z)\in\Abound$, we have 
\begin{equation*}
\BHbarre \cap Z^{K, trapped}(\ve, \ell_z) = \emptyset \quad\text{and}\quad  Z^{K, trapped}(\ve, \ell_z)\subset \BAbarre. 
\end{equation*}
\end{lemma}
\begin{proof}
First of all, we claim that $\forall d\in]0, 1[$
\begin{equation*}
\tilde \rho^{mb, -}(d)> \sqrt{1-d^2} \quad\text{and}\quad  \tilde \rho^{mb, +}(d)> \sqrt{1-d^2}. 
\end{equation*}
\\ Therefore, 
\begin{equation*}
\forall (\ve, \ell_z)\in \Abound\;,\; \tilde \rho_1^{K, \pm}(\ve, \ell_z) > \tilde \rho^{mb, \pm}(d) > \left(1 + \frac{1}{b}\right)\sqrt{1 - d^2}. 
\end{equation*}

\end{proof}
\begin{remark}
We note that,  by Proposition \ref{domain:trapped}, in a sub-extremal Kerr black hole, all trapped timelike geodesics lie in $\BB^{\pm, trapped}(a, M)$. However, in the limiting case (extremal Kerr), $\BB^{+, trapped}(a, M)$ coincides with the whole exterior region, $\BB$. As for, $\BB^{-, trapped}(a, M)$, it is located at approximatively $4.83M$ away from the horizon (See proof of Proposition  \ref{domain:trapped}). 
\\In this work, we use the fact that trapped timelike geodesics lie away from the horizon and thus we consider only the sub extremal case. In fact, we will need a non-trivial lower bound on the inner boundary of the support of the Vlasov matter. 
\\ In the extremal case, one possibility could be to consider only orbits which are retrograde. 
\end{remark}
 
\noindent One can wonder what is the location of trapped non-spherical orbits with respect to the ergoregion $\mathscr{E}$.  Recall that the ergosurface, the boundary of $\mathscr{E}$,  denoted by $\mathscr S$, is defined to be the set of points $(\rho, z)\in\BB$ such that 
\begin{equation*}
g(T, T) = 0. 
\end{equation*}
In a sub-extremal Kerr exterior $0<d<1$, the latter is equivalent to the set of points $(\rho, z)\in \Bbarre$ such that 
\begin{equation*}
V_K(\rho, z) = 0. 
\end{equation*}
Direct computations lead to 
\begin{equation*}
\mathscr S = \left\{(\rho, z)\in\Bbarre \;:\; z^2 = \left(1 - \frac{\rho}{d}\right)\left(1 - d^2\left(1 - \frac{\rho}{d}\right) \right) \right\}. 
\end{equation*}
We refer to Figure \ref{ergoregion} for the shape of $\mathscr S$. On the other hand, we have
\begin{itemize}
\item $\displaystyle \forall d\in]0, 1[,$
\begin{equation*}
\tilde \rho^{mb, -}(d) > d. 
\end{equation*}
\item $\displaystyle d_0:=   2(\sqrt 2 - 1)$ is the unique solution in $]0, 1[$ to 
\begin{equation*}
\tilde \rho^{mb, +}(d) = d. 
\end{equation*}
\end{itemize}
Note that $\mathscr S$ intersect the equatorial plane uniquely at the point $\rho_{eq} = d$. Therefore, it is straightforward to see that 
\begin{equation}
\label{retro::ergo}
\BB^{-, trapped}(a, M)\cap \mathscr E = \emptyset. 
\end{equation}
and 
\begin{equation}
\BB^{+, trapped}(a, M)\cap \mathscr E = \emptyset\quad\text{if and only if}\quad d<d_0. 
\end{equation}
We refer to Figure \ref{direct:ergo} to visualise the intersection when $d>d_0$ . 
\begin{figure}[h!]
\includegraphics[width=\linewidth]{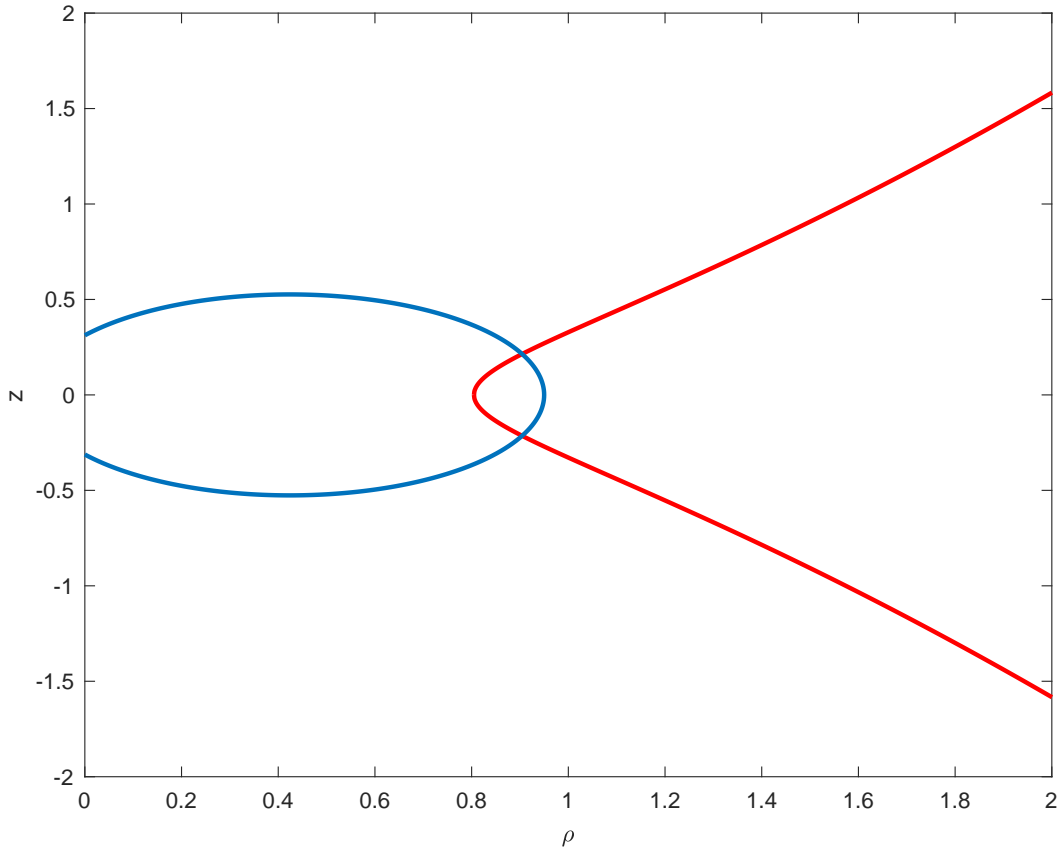}
\caption{\it Intersection of $\mathscr E(d)$ (blue curve) and $Z^{K, trapped}(\ve, \ell_z)$ (red curve) when $d = 0.95$. The remaining parameters are set to :$\ve = 0.99$ and $\ell_z = 2.5$. }
\label{direct:ergo}
\end{figure}
\begin{figure}[h!]
\includegraphics[width=\linewidth]{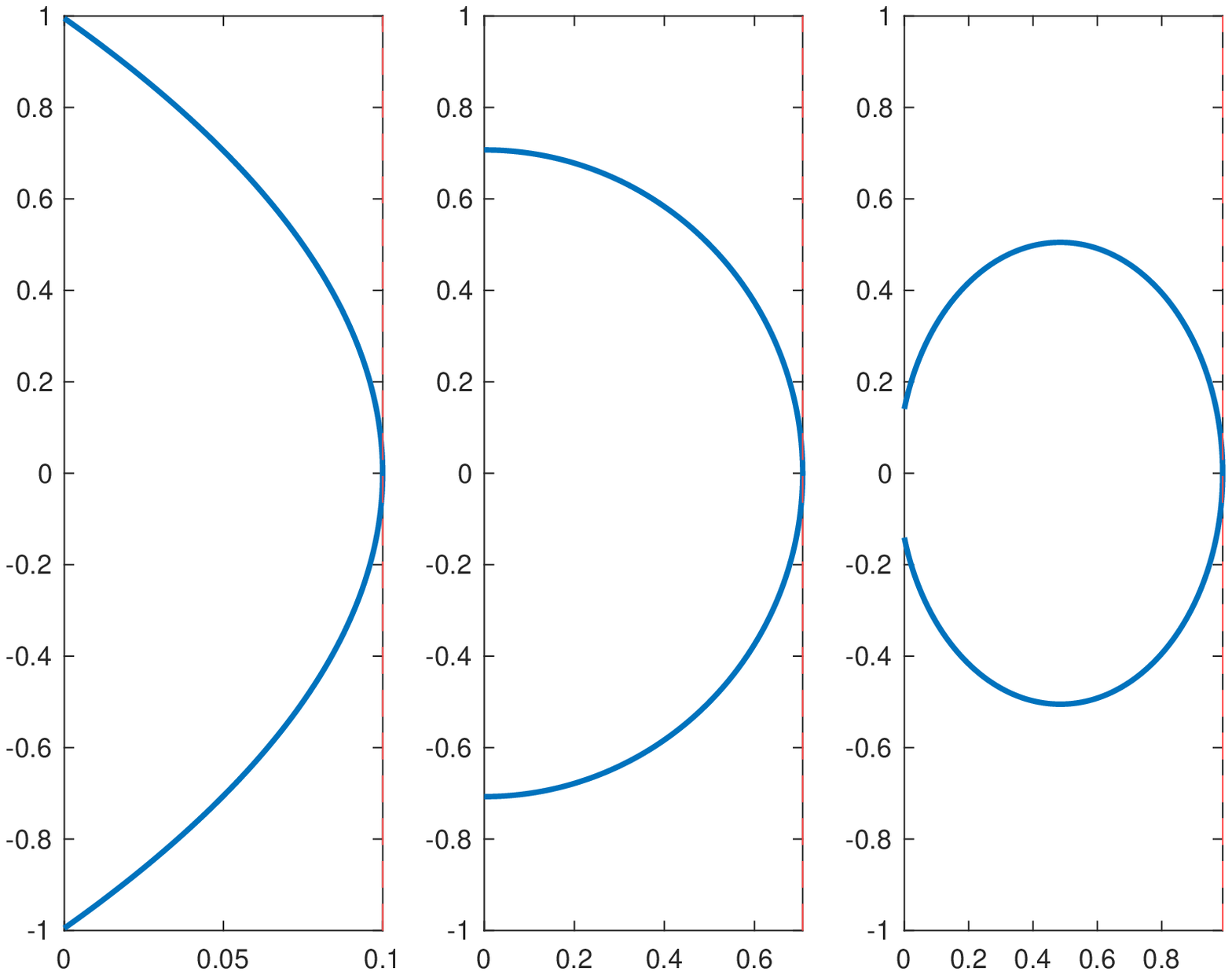}
\caption{\it Shape of $S$ in the following cases from the left to the right: $d = 0.1$, $d = \sqrt{0.5}$ and $d = 0.99$}
\label{ergoregion}
\end{figure}
\begin{remark}
\begin{itemize}
\item[-] Recall that a necessary condition for energy extraction from the black hole and thus a negative energy $\ve$ is that $d\ell_z< 0$ and $g_{tt} > 0$, which means that retrograde orbits that lie inside the ergoregion have negative energies. This gives another explanation for \eqref{retro::ergo}. 
\item[-] Finally, we note that when computing the matter terms in \ref{Further:computations}, we will have to take into account the ergoregion. In this case, we will have to split the matter sources into a term supported in the ergoregion and  another which is supported in the remaining region of the spacetime exterior. 
\end{itemize}
\end{remark}

\section{Reduced Einstein-Vlasov system}
\label{EV:reduced:system}
In this section, we compute the components of the energy momentum tensor and we reduce the Einstein equations with a source arising from a Vlasov field to a system of elliptic equations.
\subsection{Assumptions and General Framework}
Recall from Section \ref{s:metric:ansatz} the metric ansatz considered throughout this work: 
\begin{equation}
\label{ansatz:metric}
g = -Vdt^2 + 2Wdtd\phi + Xd\phi^2 + e^{2\lambda}\left( d\rho^2 + dz^2\right),
\end{equation}
where $V, W, X, \lambda : \BB\to \mathbb R$ and $\BB = \left\{(\rho, z)\;:\; \rho>0 \;,\; z\in\mathbb R \right\}.$
Following the work \cite{chodosh2015stationary}, we replace the metric components $V, W, X, \lambda$ by $(X, W, \theta, \sigma, \lambda)$, called "metric data", which reduces under symmetries in a nice manner and where $\theta$ and $\sigma$ are given by \eqref{twist:phi} and \eqref{def:sigma} respectively. 
\subsection{Ansatz for the distribution function}
\label{Ansatz:for:the:distribution:function}
We are interested in stationary and axisymmetric distribution functions. Therefore, we assume that $f : \Gamma_1\to\mathbb R_+$ takes the form
\begin{equation}
\label{ansatz:for:f}
f(t, \phi, \rho, z, \phi, v^\rho, v^\phi, v^z) = \Phi(\ve, \ell_z)\Psi_\eta(\rho, (\ve, \ell_z), (X, W, \sigma))
\end{equation}
where 
\begin{itemize}
\item $\displaystyle \Phi : \mathbb R\times\mathbb R \to \mathbb R_+$ is a $C^2$ function  and is supported on some compact set  $\displaystyle \Bbound$ of $\Abound$. Without loss of generality (See Lemma \ref{compact:Bbound}), we assume that  $\Bbound$ has the form
\begin{equation*}
\Bbound := \Bbound^-\cup \Bbound^+
\end{equation*}
where $\Bbound^-$ and $\Bbound^+$ are defined by 
\begin{equation*}
\Bbound^\pm := [\ve_1^\pm, \ve_2^\pm]\times[\ell_1^\pm, \ell_2^\pm] 
\end{equation*}
where $\displaystyle \ve^{i, \pm}_j$ and $\displaystyle \ell^{i, \pm}_j$ satisfy 
\begin{equation}
\label{bounds:}
\ve^\pm_{min}(d)<\ve^\pm_1<\ve^\pm_2<1 \quad\text{,}\quad \ell_{lb}(\ve^+_2)<\ell^+_1<\ell^+_2<\ell_{ub}(\ve^+_1) \quad\text{and}\quad \ell_{ub}(\ve^-_1)<\ell^-_1<\ell^-_2<\ell_{lb}(\ve^-_2). 
\end{equation} 
\item $\eta> 0$ is a constant that will be specified later (see Section \ref{perturbed:Kgeo}), $\Psi_\eta(\cdot, \cdot, h)\in C^\infty(]0, \infty[\times\Adm, \mathbb R_+)$ is a cut-off function depending on the metric data $h := (X, W, \sigma)$, such that     
\begin{equation}
\label{cut::off}
\Psi_\eta(\cdot, (\ve, \ell), h) := 
\left\{
\begin{aligned}
& \Chi_\eta(\cdot - \rho_1(h, (\ve, \ell_z))), \quad (\ve, \ell_z)\in \Bbound,\\
& 0 \quad\quad\quad\quad\quad\quad\quad\quad\quad\quad (\ve, \ell_z)\notin \Abound,
\end{aligned}
\right.
\end{equation}
where $\rho_1$ is a positive function of $(h, \ve, \ell_z)$ which will be defined later \footnote{$\rho_1$ is defined by \eqref{pert::quantities}. It can be seen as the perturbation of $\rho^1_K(\ve, \ell_z)$, the second largest root of the equation $\displaystyle E^K_{\ell_z}(\rho, 0) = 0$.} and $\Chi_\eta\in C^\infty(\mathbb R, \mathbb R_+)$ is a cut-off function such that 
\begin{equation}
\label{cut:off:bis}
\Chi_\eta(s) =
\left\{
\begin{aligned}
& 1 \quad s\geq 0 , \\
& \leq1 \quad s\in [-\eta, 0] , \\
& 0 \quad s< -\eta. 
\end{aligned}
\right. 
\end{equation}

\item $\ve$ and $\ell_z$ are defined  by \eqref{E:energy} and \eqref{L:momentum}.
\end{itemize}

\subsection{Reparametrization of the fibre $\Gamma_{x}$ and the components of the energy-momentum tensor}
\label{compu:Tab}
In this section, we will compute the components of the energy momentum tensor $\T{\alpha}{\beta}$ provided $g$ has the form \eqref{metric:ansatz} and $f$ has the form \eqref{ansatz:for:f}. First of all, we recall the definition of $\Omega$:
\begin{equation*}
\Omega = \frac{\partial}{\partial t} + \omega\frac{\partial}{\partial \phi} \quad\text{where} \quad \omega = -\frac{W}{X}, 
\end{equation*}
the timelike vector field defined on $\BB$ which was fixed in Section \ref{2:DOF} for the time orientation. 
\\ Recall that by definition, if $c_\alpha$ is the coordinate  basis associated to the spacetime coordinates,  then any tangent vector can be written as 
\begin{equation*}
v = v^\alpha c_\alpha.
\end{equation*} 
The $v^\alpha$ are then called the conjugate coordinates to the spacetime coordinates. 
\\ Now let  $(v^t, v^\phi, v^\rho, v^z)$ be the conjugate coordinates to the spacetime coordinates $(t, \phi, \rho, z)$. Let $x = (t, \phi, \rho, z)\in\BB$ and denote by 
\begin{equation*}
c_0 = \frac{\partial}{\partial t} \;,\;  c_1 = \frac{\partial}{\partial \phi} \;,\; c_2 =  \frac{\partial}{\partial \rho} \;,\; c_3 = \frac{\partial}{\partial z}
\end{equation*}
the canonical basis determined by the local coordinates.   
We consider the orthonormal frame defined by
\begin{equation*}
e_0 := \sqrt{\frac{X}{\sigma^2}}\Omega\;,\; e_1 := \sqrt{\frac{1}{X}}\frac{\partial}{\partial\phi}\;,\; e_2 := e^{-\lambda}\frac{\partial}{\partial\rho}\;,\; e_3 := e^{-\lambda}\frac{\partial}{\partial z}. 
\end{equation*}
Let $p^\alpha$ be the corresponding coordinates. Therefore: 
\begin{equation*}
\displaystyle (v^\alpha) = G (p^\alpha) \quad\text{where}\quad\quad G = \begin{pmatrix} \frac{\sqrt{X}}{\sigma} & 0 & 0 & 0 \\ \frac{\omega\sqrt{X}}{\sqrt{\sigma}} & \frac{1}{\sqrt{X}} & 0 & 0 \\ 0 & 0 & e^{-\lambda} & 0 \\ 0 & 0 & 0 & e^{-\lambda}  \end{pmatrix}
\end{equation*} 
In the orthonormal frame, the mass shell condition is given by 
\begin{equation*}
-(p^0)^2 + (p^1)^2 + (p^2)^2 + (p^3)^2 = -1.
\end{equation*}
Now, we can express the $p^0$ in terms of the remaining $p^i$: 
\begin{equation*}
p^0 = \sqrt{1 + |p|^2}. 
\end{equation*}
We recall the energy-momentum tensor associated to the metric $g$ and the distribution function $f$, 
\begin{equation*}
\forall (t, \phi, \rho, z)\in\spacetime\;\quad\T{\alpha}{\beta}(t, \phi, \rho, z)= \int_{\Gamma_{x}} v_\alpha v_\beta f(t, \phi, \rho, z,v^\rho, v^\phi, v^z)\;d\text{vol}_x(v),
\end{equation*}
where $\Gamma_{x}$ is given by
\begin{equation*}
 \Gamma_{x}:= \left\{ v^\alpha\in \mathbb T_x\spacetime \;:\;  \ginv{\alpha}{\beta}v_\alpha v_\beta = - 1, \; v^0>0 \right\}, \end{equation*}
\noindent where we recall that the condition $v^0>0$ is equivalent to $v$ is future-pointed.  Thus, after a first change of variables in the fibre $\Gamma_{x}$, the energy momentum tensor is given by 
\begin{equation*}
\T{\alpha}{\beta}[f] = \int_{\mathbb R^3}\; v_{\alpha}(p)v_{\beta}(p)f(x, v(p))\frac{dp^1 dp^2 dp^3}{\sqrt{1+|p|^2}}. 
\end{equation*}
We will make a second change of variables in order to simplify the expressions of the energy momentum tensor components. To this end, we compute $\varepsilon$ and $\ell_z$ in terms of the $p^i$s: 
\begin{align*}
\varepsilon = - v_0 = \frac{\sigma}{\sqrt{X}}\left(1 + |p|^2\right)^{\frac{1}{2}} + \omega\ell_z
\end{align*}
and 
\begin{equation*}
 \ell_z = \sqrt{X}p^1. 
\end{equation*}
From the mass shell condition, we have 
\begin{equation}
\label{pos:J}
e^{2\lambda}\left((v^\rho)^2 + (v^z)^2\right) = J(\rho, z, \varepsilon, \ell_z),
\end{equation}
where $J$ is given by 
\begin{equation*}
J(\rho, z, \varepsilon, \ell_z) =  -1 +  \frac{X}{\sigma^2}\varepsilon^2 + \frac{2W}{\sigma^2}\varepsilon\ell_z - \frac{V}{\sigma^2}\ell_z^2.
\end{equation*}
We recall that $\sigma^2 = XV + W^2$. We introduce the polar variables $(\sqrt{J}, \vartheta)\in[0, \infty[\times [0, 2\pi]$ in the following way:
\begin{equation*}
\begin{aligned}
v^2 &= e^{-\lambda}\sqrt J \sin\vartheta, \\
v^3 &= e^{-\lambda}\sqrt J \cos\vartheta. 
\end{aligned}
\end{equation*}
Therefore, 
\begin{equation*}
\begin{aligned}
p^2 &= \sqrt J \sin\vartheta, \\
p^3 &= \sqrt J \cos\vartheta. 
\end{aligned}
\end{equation*}
Now we make the change of variables from $(p^1, p^2, p^3)$ to $(p^1, \sqrt J, \vartheta)$. We have \begin{equation*}
\frac{dp^1 dp^2 dp^3}{\sqrt{1 + |p|^2}} = \frac{\sqrt{J}}{\sqrt{1 + J + (p^1)^2}}d\sqrt{J} d\vartheta dp^\phi. 
\end{equation*}
Therefore,  
\begin{equation}
\label{3::4::7}
\T{\alpha}{\beta}[f] = \int_{\mathbb R_+}\int_0^{2\pi}\int_{\mathbb R}\; v_{\alpha}(p)v_{\beta}(p)f(x, v(p))\frac{\sqrt{J}}{\sqrt{1 + J + (p^1)^2}}d\sqrt{J} d\vartheta dp^1. 
\end{equation}
We make a last change of variables $H(\rho, z): (p^1, \sqrt J, \vartheta)\mapsto(E, L, \vartheta)$ defined by 
\begin{equation}
\label{def:E:L}
\begin{aligned}
E &:= \varepsilon -\omega\ell_z  =  \frac{\sigma}{\sqrt{X}}\left(1 +  J + (p^1)^2 \right)^{\frac{1}{2}}, \\
L &:= \frac{\sqrt X}{\rho}p^1.
\end{aligned}
\end{equation}
It is easy to see that $H(\rho, z)$ is a smooth diffeomorphism on its image and we have 
\begin{equation*}
\frac{\sqrt{J}}{\sqrt{1 + J + (p^1)^2}}dp^1d\sqrt{J} d\vartheta = \frac{\rho}{\sigma}dE dL d\vartheta. 
\end{equation*}	
\noindent It remains to compute the domain of the variables $E$ and $L$, which is the image of above change of variables.  From \eqref{def:E:L}, $\forall J\geq 0, \forall p^1\in\mathbb R$, we have 
\begin{equation*}
E\geq \frac{\sigma}{\sqrt X}.
\end{equation*}
By straightforward computations, we obtain
\begin{equation*}
J = \frac{X}{\sigma^2}E^2 - 1 - \frac{\rho^2}{X}L^2.
\end{equation*}
Since $J$ is positive, $L$ satisfies 
\begin{equation}
\label{tilde::L}
|L| \leq  \frac{\sqrt{X}}{\rho}\left( - 1 + \frac{X}{\sigma^2}E^2 \right)^{\frac{1}{2}} =: \tilde L(E, X, \sigma, \rho, z). 
\end{equation}
Now we define $D(\rho, z)$ to be the set: 
\begin{equation}
\label{D:rho:z}
D(\rho, z) := \left\{ (E , L)\;:\; E\geq \frac{\sigma}{\sqrt X} \quad\text{and }\quad |L| \leq \tilde L(E, X, \sigma, \rho, z) \right\} . 
\end{equation}
Therefore, 
\begin{equation*}
\text{Im}\, H(\rho, z) = D(\rho, z). 
\end{equation*}
\noindent By symmetry considerations, we have 
\begin{equation*}
\T{\rho}{\rho}[f](\rho, z) = \T{z}{z}[f](\rho, z).
\end{equation*} 
Moreover, the only non-vanishing components of $\mathbb T_{\alpha\beta}$ are $\T{t}{t}, \T{t}{\phi}, \T{\phi}{\phi}, \T{\rho}{\rho}$ and $\T{z}{z}$. Now we compute  
\begin{equation*}
\begin{aligned}
\T{t}{t}[f](\rho, z) &= \int_{\mathbb R_+}\int_0^{2\pi}\int_{\mathbb R}\; (v_{0}(p))^2f(x, v(p))\frac{\sqrt{J}}{\sqrt{1 + J + (p^1)^2}}d\sqrt{J} d\chi dp^1 \\
&= \int_{D(\rho, z)}\int_0^{2\pi}\; (E + \rho\omega L)^2\Phi(E + \rho\omega, L)\frac{\rho}{\sigma}\Psi_\eta(\rho, (E + \rho\omega L, \rho L), (X, W, \sigma))dE dL d\vartheta \\
&= \frac{2\pi\rho}{\sigma}\int_{D(\rho, z)}\left( E + \rho\omega L\right)^2\Phi(E + \rho\omega L, \rho L)\Psi_\eta(\rho, (E + \rho\omega L, \rho L), (X, W, \sigma))\,dE dL. 
\end{aligned}
\end{equation*}
Here, we used that $v_0(p) = -\ve = -(E + \rho\omega L )$, $\ell_z = \rho L$ and the independence of the integrand on $\vartheta$. The components $\T{t}{t}[f](\rho, z)$, $\T{t}{\phi}[f](\rho, z)$ and $\T{\phi}{\phi}[f](\rho, z)$ are computed in the same way. As for $\T{\rho}{\rho}[f](\rho, z)$, we have
\begin{equation*}
\begin{aligned}
\T{\rho}{\rho}[f](\rho, z) &= \int_{\mathbb R_+}\int_0^{2\pi}\int_{\mathbb R}\; (v_{\rho}(p))^2f(x, v(p))\frac{\sqrt{J}}{\sqrt{1 + J + (p^1)^2}}d\sqrt{J} d\chi dp^1 \\
&= e^{2\lambda}\int_{D(\rho, z)}\int_0^{2\pi}\; J(E, L)\sin^2\vartheta\frac{\rho}{\sigma}\Phi(E + \rho\omega L, \rho L)\Psi_\eta(\rho, (E + \rho\omega L, \rho L), (X, W, \sigma))dE dL d\vartheta \\
&= \frac{\pi\rho}{\sigma}e^{2\lambda}\int_{D(\rho, z)}\left( \frac{X}{\sigma^2}E^2 - 1 - \frac{\rho^2}{X}L^2\right)\Phi(E + \rho\omega L, \rho L)\Psi_\eta(\rho, (E + \rho\omega L, \rho L), (X, W, \sigma))\,dE dL \\
&= \frac{\pi\rho}{\sigma}e^{2\lambda}\int_{D(\rho, z)}\left(\tilde L (E, X, \sigma, \rho, z)^2 - L^2 \right)\Phi(E + \rho\omega L, \rho L)\Psi_\eta(\rho, (E + \rho\omega L, \rho L), (X, W, \sigma))\,dE dL.
\end{aligned}
\end{equation*}
The final expression is obtained by \eqref{tilde::L}.  Hence,  
\begin{equation}
\label{T:t:t}
\T{t}{t}[f](\rho, z) = \frac{2\pi\rho}{\sigma}\int_{D(\rho, z)}\left( E + \rho\omega L\right)^2\Phi(E + \rho\omega L, \rho L)\Psi_\eta(\rho, (E + \rho\omega L, \rho L), (X, W, \sigma))\,dE dL,  
\end{equation}
\begin{equation}
\label{T:t:phi}
\T{t}{\phi}[f](\rho, z) = -\frac{2\pi\rho}{\sigma}\int_{D(\rho, z)}\rho L\left( E + \rho\omega L\right)\Phi(E + \rho\omega L, \rho L)\Psi_\eta(\rho, (E + \rho\omega L, \rho L), (X, W, \sigma))\,dE dL, 
\end{equation}
\begin{equation}
\label{T:phi:phi}
\T{\phi}{\phi}[f](\rho, z) = \frac{2\pi\rho}{\sigma}\int_{D(\rho, z)}(\rho L)^2\Phi(E + \rho\omega L, \rho L)\Psi_\eta(\rho, (E + \rho\omega L, \rho L), (X, W, \sigma))\,dE dL,
\end{equation}
\begin{equation}
\label{T:rho:rho}
\T{\rho}{\rho}[f](\rho, z) = \frac{\pi\rho}{\sigma}e^{2\lambda}\int_{D(\rho, z)}\left(\tilde L (E, X, \sigma, \rho, z)^2 - L^2 \right)\Phi(E + \rho\omega L, \rho L)\Psi_\eta(\rho, (E + \rho\omega L, \rho L), (X, W, \sigma))\,dE dL,  
\end{equation}
\begin{equation}
\label{T:z:z}
\T{z}{z}[f](\rho, z) = \frac{\pi\rho}{\sigma}e^{2\lambda}\int_{D(\rho, z)}\left(\tilde L (E, X, \sigma, \rho, z)^2 - L^2 \right)\Phi(E + \rho\omega L, \rho L)\Psi_\eta(\rho, (E + \rho\omega L, \rho L), (X, W, \sigma))\,dE dL,
\end{equation}
where $\tilde L$ was defined in \eqref{tilde::L}.
\noindent It remains to compute the intersection of $D(\rho, z)$ and the support of $\Phi$ and $\Psi_\eta$. This will be done in  Section \ref{regularity} concerning the regularity of the matter terms.
\subsection{Static and axisymmetric black holes with matter}
In this section, we apply Theorem $1.1$ in \cite{chodosh2015stationary} concerning the modified Carter-Robinson theory. First of all, we recall the following theorem 
\begin{theoreme}[O.CHODOSH, Y.SHLAPENTOKH-ROTHMAN]
\label{Yakov:Otis}
Suppose that $(\spacetime, g)$ solves the Einstein equations for some energy-momentum tensor $\mathbb T$ satisfying 
\begin{equation}
\label{symmetry:T}
\mathbb T(T, \partial_\rho) = \mathbb T(T, \partial_z) = \mathbb T(\Phi, \partial_\rho) =   \mathbb T(\Phi, \partial_z) = 0.  
\end{equation}
Then, the metric data $(X, W, \theta, \sigma, \lambda)$ satisfies the following equations on $\BB = \left\{(\rho, z)\in\mathbb R^2\;:\; \rho >0 \right\}$:
\begin{enumerate}
\item X satisfies 
\begin{equation}
 \sigma^{-1}\partial_{\rho}(\sigma\partial_{\rho}X) + \sigma^{-1}\partial_{z}(\sigma\partial_{z}X) = 
 e^{2\lambda}\left( - 2\mathbb T(\Phi, \Phi) + {\text{Tr}(\mathbb T)}X\right) +  \frac{(\partial_{\rho}X)^2 + (\partial_{z}X)^2 - \theta_{\rho}^2 - \theta_z^2}{X} ,
\end{equation}
\item W satisfies 
\begin{equation}
 \partial_{\rho}(X^{-1}W)d\rho + \partial_{z}(X^{-1}W)dz = \frac{\sigma}{X^2}(\theta_{\rho}dz - \theta_{z}d\rho),
\end{equation}
\item $\theta$ satisfies 
\begin{equation}
d\theta = 2\sigma^{-1}e^{2\lambda}\left(\mathbb T(\Phi, \Phi)W - \mathbb T(\Phi, T)X\right)d\rho\wedge dz,
\end{equation}
as well as 
\begin{equation}
\label{theta::2}
 \sigma^{-1}\partial_{\rho}(\sigma\partial_{\rho}\theta) + \sigma^{-1}\partial_{z}(\sigma\partial_{z}\theta) = \frac{2\theta_{\rho}\partial_{\rho}X + 2\theta_{z}\partial_{z}X}{X}
\end{equation}
\item $\sigma$ satisfies 
\begin{equation}
X^{-1}\exp{-2\lambda}\sigma(\partial^2_{\rho}\sigma + \partial^2_{z}\sigma) = \mathbb T\left(T - X^{-1}W\Phi, T - X^{-1}W\Phi \right) - X^{-2}\sigma^2\mathbb T(\Phi, \Phi) + X^{-1}\sigma^2\text{Tr}(\mathbb T) , 
\end{equation}

\item $\lambda$ satisfies the following equations at the points where $|\partial\sigma| \neq 0 $
\begin{equation}
\label{eq:lambda:bis}
    \partial_{\rho}\lambda = \alpha_{\rho} - \frac{1}{2}\partial_{\rho}\log X, \quad \partial_{z}\lambda = \alpha_{z} - \frac{1}{2}\partial_{z}\log X, 
\end{equation}
where 
\begin{align*}
 ((\partial_{\rho}\sigma)^2 + (\partial_{z}\sigma)^2)\alpha_{\rho} &= \frac{1}{2} (\partial_{\rho}\sigma)\sigma\left(\mathbb T(\partial_\rho, \partial_\rho) - \mathbb T(\partial_z, \partial_z) + \frac{1}{2}\frac{(\partial_{\rho}X)^2-(\partial_{z}X)^2+(\theta_{\rho})^2-(\theta_{z})^2}{X^2}\right)  \\
 &+ (\partial_{\rho}\sigma)(\partial^2_{\rho}\sigma - \partial^2_{z}\sigma) +(\partial_{z}\sigma)(\partial^2_{\rho,z}\sigma)  \\
    &+(\partial_{z}\sigma)\sigma\left(\mathbb T(\partial_\rho, \partial_z) + \frac{1}{2}X^{-2}((\partial_{\rho}X)(\partial_{z}X) + (\theta_{\rho})(\theta_{z}))\right),
\end{align*}
\begin{align*}
((\partial_{\rho}\sigma)^2 + (\partial_{z}\sigma)^2)\alpha_{z} &= -\frac{1}{2} (\partial_{z}\sigma)\sigma\left(\mathbb T(\partial_\rho, \partial_\rho) - \mathbb T(\partial_z, \partial_z) + \frac{1}{2}\frac{(\partial_{\rho}X)^2-(\partial_{z}X)^2+(\theta_{\rho})^2-(\theta_{z})^2}{X^2}\right)   \\
 &+(\partial_{\rho}\sigma)((\partial^2_{\rho,z}\sigma)) - (\partial_{z}\sigma)(\partial^2_{\rho}\sigma - \partial^2_{z}\sigma) +(\partial_{\rho}\sigma)((\partial^2_{\rho,z}\sigma))  \\
    &+(\partial_{\rho}\sigma)\sigma\left(\mathbb T(\partial_\rho, \partial_z) + \frac{1}{2}X^{-2}((\partial_{\rho}X)(\partial_{z}X) + (\theta_{\rho})(\theta_{z}))\right),
\end{align*}
Independent of the behaviour of  $\sigma$, $\lambda$ satisfies
    \begin{align*}
        2\partial^2_{\rho}\lambda+2\partial^2_{z}\lambda &=  -\partial^2_{\rho}\log{X}-\partial^2_{z}\log{X} + \sigma^{-1}((\partial_{\rho}\sigma)^2 + (\partial_{z}\sigma)^2)  \\
        &+e^{2\lambda}\text{Tr}(\mathbb T) - \mathbb T(\partial_\rho, \partial_\rho) - \mathbb T(\partial_z, \partial_z) \\
        &- 2X^{-1}\left( \mathbb T(\Phi, \Phi) - \frac{1}{2}\text{Tr}(\mathbb T)X\right)e^{2\lambda} \\
        &-\frac{1}{2}X^{-2}((\partial_{\rho}X)^2+(\partial_{z}X)^2 +(\theta_{\rho})^2+(\theta_{z})^2). 
    \end{align*}
\end{enumerate}
 Conversely, if the metric data  solves each of these equations, and $|\partial\sigma|\neq 0$ on $\BB$, then we may recover the metric  $g$ defined on $\spacetime = \mathbb R\times(0, 2\pi)\times\BB$ such that $(\spacetime, g)$ solves the Einstein equations with energy-momentum tensor $\mathbb T$.
\end{theoreme}
\noindent In our case, all the assumptions of Theorem \ref{Yakov:Otis}  are satisfied. Therefore,  we apply the latter  with the components of the energy momentum tensor given by \eqref{T:t:t}, \eqref{T:t:phi}, \eqref{T:phi:phi}, \eqref{T:rho:rho} and \eqref{T:z:z} and obtain 
\begin{Propo}
\label{PDEs::1}
Suppose that $(\spacetime ,g, f)$ solves the Einstein-Vlasov equations where $g$ is given by 
\begin{equation*}
g = -Vdt^2 + 2Wdtd\phi + Xd\phi^2 + e^{2\lambda}\left( d\rho^2 + dz^2\right)
\end{equation*}
and $f$ is given by  
\begin{equation*}
f(x,v) = \Phi(\varepsilon, \ell_z)\Psi_\eta\left(\rho, (\varepsilon, \ell_z), h:= (X, W, \sigma)\right). 
\end{equation*}
Then the metric data $(X, W, \theta, \sigma, \lambda)$ satisfies the following equations on $\BB$
\begin{enumerate}
\item X satisfies 
\begin{equation}
\label{eq:X}
 \sigma^{-1}\partial_{\rho}(\sigma\partial_{\rho}X) + \sigma^{-1}\partial_{z}(\sigma\partial_{z}X) = 
 F_1(W, X, \sigma)(\rho, z) +  \frac{(\partial_{\rho}X)^2 + (\partial_{z}X)^2 - \theta_{\rho}^2 - \theta_z^2}{X} ,
\end{equation}
\item W satisfies 
\begin{equation}
\label{eq:W}
\partial_{\rho}(X^{-1}W)d\rho + \partial_{z}(X^{-1}W)dz = \frac{\sigma}{X^2}(\theta_{\rho}dz - \theta_{z}d\rho),
\end{equation}
\item $\theta$ satisfies 
\begin{equation}
\label{eq::::theta}
d\theta = 2\sigma^{-1}e^{2\lambda}F_2(W, X, \sigma)(\rho, z)d\rho\wedge dz,
\end{equation}
as well as 
\begin{equation}
 \sigma^{-1}\partial_{\rho}(\sigma\partial_{\rho}\theta) + \sigma^{-1}\partial_{z}(\sigma\partial_{z}\theta) = \frac{2\theta_{\rho}\partial_{\rho}X + 2\theta_{z}\partial_{z}X}{X}.
\end{equation}
\item $\sigma$ satisfies 
\begin{equation}
\label{eq:sigma}
X^{-1}\exp{(-2\lambda)}\sigma(\partial^2_{\rho}\sigma + \partial^2_{z}\sigma) = F_3(W, X, \sigma)(\rho, z) , 
\end{equation}

\item $\lambda$ satisfies the following equations at the points where $|\partial\sigma| \neq 0 $
\begin{equation}
    \partial_{\rho}\lambda = \alpha_{\rho} - \frac{1}{2}\partial_{\rho}\log X, \quad \partial_{z}\lambda = \alpha_{z} - \frac{1}{2}\partial_{z}\log X, 
\end{equation}
where 
\begin{align*}
 ((\partial_{\rho}\sigma)^2 + (\partial_{z}\sigma)^2)\alpha_{\rho} &= \frac{1}{4} (\partial_{\rho}\sigma)\sigma\frac{(\partial_{\rho}X)^2-(\partial_{z}X)^2+(\theta_{\rho})^2-(\theta_{z})^2}{X^2} + (\partial_{\rho}\sigma)(\partial^2_{\rho}\sigma - \partial^2_{z}\sigma)  \\
    &+ (\partial_{z}\sigma)(\partial^2_{\rho,z}\sigma) + \frac{1}{2}X^{-2}((\partial_{\rho}X)(\partial_{z}X) + (\theta_{\rho})(\theta_{z}))),
\end{align*}
\begin{align*}
((\partial_{\rho}\sigma)^2 + (\partial_{z}\sigma)^2)\alpha_{z} &= -\frac{1}{4} (\partial_{z}\sigma)\sigma\frac{(\partial_{\rho}X)^2-(\partial_{z}X)^2+(\theta_{\rho})^2-(\theta_{z})^2}{X^2}- (\partial_{z}\sigma)(\partial^2_{\rho}\sigma - \partial^2_{z}\sigma)  \\
 &+ (\partial_{\rho}\sigma)(\partial^2_{\rho,z}\sigma) + \frac{1}{2}X^{-2}((\partial_{\rho}X)(\partial_{z}X) + (\theta_{\rho})(\theta_{z}))),
\end{align*}
Independent of the behaviour of  $\sigma$, $\lambda$ satisfies
    \begin{align*}
        2\partial^2_{\rho}\lambda+2\partial^2_{z}\lambda &=  -\partial^2_{\rho}\log{X}-\partial^2_{z}\log{X} + \sigma^{-1}((\partial_{\rho}\sigma)^2 + (\partial_{z}\sigma)^2) +F_4(W, X, \sigma)(\rho, z) \\
        &-\frac{1}{2}X^{-2}((\partial_{\rho}X)^2+(\partial_{z}X)^2 +(\theta_{\rho})^2+(\theta_{z})^2),
    \end{align*}
\end{enumerate}
where $F_1, F_2, F_3, F_4$ are given by: 

\begin{equation}
\label{F::1}
F_1(W, X, \sigma)(\rho, z):= -e^{2\lambda}\frac{2\pi\rho}{\sigma}\int_{D(\rho, z)}(X + 2(\rho L)^2)\Phi(E + \rho\omega L, \rho L)\Psi_\eta(\rho, (E + \rho\omega L, \rho L), (X, W, \sigma))\,dE dL,
\end{equation}
\begin{equation}
\label{F::2}
F_2(W, X, \sigma)(\rho, z):= \frac{2\pi\rho}{\sigma}X\int_{D(\rho, z)}\rho L E\Phi(E + \rho\omega L, \rho L)\Psi_\eta(\rho, (E + \rho\omega L, \rho L), (X, W, \sigma))\,dE dL,
\end{equation}
\begin{equation}
\label{F::3}
\begin{aligned}
F_3(W, X, \sigma)(\rho, z)&:= \frac{2\pi\rho^3\sigma}{X^2}\int_{D(\rho, z)}\left(\tilde L^2 - L^2\right)\Phi(E + \rho\omega L, \rho L)\Psi_\eta(\rho, (E + \rho\omega L, \rho L), (X, W, \sigma))\,dE dL, 
\end{aligned}
\end{equation}
\begin{equation}
\label{F::4}
\begin{aligned}
F_4(W, X, \sigma, \lambda)(\rho, z):= -\frac{4\pi e^{2\lambda}\rho}{\sigma}&\int_{D(\rho, z)}\left(\frac{X^2}{\rho^2\sigma^2}E^2 + \left(1 - \frac{X}{\rho^2}\right)\left( 1 + \frac{\rho^2}{X}L^2\right)\right)\Phi(E + \rho\omega L, \rho L) \\
&\Psi_\eta(\rho, (E + \rho\omega L, \rho L), (X, W, \sigma))\,dE dL,  \\
\end{aligned}
\end{equation}
Conversely, if the metric data  solves each of these equations, and $|\partial\sigma|\neq 0$ on $\BB$, then we may recover the metric and the distribution function  $(\spacetime,g, f)$, solving the Einstein-Vlasov equations. 
\end{Propo}
\begin{proof}
We apply Theorem \ref{Yakov:Otis} and we compute the components of the energy-momentum tensor. 
\\ First of all, we have 
\begin{equation*}
\begin{aligned}
\T{\rho}{\rho} &=  \T{z}{z} ,\\
\T{\rho}{z} &= 0. 
\end{aligned}
\end{equation*}
Therefore, $\alpha_\rho$ and $\alpha_z$ do not depend on the the matter terms. Now, we compute 

\begin{equation*}
\begin{aligned}
\text{Tr}(\mathbb T) &= \ginv{\alpha}{\beta}\T{\alpha}{\beta}  \\
&= \ginv{t}{t}\T{t}{t} + 2\ginv{t}{\phi}\T{t}{\phi} + \ginv{\phi}{\phi}\T{\phi}{\phi} + 2e^{-2\lambda}\T{\rho}{\rho} \\
&= \frac{1}{\sigma^2}\left(-X\T{t}{t} + 2W\T{t}{\phi} + V\T{\phi}{\phi}\right) + 2e^{-2\lambda}\T{\rho}{\rho} \\
&= \frac{2\pi\rho}{\sigma^3}\int_{D(\rho, z)}\left(-X(E + \rho\omega L)^2 - 2W\rho L(E + \rho\omega L) + V(\rho L)^2 + \sigma^2\left( \tilde L(E, X, \sigma, \rho)^2 - L^2\right) \right) \\
&\Phi(E + \rho\omega L, \rho L)\Psi_\eta(\rho, (E + \rho\omega L, \rho L), (X, W, \sigma))\,dE dL,
\end{aligned}
\end{equation*}
Now we recall that $\sigma^2 = XV + W^2$ and $\displaystyle \tilde L(E, X, \sigma, \rho) = \frac{\sqrt X}{\rho}\left( -1 + \frac{X}{\sigma^2}E^2\right)^{\frac{1}{2}}$  . We use the latter to obtain 
\begin{equation*}
\begin{aligned}
&-X(E + \rho\omega L)^2 - 2W\rho L(E + \rho\omega L) + V(\rho L)^2 + \sigma^2\left( \tilde L(E, X, \sigma, \rho)^2 - L^2\right) = -\sigma^2. 
\end{aligned}
\end{equation*}
Therefore, 
\begin{equation*}
\text{Tr}(\mathbb T) = -\frac{2\pi\rho}{\sigma}\int_{D(\rho, z)}\, \Phi(E + \rho\omega L, \rho L)\Psi_\eta(\rho, (E + \rho\omega L, \rho L), (X, W, \sigma))\,dE dL. 
\end{equation*}
Now, we compute
\begin{equation*}
\begin{aligned}
F_1(W, X, \sigma)(\rho, z) &= e^{2\lambda}\left( - 2\T{\phi}{\phi} + {\text{Tr}(\mathbb T)}X\right)  \\
&= -e^{2\lambda}\frac{2\pi\rho}{\sigma}\int_{D(\rho, z)}\,(X + 2(\rho L)^2) \Phi(E + \rho\omega L, \rho L)\Psi_\eta(\rho, (E + \rho\omega L, \rho L), (X, W, \sigma))\,dE dL.
\end{aligned}
\end{equation*}
\begin{equation*}
\begin{aligned}
&F_2(W, X, \sigma)(\rho, z) = \T{\phi}{\phi}W - \T{\phi}{t}X \\
&=  \frac{2\pi\rho}{\sigma}\int_{D(\rho, z)}\,\left((\rho L)^2W + X\rho L(E + \rho\omega L)\right)\Phi(E + \rho\omega L, \rho L)\Psi_\eta(\rho, (E + \rho\omega L, \rho L), (X, W, \sigma))\,dE dL \\
&= \frac{2\pi\rho}{\sigma}X\int_{D(\rho, z)}\,\rho LE\Phi(E + \rho\omega L, \rho L)\Psi_\eta(\rho, (E + \rho\omega L, \rho L), (X, W, \sigma))\,dE dL.  \\
\end{aligned}
\end{equation*}
We have used $\omega:= -WX^{-1}$ to obtain the latter expression. Now we compute 
\begin{equation*}
\begin{aligned}
& \mathbb T\left(T - X^{-1}W\Phi, T - X^{-1}W\Phi \right) - X^{-2}\sigma^2\mathbb T(\Phi, \Phi) + X^{-1}\sigma^2\text{Tr}(\mathbb T) \\
&= \T{t}{t} + \omega^2 \T{\phi}{\phi} + 2\omega\T{t}{\phi} \\
&= \frac{2\pi\rho}{\sigma}\int_{D(\rho, z)}\,\left((E + \rho\omega L)^2 + (\rho\omega L)^2 - 2\rho\omega L(E + \rho\omega L)\right)\Phi(E + \rho\omega L, \rho L) \\
&\Psi_\eta(\rho, (E + \rho\omega L, \rho L), (X, W, \sigma))\,dE dL.  \\
&= \frac{2\pi\rho}{\sigma}\int_{D(\rho, z)}\, E^2\Phi(E + \rho\omega L, \rho L)\Psi_\eta(\rho, (E + \rho\omega L, \rho L), (X, W, \sigma))\,dE dL.   \\
\end{aligned}
\end{equation*}
Hence,  
\begin{equation*}
\begin{aligned}
&F_3(W, X, \sigma)(\rho, z) = \mathbb T\left(T - X^{-1}W\Phi, T - X^{-1}W\Phi \right) - X^{-2}\sigma^2\mathbb T(\Phi, \Phi) + X^{-1}\sigma^2\text{Tr}(\mathbb T) \\
&= \frac{2\pi\rho}{\sigma}\int_{D(\rho, z)}\, \left(E^2 - X^{-2}\sigma^2(\rho L)^2 - X^{-1}\sigma^2\right)\Phi(E + \rho\omega L, \rho L)\Psi_\eta(\rho, (E + \rho\omega L, \rho L), (X, W, \sigma))\,dE dL   \\
&= \frac{2\pi\rho}{\sigma}\int_{D(\rho, z)}\, X^{-2}\sigma^2\rho^2\left(\frac{X^2}{\rho^2\sigma^2}E^2 - L^2 - \frac{X}{\rho^2}\right)\Phi(E + \rho\omega L, \rho L)\Psi_\eta(\rho, (E + \rho\omega L, \rho L), (X, W, \sigma))\,dE dL   \\
&= \frac{2\pi\rho^3\sigma}{X^2}\int_{D(\rho, z)}\, \left(\frac{X}{\rho^2}\left(\frac{X}{\sigma^2}E^2 - 1\right) - L^2\right)\Phi(E + \rho\omega L, \rho L)\Psi_\eta(\rho, (E + \rho\omega L, \rho L), (X, W, \sigma))\,dE dL   \\
&= \frac{2\pi\rho^3\sigma}{X^2}\int_{D(\rho, z)}\, \left(\tilde L^2(E, X, \sigma, \rho, z) - L^2\right)\Phi(E + \rho\omega L, \rho L)\Psi_\eta(\rho, (E + \rho\omega L, \rho L), (X, W, \sigma))\,dE dL . 
\end{aligned}
\end{equation*}
Finally, we compute
\begin{equation*}
\begin{aligned}
F_4(W, X, \sigma, \lambda)(\rho, z) &= e^{2\lambda}\text{Tr}(\mathbb T) - \mathbb T(\partial_\rho, \partial_\rho) - \mathbb T(\partial_z, \partial_z) - 2X^{-1}\left( \mathbb T(\Phi, \Phi) - \frac{1}{2}\text{Tr}(\mathbb T)X\right)e^{2\lambda} \\
&= 2e^{2\lambda}\left( \text{Tr}(\mathbb T) - e^{-2\lambda}\T{\rho}{\rho} - X^{-1} \T{\phi}{\phi}  \right) \\
&= -2e^{2\lambda}\frac{2\pi\rho}{\sigma}\int_{D(\rho, z)}\,\left( 1 + \frac{1}{2}\left(\tilde L^2(E, X, \sigma, \rho, z) - L^2\right) + X^{-1}(\rho L)^2\right)\Phi(E + \rho\omega L, \rho L) \\
&\Psi_\eta(\rho, (E + \rho\omega L, \rho L), (X, W, \sigma))\,dE dL \\ 
&= -\frac{4\pi e^{2\lambda}\rho}{\sigma}\int_{D(\rho, z)}\left(\frac{X^2}{\rho^2\sigma^2}E^2 + \left(1 - \frac{X}{\rho^2}\right)\left( 1 + \frac{\rho^2}{X}L^2\right)\right)\Phi(E + \rho\omega L, \rho L) \\
&\Psi_\eta(\rho, (E + \rho\omega L, \rho L), (X, W, \sigma))\,dE dL. 
\end{aligned}
\end{equation*}
\end{proof}

\subsection{Renormalised unknowns and their equations}
\label{Renormalised:::unknowns}
This section follows closely \cite{chodosh2017time}. Nonetheless, we detail the computations and the arguments in order to be self contained. 
\\In order to apply the fixed point theorem in a neighbourhood of a fixed Kerr solution, we introduce new quantities which are normalised with respect to the Kerr metric. This choice of variables was first considered in \cite{chodosh2017time} and allows to subtract off the leading order singular behaviour near the axis and the horizon. 
\\ We begin by defining an adapted Ernst potential, $Y: \BB\to \mathbb R$. Let us review the construction in the vacuum case just for comparison. 
\\ In vacuum, the twist one-form associated to $\displaystyle \Phi = \frac{\partial}{\partial \phi}$ vanishes: 
\begin{equation*}
d\theta = 0.
\end{equation*} 
Therefore, since $\BB$ is simply connected, there exists a function $Y_{vacuum}: \BB\to \mathbb R$, which satisfies
\begin{equation}
\label{Ernst:vacuum}
\theta = dY_{vacuum}.
\end{equation}
This leads to a harmonic map system in $(X, Y_{vacuum})$ which decouples from the remaining metric data. Moreover, the requirements of asymptotic flatness and regular extensions to the axis and to the horizon lead to boundary conditions for $X$ and $Y_{vacuum}$. This allows one to determine uniquely $(X, Y_{vacuum})$ (See \cite{weinstein1992stationary}, \cite{carter2009republication}).  Now, given $(X, Y_{vacuum})$ and the boundary conditions, the rest of the metric components are uniquely determined by quadratures, see Section \ref{main::difficulties} of the introduction and \cite[Chapter 10]{heusler1996black} for details. 
\\ In the presence of matter, the twist one-form is no longer closed. We recall from \eqref{eq::::theta} that $\theta$ satisfies
\begin{equation}
d\theta = 2\sigma^{-1}e^{2\lambda}F_2(W, X, \sigma)(\rho, z) d\rho\wedge dz. 
\end{equation}
However, we can still define an "Ernst potential" $Y$, which will define, together with $X$ a harmonic map system. The idea is to split the twist one-form $\theta$ into an Ernst potential piece $dY$ and an other one-form $B$, defined on $\Bbarre$ which verifies
\begin{equation}
\label{def::B}
dB = 2\sigma^{-1}e^{2\lambda}F_2(W, X, \sigma)(\rho, z) d\rho\wedge dz \quad\text{on}\quad \BB. 
\end{equation}
\begin{enumerate}
\item We begin by extending the two-form $\omega\in\Omega^2(\BB)$ defined by 
\begin{equation*}
\omega := \omega_{\rho z} d\rho\wedge dz :=  2\sigma^{-1}e^{2\lambda}F_2(W, X, \sigma)(\rho, z) d\rho\wedge dz 
\end{equation*}
to a two-form $\overline\omega$ defined on $\Bbarre$ so that 
\begin{equation*}
 \overline \omega = \omega \quad\quad\text{on}\quad\BB.
\end{equation*}
First, assume that $\omega_{\rho z}$ can be extended to a function $\overline\omega_{\rho z}$ \footnote{We identify $\overline\omega_{\rho z}$ with its coordinate representations in every chart $\Bnbarre$, $\Bsbarre$, $\BAbarre\cup\BHbarre$.} defined on $\Bbarre$. Now,  let $(\xi_N, \xi_S, 1 - \xi_N -\xi_S)$ be the partition of unity subordinate to $(\Bnbarre, \Bsbarre, \BAbarre\cup\BHbarre)$ and let $p\in\Bbarre$
\begin{itemize}
\item If $p\in(\BAbarre\cup\BHbarre)$, then 
\begin{equation*}
\omega_{\rho z}(p) = \overline\omega_{\rho z}^{A}(\rho, z) 
\end{equation*}
\item If $p\in \Bnbarre$, then
\begin{equation*}
\omega_{\rho z}(p) = \overline\omega_{\rho z}^N(s, \chi) 
\end{equation*}
\item If $p\in \Bsbarre$, then
\begin{equation*}
\omega_{\rho z}(p) = \overline\omega_{\rho z}^S(s', \chi') 
\end{equation*}
\end{itemize}
\noindent where $\overline\omega_{\rho z}^{A}$, $\overline\omega_{\rho z}^{N}$ and $\overline\omega_{\rho z}^{S}$ are defined on $\BAbarre\cup\BHbarre$, $\Bnbarre$ and $\Bsbarre$ respectively. 
\noindent We have 
\begin{equation*}
\begin{aligned}
\omega(p) &= \omega_{\rho z}(p) d\rho\wedge dz  =  \overline\omega_{\rho z}(s, \chi) (s^2 + \chi^2) ds\wedge d\chi \quad \forall p\in \BB_N.
\end{aligned}
\end{equation*}
Since the coordinate system $(s, \chi)$ is defined on $\Bnbarre$, we set:
\begin{equation*}
\overline\omega (p) :=\overline\omega^N_{\rho z}(s, \chi) (s^2 + \chi^2) ds\wedge d\chi \quad \forall p\in\Bnbarre. 
\end{equation*}
Similarly, we set 
\begin{equation*}
\overline\omega (p) := \overline\omega^S_{\rho z}(s', \chi') ((s')^2 + (\chi')^2) ds'\wedge d\chi',  \quad \forall p\in \Bsbarre. 
\end{equation*}
Finally, we set 
\begin{equation*}
\begin{aligned}
\overline\omega(p) &:= \overline\omega_{\rho z}(p) d\rho\wedge dz  ,  \quad \forall p\in \BAbarre\cup\BHbarre.
\end{aligned}
\end{equation*}
This defines $\overline \omega$ on $\Bbarre$. 
\item In order to construct $B$, or rather to find the equations for $B$,  we first make the following ansatz: 
\begin{equation}
\label{ansatz::for::B}
B := B^{(A)} + \xi_NB^{(N)} + \xi_SB^{(S)}
\end{equation}
where $B^{(N)}$, $B^{(S)}$ and $B^{(A)}$ are one-forms that are defined on $\Bnbarre$, $\Bsbarre$ and $\BAbarre\cup\BHbarre$ respectively. Therefore, solving the equations for $B^{(N)}$, $B^{(S)}$ and $B^{(A)}$ will allow us to determine $B$. Hence,  we determine the equations for $B^{(N)}$, $B^{(S)}$ and $B^{(A)}$ based on \eqref{def::B}: 
\begin{itemize}
\item On $\Bnbarre$, we construct $B^{(N)}$ such that 
\begin{equation*}
dB^{(N)} = \xi_N\overline\omega. 
\end{equation*}
In the local coordinates $(s, \chi)$, we have 
\begin{equation*}
B^{(N)} = B^{(N)}_s ds + B^{(S)}_\chi d\chi.
\end{equation*}
This implies
\begin{equation*}
dB^{(N)} = \partial_\chi B^{(N)}_s d\chi\wedge ds +  \partial_s B_\chi^{(N)} ds\wedge d\chi. 
\end{equation*}
Therefore, 
\begin{equation*}
\xi_N\overline\omega_{\rho z}(s, \chi) (s^2 + \chi^2) ds\wedge d\chi = \partial_\chi B^{(N)}_s d\chi\wedge ds +  \partial_s B_\chi^{(N)} ds\wedge d\chi.
\end{equation*}
We make the following gauge choice: 
\begin{equation*}
 B^{(N)}_s(s, \chi) = 0\;,\; B_{\chi}^{(N)}(0, \chi) = B_{\chi}^{(N)}(s, 0) = 0.  
\end{equation*}
Moreover, we require $B^{(N)}_s$ to satisfy the equation. 
\begin{equation*}
\partial_s B^{(N)}_\chi = \xi_N\overline\omega_{\rho z}(s, \chi) (s^2 + \chi^2).
\end{equation*}
Consequently, we define the one-form $B^{(N)}$ on $\Bnbarre$ to be the solution the equations
\begin{equation}
\label{eq:for:Bn}
\begin{aligned}
B_{\chi}^{(N)}(0, \chi) &= 0, \\
\partial_s B^{(N)}_\chi &= 2\xi_N\sigma^{-1}e^{2\lambda}F_2(W, X, \sigma)(\rho, z)(s^2 + \chi^2), \\
 B^{(N)}_s(s, \chi) &= 0.
\end{aligned}
\end{equation}
\item Similarly, we define  the one-form $B^{(S)}$ on $\Bsbarre$ to be the solution the equations
\begin{equation}
\label{eq:for:Bs}
\begin{aligned}
B_{\chi'}^{(S)}(0, \chi') &= 0, \\
\partial_{s'} B^{(S)}_{\chi'} &= 2\xi_S\sigma^{-1}e^{2\lambda}F_2(W, X, \sigma)(s', \chi') ((s')^2 + (\chi')^2), \\
 B^{(S)}_{s'}(s', \chi') &= 0.
\end{aligned}
\end{equation}
\item Finally, we use \eqref{ansatz::for::B} in order to find the equations satisfied by $B^{(A)}$. We have 
\begin{equation*}
\begin{aligned}
dB^{(A)} &= dB - d(\xi_NB^{(N)}) - d(\xi_SB^{(S)})  \\
&= dB - \xi_Nd(B^{(N)}) - d\xi_N\wedge B^{(N)} - \xi_Sd(B^{(S)}) - d\xi_S\wedge B^{(S)} \\
&= (1 - \xi_N^2 - \xi_S^2)dB  - d\xi_N\wedge B^{(N)} - d\xi_S\wedge B^{(S)} .
\end{aligned}
\end{equation*}
On $\BAbarre\cup\BHbarre$, we have 
\begin{equation*}
dB =  2\sigma^{-1}e^{2\lambda}F_2(W, X, \sigma)(\rho, z) d\rho\wedge dz, 
\end{equation*}
\begin{equation*}
B^{(N)} = B^{(N)}_\rho d\rho + B^{(S)}_z dz \quad \text{,} \quad  B^{(N)} = B^{(S)}_\rho d\rho + B^{(S)}_z dz \quad {,} \quad B^{(A)} = B^{(A)}_\rho d\rho + B^{(A)}_z dz
\end{equation*}
We compute
\begin{equation*}
dB^{(A)} =  \partial_zB_\rho^{(A)} dz\wedge d\rho + \partial_\rho B_z^{(A)} d\rho\wedge dz,
\end{equation*}
\begin{equation*}
d\xi_N\wedge B^{(N)} = (\partial_\rho\xi_N)B_z^{(N)} - (\partial_z\xi_N)B_\rho^{(N)}. 
\end{equation*}
We make the following gauge choice 
\begin{equation*}
B_{\rho}^{(A)} = 0\quad\text{,}\quad B_z^{(A)}(0, z) = 0
\end{equation*}
and we require $B_z^{(A)}$ to satisfy the equation 
\begin{equation*}
\begin{aligned}
\partial_\rho B_z^{(A)}(\rho, z) &= 2 (1 - \xi_N^2 - \xi_S^2) \sigma^{-1}e^{2\lambda}F_2(W, X, \sigma)(\rho, z)  \\
&-  (\partial_\rho\xi_N)B_z^{(N)} + (\partial_z\xi_N)B_\rho^{(N)} -(\partial_\rho\xi_S)B_z^{(S)} + (\partial_z\xi_S)B_\rho^{(S)}.
\end{aligned}
\end{equation*}
Therefore, we define the one-form $B^{(A)}$ to be the solution of the equation 
\begin{equation}
\label{eq:for:BA}
\begin{aligned}
B_z^{(A)}(0, z) &= 0, \\
\partial_\rho B_z^{(A)}(\rho, z) &= 2 (1 - \xi_N^2 - \xi_S^2) \sigma^{-1}e^{2\lambda}F_2(W, X, \sigma)(\rho, z)  \\
&-  (\partial_\rho\xi_N)B_z^{(N)} + (\partial_z\xi_N)B_\rho^{(N)} -(\partial_\rho\xi_S)B_z^{(S)} + (\partial_z\xi_S)B_\rho^{(S)}, \\
B_{\rho}^{(A)}(\rho, z) &= 0. 
\end{aligned}
\end{equation}
\end{itemize}
\item With this construction, $B$ is defined by $\displaystyle \left(B_\chi^{(N)}, B_{\chi'}^{(S)}, B_z^{(A)}\right)$ and verifies \eqref{def::B}
\end{enumerate}
\noindent Therefore, we can define the Ernst potential $Y$: 
\begin{definition}
We define, up to a constant,  the Ernst potential $Y : \BB\to\mathbb R$ to be the function which satisfies: 
\begin{equation}
\label{ernst:Yzero}
dY = \theta - B. 
\end{equation}
\end{definition}
Now, we introduce the renormalised unknowns to be the following set of functions and a one-form, which are all assumed to be continuous on $\Bbarre$ and satisfy the following definitions on $\BB$: 
\begin{align}
    \sigmazero &:= \frac{\sigma-\sigma_K}{\sigma_K}\quad \sigma_K = \rho,  \\
    \Xzero &:= X_K^{-1}(X-X_K), \quad\Yzero := X_K^{-1}(Y-Y_K), \\
         \thetazero &:= X^{-1}W - X_K^{-1}W_K, \\
         \lambdazero &:= \lambda - \lambda_K.
\end{align}
Henceforth, the quantities $\displaystyle \left(\sigmazero, \left(B_\chi^{(N)}, B_{\chi'}^{(S)}, B_z^{(A)}\right), \left(\Xzero, \Yzero \right), \thetazero, \lambdazero \right) $ will be called the "renormalised unknowns".
\\ Now, given $\displaystyle \left(\sigmazero, \displaystyle \left(B_\chi^{(N)}, B_{\chi'}^{(S)}, B_z^{(A)}\right), \left(\Xzero, \Yzero \right), \thetazero, \lambdazero\right)$, we can recover the original unknowns in the following way: 
\begin{equation}
    \sigma = \sigma_K(1 + \sigmazero), 
\end{equation}
\begin{equation}
X= X_K(1 + \Xzero), \quad Y = X_K(1 + \Yzero),
\end{equation}
\begin{equation}
         X^{-1}W  =  \thetazero + X_K^{-1}W_K,
\end{equation}
\begin{equation}
\lambda = \lambdazero + \lambda_K.
\end{equation}
In terms of the renormalised unknowns, the matter terms become:
\begin{equation*}
\begin{aligned}
F_1(\thetazero, \Xzero, \sigmazero)(\rho, z):= -\frac{2\pi e^{2\left(\lambdazero+ \lambda_K\right)}}{1 + \sigmazero}&\int_{D(\rho, z)}(X_K(1 + \Xzero) + 2(\rho L)^2)\Phi(E + \rho\left(-\thetazero + \omega_K\right) L, \rho L) \\
&\Psi_\eta(\rho, (E + \rho(-\thetazero + \omega_K) L, \rho L), (\thetazero, \Xzero, \sigmazero))\,dE dL,
\end{aligned}
\end{equation*}
\begin{equation*}
\begin{aligned}
F_2(\thetazero, \Xzero, \sigmazero)(\rho, z):= \frac{2\pi}{1 + \sigmazero}X_K(1 + \Xzero)&\int_{D(\rho, z)}\rho L E\Phi(E + \rho\left(-\thetazero + \omega_K\right) L, \rho L) \\
&\Psi_\eta(\rho, (E + \rho\left(-\thetazero + \omega_K\right) L, \rho L), (\thetazero, \Xzero, \sigmazero))\,dE dL,
\end{aligned}
\end{equation*}
\begin{equation*}
\begin{aligned}
F_3(\thetazero, \Xzero, \sigmazero)(\rho, z):= \frac{\rho^4}{X^2_K}\frac{2\pi(1 + \sigmazero)}{\left(1 + \Xzero\right)^2}&\int_{D(\rho, z)}\left(\tilde L^2 - L^2\right)\Phi(E + \rho\left(-\thetazero + \omega_K\right) L, \rho L) \\ 
&\Psi_\eta(\rho, (E + \rho\left(-\thetazero + \omega_K\right) L, \rho L), (\thetazero, \Xzero, \sigmazero))\,dE dL, 
\end{aligned}
\end{equation*}
\begin{equation*}
\begin{aligned}
F_4(\thetazero, \Xzero, \sigmazero, \lambdazero)(\rho, z)&:= -\frac{4\pi e^{2\left(\lambdazero+ \lambda_K\right)}}{1 + \sigmazero} \\
&\int_{D(\rho, z)}\left(\frac{X_K^2\left(1 + \Xzero\right)^2}{\rho^4\left(1 + \sigmazero\right)^2}E^2 + \left(1 - \frac{X_K}{\rho^2}\left(1 + \Xzero \right)\right)\left( 1 + \frac{\rho^2}{X_K}\frac{1}{1 + \Xzero}L^2\right)\right) \\ 
&\Phi(E + \rho\left(-\thetazero + \omega_K\right) L, \rho L)\Psi_\eta(\rho, (E + \rho\left(-\thetazero + \omega_K\right), \rho L), (\thetazero, \Xzero, \sigmazero))\,dE dL,  \\
\end{aligned}
\end{equation*}
\noindent We apply Proposition  \ref{PDEs::1}  in order to obtain the equations for the renormalised unknows.
\begin{Propo}
\label{reduced:reduced::ev}
The renormalised unknowns $\displaystyle \left( \sigmazero, B, \left(\Xzero, \Yzero \right), \thetazero, \lambdazero\right)$verify the following equations
\begin{itemize}
\item $\sigmazero$ satisfies
    \begin{equation}
\label{sigmazero}
            \Delta_{\mathbb{R}^4}{\sigmazero} = \rho^{-1}\sigma^{-1}{X}e^{2\lambda}F_3(\Xzero, \thetazero, \sigmazero)(\rho, z),
        \end{equation}
        where $\Delta_{\mathbb R^4}$ is the Laplacian corresponding to the flat metric on $\mathbb R^4$ given by $g_{\mathbb R^4} = d\rho^2 + dz^2 + \rho^2d\mathbb{S}^2$. 
\item $B$ satisfies

  \begin{equation}
\label{Bzero}
\begin{aligned}
&\partial_\chi B^{(N)}_s = 2\xi_N(s\chi)^{-1}(1 + \sigmazero)^{-1}e^{2\lambda + 2\lambdazero}F_2(\Xzero, \thetazero, \sigmazero)(s, \chi)(s^2 + \chi^2), \\
 & \partial_{\chi'} B^{(S)}_{s'} = 2\xi_S(s'(\chi)')^{-1}(1 + \sigmazero)^{-1}e^{2\lambda + 2\lambdazero}F_2(\Xzero, \thetazero, \sigmazero)(s', \chi') ((s')^2 + (\chi')^2), \\
   & \partial_{\rho}B^{(A)}_{z} (\rho,z) + (\partial_\rho\xi_N)B_z^{(N)} - (\partial_z\xi_N)B_\rho^{(N)} + (\partial_\rho\xi_S)B_z^{(S)} - (\partial_z\xi_S)B_\rho^{(S)} =  \\
   &2(1 - \xi_N^2 - \xi_S^2)\rho^{-1}(1 + \sigmazero)^{-1}e^{2\lambda + 2\lambdazero}F_2(\Xzero, \thetazero, \sigmazero)(\rho, z),
  \end{aligned}
     \end{equation}
\item $(\Xzero,\Yzero)$ satisfies 

        \begin{equation}
\label{XYzero}
        \begin{aligned}
        \Delta_{\mathbb{R}^3}\Xzero + \frac{2\partial Y_K\cdot\partial\overset{\circ}{Y}}{X_K} - \frac{2|\partial Y_K|^2}{X_K^2}\overset{\circ}{X} + 2\frac{\partial X_K\cdot\partial Y_K}{X_K^2}\overset{\circ}{Y} &= N_X := N^{(1)}_X + N^{(2)}_X, \\
        \Delta_{\mathbb{R}^3}\overset{\circ}{Y} - \frac{2\partial Y_K\cdot\partial\overset{\circ}{X}}{X_K} - \frac{(|\partial X_K|^2+|\partial Y_K|^2)}{X_K^2}\overset{\circ}{Y} &= N_Y := N^{(1)}_Y + N^{(2)}_Y,
        \end{aligned}
        \end{equation}
where
\begin{equation*}
    N^{(1)}_X := \frac{X_K^2(|\partial\overset{\circ}{X}|^2-|\partial\overset{\circ}{Y}|^2) + (\overset{\circ}{X}\partial Y_K -\overset{\circ}{Y}\partial X_K)\cdot(2X_K\partial\overset{\circ}{Y}-\overset{\circ}{X}\partial Y_K+\overset{\circ}{Y}\partial X_K)}{X_K^2(1+\overset{\circ}{X})},    
    \end{equation*}
    \begin{equation*}
    N^{(1)}_Y := \frac{\partial\overset{\circ}{X}\cdot\partial\overset{\circ}{Y}+2X_K(\overset{\circ}{Y}\partial X_K-\overset{\circ}{X}\partial Y_K)\cdot\partial\overset{\circ}{X}}{X_K^2(1+\overset{\circ}{X})},    
    \end{equation*}
    \begin{equation*}
    \begin{aligned}
    N^{(2)}_X &:= (\rho^{-1}-\sigma^{-1}\partial_{\rho}\sigma)\frac{ \partial_{\rho}(X_K(1+\overset{\circ}{X}))}{X_K}-\sigma^{-1}\partial_z\sigma\frac{\partial_z(X_K(1+\overset{\circ}{X}))}{X_K} - \frac{2\partial_{\rho}(Y_K+X_K\overset{\circ}{Y})B_{\rho}}{X_K^2(1+\overset{\circ}{X})}   \\
    &- \frac{2\partial_{z}(Y_K+X_K\overset{\circ}{Y})B_z}{X_K^2(1+\overset{\circ}{X})} - \frac{B_{\rho}^2+B_z^2}{X_K^2(1+\overset{\circ}{X})}  + X^{-1}_KF_1(\thetazero, \Xzero, \sigmazero)(\rho, z), 
    \end{aligned}
    \end{equation*}
    \begin{equation*}
        \begin{aligned}
        N^{(2)}_Y &:= (\rho^{-1}-\sigma^{-1}\partial_{\rho}\sigma)\partial_{\rho}(Y_K+X_K\overset{\circ}{Y})X_K^{-1}-\sigma^{-1}\partial_{\rho}\sigma B_{\rho}X_K^{-1} - \sigma^{-1}\partial_{z}\sigma(\partial_{z}(Y_K+X_K\overset{\circ}{Y})+B_z)X_K^{-1} \\
        &+\frac{B_{\rho}\partial_{\rho}(X_K(1+\overset{\circ}{X}))+B_{z}\partial_{z}(X_K(1+\overset{\circ}{X}))}{X_K^2(1+\overset{\circ}{X})},
        \end{aligned}
    \end{equation*}
    and where $\Delta_{\mathbb R^3}$ is the Laplacian corresponding to the flat metric on $\mathbb R^3$ given by $g_{\mathbb R^3} = d\rho^2 + dz^2 + \rho^2d\phi^2$.

\item $\thetazero$ satisfies
    \begin{equation}
\label{thetazero}
        \begin{aligned}
        \partial_{\rho}\overset{\circ}{\Theta} &= -\frac{\sigma}{X^2}(\partial_zY+B_z)+\frac{\rho}{X_K^2}\partial_zY_K, \\
        \partial_{z}\overset{\circ}{\Theta} &= \frac{\sigma}{X^2}(\partial_{\rho}Y+B_{\rho})-\frac{\rho}{X_K^2}\partial_{\rho}Y_K.
        \end{aligned}
    \end{equation}

\item $\lambdazero$ satisfies
    \begin{equation}
\label{lambdazero}
        \begin{aligned}
        \partial_{\rho}\overset{\circ}{\lambda} &= \alpha_{\rho}-(\alpha_K)_{\rho}-\frac{1}{2}\partial_{\rho}\log(1+\overset{\circ}{X}), \\
        \partial_{z}\overset{\circ}{\lambda} &= \alpha_{z}-(\alpha_K)_{z}-\frac{1}{2}\partial_{z}\log(1+\overset{\circ}{X}),
        \end{aligned}
    \end{equation}
    where 
    \begin{equation*}
        \begin{aligned}
        (\alpha_K)_{\rho} &=\frac{1}{4}\rho X_K^{-2}((\partial_{\rho}X_K)^2-(\partial_{z}X_K)^2 + (\partial_{\rho}Y_K)^2)- (\partial_{z}Y_K)^2, \\
        (\alpha_K)_{z} &=\frac{1}{4}\rho X_K^{-2}((\partial_{\rho}X_K)(\partial_{z}X_K)+ (\partial_{\rho}Y_K)(\partial_{z}Y_K)),
        \end{aligned}
    \end{equation*}
where $\alpha_\rho$ and $\alpha_z$ satisfy 
\begin{align*}
 ((\partial_{\rho}\sigma)^2 + (\partial_{z}\sigma)^2)\alpha_{\rho} &= \frac{1}{4} (\partial_{\rho}\sigma)\sigma\frac{(\partial_{\rho}X)^2-(\partial_{z}X)^2+(\theta_{\rho})^2-(\theta_{z})^2}{X^2} + (\partial_{\rho}\sigma)(\partial^2_{\rho}\sigma - \partial^2_{z}\sigma) + (\partial_{z}\sigma)((\partial^2_{\rho,z}\sigma)) \\
    &+ \frac{1}{2}X^{-2}((\partial_{\rho}X)(\partial_{z}X) + (\theta_{\rho})(\theta_{z}))),
\end{align*}
\begin{align*}
((\partial_{\rho}\sigma)^2 + (\partial_{z}\sigma)^2)\alpha_{z} &= -\frac{1}{4} (\partial_{z}\sigma)\sigma\frac{(\partial_{\rho}X)^2-(\partial_{z}X)^2+(\theta_{\rho})^2-(\theta_{z})^2}{X^2}- (\partial_{z}\sigma)(\partial^2_{\rho}\sigma - \partial^2_{z}\sigma) + (\partial_{\rho}\sigma)((\partial^2_{\rho,z}\sigma)) \\
 &+ \frac{1}{2}X^{-2}((\partial_{\rho}X)(\partial_{z}X) + (\theta_{\rho})(\theta_{z}))
\end{align*}
\end{itemize}

\end{Propo}

\begin{proof}
\begin{enumerate}
\item First,  we derive the equation for $\sigmazero$:
\noindent For this, we compute the Laplacian of $\sigmazero$ with respect to the flat metric $g = d\rho^2 + dz^2 + \rho^2d\mathbb{S}^2$:
    \begin{align*}
    \Delta_{\mathbb{R}^4}\sigmazero &= \Delta_{\mathbb{R}^4}\left(\frac{\sigma}{\rho}\right) \\
    &= \frac{1}{\sqrt{\det g}}\partial_i\left(\sqrt{\det g}\ginv{i}{j}\partial_j\left(\frac{\sigma}{\rho}\right)\right), \quad i, j\in\left\{1, \cdots 4\right\} \\
    &= \frac{1}{\rho^2\sin\theta}\partial_{\rho}\left(\rho^2\sin\theta\ginv{\rho}{\rho}\partial_{\rho}\left(\frac{\sigma}{\rho}\right)\right) + \frac{1}{\rho^2\sin\theta}\partial_z\left(\rho^2\sin\theta\ginv{z}{z}\partial_z\left(\frac{\sigma}{\rho}\right)\right) \\    
    &= \partial_{\rho\rho}\left(\frac{\sigma}{\rho}\right) + \frac{2}{\rho}\partial_{\rho}\left(\frac{\sigma}{\rho}\right) + \frac{1}{\rho}\partial_{zz}\sigma \\
    &= \frac{1}{\rho}(\partial_{\rho\rho}\sigma + \partial_{zz}\sigma)
 \end{align*}
\noindent By \eqref{eq:sigma}, we have
\begin{equation*}
X^{-1}\exp{(-2\lambda)}\sigma(\partial^2_{\rho}\sigma + \partial^2_{z}\sigma) = F_3(\thetazero, \Xzero, \sigmazero)(\rho, z), 
\end{equation*}
Therefore, 
\begin{equation*}
\Delta_{\mathbb{R}^4}\sigmazero = \frac{X_K}{\rho^2}e^{2\lambda}\frac{1 + \Xzero}{1 + \sigmazero}F_3(\thetazero, \Xzero, \sigmazero)(\rho, z). 
\end{equation*}

\item The equation for $B$ follows from the definition of $\displaystyle \left(B_\chi^{(N)}, B_{\chi'}^{(S)}, B_z^{(A)}\right)$ and the equations \eqref{eq:for:Bn}, \eqref{eq:for:Bs}, \eqref{eq:for:BA}.  
\item We turn to the equations for $(\Xzero, \Yzero)$: 
\begin{itemize}

\item First of all,  we recall from the classical Carter Robinson theory (see for example \cite{weinstein1992stationary} )that $(X_K,Y_K)$ forms a harmonic map system and satisfies the equations
    \begin{equation}\label{XK::YK}
    \left\{
        \begin{aligned}
            \rho^{-1}\partial_{\rho}(\rho\partial_{\rho}X_K) + \rho^{-1}\partial_{z}(\rho\partial_{z}X_K) &= \frac{(\partial_{\rho}X_K)^2 + (\partial_{z}X_K)^2 - (\partial_{\rho}Y_K)^2 - (\partial_{z}Y_K)^2}{X_K}, \\
            \rho^{-1}\partial_{\rho}(\rho\partial_{\rho}Y_K) + \rho^{-1}\partial_{z}(\rho\partial_{z}Y_K) &= \frac{2(\partial_{\rho}Y_K)(\partial_{\rho}X_K) + 2(\partial_{z}Y_K)(\partial_{z}X_K)}{X_K}.
        \end{aligned}
        \right.
    \end{equation}

\item Now,  we compute
\begin{align*}
    \Delta_{\mathbb{R}^3}\Xzero &= \Delta_{\mathbb{R}^3}\left(\frac{X-X_K}{X_K}\right) \\
    &= \frac{1}{\rho}\partial_{\rho}\left(\rho\partial_{\rho}\left(\frac{X-X_K}{X_K}\right)\right) + \partial_{zz}\left(\frac{X-X_K}{X_K}\right) \\
    &= \frac{1}{\rho}\frac{(\partial_{\rho}X)X_K-(\partial_{\rho}X_K)X}{X^2_K}+\partial_{\rho}\left(\frac{(\partial_{\rho} X)X_K-(\partial_{\rho}X_K)X}{X^2_K}\right)+ \partial_z\left(\frac{(\partial_{z}X)X_K-(\partial_{z}X_K)X}{X^2_K}\right).
\end{align*}
We expand  
\begin{align*}
    \partial_{\rho}\left(\frac{(\partial_{\rho}X)X_K-(\partial_{\rho}X_K)X}{X^2_K}\right) &= \frac{(\partial_{\rho\rho}X)X_K-(\partial_{\rho\rho}X_K)X}{X^2_K} - 2\frac{(\partial_{\rho}X_K)(\partial_{\rho}X)}{X^2_K} + 2X\frac{(\partial_{\rho}X_K)(\partial_{\rho}X)}{X^3_K} \\
    &= \frac{(\partial_{\rho\rho}X)X_K-(\partial_{\rho\rho}X_K)X}{X^2_K} - 2\frac{1}{X_K}\partial_\rho\Xzero\partial_\rho X_K, 
\end{align*}
and 
\begin{align*}
    \partial_{z}\left(\frac{(\partial_{z}X)X_K-(\partial_{z}X_K)X}{X^2_K}\right) &= \frac{(\partial_{zz}X)X_K-(\partial_{zz}X_K)X}{X^2_K} - 2\frac{(\partial_{z}X_K)(\partial_{z}X)}{X^2_K} + 2X\frac{(\partial_{z}X_K)(\partial_{z}X)}{X^3_K} \\
    &= \frac{(\partial_{\rho\rho}X)X_K-(\partial_{\rho\rho}X_K)X}{X^2_K} - 2\frac{1}{X_K}\partial_z\Xzero\partial_z X_K. 
\end{align*}
Moreover, from \eqref{XK::YK}, we have 
\begin{equation*}
    \frac{\partial_{\rho}X_K}{\rho} + \partial_{\rho\rho}X_K + \partial_{zz}X_K = \frac{(\partial_{\rho}X_K)^2 + (\partial_{z}X_K)^2 - (\partial_{\rho}Y_K)^2 - (\partial_{z}Y_K)^2}{X_K}. 
\end{equation*}
Recall that 
\[
\Delta_{\mathbb{R}^2} = \partial_{\rho\rho} + \partial_{zz}.
\]
Hence, 
\begin{align*}
    \Delta_{\mathbb{R}^3}\Xzero &= \frac{1}{X_K}\left(\frac{1}{\rho}\left(\partial_{\rho}X\right)+\Delta_{\mathbb{R}^2}X\right) - 2\frac{1}{X_K}\partial_\rho\Xzero\partial_\rho X_K - 2\frac{1}{X_K}\partial_z\Xzero\partial_z X_K \\
    &= \frac{1}{X_K}\left(\frac{1}{\rho}\left(\partial_{\rho}X\right)+\Delta_{\mathbb{R}^2}X\right) - \frac{X}{X^2_K}\left(\frac{(\partial_{\rho}X_K)^2 + (\partial_{z}X_K)^2 - (\partial_{\rho}Y_K)^2 - (\partial_{z}Y_K)^2}{X_K}\right) \\
     &-2\frac{1}{X_K}\partial\Xzero\cdot\partial X_K \\
        &= \frac{1}{X_K}\left(\frac{1}{\rho}\left(\partial_{\rho}X\right)+\Delta_{\mathbb{R}^2}X\right) - \frac{1 + \Xzero}{X_K}\left( |\partial X_K|^2 - |\partial Y_K|^2\right) - 2\frac{1}{X_K}\partial\Xzero\cdot\partial X_K. 
\end{align*}

\item We replace $X$ by $X_K(1 + \Xzero)$ in the left hand side of  \eqref{eq:X} and we compute 
\begin{align*}
\sigma^{-1}\partial_{\rho}(\sigma\partial_{\rho}X) + \sigma^{-1}\partial_{z}(\sigma\partial_{z}X) &= \frac{1}{\sigma}(\partial_{\rho}\sigma\partial_{\rho}X+\sigma\partial_{\rho\rho}X)+\frac{1}{\sigma}(\partial_{z}\sigma\partial_{z}X+\sigma\partial_{zz}X)  \\
&= \frac{1}{\sigma}(\partial_{\rho}\sigma\partial_{\rho}X + \partial_{z}\sigma\partial_{z}X) + \Delta_{\mathbb{R}^2}X  \\ 
&= \frac{1}{\sigma}\partial_{\rho}\sigma\partial_{\rho}(X_K(1+\Xzero))+ \frac{1}{\sigma}\partial_{z}\sigma\partial_{z}(X_K(1+\Xzero)).
\end{align*}
Therefore, by  \eqref{eq:X}, we have 
\begin{equation}
\label{aux:1}
\begin{aligned}
    \Delta_{\mathbb{R}^2}X &= - \frac{1}{\sigma}\partial_{\rho}\sigma\partial_{\rho}(X_K(1+\Xzero)) -  \frac{1}{\sigma}\partial_{z}\sigma\partial_{z}(X_K(1+\Xzero)) + F_1(\thetazero, \Xzero, \sigmazero)(\rho, z)  \\
    &+  \frac{(\partial_{\rho}X)^2 + (\partial_{z}X)^2 - \theta_{\rho}^2 - \theta_z^2}{X}.
    \end{aligned}
\end{equation}

\item Recall that $\displaystyle dY = \theta - B$. Therefore, 
\begin{equation*}
    \begin{aligned}
    \theta_{\rho} &= \partial_{\rho}Y + B_\rho, \\
    \theta_{z} &= \partial_{z}Y + B_z.
    \end{aligned}
\end{equation*}
This implies
\begin{align*}
    \theta_{\rho}^2 &= (\partial_{\rho}Y)^2+B^2_{\rho}+2B_{\rho}\partial_{\rho}Y, \\
    \theta_{z}^2 &= (\partial_{z}Y)^2+B^2_{z}+2B_{z}\partial_{z}Y.
\end{align*}
\item Therefore, 
\begin{equation}
\label{aux:2}
\begin{aligned}
\frac{(\partial_{\rho}X)^2+(\partial_{z}X)^2-\theta_{\rho}^2-\theta_z^2}{X} &= \frac{|\partial X|^2 - |\partial Y|^2}{X_K(1+\Xzero)}  - \frac{B_\rho^2 + B_z^2}{X_K(1 + \Xzero)} - 2 \frac{B_\rho\partial_\rho(Y_K+\Yzero X_K)}{X_K(1 + \Xzero)}  \\
&- 2 \frac{B_z\partial_z(Y_K+\Yzero X_K)}{X_K(1 + \Xzero)} .
\end{aligned}
\end{equation}
\item We set
\begin{equation}
\label{aux:3}
\begin{aligned}
    N^{(2)}_X &:= (\rho^{-1}-\sigma^{-1}\partial_{\rho}\sigma)\left(\frac{\partial_{\rho}(X_K(1+\Xzero))}{X_K}\right) -\sigma^{-1}\frac{\partial_z\sigma\partial_z(X_K(1+\Xzero))}{X_K} + X^{-1}_KF_1(\thetazero, \Xzero, \sigmazero)(\rho, z) \\
    &-\frac{B^2_{\rho}+B^2_{z}}{X^2_K(1+\Xzero)}- \frac{2B_{\rho}\partial_{\rho}(Y_K+\Yzero X_K)}{X^2_K(1+\Xzero)}- \frac{2B_{z}\partial_{z}(Y_K+\Yzero X_K)}{X^2_K(1+\Xzero)}.
\end{aligned}
\end{equation}
\item From \eqref{aux:1}, \eqref{aux:2} and \eqref{aux:3}, we obtain 
\begin{align*}
\Delta_{\mathbb R^3}\Xzero &= N_X^{(2)} -  \frac{1 + \Xzero}{X_K}\left( |\partial X_K|^2 - |\partial Y_K|^2\right) - 2\frac{1}{X_K}\partial\Xzero\cdot\partial X_K + \frac{\left(|\partial X|^2 - |\partial Y|^2\right)}{X_K^2(1 + \Xzero)}. 
\end{align*}
Now, we replace $X$ by $X_K(1 + \Xzero)$ and $Y$ by $Y_K + X_K\Yzero$ in the last term of the right hand side and we expand so that we obtain 
\begin{align*}
\Delta_{\mathbb R^3}\Xzero &= N_X^{(2)} -  \frac{1 + \Xzero}{X_K}\left( |\partial X_K|^2 - |\partial Y_K|^2\right) - 2\frac{1}{X_K}\partial\Xzero\cdot\partial X_K + \frac{\left(|\partial X|^2 - |\partial Y|^2\right)}{X_K^2(1 + \Xzero)} \\
&= \frac{X_K^2(|\partial\Xzero|^2 - |\partial\Yzero|^2) + \Xzero(\Xzero + 2)|\partial Y_K|^2 - |\partial X_K|^2\Yzero^2 - 2\Yzero\partial Y_K\cdot \partial X_K}{X_K^2(1 + \Xzero)} \\
&-\frac{2X_K\Yzero\partial\Yzero\cdot \partial X_K - 2X_K\partial Y_K\cdot\partial\Yzero}{X_K^2(1 + \Xzero)} \\
&= \frac{X_K^2(|\partial\Xzero|^2 - |\partial\Yzero|^2)  + 2\Xzero(\Xzero + 1)|\partial Y_K|^2 - \Xzero^2|\partial Y_K|^2 - 2(\Xzero + 1)\Yzero\partial Y_K\cdot\partial X_K }{X_K^2(1 + \Xzero)} \\
&+\frac{2\Xzero\Yzero \partial Y_K\cdot X_K - 2X_K\Yzero\partial X_K\cdot\partial\Yzero - 2(1 + \Xzero)X_K\partial Y_K\cdot\partial\Yzero + 2\Xzero X_K\partial Y_K\cdot\partial\Yzero}{X_K^2(1 + \Xzero)}.
\end{align*}
\item We set 
\begin{equation*}
    N^{(1)}_X := \frac{X_K^2(|\partial\overset{\circ}{X}|^2-|\partial\overset{\circ}{Y}|^2) + (\overset{\circ}{X}\partial Y_K -\overset{\circ}{Y}\partial X_K)\cdot(2X_K\partial\overset{\circ}{Y}-\overset{\circ}{X}\partial Y_K+\overset{\circ}{Y}\partial X_K)}{X_K^2(1+\overset{\circ}{X})},    
    \end{equation*}
\item Hence, 
\begin{equation*}
\Delta_{\mathbb R^3}\Xzero = N_X^{(2)} + N_X^{(1)}  - \frac{2\partial Y_K\cdot\partial\overset{\circ}{Y}}{X_K} + \frac{2|\partial Y_K|^2}{X_K^2}\overset{\circ}{X} - 2\frac{\partial X_K\cdot\partial Y_K}{X_K^2}\overset{\circ}{Y}.
\end{equation*}
which is equivalent to \eqref{XYzero}. 
\item The equation for $\Yzero$ is derived in the same way. 
\end{itemize}

\item We derive the equation for $\thetazero$: $\thetazero$ is defined by 
\begin{equation*}
  X^{-1}W  =  \thetazero + X_K^{-1}W_K.
\end{equation*}
Moreover, by \eqref{eq:W}, we have 
\begin{equation*}
\partial_{\rho}(X^{-1}W)d\rho + \partial_{z}(X^{-1}W)dz = \frac{\sigma}{X^2}(\theta_{\rho}dz - \theta_{z}d\rho) 
\end{equation*}
and $X_K^{-1}W_K$ satisfies 
\begin{equation*}
\partial_{\rho}(X_K^{-1}W_K)d\rho + \partial_{z}(X_K^{-1}W_K)dz = \frac{\rho}{X_K^2}(\partial_\rho Y_Kdz - \partial_zY_Kd\rho).  
\end{equation*}
We recall that 
\begin{equation*}
    \begin{aligned}
    \theta_{\rho} &= \partial_{\rho}Y + B_\rho, \\
    \theta_{z} &= \partial_{z}Y + B_z.
    \end{aligned}
\end{equation*}
Therefore, 
\begin{equation*}
\begin{aligned}
\partial_\rho\thetazero d\rho + \partial_z\thetazero dz &= \partial_{\rho}(X_K^{-1}W_K)d\rho + \partial_{z}(X_K^{-1}W_K)dz - \left( \partial_{\rho}(X_K^{-1}W_K)d\rho + \partial_{z}(X_K^{-1}W_K)dz \right) \\ 
&=  \frac{\sigma}{X^2}(\theta_{\rho}dz - \theta_{z}d\rho)  -  \frac{\rho}{X_K^2}(\partial_\rho Y_Kdz - \partial_zY_Kd\rho) \\
&=  \frac{\sigma}{X^2}(\partial_\rho Ydz - \partial_zYd\rho) + \frac{\sigma}{X^2}(B_\rho dz - B_z d\rho) - \frac{\rho}{X_K^2}(\partial_\rho Y_Kdz - \partial_zY_Kd\rho).
\end{aligned}
\end{equation*}
Thus, we find the following pair of equations for $\thetazero$: 
 \begin{equation*}
         \begin{aligned}
        \partial_{\rho}\overset{\circ}{\Theta} &= -\frac{\sigma}{X^2}(\partial_zY+B_z)+\frac{\rho}{X_K^2}\partial_zY_K, \\
        \partial_{z}\overset{\circ}{\Theta} &= \frac{\sigma}{X^2}(\partial_{\rho}Y+B_{\rho})-\frac{\rho}{X_K^2}\partial_{\rho}Y_K.
        \end{aligned}
    \end{equation*}

\item Finally, we derive the equations for $\lambdazero$: $\lambda$ satisfies
 
 \begin{equation*}
    \partial_{\rho}\lambda = \alpha_{\rho} - \frac{1}{2}\partial_{\rho}\log X_K(1+\Xzero).
\end{equation*}
This implies  
\begin{equation*}
    \partial_{\rho}\lambdazero = \alpha_{\rho} - (\partial_{\rho}\lambda_K+\frac{1}{2}\log X_K) - \frac{1}{2}\partial_{\rho}\log(1+\Xzero).
\end{equation*}
Similarly,
\begin{equation*}
    \partial_{z}\lambdazero = \alpha_{z} - (\partial_{z}\lambda_K+\frac{1}{2}\log X_K) - \frac{1}{2}\partial_{z}\log(1+\Xzero).
\end{equation*}
Now, recall that $\lambda_K$ satisfies 
        \begin{equation*}
            \left\{
            \begin{aligned}
                \partial_{\rho}\lambda_K &= \frac{1}{4}\rho{X_K}^{-2} ((\partial_{\rho}X_K)^2-(\partial_{z}X_K)^2+(\partial_{\rho}Y_K)^2-(\partial_{z}Y_K)^2) - \frac{1}{2}\partial_{\rho}\log{X_K}, \\
                \partial_{z}\lambda_K &= \frac{1}{4}\rho{X_K}^{-2} ((\partial_{\rho}X_K)(\partial_{z}X_K)+(\partial_{\rho}Y_K)(\partial_{z}Y_K)) - \frac{1}{2}\partial_{z}\log{X_K}.
            \end{aligned}
            \right.
            \end{equation*}
Define $(\alpha_K)_{\rho}$ and $(\alpha_K)_{z}$ by setting
\begin{equation*}
    \begin{aligned}
        (\alpha_K)_{\rho} &:= \partial_{\rho}\lambda_K + \frac{1}{2}\log X_K, \\
        (\alpha_K)_{z} &:= \partial_{z}\lambda_K + \frac{1}{2}\log X_K.
    \end{aligned}
\end{equation*}
Therefore, $\lambdazero$ satisfies
\begin{equation*}
        \begin{aligned}
        \partial_{\rho}\overset{\circ}{\lambda} &= \alpha_{\rho}-(\alpha_K)_{\rho}-\frac{1}{2}\partial_{\rho}\log(1+\overset{\circ}{X}), \\
        \partial_{z}\overset{\circ}{\lambda} &= \alpha_{z}-(\alpha_K)_{z}-\frac{1}{2}\partial_{z}\log(1+\overset{\circ}{X}).
        \end{aligned}
    \end{equation*}

\end{enumerate}
\end{proof}
 
\begin{remark}
\label{order::solving}
As in \cite{chodosh2017time}, the order we have presented the renormalised unknowns reflects the order in which we will treat their equations. We will first solve for $\sigmazero$,  for $B$, for $(\Xzero, \Yzero)$, for $\thetazero$ then for $\lambdazero$. Indeed, one has to solve for $\Yzero$ and $B$ before solving the $\thetazero$ equation since the necessary boundary condition to integrate the equation for $\thetazero$.  
\end{remark}


\subsection{Functional spaces on $\Bbarre$}
\label{function::spaces:bis}
In this section, we define the functional spaces for the renormalised unknowns in order  to apply standard elliptic theory to solve  non-homogeneous linear problems and to establish non-linear estimates. 
\\ Firstly, for any $(x, y , z)\in\mathbb R^3$,  let  $(\rho,\vartheta, z)\in[0, \infty[\times\mathbb R\times(0, 2\pi)$ be its cylindrical coordinates  defined by 
\begin{equation*}
\begin{aligned}
x &= \rho\cos\vartheta, \\
y &= \rho\sin\vartheta, \\
z &= z. 
\end{aligned}
\end{equation*}
To any function $f:\Bbarre\mapsto \mathbb R$ we associate an axisymmetric function $f_{\mathbb R^3}:{\mathbb R^3}\mapsto \mathbb R$ by setting 
\begin{equation}
\label{extended:f}
f_{\mathbb R^3}(x, y, z) := f(\rho(x, y), z).
\end{equation}
\noindent Now, we introduce the following function spaces which are  associated to $\Bbarre$: 
   \begin{equation*}
            \begin{aligned}
            \dot W^{k,p}_{axi}(\Bbarre) &:= \left\{ f(\rho,z), \quad f_{\mathbb{R}^3}\in\dot W^{k,p}(\mathbb{R}^3) \right\} ,\\
            W^{k,p}_{axi}(\Bbarre) &:= \left\{ f(\rho,z), \quad f_{\mathbb{R}^3}\in W^{k,p}(\mathbb{R}^3) \right\} , \\
            C^{k,\alpha}_{axi}(\Bbarre) &:= \left\{ f(\rho,z), \quad f_{\mathbb{R}^3}\in C^{k,\alpha}(\mathbb{R}^3) \right\}, \\
            C^{k,\alpha}_{0,axi}(\Bbarre) &:= \left\{ f(\rho,z), \quad f_{\mathbb{R}^3}\in C^{k,\alpha}_{0}(\mathbb{R}^3) \right\}.  \\
            \end{aligned}
\end{equation*}

\noindent The above spaces can also be seen as the Sobolev and Hölder spaces of axially symmetric functions defined on $\mathbb R^3$. 
Now, for any $(x, y, u, v)\in\mathbb R^4$, let $(s, \vartheta_1, \chi, \vartheta_2)\in[0, \infty[\times(0, 2\pi)\times[0, \infty[\times(0, 2\pi) $ or $(s', \vartheta_1, \chi', \vartheta_2)\in[0, \infty[\times(0, 2\pi)\times[0, \infty[\times(0, 2\pi) $  be its polar coordinates defined by
\begin{equation}
\label{xy::uv::}
\begin{aligned}
x &= s\cos\vartheta_1, \\
y &= s\sin\vartheta_1, \\
u &= \chi\cos\vartheta_2, \\
v &= \chi\sin\vartheta_2.  \\
\end{aligned}
\end{equation}
Now, to any function $f^N:\Bnbarre\mapsto \mathbb R$ or $f^S:\Bsbarre\mapsto \mathbb R$, we associate a function $f^N_{\mathbb R^4}: \mathbb R^4\mapsto\mathbb R$ or $f^S_{\mathbb R^4}: \mathbb R^4\mapsto\mathbb R$ by setting
 \begin{equation*}
f^N_{\mathbb R^4}(x, y, u, v)  := f^N(s(x, y), \chi(u, v)), 
\end{equation*}   
  \begin{equation*}
f^S_{\mathbb R^4}(x, y, u, v)  := f^N(s'(x, y), \chi'(u, v)).
\end{equation*} 
Another family of function spaces will be used during the analysis. They are defined by 
   \begin{equation*}
            \begin{aligned}
            \widehat{\dot W}^{k,p}_{axi}(\Bbarre) &:= \left\{ f(\rho,z), \;(\xi_Nf)_{\mathbb{R}^4}, (\xi_Sf)_{\mathbb{R}^4}\in \dot W^{k,p}(\mathbb{R}^4),\; ((1-\xi_N-\xi_S)f)_{\mathbb{R}^3}\in \dot W^{k,p}(\mathbb{R}^3) \right\} ,\\
            \widehat{W}^{k,p}_{axi}(\Bbarre) &:= \left\{ f(\rho,z), \; (\xi_Nf)_{\mathbb{R}^4}, (\xi_Sf)_{\mathbb{R}^4}\in W^{k,p}(\mathbb{R}^4),\; ((1-\xi_N-\xi_S)f)_{\mathbb{R}^3}\in W^{k,p}(\mathbb{R}^3) \right\}, \\
            \widehat{C}^{k,\alpha}_{axi}(\Bbarre) &:= \left\{ f(\rho,z),\; (\xi_Nf)_{\mathbb{R}^4}, (\xi_Sf)_{\mathbb{R}^4}\in C^{k,\alpha}(\mathbb{R}^4), \; ((1-\xi_N-\xi_S)f)_{\mathbb{R}^3}\in C^{k,\alpha}(\mathbb{R}^3) \right\}, \\
            \widehat{C}^{k,\alpha}_{0,axi}(\Bbarre) &:= \left\{ f(\rho,z),\;(\xi_Nf)_{\mathbb{R}^4}, (\xi_Sf)_{\mathbb{R}^4}\in C^{k,\alpha}_{0}(\mathbb{R}^4),\;((1-\xi_N-\xi_S)f)_{\mathbb{R}^3}\in C^{k,\alpha}_{0}(\mathbb{R}^3) \right\}. 
            \end{aligned}
\end{equation*}
\noindent Similarly, we define the spaces $\displaystyle \widehat C^k_{axi}$ and $\displaystyle \widehat C^k_{axi, 0}$, and we set $\displaystyle C^\infty := \cap_{k=1}^\infty \widehat C^k_{axi}$ and $\displaystyle C^\infty := \cap_{k=1}^\infty \widehat C^k_{axi, 0}$.
\begin{remark}
When there is unlikely to be any confusion, we will often drop the subscript "axi". 
\end{remark}
We state the following lemma on the relationship between the hatted and non-hatted Hölder spaces
\begin{lemma}
Let $\alpha\in(0,1)$ and $k\in\mathbb{N}$. We have the continuous inclusion 
\begin{equation*}
    C^{k,\alpha}_{axi}(\Bbarre)\subset\hat{C}^{k,\alpha}_{axi}(\Bbarre).
\end{equation*}
\end{lemma}
\begin{proof}
See proof of Lemma 3.2.1 in  \cite{chodosh2017time}. 
\end{proof}
\noindent In the following section, we introduce the function spaces for the renormalised unknowns $$\displaystyle \left(\sigmazero, B, \left(\Xzero, \Yzero \right), \thetazero, \lambdazero \right)$$.

\subsection{Function spaces for the renormalised unknowns}
Let $\alpha_0\in(0, 1)$ be fixed and let  the function $r: [0, \infty[\times\mathbb R\to\mathbb R_+$  be defined by $r(\rho, z):= \sqrt{1 + \rho^2 + z^2}$ .
\noindent Note that $r(\rho, z)$ is different from the radial coordinate appearing in BL coordinates. See Section \ref{Kerr:geo:aux}.  
\subsubsection{Function spaces for ${\sigma}$}
\begin{definition}
The Banach space  $\left(\Lsigma, ||\cdot||_{\Lsigma} \right)$ is defined to be the completion of smooth functions $f\in \hat C^{\infty}_0(\Bbarre)$ under the norm
\begin{equation*}
\begin{aligned}
 ||f||_{\Lsigma} &:= ||f||_{C^{3,\alpha_0}(\overline{\BB})} + ||r^2f||_{L^{\infty}(\Bbarre)} + ||r^3\partial f||_{L^{\infty}(\Bbarre)} +  ||r^4\log^{-1}(4r)\partial^2 f||_{L^{\infty}(\Bbarre)}  \\
 &+ ||r^4\log^{-1}(4r)\partial^2 f||_{C^{0,\alpha_0}(\Bbarre)}. 
    \end{aligned}
\end{equation*}
The Banach space $\left(\Nsigma, ||\cdot||_{\Nsigma} \right)$ is defined to be the completion of smooth functions $f\in \hat C^{\infty}_0(\Bbarre)$ under the norm
\begin{equation*}
 ||f||_{\Nsigma} := ||r^5f||_{C^{1,\alpha_0}({\Bbarre})} . 
\end{equation*}
\end{definition}

\subsubsection{Function spaces for ${B}$}
\begin{definition}
The Banach space  $\left(\LB, ||\cdot||_{\LB} \right)$ is defined to be the completion of triples $\left(F_z^{(A)}, F_\chi^{(N)}, F_{\chi'}^{(S)}\right)\in \left(\hat C^{\infty}_0(\Bbarre)\right)^3$ under the norm
\begin{equation*}
\begin{aligned}
 ||\left(F_z^{(A)}, F_\chi^{(N)}, F_{\chi'}^{(S)}\right)||_{\LB} &:= \left|\left|\frac{(1 + \rho^{10})(1 + r^{10})}{\rho^{10}}F_z^{(A)}\right|\right|_{\hat C^{1,\alpha_0}(\BAbarre\cup\BHbarre)}  \\
 &+ \left|\left|s^{-10} F^{(N)}_\chi\right|\right|_{\hat C^{1,\alpha_0}(\Bnbarre)} + \left|\left|(s')^{-10} F^{(S)}_{\chi'}\right|\right|_{\hat C^{1,\alpha_0}(\Bsbarre)}.
 \end{aligned}
\end{equation*}
The Banach space $\left(\NB, ||\cdot||_{\NB} \right)$ is defined to be the completion of smooth functions $f\in \hat C^{\infty}_0(\Bbarre)$ under the norm
\begin{equation*}
\begin{aligned}
 ||\left(F_z^{(A)}, F_\chi^{(N)}, F_{\chi'}^{(S)}\right)||_{\NB} &:= \left|\left|\frac{(1 + \rho^{15})(1 + r^{10})}{\rho^{15}}F_z^{(A)}\right|\right|_{\hat C^{1,\alpha_0}(\BAbarre\cup\BHbarre)}  \\
 &+ \left|\left|s^{-15} F^{(N)}_\chi\right|\right|_{\hat C^{1,\alpha_0}(\Bnbarre)} + \left|\left|(s')^{-15} F^{(S)}_{\chi'}\right|\right|_{\hat C^{1,\alpha_0}(\Bsbarre)}.
 \end{aligned}
\end{equation*}
\end{definition}

\subsubsection{Function spaces for $X$ and $Y$ }
\begin{definition}
The Banach space $\left(\LX, ||\cdot||_{\LX} \right)$ is defined to be the completion of smooth functions $f\in \hat C^{\infty}_0(\Bbarre)$ under the norm
 \begin{equation*}
      ||f||_{\LX} := ||f||_{\dot W^{2,p}_{axi}(\Bbarre)} + ||f||_{\hat{C}^{2,\alpha_0}(\Bbarre)} + ||rf||_{L^{\infty}(\Bbarre)} + ||r^2\hat{\partial} f||_{L^{\infty}(\Bbarre)} + ||r^3\log^{-1}(4r)\hat{\partial}^2 f||_{C^{0,\alpha_0}(\Bbarre)}.
             \end{equation*}
The Banach space $\left(\NX, ||\cdot||_{\NX} \right)$ is defined to be the completion of smooth functions $f\in \hat C^{\infty}_0(\Bbarre)$ under the norm
  \begin{equation*}
      ||f||_{\NX} := ||r^3(1-\xi_N-\xi_S)f||_{C^{0,\alpha_0}(\mathbb{R}^3)} + ||(\chi^2+s^2)\xi_Nf||_{C^{0,\alpha_0}(\Bnbarre)} + ||((\chi')^2+(s')^2)\xi_Sf||_{C^{0,\alpha_0}(\Bsbarre)}.
    \end{equation*}
\end{definition}

\begin{definition}
The Banach space $\left(\LY, ||\cdot||_{\LY} \right)$ is defined to be the completion of smooth functions $f\in \hat C^{\infty}_0(\Bbarre)$ under the norm
 \begin{equation*}
 \begin{aligned}
||f||_{\LY} &:= ||f||_{\dot W^{2,p}_{axi}(\Bbarre)} + |||\partial h|f ||_{L^2(\mathbb{R}^3)} + ||f||_{\hat{C}^{2,\alpha_0}(\Bbarre)} + ||X_K^{-1}f||_{\hat{C}^{2,\alpha_0}(\Bbarre)} + ||r^3X_K^{-1}f||_{L^{\infty}(\Bbarre)}  \\
&+ ||r^4\hat{\partial}( X^{-1}_Kf)||_{L^{\infty}(\Bbarre)} + ||r^5\log^{-1}(4r)\hat{\partial}^2( X^{-1}_Kf)||_{C^{0,\alpha_0}(\Bbarre)}.
\end{aligned}
             \end{equation*}
The Banach space $\left(\NY, ||\cdot||_{\NY} \right)$ is defined to be the completion of smooth functions $f\in \hat C^{\infty}_0(\Bbarre)$ under the norm
  \begin{equation*}
  \begin{aligned}
      ||f||_{\mathcal{N}_{Y}} &:= ||fr^5X^{-1}_K||_{\hat{C}^{0,\alpha_0}\left(\left(\BHbarre\cup \BAbarre\right)\cap \left\{ \rho\le 1\right\}\right)} +  ||fr^4||_{\hat{C}^{0,\alpha_0}(\Bbarre\cap\left\{ \rho\ge 1\right\})} + ||(\chi^2+s^2)\xi_Nf||_{C^{0,\alpha_0}(\Bnbarre)}  \\
      &+ ||((\chi')^2+(s')^2)\xi_Sf||_{C^{0,\alpha_0}(\Bsbarre)}. 
      \end{aligned}
    \end{equation*}
\end{definition}

\subsubsection{Function spaces for $\Theta$}

\begin{definition}
The Banach space $\left(\Ltheta, ||\cdot||_{\Ltheta} \right)$ is defined to be the completion of smooth functions $f\in \hat C^{\infty}_0(\Bbarre)$ under the norm
    \begin{equation*}
        ||f||_{\Ltheta} := ||r^2f||_{\hat C^{2,\alpha_0}(\Bbarre)}.
    \end{equation*}
The Banach space $\left(\Ntheta, ||\cdot||_{\Ntheta} \right)$ is defined to be the completion of pairs of smooth compatly supported closed $1-$forms $F$ under the norm 
\begin{equation*}
\begin{aligned}
      ||F||_{\Ntheta} &:= ||r^3(1+\rho^{-1})F_{\rho}||_{\hat{C}^{1,\alpha_0}\left(\BHbarre\cup\BAbarre\right)} + ||r^3F_{z}||_{\hat{C}^{1,\alpha_0}(\left(\BHbarre\cup\BAbarre\right))} \\
      &+||s^{-1}F_{s}||_{\hat{C}^{1,\alpha_0}(\Bnbarre)} + ||F_{\chi}||_{\hat{C}^{1,\alpha_0}(\Bnbarre)} \\
      &+ ||(s')^{-1}F_{s'}||_{\hat{C}^{1,\alpha_0}(\Bsbarre)} + ||F_{\chi'}||_{\hat{C}^{1,\alpha_0}(\Bsbarre)}
      \end{aligned}
    \end{equation*}

\end{definition}

\subsubsection{Function spaces for $\lambda$}
\begin{definition}
The Banach space $\left(\Llambda, ||\cdot||_{\Llambda} \right)$ is defined to be the completion of smooth functions $f\in \hat C^{\infty}(\Bbarre)$ under the norm
    \begin{equation*}
      ||f||_{\Llambda} := ||f||_{\hat C^{1,\alpha_0}(\Bbarre)}. 
    \end{equation*}
\end{definition}

It is easy to see that all these spaces are Banach spaces.

\section{Main Result}
\label{main::result:two:ref}
In this section, we give a more detailed formulation of our result. 
More precisely, we have 
\begin{theoreme}
\label{main::result}
Let 
\begin{enumerate}
\item $0< |a|< M$ and $\spacetime = \mathbb R\times(0, 2\pi)\times \BB$ be the the domain of outer communications minus the axis of symmetry  parametrised by the standard Weyl coordinates $(t, \phi, \rho, z)$,
\item $\Bbound\subset\subset \mathcal A_{bound}$ be a compact subset of the set $\Abound$ defined by Definition \ref{A:bound:def},  
\item $\Phi:\Abound\times\mathbb R_+ \to \mathbb R^\ast_+$ be a $C^2$ function with respect to the first two variables, $C^1$ with respect to the third variable and such that
\begin{itemize}
\item $\displaystyle \forall\delta\in[0, \infty[\;$, $\displaystyle \Phi(\cdot, \cdot ; \delta)$ is supported in $\displaystyle \Bbound$.  
\item $\displaystyle \forall (E, \ell)\in\Abound\,,\;\; \Phi(E, \ell; 0) = \partial_\ell\Phi(E, \ell; 0) = 0$ and $\displaystyle\forall\delta>0, \Phi(\cdot, \cdot, \delta)$ does not identically vanish on $\Bbound$.   
\end{itemize}
\end{enumerate}
Then, there exists $\delta_0>0$ and a one-parameter family of functions 
\begin{equation*}
(V_\delta, W_\delta, X_\delta, \lambda_\delta)_{\delta\in[0, \delta_0[}\in(C^{2,\alpha}(\BB))^4, \; f^\delta\in C^2(\BB\times\mathbb R^3)
\end{equation*}
with the following properties 
\begin{enumerate}
\item $(V_0, W_0, X_0, \lambda_0) = (V_K, W_K, X_K, \lambda_K)$ corresponds to a Kerr solution with parameters $(a, M)$. 
\item For all $(E, \ell_z)\in \Bbound$, the equation 
\begin{equation}
\label{sol:Z:pert}
E_{\ell_z}(W_\delta, X_\delta, \sigma_\delta, \rho, z) = E^2 
\end{equation} 
admits a unique solution curve with two connected components, such that one of them is diffeomorphic to $\mathbb S^1$.  Moreover, there exists $\eta>0$ depending only on $\delta_0$ such that 
\begin{equation*}
\rho_0(W_\delta, X_\delta, \sigma_\delta, E, \ell_z) + \eta < \rho_1(W_\delta, X_\delta, \sigma_\delta, E, \ell_z)
\end{equation*}
where $\displaystyle \rho_i(W_\delta, X_\delta, \sigma_\delta, E, \ell_z)$ are the two smallest solutions of the equation \eqref{sol:Z:pert} with $z = 0$. 
\item  The function $\displaystyle f^\delta$ takes the form 
\begin{equation*}
f^\delta(x,v) =\Phi(E^\delta, \ell; \delta)\Psi_\eta\left(r, (\ve_\delta, (\ell_z)_\delta), W_\delta, X_\delta, \sigma_\delta\right), 
\end{equation*}
for $(x, v)\in\mathcal O\times \mathbb R^3$ with coordinates $(t, r, \theta, \phi, v^r, v^\theta, v^\phi)$ and where
\begin{equation*}
\begin{aligned}
\varepsilon_\delta &= \frac{\sigma_\delta}{\sqrt{X_\delta}}(1 + |p|^2)^{\frac{1}{2}} - \frac{W_\delta}{X_\delta}(\ell_z)_\delta , \\
(\ell_z)_\delta &= \sqrt{X_\delta}p^\phi,
\end{aligned}
\end{equation*}
and $\Psi_\eta$ is defined by \eqref{cut::off}. 
\item Let $g_\delta$ be defined on $\Bbarre$ by
\begin{equation*}
g_\delta := -V_\delta dt^2 + 2W_\delta dtd\phi + X_\delta d\phi^2 + e^{2\lambda_\delta}\left( d\rho^2 + dz^2\right)
\end{equation*}
then $(\mathcal O, g_\delta, f^\delta)$ is a stationary and axially symmetric solution to the Einstein-Vlasov system \eqref{EFE} - \eqref{Vlasov2} - \eqref{EM_tensor} describing a  matter shell orbiting a Kerr like black hole in the following sense: 
\begin{itemize}
\item $\displaystyle \exists \rho_{min}^\delta, \rho_{max}^\delta \in]0,\infty[$ and $Z_{min}^\delta, Z_{max}^\delta \in \mathbb R$  which satisfy 
\begin{equation*}
\begin{aligned}
&\rho_{min}^\delta < \rho_{max}^\delta, \quad\text{and}\quad \rho_{min}^\delta>\rho^{mb, +}(a, M) \\
& Z_{min}^\delta < 0 < Z_{max}^\delta, 
\end{aligned}
\end{equation*}
\item we have
\begin{equation*}
\supp_{(\rho, z)}f\subset\subset\left[\rho_{min}(h), \rho_{max}(h)\right]\times\left[Z_{min}(h), Z_{max}(h) \right]. 
\end{equation*}
and  $\exists (\tilde\rho, \tilde z)\in \left[\rho_{min}(h), \rho_{max}(h)\right]\times\left[Z_{min}(h), Z_{max}(h) \right] ,$
\begin{equation*}
f^\delta(\tilde \rho, \tilde z, \cdot) > 0.   
\end{equation*}
\item The region $\Horizon = \left\{(\rho, z)\;:\; \rho = 0 \,,\, |z|<\beta  \right\}$  corresponds to a non-degenerate bifurcate Killing horizon on which the metric has a $C^{2, \alpha}$ extension in sense of Definition \ref{extendibility}.
\end{itemize}
\end{enumerate}
\end{theoreme}

\subsubsection*{Overview of the poof}
In this section, we give an overview of the proof of Theorem \ref{main::result}. 
\begin{enumerate}
\item First of all, we  show  in Section \ref{perturbed:Kgeo} that the compact connected component of ZVC associated to a trapped geodesic  with parameters $(\ve, \ell_z)\in\Bbound$ moving in a Kerr exterior, $Z^K(\ve, \ell_z)$, remains stable under stationary and axisymmetric perturbations of the Kerr metric. Then, using the compactness of $\Bbound$, we show that trapped geodesics moving in the perturbed spacetimes lie in a compact region of $\BB$ which is uniform in $(\ve, \ell_z)$. This allows us to obtain a distribution function which is compactly supported in $\BB$. Consequently, all the matter terms $F_i(\thetazero, \Xzero, \sigmazero)$ are compactly supported in $\BB$ and vanish in a neighbourhoods of the horizon, the axis of symmetry and  the poles. 
\item Then, to resolve the nonlinear aspects of the problem, we will used  two fixed point lemmas, which are introduced in Section \ref{Fixed::point::lemmaa}. We will start with the study of a toy model which illustrates the application of these lemmas. In the general case, we will have to deal with the difficulty related to the nonlinear coupling of the equations. 
\item At this stage, we introduce a bifurcation parameter $\delta\geq 0 $ in the ansatz for the distribution function which turns on  the presence of Vlasov matter. This allows us to transform the problem of finding solutions to the reduced EV system for the renormalised quantities into that of finding a one-parameter family of solutions which depends on $\delta$, by applying a fixed point lemma, considered as a zero of a well-defined operator.  
\item Note that we will solve each equation separately and the order in which we solve them matters. See Remark \ref{order::solving}. More precisely:
\begin{itemize}
\item We begin by solving the equation for $\sigmazero$ in terms of the remaining quantities and the bifurcation parameter $\delta$. The regularity for the matter terms will allow us to have a $C^1$ dependence of $\sigmazero$ in $(\Xzero, \thetazero, \lambdazero)$ and a continuous dependance with respect to $\delta$. To this end, we apply a fixed point lemma.
\item Then, we solve the equations for $B$ in terms of $(\sigmazero, \Xzero, \thetazero)$. Note that $\sigmazero$ depends on the other renormalised quantities and $\delta$. Therefore, after the application of the fixed point theorem, we will obtain a one parameter family of solutions $(B, \sigmazero)$ which depend in $C^1$ manner of $(\Xzero, \thetazero, \lambdazero)$ and continuously on $\delta$. 
\item We iterate the solving process in order to solve the equations for $(\Xzero, \Yzero)$ in terms of $(\thetazero, \lambdazero; \delta)$, then $\thetazero$ in terms of $(\lambdazero; \delta)$ and finally the equations for $\lambdazero$ in terms of $\delta$ only. 
\item Consequently, we obtain a one-parameter family of solutions $(\sigmazero, B, \Xzero, \Yzero, \thetazero, \lambdazero)$ which depends continuously on $\delta$. 
\end{itemize}
\end{enumerate}

\section{Perturbation of trapped Kerr geodesics}
\label{perturbed:Kgeo}
In this section, we show that given a compact set of parameters $(\ve, \ell_z)$ leading to trapped orbits, the latter remain stable under small metric data perturbations. 
\\ First of all, let $\Bbound$ be a compact subset of the form \eqref{compact:Bbound} and let $(\ve, \ell_z)\in\Bbound$. We recall from \ref{repara:Z:K} that 
\begin{enumerate}
\item there exists $\displaystyle \Phi^{K, abs}_{(\ve, \ell_z)}: ]-\beta, \beta[\mapsto ]0, \rho_0^K(\ve, \ell_z)]$ such that 
\begin{equation*}
Z^{K, abs}(\ve, \ell_z) = \text{Gr} \left(\Phi^{K, abs}_{(\ve, \ell_z)}\right),
\end{equation*} 
\item there exist $\displaystyle \Phi^{K, i}_{(\ve, \ell_z)}: I^i_{(\ve, \ell_z)}\mapsto \mathbb R,\; i=1\cdots 4$ such that 
\begin{equation*}
Z^{K, trapped}(\ve, \ell_z) = \bigcup_{i=1\cdots 4}\text{Gr} \left(\Phi^{K, i}_{(\ve, \ell_z)}\right),
\end{equation*} 
where $\displaystyle I^i_{(\ve, \ell_z)}, \Phi^{K, i}_{(\ve, \ell_z)}$ and $\displaystyle \Phi^{K, abs}_{(\ve, \ell_z)}$ are defined in \eqref{I:i:def}. 
\end{enumerate}
Furthermore, we recall from Lemma \ref{decomp:B} that $\forall (\ve, \ell_z)\in\Bbound$, there exist $\BB_i,\; i= 1\cdots 5$ and $\BB^{abs}$ such that 
\begin{equation}
\label{B::decomp}
\BB = \BB^{abs}\cup\left(\cup_{i= 1\cdots 5}\BB_i \right). 
\end{equation}
\noindent Before we state the main result of this section, we give the following definition of $\delta$-perturbations of a  zero velocity curve, $Z^K(\ve, \ell_z)$ associated to a timelike future-directed geodesic with constants of motion $(\ve, \ell_z)\in\Abound$ moving in Kerr exterior. 
\begin{definition}[$Z^K(\ve, \ell_z)$ perturbations]
\label{Z:K::perturb}
Let $(\ve, \ell_z)\in\Abound$ and $\delta_0>0$. A {\it{$\delta_0$-perturbation of $Z^K(\ve, \ell_z)$}} is a continuous one-parameter family of curves $(Z(\delta, \ve, \ell_z))_{\delta\in[0, \delta_0[}$ such that 
\begin{enumerate}
\item $\displaystyle Z(0, \ve, \ell_z) = Z^K(\ve, \ell_z)$. 
\item $\forall \delta\in[0, \delta_0[$, $Z(\delta, \ve, \ell_z)$ consists of two connected components: $\displaystyle Z^{abs}(\delta, \ve, \ell_z)$ which is diffeomorphic to $\mathbb R$ and $Z^{trapped}(\delta, \ve, \ell_z)$ which is diffeomorphic to $\mathbb S^1$ and such that 
\begin{itemize}
\item there exists $\Phi_{(\ve, \ell_z)}^{\delta, abs}:]-\beta, \beta[\to]0, \infty[$ such that 

\begin{equation*}
Z^{abs}(\delta, \ve, \ell_z) = \text{Gr} \left(\Phi^{\delta, abs}_{(\ve, \ell_z)}\right)
\end{equation*} 
and
\begin{equation*}
\left|\left| \Phi^{\delta, abs}_{(\ve, \ell_z)} -  \Phi^{K, abs}_{(\ve, \ell_z)} \right|\right|_{C^1(]-\beta, \beta[)} < \delta_0. 
\end{equation*}
\item there exist $\displaystyle \Phi^{\delta, i}_{(\ve, \ell_z)}: I^i_{(\ve, \ell_z)}\mapsto \mathbb R,\; i=1\cdots 4$ such that 
\begin{equation*}
Z^{trapped}(\delta, \ve, \ell_z) = \bigcup_{i=1\cdots 4}\text{Gr} \left(\Phi^{\delta, i}_{(\ve, \ell_z)}\right),
\end{equation*} 
and 
\begin{equation*}
\left|\left| \Phi^{\delta, i}_{(\ve, \ell_z)} -  \Phi^{K, i}_{(\ve, \ell_z)} \right|\right|_{C^1\left(\overline I^i_{(\ve, \ell_z)}\right)} < \delta_0. 
\end{equation*}
\end{itemize}
\end{enumerate}
\end{definition}
\noindent Now, we recall the definition of the effective potential energy associated to a timelike future-directed geodesic moving in the exterior region of a stationary and axisymmetric spacetime with metric given by \eqref{metric:ansatz}:
\begin{equation*}
E_{\ell_z}(W, X, \sigma, \rho, z) := \frac{-W(\rho, z)}{X(\rho, z)}\ell_z + \frac{\sigma}{X(\rho, z)}\sqrt{\ell_z^2 + X(\rho, z)}. 
\end{equation*} 
We rewrite $E_{\ell_z}$ in terms of the renormalised unknowns $\displaystyle h:= (\thetazero, \sigmazero, \Xzero)$: 
\begin{equation}
\label{E:ellz:pert}
E_{\ell_z}(h, \rho, z) := -\thetazero\ell_z - W_KX_K^{-1}\ell_z + \frac{\rho}{X_K}\frac{\sigmazero + 1}{\Xzero + 1}\sqrt{\ell_z^2 + X_K(\Xzero+1)}
\end{equation}
Henceforth, $E_{\ell_z}$ is seen as a function defined on $\displaystyle \Ltheta\times\Lsigma\times\LX\times\BB$. Moreover, we recall that the set of solutions to the equation
\begin{equation*}
E_{\ell_z}(h, \rho, z) = \ve
\end{equation*}
is called the zero velocity curve $Z(h, \ve, \ell_z)$. The latter can also be seen as the level curve of $E_{\ell_z}(h, \cdot)$ at $\ve$. 
\noindent In the following, we show that when $E_{\ell_z}(h, \cdot)$ and $E^K_{\ell_z}$ are similar, that is when $h$ is small, then their level sets at a level $\ve$, such that $(\ve, \ell_z)\in \Abound$ have the same shape. More precisely, we state the following main result
\begin{Propo}
\label{Pert:kerr}
Let $\tilde \delta_0>0$. Then there exists $0<\delta_0 \leq \tilde\delta_0$ such that $\forall h := \left( \thetazero, \sigmazero, \Xzero \right)\in B(0, \delta_0)\subset\Ltheta\times\Lsigma\times\LX$, $\forall (\varepsilon, \ell_z)\in \Bbound$, there exists a unique curve $Z(h, \ve, \ell_z)\subset \BB$  solution to the equation 
\begin{equation}
\label{Z:pert}
E_{\ell_z}(h, \rho, z) = \ve
\end{equation}
Moreover, the one-parameter family of solutions $(Z(h, \ve, \ell_z))_{||h||\in[0, \delta_0[}$ is a $\delta_0-$perturbation of $Z^K(\ve, \ell_z)$ in the sense of Definition \ref{Z:K::perturb}. 
\end{Propo} 
\noindent In the case of Kerr, i.e.~ $h=0$, the set of solutions to the equation \eqref{Z:pert} cannot be written globally as the graph of a unique function depending on $\rho$ or on $z$, since the implicit function theorem cannot be applied globally to \eqref{Z:pert} in order to write $\rho$ in terms of $z$ or $z$ in terms of $\rho$. This must be taken into account when defining the space of solutions. In order to overcome this technical difficulty, we have decomposed $\BB$ into several regions and we have used the reparameterization of $Z^K(\ve, \ell_z)$ so that the implicit function theorem can be applied in each region to solve \eqref{Z:pert} with $h = 0$. We will now apply the fixed point theorem on each region in order to solve \eqref{Z:pert} with small $h$. The set of solutions to the latter is the graph of some function. Finally, we will obtain the set of solutions in the whole region $\BB$ by gluing all the graphs that we have obtained. 
\\ 
\\ The remaining of this section is devoted to the proof of Proposition \ref{Pert:kerr}.

\subsection{Proof of Proposition \ref{Pert:kerr}}
Let $\delta_0>0$, let $h\in B_{\delta_0}\subset \Ltheta\times\Lsigma\times\LX$ where $B_{\delta_0}$ is the open ball of radius $\delta_0$ centred around $0$ and let $(\ve, \ell_z)\in\Bbound$. The problem of finding solutions in $\BB$ is equivalent to that of solving the above equation in $\BB_i$ and $\BB^{abs}$. Hence, we start with solving \eqref{Z:pert} on $\BB_1$. We state the following lemma

\begin{lemma}
\label{phi1:zvc}
There exists $\delta_0>0$ such that  $\forall h\in B_{\delta_0}$, there exists a unique function $\Phi^{h, 1}_{(\ve, \ell_z)}: \overline I^1_{(\ve, \ell_z)}\mapsto\mathbb R$ such that the set of solutions to \eqref{Z:pert} on $\BB_1$ is given by $\displaystyle \text{Gr} \left(\Phi^{h, 1}_{(\ve, \ell_z)}\right)$. Moreover, we have 
\begin{itemize}
\item $\Phi^{h,1}$ is smooth with respect to $(\ve, \ell_z)$. 
\item $\Phi^{h,1}_{\ve, \ell_z}$ is continuously Fréchet differentiable with respect to $h$, 
\item $\Phi^{h,1}_{\ve, \ell_z}$ is $C^1$ on $I^1_{(\ve, \ell_z)}$ and satisfies
\begin{equation}
\left|\left|\Phi^{1}_{\ve, \ell_z} - \Phi^{K, 1}_{\ve, \ell_z} \right|\right|_{C^1\left(\overline I^1_{(\ve, \ell_z)}\right)} < \delta_0. 
\end{equation} 
\end{itemize}
\end{lemma}

\begin{proof}
Let $(\ve, \ell_z)\in \Bbound$ and recall the definition of  $\displaystyle \overline I^1_{(\ve, \ell_z)}$. Define the mapping $\Fpert$ by 
\begin{equation*}
\Fpert\left(\thetazero, \sigmazero, \Xzero, z\right) = -\thetazero\ell_z - W_KX_K^{-1}\ell_z + \frac{\rho}{X_K}\frac{\sigmazero + 1}{\Xzero + 1}\sqrt{\ell_z^2 + X_K(\Xzero+1)} - \ve
\end{equation*}
on the domain $\displaystyle B_{\tilde\delta_0}\times \mathbb R$, where $B_{\tilde\delta_0}$ is the ball centred at $(0,0,0)$ with radius $\tilde\delta_0$ of the product space $\Ltheta\times\Lsigma\times\LX$. 
\begin{itemize}
\item \underline{Existence}
\begin{enumerate}
\item It is easy to see that $\forall \rho\in \overline I^1_{(\ve, \ell_z)}, $ the point $(h_0:= (0, 0, 0), z_0:= \Phi^{K, 1}_{\ve, \ell_z}(\rho))$ is a zero for $\displaystyle \Fpert$. 
\item Since $(\thetazero,\sigmazero, \Xzero)\in C^{2}(\BB)$,  $\Fpert$ is continuously differentiable with respect to $z$ on $\mathbb R$. 
\item We compute:
\begin{equation*}
\frac{\partial \Fpert}{\partial z}(h_0, z_0) = \frac{\partial E_{\ell_z}}{\partial z}(h_0, \rho, z_0).
\end{equation*}
The latter vanishes if and only if $z_0 = 0$, which is not the case. Therefore, 
\begin{equation*}
\phi(\ve, \ell_z,  \rho) := \left(\frac{\partial\Fpert}{\partial z}(h_0, z_0) \right)^{-1}
\end{equation*}
is well-defined.  
\item We consider the mapping $\Fpert_h$ defined on $\displaystyle \mathbb R$ by 
\begin{equation*}
\Fpert_h(z) :=   z - \phi(\ve, \ell_z, \rho)\Fpert(h, z).
\end{equation*}

We will show that after shrinking $\tilde\delta_0$ uniformly in $(\ve, \ell_z)$ and $\rho$, that $\forall h\in B(h_0, \delta_0)$, $\Fpert_h$ is a contraction on $\overline B(z_0, \delta_0)$. 
\\ First, by compactness of $\Bbound$ and $I^1_{(\ve, \ell_z)}$, there exists $C>0$ such that $\forall(\ve, \ell_z)\in \Bbound$, $\forall \rho\in I^1_{(\ve, \ell_z)}$ :
\begin{equation*}
\left|\phi(\ve, \ell_z, \rho)\right| \leq C. 
\end{equation*}
Now, let $h\in B(h_0, \tilde\delta_0)$ and $z\in B(z_0, \tilde\delta_0)$. We compute
\begin{equation*}
\frac{\partial\Fpert_h}{\partial z}(z) = 1 -  \phi(\ve, \ell_z, \rho)\frac{\partial E_{\ell_z}}{\partial z}(h, \rho, z). 
\end{equation*}
Next, we show that there exists $C$ independent of $((\ve, \ell_z); \rho)$ such that $\forall (h, z)\in B_{\tilde\delta_0}(h_0)\times B_{\tilde\delta_0}(z_0)$ 
\begin{equation*}
\left| \partial_z \Fpert(h, z) -  \partial_z \Fpert(h_0, z_0)\right| \leq C\tilde\delta_0. 
\end{equation*}
We have 
\begin{align*}
\partial_z \Fpert(h, z) -  \partial_z \Fpert(h_0, z_0) &=   \partial_z E_{\ell_z}(\rho, z) -  \partial_z E^K_{\ell_z}(\rho, z_0)  \\
&-\partial_z \thetazero\rz\ell_z - \partial_z (W_KX_K^{-1})\rz\ell_z  \\
&+ \partial_z\left( \frac{\rho(\sigmazero + 1)}{X_K(\Xzero + 1)}\sqrt{\ell_z^2 + X_K(\Xzero + 1)}\right)\rz   \\
&+ \partial_z (W_KX_K^{-1})(\rho, z_0)\ell_z - \partial_z\left( \frac{\rho}{X_K}\sqrt{\ell_z^2 + X_K}\right)(\rho, z_0)\\
&=  I + II + III, 
\end{align*}

where 
\begin{equation*}
\begin{aligned}
I &:= -\partial_z \thetazero\rz\ell_z,  \\
II &:=  \ell_z\left( - (W_KX_K^{-1})(\rho, z) + (W_KX_K^{-1})(\rho, z_0) \right), \\
III &:= \partial_z\left( \frac{\rho(\sigmazero + 1)}{X_K(\Xzero + 1)}\sqrt{\ell_z^2 + X_K(\Xzero + 1)}\right)\rz -  \partial_z\left( \frac{\rho}{X_K}\sqrt{\ell_z^2 + X_K}\right)(\rho, z_0).
\end{aligned}
\end{equation*}
Since , $\ell_z\in [\ell_1, \ell_2]$ and $\thetazero\in B_{\tilde\delta_0}(\Ltheta)$, there exists $C>0$ uniform in $\rho$, $(\ve, \ell_z)$ such that 
\begin{equation*}
|I| \leq C\tilde\delta_0. 
\end{equation*}
Concerning the second term, we write: 
\begin{align*}
|II| &\leq \ell_2 \left|(W_KX_K^{-1})\rz - (W_KX_K^{-1})(\rho, z_0) \right|   \\
& \leq \ell_2\left( \sup_{\overline z \in(z, z_0)} \left|\partial{z}(W_KX_K^{-1})(\rho, \overline z)\right| \right)\left| z - z_0 \right|. 
\end{align*}
By \ref{extendibility},  $\displaystyle W_KX_K^{-1}(\rho, z) := \frac{2dr\rz}{\Pi\rz}$ is smooth and bounded (with its derivatives) on $\BB$. Therefore, there exists $C>0$ independent of $(\rho, z)$  and  $(\ve, \ell_z)$ so that
\begin{equation*}
|II| \leq C\tilde\delta_0. 
\end{equation*}
As for $III$, we set 
\begin{equation*}
A_{\ell_z}((\rho, z), (x, y)):= \frac{\rho}{X_K}\frac{x + 1}{y + 1}\sqrt{\ell_z^2 + X_K(y + 1)} \quad\text{where}\quad (x, y)\in[-\overline\delta_0, \overline\delta_0]^2. 
\end{equation*}
$A_{\ell_z}$ is smooth on $\BB_1\times]-\overline\delta_0, \overline\delta_0[^2 $ and we have 
\begin{equation*}
\nabla_{(\rho, z)}A_{\ell_z}((\rho, z), (x, y)) = \frac{x + 1}{y + 1}\left( \nabla_{(\rho, z)}\left(\frac{\rho}{X_K}\right)\sqrt{\ell_z^2 + X_K(y + 1)} + \frac{\rho(x + 1)}{X_K}\frac{\nabla_{(\rho, z)}X_K}{\sqrt{\ell_z^2 + X_K(y + 1)}}\right) 
\end{equation*}
and 
\begin{equation*}
\nabla_{(x, y)}A_{\ell_z}((\rho, z), (x, y)) = \frac{\rho}{X_K}\left( 
\begin{aligned}
& \frac{1}{y + 1}\sqrt{\ell_z^2 + X_K(y + 1)} \\
& -\frac{(x + 1)}{(y + 1)^2}\sqrt{\ell_z^2 + X_K(y + 1)} + \frac{x + 1}{y + 1}\frac{X_K}{2\sqrt{\ell_z^2 + X_K(y + 1)}}
\end{aligned}
\right).
\end{equation*}
The functions $\displaystyle (\rho, z)\to\frac{\rho}{X_K}$ and $\displaystyle (\rho, z)\to X_K(\rho, z)$ are  bounded on $\BB_1 = \BB_1(\ve, \ell_z)$ and by compactness of $\Bbound$, the latter are bounded independently from $(\ve, \ell_z)$. 
\\ Therefore, there exists $C>0$ uniform in $(\ve, \ell_z)$ such that $\forall (\rho, z)\in \BB_1\,, \, \forall (x, y)\in ]-\delta_0, \delta_0[^2$
\begin{equation*}
\left|\partial_z A_{\ell_z}((\rho, z), (x, y))\right| \leq C\overline\delta_0. 
\end{equation*} 
\noindent Now, we have 
\begin{equation}
\label{III:estimates}
\begin{aligned}
III &=   \partial_z\left(A_{\ell_z}((\rho, z), (\sigmazero(\rho, z), \Xzero(\rho, z)))\right) -  \partial_z\left( A_{\ell_z}((\rho, z), (0, 0))\right)(\rho, z_0) \\
&= \partial_z A_{\ell_z}((\rho, z), (\sigmazero(\rho, z), \Xzero(\rho, z))) +  \partial_z\sigmazero\partial_xA_{\ell_z}((\rho, z), (\sigmazero(\rho, z), \Xzero(\rho, z))) \\
&+ \partial_z\Xzero\partial_yA_{\ell_z}((\rho, z), (\sigmazero(\rho, z), \Xzero(\rho, z)))  -  \partial_zA_{\ell_z}((\rho, z_0), (0, 0))
\end{aligned}
\end{equation}
\noindent Since $\sigmazero\in B_{\overline \delta_0}\subset \LX$ and $\Xzero\in B_{\overline \delta_0} \subset \LX$, the functions $\Xzero$ and $\sigma_0$ are bounded (with their derivatives) on $\Bbarre$. In particular, we have
\begin{equation*}
\forall (\rho, z)\in I^1_{(\ve, \ell_z)}\;,\; |\sigmazero(\rho, z)|\,,\, |\Xzero(\rho, z)|\,,\, \left|\partial_z\sigmazero(\rho, z)\right| \,,\, \left|\partial_z\Xzero(\rho, z)\right| \leq \overline\delta_0.
\end{equation*}
Moreover, 
\begin{equation*}
\partial_z A_{\ell_z}((\rho, z), (\sigmazero(\rho, z), \Xzero(\rho, z))) -   \partial_zA_{\ell_z}((\rho, z_0), (0, 0)) = \text{(a)} + \text{(b)} 
\end{equation*}
where 
\begin{equation*}
\text{(a)} := \partial_z A_{\ell_z}((\rho, z), (\sigmazero(\rho, z), \Xzero(\rho, z))) - \partial_z A_{\ell_z}((\rho, z_0), (\sigmazero(\rho, z_0), \Xzero(\rho, z_0)))
\end{equation*}
and 
\begin{equation*}
\text{(b)} := \partial_z A_{\ell_z}((\rho, z_0), (\sigmazero(\rho, z_0), \Xzero(\rho, z_0)))  - \partial_zA_{\ell_z}((\rho, z_0), (0, 0))
\end{equation*}
We have $\forall z\in B(z_0, \overline\delta_0)$, 
\begin{align*}
\text{(a)} &:= (z - z_0)\partial_{z}\left(\sigmazero(\rho, z)\; \Xzero(\rho, z)\right)\cdot\nabla_{(x, y)}\partial_zA_{\ell_z}\left( (\rho, z_0), (\sigmazero(\rho, z_0), \Xzero(\rho, z_0))\right)  \\ 
&+ (z - z_0)\partial_{zz} A_{\ell_z}\left( (\rho, z_0), (\sigmazero(\rho, z_0), \Xzero(\rho, z_0))\right) + O((\overline \delta_0)^2)
\end{align*}
By similar arguments to the first estimates, we can find $C>0$ uniform in $(\ve, \ell_z)$ such that $\forall (\rho, z) $ we have 
\begin{equation*}
|\text{(a)}| \leq C\overline\delta_0. 
\end{equation*}
In order to estimate (b), we write 
\begin{equation*}
\text{(b)} = \left(\sigmazero(\rho, z_0)\; \Xzero(\rho, z_0)\right)\cdot \nabla_{(x, y)}\partial_zA_{\ell_z}((\rho, z_0), (0, 0)) + O((\overline \delta_0)^2). 
\end{equation*}
Again by similar arguments as above, we can find $C>0$ uniform in $(\ve, \ell_z)$ such that $\forall (\rho, z) $ we have 
\begin{equation*}
|\text{(b)}| \leq C\overline\delta_0. 
\end{equation*}
Finally, after shrinking $\overline\delta_0$, there exists $C>0$ such that 
\begin{equation*}
| III | \leq C\delta_0 
\end{equation*}
Hence, we choose $\tilde\delta_0$ so that 
\begin{equation*}
\left| \partial_z \Fpert_h(z) -  \partial_z \Fpert_{h_0}(z_0)\right| \leq \frac{1}{2}. 
\end{equation*}
Since, 
\begin{equation}
\partial_z \Fpert_{h_0}(z_0) = 0,  
\end{equation}
we are left with a bound on the derivative of $\Fpert_h$ with respect to $z$:  
\begin{equation}
\label{deriv:1:2}
\left| \partial_z \Fpert_h(z) \right| \leq \frac{1}{2}. 
\end{equation}

\item We claim that there exists $\delta_0\leq \overline \delta_0$ such that 
\begin{equation*}
\forall z\in B\left(\Phi^{K, 1}_{(\ve, \ell_z)}(\rho), \overline\delta_0\right)\;:\quad \left|  \Fpert_h(z) -  \Phi^{K, 1}_{(\ve, \ell_z)}(\rho)\right| \leq\frac{1}{2}. 
\end{equation*}
Indeed, $\displaystyle\forall (\ve, \ell_z)\in \Bbound\;,\;\displaystyle\forall \rho\in I^1_{(\ve, \ell_z)}\;,\; \forall h\in B(h_0, \overline \delta_0)$ and $\displaystyle \forall z\in B(\Phi^{K, 1}_{(\ve, \ell_z)}(\rho), \overline\delta_0)$
\begin{equation*}
\begin{aligned}
\Fpert_h(z) -  \Phi^{K, 1}_{(\ve, \ell_z)}(\rho) &= \Fpert_h(z) -  \Fpert_{h_0}(z_0 = \Phi^{K, 1}_{(\ve, \ell_z)}(\rho)) \\
&=   \Fpert_h(z) -  \Fpert_{h}(z_0) +  \Fpert_h(z_0) -  \Fpert_{h_0}(z_0)
\end{aligned}
\end{equation*}
By the mean value theorem, 
\begin{equation*}
\left|\Fpert_h(z) -  \Fpert_h(z_0)\right| \leq \frac{1}{2}\left|z - z_0 \right| \leq \frac{1}{2}\overline\delta_0.  
\end{equation*}
Moreover, by similar estimates to those performed in \eqref{III:estimates}, there exists $C>0$ uniform in $(\ve, \ell_z)$ and $z$ such that 
\begin{equation*}
\left|\Fpert_h(z_0) - \Fpert_{h_0}(z_0) \right| \leq C\overline \delta_0. 
\end{equation*}
Therefore, we choose $\delta_0\leq \overline\delta_0$ so that 
\begin{equation*}
\left|\Fpert_h(z) - \Fpert_{h_0}(z_0)\right| \leq \delta_0. 
\end{equation*}
By \eqref{deriv:1:2},  $\displaystyle \Fpert_h(\rho)$ is a contraction on $\overline B(\Phi^{K, 1}_{(\ve, \ell_z)}(z), \delta_0)$. 

\item By the fixed point theorem, there exists $\delta_0\leq\tilde\delta_0$ such that  $\displaystyle \forall (\ve, \ell_z)\in \Bbound\;,\;  \forall \rho\in I^1_{(\ve, \ell_z)}\;,\; $ there exists a mapping  $\Phi^{\cdot., 1}_{(\ve, \ell_z)}(\rho) : B(0, \delta_0)\mapsto B(\Phi^{K, 1}_{(\ve, \ell_z)}(\rho), \delta_0)$ such that $\displaystyle \Phi^{h, 1}_{(\ve, \ell_z)}(\rho)\in B_{\delta_0}(\Ltheta\times\Lsigma\times\LX)$ and  satisfies  
\begin{equation*}
\Fpert(h, \Phi^{h, 1}_{(\ve, \ell_z)}(\rho)) = 0. 
\end{equation*}
\end{enumerate}

\item \underline{Regularity:} 

\begin{itemize}

\item \underline{Regularity with respect to $h$:} we prove that $h\to\Phi^{h, 1}$ is $C^1$ with respect to $h$. For this, let $(\ve, \ell_z)\in \Bbound$ and $\rho \in I_{(\ve, \ell_z)}^1$. In order to lighten the expressions, we will not write the dependence of $\Phi^{h, 1}$ on $((\ve, \ell_z); \rho)$ and also sometimes write $\Phi^{h, 1}$ as $\Phi^{1}(h)$. 
\\First, we show that $\Phi^1$ is Lipschitz. For this, let $h, \overline h\in B_{\delta_0}$ and set
\begin{equation*}
z = \Phi^1(h), \quad\quad \overline z =  \Phi^1(\overline h). 
\end{equation*}
We have
\begin{align*}
z - \overline z &= \Phi^1(h) - \Phi^1(\overline h) \\
&= F_h(z) - F_{\overline h}(\overline z) \\
&= F_h(z) -  F_{h}(\overline z)  + F_h(\overline z) - F_{\overline h}(\overline z) \\
&= F_h(z) -  F_{h}(\overline z) + \phi(\ve, \ell_z, \rho)\left( F^{(\ve, \ell_z), z}(h, \overline z) - F^{\ve, \ell_z, z}(\overline h, \overline z)\right). 
\end{align*}
We have 
\begin{equation*}
\left|  F_h(z) -  F_{h}(\overline z) \right| \leq \frac{1}{2}\left| z - \overline z\right| 
\end{equation*}
and
\begin{align*}
\left| \phi(\ve, \ell_z, \rho)\left( F^{(\ve, \ell_z), \rho}(h, \overline z) - F^{\ve, \ell_z, \rho}(\overline h, \overline z)  \right)\right|\leq C\left|\left| h  - \overline h\right|\right|_{\mathcal L_\Theta \times \mathcal L_\sigma \times \mathcal L_X}.
\end{align*}
Therefore,
\begin{equation*}
|z - \overline z| \leq 2C \left|\left|h - \overline h \right|\right|_{\mathcal L_\Theta \times \mathcal L_\sigma \times \mathcal L_X}. 
\end{equation*}
Thus, $\Phi^1$ is Lipschitz, so continuous on $B_{\delta_0}$. Since 
\begin{equation*}
\forall \Phi^1(\rho)\in B(\Phi^{K, 1}(\rho), \delta_0), \;\quad \left| \frac{\partial F_{h}^{\ve, \ell_z, \rho}}{\partial z}(z) \right|\leq \frac{1}{2}, 
\end{equation*}
Hence, 
\begin{equation*}
\forall z\in B(\Phi^{K, 1}(\rho), \delta_0),\,\forall h\in B_{\delta_0}, \quad\quad  \frac{\partial F^{(\ve, \ell_z), \rho}}{\partial z}(h, z) \neq 0. 
\end{equation*}

Since $F^{(\ve, \ell_z), \rho}$ is differentiable at $(\overline h, \overline z)$, we have
\begin{align*}
0 &= F^{(\ve, \ell_z), \rho}(h, z) - F^{(\ve, \ell_z), \rho}(\overline h, \overline z)  \\
&= D_h F^{(\ve, \ell_z), \rho}(\overline h, \overline z)\cdot(h - \overline h) + \partial_z F^{(\ve, \ell_z), \rho}(\overline h, \overline z)(z - \overline z) +  o\left(||h - \overline h||_{\mathcal L_\Theta \times \mathcal L_\sigma \times \mathcal L_X} + |z - \overline z|\right).  
\end{align*}
By the above estimates we have 
\begin{equation*}
o(|z - \overline z| = O (||h - \overline h||_{\mathcal L_\Theta \times \mathcal L_\sigma \times \mathcal L_X}).
\end{equation*}
Therefore, 
\begin{equation*}
 z - \overline z = - \left(\partial_\rho F^{(\ve, \ell_z), \rho}(\overline h, \overline z)\right)^{-1}D_h F^{(\ve, \ell_z), \rho}(\overline h, \overline z)\cdot(h - \overline h) + o (||h - \overline h||).
 \end{equation*}
 
 \item \underline{Regularity with respect to $z$:} Let $(\ve, \ell_z)\in \Bbound$ and let $h\in B_{\delta_0}(\Ltheta\times\Lsigma\times\LX)$.
 \\ First we show that $\Phi^{h, 1}_{(\ve, \ell_z)}$ is Lipschitz on $I^1_{(\ve, \ell_z)}$. In order to lighten the expressions, we drop the dependence in $(\ve, \ell_z)$ and $h$ so that $\Phi^{h, 1}_{(\ve, \ell_z)}$ will be denoted by $\Phi^{1}$. Similarly, $\Phi^{K, 1}_{(\ve, \ell_z)}$ will be denoted by $\Phi^{K, 1}$.  Now, let $\rho_1, \rho_2\in I^1_{(\ve, \ell_z)}$. 
\\We will express $\displaystyle \Phi^{1}(\rho_1) - \Phi^{1}(\rho_2)$ in terms of $\rho_1 - \rho_2$. To this end, we write, 
\begin{align*}
z_1 - z_2 &:= \Phi^1(\rho_1) - \Phi^1(\rho_2) \\
&= F^{\rho_1}(z_1) - F^{\rho_2}(z_2) \\
&= F^{\rho_1}(z_1) -  F^{\rho_1}(z_2)  + F^{\rho_1}(z_2) - F^{\rho_2}(z_2) \\
&= F^{\rho_1}(z_1) -  F^{\rho_1}(z_2) + \phi(\rho_1)\left( F^{\rho_1}(z_2) - F^{\rho_2}(z_2)\right) + \left(\phi(\rho_1) - \phi(\rho_2) \right)F^{\rho_2}(z_2). 
\end{align*}
Again, we have 
 \begin{equation*}
\left|  F^{\rho_1}(z_1) -  F^{\rho_1}(z_2) \right| \leq \frac{1}{2}\left| z_1 - z_2\right|. 
\end{equation*}
 We have
 \begin{equation*}
\phi(z_i) = \left(\frac{\partial F^{\rho_i}}{\partial z}(0, \Phi^{K, 1}(\rho_i)) \right)^{-1} = \left(\frac{\partial E_{\ell_z}^K}{\partial z}(\rho_i, \Phi^{K, 1}(\rho_i)) \right)^{-1}. 
\end{equation*}
Therefore, 
\begin{align*}
\phi(\rho_1) - \phi(\rho_2) &= \left(\frac{\partial E_{\ell_z}^K}{\partial z}(\rho_1, \Phi^{K, 1}(\rho_1)) \right)^{-1} - \left(\frac{\partial E_{\ell_z}^K}{\partial z}(\rho_2, \Phi^{K, 1}(\rho_2)) \right)^{-1} \\
&=  \left(\frac{\partial E_{\ell_z}^K}{\partial z}(\rho_1, \Phi^{K, 1}(\rho_1))\frac{\partial E_{\ell_z}^K}{\partial z}(\rho_2, \Phi^{K, 1}(\rho_2)) \right)^{-1}  \\
&\left(\frac{\partial E_{\ell_z}^K}{\partial z}(\rho_2, \Phi^{K, 1}(\rho_2))- \frac{\partial E_{\ell_z}^K}{\partial z}(\rho_1, \Phi^{K, 1}(\rho_1)) \right). 
\end{align*} 
\begin{align*}
\left| \frac{\partial E_{\ell_z}^K}{\partial z}(\rho_1, \Phi^{K, 1}(\rho_1)) - \frac{\partial E_{\ell_z}^K}{\partial z}(\rho_2, \Phi^{K, 1}(\rho_2)) \right| &\leq C \sup_{\BB_1}\left|\right|\nabla^2E_{\ell_z}(\rho, z)\left|\right|\, |\rho_1 - \rho_2| \\ 
&\leq C|\rho_1 - \rho_2|, 
\end{align*}
where $C$ is independent of $(\ve, \ell_z)$. 
We also have 
\begin{equation*}
 \left(\frac{\partial E_{\ell_z}^K}{\partial z}(\rho_1, \Phi^{K, 1}(\rho_1))\frac{\partial E_{\ell_z}^K}{\partial z}(\rho_2, \Phi^{K, 1}(\rho_2)) \right)^{-1} \leq C, 
\end{equation*}
for some $C$ uniform in $(\ve, \ell_z)$. Moreover, 
\begin{align*}
F^{\rho_1}(z_2) - F^{\rho_2}(z_2) &= E_{\ell_z}(h, \rho_1, z_2) - E_{\ell_z}(h, \rho_2, z_2). 
\end{align*}
Since $E_{\ell_z}(h, ,\cdot, z_2)$ is differentiable on $I^1_{(\ve, \ell_z)}$, we have 
\begin{align*}
\left| F^{\rho_1}(z_2) - F^{\rho_2}(z_2) \right| &\leq \sup_{\BB_1}\left|\nabla_{\rho, z}E_{\ell_z}(\rho, z)\right|\left| \rho_1 - \rho_2 \right| \\
&\leq C\left| \rho_1 - \rho_2 \right|,
\end{align*}
where $C>0$ is some constant independent of  $(\ve, \ell_z)$. Therefore, there exists a constant independent of $(\ve, \ell_z)$ and $h$ such that 
\begin{equation}
\label{lipschitz:phi}
\left|\Phi^1(\rho_1) - \Phi^1(\rho_2)\right| \leq C|\rho_1 - \rho_2|. 
\end{equation}
Hence, $\Phi^1$ is Lipschitz on $I^1_{(\ve, \ell_z)}$. It remains to show that $\Phi^1$ is continuously differentiable on $I^1(\ve, \ell_z)$. 
\\ Now, recall that 
\begin{equation*}
\forall \rho\in I^1_{(\ve, \ell_z)} \in \forall z\in B(\Phi^{K, 1}(\rho), \delta_0),\,\forall h\in B_{\delta_0}, \quad\quad  \frac{\partial F^{(\ve, \ell_z), \rho}}{\partial z}(h, z) \neq 0. 
\end{equation*}
Since, $(\rho, z)\to F^\rho(z)$ is differentiable on $\BB_1$, we have 
\begin{align*}
0 &= F^{(\ve, \ell_z), \rho_1}(h, z_1) - F^{(\ve, \ell_z), \rho_2}(h, z_2)   \\
&= F^{(\ve, \ell_z), \rho_1}(h, z_1) - F^{(\ve, \ell_z), \rho_2}(h, z_1) + F^{(\ve, \ell_z), \rho_2}(h, z_1) - F^{(\ve, \ell_z), \rho_2}(h, z_2) \\
&= \partial_\rho F^{(\ve, \ell_z), \rho_1}(h, z_1)(\rho_1 - \rho_2) + \partial_z F^{(\ve, \ell_z), \rho_2}(h, z_2)(z_1 - z_2) +  o\left(|\rho_1- \rho_2| + |z_1 - z_2|\right).  
\end{align*}
By \eqref{lipschitz:phi}, we have 
\begin{equation*}
o(|z_1 - z_2|) =  O(|\rho_1 - \rho_2|). 
\end{equation*}
Hence,
\begin{align*}
\Phi^1(\rho_1) - \Phi^1(\rho_2) &= -\partial_z F^{(\ve, \ell_z), \rho_2}(h, z_2)^{-1}\partial_\rho F^{(\ve, \ell_z), \rho_1}(h, z_1)(\rho_1 - \rho_2) +  O(|\rho_1 - \rho_2|) \\
&= -\frac{\partial_\rho E_{\ell_z}(h, \rho_1, \Phi^1(\rho_2))}{\partial_z E_{\ell_z}(h, \rho_2, \Phi^1(\rho_2))}(\rho_1 - \rho_2) +  O\left(|\rho_1 - \rho_2|\right). 
\end{align*} 

\item \underline{Regularity with respect to $(\ve, \ell_z)$: } We claim that $\Phi^1$ is continuously differentiable on $\Bbound$.  The regularity in this case is proven in the the same manner. 

\end{itemize}

\item \underline{Uniqueness:} We show, after possibly shrinking $\delta_0$, that $\forall h\in B_{\delta_0}\,,\forall (\ve, \ell_z)\in \Bbound\,,\,  \forall \rho\in I^1_{(\ve, \ell_z)} \;$ there are no others solutions to the equation \eqref{Z:pert} in $[\overline z_{max}(\ve, \ell_z), \infty[$. 
First, we set 
\begin{equation*}
\forall z\in [\overline z_{max}(\ve, \ell_z), z_{max}(\ve, \ell_z)]\;,\quad P^{h, \rho}_{\ve, \ell_z}(z) := E_{\ell_z}(h, \rho, z) - \ve.
\end{equation*}
and 
\begin{equation*}
\forall z\in [\overline z_{max}(\ve, \ell_z), z_{max}(\ve, \ell_z)];,\quad P^{K, \rho}_{\ve, \ell_z}(z) := E^K_{\ell_z}\rz - \ve.
\end{equation*}
We claim that for $\delta_0$ sufficiently small, there exists $\displaystyle C>0$ such that $\displaystyle \forall (\ve, \ell_z)\in \Bbound\,,   \forall \rho\in I^1_{(\ve, \ell_z)} \,$, $z \in J :=  [\overline z_{max}(\ve, \ell_z), \infty[ \backslash B(\Phi^{K, 1}_{\ve, \ell_z}(\rho), \delta_0)$ we have  
\begin{equation*}
\left| P^{K, \rho}_{\ve, \ell_z}(\rho)\right| > C\delta_0. 
\end{equation*} 
$P^{K, \rho}_{\ve, \ell_z}$ is monotonically increasing on $]\overline z_{max}(\ve, \ell_z), \infty[$. Therefore,  $\displaystyle \forall (\ve, \ell_z)\in \Bbound\,$,   $\displaystyle \forall \rho\in I^1_{(\ve, \ell_z)} \,$, $\displaystyle \forall z\in J $ \;:\; 
\begin{equation*}
\left| P^{K, \rho}_{\ve, \ell_z}(z) \right| > \max\left\{\left|P^{K, \rho}_{\ve, \ell_z}(\Phi^{K, 1}_{\ve, \ell_z} (\rho) - \delta_0)\right|, \left|P^{K, \rho}_{\ve, \ell_z}(\Phi^{K, 1}_{\ve, \ell_z}(\rho) + \delta_0)\right|\right\}.
\end{equation*}
Now, we show that for all $|h|$ small, there exist $C(h)> 0$ uniform in $(\ve, \ell_z)$ such that 
\begin{equation*}
|P^{K, \rho}_{\ve, \ell_z}(\Phi^{K, 1}_{\ve, \ell_z}(\rho) + h)| > C(h). 
\end{equation*}
We have 
\begin{equation*}
P^{K, \rho}_{\ve, \ell_z}(\Phi^{K, 1}_{\ve, \ell_z}(\rho)+ h) = P^{K, \rho}_{\ve, \ell_z}(\Phi^{K, 1}_{\ve, \ell_z}(\rho)) + h(P^{K, \rho}_{\ve, \ell_z})'(\Phi^{K, 1}_{\ve, \ell_z}(\rho)) + o(h),
\end{equation*}
where $o(h)$ is uniform in $(\ve, \ell_z)$ by continuity of $\displaystyle (P^{K, \rho}_{\ve, \ell_z})^{(k)}$ with respect to $((\ve, \ell_z); \rho)$, compactness of $\Bbound$ and the fact that $z\in I^1_{(\ve, \ell_z)}\subset K$ where $K$ is some compact independent from $(\ve, \ell_z)$. Moreover,  $\Phi^{K, 1}_{\ve, \ell_z}(\rho)$ is a simple root in the sense that  $(P^{K, \rho}_{\ve, \ell_z})'(\Phi^{K, 1}_{\ve, \ell_z}(\rho)) = \frac{\partial E_{\ell_z}}{\partial z}(\rho, \Phi^{K, 1}_{(\ve, \ell_z)}(\rho)) \neq 0$ and $\displaystyle (\ve, \ell_z, \rho)\mapsto (P^{K, \rho}_{\ve, \ell_z})'(\Phi^{K, 1}_{\ve, \ell_z, d}(\rho)) $ is continuous on $\Bbound\times K $. Hence, for $h$ sufficiently small, we obtain
\begin{equation*}
\left|P^{K, \rho}_{\ve, \ell_z}\left(\Phi^{K, 1}_{\ve, \ell_z}(\rho)+ h\right)\right| > |h|\left|(P^{K, \rho}_{\ve, \ell_z})'(\Phi^{K, 1}_{\ve, \ell_z}(\rho))\right|> C|h|, 
\end{equation*}
where $C$ is some constant which is uniform in $\delta_0$ and $(\ve, \ell_z)$. Therefore, we update $\delta_0$ so that 
\begin{equation*}
\left| P^{K, \rho}_{\ve, \ell_z}(\Phi^{K, 1}_{\ve, \ell_z}(z) \pm \delta_0) \right| > C\delta_0.
\end{equation*}  
Now, let $\delta_1 <\delta_0$.  By uniqueness in the fixed point theorem,  $\displaystyle  \forall h\in B_{\delta_1}\subset B_{\delta_0} $,  $\displaystyle \forall (\ve, \ell_z)\in \Bbound, $ $\displaystyle \forall \rho\in I^1_{(\ve, \ell_z)} \,$    $\; \Phi^{h, 1}_{\ve, \ell_z}(\rho)$ is the unique solution in the ball $B(\Phi^{K, 1}_{\ve, \ell_z}(\rho), \delta_0)$.  Moreover, $\forall z \in J$, we have 
\begin{equation*}
P^{h, \rho}_{\ve, \ell_z}(z) = P^{h, \rho}_{\ve, \ell_z}(z) - P^{K, \rho}_{\ve, \ell_z}(z) + P^{K, \rho}_{\ve, \ell_z}(z).
\end{equation*}
By the triangular inequality, the latter implies for $\left| P^{h, \rho}_{\ve, \ell_z}(z) - P^{K, \rho}_{\ve, \ell_z}(z) \right| < \tilde C\delta_1 < C\delta_0$ that 
\begin{equation*}
\forall z\in J, \quad\quad |P^{h, \rho}_{\ve, \ell_z}(z)| > -\tilde C\delta_1 +  C\delta_0 > 0.
\end{equation*}
Therefore, after updating $\delta_1>0$, $P^{h, \rho}_{\ve, \ell}$ does not vanish outside the ball $B(\Phi^{K, 1}_{\ve, \ell_z}(\rho) , \delta_1)$.  This yields the uniqueness. 

\end{itemize}
Thus, we update $\delta_0$ in order to obtain that  that $\forall h\in B_{\delta_0}\,,\forall (\ve, \ell_z)\in \Bbound\,,\,  \forall \rho\in I^1_{(\ve, \ell_z)} \;$ there exist a unique solution to the equation \eqref{Z:pert} in  the region $[\overline z_{max}(\ve, \ell_z), \infty[$. 
\end{proof}
In the same way, we prove the following lemmas: 

\begin{lemma}
\label{phii:zvc}
Let $i= 2,\cdots, 4$. There exists $ \delta_0>0$ such that  $\forall h\in B_{\delta_0}\subset \mathcal L_\Theta \times \mathcal L_\sigma \times \mathcal L_X$, there exists a unique $\Phi^{h, i}_{(\ve, \ell_z)}: I^i_{(\ve, \ell_z)}\mapsto\mathbb R$ such that the set of solutions to \eqref{Z:pert} on $\BB_i$ is given by $\displaystyle \text{Gr} \left(\Phi^{h, i}_{(\ve, \ell_z)}\right)$. Moreover, we have 
\begin{itemize}
\item $\Phi^{h, i}$ is smooth with respect to $(\ve, \ell_z)$. 
\item $\Phi^{h, i}_{\ve, \ell_z}$ is continuously Fréchet differentiable with respect to $h$, 
\item $\Phi^{h, i}_{\ve, \ell_z}$ is $C^1$ on $I^i_{(\ve, \ell_z)}$ and satisfies
\begin{equation}
\left|\left|\Phi^{i}_{\ve, \ell_z} - \Phi^{K, i}_{\ve, \ell_z} \right|\right|_{C^1\left(\overline I^i_{(\ve, \ell_z)}\right)} < \delta_0. 
\end{equation} 
\end{itemize}
\end{lemma}

\begin{lemma}
\label{phi:abs}
There exists $\delta_0>0$ such that $\forall (\ve, \ell_z)\in \Bbound$, $\forall h\in B_{\delta_0} \subset \mathcal L_\Theta \times \mathcal L_\sigma \times \mathcal L_X$, $\forall K\subset\subset]-\beta, \beta[$, $\forall z\in K$, there exists a unique $\Phi^{h, abs}_{\ve, \ell_z}(z)\in B_{\delta_0}(\Phi^{K, abs}_{\ve, \ell_z}(z))$ which solves \eqref{Z:pert}. Moreover, we have 
\begin{itemize}
\item $\Phi^{h, abs}$ is smooth respect to $(\ve, \ell_z)$. 
\item $\Phi^{h, abs}$ is continuously Fréchet differentiable with respect to $h$, 
\item $\Phi^{h, abs}$ has the same regularity as $h$ seen as functions of $z$ on $K$. 
\end{itemize}
\end{lemma}

\begin{remark}
We could have applied the implicit function theorem to show the existence of solutions in a neighbourhood of each point of $Z^K$, but the neighbourhood will a priori depend on the point and on $(\ve, \ell_z)$. Hence, we used the fixed point theorem instead to obtain a uniform $\delta_0$, using the compactness of $\Bbound$. 
\end{remark}
\noindent Now, we choose $\delta_0$ so that Lemma \ref{phi1:zvc}, \ref{phii:zvc} and \ref{phi:abs} are satisfied.  It remains to show that $\forall(\ve, \ell_z)\in\Bbound$, $\forall h\in B_{\delta_0}$, $Z(h, \ve, \ell_z)$ consists of two connected components: $Z^{abs}(h, \ve, \ell_z)$and $Z^{trapped}(h, \ve, \ell_z)$ which is diffeomorphic to $\mathbb S^1$. 
Now, we prove Proposition \ref{Pert:kerr}
\begin{proof}
First, we apply lemmas \ref{phi1:zvc}, \ref{phii:zvc} and \ref{phi:abs}: $\forall (\ve, \ell_z)\in \Bbound\;,\;\forall h\in B_{\delta_0}\subset\Ltheta\times\Lsigma\times\LX$ there exists a unique set of solutions $Z^i(h, \ve, \ell_z)$ in the region $\BB_i$ given by
\begin{itemize}
\item on every compact of $\BB_{abs}$:
\begin{equation*}
Z^{abs}(h, \ve, \ell_z) = Graph\left(\Phi^{abs}_{(\ve, \ell_z)} \right),
\end{equation*}
\item on $\BB_1$: 
\begin{equation*}
Z^1(h, \ve, \ell_z) = Graph\left(\Phi^{h, 1}_{(\ve, \ell_z)} \right),
\end{equation*}
\item on $\BB_2$: 
\begin{equation*}
Z^2(h, \ve, \ell_z) = Graph\left(\Phi^{h, 2}_{(\ve, \ell_z)} \right),
\end{equation*}
\item on $\BB_3$: 
\begin{equation*}
Z^3(h, \ve, \ell_z) = Graph\left(\Phi^{h, 3}_{(\ve, \ell_z)} \right),
\end{equation*}
\item on $\BB_4$: 
\begin{equation*}
Z^4(h, \ve, \ell_z) = Graph\left(\Phi^{h, 4}_{(\ve, \ell_z)} \right). 
\end{equation*} 
\end{itemize}
By uniqueness of  solutions on $\BB_i$, we obtain: 
\begin{equation}
\label{cond:closed}
Z^i(h, \ve, \ell_z) = Z^j(h, \ve, \ell_z) \quad\quad\quad\text{on}\quad \BB_i\cap \BB_j\quad\text{such that}\quad i\neq j. 
\end{equation}
We set
\begin{equation}
\label{Zh:trapped}
Z^{trapped}(h, \ve, \ell_z):= \bigcup_{i=1\cdots4}Z^i(h, \ve, \ell_z). 
\end{equation}
and
\begin{equation}
\label{Zh:abs}
Z^{abs}(h, \ve, \ell_z):= Z^{abs}(h, \ve, \ell_z). 
\end{equation}
It is easy to see that 
\begin{equation}
Z^{abs}(h, \ve, \ell_z)\cap Z^{trapped}(h, \ve, \ell_z) = \emptyset. 
\end{equation}
Therefore, $Z(h, \ve, \ell_z)$ has two connected components. 
It remains to show that $Z^{trapped}(h, \ve, \ell_z)$ is diffeomorphic to $\mathbb S^1$. To see this, we construct a complete curve $\alpha:\mathbb R\to Z^{trapped}(h, \ve, \ell_z)$ which is periodic and which will parametrise $Z^{trapped}(h, \ve, \ell_z)$.  
\begin{enumerate}
\item there exist two unique points $(\rho_i(h, \ve, \ell_z), 0)\in Z^i(h, \ve, \ell_z)$. 
\item We set 
$$\alpha(0):= (\rho_1(h, \ve, \ell_z), 0)\quad\;\quad \alpha\left(\frac{1}{2}\right):= (\rho_2(h, \ve, \ell_z), 0).$$ 
\item Let $t_0<t_1\in  \left]0, \frac{1}{2}\right[$. We set:
\begin{equation*}
\rho_0:= \Phi^1_{(\ve, \ell_z)}\left(z^K_{min} - \frac{c}{2}\right)\quad\text{and}\quad \rho_1:= \Phi^3_{(\ve, \ell_z)}\left((z^K_{min} - \frac{c}{2}\right)
\end{equation*}
and
\begin{equation*}
z_i := z^K_{min} - \frac{c}{2} \quad\quad i\in\left\{0, 1 \right\}
\end{equation*} 
so that 
\begin{equation*}
(\rho_0, z_0)\in Z^1(h, \ve, \ell_z)\cap Z^3(h, \ve, \ell_z) \quad\text{and}\quad (\rho_1, z_1)\in Z^2(h, \ve, \ell_z)\cap Z^3(h, \ve, \ell_z).
\end{equation*}
\item We construct the path $\alpha$ on $]0, t_0[$ by setting:
\begin{equation*}
\alpha(t) :=  \left(\Phi^1_{\ve, \ell_z}(z(t)), z(t):= z_0\frac{t}{t_0}\right).
\end{equation*}
\item On $]t_0, t_1[$, $\alpha$ is defined by
\begin{equation*}
\alpha(t) :=  \left(\rho(t):= \rho_0\frac{t_1 - t}{t_1 - t_0} + \rho_1\frac{t - t_0}{t_1 - t_0}, z(t):= \Phi^3_{\ve, \ell_z}(\rho(t))\right).
\end{equation*}
\item On $\left]t_1, \frac{1}{2}\right[$, $\alpha$ is defined by
\begin{equation*}
\alpha(t) :=  \left(\Phi^2_{\ve, \ell_z}(z(t)), z(t):= \frac{z_1}{t_1 - \frac{1}{2}}t - \frac{z_1}{2t_1 - 1}\right).
\end{equation*}
One can easily check that $\alpha$ is continuous on $\left[0, \frac{1}{2}\right]$. Now, by similar constructions, we extend continuously  the construction to $\left[\frac{1}{2}, 1\right]$ so that we obtain 
$$\alpha(0) = \alpha(1).$$
\item Note that the constructed curve is in fact $C^1$ due to regularity of $\Phi^i$'s. 
\item Therefore, $\alpha$ is homeomorphic to $\mathbb S^1$. 
\end{enumerate}
\end{proof}
\begin{remark}
We note that $\delta_0$ depends on $\Bbound$ and the parameters of the Kerr black hole $(a, M)$. 
\end{remark}

\subsection{Further properties of the functions $\Phi^i$}
In order to introduce the dimensions of the matter cloud, we need the following lemma concerning the extrema of the functions $\Phi^i_{(\ve, \ell_z)}(h, \cdot)$ at a fixed $h$ and $(\ve, \ell_z)$
\begin{lemma}
Let $\delta_0>0$ be given by Proposition \ref{Pert:kerr}. Then, after possibly shrinking $\delta_0$, $\forall h\in B(0, \delta_0)$, $\forall (\ve, \ell_z)\in \Bbound$, there exist
\begin{enumerate}
\item a unique $z_0(h, \ve, \ell_z)\in B_{\delta_0}(0)\subset ]-\gamma, \gamma[$ such that $$\frac {\partial\Phi^{abs}_{\ve, \ell_z}}{\partial z}(h, z_0) = 0.\;$$ Moreover, $\Phi^{abs}_{\ve, \ell_z}(h, \cdot)$ is maximal at this point. 
\item a unique $\rho_{max}^1(h, \ve, \ell_z)\in B_{\delta_0}(\rho_{max}(\ve, \ell_z))\subset I^1_{(\ve, \ell_z)}$ such that $$\frac{\partial\Phi^1_{\ve, \ell_z}}{\partial \rho}(h, \rho_{max}^1(h, \ve, \ell_z)) = 0\;,$$ Moreover, $\Phi^1_{\ve, \ell_z}(h, \cdot)$ is maximal at this point. 
\item a unique $\rho_{max}^2(h, \ve, \ell_z)\in B_{\delta_0}(\rho_{max}(\ve, \ell_z))\subset I^2_{(\ve, \ell_z)}$ such that $$\frac{\partial\Phi^2_{\ve, \ell_z}}{\partial \rho}(h, \rho_{max}^2(h, \ve, \ell_z)) = 0\;,$$ Moreover, $\Phi^2_{\ve, \ell_z}(h, \cdot)$ is minimal at this point. 
\item a unique $z_{c1}(h, \ve, \ell_z)\in B_{\delta_0}(0)\subset I^3_{(\ve, \ell_z)}$ such that $$\frac{\partial\Phi^3_{\ve, \ell_z}}{\partial z}(h, z_{c2}(h, \ve, \ell_z)) = 0\;,$$Moreover, $\Phi^3_{\ve, \ell_z}(h, \cdot) $ is minimal at this point. 
\item a unique $z_{c2}(h, \ve, \ell_z)\in B_{\delta_0}(0)\subset I^4_{(\ve, \ell_z)}$ such that $$\frac {\partial\Phi^4_{\ve, \ell_z}}{\partial z}(h, z_{c2}(h, \ve, \ell_z)) = 0\;,$$Moreover, $\Phi^4_{\ve, \ell_z}(h, \cdot)$ is maximal at this point. 
\end{enumerate}
\end{lemma}

\begin{proof}
For example, in order to obtain the second point, we consider the mapping $\partial _\rho\Phi^1_{\ve, \ell_z}: B_{\delta_0}(0)\times  I^1_{(\ve, \ell_z)}\to \mathbb R$ defined by  
\begin{equation*}
\frac{\partial \Phi^1_{\ve, \ell_z}}{\partial \rho}(h, \rho).  
\end{equation*}
and we consider the equation 
\begin{equation*}
\frac{\partial \Phi^1_{\ve, \ell_z}}{\partial \rho}(h, \rho) = 0. 
\end{equation*}
By Lemma \ref{lemma::50}, the point $(0,  \rho_{max}(\ve, \ell_z))$ is zero of the above equation. Moreover, by Lemma \ref{phi1:zvc}, $\displaystyle \partial _\rho\Phi^1_{\ve, \ell_z}$ is continuously Fréchet differentiable on its domain. Furthermore, 
\begin{equation*}
\frac{\partial^2 \Phi^1_{\ve, \ell_z}}{\partial \rho^2}(0, \rho_{max}(\ve, \ell_z))  = \frac{\partial^2 \Phi^{K, 1}_{\ve, \ell_z}}{\partial \rho^2}(\rho_{max}(\ve, \ell_z))  \neq  0. 
\end{equation*}
Therefore, we apply the same method used in Lemma \ref{phi1:zvc} (fixed point arguments) to obtain that there exists $\delta_0>0$ independent of $(\ve, \ell_z)$ (eventually smaller than the one chosen by Proposition \ref{Pert:kerr}) such that $\displaystyle \forall h\in B_{\delta_0}(0)\subset \Ltheta\times\Lsigma\times\LX $ there exists a unique $\rho_{max}^1(h, \ve, \ell_z)\in B_{\delta_0}(\rho_{max}(\ve, \ell_z))\subset I^1_{(\ve, \ell_z)}$ such that $$\frac{\partial\Phi^1_{\ve, \ell_z}}{\partial \rho}(h, \rho_{max}^1(h, \ve, \ell_z)) = 0\;,$$ Moreover, $\Phi^1_{\ve, \ell_z}(h, \cdot)$ is maximal at this point and $\rho_{max}^1(\cdot, \ve, \ell_z)$ is continuously Fréchet differentiable on $B_{\delta_0}(0)$. 
\end{proof}

\subsection{Localisation of the matter cloud}
In this section, we will localise the matter cloud moving in the perturbed spacetime. More precisely, we will prove that for $\delta_0$ sufficiently small, all trapped non-spherical  timelike future-directed geodesics with $(\ve, \ell_z)\in \Bbound$ moving the perturbed spacetime are located inside $\BB^{\pm, trapped}(a, M)$, defined in Proposition \ref{domain:trapped}. We will also construct a compact region of $\BB^{\pm, trapped}(a, M)$  which depends only on $\Bbound$ and $(a, M)$ which contains the support of the distribution function. Finally, we will arrange so that the domain of trapped non-spherical geodesics moving in the perturbed spacetime remains in $\BAbarre$. 
\\ 
\\ First, we recall 
 \begin{equation*}
 \begin{aligned}
          \BHbarre&= \left\{(\rho,z)\in\overline{\mathscr{B}}, \rho^2+(z\pm\beta)^2>\frac{\beta}{a}, |z|+|\rho|<\left( 1 + \frac{1}{b}\right)\beta   \right\}, \\        
             \BAbarre&= \left\{(\rho,z)\in\overline{\mathscr{B}}, \rho^2+(z\pm\beta)^2>\frac{\beta}{a}, |z|+|\rho|>\left( 1 - \frac{1}{b}\right)\beta   \right\}.
             \end{aligned}
        \end{equation*}

 We generalise Lemma \ref{support:matter} to the perturbed spacetimes in the following way. 
\begin{Propo}
\label{support:matter:pert}
There exist $\;\delta_0>0\,,\,b>0$ such that $\displaystyle \forall (\ve, \ell_z)\in\Bbound$, $\displaystyle h\in B(0, \delta_0)$ we have 
\begin{equation*}
\BHbarre \cap Z^{trapped}(h, \ve, \ell_z) = \emptyset \quad\text{and}\quad  Z^{trapped}(h, \ve, \ell_z)\subset \BAbarre. 
\end{equation*}
\end{Propo}
We set the quantities :
\begin{equation}
\begin{aligned}
\label{pert::quantities}
\rho_1(h, \ve, \ell_z)&:= \min_{z\in]-z_c(\ve, \ell_z), z_c(\ve, \ell_z)[}\Phi^1_{(\ve, \ell_z)}(z) \;,\; \rho_2(h, \ve, \ell_z):= \max_{z\in]-z^K_{min}, z^K_{min}[}\Phi^2_{(\ve, \ell_z)}(z) \;,\; \\
z^{max}(h, \ve, \ell_z)&:= \max_{]\rho_c^1(\ve, \ell_z), \rho_c^2(\ve, \ell_z)[} \Phi^3_{(\ve, \ell_z)}(\rho)\;,\; z^{min}(h, \ve, \ell_z):= \min_{]\rho_c^1(\ve, \ell_z), \rho_c^2(\ve, \ell_z)[} \Phi^4_{(\ve, \ell_z)}(\rho)
\end{aligned}
\end{equation}
Now, we introduce the "dimensions of the matter shell" by defining: 
\begin{equation*}
\begin{aligned}
\rho_{min}(h) &:= \min_{\Bbound}\rho_1(h, \ve, \ell_z)\;,\; \rho_{max}(h) := \max_{\Bbound}\rho_2(h, \ve, \ell_z)\;,\; \\
Z_{max}(h) &:= \max_{\Bbound}z^{max}(h, \ve, \ell_z)\;,\; Z_{min}(h) := \min_{\Bbound}z^{min}(h, \ve, \ell_z)\;.\; 
\end{aligned}
\end{equation*}
We give the proof of Proposition \ref{support:matter:pert}
\begin{proof}
Let $\delta_0$ be sufficiently small and let $(\ve, \ell_z)\in \Bbound$. Let $(\rho, z)\in Z^{trapped}(h, \ve, \ell_z)$. Then, $\rho\geq \rho_1(h, \ve, \ell_z)\geq  \rho_{min}(h)$. Set
\begin{equation*}
\rho_{min}(h) = \rho_1(h, \ve_0, \ell_0) \quad\text{for some}\quad (\ve_0, \ell_0)\in \Bbound. 
\end{equation*} 
By Lemma \ref{support:matter}, 
\begin{equation*}
\rho_1^K(\ve_0, \ell_0) > \left(1 + \frac{1}{b} \right)\beta. 
\end{equation*}
Recall that,  
\begin{equation*}
\forall (\ve, \ell_z)\in\Abound\;,\; \rho_1^{K}(\ve, \ell_z) > \rho^{mb}(a, M) > \left(1 + \frac{1}{b}\right)\beta. 
\end{equation*}
Therefore, there exists $r>0$ uniform in $(\ve_0, \ell_0)$ such that 
\begin{equation*}
B(\rho_1^K(\ve_0, \ell_0), r)\subset \left]\left(1 + \frac{1}{b} \right)\beta, \infty\right[. 
\end{equation*}
From Proposition \ref{Pert:kerr}, we have 
\begin{equation*}
\rho_1(h, \ve_0, \ell_0)\in B(\rho_1^K(\ve_0, \ell_0), \delta_0). 
\end{equation*}
Now, we adjust $\delta_0 $ so that 
\begin{equation*}
B(\rho_1^K(\ve_0, \ell_0), \delta_0) \subset B(\rho_1^K(\ve_0, \ell_0), r)\subset \left]\left(1 + \frac{1}{b} \right)\beta, \infty\right[. 
\end{equation*}
Therefore, 
\begin{equation*}
\rho_1(h, \ve_0, \ell_0)\in \left]\left(1 + \frac{1}{b} \right)\beta, \infty\right[. 
\end{equation*}
Hence, $(\rho, z)\notin\BHbarre$. 
\end{proof}
Before we construct the compact support of the distribution function (for the spacetime variables), we will need the following lemma
\begin{lemma}
\label{dist:pert}
There exists $\delta_0>0$ such that for all $h\in B(0, \delta_0)$ and for all $\displaystyle (\ve, \ell_z)\in \Bbound$ we have 
\begin{equation*}
\rho_1(h, \ve, \ell_z) - \rho_0(h, \ve, \ell_z)> \eta,
\end{equation*}
where $\eta$ is given by Lemma \ref{dist:kerr}. 
\end{lemma}
\begin{proof}
Let $\delta_0>0$ be given by Proposition \ref{Pert:kerr}. Let $h\in B(0, \delta_0)$ and $(\ve, \ell_z)\in \Bbound$. Then, for sufficiently small $\delta_0$, we have 
\begin{equation*}
\rho_1(h, \ve, \ell_z) > \rho^K_{1, min} - \frac{\eta}{2}  \quad\quad\text{and}\quad\quad  \rho_0(h, \ve, \ell_z) < \rho^K_{0, max} + \frac{\eta}{2}. 
\end{equation*}
Therefore, 
\begin{equation*}
\rho_1(h, \ve, \ell_z) - \rho_1(h, \ve, \ell_z) > \rho^K_{1, min} -  \rho^K_{0, max} - \eta. 
\end{equation*}
By \eqref{dist:eta}, we obtain 
\begin{equation*}
\rho_1(h, \ve, \ell_z) - \rho_1(h, \ve, \ell_z) > 2\eta - \eta = \eta. 
\end{equation*}
\end{proof}
Finally, we obtain
\begin{Propo}
\label{support:f}
Assume that the distribution function $f$ has the ansatz defined by \eqref{ansatz:for:f}. Then
\begin{equation}
\supp_{(\rho, z)}f\subset\subset\left[\rho_{min}(h), \rho_{max}(h)\right]\times\left[Z_{min}(h), Z_{max}(h) \right]. 
\end{equation}
\end{Propo}
\begin{proof}
Let $\displaystyle (\rho, z)\in\supp_{(\rho, z)}f$. Then, 
\begin{equation*}
(\ve, \ell_z)\in \Bbound\quad\quad\text{and}\quad\quad \Chi_\eta\left(\rho - \rho_1(h, (\ve, \ell_z))\right)>0. 
\end{equation*}
By \eqref{allowed::region} and since $(\ve, \ell_z)\in \Bbound$, $(\rho, z)$ must lie either in the region bounded by $\displaystyle Z^{trapped}(h, \ve, \ell_z)$ or in the region bounded by $\displaystyle \partial(\BHbarre\cup\BAbarre) \bigcup Z^{abs}(h, \ve, \ell_z)$. 
\\ Since $$\Chi_\eta\left(\rho - \rho_1(h, (\ve, \ell_z))\right)>0,$$
 by Lemma \ref{dist:pert}, we have 
\begin{equation*}
\rho > \rho_1(h, \ve, \ell_z) - \eta > \rho_0(h, \ve, \ell_z).
\end{equation*}
Thus, $(\rho, z)$ cannot lie in the region bounded by $\displaystyle \partial(\BHbarre\cup\BAbarre) \bigcup Z^{abs}_c(h, \ve, \ell_z)$. We conclude that 
\begin{equation*}
(\rho, z)\in\left[\rho_1(h, \ve, \ell_z), \rho_2(h, \ve, \ell_z)\right]\times\left[z_{min}(h, \ve, \ell_z), z_{max}(h, \ve, \ell_z)\right]\subset \left[\rho_{min}(h), \rho_{max}(h)\right]\times\left[Z_{min}(h), Z_{max}(h)\right]. 
\end{equation*}
This ends the proof. 
\end{proof}

\section{Set-up for solving the renormalised equations}
\label{Set::Up::Kerr}
\subsection{Two fixed point lemmas}
\label{Fixed::point::lemmaa}
In this section, we state two variations of the classical fixed point theorem whose applications will allow us to solve the system of equations for the renormalised unknowns. These versions were derived and used in \cite{chodosh2017time}. 
\begin{theoreme}[Banach fixed point theorem]
\label{Banach}
Let $(X,d)$ be a non-empty complete metric space with a contraction mapping $T:X\to X$. Then T admits a unique fixed point in X.
\end{theoreme}
\noindent The following theorem is a consequence of \ref{Banach}.
\begin{theoreme}\label{Fixed::Point::1}
Suppose that we have Banach spaces $\mathcal{L}$, $\mathcal{Q}$, and $\mathcal{P}$, $\epsilon>0$ and a map 
\begin{equation*}
    \mathcal{T}: B_{\epsilon}(\mathcal{L}) \times B_{\epsilon}(\mathcal{Q}) \times B_{\epsilon}(\mathcal{P}) \to B_{\epsilon}(\mathcal{L}).
\end{equation*}
Furthermore, suppose that 
   \begin{enumerate}
        \item There exists a constant $D>0$ such that $(l,q,p)\in B_{\epsilon}(\mathcal{L})\times B_{\epsilon}(\mathcal{Q}) \times B_{\epsilon}(\mathcal{P})$, we have
        \[
        ||\mathcal{T}(l,q,p)||_{\mathcal{L}} \le D(||l||^2_{\mathcal{L}} + ||q||^2_{\mathcal{Q}} + ||p||_{\mathcal{P}}).
        \]
        \item There exists a constant $D>0$ such that $(l_1,q_1,p_1),(l_2,q_2,p_2)\in B_{\epsilon}(\mathcal{L})\times B_{\epsilon}(\mathcal{Q}) \times B_{\epsilon}(\mathcal{P})$ implies
        \begin{align*}
        &||\mathcal{T}(l_1,q_1,p_1)-\mathcal{T}(l_2,q_2,p_2)||_{\mathcal{L}}  \\
        &\le D \left[(||l_1||_{\mathcal{L}}+||l_2||_{\mathcal{L}})||l_1 -l_2|| + ||q_1||_{\mathcal{Q}}+||q_2||_{\mathcal{Q}})||q_1 -q_2||_{\mathcal{Q}} + ||p_1 -p_2||_{\mathcal{P}} \right].
        \end{align*}
    \end{enumerate}
Then after choosing $\epsilon>0$ sufficiently small, there exists a solution map $\mathcal{G}:B_{\epsilon}(\mathcal{Q}) \times B_{\epsilon}(\mathcal{P})\to B_{\epsilon}(\mathcal{L})$ such that

    \begin{enumerate}
        \item $(q,p) \in B_{\epsilon}(\mathcal{Q}) \times B_{\epsilon}(\mathcal{P})$ implies
        \[
        \mathcal{T}(\mathcal{G}(q,p),q,p) = \mathcal{G}(q,p).
        \]
        \item There exists a constant $D>0$ such that $(q,p)\in B_{\epsilon}(\mathcal{Q}) \times B_{\epsilon}(\mathcal{P})$ implies 
        \[
        ||\mathcal{G}(q,p)||_{\mathcal{L}} \le D\left( ||q||^2_{\mathcal{Q}} + ||p||_{\mathcal{P}}\right).
        \]
        \item There exists a constant $D>0$ such that $(q_1,p_1),(q_2,p_2)\in B_{\epsilon}(\mathcal{Q}) \times B_{\epsilon}(\mathcal{P})$ implies
        \[
        ||\mathcal{G}(q_1,p_1)-\mathcal{G}(q_2,p_2)||_{\mathcal{L}} \le D \left( (||q_1||_{\mathcal{Q}}+||q_2||_{\mathcal{Q}})||q_1 -q_2||_{\mathcal{Q}} + ||p_1 - p_2||_{\mathcal{P}} \right).
        \]
    \end{enumerate}
\end{theoreme}

\begin{proof}
We apply Theorem \ref{Banach} to 
\[
\mathcal{T}(.,q,p):B_{\epsilon}(\mathcal{L})\to B_{\epsilon}(\mathcal{L}).
\]
For this we choose $\epsilon$ sufficiently small so that $\mathcal{T}(.,p,q)$ is a contraction map.

First we fix (q,p). By the second assumption, one has
    \begin{align*}
        ||\mathcal{T}(l_1,q,p)-\mathcal{T}(l_2,q,p)||_{\mathcal{L}} &\le D \left((||l_1||_{\mathcal{L}}+||l_2||_{\mathcal{L}})||l_1 -l_2||_{\mathcal{L}}  \right) \\
        &\le 2D\epsilon ||l_1 -l_2||_{\mathcal{L}}.
    \end{align*}
We choose $\epsilon$ such that 
\begin{equation*}
    2D\epsilon < 1.
\end{equation*}
\noindent We apply Theorem \ref{Banach} to have the first point of the theorem.
\\Now we derive the quadratic estimates,

\begin{align*}
    ||\mathcal{G}(q,p)||_{\mathcal{L}} &= ||\mathcal{T}(\mathcal{G}(q,p),q,p)||_{\mathcal{L}} \\
    &\le D(||\mathcal{G}(q,p)||^2_{\mathcal{L}} + ||q||^2_{\mathcal{Q}} + + ||p||_{\mathcal{P}}), \\
    &\le D\epsilon||\mathcal{G}(q,p)||_{\mathcal{L}} + D(||q||^2_{\mathcal{Q}}+ ||p||_{\mathcal{P}}),\\
    &\le \frac{1}{2}||\mathcal{G}(q,p)||_{\mathcal{L}} + D(||q||^2_{\mathcal{Q}} + ||p||_{\mathcal{P}}). 
\end{align*}
Hence

\begin{equation*}
    ||\mathcal{G}(q,p)||_{\mathcal{L}} \lesssim ||q||^2_{\mathcal{Q}} + ||p||_{\mathcal{P}}. 
\end{equation*}

For the second quadratic estimates, we proceed in the same way:
\begin{align*}
&||\mathcal{G}(q_1,p_1)-\mathcal{G}(q_2,p_2)||_{\mathcal{L}} = ||\mathcal{T}(\mathcal{G}(q_1,p_1),q_1,p_1)-\mathcal{T}(\mathcal{G}(q_1,p_1),q_2,p_2)||_{\mathcal{L}}, \\
&\le D((||\mathcal{G}(q_1,p_1)||_{\mathcal{L}} + ||\mathcal{G}(q_2,p_2)||_{\mathcal{L}})||\mathcal{G}(q_1,p_1)-\mathcal{G}(q_2,p_2)||_{\mathcal{L}} +   (||q_1||_{\mathcal{Q}}+||q_2||_{\mathcal{Q}})||q_1 -q_2||_{\mathcal{Q}} + ||p_1 - p_2||_{\mathcal{P}} )\\
&\le 2D\epsilon||\mathcal{G}(q_1,p_1)-\mathcal{G}(q_2,p_2)||_{\mathcal{L}} +  D( (||q_1||_{\mathcal{Q}}+||q_2||_{\mathcal{Q}})||q_1 -q_2||_{\mathcal{Q}} + ||p_1 - p_2||_{\mathcal{P}} )\\
&\le \frac{1}{2}||\mathcal{G}(q_1,p_1)-\mathcal{G}(q_2,p_2)||_{\mathcal{L}} +  D( (||q_1||_{\mathcal{Q}}+||q_2||_{\mathcal{Q}})||q_1 -q_2||_{\mathcal{Q}} + ||p_1 - p_2||_{\mathcal{P}} ).
\end{align*}
Hence 
\begin{equation*}
    ||\mathcal{G}(q_1,p_1)-\mathcal{G}(q_2,p_2)||_{\mathcal{L}} \lesssim \left( (||q_1||_{\mathcal{Q}}+||q_2||_{\mathcal{Q}})||q_1 -q_2||_{\mathcal{Q}} + ||p_1 - p_2||_{\mathcal{P}} \right).
\end{equation*}

\end{proof}
\noindent As a consequence of Theorem \ref{Fixed::Point::1}, we have the following version of the fixed point theorem
\begin{theoreme}\label{Fixed::Point::2}
    Suppose we have a linear operator $L : \mathcal{L}\to\tilde{\mathcal{L}}$, an operator $N: \mathcal{L}\times\mathcal{Q}\times\mathcal{P}\to\tilde{\mathcal{L}}$, $E : B_{\epsilon}(\mathcal{L})\times B_{\epsilon}(\mathcal{Q})\times B_{\epsilon}(\mathcal{P})\to \tilde{\mathcal L}$ for some $\epsilon>0$, such that
    \begin{enumerate}
        \item For all $(l,q, p)\in B_{\epsilon}(\mathcal{L})\times B_{\epsilon}(\mathcal{Q})\times B_{\epsilon}(\mathcal{P})$, we have
        \[
        E(l,q, p) = L(l) - N(l,q, p).
        \]
        \item We have a Banach space $\mathcal{N}\subset\tilde{\mathcal{L}}$ and a bounded map $L^{-1}:\mathcal{N}\to\mathcal{L}$ such that $H\in\mathcal{N}$ implies 
        \[
        L(L^{-1}(H)) = H.
        \]
        \item We have $N(B_{\epsilon}(\mathcal{L})\times B_{\epsilon}(\mathcal{Q})\times B_{\epsilon}(\mathcal{P}))\subset\mathcal{N}$ and there exists a constant $D>0$ such that $(l,q, p)\in B_{\epsilon}(\mathcal{L})\times B_{\epsilon}(\mathcal{Q})\times B_{\epsilon}(\mathcal{P})$ implies 
        \[
        ||N(l,q, p)||_{\mathcal{N}} \le D(||l||^2_{\mathcal{L}} + ||q||^2_{\mathcal{Q}} + ||p||_{\mathcal P})
        \]
        \item There exists a constant $D>0$ such that $(l_1,q_1, p_1),(l_2,q_2, p_2)\in B_{\epsilon}(\mathcal{L})\times B_{\epsilon}(\mathcal{Q})\times B_{\epsilon}(\mathcal{P})$ implies
        \[
        ||N(l_1,q_1, p_1)-N(l_2,q_2, p_1)||_{\mathcal{N}} \le D\left[(||l_1||_{\mathcal{L}}+||l_2||_{\mathcal{L}})||l_1 -l_2|| + ||q_1||_{\mathcal{Q}}+||q_2||_{\mathcal{Q}})||q_1 -q_2|| + ||p_1 - p_2||_{\mathcal P}\right].
        \]
    \end{enumerate}
    Then, after choosing a sufficiently small $\epsilon$, there exists a solution map $\mathcal{G}:B_{\epsilon}(\mathcal{Q})\times B_{\epsilon}(\mathcal{P})\to B_{\epsilon}(\mathcal{L})$ such that 
    \begin{enumerate}
        \item $(q, p) \in B_{\epsilon}(\mathcal{Q})\times B_{\epsilon}(\mathcal{P})$ implies
        \[
        E(\mathcal{G}(q, p),q, p) = 0.
        \]
        \item There exists a constant $D>0$ such that $(q, p)\in B_{\epsilon}(\mathcal{Q})\times B_{\epsilon}(\mathcal{P})$ implies 
        \[
        ||\mathcal{G}(q, p)||_{\mathcal{L}} \le D\left(||q||^2_{\mathcal{Q}} + ||p||_{\mathcal P}\right).
        \]
        \item There exists a constant $D>0$ such that $\displaystyle q_1,q_2\in B_{\epsilon}(\mathcal{Q})$ and $\displaystyle p_1,p_2\in B_{\epsilon}(\mathcal{P})$ imply
        \[
        ||\mathcal{G}(q_1, p_1)-\mathcal{G}(q_2, p_2)||_{\mathcal{L}} \le D\left[(||q_1||_{\mathcal{Q}}+||q_2||_{\mathcal{Q}})||q_1 -q_2||_{\mathcal{Q}} + ||p_1 - p_2||_{\mathcal P} \right].
        \]
    \end{enumerate}

\end{theoreme}

\subsection{Toy Model}
\label{Toy:model}
In this section, we present a model problem which indicates the general structure that we will exploit when we solve the equations for the renormalised unknowns.
\\ Consider the following non-linear Poisson equation on the open unit ball of $\mathbb R^n$ with Dirichlet boundary condition:
\begin{equation}
\label{poisson}
\left\{
\begin{aligned}
\Delta h &= N(h, \delta)(x)\quad\text{on}\; B_1 \\
\left. h\right|_{\partial B_1} &= 0. 
\end{aligned}
\right.
\end{equation}
where $N\,:\,B_{\delta_0}(\mathcal L) \times [0, \delta_0[ \to \mathcal L $ is the mapping defined by 
\begin{equation*}
N(h, \delta)(x) := \int_K\; F(x, v, \delta)\Psi(h(x), v)\, dv
\end{equation*}
where $K\subset\subset \mathbb R^3$, $\displaystyle \mathcal L := C^{k+2, \alpha_0}(\overline{B_1})$ with $\alpha_0\in(0, 1)$, $k\geq 0$ and $B_{\delta_0}(\mathcal L)$ is the open ball  of $\mathcal L$ of radius $\delta_0>0$ centred at $0$. 
\\We make the following assumptions on $F$:
\begin{itemize}
\item $\displaystyle F: \mathbb R^n\times K\times[0, \delta_0[\to \mathbb R_+$ is continuous and $C^{k, \alpha_0}$ with respect to the first variable,
\item $\displaystyle \forall v\in K\;,\; \forall \delta\in[0, \delta_0[\;,\;  F(\cdot, v, \delta)$ is supported on $B_1$,
\item $\displaystyle \forall x\in B_1\;,\; \forall v\in K\;,\; F(x, v, 0) = 0$,
\item $\displaystyle \forall x\in B_1\;,\;\forall v\in K\;,\;F(x, v, \cdot)$ is differentiable at $\delta = 0$, 
\item $\displaystyle \Psi: \mathbb R\times K \to \mathbb R_+$ is smooth and compactly supported. 
\end{itemize}
We have the trivial solution given by $h = 0$ with $\delta = 0$. We will use the fixed point argument in order to construct a one-parameter family of solutions $(h(\delta))_{[0,\delta_0[}$ of \eqref{poisson} which equals the trivial solution when $\delta = 0$. We state the following result: 

\begin{Propo}
\label{problem:model}
There exists $\delta_0>0$ sufficiently small such that there exists a solution map $h: [0, \delta_0[\to B_{\delta_0}(\mathcal L)$ such that 
\begin{enumerate}
\item $\displaystyle \forall\delta\in[0, \delta_0[\;,\; h(\delta)$ solves \eqref{poisson}
and 
\item there exists $C>0$ such that $\forall \delta\in[0, \delta_0[$, 
\begin{equation*}
||h(\delta)||_{\mathcal L}\leq C\delta
\end{equation*} 
\item The one parameter family of solutions bifurcates from the trivial solution in the sense that $\delta\mapsto h(\delta)$ is differentiable at $\delta = 0$ and $\displaystyle \lim_{\delta\to 0}\, \delta^{-1}h(\delta) = \hat h$ where $\hat h$ is the solution of linear Poisson problem: 
\begin{equation}
\label{poisson:linear}
\left\{
\begin{aligned}
\Delta h &=\int_K\; \partial_\delta F(x, v, 0)\Psi(0, v)\, dv\quad\text{on}\; B_1 \\
\left. h\right|_{\partial B_1} &= 0. 
\end{aligned}
\right.
\end{equation}
\end{enumerate}
\end{Propo}

\begin{proof}
Let $\delta_0>0$. The proof relies on the application Theorem \ref{Fixed::Point::2} with $L = \Delta : \mathcal L\to C^{0}(\overline{B_1})$ and $N: \mathcal L\times \mathbb R\to C^{0}(\overline{B_1})$. To this end, we check the following assumptions:
\begin{enumerate}
\item $E: B_{\delta_0}(\mathcal L)\times[0, \delta_0[\to C^{0}(\overline{B_1})$ has the form 
\begin{equation*}
E(h, \delta) = \Delta(h) - N(h, \delta). 
\end{equation*}
It is well defined by the dominated convergence theorem. 
\item We consider the linear Poisson equation with the previous Dirichlet boundary conditions : 
\begin{equation}
\label{hom:problem}
\Delta h = H \quad\text{on } B_1, \quad\quad h|_{\partial B_1} = 0, 
\end{equation}
where $H$ is compactly compactly supported on $B_1$ and lies in $\mathcal N := C^{k, \alpha_0}(\overline{B_1})$.  Then, we can solve uniquely\footnote{We refer to \cite[Chapter 2.2]{evans1998partial} for the derivation of the representation formula.} for $h\in C^{k+2, \alpha_0}(\overline {B_1})$ :
More precisely, $h$ is given by 
\begin{equation}
\forall x\in B_1\;,\; h(x) = -\int_{B_1}\; G(x, y)H(y)dy
\end{equation}
where $G$ is the Green function for the unit ball. Now, the a priori estimates on $h$ result from the potential estimates.  More precisely, we apply Theorem $4.13$ of \cite{gilbarg2015elliptic} in order to obtain 
\begin{equation*}
h\in C^{k+2, \alpha_0}(\overline{B_1})\quad 
\end{equation*}
and we have
\begin{equation}
||h||_{C^{k+2, \alpha_0}(\overline {B_1})} \leq C ||H||_{\mathcal N}. 
\end{equation}
Therefore, $L^{-1}: \mathcal N \to \mathcal L$ is well defined, it is bounded and we have 
\begin{equation*}
\forall H\in \mathcal N\;,\; L\left(L^{-1}(H) \right) = H. 
\end{equation*}
\item We need to check that  
\begin{equation*}
N\left( B_{\delta_0}(\mathcal L)\times[0, \delta_0[\right) \subset\mathcal N 
\end{equation*}
and that there exists $D>0$ such that $(h, \delta)\in B_{\delta_0}(\mathcal L)\times[0, \delta_0[$ implies
\begin{equation}
\label{toy::estimate}
||N(h, \delta)||_{\mathcal N}\leq D\left( ||h||^2_{\mathcal L} + \delta\right). 
\end{equation}
We show the result for $k = 0$. The general case is obtained using the same arguments. 
\begin{itemize}
\item Let  $(h, \delta)\in B_{\delta_0}(\mathcal L)\times[0, \delta_0[$. Then, by the regularity assumptions on $F$,  $\forall v\in K $,  the function $F(\cdot, v, \delta)\Psi(h(\cdot), v)$ lies in $C^{0, \alpha_0}$.  Since $\Psi$ and $F(\cdot, \cdot, \delta)$ are compactly supported, we apply the dominated convergence theorem. Therefore $N(h, \delta)$ lies in $C^0(\overline B_1)$. 
\item It remains to show that $N(h, \delta)$ is Hölder continuous. This follows using the Hölder regularity assumption for $F$ and $h$ and the dominated convergence theorem. 
\item Now, we show the estimates \eqref{toy::estimate}. We claim that the mapping $N$ is Fréchet differentiable at $(0, 0)$. 
\\ Therefore, we have $\displaystyle \forall (h, \delta)\in B_{\delta_0}(\mathcal L)\times[0, \delta_0[$ with $\delta_0$ sufficiently small, we have 
\begin{equation*}
N(h, \delta) = N(0, 0) + D_hN(0, 0)\cdot h + D_\delta N(0, 0)\cdot \delta + O(||h||^2_{\mathcal L} + \delta^2). 
\end{equation*}
By the assumptions on $F$, we have 
\begin{equation}
N(0, 0) = 0 \quad\text{and}\quad D_hN(0, 0) = 0. 
\end{equation}
Hence, 
\begin{equation*}
||N(h, \delta)||_{C^0(\overline B_1)} \leq C(||h||^2_{\mathcal L} + \delta). 
\end{equation*}
The Hölder part is estimated using similar arguments. 
\end{itemize}
\item  The remaining step is to show that there exists $C>0$ such that $\delta_1, \delta_2\in[0, \delta_0[$ implies 
\begin{equation*}
||h(\delta_1) - h(\delta_2)||_{\mathcal L} \leq C||\delta_1 - \delta_2||. 
\end{equation*}

\end{enumerate}
\end{proof}
\noindent Now, we make an analogy with the reduced EV system: $h$ corresponds to the renormalised quantities, the function $F$ corresponds to the distribution function and $K$ corresponds to $\Bbound$.  Moreover, the compact support of $F$  corresponds to the compact support of the matter terms. Finally, the ability to invert the Laplacian and solve for $h$ corresponds to the use of the modified Carter-Robinson theory.

\section{Solving for the renormalised quantities}
\label{solving::unknowns}
\subsection{Further analytical properties of the Kerr metric}

As in \cite{chodosh2017time}, We start by introducing the following function defined on $\BB$ by : 
\begin{equation}
h(\rho, z) := \log\left(\sqrt{\rho^2 + (z - \beta)^2} - (z - \beta) \right) + \log\left(\sqrt{\rho^2 + (z + \beta)^2} + (z + \beta) \right).  
\end{equation}
This function will allow us to capture the singular behaviour of the Kerr metric coefficients. In fact, it behaves in the same way as  $\log(X_K)$ near the horizon, the axis of symmetry, the poles and near at infinity. More precisely, we have
\begin{lemma}
\label{x::K:reg}
Define a function $$ x_K(\rho, z):= \log(X_K) - h. $$ Then, $x_K\in \hat C^\infty(\Bbarre)$. 
\end{lemma}

\begin{proof}
The proof is based on Taylor expansions of the different metric components around the singularities. First of all,  note that away from the horizon, the axis of symmetry and the poles, $x_k$ is smooth and all the derivatives are bounded.
\begin{itemize}
    \item \underline{Near the axis:} 
    
    Define $\displaystyle \Axis_N := \left\{ (\rho,z) \in \BAbarre\backslash\Axis,\quad z>\beta\right\}$ and let $(\rho,z)\in\Axis_N$.
    
    Then, 
    \begin{equation}
    \label{hA:exp}
        h = 2\log\rho+\log\left(\frac{\sqrt{\rho^2+(z+\beta)^2}+(z+\beta)}{\sqrt{\rho^2+(z-\beta)^2}+(z-\beta)}\right).
    \end{equation}
    Thus,
    \begin{align*}
        x_k &= \log\left(\rho^2\frac{\Pi}{\Sigma^2\Delta}\right)- 2\log\rho-\log\left(\frac{\sqrt{\rho^2+(z+\beta)^2}+(z+\beta)}{\sqrt{\rho^2+(z-\beta)^2}+(z-\beta)}\right) ,\\
        &= \log\left(\frac{\Pi}{\Sigma^2\Delta}\right)-\log\left(\frac{\sqrt{\rho^2+(z+\beta)^2}+(z+\beta)}{\sqrt{\rho^2+(z-\beta)^2}+(z-\beta)}\right) .
    \end{align*}
    Next note that $\displaystyle \frac{\Pi}{\Sigma^2\Delta}$ is smooth on $\Axis_N$ since both functions $\Sigma^2$ and $\Delta$ do not vanish on this set. Thus $x_K$ is smooth around $\Axis_N$. In the same way $x_K$ is smooth on $\Axis_S := \left\{ (\rho,z) \in \BAbarre\backslash\Axis,\quad z<-\beta\right\} $.
    \item \underline{Near the horizon:}
    Similarly, let $\tilde{\mathcal{H}} \subset \Horizon$ be a neighbourhood of the horizon. By the extendibility of Kerr around the horizon,  we have 
    \begin{equation*}
        \log(X_K) = \log(X_{\Horizon}), 
    \end{equation*}
    where $X_{\Horizon}(0,z) >0$. Thus,  $\log(X_K)$ is smooth near the horizon.
    
    By Taylor expansion of $h$ around the horizon, we find that the latter is also regular. 
    
    \item \underline{Near $p_N$:}
    
  Let $(\rho, z)\in \Bbarre_N$. We compute 
    \begin{align*}
        h &= \log\left(\sqrt{\rho^2+(z-\beta)^2}-(z-\beta)\right)+\log\left(\sqrt{\rho^2+(z+\beta)^2}+(z+\beta)\right)\\
        &= \log\left(\sqrt{s^2\chi^2+\frac{1}{4}(\chi^2-s^2)^2}-\frac{1}{2}(\chi^2-s^2)\right)+\log\left(\sqrt{s^2\chi^2+(\frac{1}{2}(\chi^2-s^2)+2\beta)^2}+(\frac{1}{2}(\chi^2-s^2)+2\beta)\right).\\
    \end{align*}
    The first term of the right hand side is given by 
    \begin{align*}
        I &:= \log\left(\sqrt{s^2\chi^2+\frac{1}{4}(\chi^2+s^2-2s\chi)}-\frac{1}{2}(\chi^2-s^2)\right) ,\\
    &= \log\left(\sqrt{\frac{1}{4}(\chi^2+s^2+2s\chi)}-\frac{1}{2}(\chi^2-s^2)\right) ,\\
    &= 2\log s.
    \end{align*}
    We compute the second term of the right hand side
    \begin{align*}
        II &:= \log\left(\sqrt{s^2\chi^2+\frac{1}{4}(\chi^2+s^2-2s^2\chi^2)+\beta(\chi^2-s^2)+2\beta^2}+(\frac{1}{2}(\chi^2-s^2)+2\beta)\right) ,\\
        &= \log\left(\frac{1}{2}\sqrt{(\chi^2+s^2)^2+4\beta(\chi^2-s^2)+4\beta^2}+(\frac{1}{2}(\chi^2-s^2)+2\beta)\right).
    \end{align*}
    $II$ is clearly smooth on $\Bbarre_N$. 
    Now, we have 
    \begin{align*}
        x_k(s,\chi) &= \log\left(\rho^2\frac{\Pi}{\Sigma^2\Delta} \right)-h ,\\
        &= 2\log s+\log \frac{\chi^2}{\Delta} + \log \frac{\Pi}{\Sigma^2} -h, \\
        &= \log \frac{\chi^2}{\Delta} + \log \frac{\Pi}{\Sigma^2}- \log\left(\frac{1}{2}\sqrt{(\chi^2+s^2)^2+4\beta(\chi^2-s^2)+4\beta^2}+(\frac{1}{2}(\chi^2-s^2)+2\beta)\right).
    \end{align*}
    $\Delta$ vanishes at $p_N$. We make a Taylor expansion for $\log \frac{\chi^2}{\Delta}$ around $p_N$:
    \begin{equation*}
        \tilde{r} = \tilde{r}_++\chi^2\tilde{r}_N(s,\chi) ,
    \end{equation*}
    where $\tilde{r}_N$ is smooth and does not vanish on $\Bbarre_N$. It is given by
    \begin{equation*}
        \tilde{r}_N(s,\chi) := \frac{2\beta}{4\beta-s^2}+O(s^2\chi^2).
    \end{equation*}
    Moreover, 
    \begin{equation*}
        \Delta = (\tilde{r}-\tilde{r}_+)(\tilde{r}-\tilde{r}_-).
    \end{equation*}
    Hence, 
    \begin{equation*}
    \log\frac{\chi^2}{\Delta} = \log\frac{1}{\tilde{r}_N(\tilde{r}-\tilde{r}_-)}.
    \end{equation*}
\end{itemize}
The latter  is smooth on $\Bnbarre$. The same arguments are applied near $p_S$. 
\end{proof}

Now, we recall from  \cite{chodosh2017time} the following decay estimates for $Y_K$, the Ernst potential defined by \eqref{Ernst:vacuum}:
\begin{lemma}
\label{decay:estimates}
On $\BAbarre\cup\BHbarre$, we have
\begin{equation*}
    \frac{\rho}{X^2_K}|\partial Y_K| \le Cr^{-4}, \quad |\partial\left( X^{-1}_K\partial Y_K\right)| \le Cr^{-4}, \quad |\partial^2\left( X^{-1}_K\partial Y_K\right)| \le Cr^{-5}.
\end{equation*}
We have the stronger bound for the $z-$component of the derivative of $Y_K$ on $\BAbarre\cup\BHbarre$,
\begin{equation*}
    \frac{1}{X^2_K}|\partial_z Y_K| \le Cr^{-5}, \quad |\partial\left( X^{-2}_K\partial Y_K\right)| \le Cr^{-6}, \quad |\partial^2\left( X^{-2}_K\partial Y_K\right)| \le Cr^{-7}.
\end{equation*}
On $\Bbarre_N$ we have the estimates
\begin{equation*}
    |\underline{\partial}Y_K|\le Cs^3, \quad |\underline{\partial}^2Y_K|\le Cs^2, \quad |\underline{\partial}^3Y_K|\le Cs.
\end{equation*}
The latter estimates remain valid on $\Bsbarre$. 
Here, $r$ is defined by: $$r := \sqrt{1 + \rho^2 + z^2}. $$
\end{lemma}

\begin{lemma}
\label{asymptotics:h1}
We have 
\begin{itemize}
\item  
\begin{equation*}
\frac{|\partial X_K|}{X_K} \sim |\partial h| \quad \text{near}\; \Horizon , \Axis, p_{N}, p_{S}. 
\end{equation*}
\item There exist $c, C>0$ such that $\forall (\rho, z)\in \Bbarre$, we have 

\begin{equation*}
c e^h\left|\partial h\right| \leq \left|\partial X_K\right| \leq C e^h\left|\partial h\right|. 
\end{equation*}
\end{itemize}
\end{lemma}

\begin{proof}
\begin{enumerate}
\item We prove the equivalence in every region: $\BHbarre$, $\BAbarre$, $\Bnbarre$ and $\Bsbarre$. 
\begin{enumerate}
\item Near $\Horizon$: We have 
\begin{equation*}
\begin{aligned}
\partial\log X_K &= \partial(\log X_\Horizon(\rho^2, z)) \\
&= \frac{1}{X_\Horizon(\rho^2, z)}\left(2\rho\partial_\rho X_\Horizon(\rho^2, z)\;\;   \partial_z X_\Horizon(\rho^2, z)\right)^t
\end{aligned}
\end{equation*}
and 
\begin{equation*}
\partial h (\rho, z) = \left( 2\rho\partial_\rho\overline h(\rho^2, z)\;\; -\frac{2z}{\beta^2 - z^2} + \partial_z\overline h(\rho^2, z) \right)^t
\end{equation*}
where $\overline h$ is some smooth function defined on $\BHbarre$. Recall that, $X_\Horizon$ is given by 
\begin{equation*}
X_\Horizon(\rho^2, z) = \frac{\beta^2 - z^2}{\beta^2}\frac{4 r_H^2}{\Sigma^2(\rho^2, z)} + \overline X_\Horizon(\rho^2, z)
\end{equation*}
where $\overline X_\Horizon $ is some smooth function defined on $\BHbarre$.  Now, we compute:
\begin{equation*}
\begin{aligned}
\partial_z X_\Horizon(\rho^2, z) &= \partial_z\left(\frac{\beta^2 - z^2}{\beta^2}\right)\frac{4 r_H^2}{\Sigma^2(\rho^2, z)} -   4 r_H^2\frac{\beta^2 - z^2}{\beta^2}\frac{\partial_z \Sigma^2(\rho^2, z)}{\Sigma^4(\rho^2, z)}+\partial_zX_\Horizon(\rho^2, z). 
\end{aligned}
\end{equation*}
On the horizon, we have 
\begin{equation*}
\partial_zX_\Horizon(0, z) = 0 \;,\; \overline X_\Horizon(\rho^2, z) = 0\;,\; \partial_z\overline h(0, z) = 0\;,\; \Sigma^2(0, z) = r_H^2 + \frac{z^2}{\beta^2}. 
\end{equation*}
Now, straightforward computations imply: 
\begin{equation*}
\frac{\partial_z X_\Horizon(0, z)}{X_\Horizon(0, z)} = -\frac{4r_Hz}{\beta^2 - z^2} \quad\text{and}\quad \partial_z h(0, z) = -\frac{2z}{\beta^2 - z^2}.
\end{equation*}
Therefore, 
\begin{equation*}
\lim_{\rho\to0 ; |z|<\beta}\,\frac{|\partial X_K|}{X_K|\partial h|} = 4 r_H^2. 
\end{equation*}
\item Near $\Axis$: We have
\begin{equation*}
\begin{aligned}
\partial \log X_K &= 2\left(\frac{1}{\rho}\;\;0\right)^t + \partial f_\Axis,  \\
\partial h &= 2\left(\frac{1}{\rho}\;\;0\right)^t + \partial g_\Axis,
\end{aligned}
\end{equation*}
where $f_\Axis, g_\Axis\in \hat C ^{\infty}(\BAbarre)$ with bounded derivatives.  Therefore, 
\begin{equation*}
\begin{aligned}
\frac{|\partial\log X_K|^2}{|\partial h|^2} &= \displaystyle \frac{\frac{4}{\rho^2}+ \frac{4}{\rho}\partial_\rho f + |\partial f|^2 }{\frac{4}{\rho^2}+ \frac{4}{\rho}\partial_\rho g + |\partial g|^2}. 
\end{aligned}
\end{equation*}
The latter goes to $1$ when $\rho\to 0$. 
\item Near $p_N$ and $p_S$:
\begin{equation*}
\begin{aligned}
\log X_K &= 2\log s + f_N(s, \chi) \\
h &= 2\log s + g_N(s, \chi)
\end{aligned}
\end{equation*}
where $f_N, g_N\in \hat C^\infty(\Bnbarre)$. Hence, 
\begin{equation*}
\frac{|\partial X_K|^2}{X_K^2|\partial h|^2} =  \frac{|\underline\partial X_K|^2}{X_K^2|\underline \partial h|^2} = \frac{\left(\frac{2}{s} + \partial_s f_N \right)^2 + (\partial_\chi f_N)^2}{\left(\frac{2}{s} + \partial_s g_N \right)^2 + (\partial_\chi g_N)^2}.
\end{equation*}
The latter goes to $1$ when $(s, \chi)\to (0, 0)$. 
\end{enumerate}
\item Now, we prove that there exist $c, C>0$ such that $\forall (\rho, z)\in \Bbarre$, we have 
\begin{equation*}
c e^h\left|\partial h\right| \leq \left|\partial X_K\right| \leq C e^h\left|\partial h\right|. 
\end{equation*}
We write $\forall (\rho, z)\in \Bbarre$:
\begin{equation*}
\frac{|\partial X_K|}{e^h|\partial h|} = \frac{|\partial X_K|}{X_K|\partial h|} X_K e^{-h} = e^{x_K}\frac{|\partial X_K|}{X_K|\partial h|}. 
\end{equation*}
By Lemma \ref{x::K:reg},  $x_K\in \hat C^\infty(\Bbarre)$. Moreover, we have 
\begin{equation*}
\frac{|\partial X_K|}{X_K} \sim |\partial h| \quad \text{near}\; \Horizon , \Axis, p_{N}, p_{S}
\end{equation*}
by the first point. Hence 
\begin{equation*}
|\partial X_K| \sim e^h|\partial h| \quad\text{on}\quad \Bbarre. 
\end{equation*}
\end{enumerate}
\end{proof}

\begin{lemma}
\label{asymptotics:h2}
There exist smooth vector fields $e_A$, $e_N$ and $e_S$ defined on $\BAbarre$, $\Bnbarre$ and $\Bsbarre$ respectively, which all extend to smooth vector fields on $\Bbarre$, such that 
\begin{equation*}
    \partial h_{\BAbarre\cap\Axis} = \frac{2}{\rho}\partial_\rho+e_A, \quad \underline{\partial}h_{\Bbarre_N} = \frac{2}{s}\underline{\partial}_s+e_s, \quad 
    \underline{\partial}h_{\Bbarre_S} = \frac{2}{s'}\underline{\partial}_s'+e_s,
\end{equation*}
Moreover, $e_A$ verifies the decay estimate at infinity:
\begin{equation}
\label{asymptotics:eA}
\left|e_A\right| = O_{r\to +\infty}\left( \frac{1}{r^2}\right).
\end{equation}

\end{lemma}

\begin{proof}
\begin{enumerate}
\item Let $\tilde \Axis\subset \BAbarre$ be a neighbourhood of the axis and let $(\rho, z)\in \tilde\Axis $. Then 
\begin{equation*}
\begin{aligned}
\left. \partial\right|_{ \tilde\Axis} &= 2\left(\frac{1}{\rho}\;\;0\right)^t + \partial g_\Axis \\
&= 2\frac{1}{\rho}\partial_\rho + e_A, 
\end{aligned}
\end{equation*}
where $e_A := \partial g_\Axis\in \hat C^{\infty}(\BAbarre)$ and which can be extended smoothly to $\Bbarre$. 

By \eqref{hA:exp}, we have 
\begin{equation*}
\begin{aligned}
\left|\partial g_\Axis\right|^2 &= \left|\partial\left(\log\left(\frac{\sqrt{\rho^2+(z+\beta)^2}+(z+\beta)}{\sqrt{\rho^2+(z-\beta)^2}+(z-\beta)}\right) \right)\right|^2. \\
\end{aligned}
\end{equation*}
Set 
\begin{equation*}
P_\pm(\rho, z) := \sqrt{\rho^2+(z\pm \beta)^2}+(z\pm \beta).
\end{equation*}
Then 
\begin{equation*}
\begin{aligned}
\left|\partial g_\Axis\right|^2 &= \frac{P^2_-(\rho, z)}{P^2_+(\rho, z)}\frac{1}{P^4_-(\rho, z)}\left(|\partial P_+(\rho, z)|^2P^2_-(\rho, z) + |\partial P_-(\rho, z)|^2P^2_+(\rho, z)  \right.\\
&\left. - 2P_+(\rho, z)P_-(\rho, z)\partial P_+(\rho, z)\cdot \partial P_-(\rho, z)\right). 
\end{aligned}
\end{equation*}
When $r = \sqrt{1 + \rho^2 + z^2}\to \infty$, we have 
\begin{equation*}
\frac{P^2_-(\rho, z)}{P^2_+(\rho, z)} = O(1)\quad\text{and}\quad \frac{1}{P^4_-(\rho, z)} = O(r^{-4}).
\end{equation*}
Moreover, 
\begin{equation*}
|\partial P_\pm(\rho, z)|^2 = \left(\frac{\rho}{\rho^2+(z\pm \beta)^2}\right)^2 + \left(\frac{z}{\rho^2+(z\pm \beta)^2} + 1\right)^2 = O(r^{-2})
\end{equation*}
and 
\begin{equation*}
\left| \partial P_+(\rho, z)\cdot \partial P_-(\rho, z) \right| = O(r^{-2}). 
\end{equation*}
Therefore, 
\begin{equation*}
\left|\partial g_\Axis\right| = O_{r\to\infty}\left(\frac{1}{r^2}\right). 
\end{equation*}
\item The expansion near $p_N$ and near $p_S$ is obtained in the same manner. 
\end{enumerate}
\end{proof}

\subsection{Regularity of the matter terms $F_i$}
First of all, we recall that $F_i$ are given by
\begin{equation}
\label{F::11}
F_1(W, X, \sigma)(\rho, z):= -e^{2\lambda}\frac{2\pi\rho}{\sigma}\int_{D(\rho, z)}(X + 2(\rho L)^2)\Phi(E + \rho\omega L, \rho L)\Psi_\eta(\rho, (E + \rho\omega L, \rho L), (X, W, \sigma))\,dE dL,
\end{equation}
\begin{equation}
\label{F::22}
F_2(W, X, \sigma)(\rho, z):= \frac{2\pi\rho}{\sigma}X\int_{D(\rho, z)}\rho L E\Phi(E + \rho\omega L, \rho L)\Psi_\eta(\rho, (E + \rho\omega L, \rho L), (X, W, \sigma))\,dE dL,
\end{equation}
\begin{equation}
\label{F::33}
\begin{aligned}
F_3(W, X, \sigma)(\rho, z)&:= \frac{2\pi\rho^3\sigma}{X^2}\int_{D(\rho, z)}\left(\tilde L^2 - L^2\right)\Phi(E + \rho\omega L, \rho L)\Psi_\eta(\rho, (E + \rho\omega L, \rho L), (X, W, \sigma))\,dE dL, 
\end{aligned}
\end{equation}
\begin{equation}
\label{F::44}
\begin{aligned}
F_4(W, X, \sigma, \lambda)(\rho, z):= -\frac{4\pi e^{2\lambda}\rho}{\sigma}&\int_{D(\rho, z)}\left(\frac{X^2}{\rho^2\sigma^2}E^2 + \left(1 - \frac{X}{\rho^2}\right)\left( 1 + \frac{\rho^2}{X}L^2\right)\right)\Phi(E + \rho\omega L, \rho L) \\
&\Psi_\eta(\rho, (E + \rho\omega L, \rho L), (X, W, \sigma))\,dE dL,  \\
\end{aligned}
\end{equation}

\label{regularity}
\begin{lemma}
Let $\delta_0>0$ be given by Proposition \ref{Pert:kerr} and let $h := \left(\thetazero, \sigmazero, \Xzero \right)\in B_{\delta_0}\subset \Ltheta\times\Lsigma\times\LX$. Then, the matter terms $F_i(h)$ as defined in \eqref{F::11}-\eqref{F::44} are well-defined on $\Bbarre$. 
\end{lemma}
\begin{proof}
By Proposition \ref{support:f}, we have 
\begin{equation*}
\left(\BHbarre\cup\Bnbarre\cup\Bsbarre\right)\cap \supp_{(\rho, z)} f = \emptyset 
\end{equation*}
and 
\begin{equation*}
\supp_{(\rho, z)} f\subset\subset\BAbarre. 
\end{equation*}
Now let $(\rho, z)\in \BAbarre$. If $(\rho, z)\in \tilde\Axis$ where $\tilde\Axis$ is some open neighbourhood of the axis, then by Proposition \ref{extendibility}, there exists a smooth function $X_\Axis: \tilde\Axis\mapsto\mathbb R$ such that $\displaystyle X_\Axis(0, z)>0$ and 
\begin{equation}
\label{ext:kerr:axis}
X_K(\rho, z) = \rho^2X_\Axis(\rho^2, z). 
\end{equation}
Hence, near the axis,  $\tilde L$  is given by 
\begin{equation*}
\begin{aligned}
\tilde L(E, \Xzero, \sigmazero, \rho, z) &= \frac{\sqrt {X_K}}{\rho}\sqrt{1 + \Xzero}\left(-1 + \frac{X_K}{\rho^2}\frac{1 + \Xzero}{(1 + \sigmazero)^2}E^2  \right)^{\frac{1}{2}} \\
&= \sqrt {X_\Axis}\sqrt{1 + \Xzero}\left(-1 + X_\Axis\frac{1 + \Xzero}{(1 + \sigmazero)^2}E^2  \right)^{\frac{1}{2}}
\end{aligned}
\end{equation*}
which is well defined on $\tilde\Axis$. The term $\displaystyle \frac{\sigma}{\sqrt X}$ is also well defined on $\tilde\Axis$ since
\begin{equation*}
\frac{\sigma}{\sqrt X} = \frac{\rho(1 + \sigmazero)}{\sqrt{X_K(1 + \Xzero)}} = \frac{1}{\sqrt {X_\Axis}}\frac{1 + \sigmazero}{\sqrt{1 + \Xzero}}. 
\end{equation*}
Therefore, the matter terms $F_i(\thetazero, \Xzero, \sigmazero)$ are well defined on $\tilde\Axis$, thus on $\BAbarre$.
\\ Now, if $(\rho, z)\in \left(\BHbarre\cup\Bnbarre\cup\Bsbarre\right)$, then  $F_i(\thetazero, \Xzero,  \sigmazero)$ vanish. Hence, they are  well defined. 
\end{proof}
\subsubsection{Further computations of the matter terms $F_i$}
\label{Further:computations}
In this section, we compute explicitly the intersection of $D(\rho, z)$ with $\supp\Phi$. To this end, we will have to distinguish between direct and retrograde orbits. In this case, we write 
\begin{equation}
f(x, v) = \left( \Phi_+(\ve, \ell_z) + \Phi_-(\ve, \ell_z)\right)\Psi_\eta(\rho, (\ve, \ell_z), h)
\end{equation}
where $\Phi_+$ is supported on $\Bbound^+$ and  $\Phi_-$ is supported on $\Bbound^-$. 
\\Let $(\rho, z)\in\BAbarre$. Recall the definition of $D(\rho, z)$: 
\begin{equation*}
D(\rho, z) = \left\{ (E , L)\;:\; E\geq \frac{\sigma}{\sqrt X} \quad\text{and }\quad |L| \leq \tilde L(E, X, \sigma, \rho, z) \right\} 
\end{equation*}
where 
\begin{equation}
\label{L:::tilde}
\tilde L(E, X, \sigma, \rho, z) :=  \frac{\sqrt{X}}{\rho}\left( - 1 + \frac{X}{\sigma^2}E^2 \right)^{\frac{1}{2}}. 
\end{equation}
Now, let $(E, L)\in D(\rho, z)$. Then, 
\begin{equation}
E\geq\frac{\sigma}{\sqrt X} \quad\text{and}\quad 0 \leq \frac{\rho^2}{X}L^2\leq \frac{X}{\sigma^2}E^2 - 1.
\end{equation}
This is equivalent to 
\begin{equation}
E\geq\frac{\sigma}{\sqrt X} \quad\text{and}\quad \frac{\sigma^2}{X} \leq \frac{\sigma^2}{X}\left(\frac{\rho^2}{X}L^2 + 1\right) \leq E^2.
\end{equation}
Thus, 
\begin{equation*}
E \geq \tilde E(L, X, \sigma, \rho, z)
\end{equation*}
where $\tilde E$ is defined by 
\begin{equation}
\label{E:tilde}
\tilde E(L, X, \sigma, \rho, z) := \frac{\sigma}{\sqrt X}\sqrt{1 + \frac{\rho^2}{X}L^2}. 
\end{equation}
Now, we introduce the following change of variables  
\begin{equation}
\begin{aligned}
\xi&:= E + \rho\omega s ,\\
s&= L.   
\end{aligned}
\end{equation}
Therefore\footnote{Recall that $E_{\ell_z}(h, \cdot, \cdot)$ is the effective potential energy of a particle with angular momentum $\ell_z$ defined by \eqref{E:ellz:pert}.}, 
\begin{equation*}
D(\rho, z) = \left\{ (\xi , s)\;:\; s\in\mathbb R \quad\text{and }\quad \xi\geq \tilde E(s, X, \sigma, \rho, z) + \rho\omega s = E_{\rho s}((\thetazero, \Xzero, \sigmazero), \rho, z)\right\}
\end{equation*}
and the matter terms are given by 
\begin{equation*}
F_1(\thetazero, \Xzero, \sigmazero)(\rho, z)= -\frac{2\pi e^{2\left(\lambdazero + \lambda_K\right)}}{1 + \sigmazero}\int_{D(\rho, z)}(X_K(1 + \Xzero) + 2(\rho s)^2)\Phi(\xi, \rho s)\Psi_\eta(\rho, (\xi, \rho s), (\thetazero, \Xzero, \sigmazero))\,d\xi ds,
\end{equation*}
\begin{equation*}
F_2(\thetazero, \Xzero, \sigmazero)= \frac{2\pi X_K(1 + \Xzero)}{1 + \sigmazero}\int_{D(\rho, z)}\rho s (\xi - \rho\left(-\thetazero + \omega_K\right) s)\Phi(\xi, \rho s)\Psi_\eta(\rho, (\xi, \rho s), (\thetazero, \Xzero, \sigmazero))\,d\xi ds,
\end{equation*}
\begin{equation*}
\begin{aligned}
F_3(\thetazero, \Xzero, \sigmazero)(\rho, z)= \frac{\rho^4}{X_K^2}\frac{2\pi(1 + \sigmazero)}{(1 + \Xzero)^2}&\int_{D(\rho, z)}\left(\tilde L^2\left(\xi-\rho\left(-\thetazero + \omega_K\right) s, X_K(1 + \Xzero ), \sigma, \rho, z\right) - s^2\right)\Phi(\xi, \rho s) \\
&\Psi_\eta(\rho, (\xi, \rho s), (\thetazero, \Xzero, \sigmazero))\,d\xi ds, 
\end{aligned}
\end{equation*}
\begin{equation*}
\begin{aligned}
F_4(\thetazero, \Xzero, \sigmazero)(\rho, z)&= -\frac{4\pi e^{2\left(\lambdazero + \lambda_K\right)}}{1 + \sigmazero}\int_{D(\rho, z)}\left(\frac{X_K^2(1 + \Xzero)^2}{\rho^4\left(1 + \sigmazero\right)^2}\left(\xi - \rho\left(-\thetazero + \omega_K\right) s\right)^2  \right.  + \\
&\left. \left(1 - \frac{X_K(1 + \Xzero)}{\rho^2}\right)\left( 1 + \frac{\rho^2}{X_K(1 + \Xzero)}s^2\right)\right)\Phi(\xi, \rho s)\Psi_\eta(\rho, (\xi, \rho s), (\thetazero, \Xzero, \sigmazero))\,d\xi ds. 
\end{aligned}
\end{equation*}
Now, we compute $\displaystyle D(\rho, z)\cap\supp\Phi(\cdot, \rho\cdot)$ provided $(\rho, z)\in\Bbarre$. 
Since the matter shell contains retrograde and direct orbits, we introduce 
\begin{equation}
\label{D:plus}
D^+(\rho, z) = \left\{ (\xi, s)\in D(\rho, z)\;:\; \omega s\geq 0 \right\},  
\end{equation}

\begin{equation}
\label{D:moins}
D^-(\rho, z) = \left\{ (\xi, s)\in D(\rho, z)\;:\; \omega s\leq 0 \right\},  
\end{equation}
Hence, 
\begin{equation*}
D(\rho, z)\cap\supp\Phi(\cdot, \rho\cdot) = \left( D^+(\rho, z)\cap \Bbound^+ \right)\cup \left( D^-(\rho, z)\cap \Bbound^- \right)
\end{equation*}
We state the following result:
\begin{lemma}
Let $(\rho, z)\in \BAbarre$. Then,
\begin{itemize}
\item if $\rho = 0$, then 
\begin{equation}
D(\rho, z)\cap\supp\Phi(\cdot, \rho\cdot)  = \emptyset. 
\end{equation}
\item Otherwise, 
\begin{equation*}
\begin{aligned}
\supp\Phi_+(\cdot, \rho\cdot) \cap D^+(\rho, z) &\subset \left\{ (\xi, s)\;:\; E_{\rho s}(h, \rho, z)\leq \xi \leq \ve^+_2\quad\text{and}\quad \tilde L(\ve^+_2, h, \rho, z)\geq s \geq \frac{\ell_1^+}{\rho}\right\},  \\
\supp\Phi_-(\cdot, \rho\cdot) \cap D^-(\rho, z) & \subset \left\{ (\xi, s)\;:\; E_{\rho s}(h, \rho, z)\leq \xi \leq \ve^-_2\quad\text{and}\quad -\tilde L(\ve^-_2 - \omega\ell_1^-, h, \rho, z)\leq s \leq \frac{\ell_2^-}{\rho}\right\},  \\
\end{aligned}
\end{equation*}
\end{itemize}
\end{lemma}
\begin{proof}
Let $\displaystyle (\xi, s)\in D(\rho, z)\cap\supp\Phi(\cdot, \rho\cdot)$. Then, 
\begin{equation*}
\xi \geq E_{\rho s}(h, (\rho, z)) \quad\text{and}\quad (\xi, \rho s)\in \Bbound. 
\end{equation*}
\begin{itemize}
\item If $\rho = 0$. Then, $(\xi, \rho s)\notin \Bbound$. Hence, $\displaystyle D(\rho, z)\cap\supp\Phi(\cdot, \rho\cdot) = \emptyset$.  
\item Otherwise, 

\begin{enumerate}
\item If $\displaystyle (\xi, s)\in D^+(\rho, z)\cap\supp\Phi_+(\cdot, \rho\cdot)$, then 
\begin{equation*}
\begin{aligned}
\omega s\geq 0\quad\text{and}&\quad \xi \geq E_{\rho s}(h, \rho, z),  \\
 \ve_1^+\leq\ve\leq\ve_2^+ &\quad\text{and}\quad \frac{\ell_1^+}{\rho} \leq s \leq \frac{\ell_2^+}{\rho}. 
\end{aligned}
\end{equation*}
This implies
\begin{equation*}
E_{\rho s}(h, \rho, z) = \tilde E(s, h, \rho, z) + \omega\rho s\leq \ve_2^+.
\end{equation*}
Since $\omega s\geq 0$, we obtain 
\begin{equation*}
\tilde E(s, h, \rho, z) \leq  \ve_2^+. 
\end{equation*}
Since $\tilde E$ is monotonically increasing on $\mathbb R_+$ with respect to $s$, we obtain a upper bound on $s$: 
\begin{equation*}
s \leq \tilde L (\ve_2^+, h, \rho, z). 
\end{equation*}
Finally, we recall that we also have an upper bound on $s$, given by $\displaystyle \frac{\ell_1}{\rho}$, which provides the desired result. 
\item Similarly, if $\displaystyle (\xi, s)\in D^-(\rho, z)\cap\supp\Phi_-(\cdot, \rho\cdot)$, then, 
\begin{equation*}
\begin{aligned}
\omega s\leq 0\quad\text{and}&\quad \xi \geq E_{\rho s}(h, \rho, z),  \\
 \ve_1^-\leq\xi\leq\ve_2^- &\quad\text{and}\quad \frac{\ell_1^-}{\rho} \leq s \leq \frac{\ell_2^-}{\rho}. 
\end{aligned}
\end{equation*}
This implies, 
\begin{equation*}
\omega\ell_2^- \leq -\omega\rho s \leq \omega\ell_1^- \quad\text{and}\quad E_{\rho s}(h, \rho, z)\leq \ve_2^-. 
\end{equation*}
Hence, 
\begin{equation*}
\tilde E(s, h, \rho, z) \leq \ve^-_2 - \omega\rho s \leq \ve^-_2 - \omega\ell_1^-. 
\end{equation*}
Since$\tilde E$ is monotonically decreasing on $\mathbb R_-$ with respect to $s$, we obtain a lower bound on $s$: 
\begin{equation*}
s \geq -\tilde L (\ve_2^- - \omega\ell_1^-, h, \rho, z)
\end{equation*}
and we also have 
\begin{equation*}
s\leq \frac{\ell_2^-}{\rho}. 
\end{equation*}
This ends the proof. 
\end{enumerate}
\end{itemize}
\end{proof}
Finally, we set $f_{\phi_\pm, \eta}^{i, \pm}$ to be 

\begin{align}
f_{\Phi_\pm, \eta}^{1, \pm}(s, h, \rho, z) &:= \int_{E_{\rho s}(h, \rho, z)}^{\ve^\pm_2}(X_K(1 + \Xzero) + 2(\rho s)^2)\Phi_\pm(\xi, \rho s)\Psi_\eta(\rho, (\xi, \rho s), h)\,d\xi , \\
f_{\Phi_\pm, \eta}^{2, \pm}(s, h, \rho, z) &:= \int_{E_{\rho s}(h, \rho, z)}^{\ve^\pm_2} \rho s (\xi - \rho\left(-\thetazero + \omega_K\right) s)\Phi_\pm(\xi, \rho s)\Psi_\eta(\rho, (\xi, \rho s), h)\,d\xi  , \\
f_{\Phi_\pm, \eta}^{3, \pm}(s, h, \rho, z) &:= \int_{E_{\rho s}(h, \rho, z)}^{\ve^\pm_2} \left(\tilde L^2\left(\xi-\rho\left(-\thetazero + \omega_K\right) s, X_K(1 + \Xzero ), \sigma, \rho, z\right) - s^2\right)\Phi_\pm(\xi, \rho s)\Psi_\eta(\rho, (\xi, \rho s), h)\,d\xi, \\
f_{\Phi_\pm, \eta}^{4, \pm}(s, h, \rho, z) &:= \int_{E_{\rho s}(h, \rho, z)}^{\ve^\pm_2} g^{4}(s, h, \rho, z, \xi)\Phi_\pm(\xi, \rho s)\Psi_\eta(\rho, (\xi, \rho s), h)\,d\xi, \;\text{where} \\
g^{4}(s, h, \rho, z, \xi)&:= \left(\frac{X_K^2(1 + \Xzero)^2}{\rho^4\left(1 + \sigmazero\right)^2}\left(\xi - \rho\left(-\thetazero + \omega_K\right) s\right)^2 +  \left(1 - \frac{X_K(1 + \Xzero)}{\rho^2}\right)\left( 1 + \frac{\rho^2}{X_K(1 + \Xzero)}s^2\right)\right) 
\end{align}
Hence, $F_i(h)$ are given by: 
\begin{itemize}
\item $\displaystyle \forall (\rho, z)\in \BB_A$
\begin{equation}
F_i(h, \lambdazero)(\rho, z):= F_+^i(h, \lambdazero, \rho, z) + F_-^i(h, \lambdazero, \rho, z),
\end{equation}
where
\begin{align}
F_+^1(h, \rho, z)  &:= -\frac{2\pi e^{2\left(\lambdazero + \lambda_K\right)}}{1 + \sigmazero}\int_{\frac{\ell_1^+}{\rho}}^{\tilde L(\ve_2^+, h, \rho, z)}\;f_{\Phi_+, \eta}^{1, +}(s, h, \rho, z)\, ds , \\ 
F_+^2(h, \rho, z)  &:=  \frac{2\pi X_K(1 + \Xzero)}{1 + \sigmazero}\int_{\frac{\ell_1^+}{\rho}}^{\tilde L(\ve_2^+, h, \rho, z)}\;f_{\Phi_+, \eta}^{2, +}(s, h, \rho, z)\, ds , \\
F_+^3(h, \rho, z)  &:=  \frac{\rho^4}{X_K^2}\frac{2\pi(1 + \sigmazero)}{(1 + \Xzero)^2}\int_{\frac{\ell_1^+}{\rho}}^{\tilde L(\ve_2^+, h, \rho, z)}\;f_{\Phi_+, \eta}^{3, +}(s, h, \rho, z)\, ds , \\
F_+^4(h, \lambdazero, \rho, z)  &:= -\frac{4\pi e^{2\left(\lambdazero + \lambda_K\right)}}{1 + \sigmazero}\int_{\frac{\ell_1^+}{\rho}}^{\tilde L(\ve_2^+, h, \rho, z)}\;f_{\Phi_+, \eta}^{4, +}(s, h, \rho, z)\, ds ,
\end{align}

\begin{align}
F_-^1(h, \rho, z)  &:= -\frac{2\pi e^{2\left(\lambdazero + \lambda_K\right)}}{1 + \sigmazero}\int^{\frac{\ell_1^-}{\rho}}_{-\tilde L(\ve_2^-, h, \rho, z)}\;f_{\Phi_-, \eta}^{1, -}(s, h, \rho, z)\, ds, \\
F_-^2(h, \rho, z)  &:=  \frac{2\pi X_K(1 + \Xzero)}{1 + \sigmazero}\int^{\frac{\ell_1^-}{\rho}}_{-\tilde L(\ve_2^-, h, \rho, z)}\;f_{\Phi_-, \eta}^{2, -}(s, h, \rho, z)\, ds, \\ 
F_-^3(h, \rho, z)  &:=  \frac{\rho^4}{X_K^2}\frac{2\pi(1 + \sigmazero)}{(1 + \Xzero)^2}\int^{\frac{\ell_1^-}{\rho}}_{-\tilde L(\ve_2^-, h, \rho, z)}\;f_{\Phi_-, \eta}^{3, -}(s, h, \rho, z)\, ds , \\
F_-^4(h, \lambdazero, \rho, z)  &:= -\frac{4\pi e^{2\left(\lambdazero + \lambda_K\right)}}{1 + \sigmazero}\int^{\frac{\ell_1^-}{\rho}}_{-\tilde L(\ve_2^-, h, \rho, z)}\;f_{\Phi_-, \eta}^{4, -}(s, h, \rho, z)\, ds. 
\end{align}

\item Otherwise, they vanish. 
 \end{itemize}

Following the previous lemma, we obtain 
\begin{lemma}
\label{vanish:near:A}
 $\forall h\in B_{\delta_0}\left(\Ltheta\times\LX\times\Lsigma \right)$, there exists an open neighbourhood of the axis of symmetry $\Axis$,  say $\tilde \Axis\subset\BAbarre$,  such that the matter terms $F^i(h, \cdot)$ vanish in $\tilde\Axis$.
 \end{lemma}
 
 \begin{proof}
 Recall that $\Axis \subset \left\{(\rho, z)\,:\; \rho = 0 \right\}$. By Proposition \ref{support:f}, we have $$\supp_{(\rho, z)} f\subset\subset \left[\rho_{min}(h), \rho_{max}(h)\right]\times\left[Z_{min}(h), Z_{max}(h) \right].$$
 \noindent Since $\rho_{min}(h)>0$, we choose any  open subset of $\BAbarre\cap([0, \rho_{min}(h)[\times\mathbb R)$. 
 \end{proof}
\noindent Before we study the regularity of the matter terms, we state the following lemmas from which we obtain the regularity of the different matter terms. 
 
 \begin{lemma}
 \label{diff:E:u}
 Let $u\in\mathbb R$ and define the mapping $\displaystyle E_{u}: B_{\delta_{0}}\subset \Ltheta\times\LX\times\Lsigma\to\hat C^{2, \alpha}(\Bbarre)$ by 
 \begin{equation*}
 E_{u}(\thetazero, \Xzero, \sigmazero)(\rho, z):= E_{\rho u}(\thetazero, \Xzero, \sigmazero, \rho, z) = -\thetazero\rho u - W_KX_K^{-1}\rho u + \frac{\rho}{X_K}\frac{\sigmazero + 1}{\Xzero + 1}\sqrt{(\rho u)^2 + X_K(\Xzero+1)}. 
 \end{equation*}
 Then $E_{u}$ is well defined on $B_{\delta_{0}}$ and it is continuously Fréchet differentiable on $B_{\delta_{0}}$  with derivative: $\displaystyle \forall (\thetazero, \Xzero, \sigmazero)\in B_{\delta_{0}}\,,\, \forall (\tilde\Theta, \tilde X, \tilde\sigma)\in \Ltheta\times\LX\times\Lsigma\,,\,\forall (\rho, z)\in\Bbarre$
 \begin{equation}
 \label{diff:Eu}
 \begin{aligned}
 &DE_{u}(\thetazero, \Xzero, \sigmazero)[\tilde\Theta, \tilde X, \tilde\sigma](\rho, z) =  \\
 &-\rho u\tilde\Theta + \frac{\rho}{X_{K}}\frac{\tilde\sigma}{1 + \Xzero}\sqrt{\rho^2u^2 + X_{K}(1 + \Xzero)} + \frac{\rho}{X_{K}}\tilde X (1 + \sigmazero) \\
 &\left(-\frac{\sqrt{ \rho^2u^2 + X_{K}(1 + \Xzero)}}{(1 + \Xzero)^2} + \frac{X_{K}}{2(1 + \Xzero)\sqrt{ \rho^2u^2 + X_{K}(1 + \Xzero)}} \right)  
 \end{aligned}
 \end{equation}
 \end{lemma}
 
 \begin{proof}
 Let $u\in \mathbb R$ and define, for $x\in]-\delta_0, \delta_0[^3, (\rho, z)\in \Bbarre$, 
 \begin{equation*}
 \mathscr E_u(x, (\rho, z)) := x_1\rho u - W_KX_K^{-1}(\rho, z)\rho u + \frac{\rho}{X_K(\rho, z)}\frac{1 + x_2}{1 + x_3}\sqrt{(\rho u)^2 + X_K(\rho, z)(x_2 + 1)}.  
 \end{equation*}
 By Proposition \ref{extendibility}, $\forall u\in \mathbb R\,,\; \mathscr E_u$ is well-defined on $\Bbarre$ and it is smooth on $]-\delta_0, \delta_0[^3\times \Bbarre$. Moreover, $\forall u\in\mathbb R, \forall (x, (\rho, z))\in ]-\delta_0, \delta_0[^3\times \Bbarre\; :$
 \begin{equation*}
 \partial_x\mathscr E_u(x, (\rho, z)) =  \left(
 \begin{aligned}
 &\rho u \\
 &\frac{\rho}{X_K(\rho, z)}\frac{1}{1 + x_3}\left(\sqrt{(\rho u)^2 + X_K(\rho, z)(x_2 + 1)} + \frac{(1 + x_2)X_K(\rho, z)}{2\sqrt{(\rho u)^2 + X_K(\rho, z)(x_2 + 1)}} \right) \\
 &-\frac{\rho}{X_K(\rho, z)}\frac{1 + x_2}{(1 + x_3)^2}\sqrt{(\rho u)^2 + X_K(\rho, z)(x_2 + 1)}
 \end{aligned}
 \right)
 \end{equation*}
Now, we set
  \begin{equation*}
 DE_{u}(\thetazero, \Xzero, \sigmazero)[\tilde\Theta, \tilde X, \tilde\sigma](\rho, z) = ((\tilde\Theta(\rho, z) \;\; \tilde X(\rho, z)\;\;  \tilde\sigma(\rho, z))\partial_x\mathscr E_u((\thetazero, \Xzero, \sigmazero)(\rho, z), (\rho, z)).
 \end{equation*}
 Set $h = (\thetazero, \Xzero, \sigmazero)$ and $\tilde h = (\tilde\Theta, \tilde X, \tilde\sigma)$. We will show that for
 \begin{equation*}
 \lim_{||\tilde h||_{\Ltheta\LX\times\Lsigma}\to 0 }\frac{||E_u(h + \tilde h) - E_u(h) - DE_{u}(h)[\tilde h] ||_{\hat C^{2, \alpha}(\Bbarre)}}{||\tilde h||_{\Ltheta\times\LX\times\Lsigma}}
 \end{equation*}
 \noindent $\forall (\rho, z)\in \Bbarre$, we have 
 \begin{equation*}
 \begin{aligned}
 &E_u(h + \tilde h)(\rho, z) - E_u(h)(\rho, z) - DE_{u}(h)[\tilde h](\rho, z) =  \\
 &\mathscr E_u(h(\rho, z) + \tilde h(\rho, z), (\rho, z)) - \mathscr E_u(h(\rho, z), (\rho, z)) - \tilde h(\rho, z)\cdot  \partial_x\mathscr E_u(h(\rho, z), (\rho, z)).
 \end{aligned}
 \end{equation*}
 Denote by $'$ the differential with respect to $(\rho, z)$, then 
  \begin{equation*}
 \begin{aligned}
 &\left( E_u(h + \tilde h)(\rho, z) - E_u(h)(\rho, z) - DE_{u}(h)[\tilde h](\rho, z)\right)' =  \\
 &(\partial_x\mathscr E_u(h(\rho, z) + \tilde h(\rho, z), (\rho, z)))^t(h'(\rho, z) + \tilde h'(\rho, z)) +  (\mathscr E_u)'(h(\rho, z) + \tilde h(\rho, z), (\rho, z))\\
 &- (\partial_x\mathscr E_u(h(\rho, z), (\rho, z)))^t(h'(\rho, z)) -  (\mathscr E_u)'(h(\rho, z), (\rho, z)) \\
 &- (\partial_x\mathscr E_u(h(\rho, z), (\rho, z)))^t\tilde h'(\rho, z) - \tilde h(\rho, z)(\partial_x\mathscr E_u)'(h(\rho, z), (\rho, z)) - \tilde h(\rho, z)\partial_{xx}\mathscr E_u(h(\rho, z), (\rho, z))h'(\rho, z). 
 \end{aligned}
 \end{equation*}
 \noindent In the same way, we compute the second derivatives. We provide only the second derivative of the term $E_u(h)(\rho, z)$ in order to detail the Hölder estimates. The other terms are estimated using similar arguments:
 \begin{equation*}
 \begin{aligned}
 E_u(h)''(\rho, z) &= ((\partial_x\mathscr E_u(h(\rho, z), (\rho, z)))^t(h'(\rho, z)))' + ((\mathscr E_u)'(h(\rho, z), (\rho, z)))' \\
 &=  \sum_{i = 1}^3(\partial_x\mathscr E_u)_i(h(\rho, z), (\rho, z))h''_i(\rho, z)  \\
 &+ h'(\rho, z)^t(\mathscr E_u)'(h(\rho, z), (\rho, z)) +((\partial_x\mathscr E_u)'(h(\rho, z), (\rho, z)))^th''(\rho, z) + (\mathscr E_u)''(h(\rho, z), (\rho, z)) \\
 &+ \sum_{i = 1}^3(h'(\rho, z))^t(\partial_{xx}\mathscr E_u)_i(h(\rho, z), (\rho, z))h'(\rho, z). 
 \end{aligned}
 \end{equation*}
 Now, we estimate the $C^{2, \alpha}(\Bbarre)$ norm of $E_u(h + \tilde h) - E_u(h) - DE_{u}(h)[\tilde h] $
 \begin{enumerate}
 \item Since $\mathscr E_u$ is smooth and bounded with respect to $x$, there exists $C = C(u, \delta_0)$ such that $\forall (\rho, z)\in\Bbarre$
 \begin{equation*}
 \left| E_u(h + \tilde h)(\rho, z) - E_u(h)(\rho, z) - DE_{u}(h)[\tilde h](\rho, z)\right| \leq C ||\tilde h||^2_{\Ltheta\times\LX\times\Lsigma}. 
 \end{equation*}
 \item Moreover,
 \begin{equation*}
 \begin{aligned}
& (\mathscr E_u)'(h(\rho, z) + \tilde h(\rho, z), (\rho, z))  - (\mathscr E_u)'(h(\rho, z), (\rho, z)) - \tilde h(\rho, z)(\partial_x\mathscr E_u)'(h(\rho, z), (\rho, z)) =  \\
& (\mathscr E_u)'(h(\rho, z) + \tilde h(\rho, z), (\rho, z))  - (\mathscr E_u)'(h(\rho, z), (\rho, z)) - \tilde h(\rho, z)(\partial_x(\mathscr E_u)')(h(\rho, z), (\rho, z)) =  \\
& O(|\tilde h(\rho, z)|^2). 
 \end{aligned}
 \end{equation*}
 Since $\mathscr E_u$ is smooth and bounded with respect to $(\rho, z)$ and their derivatives with respect to $(\rho, z)$ are also bounded on $\Bbarre$ by Proposition \ref{extendibility}, then there exists $C = C(u, \delta_0)>0$ such that $\forall (\rho, z)\in\Bbarre$, we have 
  \begin{equation*}
 \left|(\mathscr E_u)'(h(\rho, z) + \tilde h(\rho, z), (\rho, z))  - (\mathscr E_u)'(h(\rho, z), (\rho, z)) - \tilde h(\rho, z)(\partial_x\mathscr E_u)'(h(\rho, z), (\rho, z))\right| \leq C ||\tilde h||^2_{\Ltheta\times\LX\times\Lsigma}.   \\
 \end{equation*}
 \item Furthermore, 
 \begin{equation*}
 \begin{aligned}
&(\partial_x\mathscr E_u(h(\rho, z) + \tilde h(\rho, z), (\rho, z)))^t(h'(\rho, z) + \tilde h'(\rho, z)) - (\partial_x\mathscr E_u(h(\rho, z), (\rho, z)))^t(h'(\rho, z))  \\
&- (\partial_x\mathscr E_u(h(\rho, z), (\rho, z)))^t\tilde h'(\rho, z) - \tilde h(\rho, z)\partial_{xx}\mathscr E_u(h(\rho, z), (\rho, z))h'(\rho, z) = \\
& (\partial_x\mathscr E_u(h(\rho, z) + \tilde h(\rho, z), (\rho, z)) - \partial_x\mathscr E_u(h(\rho, z), (\rho, z)) - \partial_{xx}\mathscr E_u(h(\rho, z), (\rho, z))(\tilde h(\rho, z))^t)^t(h'(\rho, z)) \\
&+\left( \partial_x\mathscr E_u(h(\rho, z) + \tilde h(\rho, z), (\rho, z)) - \partial_x\mathscr E_u(h(\rho, z), (\rho, z)))^t(h'(\rho, z))\right) \\ 
& (\partial_x\mathscr E_u(h(\rho, z) + \tilde h(\rho, z), (\rho, z)) - \partial_x\mathscr E_u(h(\rho, z), (\rho, z)) - \partial_{xx}\mathscr E_u(h(\rho, z), (\rho, z))(\tilde h(\rho, z))^t)^t(\tilde h'(\rho, z)) \\
&+ (\tilde h(\rho, z)\partial_x\mathscr E_u(h(\rho, z), (\rho, z)))^t\tilde h'(\rho, z)  \\
& = (h'(\rho, z))O((\tilde h(\rho, z))^2) +  O(|\tilde h'(\rho, z)| |\tilde h(\rho, z)|)\\
 \end{aligned}
 \end{equation*}
Thus, there exists $C = C(u, \delta_0)>0$ such that $\forall (\rho, z)$, we have 
\begin{equation*}
\begin{aligned}
&\left|\left( E_u(h + \tilde h)(\rho, z) - E_u(h)(\rho, z) - DE_{u}(h)[\tilde h](\rho, z)\right)'  \right| \leq C  ||\tilde h||^2_{\Ltheta\times\LX\times\Lsigma}. 
\end{aligned}
\end{equation*}
\item The $C^{2, \alpha}$ estimates and the Hölder semi-norm of the second derivatives are obtained as above: we use the smoothness ad boundedness of $\mathscr E_u$ as well as its derivatives. 
 \end{enumerate}
 
 \noindent Therefore, $E_u$ is Fréchet-differentiable on $B_{\delta_0}$ and its Fréchet differential is given by 
It remains to show that $DE_{u}$ is continuous on $B_{\delta_0}$. Let $h_1, h_2\in B_{\delta_0}$ and let $\tilde h\in \Ltheta\times\LX\times\Lsigma$ such that $||\tilde h||_{\Ltheta\times\LX\times\Lsigma}\leq 1$. We have 
 \begin{equation*}
 \begin{aligned}
 ||DE_u(h_1)[\tilde h] -  DE_u(h_2)[\tilde h]||_{\hat C^{2, \alpha}(\Bbarre)} &= ||\tilde h\cdot(\partial_x \mathscr E_u(h_1(\cdot), \cdot) - \partial_x \mathscr E_u(h_2(\cdot), \cdot)) ||_{\hat C^{2, \alpha}(\Bbarre)} \\
 &\leq ||\tilde h||_{\Ltheta\times\LX\times\Lsigma}\;||(\partial_x \mathscr E_u(h_1(\cdot), \cdot) - \partial_x \mathscr E_u(h_2(\cdot), \cdot)) ||_{\hat C^{2, \alpha}(\Bbarre)}
 \end{aligned}
 \end{equation*}
 Since $\mathscr E_u$ is smooth, we have $\forall (\rho, z)\in \Bbarre, $
 \begin{equation*}
 \partial_x \mathscr E_u(h_1(\rho, z), (\rho, z)) - \partial_x \mathscr E_u(h_2(\rho, z), (\rho, z)) = \partial_{xx}\mathscr E_u(h_1(\rho, z), (\rho, z))\cdot(h_1(\rho, z) - h_2(\rho, z))^t + O(|h_1(\rho, z) - h_2(\rho, z)|^2)
 \end{equation*}
 Now, since $h_1\in B_{\delta_0}$ and by Proposition \ref{extendibility}, there exists $C = C(u, \delta_0)$ such that $\forall (\rho, z)\in \Bbarre$
 \begin{equation*}
 |\partial_{xx}\mathscr E_u(h_1(\rho, z), (\rho, z))| \leq C \quad\text{and}\quad O(|h_1(\rho, z) - h_2(\rho, z)|^2) = O(||h_1 - h_2||^2). 
 \end{equation*}
 Consequently, 
 \begin{equation*}
 ||(\partial_x \mathscr E_u(h_1(\cdot), \cdot) - \partial_x \mathscr E_u(h_2(\cdot), \cdot)) ||_{\hat C^{0}(\Bbarre)} \leq C(u, \delta_0)||h_1 - h_2||. 
 \end{equation*}
 In the same way, we estimate the $\hat C^1$, $\hat C^2$ norms and the Hölder part. 
 We conclude that $E_u$ is continuously Fréchet differentiable on $B_{\delta_0}$
 \end{proof}
 \begin{remark}
As in the previous lemma,  we will show  the Fréchet differentiabilty of functionals of $(\thetazero, \Xzero, \sigmazero)$ defined on $B_{\delta_0}(0)\subset \Ltheta\times \LX\times \Lsigma$ which also  depend (smoothly)  on $(\rho, z)\in \Bbarre$. The steps are similar to what has been done in the previous lemma and we only compute the differential in the proofs.  
  \end{remark}
 \begin{lemma}
 \label{diff:tilde:L}
 Let $E\in \mathbb R$ and consider the mapping $\tilde L(E, \cdot, \cdot)\,:\, B_{\delta_0}\subset (\LX\times \Lsigma)\to \widehat C^{2, \alpha}(\BAbarre)$ defined by 
 \begin{equation*}
 \tilde L(E, h, \rho, z) :=  \frac{\sqrt{X_K}}{\rho}\sqrt{1 + \Xzero}\left( - 1 + \frac{1 + \Xzero}{(1 + \sigmazero)^2}\frac{X_K}{\rho^2}E^2 \right)^{\frac{1}{2}}.
 \end{equation*}
 Then, $\tilde L(E, \cdot, \cdot)$ is well-defined and it is continuously Fréchet differentiable on $B_{\delta_0}$. 
 \end{lemma}
 \begin{proof}
 Let $E\in\mathbb R$. 
 \begin{enumerate}
 \item First, we introduce the mapping $g_E: ]-\delta_0, \delta_0[^2\times \BAbarre$
 \begin{equation*}
 g_E(x, y, \rho, z):= \frac{\sqrt{X_K}}{\rho}\sqrt{1 + x}\left( - 1 + \frac{1 + x}{(1 + y)^2}\frac{X_K}{\rho^2}E^2 \right)^{\frac{1}{2}}.
 \end{equation*}
 By Proposition \ref{extendibility}, $\forall (x, y)\in]-\delta_0, \delta_0[^2\,;\, g_E(x, y, \cdot)$ is well-defined on $\BAbarre$ and it is smooth. Moreover, $\forall(\rho, z)\in \BAbarre$, $g_E(\cdot, \rho, z)$ is smooth on $]-\delta_0, \delta_0[^2$ and  we have, 
 \begin{equation*}
 \forall (\rho, z)\in \BAbarre\;,\;   \nabla_{(x,y)} g_E(x, y, \rho, z) = \frac{\sqrt{X_K}}{\rho}\frac{(\nabla_{(x, y)}\left( -(1 + x) + \frac{(1 + x)^2}{(1 + y)^2}E^2\right)}{2\left( -(1 + x)^2 + \frac{(1 + x)^2}{(1 + y)^2}E^2\right)}.
 \end{equation*}
 \begin{equation*}
 \forall (x, y)\in ]-\delta_0, \delta_0[^2\;, \;    \nabla_{(\rho,z)} g_E(x, y, \rho, z) = \sqrt{1 + x}\frac{\nabla_{(\rho, z)} \left( -\frac{X_K}{\rho^2} + \frac{1 + x}{(1 + y)^2}\frac{X_K^2}{\rho^4}E^2\right)}{2\left( -\frac{X_K}{\rho^2} + \frac{1 + x}{(1 + y)^2}\frac{X_K^2}{\rho^4}E^2\right)}. 
 \end{equation*}

 \item Let $h = (\Xzero, \sigmazero)\in B_{\delta_0}$. First, we show that $\tilde L(E, h, \cdot)\in \widehat C^{2, \alpha}(\BAbarre)$.  We have  
 \begin{equation*}
 \tilde L(E, h, \rho, z) = g_E(\Xzero(\rho, z), \sigmazero(\rho, z), \rho, z). 
 \end{equation*} 
$g_E(x, \cdot)$ is smooth on $\BAbarre$ and $\Xzero, \sigmazero\in  \widehat C^{2, \alpha}(\BAbarre)$. Thus,  $\tilde L(E, h, \cdot)$ is well-defined and it also lies in $\widehat C^{2, \alpha}(\BAbarre)$. 
 \item Now,  $\tilde L(E, \cdot, \cdot)$ is Fréchet differentiable on $B_{\delta_0}$ with derivative $\forall h\in B_{\delta_0}$, $\forall \tilde h = (\tilde X, \tilde \sigma)\in\LX\times\Lsigma$
 \begin{equation}
 \label{diff:L:tilde}
 D_h \tilde L(E, h, \cdot)[\tilde h] (\rho, z) =  (\tilde X \;\;\;\tilde\sigma)\nabla_{(x,y)} g_E(\Xzero(\rho, z), \sigmazero(\rho, z), \rho, z). 
 \end{equation}
 \noindent  Moreover , by smoothness of $g_E$ with respect to $(x, y)$, there exists $C = C(\delta_0, E)$ such that $\forall \overline h \in \Ltheta\times\Lsigma\times \LX$
 \begin{equation*}
 \left| \left|  D_h \tilde L(E, h, \cdot)[\overline h]  \right|\right|_{\widehat C^{2, \alpha}(\BAbarre)} \leq C(\delta_0) ||h||. 
 \end{equation*}
 The latter follows by similar arguments in the proof of Lemma \ref{diff:E:u}
 \end{enumerate}
 \end{proof}

 \begin{lemma}
 \label{diff:fi}
 Let $s\in\mathbb R$ and assume that 
 \begin{itemize}
 \item $\Phi_\pm$ is compactly supported on $\Bbound^\pm$,
 \item $\forall \xi\in\mathbb R\;,\; \ell_z\to\Phi_\pm(\xi, \ell_z)$ lies in $C^{2, \alpha}(\mathbb R)$, 
 \item $\forall h = (\thetazero, \Xzero, \sigmazero)\in \Ltheta\times\LX\times\Lsigma, \forall\xi\in\mathbb R\;,\; (\rho, \ell_z)\to\Psi_\eta(\rho, (\xi, \ell_z), h)$ lies in $C^{2, \alpha}(\mathbb R^2)$.  
 \end{itemize}
 Then, $\forall i\in\left\{1, 2 \right\}$, 
 \begin{equation}
 \begin{aligned}
 f^{i, \pm}_{\Phi, \eta}(s, \cdot, \cdot): &B_{\delta_0}\to \hat C^{2, \alpha}(\Bbarre) \\
 & h= (\thetazero, \Xzero, \sigmazero)\mapsto  f^{i, \pm}_{\Phi, \eta}(s, h, \cdot)
 \end{aligned}
 \end{equation}
 
 and $\forall j\in\left\{3, 4 \right\}$, 
 \begin{equation}
 \begin{aligned}
 f^{j, \pm}_{\Phi, \eta}(s, \cdot, \cdot): &B_{\delta_0}\to \hat C^{2, \alpha}(\BAbarre) \\
 & h= (\thetazero, \Xzero, \sigmazero)\mapsto  f^{j, \pm}_{\Phi, \eta}(s, h, \cdot)
 \end{aligned}
 \end{equation}
 are well-defined and are continuously Fréchet differentiable mappings on $B_{\delta_0}$ with Fréchet differentials: $\forall h = (\thetazero, \Xzero, \sigmazero)\in B_{\delta_0}\;,\; \forall \tilde h = (\tilde\Theta, \tilde X, \tilde\sigma)\in \Ltheta\times\LX\times\Lsigma$ we have 

\begin{equation}
\begin{aligned}
&Df^{i, \pm}_{\phi_\pm, \eta}(s, h)[\tilde h](\rho, z) = -DE_s(h)[\tilde h](\rho, z)P^i(h, s, \xi)(\rho, z) + \int^{\ve_2^\pm}_{E_{\rho s}(h, \rho, z)}\, D_h P^i(h, s, \xi)[\tilde h](\rho, z)\,d\xi
\end{aligned}
\end{equation}
 where $\displaystyle P^i:B_{\delta_0}\times\mathbb R^2\to \hat C^{2, \alpha}(\Bbarre)$ and $\displaystyle P^j:B_{\delta_0}\times\mathbb R^2\to \hat C^{2, \alpha}(\BAbarre)$ are the continuously Fréchet différentiable mappings defined by 
\begin{equation}
\label{P::1}
P^1(h, s, \xi)(\rho, z) := (X_K(1 + \Xzero) + 2(\rho s)^2)\phi_\pm(\xi, \rho s)\Psi_\eta(\rho, (\xi, \rho s), h),
\end{equation} 
\begin{equation}
\label{P::2}
P^2(h, s, \xi)(\rho, z) := \rho s (\xi - \rho\left(-\thetazero + \omega_K\right) s)\phi_\pm(\xi, \rho s)\Psi_\eta(\rho, (\xi, \rho s), h),
\end{equation}
\begin{equation}
\label{P::3}
P^3(h, s, \xi)(\rho, z) := \left(\tilde L^2\left(\xi-\rho\left(-\thetazero + \omega_K\right) s, X_K(1 + \Xzero ), \sigma, \rho, z\right) - s^2\right)\phi_\pm(\xi, \rho s)\Psi_\eta(\rho, (\xi, \rho s), h)
\end{equation}
where $\tilde L$ is defined by \eqref{L:::tilde} and 
\begin{equation}
\label{P::4}
\begin{aligned}
P^4(h, s, \xi)(\rho, z) &:=  \left(\frac{X_K^2(1 + \Xzero)^2}{\rho^4\left(1 + \sigmazero\right)^2}\left(\xi - \rho\left(-\thetazero + \omega_K\right) s\right)^2 +  \left(1 - \frac{X_K(1 + \Xzero)}{\rho^2}\right)\left( 1 + \frac{\rho^2}{X_K(1 + \Xzero)}s^2\right)\right)\\
&\phi_\pm(\xi, \rho s)\Psi_\eta(\rho, (\xi, \rho s), h), 
\end{aligned}
\end{equation}

 with derivatives: 
 \begin{equation}
 \label{DP::1}
 \begin{aligned}
 D_h P^1(h, s, \xi)[\tilde h](\rho, z) &= -D_h\rho_1(h, \xi, \rho s)[\tilde h]\Chi_\eta'(\rho - \rho_1(h, \xi, \rho s))(X_K(1 + \Xzero) + 2(\rho s)^2) \phi_\pm(\xi, \rho s)  \\
 &+ X_K\tilde X\phi_\pm(\xi, \rho s)\Psi_\eta(\rho, (\xi, \rho s), h), 
 \end{aligned}
 \end{equation}
 
 \begin{equation}
 \begin{aligned}
D_hP^2(h, s, \xi)[\tilde h](\rho, z) &= -\rho s D_h\rho_1(h, \xi, \rho s)[\tilde h]\Chi_\eta'(\rho - \rho_1(h, \xi, \rho s)) (\xi - \rho\left(-\thetazero + \omega_K\right) s)\phi_\pm(\xi, \rho s)  \\
&+ \rho^2 s\tilde\Theta\phi_\pm(\xi, \rho s)\Psi_\eta(\rho, (\xi, \rho s), h),
\end{aligned}
\end{equation}

\begin{equation}
\begin{aligned}
&D_h P^3(h, s, \xi)[\tilde h](\rho, z) = \\
&-D_h\rho_1(h, \xi, \rho s)[\tilde h]\Chi_\eta'(\rho - \rho_1(h, \xi, \rho s))\left(\tilde L^2\left(\xi-\rho\left(-\thetazero + \omega_K\right) s, X_K(1 + \Xzero ), \sigma, \rho, z\right) - s^2\right)\phi_\pm(\xi, \rho s) \\
&+ 2\tilde L\left(\xi-\rho\left(-\thetazero + \omega_K\right) s, X_K(1 + \Xzero), \sigma, \rho, z\right)\phi_\pm(\xi, \rho s)\Psi_\eta(\rho, (\xi, \rho s), h) \\
&\left( \rho s\tilde\Theta\frac{\partial \tilde L}{\partial E}\left(\xi-\rho\left(-\thetazero + \omega_K\right) s, X_K(1 + \Xzero), \sigma, \rho, z\right) + X_K\tilde X\frac{\partial \tilde L}{\partial X}\left(\xi-\rho\left(-\thetazero + \omega_K\right) s, X_K(1 + \Xzero), \sigma, \rho, z\right)  \right. \\
&\left. + \rho\tilde\sigma\frac{\partial \tilde L}{\partial \sigma}\left(\xi-\rho\left(-\thetazero + \omega_K\right) s, X_K(1 + \Xzero), \sigma, \rho, z\right)\right)
\end{aligned}
\end{equation}
 and 
 \begin{equation}
 \begin{aligned}
 &D_h P^4(h, s, \xi)[\tilde h](\rho, z) = -D_h\rho_1(h, \xi, \rho s)[\tilde h]\Chi_\eta'(\rho - \rho_1(h, \xi, \rho s)) \\
&\left(\frac{X_K^2(1 + \Xzero)^2}{\rho^4\left(1 + \sigmazero\right)^2}\left(\xi - \rho\left(-\thetazero + \omega_K\right) s\right)^2 +  \left(1 - \frac{X_K(1 + \Xzero)}{\rho^2}\right)\left( 1 + \frac{\rho^2}{X_K(1 + \Xzero)}s^2\right)\right)\phi_\pm(\xi, \rho s)  \\
&+ \phi_\pm(\xi, \rho s)\Psi_\eta(\rho, (\xi, \rho s), h)  \\
&\left(\frac{X_K^2(1 + \Xzero)^2}{\rho^3\left(1 + \sigmazero\right)^2}s\tilde\Theta\left(\xi - \rho\left(-\thetazero + \omega_K\right)s\right) -2\tilde\sigma\left(1 + \sigmazero\right)\frac{X_K^2(1 + \Xzero)^2}{\rho^4\left(1 + \sigmazero\right)^4}\left(\xi - \rho\left(-\thetazero + \omega_K\right) s\right)^2  \right.  \\
&\left. + 2\frac{X_K^2\tilde X(1 + \Xzero)}{\rho^4\left(1 + \sigmazero\right)^2} \left(\xi - \rho\left(-\thetazero + \omega_K\right) s\right)^2 -\frac{X_K\tilde X}{\rho^2}\left( 1 + \frac{\rho^2}{X_K(1 + \Xzero)}s^2\right) - \frac{\tilde X\rho^2 s^2}{X_K(1 + \Xzero)^2} \left(1 - \frac{X_K(1 + \Xzero)}{\rho^2}\right)\right). 
 \end{aligned}
 \end{equation}

 \end{lemma}
 
 \begin{proof}
 The proof follows from Lemma \ref{diff:E:u}, Lemma \ref{diff:tilde:L} and the dominated convergence theorem. We give details for $f^{1, \pm}_{\Phi, \eta}(s, \cdot, \cdot)$. Regularity for the remaining functions is proved in the same manner. 
 \\ First of all, let $s\in\mathbb R$ and recall the definition of $\displaystyle f^{1, \pm}_{\Phi, \eta}(s, \cdot, \cdot)$: 
 \begin{equation*}
 f_{\Phi, \eta}^{1, \pm}(s, h, \rho, z) := \int_{E_{\rho s}(h, \rho, z)}^{\ve^\pm_2}P^1(h, s, \xi)(\rho, z)\,d\xi  
 \end{equation*}
 where 
 \begin{equation*}
P^1(h, s, \xi)(\rho, z) := (X_K(1 + \Xzero) + 2(\rho s)^2)\phi_\pm(\xi, \rho s)\Psi_\eta(\rho, (\xi, \rho s), h),
\end{equation*} 
\begin{itemize}
\item We recall 
\begin{equation*}
\Psi_\eta(\cdot, (\ve, \ell), h) := 
\left\{
\begin{aligned}
& \Chi_\eta(\cdot - \rho_1(h, (\ve, \ell_z))), \quad (\ve, \ell_z)\in \Bbound\\
& 0 \quad\quad\quad\quad\quad\quad\quad\quad\quad\quad (\ve, \ell_z)\notin \Abound,
\end{aligned}
\right.
\end{equation*}
By Lemma \ref{phi1:zvc}, $\forall (\ve, \ell_z)\in \Bbound\,, \,\rho_1(\cdot, (\ve, \ell_z))$ is continuously Fréchet differentiable on $B_{\delta_0}$.  By smoothness of $\Chi_\eta$, $\Psi_\eta$ is continuously Fréchet differentiable with respect to $h$ with derivative $\forall \tilde h \in \Ltheta\times\LX\times\Lsigma$
\begin{equation*}
D_h\Psi_{\eta}(\rho, (\xi, \rho s), h) = \left\{  
\begin{aligned}
&-D_h\rho_1(h, \xi, \rho s)[\tilde h]\Chi_\eta'(\rho - \rho_1(h, \xi, \rho s))  \quad\text{if}\; (\xi, \rho s)\in \Bbound, \\
& 0 \quad\text{otherwise.}
\end{aligned}
\right. 
\end{equation*}
\item Hence, $P^1$ is Fréchet differentiable with respect to $h$ with derivative given by \eqref{DP::1}. 
\item Now, let $\overline P^1(h, s, \cdot)(\rho, z)$ be a primitive of $P^1(h, s, \cdot)(\rho, z)$.  We write 
\begin{equation*}
 f_{\Phi, \eta}^{1, \pm}(s, h, \rho, z) =  \overline P^1(h, s, \ve_2^\pm)(\rho, z) - \overline P^1(h, s, E_{\rho s}(h, \rho, z))(\rho, z).  
\end{equation*}
Hence, $\forall \tilde h\in \Ltheta\times \LX \times \Lsigma$, we have 
\begin{equation*}
\begin{aligned}
D_hf_{\Phi, \eta}^{1, \pm}(s, h, \rho, z) &=  D_h\overline P^1(h, s, \ve_2^\pm)[\tilde h](\rho, z) - D_h\overline P^1(h, s, E_{\rho s}(h, \rho, z))[\tilde h](\rho, z)  \\
&- D_h E_{s}(h)[\tilde h](\rho, z)\partial_\xi\overline P^1(h, s, E_{\rho s}(h, \rho, z))(\rho, z)  \\
&=  \int^{\ve_2^\pm}_{E_{\rho s}(h, \rho, z)}\, D_h P^1(h, s, \xi)[\tilde h](\rho, z)\,d\xi -DE_s(h)[\tilde h](\rho, z)P^1(h, s, \xi)(\rho, z).
\end{aligned}
\end{equation*}
The first term of the latter expression was obtained by the dominated convergence theorem and the second term is obtained by the definition of $\overline P^1$.
\item It remains to show that the derivative of $f_{\Phi, \eta}^{1, \pm}$ with respect to $h$ is continuous. This follows using similar arguments. 
\end{itemize}
 \end{proof}

 \begin{Propo}
 \label{Fi::regularity}
 \begin{enumerate}
\item Let $(\thetazero, \Xzero, \sigmazero, \lambdazero)\in \Ltheta\times\LX\times\Lsigma\times\Llambda$. Then, $F_1(\thetazero, \Xzero, \sigmazero)$, $F_2(\thetazero, \Xzero, \sigmazero)$, $F_3(\thetazero, \Xzero, \sigmazero)$  lie in $\hat C^{2, \alpha}(\Bbarre)$ and $F_4(\thetazero, \Xzero, \sigmazero, \lambdazero)$ lies in $\hat C^{1, \alpha}(\Bbarre)$. Moreover, they are compactly supported in $\BAbarre$. 
\item Let $\delta_0>0$ and let $F_i$ be defined on $B_{\delta_0}\footnote{$\displaystyle B_{\delta_0}\subset \Ltheta\times\LX\times\Lsigma\times\Llambda$ is the closed  ball of radius $\delta_0$ centred in $(0, 0, 0, 0)$}$. Then $F_i$, $i\in\left\{1, 2, 3\right\}$, are  a well-defined mappings from $B_{\delta_0}$ to $\hat C^{2, \alpha}(\Bbarre)$ and $F_4$ is a well-defined mapping from $B_{\delta_0}$ to $\hat C^{1, \alpha}(\Bbarre)$. Furthermore, $F_i$, $i\in\left\{ 1, 2, 3, 4\right\}$, are continuously Fréchet differentiable on $B_{\delta_0}$ with Fréchet differential given by $\forall (h = (\thetazero, \Xzero, \sigmazero), \lambdazero)\in B_{\delta_0}, (\forall \tilde h = (\tilde\Theta, \tilde X, \tilde\sigma), \tilde\lambda)\in \Ltheta\times\LX\times\Lsigma\times\Llambda$, $\forall (\rho, z)\in \Bbarre: $

\begin{equation}
\begin{aligned}
&DF_1(h)[\tilde h](\rho, z) =  \\
& -\frac{2\pi}{1 + \sigmazero}\left( D\tilde L(\ve_2^+, h, \rho, z)[\tilde h]f_{\phi_+, \eta}^{1, +}\left(\tilde L(\ve_2^+, h, \rho, z)  h,\rho, z\right) + D\tilde L(\ve_2^-, h, \rho, z)[\tilde h]f_{\phi_-, \eta}^{1, -}\left(\tilde L(\ve_2^-, h, \rho, z)  h,\rho, z\right) \right) \\&+ \int_{\frac{\ell_1^+}{\rho}}^{\tilde L(\ve_2^+, h, \rho, z)}D_h\left(-\frac{2\pi}{1 + \sigmazero}f_{\phi_+, \eta}^{1, +}(s, h, \rho, z)\right)[\tilde h]\, ds + \int^{\frac{\ell_1^-}{\rho}}_{-\tilde L(\ve_2^-, h, \rho, z)}D_h\left(-\frac{2\pi}{1 + \sigmazero}f_{\phi_-, \eta}^{1, -}(s, h, \rho, z)\right)[\tilde h]\, ds
\end{aligned}
\end{equation}

\begin{equation}
\begin{aligned}
DF_2(h)[\tilde h](\rho, z) &=  \frac{2\pi X_K(1 + \Xzero)}{1 + \sigmazero}\left( D\tilde L(\ve_2^+, h, \rho, z)[\tilde h]f_{\phi_+, \eta}^{2, +}\left(\tilde L(\ve_2^+, h, \rho, z)  h,\rho, z\right) \right. \\
&\left. + D\tilde L(\ve_2^-, h, \rho, z)[\tilde h]f_{\phi_-, \eta}^{2, -}\left(\tilde L(\ve_2^-, h, \rho, z)  h,\rho, z\right) \right) \\
&+ \int_{\frac{\ell_1^+}{\rho}}^{\tilde L(\ve_2^+, h, \rho, z)}D_h\left(\frac{2\pi X_K(1 + \Xzero)}{1 + \sigmazero}f_{\phi_+, \eta}^{2, +}(s, h, \rho, z)\right)[\tilde h]\, ds  \\
&+ \int^{\frac{\ell_1^-}{\rho}}_{-\tilde L(\ve_2^-, h, \rho, z)}D_h\left(-\frac{2\pi X_K(1 + \Xzero)}{1 + \sigmazero}f_{\phi_-, \eta}^{2, -}(s, h, \rho, z)\right)[\tilde h]\, ds
\end{aligned}
\end{equation}

\begin{equation}
\begin{aligned}
DF_3(h)[\tilde h](\rho, z) &= \frac{\rho^4}{X_K^2}\frac{2\pi(1 + \sigmazero)}{(1 + \Xzero)^2}\left( D\tilde L(\ve_2^+, h, \rho, z)[\tilde h]f_{\phi_+, \eta}^{3, +}\left(\tilde L(\ve_2^+, h, \rho, z)  h,\rho, z\right)  \right. \\
&\left. + D\tilde L(\ve_2^-, h, \rho, z)[\tilde h]f_{\phi_-, \eta}^{3, -}\left(\tilde L(\ve_2^-, h, \rho, z)  h,\rho, z\right) \right) \\
&+ \int_{\frac{\ell_1^+}{\rho}}^{\tilde L(\ve_2^+, h, \rho, z)}D_h\left(\frac{\rho^4}{X_K^2}\frac{2\pi(1 + \sigmazero)}{(1 + \Xzero)^2}f_{\phi_+, \eta}^{3, +}(s, h, \rho, z)\right)[\tilde h]\, ds  \\
&+ \int^{\frac{\ell_1^-}{\rho}}_{-\tilde L(\ve_2^-, h, \rho, z)}D_h\left(\frac{\rho^4}{X_K^2}\frac{2\pi(1 + \sigmazero)}{(1 + \Xzero)^2}f_{\phi_-, \eta}^{3, -}(s, h, \rho, z)\right)[\tilde h]\, ds
\end{aligned}
\end{equation}

\begin{equation}
\begin{aligned}
DF_4(h, \lambdazero)[\tilde h, \tilde \lambda](\rho, z) &= -\frac{8\pi\tilde\lambda e^{2\left(\lambdazero + \lambda_K\right)}}{1 + \sigmazero}\left(\int_{\frac{\ell_1^+}{\rho}}^{\tilde L(\ve_2^+, h, \rho, z)}\;f_{\phi_+, \eta}^{4, +}(s, h, \rho, z)\, ds , \right. \\
&\left. +  \int^{\frac{\ell_1^-}{\rho}}_{-\tilde L(\ve_2^-, h, \rho, z)}\;f_{\phi_-, \eta}^{4, -}(s, h, \rho, z)\, ds \right)  \\
&-\frac{4\pi e^{2\left(\lambdazero + \lambda_K\right)}}{1 + \sigmazero}\left( D\tilde L(\ve_2^+, h, \rho, z)[\tilde h]f_{\phi_+, \eta}^{4, +}\left(\tilde L(\ve_2^+, h, \rho, z)  h,\rho, z\right)  \right. \\
&\left. + D\tilde L(\ve_2^-, h, \rho, z)[\tilde h]f_{\phi_-, \eta}^{4, -}\left(\tilde L(\ve_2^-, h, \rho, z)  h,\rho, z\right) \right) \\
&+ \int_{\frac{\ell_1^+}{\rho}}^{\tilde L(\ve_2^+, h, \rho, z)}D_h\left(-\frac{4\pi e^{2\left(\lambdazero + \lambda_K\right)}}{1 + \sigmazero}f_{\phi_+, \eta}^{4, +}(s, h, \rho, z)\right)[\tilde h]\, ds  \\
&+ \int^{\frac{\ell_1^-}{\rho}}_{-\tilde L(\ve_2^-, h, \rho, z)}D_h\left(-\frac{4\pi e^{2\left(\lambdazero + \lambda_K\right)}}{1 + \sigmazero}f_{\phi_-, \eta}^{4, -}(s, h, \rho, z)\right)[\tilde h]\, ds
\end{aligned}
\end{equation}
\end{enumerate}
\end{Propo}

\begin{proof}
\begin{enumerate}
\item Since the matter terms vanish in the regions $\Bnbarre, \Bsbarre$, $\BHbarre$ and near the axis we prove the result only in $\BAbarre\backslash\tilde\Axis$, where $\tilde\Axis$ is some neighbourhood of the axis given by Lemma \ref{vanish:near:A}.

\item By the first point of the proposition, the mappings $F_i$ are well defined. In order the show the Fréchet differentiability, we apply lemmas \ref{diff:tilde:L}, \ref{diff:fi} and Lebesgue dominated convergence theorem. 

\item Since $\Chi_{\eta}$ is either $0$ or $1$ on the support of $\phi_{\pm}(\xi, \rho s)$, $D\Psi_{\eta}(\cdot, (\xi, \rho s), h)$ vanishes. Hence, we can eliminate the terms including $D\Psi_{\eta}$ in the derivatives of $P^i$'s.  

\item The estimates follow by using similar arguments from the previous lemmas. 
\end{enumerate}

\end{proof}

\subsection{Introduce the bifurcation parameter $\delta$}
Solutions to the Einstein-Vlasov system will be obtained by perturbing the Kerr spacetime using a bifurcation parameter $\delta\geq 0$. The latter turns on in the presence of Vlasov matter supported on $\Bbound\subset\subset \Abound$. In order to introduce the latter in the equations, we adjust the ansatz for $f$ \eqref{ansatz:for:f} in oder to make the dependence on $\delta$ explicit: 
\begin{equation}
\label{ansatz:for:f:delta}
f^\delta(t, \phi, \rho, z, \phi, p^\rho, p^\phi, p^z) = \Phi(\ve, \ell_z; \delta)\Psi_\eta(\rho, (\ve, \ell_z), (X, W, \sigma))
\end{equation}
 such that $$\forall (\ve, \ell_z)\in \Abound\;,\; \quad \Phi(\ve, \ell_z; 0) = 0 ,$$ where $\Phi: \Abound\times\mathbb R_+\to \mathbb R_+$. We will impose the regularity assumptions on $\Phi$ given in  Lemma \ref{diff:fi} so that the solution operator is well defined. Assuming that $(g, f^\delta)$ solves the EV-system, we can apply Theorem \eqref{PDEs::1} with the ansatz \eqref{ansatz:for:f:delta} to obtain the explicit dependence of $F_i$ on the bifurcation parameter $\delta$: 
 \begin{equation*}
\begin{aligned}
F_1(\thetazero, \Xzero, \sigmazero; \delta)(\rho, z):= -\frac{2\pi e^{2\left(\lambdazero+ \lambda_K\right)}}{1 + \sigmazero}&\int_{D(\rho, z)}(X_K(1 + \Xzero) + 2(\rho L)^2)\Phi(E + \rho\left(-\thetazero + \omega_K\right) L, \rho L) \\
&\Psi_\eta(\rho, (E + \rho(-\thetazero + \omega_K) L, \rho L), (\thetazero, \Xzero, \sigmazero))\,dE dL,
\end{aligned}
\end{equation*}
\begin{equation*}
\begin{aligned}
F_2(\thetazero, \Xzero, \sigmazero; \delta)(\rho, z):= \frac{2\pi}{1 + \sigmazero}X_K(1 + \Xzero)&\int_{D(\rho, z)}\rho L E\Phi(E + \rho\left(-\thetazero + \omega_K\right) L, \rho L) \\
&\Psi_\eta(\rho, (E + \rho\left(-\thetazero + \omega_K\right) L, \rho L), (\thetazero, \Xzero, \sigmazero))\,dE dL,
\end{aligned}
\end{equation*}
\begin{equation*}
\begin{aligned}
F_3(\thetazero, \Xzero, \sigmazero; \delta)(\rho, z):= \frac{\rho^4}{X^2_K}\frac{2\pi(1 + \sigmazero)}{\left(1 + \Xzero\right)^2}&\int_{D(\rho, z)}\left(\tilde L^2 - L^2\right)\Phi(E + \rho\left(-\thetazero + \omega_K\right) L, \rho L) \\ 
&\Psi_\eta(\rho, (E + \rho\left(-\thetazero + \omega_K\right) L, \rho L), (\thetazero, \Xzero, \sigmazero))\,dE dL, 
\end{aligned}
\end{equation*}
\begin{equation*}
\begin{aligned}
F_4(\thetazero, \Xzero, \sigmazero; \delta)(\rho, z)&:= -\frac{4\pi e^{2\left(\lambdazero+ \lambda_K\right)}}{1 + \sigmazero}\int_{D(\rho, z)}\left(\frac{X_K^2\left(1 + \Xzero\right)^2}{\rho^4\left(1 + \sigmazero\right)^2}E^2 + \left(1 - \frac{X_K}{\rho^2}\left(1 + \Xzero \right)\right)\left( 1 + \frac{\rho^2}{X_K}\frac{1}{1 + \Xzero}L^2\right)\right) \\ 
&\Phi(E + \rho\left(-\thetazero + \omega_K\right) L, \rho L)\Psi_\eta(\rho, (E + \rho\left(-\thetazero + \omega_K\right), \rho L), (\thetazero, \Xzero, \sigmazero))\,dE dL.   \\
\end{aligned}
\end{equation*}
\noindent Finally, we note that $\forall (\thetazero, \Xzero, \sigmazero), \forall (\rho, z)\,$, $F_i(\thetazero, \Xzero, \sigmazero; \cdot)(\rho, z)$ is continuously Fréchet differentiable  on $[0, \delta_0[$.

\subsection{Solving for $\sigma$}
\label{sigma::solving}
In this section, we solve for $\sigmazero$ in terms of the renormalised unkonws $\left( \Xzero, \thetazero, \lambdazero\right)$ and $\delta$. We recall that $\sigmazero$ verifies 
    \begin{equation*}
            \Delta_{\mathbb{R}^4}{\sigmazero} = \rho^{-1}\sigma^{-1}{X}e^{2\lambda}F_3(\thetazero, \Xzero, \sigmazero, \delta)(\rho, z),
        \end{equation*}
        where $\Delta_{\mathbb R^4}$ is the Laplacian corresponding to the flat metric on $\mathbb R^4$ given by $g_{\mathbb R^4} = d\rho^2 + dz^2 + \rho^2d\mathbb{S}^2$. 
\noindent To any function $f:\Bbarre \to \mathbb R$, we associate a function $f_{\mathbb R^4}: \mathbb R^4\to \mathbb R$ defined in the following way: let  $x = (x_i)_{i= 1\dots 4}\in\mathbb R$ and define its cylindrical coordinates $(\rho, \theta, \phi , z)\in [0, \infty[\times \mathbb S^2\times \mathbb R$ such that  
\begin{equation*}
\begin{aligned}
x_1 &= \rho\sin\theta\cos\phi \\
x_2 &= \rho\cos\theta\sin\phi \\
x_3 &= \rho\cos\theta \\
x_4 &= z.
\end{aligned}
\end{equation*}
Now define $f_{\mathbb R^4}$ by 
\begin{equation*}
f_{\mathbb R^4} (x) := f(\rho = \sqrt{x_1^2+ x_2^2+ x_3^2}, z = x_4). 
\end{equation*}
\noindent In the following, we confound $f$ with $f_\mathbb R^4$.  We start with solving the linear problem. 
\subsubsection{Linear problem}
We prove the following result
\begin{Propo}
\label{linear:sigma}
Let $H_\sigma\in \Nsigma$. Then, there exists a unique $\sigmazero\in\Lsigma$ which solves the equation
\begin{equation*}
\Delta_{\mathbb R^4} \sigmazero = H_\sigma.
\end{equation*}
It is given by 
\begin{equation*}
\sigmazero(\rho, z) = \sigmazero_{\mathbb R^4}(x) = -\int_{\mathbb R^4}\, \frac{1}{|x - y|^2}(H_\sigma)_{\mathbb R^4}(y)\,dy. 
\end{equation*}
Moreover, there exists $C(\alpha_0)>0$ such that 
\begin{equation*}
||\sigmazero||_{\Lsigma} \leq C(\alpha_0) ||H_\sigma||_{\Nsigma}.
\end{equation*} 
\end{Propo}
\noindent The proof is based on Theorem \ref{Newton:estimates} and the following Newtonian estimates 
\begin{lemma}
\label{Newton:1}
Let $n\ge 3$ and $F:\mathbb{R}^n\to\mathbb{R}$ satisfies $|F(x)|\le C\langle x \rangle^{-k}$ where  $\displaystyle \langle x \rangle:= (1+|x|^2)^{\frac{1}{2}}$ for some $k>2$ and $k\neq n$. Let $u:\mathbb{R}^n\to\mathbb{R}$ be the corresponding Newton potential:
\begin{equation*}
    u(x) := \int_{\mathbb{R}^n}\; |x-y|^{2-n}F(y)\,dy.
\end{equation*}
Then, we have 
\begin{equation}
\label{0:order}
    |u(x)|\le C \sup_{y\in\mathbb{R}^n} \left|\langle y \rangle^kF(y)\right|\left(\langle x \rangle^{2-n} + \langle x \rangle^{2-k} \right),
    \end{equation}
    and 
    \begin{equation}
    \label{1:order}
    |\partial u(x)|\le C \sup_{y\in\mathbb{R}^n} \left|\langle y \rangle^kF(y)\right|\left(\langle x \rangle^{1-n} + \langle x \rangle^{1-k} \right).
\end{equation}
\end{lemma}

\begin{proof}
\begin{enumerate}
\item First of all, we have $\forall x\in \mathbb R $

\begin{equation*}
\begin{aligned}
|u(x)| &\leq \int_{\mathbb R^n}\, |x - y|^{2-n}|F(y)|\,dy \\
&\leq \sup_{y\in\mathbb R^n}\left|\langle y \rangle^kF(y)\right|\int_{\mathbb R^n}\, |x - y|^{2-n}\langle y \rangle^{-k}\, dy. 
\end{aligned}
\end{equation*}
Set 
\begin{equation*}
    A(x) := \int_{\mathbb{R}^n}\; |x-y|^{2-n}\langle y \rangle^{-k}\,dy ,
\end{equation*}
we have
\begin{align*}
    A(x) &= \int_{y, |x-y|\le \frac{\langle x \rangle}{4}}\; |x-y|^{2-n}\langle y \rangle^{-k}\,dy + \int_{y, |x-y|\ge\frac{\langle x \rangle}{4}}\; |x-y|^{2-n}\langle y \rangle^{-k}\,dy = I + II. 
\end{align*}
We estimate the first term in the right hand side.
\begin{align*}
    I &:= \int_{y, |x-y|\le \frac{\langle x \rangle}{4}}\; |x-y|^{2-n}\langle y \rangle^{-k}\,dy .
\end{align*}
By the reversed triangular inequality, we have
\begin{equation*}
    |y| = |x-(x-y)| \ge ||x|-|x-y||,
\end{equation*}
Thus,
\begin{equation*}
    \langle y \rangle^{-k}  = (1+|y|^2)^{\frac{-k}{2}}\le (\langle x \rangle^2-2\langle x \rangle|x-y|+|x-y|^2)^{\frac{-k}{2}},
\end{equation*}
We obtain, 
\begin{align*}
    I &\le \int_{y, |x-y|\le \frac{\langle x \rangle}{4}}\; |x-y|^{2-n}(\langle x \rangle^2-2\langle x \rangle|x-y|+|x-y|^2)^{\frac{-k}{2}}\,dy ,\\
    &\le \int^{\frac{\langle x \rangle}{4}}_0\; \frac{r^{2-n}}{(\langle x \rangle^2-2\langle x \rangle r+r^2)^{\frac{k}{2}}}r^{n-1}\,dr ,\\
    &\le \int^{\frac{\langle x \rangle}{4}}_0\; \frac{r^{2-n}}{(\langle x \rangle^2-2\langle x \rangle r)^{\frac{k}{2}}}r^{n-1}\,dr ,\\
    &= \langle x \rangle^{-k}\int^{\frac{\langle x \rangle}{4}}_0\; \frac{r}{(1-2\langle x \rangle^{-1}r)^{\frac{k}{2}}}r^{n-1}\,dr ,\\ 
    &=  \langle x \rangle^{2-k}\int^{\frac{1}{4}}_0\;\frac{r'}{(1-2r')^{\frac{k}{2}}}\,dr' .
\end{align*}
In the last line, we changed the variable by setting $r' = \langle x \rangle^{-1}r$
Hence
\begin{equation*}
    I \le C\langle x \rangle^{2-k}.
\end{equation*}
Now we estimate the second term $II$ given by
\begin{equation*}
    II := \int_{y, |x-y|\ge \frac{\langle x \rangle}{4}}\; |x-y|^{2-n}\langle y \rangle^{-k}\,dy .
\end{equation*}
We will deal with two cases $k\le n$ and $k\ge n$. 

Suppose that $k\ge n$. Then 
\begin{align*}
    II &= \int_{y, |x-y|\ge \frac{\langle x \rangle}{4}}\; |x-y|^{2-n}\langle y \rangle^{-k}\,dy ,\\
    &\le C\langle x \rangle^{2-n}\int_{y, |x-y|\ge \frac{\langle x \rangle}{4}}\; \langle y \rangle^{-k}\,dy ,\\
    &\le C\langle x \rangle^{2-n}\int_{\mathbb{R}^n}\; \langle y \rangle^{-k}\,dy ,\\
    &\le C\langle x \rangle^{2-n}\int^\infty_0\; \frac{r^{n-1}}{(1+r^2)^{\frac{k}{2}}}\,dr ,
\end{align*}
When $r\to\infty$, we have
\begin{equation*}
    \frac{r^{n-1}}{(1+r^2)^{\frac{k}{2}}} \sim \frac{1}{r^{k-n+1}}
\end{equation*}
Since $k>n$ the integral converges. Thus,
\begin{equation*}
    I \le C\langle x \rangle^{2-n}.
\end{equation*}
Now suppose that $k\leq n$. We write 
\begin{equation*}
    II = \int_{\mathcal{D}_1(x)}\; |x-y|^{2-n}\langle y \rangle^{-k}\,dy + \int_{\mathcal{D}_2(x)}\; |x-y|^{2-n}\langle y \rangle^{-k}\,dy = II_1 + II_2, 
\end{equation*}
where 
\begin{align*}
    \mathcal{D}_1(x) &:= \left\{ y, |x-y|\ge \frac{\langle x \rangle}{4}\quad\text{and}\quad |y|\le2\langle x \rangle  \right\}, \\
    \mathcal{D}_2(x) &:= \left\{ y, |x-y|\ge \frac{\langle x \rangle}{4}\quad\text{and}\quad |y|\ge2\langle x \rangle  \right\}.
\end{align*}
Now 
\begin{align*}
    II_2 &:= \int_{\mathcal{D}_1(x)}\; |x-y|^{2-n}\langle y \rangle^{-k}\,dy ,\\
    &= \int_{|y|\ge2\langle x \rangle}\; |x-y|^{2-n}\langle y \rangle^{-k}\,dy ,  \\
    &\le \int_{|y|\ge2\langle x \rangle}\; ||x|-|y||^{2-n}\langle y \rangle^{-k}\,dy,  \quad\quad\text{by reversed triangle inequality}, \\
    &= \int_{|y|\ge2\langle x \rangle}\; (|y|-|x|)^{2-n}\langle y \rangle^{-k}\,dy ,\\ 
    &\le \int_{|y|\ge2\langle x \rangle}\; (|y|-\langle x \rangle)^{2-n}\langle y \rangle^{-k}\,dy ,\\
    &= \int^{\infty}_{2\langle x \rangle}\; \frac{1}{(r-\langle x \rangle)^{n-2}}\frac{1}{(1+r^2)^{\frac{k}{2}}}r^{n-1}\,dr,\\ 
    &= \int^{\infty}_{2\langle x \rangle}\; \frac{1}{(1-\frac{\langle x \rangle}{r})^{n-2}}\frac{r}{(1+r^2)^{\frac{k}{2}}}\,dr.
\end{align*}
Since $r\ge2\langle x \rangle$, we have the following estimate 
\begin{equation*}
    \frac{1}{(1-\frac{\langle x \rangle}{r})^{n-2}} \le 2^{n-2}.
\end{equation*}
Hence, 
\begin{align*}
    II_2 &\le C \int^{\infty}_{2\langle x \rangle}\; \frac{r}{(1+r^2)^{\frac{k}{2}}}\,dr ,\\
    &= C\left[ \frac{1}{-\frac{k}{2}+1} (1+r^2)^{-\frac{k}{2}+1}\right], \\
    &\le C\langle x \rangle^{2-k}.
\end{align*}
Now we estimate $II_1$.
\begin{align*}
    II_1 &= \int_{\left\{ |x-y|\ge \frac{\langle x \rangle}{4} \right\}\cap \left\{ |y|\le2\langle x \rangle \right\}}\; |x-y|^{2-n}\langle y \rangle^{-k}\,dy ,\\
    &= \int_{B(0,2\langle x \rangle)\backslash B(x,\frac{\langle x \rangle}{4})}\; |x-y|^{2-n}\langle y \rangle^{-k}\,dy ,\\
    &\le C\langle x \rangle^{2-n} \int_{B(0,2\langle x \rangle)\backslash B(x,\frac{\langle x \rangle}{4})}\; \langle y \rangle^{-k}\,dy ,\\
    &\leq  C\langle x \rangle^{2-n}\int^{2\langle x \rangle}_0\; \frac{r^{n-1}}{(1+r^2)^{\frac{k}{2}}}\,dr
\end{align*}
If $k<n$ then,
\begin{align*}
    II_1 &\le C\langle x \rangle^{2-n}\int^{2\langle x \rangle}_0\;\frac{1}{r^{k-n+1}}\,dr ,\\
    &\le C \langle x \rangle^{2-n}\langle x \rangle^{n-k} = \langle x \rangle^{2-k}.
\end{align*}
If $k=n$, then
\begin{align*}
    II_1 &\le C\langle x \rangle^{2-n}\int^{2\langle x \rangle}_0\;r^{n - 1}(1+r^2)^{-\frac{n}{2}}\,dr ,\\
    &\le C \langle x \rangle^{2-n}\log\langle x \rangle.
\end{align*}
\item  For the first order estimates, by the integrability condition of $F$ and the regularisation property of the convolution product  $u$ is  differentiable and we have
\begin{align*}
    \partial_{x_i} u(x) &= (2-n)\int_{\mathbb{R}^n}\; \partial_{x_i}(|x-y|)|x-y|^{1-n}F(y)\,dy, \\
    &= (2-n)\int_{\mathbb{R}^n}\; 
    (x_i-y_i)|x-y|^{-n}F(y)\,dy. \\
\end{align*}
Thus,
\begin{align*}
    |\partial u(x)| &\le (n-2)\int_{\mathbb{R}^n}\;|x_i-y_i||x-y|^{-n}F(y)\,dy , \\
    &\le (n-2)\int_{\mathbb{R}^n}\;|x-y|^{1-n}F(y)\,dy, \\
    &\le C\sup_{y\in\mathbb{R}^n} \left|\langle y \rangle^kF(y)\right|\int_{\mathbb{R}^n}\; |x-y|^{1-n}\langle y \rangle^{-k}\,dy .
\end{align*}
We use similar arguments to the first point in order to obtain the first order estimates.
\end{enumerate}
\end{proof}
\begin{lemma}
\label{Newton:2}
Let $n\ge 3$ and $F:\mathbb{R}^n\to\mathbb{R}$ satisfies $|F(x)| + \sup_{|x-y|\le 1}\frac{|F(x)-F(y)|}{|x-y|^{\alpha}} \le C\langle x \rangle^{-k}$ for some $k>2$ and $k\neq n$. Let $u:\mathbb{R}^n\to\mathbb{R}$ be the corresponding Newton potential:
\begin{equation*}
    u(x) := \int_{\mathbb{R}^n}\; |x-y|^{2-n}F(y)\,dy.
\end{equation*}
Then
\begin{align*}
    &|\partial^2u(x)| \\
    &\le C \left(\sup_{y\in\mathbb{R}^n} \left|\langle y \rangle^kF(y)\right|+ \sup_{y\in\mathbb{R}^n}\left( \langle y \rangle^k \left(\sup_{z, |z-y|\le 1}\frac{|F(z)-F(y)|}{|z-y|^{\alpha}}\right)\right)\right) \log(4\langle x \rangle)\left(\langle x \rangle^{-n} +\langle x \rangle^{-k} \right) , \\
\end{align*}
\end{lemma}

\begin{proof}
\begin{enumerate}
    \item We write \footnote{The details of computations are given in  \cite[Chapter 10]{lieb2001graduate}.} the second weak derivatives of u ,$\partial_{ij}u$:
    \begin{equation*}
        \partial^2_{x_ix_j}u(x) = \int_{\mathbb{R}^n}\; \partial^2_{x_ix_j}G(x,y)(F(y)-F(x))\,dy ,
    \end{equation*}
    where 
    \begin{equation*}
        G(x,y) := \frac{1}{|x-y|^{n-2}}.
    \end{equation*}
    \item Note that 
    \begin{equation*}
        \forall \lambda> 0\; \int_{\mathbb{S}^{n-1}}\;\frac{\partial^2G}{\partial_{x_i}\partial{x_j}}(0,\lambda\sigma)\,d\sigma = 0.
    \end{equation*}
    Indeed, we compute 
    \begin{align*}
        \frac{\partial^2G}{\partial{x_i}\partial{x_j}}(x,y) &= (2-n)\partial_{x_i}(x_j-y_j)|x-y|^{-n}, \\
        &= (2-n)\delta_{ij}|x-y|^{-n} -n(2-n)(x_i-y_i)(x_j-y_j)|x-y|^{-n-2}.
    \end{align*}
    Besides, observe that
    \begin{equation*}
        \int_{\mathbb{S}^{n-1}}\;\sigma_i\sigma_j\,d\sigma = 0, \quad\text{if $i\neq j$},
    \end{equation*}
    and 
    \begin{equation*}
        \int_{\mathbb{S}^{n-1}}\;\sigma_i^2\,d\sigma = \frac{|\mathbb{S}^{n-1}|}{n}, \quad\text{if $i=j$}.
    \end{equation*}
    Hence 
    \begin{align*}
           \int_{\mathbb{S}^{n-1}}\;\frac{\partial^2G}{\partial{x_i}\partial{x_j}}(0,\lambda\sigma)\,d\sigma &= (n-2)\lambda^{-n}(n\sigma_i\sigma_j-\delta_{ij}), \\
           &= 0.
    \end{align*}
    \item Now we proceed as in the previous lemma and we use Holder estimates for $F$ in order to control the terms when $|x - y|\leq 1$.
    
    We have
    \begin{align*}
        \partial^2_{x_ix_j}u(x) &= \int_{\mathbb{R}^n}\; \partial^2_{x_ix_j}G(x,y)(F(y)-F(x))\, ,\\
        &= \int_{y, |x-y|\le \frac{\langle x \rangle}{4}}\;\partial^2_{x_ix_j}|x-y|^{2-n}(F(y)-F(x))\,dy+ \int_{y, |x-y|\ge \frac{\langle x \rangle}{4}}\;\partial^2_{x_ix_j}|x-y|^{2-n}(F(y)-F(x))\,dy, \\ 
        = I + II.
    \end{align*}
    \begin{align*}
        I &:= \int_{\left\{ y, |x-y|\le \frac{\langle x \rangle}{4} \right\}}\;\partial^2_{x_ix_j}|x-y|^{2-n}(F(y)-F(x))\,dy ,\\
        & = \int_{\left\{ y, |y|\le \frac{\langle x \rangle}{4} \right\}}\;\frac{\partial^2G}{\partial{x_i}\partial{x_j}}(0,y)(F(y+x)-F(x))\,dy ,\\
        &= \int_{\left\{ y,|y|\le 1\right\} }\;\frac{\partial^2G}{\partial{x_i}\partial{x_j}}(0,y)(F(y+x)-F(x))\,dy + \int_{\left\{ y, 1\le|y|\le \frac{\langle x \rangle}{4}\right\} }\;\frac{\partial^2G}{\partial{x_i}\partial{x_j}}(0,y)(F(y+x)-F(x))\,dy,\\
        &= I_1 + I_2 .
    \end{align*}
    Note that the set $\left\{ y, 1 \le|y|\le \frac{\langle x \rangle}{4} \right\}$ is empty if $\langle x \rangle < 2.$
    \begin{align*}
        |I_2| &:= \left|\int_{\left\{ y, 1\le|y|\le \frac{\langle x \rangle}{4}\right\} }\;\frac{\partial^2G}{\partial{x_i}\partial{x_j}}(0,y)(F(y+x)-F(x))\,dy\right| ,\\
        &= \left|\int_{\left\{ y, 1\le|y|\le \frac{\langle x \rangle}{4}\right\} }\;\frac{\partial^2G}{\partial{x_i}\partial{x_j}}(0,y)F(y+x)\,dy\right| ,\\
        &\le \int_{\left\{ y, 1\le|y|\le \frac{\langle x \rangle}{4}\right\} }\;\frac{|F(y)|}{|x-y|^n}\,dy, \\
        &\le C\sup_{y\in\mathbb{R}^n} \left|\langle y \rangle^kF(y)\right|\int_{\left\{ y, 1\le|y|\le \frac{\langle x \rangle}{4}\right\} }\;\frac{\langle y \rangle^{-k}}{|x-y|^n}\,dy, \\
    \end{align*}
    We estimate for $\langle x \rangle \geq 2$, 
    \begin{align*}
        \int_{\left\{ y, 1\le|y|\le \frac{\langle x \rangle}{4}\right\} }\;\frac{\langle y \rangle^{-k}}{|x-y|^n}\,dy &\le \int_{\left\{ y, 1\le|y|\le \frac{\langle x \rangle}{4}\right\} }\;|x-y|^{-n}(\langle x \rangle^2-2|x||x-y|+|x-y|^2)^{-\frac{k}{2}}\,dy , \\
        &\leq \int^{\frac{\langle x \rangle}{4}}_{1} \frac{1}{r}(\langle x \rangle^2-2|x|r+r^2)^{-\frac{k}{2}}\,dr ,\\
        &\leq \langle x \rangle^{-k}\int^{\frac{1}{4}}_{\frac{1}{\langle x \rangle}}\; \frac{1}{r}(1-2\frac{|x|}{\langle x \rangle}r)^{-\frac{k}{2}}\,dr ,\\
        &\le \langle x \rangle^{-k} + \langle x \rangle^{-k}\log(4\langle x \rangle). 
    \end{align*}
    Now we  estimate $I_1$:
    \begin{align*}
        |I_1| &\le \int_{\left\{ y,|y|\le 1\right\} }\;|y|^{-n}|F(y+x)-F(x)|\,dy ,\\
        &= \int_{\left\{ y,|y|\le 1\right\} }\;|y|^{\alpha-n}\frac{|F(y+x)-F(x)|}{|y|^{\alpha}}\,dy ,\\ 
        &\le C\left(\langle x \rangle^{-k}\sup_{y\in\mathbb{R}^n}\left( \langle y \rangle^k \left(\sup_{z, |z-y|\le 1}\frac{|F(z)-F(y)|}{|z-y|^{\alpha}}\right)\right)\right)\int_{\left\{ y,|y|\le 1\right\} }\; |y|^{\alpha - n}\,dy ,\\
        &= C\left(\langle x \rangle^{-k}\sup_{y\in\mathbb{R}^n}\left( \langle y \rangle^k \left(\sup_{z, |z-y|\le 1}\frac{|F(z)-F(y)|}{|z-y|^{\alpha}}\right)\right)\right) \int^1_0\; r^{\alpha-1}\,dr.
    \end{align*}
    The integral of the last line converges. Hence, 
    \begin{equation*}
        |I_1| \le C\left(\langle x \rangle^{-k}\sup_{y\in\mathbb{R}^n}\left( \langle y \rangle^k \left(\sup_{z, |z-y|\le 1}\frac{|F(z)-F(y)|}{|z-y|^{\alpha}}\right)\right)\right).
    \end{equation*}
    Thus, 
    \begin{equation}
        |I|\le C\left(\sup_{y\in\mathbb{R}^n} \left|\langle y \rangle^kF(y)\right|+ \sup_{y\in\mathbb{R}^n}\left( \langle y \rangle^k \left(\sup_{z, |z-y|\le 1}\frac{|F(z)-F(y)|}{|z-y|^{\alpha}}\right)\right)\right) \log(4\langle x \rangle) \langle x \rangle^{-k}. 
    \end{equation}
    Now we estimate the remaining term, 
    \begin{align*}
        |II| &= \left| \int_{y, |x-y|\ge \frac{\langle x \rangle}{4}}\;\partial^2_{x_ix_j}|x-y|^{2-n}(F(y)-F(x))\,dy \right| , \\
        &= \left| \int_{y, |x-y|\ge \frac{\langle x \rangle}{4}}\;\partial^2_{x_ix_j}|x-y|^{2-n}F(y)\,dy \right| , \\
        &\le \sup_{y\in\mathbb{R}^n} \left|\langle y \rangle^kF(y)\right|\int_{\mathbb{R}^n}\; |x-y|^{1-n}\langle y \rangle^{-k}\,dy.
    \end{align*}
Set
    \begin{align*}
        II_1 &:= \int_{\left\{ y, |x-y|\ge \frac{\langle x \rangle}{4}\quad\text{and}\quad |y|\le4\langle x \rangle  \right\}}\;|x-y|^{-n}\langle y \rangle^{-k}\,dy ,\\
        II_2 &:= \int_{\left\{ y, |x-y|\ge \frac{\langle x \rangle}{4}\quad\text{and}\quad |y|\ge4\langle x \rangle  \right\}}\;|x-y|^{-n}\langle y \rangle^{-k}\,dy.
    \end{align*}
   Then, 
    \begin{align*}
        |II_1| &\le C\langle x \rangle^{-n}\int_{\left\{ y,|y|\le4\langle x \rangle\right\}}\;\langle y \rangle^{-k}\,dy ,\\
        &\le C\langle x \rangle^{-n}\int^{4\langle x \rangle}_{0}\;r^{n-k-1}\,dr ,\\
        &\le C\langle x \rangle^{-n}|B_1|+ C\langle x \rangle{-n}\int^{4\langle x \rangle}_1\;r^{n-k-1}\,dr ,\\
        &\le C\langle x \rangle^{-n}+
        \left\{
        \begin{aligned}
        &\langle x \rangle^{-k}\int^{4\langle x \rangle}_1\;\frac{1}{r} \quad\text{if $n-k\ge 0$}, \\
        &\langle x \rangle^{-n}\int^{4\langle x \rangle}_1\;\frac{1}{r} \quad\text{else},
        \end{aligned}
        \right. \\
        &\le \langle x \rangle^{-n} + \log(4\langle x \rangle)\left( \langle x \rangle^{-n} + \langle x \rangle^{-k}\right).
    \end{align*}
    Now we estimate 
    \begin{align*}
        |II_2| &\le\int_{\left\{y, |y|\ge 4\langle x \rangle \right\}}\;|x-y|^{-n}\langle y \rangle^{-k}\,dy ,\\
        &\le \int_{\left\{y, |y|\ge 4\langle x \rangle \right\}}\;||x|-|y||^{-n}\langle y \rangle^{-k}\,dy ,\\
        &= \int^{\infty}_{4\langle x \rangle}\;(r-\langle x \rangle)^{-n}\frac{r^{n-1}}{(1+r^2)^{\frac{k}{2}}}\,dr ,\\
        &\le C \int^{\infty}_{4\langle x \rangle}\;\frac{1}{r(1+r^2)^{\frac{k}{2}}}\,dr ,\\
        &\le C \int^{\infty}_{4\langle x \rangle}\;\frac{1}{r^{k+1}}\,dr ,\\
        &= C\langle x \rangle^{-k}.
    \end{align*}
    Finally we obtain the second order estimate.
\\ The Hölder part follows by using similar arguments. 
\end{enumerate}
\end{proof}

\noindent Now,  we prove Proposition \ref{linear:sigma}. 
\begin{proof}
Let $H_\sigma\in \Nsigma$. Then, 
\begin{equation*}
\forall x\in \mathbb R^4\;\; |H(x)| \leq \langle x \rangle^{-5}. 
\end{equation*}
We set 
\begin{equation*}
\sigmazero := -\int_{\mathbb R^4}|x - y|^{-2}H_\sigma(y)\,dy
\end{equation*}
and we apply Lemma \ref{Newton:1} with $n = 4$ and $k= 5$. We obtain: 
\begin{equation*}
||r^2\sigmazero||_{L^\infty(\Bbarre)} \leq C ||H_\sigma||_{\Nsigma}
\end{equation*}
and 
\begin{equation*}
||r^3\partial \sigmazero||_{L^\infty(\Bbarre)} \leq C ||H_\sigma||_{\Nsigma}. 
\end{equation*}
Since $H_\sigma\in C^{1, \alpha}(\mathbb R^4)$  and $H_\sigma\in L^1\cap L^\infty (\mathbb R^4)$, we obtain Theorem \ref{Newton:estimates} in order to obtain  $\sigmazero\in C^{3, \alpha}(\Bbarre)$
 and 
\begin{equation*}
||\sigmazero||_{C^{3, \alpha}(\Bbarre)} \leq C ||H_\sigma||_{C^{1, \alpha}(\Bbarre)} \leq C ||H_\sigma||_{\Nsigma}. 
\end{equation*}
Finally,  we apply Lemma \ref{Newton:2} with  with $n = 4$,  $k= 4$ and $F = H_\sigma$ in order to obtain:
\begin{equation*}
||r^4\log^{-1}(4r)\partial^2 f||_{L^\infty(\Bbarre)} \leq C||H_\sigma||_{C^{1, \alpha}(\mathbb R^4)} \leq C||H_\sigma||_{\Nsigma}. 
\end{equation*}
\end{proof}
\subsubsection{Non-linear estimates}
We apply Theorem \ref{Fixed::Point::2} in order to obtain
\begin{Propo}
\label{non:linear:sigma}
Let $\alpha_0\in(0, 1)$ and let $\overline\delta_0>0$. Then, there exists $0<\delta_0 \leq \overline\delta_0$ such that $\displaystyle \forall \left(\Xzero, \thetazero, \lambdazero, \delta\right)\in B_{\delta_0}\left(\LX\times\Ltheta\times\Llambda\times[0, \infty[\right)$ there exists a unique one parameter family $\sigmazero\left(\Xzero, \thetazero, \lambdazero, \delta\right)\in B_{\delta_0}\left(\Lsigma\right)$ which solves \eqref{sigmazero} and which satisfies
\begin{equation*}
        \left|\left|\sigmazero\left(\Xzero, \thetazero, \lambdazero, \delta\right)\right|\right|_{\Lsigma} \le C(\alpha_0)\left(||\Xzero||^2_{\LX} + ||\thetazero||^2_{\Ltheta} + ||\lambdazero||^2_{\Llambda} + \delta\right)
  \end{equation*}
  and $\displaystyle \forall \left(\Xzero_i, \thetazero_i, \lambdazero_i, \delta_i\right)\in B_{\delta_0}\left(\LX\times\Ltheta\times\Llambda\times[0, \infty[\right)$,
        \begin{equation*}
        \begin{aligned}
        &\left|\left|\sigmazero\left(\Xzero_1, \thetazero_1, \lambdazero_1, \delta_1\right) - \sigmazero\left(\Xzero_2, \thetazero_2, \lambdazero_2, \delta_2\right)\right|\right|_{\LX} \\ 
        &\le C(\alpha_0)\left(\left(\left|\left|\left(\Xzero_1, \thetazero_1, \lambdazero_1\right)\right|\right|_{\LX\times\Ltheta\times\Llambda} + \left|\left|\left(\Xzero_2, \thetazero_2, \lambdazero_2\right)\right|\right|_{\LX\times\Ltheta\times\Llambda}\right)\left|\left|\left(\Xzero_1, \thetazero_1, \lambdazero_1\right) -\left(\Xzero_2, \thetazero_2, \lambdazero_2\right)\right|\right|_{\LX\times\Ltheta\times\Llambda} \right. \\
        &\left. + \left|\delta_1 - \delta_2\right| \right).
        \end{aligned}
        \end{equation*}
\end{Propo}
\noindent Before we prove the above proposition, we introduce the following notations
\begin{enumerate}
\item Define $H_\sigma$ on $\Bbarre$ to be the mapping
\begin{equation*}
H_{\sigma}(\rho, z) = \frac{X_Ke^{2\lambda_K}}{\rho^2}\frac{1 + \Xzero}{1 + \sigmazero}e^{2\lambdazero}F_3(\Xzero, \thetazero, \sigmazero, \delta)(\rho, z).
\end{equation*}
\item Let $\overline\delta_0>0$. Define the nonlinear operator $N_\sigma$ on  $B_{\overline\delta_0}(\Lsigma)\times B_{\overline\delta_0}(\LX\times\Ltheta\times\Llambda\times[0, \infty[)$ by
\begin{equation*}
N_\sigma\left(\sigmazero, (\Xzero, \thetazero, \lambdazero; \delta)\right)(\rho, z) := H_{\sigma}(\rho, z). 
\end{equation*} 
\end{enumerate}
\noindent In order to prove the above proposition, we will need the following lemma
\begin{lemma}
\label{est:sigma}
\begin{enumerate}
\item There exists $\overline\delta_0>0$ such that  $\displaystyle N_\sigma\left(B_{\overline\delta_0}(\Lsigma)\times B_{\overline\delta_0}(\LX\times\Ltheta\times\Llambda\times[0, \infty[)\right)\subset \Nsigma$. 
\item There exist $0<\delta_0\leq\overline \delta_0$ and $C(\alpha_0)>0$ such that $\forall \left(\sigmazero, (\Xzero, \thetazero, \lambdazero; \delta)\right)\in B_{\overline\delta_0}(\Lsigma)\times B_{\overline\delta_0}(\LX\times\Ltheta\times\Llambda\times[0, \infty[)$
\begin{equation*}
||H_{\sigma} ||_{\Nsigma} \leq C(\alpha_0)\left( ||\sigmazero||^2_{\Lsigma} + ||(\Xzero, \thetazero, \lambdazero)||^2_{\LX\times\Ltheta\times\Llambda} + \delta \right). 
\end{equation*}
\end{enumerate}
\end{lemma}
\begin{proof}
\begin{enumerate}
\item Let $\overline \delta_0>0$ and let $(\sigmazero, \Xzero,\thetazero, \lambdazero; \delta)\in B_{\overline\delta_0}(\Lsigma)\times B_{\overline\delta_0}(\LX\times\Ltheta\times\Llambda\times[0, \infty[)$. 
\\ By Proposition \ref{Fi::regularity}, we have $F_3(\Xzero, \thetazero, \sigmazero; \delta)\in \hat C^{1, \alpha}(\Bbarre)$. Moreover,  it is compactly supported in $\BB$ and  $K := \supp F_3(\Xzero, \thetazero, \sigmazero; \delta)\not\subset(\Bnbarre\cap\Bsbarre)$. Therefore, $1 - \xi_N - \xi_S = 1$ on $K$ and we have 
Hence, 
\begin{equation*}
\begin{aligned}
||r^4 H_\sigma||_{C^{1, \alpha}(\Bbarre)} &=  ||r^4 \frac{X_Ke^{2\lambda_K}}{\rho^2}\frac{1 + \Xzero}{1 + \sigmazero}e^{2\lambdazero}F_3(\Xzero, \thetazero, \sigmazero, \delta)||_{C^{1, \alpha}(\Bbarre)} \\
&= \left|\left|r^4\frac{X_Ke^{2\lambda_K}}{\rho^2}F_3(\Xzero, \thetazero, \sigmazero; \delta)\right|\right|_{C^{1, \alpha}(K)}  \\
&\leq C\left|\left|r^4\frac{X_K e^{2\lambda_K}}{\rho^2}\right|\right|_{C^{1, \alpha}(K)}\, \left|\left| \frac{1 + \Xzero}{1 + \sigmazero}e^{2\lambdazero}\right|\right|_{C^{1, \alpha}(K)}\,||F_3(\Xzero, \thetazero, \sigmazero; \delta)||_{C^{1, \alpha}(K)} \\ 
&\leq C ||(1- \xi_N - \xi_S)F_3(\Xzero, \thetazero, \sigmazero; \delta)||_{C^{1, \alpha}(\mathbb R^3)} \\
&\leq C ||F_3(\Xzero, \thetazero, \sigmazero; \delta)||_{\hat C^{1, \alpha}(\Bbarre)}.
\end{aligned}
\end{equation*}
\item By Proposition \ref{Fi::regularity}, we have $F_3$ is continuously Fréchet differentiable on  $B_{\overline\delta_0}$. Hence, there exists $\delta_0\leq \overline\delta_0$ such that $\forall (\sigmazero, \Xzero, \thetazero;\delta)\in B_{\delta_0}$
\begin{equation*}
\begin{aligned}
F_3(\Xzero, \thetazero, \sigmazero; \delta) &= D_hF_3(0, 0, 0; 0;0)[(\Xzero, \thetazero, \sigmazero)] + D_\delta F_3(0, 0, 0; 0)\delta + O (||(\Xzero, \thetazero, \sigmazero)||^2 + \delta^2) \\
&=  D_\delta F_3(0, 0, 0; 0)\delta + O (||(\Xzero, \thetazero, \lambdazero)||^2 + \delta^2). 
\end{aligned} 
\end{equation*}
Therefore, 
\begin{equation*}
\begin{aligned}
||r^4 H_\sigma||_{C^{1, \alpha}(\Bbarre)} &\leq C ||F_3(\Xzero, \thetazero, \sigmazero; \delta)||_{\hat C^{1, \alpha}(\Bbarre)} \\
&\leq C   (||D_\delta F_3(0, 0, 0; 0)||\delta +  ||(\Xzero, \thetazero, \lambdazero)||^2 + \delta^2) \\ 
&\leq C \left( ||\sigmazero||^2_{\Lsigma} + ||\Xzero||^2_{\LX} +  ||\thetazero||^2_{\Ltheta} + \delta\right) \\ 
&\leq C(\alpha_0)\left( ||\sigmazero||^2_{\Lsigma} + ||(\Xzero, \thetazero, \lambdazero)||^2_{\LX\times\Ltheta\times\Llambda} + \delta \right). 
\end{aligned}
\end{equation*}

\end{enumerate}
\end{proof}
\noindent Now, we prove Proposition \ref{non:linear:sigma}. 
\begin{proof}
We apply Theorem \ref{Fixed::Point::2} with 
\begin{equation*}
L = \Delta_{\mathbb R^4}\;,\; N = N_{\sigma},
\end{equation*}
\begin{equation*}
\mathcal L = \Lsigma\;,\; \mathcal Q = \LX\times\Ltheta\times\Llambda \;, \; \mathcal P = [0, \delta_0[. 
\end{equation*}
By Lemma \ref{est:sigma}, the assumptions of Theorem \ref{Fixed::Point::2} are satisfied. Therefore, we obtain the desired result. 
\end{proof}

\subsection{Solving for $\displaystyle \left(B_\chi^{(N)}, B_{\chi'}^{(S)}, B_z^{(A)}\right)$}
In this section, we solve the equations 
\begin{equation}
\label{eq:for:B}
\begin{aligned}
B_{\chi}^{(N)}(0, \chi) &= 0, \\
 B_{\chi}^{(N)}(s, 0) &= 0, \\
\partial_s B^{(N)}_\chi &= 2\xi_N\sigma^{-1}e^{2\lambda}F_2(W, X, \sigma)(\rho, z)(s^2 + \chi^2), \\
 B^{(N)}_\chi(s, \chi) &= 0, \\
B_{\chi'}^{(S)}(0, \chi') &= 0, \\
 B_{\chi'}^{(S)}(s', 0) &= 0, \\
\partial_{s'} B^{(S)}_{\chi'} &= 2\xi_S\sigma^{-1}e^{2\lambda}F_2(W, X, \sigma)(s', \chi') ((s')^2 + (\chi')^2), \\
 B^{(S)}_{s'}(s', \chi') &= 0, \\
B_z^{(A)}(0, z) &= 0, \\
\partial_\rho B_z^{(A)}(\rho, z) &= 2 (1 - \xi_N^2 - \xi_S^2) \sigma^{-1}e^{2\lambda}F_2(W, X, \sigma)(\rho, z) -  (\partial_\rho\xi_N)B_z^{(N)} + (\partial_z\xi_N)B_\rho^{(N)} -(\partial_\rho\xi_S)B_z^{(S)} + (\partial_z\xi_S)B_\rho^{(S)}, \\
B_{\rho}^{(A)} &= 0,
\end{aligned}
\end{equation}
\noindent for $\displaystyle  \left(B_\chi^{(N)}, B_{\chi'}^{(S)}, B_z^{(A)}\right)$ in terms of the renormalised unknowns $\left((\Xzero, \Yzero), \thetazero, \lambdazero \right)$ and $\delta$. 
\\ By  Proposition \ref{support:f}, we have 
\begin{equation*}
\begin{aligned}
\xi_N\sigma^{-1}e^{2\lambda}F_2(W, X, \sigma)(\rho, z) &= 0, \\ 
\xi_S\sigma^{-1}e^{2\lambda}F_2(W, X, \sigma)(\rho, z) &= 0.
\end{aligned}
\end{equation*}
Therefore, 
\begin{equation*}
B_\chi^{(N)}(s, \chi) =  0\;,\; B_\chi^{(N)}(0, \chi) = 0 \;,\; B_{\chi'}^{(S)}(s', \chi') = 0 \; ,\; B_\chi^{(S)}(0, \chi') = 0, 
\end{equation*}
and 
\begin{equation*}
B^{(N)}_{\rho} = 0\; , \; B^{(N)}_{z} = 0 \; ,  \; B^{(S)}_{\rho} = 0\; , \; B^{(S)}_{z} = 0. 
\end{equation*}
\subsubsection{Linear problem}
In order to solve for $B_z^{(A)}$, we start with solving the linear problem: 
\begin{equation*}
\partial_\rho B_z^{(A)}(\rho, z) = (1 - \xi_N^2 - \xi_S^2)H^{(A)}_B\quad\quad\text{for}\quad (H^{(A)}_B, 0, 0)\in \NB 
\end{equation*}
with boundary condition 
\begin{equation*}
B_z^{(A)}(0, z) = 0. 
\end{equation*}
We state the following result
\begin{Propo}
Let $\displaystyle \left((0, 0, H^{(A)}_B\right)\in\NB.$ Then, there exists a unique solution $\left(0, 0, B_z^{(A)}\right)\in\LB$ which solves the equation
\begin{equation*}
\partial_\rho B_z^{(A)}(\rho, z) = (1 - \xi_N^2 - \xi_S^2)H^{(A)}_B\quad\quad\text{and}\quad B_z^{(A)}(0, z) = 0. 
\end{equation*}
It is given by 
\begin{equation*}
B_z^{(A)}(\rho, z) = \int_0^\rho (1 - \xi_N^2 - \xi_S^2)(\tilde \rho, z)H^{(A)}_B(\tilde \rho, z)\, d\tilde\rho. 
\end{equation*}
Moreover, there exists $C(\alpha_0)>0 $
\begin{equation*}
\left|\left| \left(0, 0, B_z^{(A)}\right) \right|\right|_{\LB} \leq C(\delta_0) \left|\left|\left(0, 0, H^{(A)}_B\right)\right|\right|_{\NB}. 
\end{equation*}
\end{Propo}

\begin{proof}

\begin{enumerate}
\item First of all, since 
\begin{equation*}
\left|\left|\frac{(1 + \rho^{15})(1 + r^{10})}{\rho^{15}}H_B^{(A)} \right|\right|_{\hat C^{1, \alpha_0}(\BAbarre\cup\BHbarre)} < \infty, 
\end{equation*}
$B_z^{(A)}$ is well-defined on $\BAbarre\cup\BHbarre$. 
\item As for the estimates, we will show that 
\begin{equation*}
\left((1 - \xi_N^2 - \xi_S^2)\frac{(1 + \rho^{10})(1 + r^{10})}{\rho^{10}}B_z^{(A)}\right)_{\mathbb R^3}\in C^{1, \alpha_0}(\mathbb R^3).
\end{equation*}
First of all, we recall the following change of variables 
 \begin{equation*}
\begin{aligned}
x &= \rho\cos\vartheta \\
y &= \rho\sin\vartheta.
\end{aligned}
\end{equation*}
We have 
\begin{equation*}
\partial_x = \frac{x}{\sqrt{x^2 + y^2}}\partial_\rho \;,\; \partial_y = \frac{y}{\sqrt{x^2 + y^2}}\partial_\rho. 
\end{equation*}
\begin{itemize}
\item if $(\rho, z)\in \tilde \Axis$ where $\tilde\Axis$ is a neighbourhood of the axis. We assume that $\rho< 1$. Then 
\begin{equation*}
\frac{\rho^{15}}{1 + \rho^{15}} \leq \frac{\rho^{10}}{1 + \rho^{10}} 
\end{equation*}
 and we have 
 \begin{equation*}
 \begin{aligned}
 \left|B_z^{(A)}\right| &\leq C\int_0^\rho \frac{\tilde\rho^{10}}{1 + \tilde\rho^{10}}\frac{1}{1 + (1 + \tilde \rho^2 + z^2)^5}\, d\tilde \rho \\ 
 &\leq C \frac{\rho^{10}}{1 + \rho^{10}}\int_0^\rho \frac{1}{1 + (1 + \tilde \rho^2 + z^2)^5}\, d\tilde \rho \\ 
  &\leq C \frac{\rho^{10}}{(1 + \rho^{10})(1 + r^{10})}\int_0^\rho \frac{1 + (1 + \rho^2 + z^2)^5}{1 + (1 + \tilde \rho^2 + z^2)^5}\, d\tilde \rho. \\ 
\end{aligned}
 \end{equation*}
 The integrand is bounded uniformly in $z$. Moreover, since $\rho\in[0, 1]$, we can bound the term $\displaystyle \int_0^\rho \frac{1 + (1 + \rho^2 + z^2)^5}{1 + (1 + \tilde \rho^2 + z^2)^5}\, d\tilde \rho$ uniformly in $(\rho, z)$. 
 \noindent 
\item Away from the axis, say $(\rho, z)\in [\rho_0, \infty[\times \mathbb R$ for some $\rho_0>1$. We have 
\begin{equation*}
1 \leq \frac{\rho^{10}}{(1 + \rho^{10})}
\end{equation*}
and 
\begin{equation*}
\begin{aligned}
\int_0^\rho \frac{\tilde\rho^{15}}{1 + \tilde\rho^{15}}\frac{1}{1 + (1 + \tilde \rho^2 + z^2)^5}\, d\tilde \rho &\leq C \int_0^\rho \frac{\tilde\rho^{8}}{1 + \tilde\rho^{15}}\frac{4\tilde\rho(1 + \tilde\rho^2 + z^2)^3}{(1 + \tilde \rho^2 + z^2)^5}\, d\tilde \rho \\
&\leq C r(\rho, z)^{-10} \leq C\frac{1}{1 + r^{10}}
\end{aligned}
\end{equation*}
Hence, 

\begin{equation*}
 \begin{aligned}
 \left|B_z^{(A)}\right| &\leq C\int_0^\rho \frac{\tilde\rho^{15}}{1 + \tilde\rho^{15}}\frac{1}{1 + (1 + \tilde \rho^2 + z^2)^5}\, d\tilde \rho \\ 
 &\leq C \frac{\rho^{10}}{(1 + \rho^{10})(1 + r^{10})}. 
\end{aligned}
 \end{equation*}

\item Away from the axis, the $L^\infty$ estimate and the Hölder estimate of the derivative are straightforward. 

\item Near the axis, we use the above change of coordinates in order to estimate the derivatives and the Hölder part. 

\end{itemize}

\end{enumerate}

\end{proof}

\subsubsection{Non-linear estimates}
We apply Theorem \ref{Fixed::Point::2} in order to obtain
\begin{Propo}
\label{non:linear:B}
Let $\alpha_0\in(0, 1)$ and let $\overline\delta_0>0$. Then, there exists $0<\delta_0 \leq \overline\delta_0$ such that $\displaystyle \forall \left(\Xzero, \thetazero, \lambdazero, \delta\right)\in B_{\delta_0}\left(\LX\times\Ltheta\times\Llambda\times[0, \infty[\right)$ there exists a unique one parameter family $\left(\sigmazero, \left(0, 0, B_z^{(A)}\right)\right)$ depending on $\left(\Xzero, \thetazero, \lambdazero, \delta\right)\in B_{\delta_0}\left(\Lsigma\times\LB\right)$ which solves \eqref{sigmazero} and \eqref{eq:for:B} and which satisfies
\begin{equation*}
        \left|\left|\left(\sigmazero, \left(0, 0, B_z^{(A)}\right)\right)\left(\Xzero, \thetazero, \lambdazero, \delta\right)\right|\right|_{\Lsigma\times\LB} \le C(\alpha_0)\left(||\Xzero||^2_{\LX} + ||\thetazero||^2_{\Ltheta} + ||\lambdazero||^2_{\Llambda} + \delta\right)
  \end{equation*}
  and $\displaystyle \forall \left(\Xzero_i, \thetazero_i, \lambdazero_i, \delta_i\right)\in B_{\delta_0}\left(\LX\times\Ltheta\times\Llambda\times[0, \infty[\right)$,
        \begin{equation*}
        \begin{aligned}
        &\left|\left|\left(\sigmazero, \left(0, 0, B_z^{(A)}\right)\right)\left(\Xzero_1, \thetazero_1, \lambdazero_1, \delta_1\right) - \left(\sigmazero, \left(0, 0, B_z^{(A)}\right)\right)\left(\Xzero_2, \thetazero_2, \lambdazero_2, \delta_2\right)\right|\right|_{\Lsigma\times \LB} \\ 
        &\le C(\alpha_0)\left(\left(\left|\left|\left(\Xzero_1, \thetazero_1, \lambdazero_1\right)\right|\right|_{\LX\times\Ltheta\times\Llambda} + \left|\left|\left(\Xzero_2, \thetazero_2, \lambdazero_2\right)\right|\right|_{\LX\times\Ltheta\times\Llambda}\right)\left|\left|\left(\Xzero_1, \thetazero_1, \lambdazero_1\right) -\left(\Xzero_2, \thetazero_2, \lambdazero_2\right)\right|\right|_{\LX\times\Ltheta\times\Llambda} \right. \\
        &\left. + \left|\delta_1 - \delta_2\right| \right).
        \end{aligned}
        \end{equation*}
\end{Propo}
\noindent Before we prove the above proposition, we introduce the following notations
\begin{enumerate}
\item Define $H_B^{(A)}$ on $\Bbarre$ to be the mapping
\begin{equation*}
H_B^{(A)}(\rho, z) = 2\sigma^{-1}e^{2\lambda}F_2(\thetazero, \Xzero, \sigmazero, \delta)(\rho, z).
\end{equation*}
\item Let $\overline\delta_0>0$ be obtained by Proposition \eqref{non:linear:sigma}. Define the nonlinear operator $N_B$ on  $B_{\overline\delta_0}(\LB)\times B_{\overline\delta_0}\left(\LX\times\Ltheta\times\Llambda\times[0, \infty[\right)$ by
\begin{equation*}
N_B\left(\left(0, 0, B_z^{(A)}\right), (\Xzero, \thetazero, \lambdazero; \delta)\right)(\rho, z) :=  \frac{2}{\rho}\left(\sigmazero (\Xzero, \thetazero, \lambdazero, \delta) \right)^{-1}e^{2\lambda}F_2(\thetazero, \Xzero, \sigmazero (\Xzero, \thetazero, \lambdazero, \delta), \delta)(\rho, z)
\end{equation*} 
where $\sigmazero$ is the mapping obtained by Proposition \ref{non:linear:sigma}. 
\end{enumerate}
\begin{lemma}
\label{est:B}
\begin{enumerate}
\item $\displaystyle N_B\left(B_{\overline\delta_0}(\LB)\times B_{\overline\delta_0}(\LX\times\Ltheta\times\Llambda\times[0, \infty[)\right)\subset \NB$. 
\item There exist $0<\delta_0\leq\overline \delta_0$ and $C(\alpha_0)>0$ such that $\forall \left(\Xzero, \thetazero, \lambdazero; \delta\right)\in B_{\overline\delta_0}(\LX\times\Ltheta\times\Llambda\times[0, \infty[)$
\begin{equation}
\label{non:linear:NB}
||(0, 0, H^{(A)}_{B}) ||_{\NB} \leq C(\alpha_0)\left( ||(\Xzero, \thetazero, \lambdazero)||^2_{\LX\times\Ltheta\times\Llambda} + \delta \right). 
\end{equation}
\end{enumerate}
\end{lemma}
\begin{proof}
We use the same arguments used to prove Lemma \ref{est:sigma}: the compact support of $F_2(\thetazero, \Xzero, \sigmazero; \delta)$ and the differentiability of $F_2$  with respect to each variable.  
\\ In order to obtain \eqref{non:linear:NB}, we write 
\begin{equation*}
\begin{aligned}
F_2(\thetazero, \Xzero, \sigmazero (\Xzero, \thetazero, \lambdazero; \delta); \delta) &= D_hF_3(0, 0, 0; 0)[(\Xzero, \thetazero, \sigmazero(\Xzero, \thetazero, \lambdazero; \delta) )] + D_\delta F_3(0, 0, 0; 0)\delta  \\
&+ O (||(\Xzero, \thetazero, \sigmazero(\Xzero, \thetazero, \lambdazero; \delta))||^2 + \delta^2) \\
&=  D_\delta F_3(0, 0, 0; 0)\delta + O (||(\Xzero, \thetazero, \sigmazero(\Xzero, \thetazero, \lambdazero; \delta) )||^2 + \delta^2). 
\end{aligned} 
\end{equation*}
By Proposition \ref{non:linear:sigma}, we have 
\begin{equation*}
        \left|\left|\sigmazero\left(\Xzero, \thetazero, \lambdazero, \delta\right)\right|\right|_{\Lsigma} \le C(\alpha_0)\left(||\Xzero||^2_{\LX} + ||\thetazero||^2_{\Ltheta} + ||\lambdazero||^2_{\Llambda} + \delta\right)
  \end{equation*}
Hence, 
\begin{equation*}
||F_2(\thetazero, \Xzero, \sigmazero (\Xzero, \thetazero, \lambdazero; \delta); \delta)||_{\hat C^{1, \alpha_0}} \leq C(\alpha_0) \left(||\Xzero||^2_{\LX} + ||\thetazero||^2_{\Ltheta} + ||\lambdazero||^2_{\Llambda} + \delta\right). 
\end{equation*}
The estimate \eqref{non:linear:NB} follows from the latter and the compact support of $F_2(\thetazero, \Xzero, \sigmazero (\Xzero, \thetazero, \lambdazero; \delta); \delta)$. 
\end{proof}
\noindent Now, we prove Proposition \ref{non:linear:B}. 
\begin{proof}
\begin{enumerate}
\item First of all, by Proposition \ref{non:linear:sigma}, there exists a solution map $\sigmazero(\Xzero, \thetazero, \lambdazero; \delta)$ defined on $B_{\delta_0}$ which solves \eqref{sigmazero}.  
\item We apply Theorem \ref{Fixed::Point::2} with $\displaystyle L = \partial_\rho\;,\; N = N_B\;,\; \mathcal L = \LB\;,\; \mathcal Q = \LX\times\Ltheta\times\Llambda$ and $\displaystyle \mathcal P = [0, \delta_0[$. By the previous lemma, all the assumptions are satisfied and we obtain the desired result. 
\end{enumerate}
\end{proof}

\subsection{Solving for $(X, Y)$}

\subsubsection{Linear problem}
The aim of this section is to solve the linear problem for $\left(\Xzero, \Yzero \right)$:  
\begin{equation}
\label{linear:X:Y}
        \begin{aligned}
        &\Delta_{\mathbb{R}^3}\overset{\circ}{X} + \frac{2\partial Y_K\cdot\partial\overset{\circ}{Y}}{X_K} - \frac{2|\partial Y_K|^2}{X_K^2}\overset{\circ}{X} + 2\frac{\partial X_K\cdot\partial Y_K}{X_K^2}\overset{\circ}{Y} = H_X , \\
        &\Delta_{\mathbb{R}^3}\overset{\circ}{Y} - \frac{2\partial Y_K\cdot\partial\overset{\circ}{X}}{X_K} - \frac{(|\partial X_K|^2+|\partial Y_K|^2)}{X_K^2}\overset{\circ}{Y} = H_Y .
        \end{aligned}
\end{equation}
where $\left(X_K, Y_K\right)$ are the renormalised quantities for the Kerr metric and where $\displaystyle \left(H_X, H_Y\right)\in \NX\times \NY$. 
\\ More precisely, the remaining of this section is devoted to the proof of  the following result
\begin{Propo}
\label{linear:est:XY}
Let $\displaystyle \left(H_X, H_Y\right)\in \NX\times \NY$. Then there exists a unique solution $\left(\Xzero, \Yzero\right)\in \LX\times \LY$ of \eqref{linear:X:Y}. Moreover, there exists $C>0$ independent of $\left(H_X, H_Y\right)$ such that 
\begin{equation*}
\left|\left|\left(\Xzero, \Yzero\right)\right|\right|_{\LX\times\LY} \leq C\left|\left|\left(H_X, H_Y\right)\right|\right|_{\NX\times\NY}
\end{equation*}
\end{Propo}
\paragraph{Existence of weak solutions to the linear non homogeneous problem}
Let $\left(H_X, H_Y\right)\in C_0^\infty(\mathbb R^3)$ and consider the system \eqref{linear:X:Y}. Set
\begin{equation}
\label{lagrangian}
    \mathcal{L}(\Xzero,\Yzero) := \int_{\mathbb{R}^3}\;\left( |\partial\Xzero+X^{-1}_K(\partial Y_K)\Yzero|^2+|\partial\Yzero-X^{-1}_K(\partial Y_K)\Xzero|^2+X^{-2}_K|\Xzero\partial Y_K-\Yzero\partial X_K|^2 + 2\Xzero H_X + 2\Yzero H_Y\right).
\end{equation}
We state the following lemmas from \cite{chodosh2017time}. We  recall the proofs in order to be self-contained. 
\begin{lemma}
\label{Lagrangian:Propo}
\begin{itemize}
\item $\mathcal L$ is well-defined on $\displaystyle \dot H_{axi}^1(\mathbb R^3)\times \left(\dot H_{axi}^1(\mathbb R^3)\cap L^2_{|\partial h|}(\mathbb R^3) \right)$. 
\item \eqref{linear:X:Y} are the Euler-Lagrange equations associated with the Lagrangian $\mathcal L$.  
\end{itemize}
\end{lemma}
\noindent Before we prove the above result, we will need the following inequalities 
\begin{lemma}\label{lemma::12}
Let $f(\rho,z)\in C^\infty_0(\mathbb{R}^3)$. Then
\begin{equation*}
    \int_{\mathbb{R}^3}\;|\partial X_K|^2f^2 \le C\int_{\mathbb{R}^3}\;X^2_K|\partial f|^2
\end{equation*}
\end{lemma}
\begin{proof}
By Lemma \ref{asymptotics:h1}, there exist $c, C>0$ such that $\forall (\rho, z)\in\Bbarre$, we have 
\begin{equation*}
c e^h\left|\partial h\right| \leq \left|\partial X_K\right| \leq C e^h\left|\partial h\right|.  
\end{equation*} 
Therefore, it suffices to show that 
\begin{equation*}
    \int_{\mathbb{R}^3}\;e^{2h}|\partial h|^2f^2 \le C\int_{\mathbb{R}^3}\;e^{2h}|\partial f|^2
\end{equation*}
Recall that $\Delta_{\mathbb R^3}h = 0$ and note that 
\begin{equation*}
\text{div}_{\mathbb R^3}(e^{2h}f^2 \partial h) =  e^{2h}f^2\Delta_{\mathbb R^3}h + \nabla h\cdot\nabla(e^{2h}f^2) =  \nabla h\cdot\nabla(e^{2h}f^2).
\end{equation*}
Now,  let $0<\epsilon<1$,  $Z>0$ large , $R> 0$ large so that $f(\rho, z) = 0 \;\forall \rho\geq 0$,  and consider the domain $ \mathcal{U}(R,\epsilon,h)\subset \mathbb R^3$ defined by 
\begin{equation*}
   {U}(R,\epsilon,Z) := \left\{ (r,\phi,z):  \; \epsilon<r<R \;, 0<\phi<2\pi\;,\; |z|<Z\right\}.
\end{equation*}
We apply the divergence theorem to obtain 
\begin{equation*}
\int_{{U}(R,\epsilon,Z)} \text{div}_{\mathbb R^3}(e^{2h}f^2 \partial h) \rho d\rho d\phi dz = -2\pi\int^{Z}_{-Z}\;(e^{2h}f^2\partial_{\rho}h)(\epsilon,z)\,dz.
\end{equation*}
Therefore,
\begin{equation*}
\int_{\mathbb R^3} \text{div}_{\mathbb R^3}(e^{2h}f^2 \partial h) \rho d\rho d\phi dz = - 2\pi\lim_{\epsilon\to 0}\int^{\infty}_{-\infty }\;(e^{2h}f^2\partial_{\rho}h)(\epsilon,z)\,dz.
\end{equation*}
We have 
\begin{align*}
    \partial_{\rho}h(e^{2h})(\rho, z) &= \left(\sqrt{\rho^2+(z-\gamma)^2}-(z-\gamma)\right)\frac{\rho}{\sqrt{\rho^2+(z-\gamma)^2}}\left(\sqrt{\rho^2+(z+\gamma)^2}+(z+\gamma)\right)^{2} \\
    &+\left( \sqrt{\rho^2+(z-\gamma)^2}-(z-\gamma)\right)^{2}\frac{\rho}{\sqrt{\rho^2+(z+\gamma)^2}}\left( \sqrt{\rho^2+(z+\gamma)^2}+(z+\gamma)\right). 
\end{align*}
Therefore, there exists $\overline f(z, \gamma)>0$ which is compactly supported in the z-variable such that $\rho\in[0, 1]$ 
\begin{equation*}
\left|(e^{2h}f^2 \partial h)(\rho, z)\right| \leq \overline f(z, \gamma). 
\end{equation*}
Moreover, 
\begin{equation*}
\lim_{\epsilon\to 0}\, \partial_{\rho}h(e^{2h})(\epsilon, z) = 0. 
\end{equation*}
Therefore, by the dominated convergence theorem, we obtain 
\begin{equation*}
\int_{\mathbb R^3}  \nabla h\cdot\nabla(e^{2h}f^2) \rho d\rho d\phi dz  = \int_{\mathbb R^3} \text{div}_{\mathbb R^3}(e^{2h}f^2 \partial h) \rho d\rho d\phi dz = 0. 
\end{equation*}
Now, we compute: 
\begin{equation*}
 \nabla h\cdot\nabla(e^{2h}f^2)  = 2 e^{2h}\left( f^2|\nabla h|^2 + f \nabla h\cdot\nabla f \right). 
\end{equation*}
Therefore
\begin{equation*}
\begin{aligned}
\int_{\mathbb R^3}\, e^{2h}|\partial h|^2 f^2 &\leq  \int_{\mathbb R^3}\, e^{2h}|f|\,|\partial f|\,|\partial h| \\
&\leq  q\int_{\mathbb{R}^3}\;e^{2h}|\partial f|^2 + \frac{1}{4q}\int_{\mathbb{R}^3}\;e^{2h}f^2|\partial h|^2 \quad \forall q>0. 
\end{aligned}
\end{equation*}
Finally, we choose $q$ so that we obtain the desired ineaquality. 
\end{proof}

\begin{lemma}\label{lemma::13}
Let $f\in C^{\infty}_0(\mathbb{R}^3)$. Then,
\begin{equation*}
    \int_{\mathbb{R}^3}\;\left<r\right>^{-2}f^2 \le C \int_{\mathbb{R}^3}\;|\partial f|^2.
\end{equation*}
where, $\left<r\right>$ is given by
\begin{equation*}
    \left<r\right> := (1+x^2+y^2+z^2)^{\frac{1}{2}}.
\end{equation*}
\end{lemma}

\begin{proof}
We apply Theorem \ref{bartnik} with $p=2$ and $\delta$ satisfying:
\begin{align*}
    -2\delta-3 = -2 \quad\Longleftrightarrow\quad \delta = -\frac{1}{2}.
\end{align*}
Thus,
\begin{align*}
    ||f_r||^2_{2,-\frac{3}{2}} &= \int_{\mathbb{R}^3}\; |\left<r\right>^{-1(x,y,z)^t\cdot\partial f}|^2\left<r\right>^{3-3} ,\\
    &\le \int_{\mathbb{R}^3}\;|\partial f|^2.
\end{align*}
Therefore, 
\begin{equation*}
    \int_{\mathbb{R}^3}\;\left<r\right>^{-2}f^2 \le C \int_{\mathbb{R}^3}\;|\partial f|^2.
\end{equation*}
\end{proof}
\noindent Now, we state the following lemma
\begin{lemma}
\label{coerc::1}
There exists $C\ge 0$ such that if 
\begin{equation*}
    \Xzero\in C^{\infty}_0(\mathbb{R}^3), \quad \Yzero\in C^{\infty}_0(\mathbb{R}^3), \quad \int_{\mathbb{R}^3}\;\frac{|\partial X_K|^2}{X^2_K}\Yzero^2 <\infty,
\end{equation*}
then
\begin{equation}
    \int_{\mathbb{R}^3}\;\left(|\partial \Xzero|^2+|\partial \Yzero|^2+|\partial h|^2\Yzero^2 \right) \le C \left(\mathcal{L}(\Xzero,\Yzero) + 2\int_{\mathbb{R}^3}\;\left( |\Xzero H_X|+|\Yzero H_Y|\right) \right).
\end{equation}
\end{lemma}
\begin{proof}
\begin{enumerate}
    \item Set
\begin{equation}
\label{L::0}
    \mathcal{L}_0(\Xzero,\Yzero) := \mathcal{L}(\Xzero,\Yzero) - 2\int_{\mathbb{R}^3}\;\left( |\Xzero H_X|+|\Yzero H_Y|\right).
\end{equation} 
\end{enumerate}
\item We have
\begin{align*}
    \left|\partial\Yzero-X^{-1}_K\partial X_k\Yzero\right|^2 &= \left| \partial\Yzero-X^{-1}_K\partial Y_k\Xzero+X^{-1}_K\partial Y_k\Xzero -X^{-1}_K\partial X_k\Yzero\right|^2 ,\\
    &\le \left|\partial\Yzero-X^{-1}_K\partial Y_k\Xzero\right|^2 +X^{-2}_K\left|\partial Y_k\Xzero-\partial X_k\Yzero\right|^2.
\end{align*}
Hence,
\begin{align*}
    \int_{\mathbb{R}^3}\;\left|\partial\Yzero-X^{-1}_K\partial X_k\Yzero\right|^2 &\le C\mathcal{L}_0(\Xzero,\Yzero).
\end{align*}
\item Now, we define $\tilde{Y}: \mathbb R^{3}\backslash\left\{0\right\}\to\mathbb R$ by $\tilde{Y} := X^{-1}_K\Yzero$ and compute
\begin{align*}
    \partial\Yzero-X^{-1}_K\partial X_k\Yzero &=  \partial X_K\tilde{Y}+ X_K\partial\tilde{Y}-X^{-1}_K\partial X_k\tilde{Y} ,\\
    &= X_K\partial\tilde{Y}.
\end{align*}
Thus,
\begin{align*}
    \int_{\mathbb{R}^3}\;\left|X_K\partial\tilde{Y}\right|^2 &\le C\mathcal{L}_0(\Xzero,\Yzero).
\end{align*}
By Lemma \ref{lemma::12}, we have 
\begin{equation*}
    \int_{\mathbb{R}^3}\;|\partial X_K|^2\tilde{Y}^2\le C\int_{\mathbb{R}^3}\;\left|X_K\partial\tilde{Y}\right|^2 \le C\mathcal{L}_0(\Xzero,\Yzero),
\end{equation*}
\begin{equation*}
    \int_{\mathbb{R}^3}\;|\partial h|^2\Yzero^2 = \int_{\mathbb{R}^3}\;X^{2}_K|\partial h|^2\tilde{Y}^2 \le C\int_{\mathbb{R}^3}\;|\partial X_K|^2\tilde{Y}^2 \le  C\mathcal{L}_0(\Xzero,\Yzero), 
\end{equation*}
and 
\begin{align*}
    \int_{\mathbb{R}^3}\;|\partial Y_K|^2X^{-2}_K\Xzero^2 &= \int_{\mathbb{R}^3}\;|\Xzero\partial Y_K|^2X^{-2}_K ,\\
    &= \int_{\mathbb{R}^3}\;X^{-2}_K|\Xzero\partial Y_K-\Yzero\partial X_K+\Yzero\partial X_K|^2 ,\\
    &\le \int_{\mathbb{R}^3}\;X^{-2}_K|\Xzero\partial Y_K-\Yzero\partial X_K|^2 + +\int_{\mathbb{R}^3}\;\tilde{Y}^2|\partial X_K|^2 ,\\
    &\le C\mathcal{L}_0(\Xzero,\Yzero).
\end{align*}
Similarly, we obtain
\begin{align*}
    \int_{\mathbb{R}^3}\;|\partial\Xzero|^2 \le C\int_{\mathbb{R}^3}\;|\partial\Xzero+X^{-1}_K\partial Y_K\Yzero|^2 + C\int_{\mathbb{R}^3}\;X^{2}_K\Yzero^2|\partial Y_K|^2.
\end{align*}
Now, we show that
\begin{equation*}
    \int_{\mathbb{R}^3}\;X^{2}_K\Yzero^2|\partial Y_K|^2 \le C\int_{\mathbb{R}^3}\;\left<r\right>^{-2}\Yzero^2.
\end{equation*}
By Lemma \ref{decay:estimates}, 
\begin{equation*}
\rho\frac{|\partial Y_K|}{X^2_K} \le C\left<r\right>^{-4} \quad\text{on}\quad \BAbarre\cup\BHbarre.     
\end{equation*}
This implies
\begin{equation*}
\rho\frac{|\partial Y_K|^2}{X^2_K} \le C\left(\frac{X_K}{\rho}\right)^2\left<r\right>^{-8}.    
\end{equation*}
Moreover, 
 \begin{equation*}
\forall (\rho, z)\in \BB\;:\;      \frac{X_K}{\rho} = \rho\frac{\Pi}{\Sigma^2\Delta} = \rho^2\frac{(r(\rho, z)^2+a^2)^2-a\rho^2}{(r(\rho, z)^2+\frac{z^2}{(r(\rho, z) - M)^2})(r(\rho, z)^2-2Mr(\rho, z)+a^2)},
 \end{equation*}
 when $\langle r\rangle$ goes to $\infty$,
 \begin{equation*}
     \left|\frac{X_K}{\rho}\right|^2 \le C \left<r\right>^2 
 \end{equation*}
 Thus,
 \begin{equation*}
     \frac{|\partial Y_K|^2}{X^2_K} \le\left<r\right>^{-6} \le C\left<r\right>^{-2}.
 \end{equation*}
Therefore, by Lemma \ref{lemma::13}, we obtain 
 \begin{equation*}
        \int_{\mathbb{R}^3}\;X^{2}_K\Yzero^2|\partial Y_K|^2 \le C\int_{\mathbb{R}^3}\;\left<r\right>^{-2}\Yzero^2 \le C\int_{\mathbb{R}^3}\;|\partial \Yzero|^2 \le C\mathcal{L}_0(\Xzero,\Yzero).
 \end{equation*}
 This finishes the proof.
\end{proof}
\noindent Now, we prove Lemma \ref{Lagrangian:Propo}
\begin{proof}
\begin{enumerate}
\item Let $(\Xzero, \Yzero)\in \dot H_{axi}^1(\mathbb R^3)\times \left(\dot H_{axi}^1(\mathbb R^3)\cap L^2_{|\partial h|}(\mathbb R^3) \right)$. We show that 
\begin{equation*}
X^{-1}_K(\partial Y_K)\Yzero\;,\; X^{-1}_K(\partial Y_K)\Xzero \in L^2(\mathbb R^3). 
\end{equation*}
By Lemma \ref{decay:estimates}, we have 
\begin{equation*}
\frac{|\partial Y_K|}{X_K} \leq Cr^{-4}\frac{X_K}{\rho} \quad\text{on}\quad \BAbarre\cup\BHbarre
\end{equation*}
and
\begin{equation*}
{|\underline\partial Y_K|} \leq Cs^3 \quad\text{on}\quad \Bnbarre \quad \text{on}\quad {|\underline\partial Y_K|} \leq C(s')^3 \quad\text{on}\quad \Bsbarre. 
\end{equation*}
Moreover, by Proposition \ref{extendibility}, we have 
\begin{equation*}
X_K(\rho, z) = s^2X_N(s^2, \chi^2) \quad\quad X_N(0, \chi), X_N(s, 0) >0. 
\end{equation*}
Therefore, 
\begin{equation*}
\begin{aligned}
\int_{\mathbb R^3}X^{-2}_K|\partial Y_K|^2\Yzero^2 &= 2\pi\int_{\Bbarre}\,X^{-2}_K|\partial Y_K|^2\Yzero^2 d\rho dz \\
&= 2\pi \int_{\BAbarre\cup \BHbarre}\,(1 - \xi_N - \xi_S)X^{-2}_K|\partial Y_K|^2\Yzero^2 \rho d\rho dz  \\
&+ 2\pi\int_{\Bnbarre}\xi_N \frac{1}{s^4X^{2}_N}|\underline \partial Y_K|^2\Yzero^2{s\chi}\, dsd\chi   \\
&+ 2\pi\int_{\Bsbarre}\xi_S \frac{1}{(s')^4X^{2}_S}|\underline \partial Y_K|^2\Yzero^2{s'\chi'}\, ds'd\chi'. \\ 
\end{aligned}
\end{equation*}
By the decay estimates for $Y_K$, the above integrals are all finite. Therefore, 
\begin{equation*}
X^{-1}_K(\partial Y_K)\Yzero\;,\; X^{-1}_K(\partial Y_K)\Xzero \in L^2(\mathbb R^3)
\end{equation*}
and $\mathcal L (\Xzero, \Yzero)$ is finite. 
\item Let $(\Xzero, \Yzero) \in  \dot H^1_{axi}(\mathbb{R})^3 \times\left(\dot H^1_{axi}(\mathbb{R})^3\cap L^2_{\partial h}(\mathbb{R}^3)\right)$ be minimiser for $\mathcal L$, let $(\phi, \psi)\in C^{\infty}_{0}(\mathbb{R}^3)$ and consider the real-valued function
\begin{equation*}
L(\tau):= \mathcal L((\Xzero, \Yzero) + \tau (\phi, \psi)). 
\end{equation*}
$L$ has a minimum at $\tau  = 0$. Therefore, 
\begin{equation*}
L'(0) = 0. 
\end{equation*}
Hence, 
\begin{equation*}
\left.\frac{d}{d\tau}\right|_{\tau = 0} \mathcal L((\Xzero, \Yzero) + \tau (\phi, \psi)) = 0. 
\end{equation*}
Now, we compute  
\begin{equation*}
\begin{aligned}
\left.\frac{d}{d\tau}\right|_{\tau = 0} \mathcal L((\Xzero, \Yzero) + \tau (\phi, \psi)) &=  \left.\frac{d}{d\tau}\right|_{\tau = 0} \int_{\mathbb R^3} \left( |\partial(\Xzero+t\phi)+X^{-1}_K(\partial Y_K)(\Yzero+t\psi)|^2\right. \\ 
& +|\partial\Yzero-X^{-1}_K(\partial Y_K)(\Xzero+ t\phi)|^2 +X^{-2}_K|(\Xzero + t\phi)\partial Y_K-(\Yzero + t\psi)\partial X_K|^2  \\
&\left.  + 2(\Xzero + t\phi) H_X + 2(\Yzero + t\psi) H_Y\right) \\
&= 2 \int_{\mathbb R^3}\, \partial\Xzero\cdot\partial \phi-\frac{2\partial Y_K\cdot\partial\overset{\circ}{Y}}{X_K}\phi+\frac{2|\partial Y_K|^2}{X_K^2}\overset{\circ}{X}\phi-2\frac{\partial X_K\cdot\partial Y_K}{X_K^2}\overset{\circ}{Y}\phi+H_X\phi \\
&+ \partial\Yzero\cdot\partial \psi+\frac{2\partial Y_K\cdot\partial\overset{\circ}{X}}{X_K}\psi+\frac{(|\partial X_K|^2+|\partial Y_K|^2)}{X_K^2}\overset{\circ}{Y}\psi+H_Y\psi
\end{aligned}
\end{equation*}
Here, we applied the dominated convergence theorem in order to invert  the derivative and the integral. Therefore, for $(\phi, 0)$ we obtain 
\begin{equation*}
\int_{\mathbb{R}^3}\;\partial\Xzero\cdot\partial\phi-\frac{2\partial Y_K\cdot\partial\overset{\circ}{Y}}{X_K}\phi+\frac{2|\partial Y_K|^2}{X_K^2}\overset{\circ}{X}\phi-2\frac{\partial X_K\cdot\partial Y_K}{X_K^2}\overset{\circ}{Y}\phi+H_X\phi = 0
\end{equation*}
and for $(0, \psi)$,  we have 
\begin{equation*}
\int_{\mathbb{R}^3}\;\partial \Yzero\cdot\partial \psi+\frac{2\partial Y_K\cdot\partial\overset{\circ}{X}}{X_K}\psi+\frac{(|\partial X_K|^2+|\partial Y_K|^2)}{X_K^2}\overset{\circ}{Y}\psi+H_Y\psi = 0. 
\end{equation*}
Finally, we integrate by part
\begin{equation*}
\begin{aligned}
& \int_{\mathbb R^3}\; -\Delta_{\mathbb{R}^3}\overset{\circ}{X}\phi - \frac{2\partial Y_K\cdot\partial\overset{\circ}{Y}}{X_K}\phi + \frac{2|\partial Y_K|^2}{X_K^2}\overset{\circ}{X}\phi - 2\frac{\partial X_K\cdot\partial Y_K}{X_K^2}\overset{\circ}{Y}\phi + H_X\phi = 0 , \\
&  \int_{\mathbb R^3}\; -\Delta_{\mathbb{R}^3}\overset{\circ}{Y}\psi + \frac{2\partial Y_K\cdot\partial\overset{\circ}{X}}{X_K}\psi + \frac{(|\partial X_K|^2+|\partial Y_K|^2)}{X_K^2}\overset{\circ}{Y}\psi + H_Y\psi = 0. 
\end{aligned}
\end{equation*}
The latter holds for all $(\phi, \psi)\in C^{\infty}_{0}(\mathbb{R}^3)$. Therefore, \eqref{linear:X:Y} are the Euler-Lagrange equations associated to $\mathcal L$. 
\end{enumerate}
\end{proof}
\noindent We apply classical variational methods in order to prove the existence and uniqueness of weak solutions to \eqref{linear:X:Y}. More precisely, we state
\begin{lemma}
\label{X:Y:weak}
Let $H_X,H_Y\in C^{\infty}_0(\mathbb{R}^3)$. Then, the system \eqref{linear:X:Y} has a unique weak solution  $(\Xzero,\Yzero)\in\dot H^1_{axi}(\mathbb{R})^3 \times\left(\dot H^1_{axi}(\mathbb{R})^3\cap L^2_{\partial h}(\mathbb{R}^3)\right)$ i.e  $(\phi_1,\phi_2)\in \dot H^1_{axi}(\mathbb{R})^3 \times\left(\dot H^1_{axi}(\mathbb{R})^3\cap L^2_{\partial h}(\mathbb{R}^3)\right)$ we have
\begin{equation}
\label{weak:solution}
\begin{aligned}
            &\int_{\mathbb{R}^3}\;\nabla\Xzero\cdot\nabla\phi_1-\frac{2\partial Y_K\cdot\partial\overset{\circ}{Y}}{X_K}\phi_1+\frac{2|\partial Y_K|^2}{X_K^2}\overset{\circ}{X}\phi_1-2\frac{\partial X_K\cdot\partial Y_K}{X_K^2}\overset{\circ}{Y}\phi_1+H_X\phi_1 \\
        &+\nabla\Yzero\cdot\nabla\phi_2+\frac{2\partial Y_K\cdot\partial\overset{\circ}{X}}{X_K}\phi_2+\frac{(|\partial X_K|^2+|\partial Y_K|^2)}{X_K^2}\overset{\circ}{Y}\phi_2+H_Y\phi_2 =0 .
\end{aligned}
\end{equation}
Finally,  the solution is uniquely determined in the class $\displaystyle \dot H^1_{axi}(\mathbb{R})^3 \times\left(\dot H^1_{axi}(\mathbb{R})^3\cap L^2_{\partial h}(\mathbb{R}^3)\right)$. 
\end{lemma}
\noindent Before we prove the above lemma, we state the following result on the properties of $\mathcal L$
\begin{lemma}
\label{coerc::2}
\begin{enumerate}
\item There exist $C, \tilde C> 0$ such that $\forall \left(\Xzero, \Yzero\right)\in \mathcal U := \dot H^1_{axi}(\mathbb{R})^3 \times\left(\dot H^1_{axi}(\mathbb{R})^3\cap L^2_{\partial h}(\mathbb{R}^3)\right)$
\begin{equation*}
\mathcal L \left(\Xzero, \Yzero\right) \geq C\left(\left|\left|\partial\Xzero\right|\right|^2_{L^2(\mathbb R^3)} + \left|\left|\partial\Yzero\right|\right|^2_{L^2(\mathbb R^3)} + \left|\left|\Yzero\right|\right|^2_{L^2_{\partial h}(\mathbb R^3)} \right) - \tilde C
\end{equation*} 
\item $\displaystyle \mathcal L$ is weakly lower semi-continuous on $\displaystyle  \mathcal U$.
\end{enumerate}
\end{lemma}
\begin{proof}
\begin{enumerate}
\item By Lemma \ref{coerc::1}, there exists $C > 0$ such that $\forall (\Xzero, \Yzero)\in C_0^{\infty}(\mathbb R^3)$ and $\displaystyle \int_{\mathbb R^3}\, \frac{|\partial X_K|^2}{X_K^2}{\Yzero}^2<\infty$, we have 
\begin{equation*}
    \int_{\mathbb{R}^3}\;\left(|\partial \Xzero|^2+|\partial \Yzero|^2+|\partial h|^2\Yzero^2 \right) \le C \left(\mathcal{L}(\Xzero,\Yzero) + 2\int_{\mathbb{R}^3}\;\left( |\Xzero H_X|+|\Yzero H_Y|\right) \right).
    \end{equation*}
    Thus, $\forall q>0$, 
\begin{equation*}
\begin{aligned}
    \int_{\mathbb{R}^3}\;\left(|\partial \Xzero|^2+|\partial \Yzero|^2+|\partial h|^2\Yzero^2 \right) &\le C\mathcal{L}(\Xzero,\Yzero) + 2Cq||\Xzero||^2_{L^2(\mathbb{R}^3)}+\frac{C}{2q}||H_X||^2_{L^2(\mathbb{R}^3)}  \\
    &+2Cq||\Yzero||^2_{L^2(\mathbb{R}^3)}+ \frac{C}{2q}||H_X||^2_{L^2(\mathbb{R}^3)}. 
    \end{aligned}
    \end{equation*}    
   By Poincaré inequality, we obtain 
\begin{equation*}
\begin{aligned}
    \int_{\mathbb{R}^3}\;\left(|\partial \Xzero|^2+|\partial \Yzero|^2+|\partial h|^2\Yzero^2 \right) &\le C\mathcal{L}(\Xzero,\Yzero) +2Cq||\partial\Xzero||^2_{L^2(\mathbb{R}^3)}+\frac{C}{2q}||H_X||^2_{L^2(\mathbb{R}^3)}  \\
    &+2Cq||\partial\Yzero||^2_{L^2(\mathbb{R}^3)}+ \frac{C}{2q}||H_Y||^2_{L^2(\mathbb{R}^3)}.  
    \end{aligned}
    \end{equation*}    
   Finally, we choose $q>0$ so that 
   \begin{equation*}
   \int_{\mathbb{R}^3}\;\left(|\partial \Xzero|^2+|\partial \Yzero|^2+|\partial h|^2\Yzero^2 \right) \le C\left( \mathcal{L}(\Xzero,\Yzero) + ||H_X||^2_{L^2(\mathbb{R}^3)} +  ||H_Y||^2_{L^2(\mathbb{R}^3)} \right)
   \end{equation*}
   \item Let $(f_k, g_k)_{k\in\mathbb N}$ be a sequence of  $\mathcal U$ which converges weakly in $\mathcal U$ to $(f, g)$. We show that 
   \begin{equation*}
   \mathcal L(f, g) \leq \underline\lim_{k\to\infty} \mathcal L(f_k, g_k). 
   \end{equation*}
   We have, 
   \begin{equation*}
   \begin{aligned}
   \mathcal L_0(f_k, g_k) &= \int_{\mathbb R^3}\, |\partial f_k|^2 + |\partial g_k|^2 + \frac{|\partial Y_K|^2}{X_K^2}\left( g_k^2 + 2f_k^2\right) + \frac{|\partial X_K|^2}{X_K^2}g_k^2 + 2g_k \partial f_k\cdot\frac{\partial Y_K}{X_K} - 2f_k \partial g_k\cdot\frac{\partial Y_K}{X_K}  \\
   &- 2f_kg_k\frac{\partial X_K\cdot\partial Y_K}{X_K^2},
   \end{aligned}
   \end{equation*}
   where $\mathcal L_0$ is defined by \eqref{L::0}. Since $\displaystyle (f_k, g_k) \rightharpoonup (f, g)$ in $\mathcal U$, there exists a subsequence of $(f_k, g_k)_{k\in\mathbb N}$ which converges strongly in $L^2(\overline B_R)$ for all $R>0$.  Now, by convexity of $x\to x^2$, we have 
   \begin{equation*}
   \int_{\overline B_R}\,|\partial f_k|^2 + |\partial g_k|^2 \geq 2\int_{\overline B_R}\partial f\cdot (\partial f_k - \partial f) +  2\int_{\overline B_R}\partial g\cdot (\partial g_k - \partial g) +  \int_{\overline B_R}\,|\partial f^2 + |\partial g|^2.
   \end{equation*}
   Since $\displaystyle (\partial f_k, \partial g_k) \rightharpoonup (\partial f, \partial g)$ in $L^2(\overline B_R)$, we obtain 
   \begin{equation*}
   \lim_{k\to \infty}\, 2\int_{\overline B_R}\partial f\cdot (\partial f_k - \partial f) +  2\int_{\overline B_R}\partial g\cdot (\partial g_k - \partial g) = 0. 
   \end{equation*}
   Therefore, 
   \begin{equation*}
    \underline\lim_{k\to\infty} \int_{\overline B_R}\,|\partial f_k|^2 + |\partial g_k|^2 \geq  \int_{\overline B_R}\,|\partial f^2 + |\partial g|^2.
   \end{equation*}
   Moreover,
   \begin{equation*}
   \int_{\overline B_R}\,  \frac{|\partial Y_K|^2}{X_K^2}g_k^2 \geq    \int_{\overline B_R}\,  \frac{|\partial Y_K|^2}{X_K^2}g_k(g_k - g) +    \int_{\overline B_R}\,  \frac{|\partial Y_K|^2}{X_K^2}g^2.  
   \end{equation*} 
   Since $(g_k)_k$ converges strongly in $L^2(\overline B_R)$, we obtain 
   \begin{equation*}
   \lim_{k\to \infty}\, \int_{\overline B_R}\,  \frac{|\partial Y_K|^2}{X_K^2}g(g_k - g)  = 0.
   \end{equation*} 
   We have, 
       \begin{equation*}
   \int_{\overline B_R}\,  \frac{|\partial X_K|^2}{X_K^2}g_k^2 \geq    \int_{\overline B_R}\,  \frac{|\partial X_K|^2}{X_K^2}g(g_k - g) +    \int_{\overline B_R}\,  \frac{|\partial X_K|^2}{X_K^2}g^2.  
   \end{equation*}
   Moreover, $\displaystyle \partial h g_k \rightharpoonup  \partial h g$ in $L^2(\overline B_R)$ and $g\partial h\in L^2(\overline B_R)$. Hence, 
       \begin{equation*}
     \underline\lim_{k\to\infty}   \int_{\overline B_R}\,  \frac{|\partial X_K|^2}{X_K^2}g_k^2 \geq \int_{\overline B_R}\,  \frac{|\partial X_K|^2}{X_K^2}g^2.  
   \end{equation*}
   The remaining terms are tackled in the same manner. Therefore, 
   \begin{equation}
   \label{lim::inf}
   \begin{aligned}
    l := \underline\lim_{k\to\infty} \mathcal L_0 (f_k, g_k) &\geq  \int_{\overline B_R}\,|\partial f|^2 + |\partial g|^2 + \frac{|\partial Y_K|^2}{X_K^2}\left( g^2 + 2f^2\right) + \frac{|\partial X_K|^2}{X_K^2}g^2 + 2g \partial f\cdot\frac{\partial Y_K}{X_K} - 2f \partial g\cdot\frac{\partial Y_K}{X_K}  \\
   &- 2fg\frac{\partial X_K\cdot\partial Y_K}{X_K^2}, \quad\quad\forall R>0.
    \end{aligned}
   \end{equation}
   Now, we have, 
   \begin{equation*}
   \begin{aligned}
    &\int_{\overline B_R}\,|\partial f|^2 + |\partial g|^2 + \frac{|\partial Y_K|^2}{X_K^2}\left( g^2 + 2f^2\right) + \frac{|\partial X_K|^2}{X_K^2}g^2 + 2g \partial f\cdot\frac{\partial Y_K}{X_K} - 2f \partial g\cdot\frac{\partial Y_K}{X_K}  - 2fg\frac{\partial X_K\cdot\partial Y_K}{X_K^2} \\ 
   &=  \int_{\mathbb R^3}\mathds{1}_{\left\{ x\,,\, |x|\leq R\right\} }\left(|\partial f|^2 + |\partial g|^2 + \frac{|\partial Y_K|^2}{X_K^2}\left( g^2 + 2f^2\right) + \frac{|\partial X_K|^2}{X_K^2}g^2 + 2g \partial f\cdot\frac{\partial Y_K}{X_K} - 2f \partial g\cdot\frac{\partial Y_K}{X_K}\right.   \\
   &\left.- 2fg\frac{\partial X_K\cdot\partial Y_K}{X_K^2} \right). 
   \end{aligned}
   \end{equation*}
   By positivity of the integrand, we can apply the monotone convergence theorem to obtain
   \begin{equation*}
   \begin{aligned}
    &\lim_{R\to\infty}\int_{\mathbb R^3}\mathds{1}_{\left\{ x\,,\, |x|\leq R\right\} }\left(|\partial f|^2 + |\partial g|^2 + \frac{|\partial Y_K|^2}{X_K^2}\left( g^2 + 2f^2\right) + \frac{|\partial X_K|^2}{X_K^2}g^2 + 2g \partial f\cdot\frac{\partial Y_K}{X_K} - 2f \partial g\cdot\frac{\partial Y_K}{X_K}   \right. 
    &\left. - 2fg\frac{\partial X_K\cdot\partial Y_K}{X_K^2} \right)\\
    &=\int_{\mathbb R^3}\left(|\partial f|^2 + |\partial g|^2 + \frac{|\partial Y_K|^2}{X_K^2}\left( g^2 + 2f^2\right) + \frac{|\partial X_K|^2}{X_K^2}g^2 + 2g \partial f\cdot\frac{\partial Y_K}{X_K} - 2f \partial g\cdot\frac{\partial Y_K}{X_K}  - 2fg\frac{\partial X_K\cdot\partial Y_K}{X_K^2} \right). 
   \end{aligned}
   \end{equation*}
   Finally,  we take the limit of \eqref{lim::inf} when $R\to\infty$, we obtain the desired result. 
   \end{enumerate}
\end{proof}
\noindent Now, we prove Lemma \ref{X:Y:weak}
\begin{proof}
We begin by setting 
\begin{equation}
\label{infimum:m}
m := \inf_{(f, g)\in \mathcal U}\;\mathcal L(f, g). 
\end{equation}
\begin{enumerate}
\item First, we claim that $m$ exists and it is finite. Indeed, the set 
\begin{equation*}
\left\{\mathcal L (f, g)\;,\; (f, g)\in \mathcal U \right\} 
\end{equation*} 
is not empty. Moreover,  by Lemma \ref{coerc::2}, $\mathcal L $ is bounded from below. Thus, $m$ exists. The latter is finite by Lemma \ref{Lagrangian:Propo}. 
\item By the first point, there exists a minimising sequence $\displaystyle (f_k, g_k)_{k\in\mathbb N}\in \mathcal U ^{\mathbb N}$ such that 
\begin{equation*}
\lim_{k\to\infty}\mathcal L (f_k, g_k) = m. 
\end{equation*}
Therefore, $(\mathcal L (f_k, g_k))_{k\in\mathbb N}$ is bounded. Moreover, there exist $C, \tilde C> 0$ such that $\forall k\in \mathbb N$, 
\begin{equation*}
\left|\left|(f_k, g_k)\right|\right|_{\mathcal U} \leq C \mathcal L (f_k, g_k) + \tilde C. 
\end{equation*} 
Hence,  $(f_k, g_k)_{k\in\mathbb N}$ is bounded in $\mathcal U$ and we can extract a subsequence that converges weakly in $\mathcal U$. 
\item By Lemma \ref{coerc::2}, $\mathcal L$ is weakly lower semi-continuous. Hence, after taking the limit,  there exists a minimizer for the Lagrangian $\mathcal L$ and we obtain \eqref{weak:solution} 
\item Finally, we show that the minimizer is unique. By the linearity of the problem \eqref{linear:X:Y}, it suffices to show that if $(\Xzero, \Yzero)$ solves the latter with $H_X = 0$ and $H_Y = 0$, then $\Xzero = 0$ and $\Yzero = 0$.   
\\  If $(\Xzero, \Yzero)$ solves weakly  \eqref{linear:X:Y}, then \eqref{weak:solution} holds with $(\phi_1, \phi_2) = (\Xzero, \Yzero)$. Therefore, $\mathcal L(\Xzero, \Yzero) = 0$. 
\\ By Lemma \ref{coerc::1}, we obtain
\begin{equation*}
||(\Xzero, \Yzero)||_{\mathcal U} = 0. 
\end{equation*}
\end{enumerate}
\end{proof}
\paragraph{Linear estimates}
We have obtained so far: 
\begin{itemize}
\item $\left(\Xzero, \Yzero\right)$ are smooth classical solutions in the region $\displaystyle \mathbb R^3\backslash\left\{(0, z)\; z\in\mathbb R \right\}. $
\item $\left(\Xzero, \Yzero\right)$ are compactly supported. 
\end{itemize}
Moreover, the equations for $(\Xzero, \Yzero)$ have a singular  behaviour on the the boundary because of  the asymptotics for $(X_K, Y_K)$ near the horizon and the axis:
\begin{itemize}
\item the equation for $\Xzero$ is singular only  near $p_N$ and $p_S$, provided $\Yzero$ in the right space. 
\item the equation for $\Yzero$ is singular near $p_N$, $p_S$ and $\BAbarre\cup\BHbarre$. 
\end{itemize}
In order to overcome this difficulty, we introduce a change of variable on every region, work in higher dimensions in order to obtain the required estimates for the new quantities on each region, then deduce the estimates for the original unknowns. 
\noindent Let $\displaystyle (H_X, H_Y)\in C_0^\infty(\mathbb R^3)$ and let $\displaystyle (\Xzero, \Yzero)\in \mathcal U$ be the unique weak solution of \eqref{linear:X:Y}. We define $\tilde Y$ on  $\displaystyle \mathbb R^3\backslash\left\{(0, z)\; z\in\mathbb R \right\}$ to be the function
\begin{equation}
\label{tilde:Y}
\tilde Y := X^{-1}_K{\Yzero}. 
\end{equation}
 \begin{lemma}
 $\tilde Y$ is a weak solution of  the equation 
 \begin{equation}
 \label{eq:tilde:Y}
    \Delta_{\mathbb{R}^3}\tilde{Y}+2\frac{\partial X_K\cdot\partial\tilde{Y}}{X_K}-2\frac{\partial Y_K\cdot\partial\Xzero}{X^2_K}-2\frac{|\partial Y_K|^2}{X^2_K}\tilde{Y} = H_YX^{-1}_K.
\end{equation}
 \end{lemma}
\begin{proof}
We compute
\begin{align*}
    \partial\Yzero &= \tilde{Y}\partial X_K + X_K\partial\tilde{Y}, \\
    \Delta_{\mathbb{R}^3}\Yzero &= \tilde{Y}\Delta_{\mathbb{R}^3} X_K+X_K\Delta_{\mathbb{R}^3}\tilde{Y}+2\partial X_K\cdot\partial\tilde{Y}.
\end{align*}
We plug the latter in the equation satisfied by $\Yzero$ in order to obtain
\begin{equation*}
    \tilde{Y}\Delta_{\mathbb{R}^3}X_K+X_K\Delta_{\mathbb{R}^3}\tilde{Y}+2\partial X_K\cdot\partial\tilde{Y}-\frac{2\partial Y_K\cdot\partial\overset{\circ}{X}}{X_K} - \frac{|\partial X_K|^2}{X_K}\tilde{Y}-\frac{|\partial X_K|^2}{X_K}\tilde{Y} = H_Y.
\end{equation*}
Now, recall  that the pair $(X_K,Y_K)$ satisfies \eqref{XK::YK}. Therefore,  on $\mathbb R^3\backslash\left\{(0, z),\; z\in\mathbb R^3\right\}$,  we have
\begin{equation*}
    \frac{|\partial X_K|^2}{X_K}\tilde{Y}- \frac{|\partial Y_K|^2}{X_K}\tilde{Y}+2\partial X_K\cdot\partial\tilde{Y}-2\frac{\partial Y_K\cdot\partial\Xzero}{X_K}-\frac{|\partial X_K|^2}{X_K}\tilde{Y}+X_K\Delta_{\mathbb{R}^3}\tilde{Y}-\frac{|\partial Y_K|^2}{X_K}\tilde{Y} = H_Y.
\end{equation*}
Hence, $\tilde Y$ satisfies \eqref{eq:tilde:Y} .
\end{proof}
\begin{lemma}
\begin{enumerate}
\item there exists a smooth vector field $\tilde e_A$ defined on $\BAbarre$ such that $\forall (\rho, z)\in \BB: $
\begin{equation*}
\frac{\partial X_K}{X_K} = \frac{2}{\rho}\partial_\rho + \tilde e_A,
\end{equation*}
\item there exist a smooth vector field $\tilde d_A$ and a smooth function $\overline d_A$ defined on $\BAbarre$  such that $\forall (\rho, z)\in \BB: $
\begin{equation*}
\frac{\partial Y_K}{X^2_K} = \tilde d_A + \frac{\overline d_A}{\rho}\partial_\rho, 
\end{equation*}
\item the function $\displaystyle \frac{\left|\partial Y_K\right|}{X_K}$ extends smoothly to $\BAbarre$. 
\end{enumerate}
\end{lemma}
\begin{proof}
This follows from Lemmas \ref{decay:estimates}, \ref{asymptotics:h1} and \ref{asymptotics:h2}. 
\end{proof}
\noindent Therefore, $\tilde Y$ satisfies the equation 
\begin{equation}
 \label{tilde:Y:A}
    \Delta_{\mathbb{R}^3}\tilde{Y}+ \frac{4}{\rho}\partial_\rho\tilde Y + \tilde e_A\cdot\partial \tilde Y - 2\frac{|\partial Y_K|^2}{X^2_K}\tilde{Y} - 2 \tilde d_A\cdot\partial\Xzero - \frac{2\overline d_A}{\rho}\partial_\rho \Xzero   = H_YX^{-1}_K
\end{equation}
 Moreover, by Lemma \ref{sph::}, we have 
\begin{equation*}
\forall z\in\mathbb R\quad \left.\frac{\partial_\rho\Xzero_{\mathbb R^3}(\cdot, \cdot, z)}{\rho}\right|_{(0, 0)} = \left.\partial^2_{x} \Xzero_{\mathbb R^3}(\cdot, \cdot, z)\right|_{\left\{(x, y) = (0, 0)\right\}}. 
\end{equation*}
For any $x\in\mathbb R^7$,  let  $(\rho,\vartheta, \tilde \vartheta, z)\in[0, \infty[\times(0, \pi)^4\times (0, 2\pi)\times\mathbb R$ be the coordinates defined in the following way  
\begin{equation*}
\begin{aligned}
x_1 &= \rho\cos\vartheta_1 \\
x_2 &= \rho\sin\vartheta_1\cos\vartheta_2 \\
x_3 &= \rho\sin\vartheta_1\sin\vartheta_2\cos\vartheta_3 \\
x_4 &= \rho\sin\vartheta_1\sin\vartheta_2\sin\vartheta_3\cos\vartheta_4 \\
x_5 &= \rho\sin\vartheta_1\sin\vartheta_2\sin\vartheta_3\sin\vartheta_4 \cos\tilde\vartheta \\
x_6 &= \rho\sin\vartheta_1\sin\vartheta_2\sin\vartheta_3\sin\vartheta_4 \sin\tilde\vartheta \\
x_7 &= z. 
\end{aligned}
\end{equation*}
Hence, $\rho$ is given by  $\displaystyle \rho = \sqrt{\sum_{i=1}^6\, x_i^2}$.  Moreover,  to any function $f:\Bbarre\mapsto \mathbb R$ we associate an axisymmetric function $f_{\mathbb R^7}:{\mathbb R^7}\mapsto \mathbb R$ by setting 
\begin{equation}
\label{extended:f}
f_{\mathbb R^7}(x) :=  f(\rho, z).
\end{equation}
In particular, we associate to $\tilde Y$ a function $\tilde Y_{\mathbb R^7}$ defined on $\mathbb R^7$. 
\\ Now, note that 
\begin{equation*}
\Delta_{\mathbb R^7} = \frac{4}{\rho}\partial_\rho + \Delta_{\mathbb R^3}.  
\end{equation*}
Therefore, $\tilde Y_{\mathbb R^7}$ satisfies the equation
\begin{equation}
 \label{tilde:Y:A:7}
    \Delta_{\mathbb{R}^7}\tilde{Y}_{\mathbb R^7} + \tilde e_A\cdot\partial \tilde Y - 2\frac{|\partial Y_K|^2}{X^2_K}\tilde{Y}_{\mathbb R^7} - 2 \tilde d_A\cdot\partial\Xzero - \frac{2\overline d_A}{\rho}\partial_\rho \Xzero   = H_YX^{-1}_K
\end{equation}
on $\displaystyle \mathbb R^7$.
\noindent Note that this identification with $\mathbb R^7$ only makes sense within the set $\BAbarre\cup\BHbarre$. In this case, we denote by $\mathbb R_A^7$ the set $(\BAbarre\cup\BHbarre)\times\mathbb S^5\subset \mathbb R^7$ and  we introduce the following norms 
\begin{equation*}
    ||f||_{W^{k,p}(\mathbb{R}^7_A)} := \sum _{|i| \leq k}\left(\int\;\int_{\BAbarre\cup\BHbarre}\;\left|\partial^{i}f\right|^p\,\rho^5d\rho dz\right)^{\frac{1}{p}}, 
\end{equation*}
\begin{equation*}
    ||u||_{C^{k,\alpha}}(\mathbb{R}^7_A):= \sum_{|\alpha|\le k}||D^\alpha u||_{C^0(\BAbarre\cup\BHbarre)}+\sum_{|\alpha|= k} [D^\alpha u]_{0,\alpha; \BAbarre\cup\BHbarre}
    \end{equation*}
and the semi norm 
\begin{equation*}
    ||f||_{\dot W^{k,p}(\mathbb{R}^7_A)} := \left(\int\;\int_{\BAbarre\cup\BHbarre}\;\left|\partial^{k}f\right|^p\,\rho^5d\rho dz\right)^{\frac{1}{p}}.
\end{equation*}
\noindent of a function $f: \BAbarre\cup\BHbarre\to\mathbb R$. 
\\Now, we recall the coefficients behavior of the elliptic operator associated to \eqref{linear:X:Y} near $p_N$ and $p_S$
\begin{lemma}
\label{asymp:poles}
\begin{enumerate}
\item there exists a smooth vector field $\tilde e_N$ defined on $\Bnbarre$ such that $\forall (s, \chi)\in \BB_N: $
\begin{equation*}
\frac{\underline \partial X_K}{X_K} = \frac{2}{s}\underline \partial_s + \tilde e_N,
\end{equation*}
\item there exists a smooth vector field $\tilde e_S$ defined on $\Bsbarre$ such that $\forall (s', \chi')\in \BB_S: $
\begin{equation*}
\frac{\underline \partial X_K}{X_K} = \frac{2}{s'}\underline \partial_{s'} + \tilde e_S,
\end{equation*}
\item the function $\displaystyle \frac{\left|\underline \partial Y_K\right|}{X_K}$ extends smoothly to $\Bnbarre$ and to $\Bsbarre$. 
\end{enumerate}
\end{lemma}
\begin{proof}
This follows from Lemmas \ref{decay:estimates}, \ref{asymptotics:h1} and \ref{asymptotics:h2}. 
\end{proof}
Now, we consider the system \eqref{linear:X:Y} near $p_N$ and $p_S$ and we write the latter in terms of the coordinates $(s, \chi)$
\begin{lemma}
$(\Xzero, \tilde Y)$ is a weak solution of the system 
\begin{equation}
\label{X:Y:pN}
\begin{aligned}
&\underline \Delta_{\mathbb R^4}\Xzero + 2\frac{\underline\partial Y_K\cdot\partial\Yzero}{X_K} - 2\frac{|\underline\partial Y_K|^2}{X_K^2}\Xzero + 2\frac{\underline\partial X_K\cdot \underline\partial Y_K}{X_K^2}\Yzero = (s^2 + \chi^2)H_X \\
&\underline \Delta_{\mathbb R^8}\tilde Y + \tilde e_N\cdot\underline\partial\tilde Y - 2\frac{|\underline \partial Y_K|^2}{X_K^2}\tilde Y - 2\frac{\underline\partial Y_K\cdot\underline\partial\Xzero}{X_K^2} = (s^2 + \chi^2)X_K^{-1}H_Y,  
\end{aligned} 
\end{equation}
on $\Bnbarre$, where $\underline\partial$ is the gradient with respect to $(s, \chi)$ coordinates and $\tilde e_N$ is given by Lemma \ref{asymp:poles}.  The above equations hold on $\Bsbarre$ provided that we replace $\tilde e_N$ by $\tilde e_S$ and we use $(s', \chi')$ coordinates. 
\end{lemma}
\begin{proof}
The computations are straightforward. 
\end{proof}
\noindent In order to prove Proposition \ref{linear:est:XY}, we will prove the following estimates:
\begin{enumerate}
\item First of all, we establish estimates for $\Xzero$ away from $\Bnbarre$ and $\Bsbarre$ using the first equation of \eqref{linear:X:Y}.  
\item Next, we establish estimates for  $\Xzero$ on $\Bnbarre$ and $\Bsbarre$ using the first equation of \eqref{X:Y:pN}. 
\item We establish estimates for $\Yzero$ away from the region $\left\{(\rho, z)\in\Bbarre\;,\; \rho = 0\;;\; z\in\mathbb R\right\}$. 
\item In order to handle the estimates for $\Yzero$ near the boundary of $\Bbarre$, we establish the estimates for $\tilde Y$ near $\Axis$ using equation \eqref{tilde:Y:A:7}. Then, we use the second equation of  \eqref{X:Y:pN} so that we obtain estimates of $\tilde Y$ on $\Bnbarre$ and $\Bsbarre$. 
\item Finally, we deduce the estimates for $\Yzero$. 
\end{enumerate}

\begin{lemma}
\label{lemma::72}
There exists $C>0$ such that 
\begin{equation*}
\left| \left|\Xzero\right|\right|_{\dot H^1_{axi}(\mathbb R^3)} + \left| \left|\Yzero\right|\right|_{\dot H^1_{axi}(\mathbb R^3)\cap L^2_{\nabla h}(\mathbb R^3)} \leq C \left( \left|\left| H_X\Xzero\right|\right|^{\frac{1}{2}}_{L^1(\mathbb R^3)} + \left|\left| H_Y\Yzero\right|\right|^{\frac{1}{2}}_{L^1(\mathbb R^3)} \right). 
\end{equation*}
\end{lemma}
\begin{proof}
Since $(\Xzero, \Yzero)$ is a weak solution, it minimises the Lagrangian $\mathcal L$. Therefore, 
\begin{equation*}
\mathcal L(\Xzero, \Yzero) \leq \mathcal L (X, Y) \quad\forall (X, Y)\in \mathcal U. 
\end{equation*} 
In particular, 
\begin{equation*}
\mathcal L(\Xzero, \Yzero) \leq \mathcal L(0, 0) = 0. 
\end{equation*} 
Now, by Lemma \ref{coerc::1}, there exists $C>0$
\begin{equation*}
\left| \left|\Xzero\right|\right|^2_{\dot H^1_{axi}(\mathbb R^3)} + \left| \left|\Yzero\right|\right|^2_{\dot H^1_{axi}(\mathbb R^3)\cap L^2_{\nabla h}(\mathbb R^3)} \leq C \left( \mathcal L(\Xzero, \Yzero) + 2\left( \left|\left| H_X\Xzero\right|\right|_{L^1(\mathbb R^3)} + \left|\left| H_Y\Yzero\right|\right|_{L^1(\mathbb R^3)} \right)\right). 
\end{equation*}
Therefore, by convexity of $x\to x^2$, we obtain the desired estimate. 
\end{proof}
\noindent Now, we write
\begin{equation*}
\begin{aligned}
\Xzero &= (1 - \xi_N - \xi_S)\Xzero + \xi_N\Xzero + \xi_S\Xzero, \\
\Yzero &=  (1 - \xi_N - \xi_S)(1 - \xi_A)\Yzero + (1 - \xi_N - \xi_S)\xi_A\Yzero + \xi_N\Yzero + \xi_S\Yzero. 
\end{aligned}
\end{equation*}
We have 
\begin{lemma}
\label{lemma::73}
There exists $C>0$ such that 
\begin{equation*}
\left|\left|(1 - \xi_N - \xi_S)\Xzero\right|\right|_{\dot H^2(\mathbb R^3)} \leq C \left( \left|\left|(1 - \xi_N - \xi_S)H_X\right|\right|_{L^2(\mathbb R^3)} + \left|\left| H_X\Xzero\right|\right|^{\frac{1}{2}}_{L^1(\mathbb R^3)} + \left|\left| H_Y\Yzero\right|\right|^{\frac{1}{2}}_{L^1(\mathbb R^3)} \right). 
\end{equation*}
\end{lemma}
\begin{proof}
\begin{enumerate}
\item First of all, we set $\Xzero_A :=  (1 - \xi_N - \xi_S)\Xzero$ and we claim that $\Xzero_A$ satisfies 
\begin{equation}
\label{Xzero::A}
\begin{aligned}
\Delta_{\mathbb R^3}\Xzero_A &= 2\partial\Xzero\cdot\partial(\xi_N + \xi_S) + \Delta_{\mathbb R^3}(\xi_N + \xi_S)\Xzero - 2(1 - \xi_N - \xi_S)\frac{\partial Y_K\cdot\partial\Yzero}{X_K} + 2(1 - \xi_N - \xi_S)\Yzero\frac{\partial X_K\cdot\partial Y_K}{X_K^2}  \\ 
& + 2\frac{|\partial Y_K|^2}{X_K^2}(1 - \xi_N - \xi_S)\Xzero + (1 - \xi_N - \xi_S)H_X.   
\end{aligned}
\end{equation}
\item We will apply Theorem \ref{Cal:Zyg} with $p = 2$, $n = 3$ and $f$ given by 
\begin{equation}
\label{f:Cal} 
\begin{aligned}
f (x)&:= 2\partial\Xzero\cdot\partial(\xi_N + \xi_S) + \Delta_{\mathbb R^3}(\xi_N + \xi_S)\Xzero - 2(1 - \xi_N - \xi_S)\frac{\partial Y_K\cdot\partial\Yzero}{X_K} + 2(1 - \xi_N - \xi_S)\Yzero\frac{\partial X_K\cdot\partial Y_K}{X_K^2} \\
& + 2\frac{|Y_K|^2}{X_K^2}(1 - \xi_N - \xi_S)\Xzero + (1 - \xi_N - \xi_S)H_X
\end{aligned}
\end{equation}
We recall here the identification of $f$ with $f_{\mathbb R^3}$. To this end, we show that $f\in L^2(\mathbb R^3)$.
\begin{itemize}
\item First of all, we have $\partial(\xi_N + \xi_S), \Delta_{\mathbb R^3}(\xi_N + \xi_S)$ are compactly supported in $\Bnbarre\cup\Bsbarre$ and they vanish near $p_N$ and $p_S$. Denote by $\overline Q$ the support of their identifications on $\mathbb R^3$. 
\item $\Xzero\in \dot H^1_{axi}(\mathbb R^3)$. Therefore, $\displaystyle -2\partial\Xzero\cdot\partial(\xi_N + \xi_S) \in L^2(\mathbb R^3)$. 
\item By Poincaré inequality, we have 
\begin{equation*}
\begin{aligned}
||\Delta_{\mathbb R^3}(\xi_N + \xi_S)\Xzero||_{L^2(\mathbb R^3)} &=  ||\Delta_{\mathbb R^3}(\xi_N + \xi_S)\Xzero||_{L^2(\overline Q)}\\ 
&\leq C||\Xzero||_{L^2(\overline Q)}\\ 
&\leq C ||\nabla\Xzero||_{L^2(\overline Q)}\\
&\leq C ||\Xzero||_{\dot H^1_{axi}(\mathbb R^3)}. 
\end{aligned}
\end{equation*}
\item By Lemma \ref{decay:estimates} and Lemma \ref{lemma::13}, we have 
\begin{equation*}
\left|\left|(1 - \xi_N - \xi_S)\Yzero\frac{\partial X_K\cdot\partial Y_K}{X_K^2} \right|\right|_{L^2(\mathbb R^3)} \leq C \left|\left|\Yzero\right|\right|_{L^2_{\nabla h}(\mathbb R^3)}. 
\end{equation*}
\item Moreover, by Lemma \ref{asymptotics:h1} and Lemma \ref{decay:estimates}, we have 
\begin{equation*}
\left|\left|(1 - \xi_N - \xi_S)\frac{\partial Y_K\cdot\partial\Yzero}{X_K} \right|\right|_{L^2(\mathbb R^3)} \leq C \left|\left|\Yzero\right|\right|_{ \dot H^1_{axi}(\mathbb R^3)}.
\end{equation*}
\item  Since $H_X\in C^\infty_0(\mathbb R^3)$, the term $\displaystyle (1 - \xi_N - \xi_S)H_X\in L^2(\mathbb R^3)$.
\end{itemize} 
Thus, $f\in L^2(\mathbb R^3)$ and we apply Theorem  \ref{Cal:Zyg} to obtain 
\begin{equation*}
\left|\left|(1 - \xi_N - \xi_S)\Xzero\right|\right|_{\dot H^2(\mathbb R^3)} \leq C \left|\left|f\right|\right|_{L^2(\mathbb R^3)}. 
\end{equation*}
\item Now, we use the latter estimates to obtain  
\begin{equation*}
\left|\left|f\right|\right|_{L^2(\mathbb R^3)} \leq C  \left( \left|\left|(1 - \xi_N - \xi_S)H_X\right|\right|_{L^2(\mathbb R^3)} +  ||\Xzero||_{\dot H^1_{axi}(\mathbb R^3)}   + \left|\left|\Yzero\right|\right|_{\dot H^1_{axi}(\mathbb R^3)} + \left|\left|\Yzero\right|\right|_{L^2_{\nabla h}(\mathbb R^3)}\right). 
\end{equation*}
\item  Finally, we use Lemma \ref{lemma::72} to conclude. 
\end{enumerate}
\end{proof}

\begin{lemma}
\label{lemma::74}
There exists $C>0$ such that 
\begin{equation*}
\left|\left|\xi_N\Xzero\right|\right|_{\dot H^2(\mathbb R^4_N)} + \left|\left|\xi_S\Xzero\right|\right|_{\dot H^2(\mathbb R^4_S)} \leq C \left( \left|\left|\xi_N(s^2 + \chi^2)H_X\right|\right|_{L^2(\mathbb R^4_N)} + \left|\left| H_X\Xzero\right|\right|^{\frac{1}{2}}_{L^1(\mathbb R^3)} + \left|\left| H_Y\Yzero\right|\right|^{\frac{1}{2}}_{L^1(\mathbb R^3)} \right)
\end{equation*}
\end{lemma}
\begin{proof}
We proceed as in Lemma \ref{lemma::73}. We only mention the main steps. 
\begin{enumerate}
\item First of all, we multiply the first equation of \eqref{X:Y:pN} by $\xi_N$ and we express $\underline \Delta_{\mathbb R^4}\left(\xi_N\Xzero\right)$ in terms of the remaining quantities:
\begin{equation*}
\begin{aligned}
\underline\Delta_{\mathbb R^4}\left(\xi_N\Xzero\right)& = \underline \Delta_{\mathbb R^4}(\xi_N)\Xzero + 2\underline\partial(\xi_N + \xi_S)\cdot\underline\partial\Xzero - 2\xi_N\frac{\underline\partial Y_K\cdot\partial\Yzero}{X_K} + 2\xi_N\frac{|\underline\partial Y_K|^2}{X_K^2}\Xzero  \\
&- 2\xi_N\frac{\underline\partial X_K\cdot \underline\partial Y_K}{X_K^2}\Yzero + \xi_N(s^2 + \chi^2)H_X
\end{aligned}
\end{equation*} 
\item We apply Theorem \ref{Cal:Zyg} with $n = 4$, $p = 2$ and $f$ given by the right hand side of the above equation. To this end, we show that $\displaystyle f\in L^2(\mathbb R^4_N)$. We give details for the term $\displaystyle \xi_N\frac{|\underline\partial Y_K|^2}{X_K^2}\Xzero $. The other terms follow in the same manner. 
\\ We have 
\begin{equation*}
\left|\left| \xi_N\frac{|\underline\partial Y_K|^2}{X_K^2}\Xzero\right| \right|^2_{L^2(\mathbb R^4_N)} = (2\pi)^2\int\int_{\Bnbarre}\, \xi_N^2\left(\frac{|\underline\partial Y_K|^2}{X_K^2}\Xzero\right)^2s\chi d\xi d\chi. 
\end{equation*}
By Lemma \ref{extendibility}, Lemma \ref{decay:estimates}, we have
\begin{equation*}
\forall (s, \chi)\in \Bnbarre\;:\; \frac{|\underline\partial Y_K|}{X_K} \leq C sX_N^{-1}(s^2, \chi^2). 
\end{equation*}
Therefore, 
\begin{equation*}
\begin{aligned}
\left|\left| \xi_N\frac{|\underline\partial Y_K|^2}{X_K^2}\Xzero\right| \right|^2_{L^2(\mathbb R^4_N)} &\leq  2\pi C\int\int_{\Bnbarre}\, \xi_N^2\Xzero^2s\chi d\xi d\chi \\
&\leq 2\pi C\int_{\mathbb R}\int_0^\infty\,\xi_N^2\Xzero^2 \rho d\rho dz \\
&\leq C \left|\left| \Xzero\right|\right|_{\dot H^1_{axi}(\mathbb R^3)}. 
\end{aligned}
\end{equation*}
Here, we used Poincaré inequality to obtain the latter estimate. 
\item We use Lemma \ref{lemma::72} to conclude.  
\end{enumerate}
\end{proof}

\begin{lemma}
There exists $C>0$ such that 
\begin{equation*}
\left|\left|(1 - \xi_N - \xi_S)(1 - \xi_A)\Yzero\right|\right|_{\dot H^2(\mathbb R^3)} \leq C \left( \left|\left|(1 - \xi_N - \xi_S)(1 - \xi_A)H_Y\right|\right|_{L^2(\mathbb R^3)} + \left|\left| H_X\Xzero\right|\right|^{\frac{1}{2}}_{L^1(\mathbb R^3)} + \left|\left| H_Y\Yzero\right|\right|^{\frac{1}{2}}_{L^1(\mathbb R^3)} \right). 
\end{equation*}
\end{lemma}
\begin{proof}
The proof is similar to the previous lemmas. Here we apply Calderon-Zygmund theory to the second equation of \eqref{linear:X:Y}. 
\end{proof}

\begin{lemma}
\label{lemma::76}
There exists $C>0$ such that 
\begin{equation*}
\begin{aligned}
\left|\left|(1 - \xi_N - \xi_S)\xi_A\tilde Y\right|\right|_{\dot H^2(\mathbb R^7_A)} \leq C \left( \left|\left|(1 - \xi_N - \xi_S)\xi_AH_YX_K^{-1}\right|\right|_{L^2(\mathbb R^7_A)} + \left|\left| H_X\Xzero\right|\right|^{\frac{1}{2}}_{L^1(\mathbb R^3)} + \left|\left| H_Y\Yzero\right|\right|^{\frac{1}{2}}_{L^1(\mathbb R^3)} \right). 
\end{aligned}
\end{equation*}
\end{lemma}

\begin{proof}
Although the proof is similar to the previous lemmas,  we give details leading to $L^2$ estimates of the different terms. 
\begin{enumerate}
\item We multiply \eqref{tilde:Y:A:7} by $\tilde \xi := (1 - \xi_N - \xi_S)\xi_A$ and we write $\Delta_{\mathbb R^7}\left(\tilde\xi\tilde Y\right)$ in terms of the remaining quantities: 
\begin{equation*}
\begin{aligned}
\Delta_{\mathbb R^7}\left(\tilde\xi\tilde Y\right) &= \tilde\xi H_Y X_K^{-1} - \left(\tilde\xi\tilde e_A - 2\partial\tilde \xi \right)\cdot\tilde Y + \left(2\frac{|\partial Y_K|^2}{X_K^2}\tilde\xi + \Delta_{\mathbb R^7}\tilde\xi\right)\tilde Y. 
\end{aligned}
\end{equation*}
\item In order to apply Calderon-Zygmund theory, we need to show that the function 
\begin{equation*}
g = \tilde\xi H_Y X_K^{-1} - \left(\tilde\xi\tilde e_A - 2\partial\tilde \xi \right)\cdot\partial \tilde Y + \left(2\frac{|\partial Y_K|^2}{X_K^2}\tilde\xi + \Delta_{\mathbb R^7}\tilde\xi\right)\tilde Y + 2\tilde\xi\tilde d_A\partial\Xzero + 2\overline d_A\tilde\xi\frac{\partial_\rho\Xzero}{\rho}
\end{equation*}
lies in $L^2(\mathbb R_A^7)$.
\begin{itemize}
\item Denote by $\tilde\Axis$ the support of $\tilde\xi$ and recall from Lemma \ref{asymptotics:h2} the decay estimate for $\tilde e_A$: 
\begin{equation*}
\left|\tilde e_A\right| = O_{r\to+\infty}\left(\langle r\rangle^{-2}\right). 
\end{equation*}
\item We have 
\begin{equation*}
\begin{aligned}
\left|\left|\tilde\xi H_YX_K^{-1}\right|\right|^2_{L^2(\mathbb R^7_A)} &= (2\pi)^5\int\int_{\tilde\Axis}\, \tilde\xi^2X^{-2}_K H_Y^2 \rho^5d\rho dz \\
&=  (2\pi)^5\int\int_{\tilde\Axis}\, \tilde\xi^2X^{-2}_{\Axis}(\rho^2, z) H_Y \rho d\rho dz
\end{aligned}
\end{equation*}
where $X_{\Axis}$ is given by Definition \ref{ext:A}. Moreover,  $X_{\Axis}$ are bounded on $\tilde\Axis$. Therefore, 
\begin{equation*}
\begin{aligned}
\left|\left|\tilde\xi H_YX_K^{-1}\right|\right|^2_{L^2(\mathbb R^7_A)} &\leq 2\pi C  \int\int_{\tilde\Axis}\, \tilde\xi^2 H^2_Y \rho d\rho dz \\ 
&\leq C \int_{\mathbb R^3} \tilde\xi^2 H^2_Y \rho d\rho dz \\
&\leq C\left|\left| \tilde\xi H_Y\right|\right|^2_{L^2(\mathbb R^3)}. 
\end{aligned}
\end{equation*}
\item We have 
\begin{equation*}
\begin{aligned}
\left|\left| \left(\tilde\xi\tilde e_A - 2\partial\tilde \xi \right)\cdot\partial \tilde Y \right|\right|^2_{L^2(\mathbb R^7_A)} &= (2\pi)^5 \int\int_{\tilde\Axis}\, \left( \left(\tilde\xi\tilde e_A - 2\partial\tilde \xi \right)\cdot\partial \tilde Y \right)^2  \rho^5d\rho dz \\
&\leq (2\pi) C \int\int_{\tilde\Axis}\, \left|\partial\tilde Y\right|^2 \rho^5 d\rho dz.
\end{aligned}
\end{equation*}
We have 
\begin{equation*}
\forall (\rho, z)\in\tilde\Axis\;:\: \tilde Y = \frac{1}{\rho^2}X_\Axis^{-1}(\rho^2, z)\Yzero. 
\end{equation*}
Hence, 
\begin{equation*}
\partial \tilde Y = -\frac{1}{\rho^3}X_\Axis^{-1}(\rho^2, z)\Yzero + \frac{1}{\rho^2}\left(\partial X_\Axis^{-1}(\rho^2, z)\Yzero + X_\Axis^{-1}(\rho^2, z)\partial \Yzero \right). 
\end{equation*}
Therefore, 
\begin{equation*}
\forall (\rho, z)\in \tilde A\;:\; \left|\partial \tilde Y\right|^2 \leq C \left( \frac{\Yzero^2}{\rho^6} +  \frac{\left|\partial \Yzero\right|^2}{\rho^4} \right)
\end{equation*}
and 
\begin{equation*}
\begin{aligned}
\left|\left| \left(\tilde\xi\tilde e_A - 2\partial\tilde \xi \right)\cdot\partial \tilde Y \right|\right|^2_{L^2(\mathbb R^7_A)} &\leq   (2\pi) C \left( \int\int_{\tilde\Axis}\,\frac{\Yzero^2}{\rho^2} \rho d\rho dz + \int\int_{\tilde\Axis}\, \left|\partial\Yzero\right|^2 \rho d\rho dz \right).
\end{aligned}
\end{equation*}
By Lemma \ref{asymptotics:h2}, 
\begin{equation*}
\forall (\rho, z)\in \tilde\Axis\;:\; \frac{1}{\rho^2} \leq |\partial h|^2. 
\end{equation*}
This yields
\begin{equation*}
\begin{aligned}
 \int\int_{\tilde\Axis}\,\frac{\Yzero^2}{\rho^2} \rho d\rho dz &\leq  \int\int_{\tilde\Axis}\, {|\partial h |\Yzero^2} \rho d\rho dz
 &\leq \left|\left|\Yzero\right|\right|^2_{L^2_{\nabla h}(\mathbb R^3)}
 \end{aligned}
\end{equation*}
and the estimate
\begin{equation*}
\begin{aligned}
\left|\left| \left(\tilde\xi\tilde e_A - 2\partial\tilde \xi \right)\cdot\partial \tilde Y \right|\right|_{L^2(\mathbb R^7_A)} &\leq C  \left|\left|\Yzero\right|\right|_{\dot H^1_{axi}(\mathbb R^3)\cap L^2_{\nabla h}(\mathbb R^3)} \\
&\leq C\left( \left|\left| H_X\Xzero\right|\right|^{\frac{1}{2}}_{L^1(\mathbb R^3)} + \left|\left| H_Y\Yzero\right|\right|^{\frac{1}{2}}_{L^1(\mathbb R^3)} \right). 
\end{aligned}
\end{equation*}
\item In order to estimate $\displaystyle  \overline d_A\tilde\xi\frac{\partial_\rho\Xzero}{\rho}$, we write 
\begin{equation*}
\begin{aligned}
\left| \left|\overline d_A\tilde\xi\frac{\partial_\rho\Xzero}{\rho} \right|\right|^2_{L^2(\mathbb R^7_A)} &\leq C  \int\int_{\tilde\Axis}\,\tilde\xi^2(\partial_\rho\Xzero)^2\rho d\rho dz \\
&\leq C  \left|\left|\Xzero\right|\right|^2_{\dot H^1_{axi}(\mathbb R^3)}. 
\end{aligned}
\end{equation*}
\item The other terms follow in the same way. 
\end{itemize} 
\end{enumerate}

\end{proof}

\begin{lemma}
There exists $C>0$ such that 
\begin{equation*}
\begin{aligned}
&\left|\left|\xi_N\tilde Y\right|\right|_{\dot H^2(\mathbb R^8_N)} \leq C \left( \left|\left|(s^2 + \chi^2)H_Y\right|\right|_{L^2(\Bnbarre)} + \left|\left| H_X\Xzero\right|\right|^{\frac{1}{2}}_{L^1(\mathbb R^3)} + \left|\left| H_Y\Yzero\right|\right|^{\frac{1}{2}}_{L^1(\mathbb R^3)} \right), \\
& \left|\left|\xi_S\tilde Y\right|\right|_{\dot H^2(\mathbb R^8_S)} \leq C \left( \left|\left|((s')^2 + (\chi')^2)H_Y\right|\right|_{L^2(\Bsbarre)} + \left|\left| H_X\Xzero\right|\right|^{\frac{1}{2}}_{L^1(\mathbb R^3)} + \left|\left| H_Y\Yzero\right|\right|^{\frac{1}{2}}_{L^1(\mathbb R^3)} \right).
\end{aligned}
\end{equation*}
\end{lemma}

\begin{proof}
The proof is similar to the previous lemmas. 
\end{proof}
\noindent Now, we derive the estimates for $\Yzero$. 
\begin{lemma}
There exists $C>0$ such that 
\begin{equation*}
\begin{aligned}
& \left|\left|(1 - \xi_N - \xi_S)\xi_A\Yzero\right|\right|_{\dot H^2(\mathbb R^3)} \leq C \left( \left|\left|(1 - \xi_N - \xi_S)\xi_AH_Y\right|\right|_{L^2(\mathbb R^3)} + \left|\left| H_X\Xzero\right|\right|^{\frac{1}{2}}_{L^1(\mathbb R^3)} + \left|\left| H_Y\Yzero\right|\right|^{\frac{1}{2}}_{L^1(\mathbb R^3)} \right),  \\
&\left|\left|\xi_N\Yzero\right|\right|_{\dot H^2(\mathbb R^4_N)} \leq C \left( \left|\left|(s^2 + \chi^2)H_Y\right|\right|_{L^2(\Bnbarre)} + \left|\left| H_X\Xzero\right|\right|^{\frac{1}{2}}_{L^1(\mathbb R^3)} + \left|\left| H_Y\Yzero\right|\right|^{\frac{1}{2}}_{L^1(\mathbb R^3)} \right), \\
& \left|\left|\xi_S\Yzero\right|\right|_{\dot H^2(\mathbb R^4_S)} \leq C \left( \left|\left|((s')^2 + (\chi')^2)H_Y\right|\right|_{L^2(\Bsbarre)} + \left|\left| H_X\Xzero\right|\right|^{\frac{1}{2}}_{L^1(\mathbb R^3)} + \left|\left| H_Y\Yzero\right|\right|^{\frac{1}{2}}_{L^1(\mathbb R^3)} \right).
\end{aligned}
\end{equation*}
\end{lemma}
\begin{proof}
\begin{enumerate}
\item First of all, we claim that $\forall f\in C^\infty_0(\mathbb R^3)$, such that 
\begin{equation*} 
\forall z\in\mathbb R \;:\; \lim_{\rho\to\infty} f(\rho, z) = 0,
\end{equation*}
we have 
\begin{equation}
\label{higher:dimen}
\begin{aligned}
&\int_0^\infty\, f^2\rho^3 d\rho dz \leq C\int_0^\infty\, (\partial_\rho f)^2\rho^5 d\rho dz \\
&\int_0^\infty\, f^2\rho d\rho dz \leq C\int_0^\infty\, (\partial_\rho f)^2\rho^3 d\rho dz
\end{aligned}
\end{equation}
Indeed, we apply the second estimate of  Theorem \ref{bartnik} to the function $\rho\to f(\rho, z)$ with $n = 1$, $p = 2$, $\delta = -1$ in order to obtain the second inequality and $\delta = -2$ in order to obtain the first inequality. 
\item  We estimate the term $\displaystyle \left|\left|(1 - \xi_N - \xi_S)\xi_A\Yzero\right|\right|_{\dot H^2(\mathbb R^3)}$. The other terms follow using similar arguments. We have 
\begin{equation*}
\begin{aligned}
\left|\left|\tilde \xi\Yzero\right|\right|^2_{\dot H^2(\mathbb R^3)} &=  2\pi\int\int_{\tilde\Axis}\, \left|\partial^2(\tilde\xi\Yzero)\right|^2 \rho d\rho dz \\
&= 2\pi\int\int_{\tilde\Axis}\, \left|\partial^2(X_K\tilde\xi\tilde Y)\right|^2 \rho d\rho dz \\
&\leq C\int\int_{\tilde\Axis}\, \left( X^2_K\left|\partial^2(\tilde\xi\tilde Y)\right|^2 +  (\tilde\xi\tilde Y)^2\left|\partial^2X_K\right|^2 + \left|\partial X_K\right|^2\left|\partial(\tilde\xi\tilde Y)\right|^2\right)\rho d\rho dz\\
&\leq C \int\int_{\tilde\Axis}\, \left( \rho^4\left|\partial^2(\tilde\xi\tilde Y)\right|^2 +  (\tilde\xi\tilde Y)^2 + \rho^2\left|\partial(\tilde\xi\tilde Y)\right|^2\right)\rho d\rho dz\\
&\leq C  \int\int_{\tilde\Axis}\,\left|\partial^2(\tilde\xi\tilde Y)\right|^2 \rho^5\,d\rho dz + \int\int_{\tilde\Axis}\,\left|\partial(\tilde\xi\tilde Y)\right|^2 \rho^3\,d\rho dz +  \int\int_{\tilde\Axis}\,  (\tilde\xi\tilde Y)^2 \rho d\rho dz. 
\end{aligned}
\end{equation*}
Now, we use \eqref{higher:dimen} to obtain 
\begin{equation*}
\begin{aligned}
&  \int\int_{\tilde\Axis}\,\left|\partial(\tilde\xi\tilde Y)\right|^2 \rho^3\,d\rho dz \leq C \int\int_{\tilde\Axis}\,\left|\partial^2(\tilde\xi\tilde Y)\right|^2 \rho^5\,d\rho dz
\end{aligned}
\end{equation*}
and 
\begin{equation*}
\begin{aligned}
\int\int_{\tilde\Axis}\,  (\tilde\xi\tilde Y)^2 \rho d\rho dz \leq C   \int\int_{\tilde\Axis}\,\left|\partial(\tilde\xi\tilde Y)\right|^2 \rho^3\,d\rho dz \\
&\leq C\int\int_{\tilde\Axis}\,\left|\partial^2(\tilde\xi\tilde Y)\right|^2 \rho^5\,d\rho dz. 
\end{aligned}
\end{equation*}
Therefore, 
\begin{equation*}
\begin{aligned}
\left|\left|\tilde \xi\Yzero\right|\right|^2_{\dot H^2(\mathbb R^3)} &\leq C  \left|\left|(1 - \xi_N - \xi_S)\xi_A\tilde Y\right|\right|_{\dot H^2(\mathbb R^7_A)}  \\
&\leq C \left( \left|\left|(1 - \xi_N - \xi_S)\xi_AH_YX_K^{-1}\right|\right|_{L^2(\mathbb R^7_A)} + \left|\left| H_X\Xzero\right|\right|^{\frac{1}{2}}_{L^1(\mathbb R^3)} + \left|\left| H_Y\Yzero\right|\right|^{\frac{1}{2}}_{L^1(\mathbb R^3)} \right).
\end{aligned}
\end{equation*}
Here, we used Lemma \ref{lemma::76} to obtain the latter estimate. Finally, by similar arguments, we show that 
\begin{equation*}
\left|\left|(1 - \xi_N - \xi_S)\xi_AH_Y\right|\right|_{L^2(\mathbb R^7_A)} \leq C  \left|\left|(1 - \xi_N - \xi_S)\xi_AH_Y\right|\right|_{L^2(\mathbb R^3)}
\end{equation*}
This ends the proof. 
\end{enumerate}
\end{proof}

\begin{lemma}
\label{lemma::79}
There exists $C>0$ such that 
\begin{equation*}
\begin{aligned}
\left|\left|\Xzero\right|\right|_{\hat C^1(\Bbarre)} &+ \left|\left|\Yzero\right|\right|_{\hat C^1(\Bbarre)} + \left|\left|\tilde Y\right|\right|_{\hat C^1(\Bbarre)} \\
&\leq C \left( \left|\left| H_X\Xzero\right|\right|^{\frac{1}{2}}_{L^1(\mathbb R^3)} + \left|\left| H_Y\Yzero\right|\right|^{\frac{1}{2}}_{L^1(\mathbb R^3)} +  \left|\left|(1 - \xi_N - \xi_S)H_X\right|\right|_{L^2(\mathbb R^3)\cap L^\infty(\mathbb R^3)}  +  \right. \\
&\left|\left|(1 - \xi_N - \xi_S)(1 - \xi_A)H_Y\right|\right|_{L^2(\mathbb R^3)\cap L^\infty(\mathbb R^3)} + \left|\left|(1 - \xi_N - \xi_S)\xi_AX^{-1}_KH_Y\right|\right|_{L^2(\mathbb R^7_A)\cap L^\infty(\mathbb R^7_A)} + \\
&\left|\left|(s^2 + \chi^2)\xi_N X_K^{-1}H_Y\right|\right|_{L^\infty(\mathbb R^8_N)} + \left|\left|((s')^2 + (\chi')^2)\xi_S X_K^{-1}H_Y\right|\right|_{L^\infty(\mathbb R^8_S)} + \\
&\left. \left|\left|(s^2 + \chi^2)\xi_N H_X\right|\right|_{L^\infty(\mathbb R^4_N)} + \left|\left|((s')^2 + (\chi')^2)\xi_S H_X\right|\right|_{L^\infty(\mathbb R^4_S)} \right). 
\end{aligned}
\end{equation*}
\end{lemma}

\begin{proof}
We will detail the $\hat C^1(\Bbarre)$ estimates for $\Xzero$. $\hat C^1(\Bbarre)$ estimates for $\Yzero$ and $\tilde Y$ follow using similar argument. 
\\ To this end, we need to show that 
\begin{equation*}
\xi_N\Xzero, \xi_S\Xzero\in C^1(\mathbb R^4) \quad\text{and}\quad (1-\xi_N - \xi_S)\Xzero\in C^1(\mathbb R^3). 
\end{equation*}
\begin{enumerate}
\item By Lemma \ref{lemma::73}, $\Xzero_A = (1-\xi_N - \xi_S)\Xzero, \partial \Xzero_A\in \dot H^1(\mathbb R^3)$. 
\item Now, we recall the continuous embedding 
\begin{equation}
\label{embed}
 \dot H^1(\mathbb R^3) \subset L^6(\mathbb R^3). 
\end{equation}
\item Recall that $(1-\xi_N - \xi_S)\Xzero$ satisfies \eqref{Xzero::A}. 
\item Now, we claim that $\partial^2\Xzero_A\in L^6(\mathbb R^3)$. To prove the latter, we use equation \eqref{Xzero::A} and we apply Theorem \ref{Cal:Zyg} with $f$ given by \eqref{f:Cal}, $p = 6$ and $n = 3$. We need to show that $f\in L^6(\mathbb R^3)$. 
\\ Recall that $f$ is given by 
\begin{equation*}
\begin{aligned}
f (x)&= 2\partial\Xzero\cdot\partial(\xi_N + \xi_S) + \Delta_{\mathbb R^3}(\xi_N + \xi_S)\Xzero - 2(1 - \xi_N - \xi_S)\frac{\partial Y_K\cdot\partial\Yzero}{X_K} + 2(1 - \xi_N - \xi_S)\Yzero\frac{\partial X_K\cdot\partial Y_K}{X_K^2} \\
& + 2\frac{|Y_K|^2}{X_K^2}(1 - \xi_N - \xi_S)\Xzero + (1 - \xi_N - \xi_S)H_X
\end{aligned}
\end{equation*}
\begin{itemize}
\item It is easy to see that the terms $\partial\Xzero\cdot\partial(\xi_N + \xi_S),  \Delta_{\mathbb R^3}(\xi_N + \xi_S)\Xzero$ and $(1 - \xi_N - \xi_S)H_X$ lie in $L^6(\mathbb R^3)$
\item By the decay estimates of $\displaystyle \frac{\partial Y_K}{X_K}$, we obtain
\begin{equation*}
\left|\left|(1 - \xi_N - \xi_S)\frac{\partial Y_K\cdot\partial\Yzero}{X_K} \right|\right|_{L^6(\mathbb R^3)} \leq C \left|\left|\Yzero\right|\right|_{\dot W^{1, 6}_{axi}(\mathbb R^3) }.
\end{equation*}

\item We have 
\begin{equation*}
\begin{aligned}
\left|\left|(1 - \xi_N - \xi_S)\Yzero\frac{\partial X_K\cdot\partial Y_K}{X_K^2} \right|\right|_{L^6(\mathbb R^3)} &\leq  \left|\left|(1 - \xi_N - \xi_S)\Yzero\frac{\partial X_K}{X_K}\right|\right|_{L^2(\mathbb R^3)}\left|\left|\frac{\partial Y_K}{X_K}\right|\right|_{L^2(\mathbb R^3)}  \\
&\leq C \left|\left|(1 - \xi_N - \xi_S)\Yzero\right|\right|_{L^2_{\nabla h}(\mathbb R^3)}
\end{aligned}
\end{equation*}
\item Moreover, by Theorem \ref{bartnik} with $n= 3$, $p = 6$ and  $\displaystyle \delta = \frac{1}{2}$
\begin{equation*}
\begin{aligned}
\left|\left|(1 - \xi_N - \xi_S)\Xzero\frac{|Y_K|^2}{X_K^2} \right|\right|_{L^6(\mathbb R^3)} &\leq C \left|\left|(1 - \xi_N - \xi_S)\Xzero \langle r \rangle^{-8}\right|\right|_{L^6(\mathbb R^3)} \\
&\leq C \left|\left|(1 - \xi_N - \xi_S)\Xzero \right|\right|_{6, \frac{1}{2}} \\
&\leq C \left|\left|\partial \Xzero_A\right|\right|_{6, -\frac{1}{2}} \\
&=  C \left|\left|\partial \Xzero_A\right|\right|_{L^6(\mathbb R^3)} \\
\end{aligned}
\end{equation*}
\end{itemize}
\item Finally, we use the Sobolev embedding 
\begin{equation*}
W^{2, 6}(\mathbb R^3) \subset C^{1, \frac{1}{2}}(\mathbb R^3) 
\end{equation*}
in order to obtain 
\begin{equation*}
\Xzero_A\in C^1(\mathbb R^3). 
\end{equation*}
Furthermore, 
\begin{equation*}
\begin{aligned}
\left|\left|\Xzero_A\right|\right|_{C^1(\mathbb R^3)} &\leq C ||f||_{L^6(\mathbb R^3)} \\
&\leq C \left( \left|\left|(1 - \xi_N - \xi_S)\Xzero\right|\right|_{\dot H^2(\mathbb R^3)} + \left|\left|(1 - \xi_N - \xi_S)\Yzero\right|\right|_{\dot H^2(\mathbb R^3)} + \left|\left|(1 - \xi_N - \xi_S)\Yzero\right|\right|_{L^2_{\nabla h}(\mathbb R^3)}  \right)\\
&\leq C\left( \left|\left| H_X\Xzero\right|\right|^{\frac{1}{2}}_{L^1(\mathbb R^3)} + \left|\left| H_Y\Yzero\right|\right|^{\frac{1}{2}}_{L^1(\mathbb R^3)} +  \left|\left|(1 - \xi_N - \xi_S)H_X\right|\right|_{L^2(\mathbb R^3)\cap L^\infty(\mathbb R^3)} + \right. \\
&\left. \left|\left|(1 - \xi_N - \xi_S)(1 - \xi_A)H_Y\right|\right|_{L^2(\mathbb R^3)\cap L^\infty(\mathbb R^3)} + \left|\left|(1 - \xi_N - \xi_S)\xi_AX^{-1}_KH_Y\right|\right|_{L^2(\mathbb R^7_A)\cap L^\infty(\mathbb R^7_A)} \right). 
\end{aligned}
\end{equation*}
\item Using similar arguments and the previous lemmas, we estimate the $C^1$ norm of $\xi_N\Xzero$ and $\xi_S\Xzero$. 
\end{enumerate}
\end{proof}
\noindent Finally, we prove Proposition \ref{linear:est:XY}
\begin{proof}
We prove that there exists $C = C(\alpha_0)>0$ such that 
\begin{equation*}
\left|\left|(\Xzero, \Yzero)\right|\right|_{\LX\times\LY} \leq C  \left|\left|(H_X, H_Y)\right|\right|_{\NX\times\NY}. 
\end{equation*}
Recall the norms: 
 \begin{equation*}
 \begin{aligned}
      ||\Xzero||_{\LX} &= ||\Xzero||_{\dot W^{1,2}_{axi}(\Bbarre)} + ||\Xzero||_{\hat{C}^{2,\alpha_0}(\Bbarre)} + ||r\Xzero||_{L^{\infty}(\Bbarre)} + ||r^2\hat{\partial} \Xzero||_{L^{\infty}(\Bbarre)} + ||r^3\log^{-1}(4r)\hat{\partial}^2 \Xzero||_{C^{0,\alpha_0}(\Bbarre)}, \\
      ||\Yzero||_{\LY} &= ||\Yzero||_{\dot W^{1,2}_{axi}(\Bbarre)} + |||\partial h|\Yzero ||_{L^2(\mathbb{R}^3)} + ||\Yzero||_{\hat{C}^{2,\alpha_0}(\Bbarre)} + ||X_K^{-1}\Yzero||_{\hat{C}^{2,\alpha_0}(\Bbarre)} + ||r^3X_K^{-1}\Yzero||_{L^{\infty}(\Bbarre)}  \\
&+ ||r^4\hat{\partial}( X^{-1}_K\Yzero)||_{L^{\infty}(\Bbarre)} + ||r^5\log^{-1}(4r)\hat{\partial}^2( X^{-1}_K\Yzero)||_{C^{0,\alpha_0}(\Bbarre)}, 
      \end{aligned}
             \end{equation*}

  \begin{equation*}
  \begin{aligned}
      ||H_X||_{\NX} &= ||r^3(1-\xi_N-\xi_S)H_X||_{C^{0,\alpha_0}(\mathbb{R}^3)} + ||(\chi^2+s^2)\xi_NH_X||_{C^{0,\alpha_0}(\Bnbarre)} + ||((\chi')^2+(s')^2)\xi_S H_Y||_{C^{0,\alpha_0}(\Bsbarre)}, \\
            ||H_Y||_{\mathcal{N}_{Y}} &= ||H_Yr^5X^{-1}_K||_{\hat{C}^{0,\alpha_0}\left(\left(\BHbarre\cup \BAbarre\right)\cap \left\{ \rho\le 1\right\}\right)} +  ||H_Yr^4||_{\hat{C}^{0,\alpha_0}(\Bbarre\cap\left\{ \rho\ge 1\right\})} + ||(\chi^2+s^2)\xi_NH_X||_{C^{0,\alpha_0}(\Bnbarre)}  \\
            &+ ||((\chi')^2+(s')^2)\xi_S H_Y||_{C^{0,\alpha_0}(\Bsbarre)}
      \end{aligned}
    \end{equation*}

\begin{enumerate}
\item First of all, by 
we have
\begin{equation*}
\begin{aligned}
&\left|\left|(1 - \xi_N - \xi_S)H_X\right|\right|_{L^2(\mathbb R^3)\cap L^\infty(\mathbb R^3)}  \leq C \left|\left|(1 - \xi_N - \xi_S)H_X\right|\right|_{C^{0, \alpha}(\mathbb R^3)}  \\ 
& \left|\left|(s^2 + \chi^2)\xi_N H_X\right|\right|_{L^\infty(\mathbb R^4_N)} \leq C \left|\left|(s^2 + \chi^2)\xi_N H_X\right|\right|_{C^{0, \alpha}(\Bnbarre)}\\
& \left|\left|((s')^2 + (\chi')^2)\xi_S H_X\right|\right|_{L^\infty(\mathbb R^4_S)} \leq C \left|\left|((s')^2 + (\chi')^2)\xi_S H_X\right|\right|_{C^{0, \alpha}(\Bsbarre)}
\end{aligned}
\end{equation*}
\item Similarly, we show that 
\begin{equation*}
\begin{aligned}
&\left|\left|(s^2 + \chi^2)\xi_N X_K^{-1}H_Y\right|\right|_{L^\infty(\mathbb R^8_N)} \leq C  \left|\left|(s^2 + \chi^2)\xi_N X_K^{-1}H_Y\right|\right|_{C^{0, \alpha}(\Bnbarre)}\\
&\left|\left|((s')^2 + (\chi')^2)\xi_S X_K^{-1}H_Y\right|\right|_{L^\infty(\mathbb R^8_S)} \leq C \left|\left|((s')^2 + (\chi')^2)\xi_S X_K^{-1}H_Y\right|\right|_{C^{0, \alpha}(\Bsbarre)}
\end{aligned}
\end{equation*}

\item By similar arguments,  we have 
\begin{equation*}
\begin{aligned}
\left|\left|(1 - \xi_N - \xi_S)(1 - \xi_A)H_Y\right|\right|_{L^2(\mathbb R^3)\cap L^\infty(\mathbb R^3)} + \left|\left|(1 - \xi_N - \xi_S)\xi_AX^{-1}_KH_Y\right|\right|_{L^2(\mathbb R^7_A)\cap L^\infty(\mathbb R^7_A)} \leq C 
\end{aligned}
\end{equation*}
\item We estimate the term $\displaystyle   \left|\left| H_X\Xzero\right|\right|^{\frac{1}{2}}_{L^1(\mathbb R^3)}$: 
\begin{equation*}
\begin{aligned}
 \left|\left| H_X\Xzero\right|\right|^{\frac{1}{2}}_{L^1(\mathbb R^3)} &\leq \left( \left|\left|H_X\right|\right|_{L^{\frac{6}{5}}(\mathbb R^3)} \left|\left| \Xzero\right|\right|_{L^6(\mathbb R^3)} \right)^{\frac{1}{2}} \\
 &\leq \left|\left|H_X\right|\right|^{\frac{1}{2}}_{L^{\frac{6}{5}}(\mathbb R^3)} \left(  \left|\left| (1 - \xi_N - \xi_S)\Xzero\right|\right|_{L^6(\mathbb R^3)} +  \left|\left|\xi_N\Xzero\right|\right|_{L^6(\mathbb R^3)} +  \left|\left|\xi_S\Xzero\right|\right|_{L^6(\mathbb R^3)} \right)^{\frac{1}{2}}. 
\end{aligned}
\end{equation*}
\begin{itemize}
\item $\xi_N\Xzero$ and $\xi_S\Xzero$ are compactly supported on $\Bnbarre$ and $\Bsbarre$ respectively. Hence, 
\begin{equation*}
 \left|\left|\xi_N\Xzero\right|\right|_{L^6(\mathbb R^3)} +  \left|\left|\xi_S\Xzero\right|\right|_{L^6(\mathbb R^3)} \leq C \left(  \left|\left|\xi_N\Xzero\right|\right|_{L^\infty(\mathbb R^3)} +  \left|\left|\xi_S\Xzero\right|\right|_{L^\infty(\mathbb R^3)} \right).  
\end{equation*}
\item By \eqref{embed}, 
\begin{equation*}
 \left|\left| (1 - \xi_N - \xi_S)\Xzero\right|\right|_{L^6(\mathbb R^3)} \leq C  \left|\left| (1 - \xi_N - \xi_S)\Xzero\right|\right|_{\dot H^1(\mathbb R^3)}. 
\end{equation*}
\item Thus, $\forall q>0$, we have 
\begin{equation*}
\begin{aligned}
\left|\left| H_X\Xzero\right|\right|^{\frac{1}{2}}_{L^1(\mathbb R^3)} &\leq C \left( q^{-1}  \left|\left|H_X\right|\right|_{L^{\frac{6}{5}}(\mathbb R^3)}  + q  \left|\left| (1 - \xi_N - \xi_S)\Xzero\right|\right|_{\dot H^1(\mathbb R^3)} +  q \left|\left|\xi_N\Xzero\right|\right|_{L^\infty(\mathbb R^3)}  \right. \\
& \left. +  q \left|\left|\xi_S\Xzero\right|\right|_{L^\infty(\mathbb R^3)} \right) \\
&\leq C \left( q^{-1}  \left|\left|(1 - \xi_N - \xi_S)H_X\right|\right|_{L^{\frac{6}{5}}(\mathbb R^3)} +   q \left|\left| (1 - \xi_N - \xi_S)\Xzero\right|\right|_{\dot H^1(\mathbb R^3)}  \right.  \\
&\left. + q^{-1}  \left|\left|\xi_NH_X\right|\right|_{L^{\frac{6}{5}}(\mathbb R^3)}   + q \left|\left|\xi_N\Xzero\right|\right|_{L^\infty(\mathbb R^3)} + q^{-1}  \left|\left|\xi_SH_X\right|\right|_{L^{\frac{6}{5}}(\mathbb R^3)} + q \left|\left|\xi_S\Xzero\right|\right|_{L^\infty(\mathbb R^3)}\right), 
\end{aligned}
\end{equation*}
\item $\xi_NH_X$ and $\xi_SH_X$ are compactly supported on $\Bnbarre$ and $\Bsbarre$ respectively. Hence, 
\begin{equation*}
\left|\left|\xi_NH_X\right|\right|_{L^{\frac{6}{5}}(\mathbb R^3)} \leq C \left|\left|\xi_NH_X\right|\right|_{L^{1}(\mathbb R^3)\cap L^{\infty}(\mathbb R^3)}
\end{equation*}
and 
\begin{equation*}
\left|\left|\xi_SH_X\right|\right|_{L^{\frac{6}{5}}(\mathbb R^3)} \leq C \left|\left|\xi_SH_X\right|\right|_{L^{1}(\mathbb R^3)\cap L^{\infty}(\mathbb R^3)}
\end{equation*}
\item Furthermore, 
\begin{equation*}
\begin{aligned}
\left|\left|\xi_N\Xzero\right|\right|_{L^\infty(\mathbb R^3)} + \left|\left|\xi_S\Xzero\right|\right|_{L^\infty(\mathbb R^3)}  &\leq C  \left|\left|\Xzero\right|\right|_{\hat C^1(\Bbarre)} \\
\end{aligned}
\end{equation*}
\item We obtain similar estimates for $H_Y\Yzero$. 
\end{itemize}
\item We choose $q>0$ so that by Lemma \ref{lemma::79}, we obtain 
\begin{equation*}
\begin{aligned}
||\Xzero||_{\dot W^{1,2}_{axi}(\Bbarre)} + ||\Yzero||_{\dot W^{1,2}_{axi}(\Bbarre)} + \left|\left|\Xzero\right|\right|_{\hat C^1(\Bbarre)} &+ \left|\left|\Yzero\right|\right|_{\hat C^1(\Bbarre)} \leq C(\alpha_0) \left|\left|(H_X, H_Y)\right|\right|_{\NX\times\NY}. 
\end{aligned}
\end{equation*}
\item Now, we estimate the  $C^{2, \alpha_0}$ part of $(\Xzero, \Yzero)$ in terms of $\left|\left|(H_X, H_Y)\right|\right|_{\NX\times \NY}$. 
\begin{itemize}
\item First, since $(H_X, H_Y)_{\mathbb R^3}$ are compactly supported in $\mathbb R^3$, $(\Xzero, \Yzero)$ are compactly supported in $\Bbarre$ and we have 
\begin{equation*}
\supp(\Xzero, \Yzero)_{\mathbb R^3} \subset \supp({H_X}_{\mathbb R^3})\cap\supp({H_Y}_{\mathbb R^3}). 
\end{equation*}
Denote by $K\subset\subset\Bbarre$ the support of $\Xzero$ and $\Yzero$ and assume for simplicity that $K = \overline B(x_0, 1)$ where $x_0\in\Bbarre$. Depending on the position of $x_0$, $K$ lies either in $\BAbarre\cup\BHbarre$, $\Bnbarre$ or $\Bsbarre$.  In the following, we establish the $C^{2, \alpha_0}$ estimates for $(\Xzero, \Yzero)$ in the regions. To this end, we use Theorem \ref{Schauder:general}. 
\\First of all, note that $(\Xzero, \Yzero)\in \hat C^{2, \alpha_0}(\Bbarre)$. Now, we have 
\begin{enumerate}
\item if  $K\subset\subset \BAbarre\cup\BHbarre$, the second order operator $(\Xzero, \tilde Y) \to L(\Xzero, \tilde Y)$ defined by  
\begin{equation*}
 L(\Xzero, \tilde Y) = \left(
\begin{aligned}
     &\Delta_{\mathbb{R}^3}\overset{\circ}{X} + \frac{2\partial Y_K\cdot\partial\overset{\circ}{Y}}{X_K} - \frac{2|\partial Y_K|^2}{X_K^2}\overset{\circ}{X} + 2\frac{\partial X_K\cdot\partial Y_K}{X_K^2}\overset{\circ}{Y}  \\
&\Delta_{\mathbb{R}^7}\tilde{Y}_{\mathbb R^7} + \tilde e_A\cdot\partial \tilde Y - 2\frac{|\partial Y_K|^2}{X^2_K}\tilde{Y}_{\mathbb R^7} - 2 \tilde d_A\cdot\partial\Xzero - \frac{2\overline d_A}{\rho}\partial_\rho \Xzero.
\end{aligned}
\right)
\end{equation*}
is uniformly elliptic on $\BAbarre\cup\BHbarre$. We apply Theorem \ref{Schauder:general} with $\Omega = \overset{\circ}{K}$, $L$ and $f = (H_X, X_K^{-1}H_Y)$
\item The remaining estimates in $\Bnbarre$ and $\Bsbarre$ follow in a similar manner. 
\end{enumerate} 
\end{itemize}
\item Now, we show that 
\begin{equation*}
\begin{aligned}
 ||r\Xzero||_{L^{\infty}(\Bbarre)} + ||r^2\hat{\partial} \Xzero||_{L^{\infty}(\Bbarre)} + ||r^3\log^{-1}(4r)\hat{\partial}^2 \Xzero||_{C^{0,\alpha_0}(\Bbarre)} \leq C \left|\left|(H_X, H_Y)\right|\right|_{\NX\times \NY}. 
\end{aligned}
\end{equation*}
 In order to obtain the latter estimate, we apply the newtonian estimates provided by Lemma \ref{Newton:1} and Lemma \ref{Newton:2} on each region $\BAbarre\cup \BHbarre$, $\Bnbarre$ and $\Bsbarre$:  
 \begin{enumerate}
 \item on the region $\BAbarre\cup\BHbarre$, with $F$ given by 
 \begin{equation*}
 F_{\mathbb R^3} := H_X  - \frac{2\partial Y_K\cdot\partial\overset{\circ}{Y}}{X_K} + \frac{2|\partial Y_K|^2}{X_K^2}\overset{\circ}{X} - 2\frac{\partial X_K\cdot\partial Y_K}{X_K^2}\overset{\circ}{Y}. 
 \end{equation*}
 Therefore, there exists $C>0$ such that $\forall x\in\mathbb R^3$
 \begin{equation*}
 \left|F(x) \right| \leq C r^{-4}.
 \end{equation*}
 
 \item In the region $\Bnbarre$ with $F$ given by 
 \begin{equation*}
 F_{\mathbb R^4}:=  (s^2 + \chi^2)H_X  - 2\frac{\underline\partial Y_K\cdot\partial\Yzero}{X_K} + 2\frac{|\underline\partial Y_K|^2}{X_K^2}\Xzero - 2\frac{\underline\partial X_K\cdot \underline\partial Y_K}{X_K^2}\Yzero.
 \end{equation*}
 \end{enumerate}

\end{enumerate}
\end{proof}

\subsubsection{Non-linear estimates}
We apply Theorem \ref{Fixed::Point::2} in order to obtain
\begin{Propo}
\label{non:linear:XY}
Let $\alpha_0\in(0, 1)$ and let $\overline\delta_0>0$. Then, there exists $0<\delta_0 \leq \overline\delta_0$ such that $\displaystyle \forall \left(\thetazero, \lambdazero, \delta\right)\in B_{\delta_0}\left(\Ltheta\times\Llambda\times[0, \infty[\right)$ there exists a unique one parameter family $\left(\sigmazero, \left(0, 0, B_z^{(A)}\right), \left(\Xzero, \Yzero \right)\right)$ depending on $\left(\thetazero, \lambdazero, \delta\right)\in B_{\delta_0}\left(\Lsigma\times\LB\times \LX\times\LY\right)$ which solves \eqref{sigmazero}, \eqref{eq:for:B} and \eqref{XYzero} and which satisfies
\begin{equation*}
        \left|\left|\left(\sigmazero, \left(0, 0, B_z^{(A)}\right), \left(\Xzero, \Yzero\right)\right)\left(\Xzero, \thetazero, \lambdazero, \delta\right)\right|\right|_{\Lsigma\times\LB} \le C(\alpha_0)\left(||\thetazero||^2_{\Ltheta} + ||\lambdazero||^2_{\Llambda} + \delta\right)
  \end{equation*}
  and $\displaystyle \forall \left(\thetazero_i, \lambdazero_i, \delta_i\right)\in B_{\delta_0}\left(\Ltheta\times\Llambda\times[0, \infty[\right)$,
        \begin{equation*}
        \begin{aligned}
        &\left|\left|\left(\sigmazero, \left(0, 0, B_z^{(A)}\right), \left(\Xzero, \Yzero \right)\right)\left(\thetazero_1, \lambdazero_1, \delta_1\right) - \left(\sigmazero, \left(0, 0, B_z^{(A)}\right), \left(\Xzero, \Yzero \right)\right)\left(\thetazero_2, \lambdazero_2, \delta_2\right)\right|\right|_{\Lsigma\times \LB\times\LX\times\LY} \\ 
        &\le C(\alpha_0)\left(\left(\left|\left|\left(\Xzero_1, \thetazero_1, \lambdazero_1\right)\right|\right|_{\LX\times\Ltheta\times\Llambda} + \left|\left|\left(\Xzero_2, \thetazero_2, \lambdazero_2\right)\right|\right|_{\LX\times\Ltheta\times\Llambda}\right)\left|\left|\left(\thetazero_1, \lambdazero_1\right) -\left(\thetazero_2, \lambdazero_2\right)\right|\right|_{\Ltheta\times\Llambda} \right. \\
        &\left. + \left|\delta_1 - \delta_2\right| \right).
        \end{aligned}
        \end{equation*}
\end{Propo}
Before we prove the above proposition, we introduce the following notations
\begin{enumerate}
\item Define $H_X$ and $H_Y$ on $\Bbarre$ to be the mappings
\begin{equation*}
\begin{aligned}
H_X(\rho, z) &= N_X^{(1)}(\sigmazero, B, (\Xzero, \Yzero), \thetazero, \lambdazero;\delta) + N_X^{(2)}(\sigmazero, B, (\Xzero, \Yzero), \thetazero, \lambdazero;\delta), \\
H_Y(\rho, z) &= N_Y^{(1)}(\sigmazero, B, (\Xzero, \Yzero), \thetazero, \lambdazero;\delta) + N_Y^{(2)}(\sigmazero, B, (\Xzero, \Yzero), \thetazero, \lambdazero;\delta), 
\end{aligned}
\end{equation*}
where

\begin{equation*}
    N^{(1)}_X(\Xzero, \Yzero)(\rho, z) := \frac{X_K^2(|\partial\overset{\circ}{X}|^2-|\partial\overset{\circ}{Y}|^2) + (\overset{\circ}{X}\partial Y_K -\overset{\circ}{Y}\partial X_K)\cdot(2X_K\partial\overset{\circ}{Y}-\overset{\circ}{X}\partial Y_K+\overset{\circ}{Y}\partial X_K)}{X_K^2(1+\overset{\circ}{X})},    
    \end{equation*}
    \begin{equation*}
    N^{(1)}_Y(\Xzero, \Yzero)(\rho, z) := \frac{\partial\overset{\circ}{X}\cdot\partial\overset{\circ}{Y}+2X_K(\overset{\circ}{Y}\partial X_K-\overset{\circ}{X}\partial Y_K)\cdot\partial\overset{\circ}{X}}{X_K^2(1+\overset{\circ}{X})},    
    \end{equation*}
    \begin{equation*}
    \begin{aligned}
    N^{(2)}_X(\sigmazero, B, (\Xzero, \Yzero), \thetazero, \lambdazero;\delta)(\rho, z) &:= (\rho^{-1}-\sigma^{-1}\partial_{\rho}\sigma)\frac{ \partial_{\rho}(X_K(1+\overset{\circ}{X}))}{X_K}-\sigma^{-1}\partial_z\sigma\frac{\partial_z(X_K(1+\overset{\circ}{X}))}{X_K}  \\
    &- \frac{2\partial_{\rho}(Y_K+X_K\overset{\circ}{Y})B_{\rho}}{X_K^2(1+\overset{\circ}{X})}- \frac{2\partial_{z}(Y_K+X_K\overset{\circ}{Y})B_z}{X_K^2(1+\overset{\circ}{X})}  \\
    &- \frac{B_{\rho}^2+B_z^2}{X_K^2(1+\overset{\circ}{X})}  + X^{-1}_KF_1(\thetazero, \Xzero, \sigmazero)(\rho, z), 
    \end{aligned}
    \end{equation*}
    \begin{equation*}
        \begin{aligned}
        N^{(2)}_Y(\sigmazero, B, (\Xzero, \Yzero), \thetazero, \lambdazero;\delta)(\rho, z) &:= (\rho^{-1}-\sigma^{-1}\partial_{\rho}\sigma)\partial_{\rho}(Y_K+X_K\overset{\circ}{Y})X_K^{-1}-\sigma^{-1}\partial_{\rho}\sigma B_{\rho}X_K^{-1}  \\
        &- \sigma^{-1}\partial_{z}\sigma(\partial_{z}(Y_K+X_K\overset{\circ}{Y})+B_z)X_K^{-1} \\
        &+\frac{B_{\rho}\partial_{\rho}(X_K(1+\overset{\circ}{X}))+B_{z}\partial_{z}(X_K(1+\overset{\circ}{X}))}{X_K^2(1+\overset{\circ}{X})},
        \end{aligned}
    \end{equation*}

\item Let $\overline\delta_0>0$ be obtained by Proposition \eqref{non:linear:B}. Define the nonlinear operator $N_{(X,Y)}$ on  $B_{\overline\delta_0}(\LX\times\LY)\times B_{\overline\delta_0}\left(\Ltheta\times\Llambda\times[0, \infty[\right)$ by
\begin{equation*}
N_{(X,Y)}\left(\left(\Xzero, \Yzero\right), (\thetazero, \lambdazero; \delta)\right)(\rho, z) :=  \left( 
\begin{aligned}
& (N^{(1)}_X + N^{(2)}_X)(\sigmazero(\Xzero, \thetazero, \lambdazero; \delta), B(\Xzero, \thetazero, \lambdazero; \delta), (\Xzero, \Yzero), \thetazero, \lambdazero; \delta) \\
& (N^{(1)}_Y + N^{(2)}_Y)(\sigmazero(\Xzero, \thetazero, \lambdazero; \delta), B(\Xzero, \thetazero, \lambdazero; \delta), (\Xzero, \Yzero), \thetazero, \lambdazero; \delta)
\end{aligned}
\right). 
\end{equation*} 
where $\sigmazero$ and $B$ are  the mappings obtained by Proposition \ref{non:linear:B}. 
\end{enumerate}
\noindent Now, we state the following lemma

\begin{lemma}
\label{est:B}
\begin{enumerate}
\item $\displaystyle N_{(X,Y)}\left(B_{\overline\delta_0}(\LX\times \LY)\times B_{\overline\delta_0}(\LX\times\Ltheta\times\Llambda\times[0, \infty[)\right)\subset \NX\times\NY$. 
\item There exist $0<\delta_0\leq\overline \delta_0$ and $C(\alpha_0)>0$ such that $\forall \left(\Xzero, \thetazero, \lambdazero; \delta\right)\in B_{\overline\delta_0}(\LX\times\Ltheta\times\Llambda\times[0, \infty[)$
\begin{equation}
\label{non:linear:NB}
||(H_X, H_Y) ||_{\NX\times\NY} \leq C(\alpha_0)\left( \left|\left|\Xzero\right|\right|^2_{\LX} + \left|\left|\Yzero\right|\right|^2_{\LY} + \left|\left|(\thetazero, \lambdazero)\right|\right|^2_{\Ltheta\times\Llambda} + \delta \right). 
\end{equation}
\end{enumerate}
\end{lemma}

\begin{proof}
\begin{enumerate}
\item We use the same arguments used to prove Lemma \ref{est:sigma} and Lemma \ref{est:B}: the compact support of $F_1(\thetazero, \Xzero, \sigmazero; \delta)$ and the differentiability of $F_1$  with respect to each variable.  
\\ In order to obtain \eqref{non:linear:NB}, we write 
\begin{equation*}
\begin{aligned}
F_1(\thetazero, \Xzero, \sigmazero (\Xzero, \thetazero, \lambdazero; \delta); \delta) &= D_hF_1(0, 0, 0; 0)[(\Xzero, \thetazero, \sigmazero(\Xzero, \thetazero, \lambdazero; \delta) )] + D_\delta F_1(0, 0, 0; 0)\delta  \\
&+ O (||(\Xzero, \thetazero, \sigmazero(\Xzero, \thetazero, \lambdazero; \delta))||^2 + \delta^2) \\
&=  D_\delta F_1(0, 0, 0; 0)\delta + O (||(\Xzero, \thetazero, \sigmazero(\Xzero, \thetazero, \lambdazero; \delta) )||^2 + \delta^2). 
\end{aligned} 
\end{equation*}
By Proposition \ref{non:linear:sigma}, we have 
\begin{equation*}
        \left|\left|\sigmazero\left(\Xzero, \thetazero, \lambdazero, \delta\right)\right|\right|_{\Lsigma} \le C(\alpha_0)\left(||\Xzero||^2_{\LX} + ||\thetazero||^2_{\Ltheta} + ||\lambdazero||^2_{\Llambda} + \delta\right). 
  \end{equation*}
 \noindent Thus, 
  \begin{equation*}
||F_1(\thetazero, \Xzero, \sigmazero (\Xzero, \thetazero, \lambdazero; \delta); \delta)||_{\hat C^{1, \alpha_0}(\Bbarre)} \leq C(\alpha_0) \left(||\Xzero||^2_{\LX} + ||\thetazero||^2_{\Ltheta} + ||\lambdazero||^2_{\Llambda} + \delta\right). 
\end{equation*}
 \noindent Now, we write
\begin{equation*}
X_K^{-1}F_1 = (1 - \xi_N - \xi_S)X_K^{-1}F_1 + \xi_NX_K^{-1}F_1 + \xi_SX_K^{-1}F_1. 
\end{equation*}
By Proposition  \ref{Fi::regularity}, $F_1$ is compactly supported in $\BAbarre$. Therefore, 
\begin{equation*}
\begin{aligned}
||X_K^{-1}F_1(\thetazero, \Xzero, \sigmazero (\Xzero, \thetazero, \lambdazero; \delta); \delta)||_{\NX} &= ||r^3(1-\xi_N-\xi_S)X_K^{-1}F_1||_{C^{0,\alpha_0}(\mathbb{R}^3)} \\
&\leq  C||F_1(\thetazero, \Xzero, \sigmazero (\Xzero, \thetazero, \lambdazero; \delta); \delta)||_{\hat C^{1, \alpha_0}(\Bbarre)}  \\
&\leq C(\alpha_0) \left(||\Xzero||^2_{\LX} + ||\thetazero||^2_{\Ltheta} + ||\lambdazero||^2_{\Llambda} + \delta\right)
\end{aligned}
\end{equation*}
\item We refer to the proof of Proposition $9.2.1$ in \cite{chodosh2017time} in order to estimate $N_X^{(1)}$,  $N_Y^{(1)}$, $N_Y^{(2)}$ and the remaining terms of $N_X^{(2)}$. 

\end{enumerate}
\end{proof}
\noindent Now, we prove Proposition \ref{non:linear:B}. 
\begin{proof}
\begin{enumerate}
\item First of all, by Proposition \ref{non:linear:sigma} and Proposition \ref{non:linear:B}, there exists a solution map $(\sigmazero, B)(\Xzero, \thetazero, \lambdazero; \delta)$ defined on $B_{\delta_0}$ which solves \eqref{sigmazero} and \eqref{Bzero}.  
\item We apply Theorem \ref{Fixed::Point::2} with $\displaystyle L$ defined by 
\begin{equation*}
L(\Xzero, \Yzero) := \left( 
\begin{aligned}
&\Delta_{\mathbb{R}^3}\overset{\circ}{X} + \frac{2\partial Y_K\cdot\partial\overset{\circ}{Y}}{X_K} - \frac{2|\partial Y_K|^2}{X_K^2}\overset{\circ}{X} + 2\frac{\partial X_K\cdot\partial Y_K}{X_K^2}\overset{\circ}{Y} \\
        &\Delta_{\mathbb{R}^3}\overset{\circ}{Y} - \frac{2\partial Y_K\cdot\partial\overset{\circ}{X}}{X_K} - \frac{(|\partial X_K|^2+|\partial Y_K|^2)}{X_K^2}\overset{\circ}{Y}
\end{aligned}
\right),
\end{equation*}
$\displaystyle N = N_{(X, Y)}\;,\; \mathcal L = \LX\times\LY\;,\; \mathcal Q = \Ltheta\times\Llambda$ and $\displaystyle \mathcal P = [0, \delta_0[$. By the previous lemma, all the assumptions are satisfied and we obtain the desired result. 
\end{enumerate}
\end{proof}

\subsection{Solving for $\Theta$}
We recall that $\thetazero$ satisfies
  \begin{equation*}
        \begin{aligned}
        \partial_{\rho}\overset{\circ}{\Theta} &= -\frac{\sigma}{X^2}(\partial_zY+B_z)+\frac{\rho}{X_K^2}\partial_zY_K, \\
        \partial_{z}\overset{\circ}{\Theta} &= \frac{\sigma}{X^2}(\partial_{\rho}Y+B_{\rho})-\frac{\rho}{X_K^2}\partial_{\rho}Y_K.
        \end{aligned}
    \end{equation*}
    
    \subsubsection{Linear problem}
    We prove the following result
    \begin{Propo}
    \label{linear:theta}
    Let $H_{\Theta}\in \Ntheta$. Then, there exists  $\thetazero\in\Ltheta$ which solves the equation 
    \begin{equation}
    \label{theta:linear}
  d\thetazero = H_{\Theta}.
   \end{equation}
It is given by 
\begin{equation}
\label{exp:theta}
\thetazero(\rho, z) := -\int_\rho^\infty\, \left(H_\Theta\right)_\rho(\tilde\rho, z)\,d\tilde \rho. 
\end{equation}
Moreover, there exists $C(\alpha_0)>0$ such that 
\begin{equation*}
||\sigmazero||_{\Lsigma} \leq C(\alpha_0) ||H_\sigma||_{\Nsigma}.
\end{equation*} 
    \end{Propo}
   
    Recall that $\Ntheta$ is the completion of smooth compactly supported closed $1-$forms under the norm
\begin{equation*}
\begin{aligned}
      ||F||_{\Ntheta} &:= ||r^3(1+\rho^{-1})F_{\rho}||_{\hat{C}^{1,\alpha_0}\left(\BHbarre\cup\BAbarre\right)} + ||r^3F_{z}||_{\hat{C}^{1,\alpha_0}\left(\BHbarre\cup\BAbarre\right)} + ||s^{-1}F_s||_{\hat{C}^{1,\alpha_0}(\Bnbarre)}  \\
      &+ ||F_{\chi}||_{\hat{C}^{1,\alpha_0}(\Bnbarre)}  + ||(s')^{-1}F_{s'}||_{\hat{C}^{1,\alpha_0}(\Bsbarre)} + ||F_{\chi'}||_{\hat{C}^{1,\alpha_0}(\Bsbarre)}. 
      \end{aligned}
    \end{equation*}

\begin{proof}
\begin{enumerate}
\item First of all, we show that $\thetazero$ given by \eqref{exp:theta} is well-defined . 
\\ When $\rho\to \infty$, we have 
\begin{equation*}
\left(H_\Theta\right)_\rho = O_{\rho\to\infty}(\rho^{-3}). 
\end{equation*}
Therefore, $\thetazero$ is well defined. 
\item Now, we show that $\thetazero$ solves \eqref{theta:linear}. 
\\ Since $H_\theta$ is closed, we have 
\begin{equation*}
d(H_\rho d\rho + H_z dz) = dH_\Theta = 0.
\end{equation*}
Therefore, 
\begin{equation*}
\partial_z \left(H_{\Theta}\right)_\rho = \partial_\rho \left(H_{\Theta}\right)_z. 
\end{equation*}
Now, we compute
\begin{equation*}
\partial_\rho\thetazero = \left(H_{\Theta}\right)_\rho(\rho, z).
\end{equation*}
We apply the dominated convergence theorem to obtain  
\begin{equation*}
\begin{aligned}
\partial_z\thetazero &= -\int_\rho^\infty\, \partial_z\left(H_\Theta\right)_\rho(\tilde\rho, z)\,d\tilde \rho \\
&=  -\int_\rho^\infty\, \partial_\rho\left(H_\Theta\right)_z(\tilde\rho, z)\,d\tilde \rho \\
&= \left(H_\Theta\right)_z. 
\end{aligned}
\end{equation*}
\item Finally, we prove that $\thetazero$  lies in $\Ltheta$, that is 
\begin{equation*}
\left((1 - \xi_N - \xi_S)r^2\thetazero\right)_{\mathbb R^3} \in C^{2, \alpha_0}(\mathbb R^3)\; ,\; \left(\xi_Nr^2\thetazero\right)_{\mathbb R^4} \in C^{2, \alpha_0}(\mathbb R^4) \;\text{and}\;  \left(\xi_Nr^2\thetazero\right)_{\mathbb R^4} \in C^{2, \alpha_0}(\mathbb R^4) . 
\end{equation*}
To this end, we prove the estimates on the different region: away from the boundary of $\Bbarre$, near the axis, the horizon and near the poles. 
\\ First of all, recall the following change of variables 
 \begin{equation*}
\begin{aligned}
x &= \rho\cos\vartheta \\
y &= \rho\sin\vartheta.
\end{aligned}
\end{equation*}
We have 
\begin{equation*}
\partial_x = \frac{x}{\sqrt{x^2 + y^2}}\partial_\rho \;,\; \partial_y = \frac{y}{\sqrt{x^2 + y^2}}\partial_\rho. 
\end{equation*}
\begin{enumerate}
\item We control the $L^\infty$ norm of $r^2\thetazero$: 
\begin{enumerate}
\item  Away from the region $\partial\Bbarre$, say $[\rho_0, \infty[\times\mathbb R$ for some $\rho_0>0$,  there exists $C>0$ such that  
\begin{equation*}
\forall (\rho, z)\in [\rho_0, \infty[\times\mathbb R \;:\; \left|\thetazero(\rho, z)\right| \leq C r^{-2}
\end{equation*}
\noindent Indeed, 
\begin{itemize}
\item if $(\rho, z)\in [\rho_0, \infty[\times\mathbb R\cap B((0, 0), 1)$, then the estimates are straightforward. 
\item Otherwise, we have $\rho^2 + z^2 >1.$ Therefore, 
\begin{equation*}
\begin{aligned}
\left|\thetazero(\rho, z)\right| &\leq C \int_{\rho}^\infty\; r^{-3}(\tilde\rho, z)\,d\tilde\rho \\
&\leq  C \int_{\rho}^\infty\; \frac{1}{(\tilde\rho^2 + z^2)^{\frac{3}{2}}}\,d\tilde\rho \\
&=  C  \int_0^\infty  \frac{1}{((\tau + \rho)^2 + z^2)^{\frac{3}{2}}}\,d\tau \\
&=  C  \int_0^\infty  \frac{1}{(\rho^2 + z^2 + 2\rho\tau +\tau^2)^{\frac{3}{2}}}\,d\tau \\
&\leq C(\rho^2 + z^2)^{-\frac{3}{2}}\int_0^\infty \frac{1}{\left(1 + \left(\frac{\tau}{\sqrt{\rho^2 + z^2}}\right)^2\right)^{\frac{3}{2}}}\,d\tau \\
&= C(\rho^2 + z^2) \int_0^\infty\,\frac{1}{(1 + u^2)^{\frac{3}{2}}}\,du \\
 &\leq C r^{-2}(\rho, z).  
\end{aligned}
\end{equation*}
\end{itemize}
Therefore, we obtain the  estimate for $r^2\thetazero$. 
\item In $\tilde\Axis\subset\BAbarre$, a neighbourhood of the axis or in $\tilde\Horizon\subset\BHbarre$, a neighbourhood of the horizon, we have 
\begin{equation*}
\begin{aligned}
\left|\thetazero(\rho, z)\right| &= \left|-\int_\rho^\infty\, \left(H_\Theta\right)_\rho(\tilde\rho, z)\,d\tilde \rho \right| \\
&\leq C  \int_{\rho}^\infty\; \frac{\tilde\rho}{1 + \tilde \rho} r^{-3}(\tilde\rho, z)\,d\tilde\rho \\
&\leq C  \int_{\rho}^\infty\; 2\tilde\rho\left(1 + \tilde \rho^2 + z^2 \right)^{-\frac{3}{2}}\,d\tilde\rho \\
\end{aligned}
\end{equation*}
\item In $\Bnbarre$, we write \eqref{theta:linear} in the $(s, \chi)$ coordinate system. We have 
\begin{equation*}
\partial_s\thetazero = (H_\Theta)_s \quad\text{;}\quad \partial_\chi\thetazero = (H_\Theta)_\chi. 
\end{equation*}
and 
\begin{equation*}
\thetazero(s, \chi) = \int_0^s\; (H_\Theta)_s(\tilde s, \chi)\,d\tilde s. 
\end{equation*}
We have 
\begin{equation*}
\begin{aligned}
\left| (r^2\xi_N\thetazero)_{\mathbb R^4}\right| \leq C \int_0^s\tilde sd\tilde s \leq C. 
\end{aligned}
\end{equation*}
Here, we used the estimate of $s^{-1}(H_\Theta)_s$ in the region $\Bnbarre$. 
\end{enumerate}

\item We control the $L^\infty$ norm of $\nabla_{\mathbb R^3}(r^2(1 - \xi_N - \xi_S)\thetazero)_{\mathbb R^3}$, $\nabla_{\mathbb R^4}(r^2\xi_N\thetazero)_{\mathbb R^4}$ and $\nabla_{\mathbb R^4}(r^2\xi_S\thetazero)_{\mathbb R^4}$
\\ We have 
\begin{equation*}
\begin{aligned}
\partial_x (r^2\thetazero)_{\mathbb R^3} &= \frac{x}{\sqrt{x^2 + y^2}}\partial_\rho\left(r^2(\rho, z)\thetazero(\rho, z)\right), \\
\partial_y (r^2\thetazero)_{\mathbb R^3} &= \frac{y}{\sqrt{x^2 + y^2}}\partial_\rho\left(r^2(\rho, z)\thetazero(\rho, z)\right), \\
\partial_z (r^2\thetazero)_{\mathbb R^3} &= \partial_z\left(r^2(\rho, z)\thetazero(\rho, z)\right). 
\end{aligned}
\end{equation*}
\begin{enumerate}
\item Away from the boundary $\left\{0\right\}\times\mathbb R$, the estimates are straightforward since $H_\Theta\in \Ntheta$.
\item In $\tilde\Axis\cup\tilde\Horizon,$ we have 
\begin{equation*}
\begin{aligned}
\partial_x (r^2\thetazero)_{\mathbb R^3} &= \frac{x}{\sqrt{x^2 + y^2}}\partial_\rho\left(r^2(\rho, z)\thetazero(\rho, z)\right) \\
&= x\thetazero(\rho, z) + r^2(\rho, z)\frac{x}{\sqrt{x^2 + y^2}}\partial_\rho\thetazero. 
\end{aligned}
\end{equation*}
The first term is easily bounded by $r^{-2}$ thanks to (a). As for the second term, we have 
\begin{equation*}
\begin{aligned}
\left| r^2(\rho, z)\frac{x}{\sqrt{x^2 + y^2}}\partial_\rho\thetazero\right| &= r^2(\rho, z)\frac{x}{\sqrt{x^2 + y^2}}\left|(H_\Theta)_\rho\right| \\ 
&\leq C r^{-1}\frac{\rho}{1 + \rho} \frac{x}{\sqrt{x^2 + y^2}} \\
&\leq C. 
\end{aligned}
\end{equation*}
We control $\partial_y (r^2\thetazero)_{\mathbb R^3}$ in a similar way. As for $\partial_z (r^2\thetazero)_{\mathbb R^3} $, it is straightforward. 
\item In $\Bnbarre$, we have 
\begin{equation*}
\partial_{s}(r^2\thetazero)_{\mathbb R^4} = \partial_{s}r^2\thetazero(s, \chi) +  r^2(H_\Theta)_s
\end{equation*}
The previous point and the estimate for $H_{\Theta}$ yield the result. 
\end{enumerate}
\item Now, we control the $L^\infty$ norm of $\left(\nabla^2_{\mathbb R^3}r^2\thetazero\right)$. 
\begin{enumerate}
\item Away from the boundary $\left\{0\right\}\times\mathbb R$, the estimates are straightforward since $H_\Theta\in \Ntheta$. 
\item In $\tilde\Axis\cup\tilde\Horizon,$ we have 
\begin{equation*}
\begin{aligned}
\partial_{xx}\left(r^2\thetazero\right)_{\mathbb R^3} &= \thetazero + \frac{x^2}{\sqrt{x^2 + y^2}}\partial_\rho\thetazero + x\partial_x\left( \frac{r^2}{\rho}(H_\Theta)_\rho\right) + \frac{r^2}{\rho}(H_\Theta)_\rho. 
\end{aligned}
\end{equation*}
The above terms are controlled using the previous estimates and the $L^\infty$ estimate of $\displaystyle \nabla_{\mathbb R^3}\left(r^3\left(1 + \frac{1}{\rho}\right)H_\Theta\right)$. 
\item In $\Bnbarre$, we have 
\begin{equation*}
\partial_{ss}(r^2\thetazero)_{\mathbb R^4} = \partial_{ss}r^2\thetazero(s, \chi) +  r^2\partial_s(H_\Theta) + 2\partial_sr^2(H_\Theta)_s.
\end{equation*}
The estimate is straigthforward. 
\end{enumerate}
\item Finally, we control 
\begin{equation*}
\sup_{X_1\ne X_2} \frac{\left|\left|\left(\nabla^2_{\mathbb R^3}r^2\thetazero\right)(X_1) - \left(\nabla^2_{\mathbb R^3}r^2\thetazero\right)(X_2) \right|\right|}{\left|X_1 - X_2\right|^{\alpha_0}} \leq C r^{-2}. 
\end{equation*}
The estimates are straightforward thanks to the control of $H_\Theta$ and the regularity of the different terms.
\end{enumerate}

\end{enumerate}
\end{proof}

\subsubsection{Non-linear estimates}
We apply Theorem \ref{Fixed::Point::2} in order to obtain
\begin{Propo}
\label{non:linear:theta}
Let $\alpha_0\in(0, 1)$ and let $\overline\delta_0>0$. Then, there exists $0<\delta_0 \leq \overline\delta_0$ such that $\displaystyle \forall \left(\lambdazero, \delta\right)\in B_{\delta_0}\left(\Llambda\times[0, \infty[\right)$ there exists a unique one parameter family $\left(\sigmazero, \left(0, 0, B_z^{(A)}\right), \left(\Xzero, \Yzero \right), \thetazero \right)\left(\lambdazero, \delta\right)\in B_{\delta_0} $ $\left(\Lsigma\times\LB\times \LX\times\LY\times\Ltheta\right)$ which solves \eqref{sigmazero}, \eqref{eq:for:B}, \eqref{XYzero} and \eqref{thetazero} and which satisfies
\begin{equation*}
        \left|\left|\left(\sigmazero, \left(0, 0, B_z^{(A)}\right), \left(\Xzero, \Yzero\right), \thetazero\right)\left(\lambdazero, \delta\right)\right|\right|_{\Lsigma\times\LB\times\LX\times\LY\times\Ltheta} \le C(\alpha_0)\left( ||\lambdazero||^2_{\Llambda} + \delta\right)
  \end{equation*}
  and $\displaystyle \forall \left(\lambdazero_i, \delta_i\right)\in B_{\delta_0}\left(\Llambda\times[0, \infty[\right)$,
        \begin{equation*}
        \begin{aligned}
        &\left|\left|\left(\sigmazero, \left(0, 0, B_z^{(A)}\right), \left(\Xzero, \Yzero \right), \thetazero\right)\left(\lambdazero_1, \delta_1\right) - \left(\sigmazero, \left(0, 0, B_z^{(A)}\right), \left(\Xzero, \Yzero \right), \thetazero\right)\left(\lambdazero_2, \delta_2\right)\right|\right|_{\Lsigma\times \LB\times\LX\times\LY\times\Ltheta} \\ 
        &\le C(\alpha_0)\left(\left(\left|\left|\lambdazero_1\right|\right|_{\Llambda} + \left|\left|\lambdazero_2\right|\right|_{\Llambda}\right)\left|\left|\lambdazero_1 - \lambdazero_2\right|\right|_{\Llambda}  + \left|\delta_1 - \delta_2\right| \right).
        \end{aligned}
        \end{equation*}
\end{Propo}
\noindent Before we prove the above proposition, we introduce the following notations
\begin{enumerate}
\item Define $H_\Theta$ on $\Bbarre$ to be the one form which expression in the $(\rho, z)$ coordinates is 
   \begin{equation*}
        \begin{aligned}
       (H_\Theta)_\rho&= -\frac{\rho}{X_K^2}\frac{(1 + \sigmazero)}{(1 + \Xzero)^2}(\partial_z(Y_K + X_K\Yzero)+B_z)+\frac{\rho}{X_K^2}\partial_zY_K, \\
       (H_\Theta)_z &= \frac{\rho}{X_K^2}\frac{(1 + \sigmazero)}{(1 + \Xzero)^2}(\partial_{\rho}(Y_K + X_K\Yzero)+B_{\rho})-\frac{\rho}{X_K^2}\partial_{\rho}Y_K.
        \end{aligned}
    \end{equation*}
\item Let $\overline\delta_0>0$ be obtained by Proposition \ref{non:linear:XY}. Define the nonlinear operator $N_\Theta$ on  $B_{\overline\delta_0}(\Ltheta)\times B_{\overline\delta_0}(\Llambda\times[0, \infty[)$ by
\begin{equation*}
\begin{aligned}
\left[N_\Theta\left(\thetazero, \lambdazero; \delta)\right)\right]_\rho(\rho, z) :=  -\frac{\rho}{X_K^2}\frac{(1 + \sigmazero(\thetazero, \lambdazero; \delta))}{(1 + \Xzero(\thetazero, \lambdazero; \delta))^2}(\partial_z(Y_K + X_K\Yzero(\thetazero, \lambdazero; \delta))+B_z(\thetazero, \lambdazero; \delta))+\frac{\rho}{X_K^2}\partial_zY_K \\
\left[N_\Theta\left(\thetazero, \lambdazero; \delta)\right)\right]_z(\rho, z) := \frac{\rho}{X_K^2}\frac{(1 + \sigmazero(\thetazero, \lambdazero; \delta))}{(1 + \Xzero(\thetazero, \lambdazero; \delta))^2}(\partial_{\rho}(Y_K + X_K\Yzero(\thetazero, \lambdazero; \delta))+B_{\rho}(\thetazero, \lambdazero; \delta))-\frac{\rho}{X_K^2}\partial_{\rho}Y_K \\
\end{aligned}
\end{equation*} 
and set 
\begin{equation*}
\begin{aligned}
\left[N_\Theta\left(\thetazero, \lambdazero; \delta)\right)\right] &= \left[N_\Theta\left(\thetazero, \lambdazero; \delta)\right)\right]_\rho d\rho  + \left[N_\Theta\left(\thetazero, \lambdazero; \delta)\right)\right]_z dz,
\end{aligned}
\end{equation*}
where $(\sigmazero, B, \Xzero, \Yzero)$ are the solution mappings obtained by Proposition \ref{non:linear:XY}. 
\end{enumerate}
\noindent In order to prove the above proposition, we will need the following lemma
\begin{lemma}
\label{est:sigma}
\begin{enumerate}
\item There exists $\overline\delta_0>0$ such that  $\displaystyle N_\Theta\left(B_{\overline\delta_0}(\Ltheta)\times B_{\overline\delta_0}(\Llambda\times[0, \infty[)\right)\subset \Ntheta$. 
\item There exist $0<\delta_0\leq\overline \delta_0$ and $C(\alpha_0)>0$ such that $\forall \left(\thetazero, \lambdazero; \delta\right)\in B_{\overline\delta_0}(\Ltheta)\times B_{\overline\delta_0}(\Llambda\times[0, \infty[)$
\begin{equation*}
||H_{\Theta} ||_{\Nsigma} \leq C(\alpha_0)\left( \left|\left|\thetazero\right|\right|^2_{\Ltheta} + ||\lambdazero||^2_{\Llambda} + \delta \right). 
\end{equation*}
\end{enumerate}
\end{lemma}

\begin{proof}
\begin{enumerate}
\item Let $\overline\delta_0>0$ be obtained by Proposition \ref{non:linear:XY} and let $(\thetazero, \lambdazero;\delta)\in (B_{\overline\delta_0}(\Ltheta)\times B_{\overline\delta_0}(\Llambda\times[0, \infty[).$ First of all, we show that $\displaystyle N_\Theta\left(\thetazero, \lambdazero; \delta\right)$ is a differentiable closed one-form on $\Bbarre$.
\begin{itemize}
\item For the differentiability, the only terms that we need to analyse are 
\begin{equation*}
\frac{\rho}{X_K^2}\partial_zY_K\quad\text{,}\quad \frac{\rho}{X_K^2}\partial_\rho Y_K\quad\text{and}\quad\frac{\rho}{X_K^2}. 
\end{equation*}
The differentiability of the remaining terms as well as the differentiability away from the $\partial\Bbarre$ follow because $(\sigmazero, B, \Xzero, \Yzero)(\thetazero, \lambdazero; \lambda)$ is differentiable.   
\item Now, we check that $\displaystyle  N_\Theta\left(\thetazero, \lambdazero; \delta\right)$ is closed. 
\begin{itemize}
\item $(\Yzero, B)(\thetazero, \lambdazero; \delta)$ solve \eqref{XYzero} and \eqref{Bzero}. We set
\begin{equation*}
\theta := dY + B \quad\text{and}\quad \theta_K := dY_K. 
\end{equation*}
\item Therefore, 
\begin{equation*}
N_\Theta\left(\thetazero, \lambdazero; \delta)\right) = -\frac{\sigma}{X^2}\left(\theta_z d\rho - \theta_\rho dz\right) + \frac{\rho}{X_K^2}\left(\theta_z^K - \theta_\rho^K dz\right). 
\end{equation*}
\item Moreover, $\theta$ verifies
\begin{equation*}
\sigma^{-1}\partial_{\rho}(\sigma\theta_{\rho}) + \sigma^{-1}\partial_{z}(\sigma\theta_{z}) = \frac{2\theta_{\rho}\partial_{\rho}X + 2\theta_{z}\partial_{z}X}{X}
\end{equation*}
\item Straightforward computations and the used of the above equation imply that the one forms $\displaystyle \frac{\sigma}{X^2}\left(\theta_z d\rho - \theta_\rho dz\right)$ and $\displaystyle  \frac{\rho}{X_K^2}\left(\theta_z^K - \theta_\rho^K dz\right)$ are closed. 
\item Hence, $\displaystyle N_\Theta\left(\thetazero, \lambdazero; \delta)\right) $ is closed. 
\end{itemize}
\end{itemize} 
\item Let $\displaystyle (\thetazero, \lambdazero;\delta)\in (B_{\overline\delta_0}(\Ltheta)\times B_{\overline\delta_0}(\Llambda\times[0, \infty[).$ We show that $N_\Theta(\thetazero, \lambdazero;\delta)$ provided $\delta$ sufficiently small. To this end, we show the estimates in the region $\BAbarre\cup\BHbarre$, then in $\Bnbarre$ and $\Bsbarre$. In order to lighten the expressions, we omit the dependence of $\sigmazero, \Xzero, \Yzero, B$ on the quantities $\displaystyle (\thetazero, \lambdazero;\delta)$. 
We re-write $\left(N_\Theta\right)_\rho$ and $\left(N_\Theta\right)_z$ on the form
\begin{equation*}
\begin{aligned}
\left(N_\Theta\right)_\rho &= -\frac{\rho}{X_K^2}\frac{(1 + \sigmazero)}{(1 + \Xzero)^2}(\partial_z(X_K\Yzero + Y_K))+B_z)+\frac{\rho}{X_K^2}\partial_zY_K  \\
&= \frac{\rho}{X_K^2}\frac{- 1 - \sigmazero + (1 + \Xzero)^2}{(1 + \Xzero^2)^2} \partial_zY_K - \frac{\rho}{X_K}\frac{(1 + \sigmazero)}{(1 + \Xzero)^2}B_z - \frac{\sigma}{X_K^2(1 + \Xzero)^2}\partial_z(X_K^2(X_K^{-1}\Yzero)) \\
&= \frac{\rho}{X_K^2}\frac{- \sigmazero + \Xzero^2 + 2\Xzero}{(1 + \Xzero^2)^2} \partial_zY_K - \frac{\rho}{X_K}\frac{(1 + \sigmazero)}{(1 + \Xzero)^2}B_z - \frac{\rho(1 + \sigmazero)}{(1 + \Xzero)^2}\partial_z(X_K^{-1}\Yzero)  \\
&- 2 \frac{\rho(1 + \sigmazero)}{(1 + \Xzero)^2}X_K^{-1}\Yzero\partial_z\log X_K.
\end{aligned}
\end{equation*}
\begin{equation*}
\begin{aligned}
\left(N_\Theta\right)_z &= \frac{\rho}{X_K^2}\frac{- \sigmazero + \Xzero^2 + 2\Xzero}{(1 + \Xzero^2)^2} \partial_\rho Y_K - \frac{\rho}{X_K}\frac{(1 + \sigmazero)}{(1 + \Xzero)^2}B_\rho - \frac{\rho(1 + \sigmazero)}{(1 + \Xzero)^2}\partial_\rho(X_K^{-1}\Yzero)  \\
&- 2 \frac{\rho(1 + \sigmazero)}{(1 + \Xzero)^2}X_K^{-1}\Yzero\partial_\rho\log X_K.
\end{aligned}
\end{equation*}

\begin{itemize}
\item In $\BAbarre\cup \BHbarre$, we prove that $\displaystyle r^3(1 + \rho^{-1})\left(N_\Theta\right)_\rho$ and $r^3(N_\Theta)_z$ are bounded in $\hat C^{1, \alpha_0}$. We have 
\begin{itemize}
\item By Lemma \ref{decay:estimates}, 
\begin{equation*}
    \frac{1}{X^2_K}|\partial_z Y_K| \le Cr^{-5}, \quad |\partial\left( X^{-2}_K\partial Y_K\right)| \le Cr^{-6}, \quad |\partial^2\left( X^{-2}_K\partial Y_K\right)| \le Cr^{-7}
\end{equation*} 
and 
\begin{equation*}
    \frac{\rho}{X^2_K}|\partial Y_K| \le Cr^{-4}, \quad |\partial\left( X^{-1}_K\partial Y_K\right)| \le Cr^{-4}, \quad |\partial^2\left( X^{-1}_K\partial Y_K\right)| \le Cr^{-5}.
\end{equation*}

Thus, the term $\displaystyle \frac{\rho}{X_K^2}\partial_z Y_K $ is bounded by $\displaystyle r^3(1 + \rho^{-1})\frac{\rho}{X_K^2}\partial_z Y_K$ is bounded in $\hat C^{1, \alpha_0}$.
\item By Lemma \ref{asymptotics:h2}, the term $\partial_z(\log X_K)\in C^\infty(\Bbarre)$ and satisfies
\begin{equation*}
\left|\partial_z\log X_K\right| = O_{r\to\infty}(r^{-2}). 
\end{equation*}
\item Moreover,  the quantities $\displaystyle r^3 X_K^{-1}\Yzero$, $\displaystyle r^4 X_K^{-1}\Yzero$ and $\displaystyle r^5 X_K^{-1}\Yzero$  are bounded in $L^\infty(\Bbarre)$, which implies  that the term $\displaystyle  r^3(1 + \rho^{-1})X_K^{-1}\Yzero$ is bounded in $\hat C^{1, \alpha_0}$
\item Finally, recall that we control $\displaystyle \frac{(1 + \rho^{10})(1 + r^{10})}{\rho^{10}}B_z$ in $\hat C^{1, \alpha_0}$. 
\item This yields to estimates in the region $\BAbarre\cup\BHbarre$. 
\end{itemize}
\item In the region $\Bnbarre$, we write $N_\Theta$ in $(s, \chi)$ coordinates 
\begin{itemize}
\item We have 
\begin{equation*}
\begin{aligned}
N_\Theta &= (N_\Theta)_\rho d\rho + (N_\Theta)_z dz \\
&=  (\chi(N_\Theta)_\rho(s, \chi) -s(N_\Theta)_z(s, \chi)) ds +  (s(N_\Theta)_\rho(s, \chi) + \chi(N_\Theta)_z(s, \chi)) d\chi.
\end{aligned}
\end{equation*}
\item Moreover, 
\begin{equation*}
    \partial_{\rho} = \frac{\chi}{\chi^2+s^2}\partial_s + \frac{s}{\chi^2 + s^2}\partial_\chi, \quad\quad  \partial_z = \frac{-s}{\chi^2+s^2}\partial_s + \frac{\chi}{\chi^2+s^2}\partial_\chi.
\end{equation*}
\item Therefore, 
\begin{equation*}
\begin{aligned}
(N_\Theta)_s &= \chi\left(-\frac{\sigma}{X^2}\left( \frac{-s}{\chi^2+s^2}\partial_sY + \frac{\chi}{\chi^2+s^2} \partial_\chi Y \right) \right) - s\left(  \frac{\sigma}{X^2}\left( \frac{\chi}{\chi^2+s^2}\partial_s Y + \frac{s}{\chi^2 + s^2}\partial_\chi Y\right)  \right)  \\
&- \frac{\sigma}{X^2}B_\chi  \\
&= \frac{s\chi}{X_K^2}\frac{- \sigmazero + \Xzero^2 + 2\Xzero}{(1 + \Xzero^2)^2} \partial_\chi Y_K  - \frac{s\chi(1 + \sigmazero)}{(1 + \Xzero)^2}\partial_\chi(X_K^{-1}\Yzero) - 2 \frac{s\chi(1 + \sigmazero)}{(1 + \Xzero)^2}X_K^{-1}\Yzero\partial_\chi\log X_K  \\
&- \frac{\sigma}{X^2}B_\chi. 
\end{aligned}
\end{equation*}
In the same manner, we obtain
\begin{equation*}
(N_\Theta)_\chi = \frac{s\chi}{X_K^2}\frac{- \sigmazero + \Xzero^2 + 2\Xzero}{(1 + \Xzero^2)^2} \partial_s Y_K  - \frac{s\chi(1 + \sigmazero)}{(1 + \Xzero)^2}\partial_s(X_K^{-1}\Yzero) - 2 \frac{s\chi(1 + \sigmazero)}{(1 + \Xzero)^2}X_K^{-1}\Yzero\partial_s\log X_K - \frac{\sigma}{X^2}B_s. 
\end{equation*}
\end{itemize}
\item Now, we estimate $s^{-1}(N_\Theta)_s$ in $\displaystyle \hat C^{1, \alpha_0}$: 
\begin{itemize}
\item We start with $\displaystyle \frac{s\chi\partial_\chi Y_K}{X_K^2}$. We recall that $Y_K$ 
\begin{equation*}
dY_K = \theta_K
\end{equation*}
which takes the following form in the $(s, \chi)$ coordinates
\begin{equation*}
\partial_s(X_K^{-1}W_K) ds + \partial_\chi (X_K^{-1}W_K) d\chi = \frac{s\chi}{X_K^2}\left(\partial_\chi Y_K - \partial_s Y_K\right). 
\end{equation*} 
Moreover, we recall the $(x, y, u, v)$ coordinates defined by \eqref{xy::uv::} so that any function  $f(s, \chi)$ defined on $\Bnbarre$ can be see as a function $f_{\mathbb R^4}(x, y, u, v)$ defined on $\mathbb R^4$
Therefore, 
\begin{equation*}
\begin{aligned}
\frac{s\chi}{X_K^2}\partial_\chi Y_K &= \partial_s(X_K^{-1}W_K)  \\
&= s\left.\partial_{xx}\left(\frac{W_K}{X_K}\right)\right|_{(0, \sqrt{x^2 + y^2}, u, v)}. 
\end{aligned}
\end{equation*}
Now, we recall that $\displaystyle \frac{W_K}{X_K}$ is smooth on $\Bbarre$. This yields the $\hat C^{1, \alpha_0}$ estimates of $\displaystyle \frac{\chi}{X_K^2}\partial_\chi Y_K$. 
\item The remaining  terms are easily controlled in $\hat C^{1, \alpha_0}$ due to the estimates for $X_K^{-1}\Yzero$ and $B$. 
\item Finally, the estimates for $(N_{\Theta})_\chi$ follow similarly. 
\end{itemize}
\item In the region $\Bsbarre$, the estimates are obtained as in $\Bnbarre$. 
\end{itemize}
\end{enumerate}
\end{proof}

\subsection{Solving for $\lambda$}
In this section, we prove the following result 

\begin{Propo}
\label{non:linear:lambda}
Let $\alpha_0\in(0, 1)$ and let $\overline\delta_0>0$. Then, there exists $0<\delta_0 \leq \overline\delta_0$ such that $\displaystyle \forall \delta\in[0, \delta_0[$ there exists a unique one parameter family $\left(\sigmazero, \left(0, 0, B_z^{(A)}\right), \left(\Xzero, \Yzero \right), \thetazero, \lambdazero \right)\left(\delta\right)\in B_{\delta_0} $ $\left(\Lsigma\times\LB\times \LX\times\LY\times\Ltheta\times \Llambda\right)$ which solves \eqref{sigmazero}, \eqref{eq:for:B}, \eqref{XYzero}, \eqref{thetazero}, \eqref{lambdazero} and which satisfies
\begin{equation*}
        \left|\left|\left(\sigmazero, \left(0, 0, B_z^{(A)}\right), \left(\Xzero, \Yzero\right), \thetazero, \lambdazero\right)\left(\delta\right)\right|\right|_{\Lsigma\times\LB\times\LX\times\LY\times\Ltheta\times\Llambda} \le C(\alpha_0)\delta
  \end{equation*}
  and $\displaystyle \delta_i\in [0, \delta_0[$,
        \begin{equation*}
        \begin{aligned}
        &\left|\left|\left(\sigmazero, \left(0, 0, B_z^{(A)}\right), \left(\Xzero, \Yzero \right), \thetazero, \lambdazero\right)\left(\delta_1\right) - \left(\sigmazero, \left(0, 0, B_z^{(A)}\right), \left(\Xzero, \Yzero \right), \thetazero, \lambdazero\right)\left(\delta_2\right)\right|\right|_{\Lsigma\times \LB\times\LX\times\LY\times\Ltheta\times\Llambda} \\ 
        &\le C(\alpha_0) \left|\delta_1 - \delta_2\right|. 
        \end{aligned}
        \end{equation*}
\end{Propo}
First of all, we recall the following Theorem 1.2 from \cite{chodosh2015stationary}

\begin{theoreme}[O.CHODOSH, Y.SHLAPENTOKH-ROTHMAN]
\label{Reduced:system:B}
Assume that the metric data $(X, W, \theta, \sigma, \lambda)$  is chosen. Assume the energy momentum tensor $\mathbb T$ is chosen so that it satisfies \eqref{symmetry:T}. Suppose that 
\begin{enumerate}
\item $\T{\alpha}{\beta}\in C^{1, \alpha}_{loc}(\BB)$, $X\in C^{2, \alpha}_{loc}(\BB)$, $W\in C^{1, \alpha}_{loc}(\BB)$, $\theta\in C^{1, \alpha}_{loc}(\BB)$, $\sigma\in C^{3, \alpha}_{loc}(\BB)$ and $\lambda\in C^{1, \alpha}_{loc}(\BB)$, 
\item $X, W, \theta, \sigma$ satisfy their respective equations on $\BB$, as listed in Theorem \ref{Yakov:Otis},
\item $|d\sigma|\neq 0$ on $\BB$. 
\item If we form the metric $g$  from \eqref{ansatz:metric}, then $\mathbb T$ is divergence free with respect to $g$.  
\end{enumerate}
Then, the $1-$form arising in Equation \eqref{eq:lambda:bis} satisfies the following compatibility condition 
\begin{equation}
\label{compatible::lambda}
d\alpha = (\beta_2)\wedge\left( d\lambda - \alpha - \frac{1}{2}d\log X \right), 
\end{equation}
where $\beta = (\beta_2)_\rho d\rho + (\beta_2)_z dz$ is a one-form which depends on my on $\sigma$ through the equations: 
\begin{equation}
\label{beta:1:form}
\begin{aligned}
(\beta_2)_z &= \frac{1}{2}\frac{(\partial_\rho \sigma)(\partial^3_{\rho\rho z}\sigma - \partial^3_z\sigma) + (\partial_z\sigma)(\partial_\rho^3\sigma - \partial_{\rho z z}^3\sigma + 2\partial^2_\rho\sigma + 2\partial_z^2\sigma)}{(\partial_\rho\sigma)^2 + (\partial_z\sigma)^2} \\
(\beta_2)_\rho &= \frac{1}{2}\frac{(\partial_z \sigma)(\partial^3_{\rho zz}\sigma - \partial^3_\rho\sigma) + (\partial_\rho\sigma)(\partial_z^3\sigma - \partial_{\rho\rho z}^3\sigma + 2\partial^2_\rho\sigma + 2\partial_z^2\sigma)}{(\partial_\rho\sigma)^2 + (\partial_z\sigma)^2}
\end{aligned}
\end{equation}
Furthermore, under the above hypothesis, if $\lambda$ satisfies its first order equation \eqref{eq:lambda:bis}, then we automatically have $\lambda\in C_{loc}^{2, \alpha}(\BB)$ and that $\lambda$ satisfies the second order equation given in Theorem \ref{Yakov:Otis}. 
\end{theoreme}
\noindent Now we recall that $\lambdazero$ satisfies the following equations 

\begin{equation}
\label{lambdazero:bis}
        \begin{aligned}
        \partial_{\rho}\overset{\circ}{\lambda} &= \alpha_{\rho}-(\alpha_K)_{\rho}-\frac{1}{2}\partial_{\rho}\log(1+\overset{\circ}{X}), \\
        \partial_{z}\overset{\circ}{\lambda} &= \alpha_{z}-(\alpha_K)_{z}-\frac{1}{2}\partial_{z}\log(1+\overset{\circ}{X}),
        \end{aligned}
    \end{equation}
    where 
    \begin{equation*}
        \begin{aligned}
        (\alpha_K)_{\rho} &=\frac{1}{4}\rho X_K^{-2}((\partial_{\rho}X_K)^2-(\partial_{z}X_K)^2 + (\partial_{\rho}Y_K)^2- (\partial_{z}Y_K)^2), \\
        (\alpha_K)_{z} &=\frac{1}{4}\rho X_K^{-2}((\partial_{\rho}X_K)(\partial_{z}X_K)+ (\partial_{\rho}Y_K)(\partial_{z}Y_K)),
        \end{aligned}
    \end{equation*}
where $\alpha_\rho$ and $\alpha_z$ satisfy 
\begin{equation}
\label{alpha:::rho}
\begin{aligned}
 ((\partial_{\rho}\sigma)^2 + (\partial_{z}\sigma)^2)\alpha_{\rho} &= \frac{1}{4} (\partial_{\rho}\sigma)\sigma\frac{(\partial_{\rho}X)^2-(\partial_{z}X)^2+(\theta_{\rho})^2-(\theta_{z})^2}{X^2} + (\partial_{\rho}\sigma)(\partial^2_{\rho}\sigma - \partial^2_{z}\sigma) + (\partial_{z}\sigma)((\partial^2_{\rho,z}\sigma)) \\
    &+ \frac{1}{2}X^{-2}((\partial_{\rho}X)(\partial_{z}X) + (\theta_{\rho})(\theta_{z}))),
\end{aligned}
\end{equation}

\begin{equation}
\label{alpha:::z}
\begin{aligned}
((\partial_{\rho}\sigma)^2 + (\partial_{z}\sigma)^2)\alpha_{z} &= -\frac{1}{4} (\partial_{z}\sigma)\sigma\frac{(\partial_{\rho}X)^2-(\partial_{z}X)^2+(\theta_{\rho})^2-(\theta_{z})^2}{X^2}- (\partial_{z}\sigma)(\partial^2_{\rho}\sigma - \partial^2_{z}\sigma) + (\partial_{\rho}\sigma)((\partial^2_{\rho,z}\sigma)) \\
 &+ \frac{1}{2}X^{-2}((\partial_{\rho}X)(\partial_{z}X) + (\theta_{\rho})(\theta_{z}))
\end{aligned}
\end{equation}
Moreover $\lambda = \lambdazero + \lambda_K$ satisfies the second order equation 

\begin{equation}
\label{eq:2:lambda}
 \begin{aligned}
        2\partial^2_{\rho}\lambda+2\partial^2_{z}\lambda &=  -\partial^2_{\rho}\log{X}-\partial^2_{z}\log{X} + \sigma^{-1}((\partial_{\rho}\sigma)^2 + (\partial_{z}\sigma)^2) +F_4(\thetazero, \Xzero, \sigmazero, \lambdazero)(\rho, z) \\
        &-\frac{1}{2}X^{-2}((\partial_{\rho}X)^2+(\partial_{z}X)^2 +(\theta_{\rho})^2+(\theta_{z})^2),
    \end{aligned}
\end{equation}
 \noindent Note that if we try to solve $\lambdazero$ by integrating directly equations \eqref{lambdazero:bis}, then we do not know a posteriori whether $\alpha$ satisfies the compatibility condition \eqref{compatible::lambda} and thus $d\alpha = 0$. 
\begin{enumerate}
\item In this case, we proceed as in \cite{chodosh2017time} and we solve the following system of equations
\begin{equation}
\label{lambdazero:modif}
        \begin{aligned}
        \partial_{\rho}\overset{\circ}{\lambda} &= \alpha_{\rho} - (\alpha_K)_{\rho}-\frac{1}{2}\partial_{\rho}\log(1+\overset{\circ}{X}), \\
        \partial_{z}\overset{\circ}{\lambda} &= \alpha_{z} - (\alpha_K)_{z}-\frac{1}{2}\partial_{z}\log(1+\overset{\circ}{X}) \\
        &-\int_\rho^\infty\left(\beta_2\wedge\left(d\lambdazero - (\alpha - \alpha_K) - \frac{1}{2}d\log(1 + \Xzero) \right) \right)_{\rho, z}(\tau, z)\,d\tau. 
        \end{aligned}
    \end{equation}

\item Next, we show that the solution map verifies the original system of equations: \eqref{lambdazero:bis}. 
\item Finally, we apply Theorem \ref{Reduced:system:B} in order to obtain that $\lambdazero$ satisfies the second order equation  \eqref{eq:2:lambda}. 
\end{enumerate}
\noindent The remaining of this section is devoted to the proof of the above steps. This will allow us to obtain Proposition \ref{non:linear:lambda}. We also note that we follow the same steps in Section 12 of \cite{chodosh2017time} and we write down details in order to be self-contained.

\subsubsection{Solving the modified equations \eqref{lambdazero:modif}}
We begin by introducing some notations: let $\overline \delta_0>0$ and let $\alpha_0\in(0, 1)$. Let $0<\delta_0\leq \overline\delta_0$ be obtained by Proposition \ref{non:linear:theta}. Let $\displaystyle \mathscr L: B_{\delta_0}(\Llambda)\times[0, \delta_0[\to \Llambda$  be the mapping defined by 
\begin{equation}
\label{L:lambda:}
\mathscr L (\lambdazero, \delta)(\rho, z) := - \int_\rho^\infty\left(\alpha_\rho - (\alpha_K)_\rho - \frac{1}{2}\partial_\rho\log(1 + \Xzero) \right)(\tau, z)\, d\tau. 
\end{equation}
Here, $(\sigma, \Xzero, \thetazero) = (\sigma, \Xzero, \thetazero)(\lambdazero, \delta)$. We omit the dependence on the $(\lambdazero; \delta)$ in order to lighten the expressions. 
\\ Now, we state the following lemma 
\begin{lemma}
\label{lemma::82}
Let $(\sigma, \Xzero, \thetazero)$ be the solution map obtained by Proposition \ref{non:linear:theta}. Then, $\forall(\lambdazero, \delta)\in B_{\delta_0}(\Llambda)\times[0, \delta_0[$, the one-form $\alpha$ defined by \eqref{alpha:::rho} and \eqref{alpha:::z} satisfies 
\begin{equation*}
d\alpha = \beta_2\wedge\left(d\lambda - \alpha - \frac{1}{2}d\log X\right). 
\end{equation*}
where $\beta_2$ is defined by \eqref{beta:1:form}. 
\end{lemma}
\begin{proof}
We check the assumptions of Theorem \ref{Reduced:system:B}
\begin{enumerate}
\item The regularity assumptions are satisfied since the quantities $\displaystyle \left(\sigmazero, B, \Xzero, \Yzero, \thetazero\right)(\lambdazero, \delta)$ as well as $\lambdazero$ lie in the right spaces. 
\item $\displaystyle \left(\sigmazero, B, \Xzero, \Yzero, \thetazero\right)(\lambdazero, \delta)$ solve their respective equations on $\Bbarre$. Therefore, the original metric data $(X, W, \theta, \sigma)(\lambdazero, \delta)$ solve their respective equations on $\BB$. 
\item Since $\sigma_K = \rho, $ $d\sigma_K = d\rho$ on $\BB$. Therefore $|d\sigma_K| \neq 0$. Now, we choose $\delta_0>0$ sufficiently small so that $|d\sigma|\neq 0$. Indeed,  
\begin{equation*}
\begin{aligned}
\left|d\sigma\right|^2 &= \left|d\left(\rho(1 + \sigmazero)\right)\right|^2 \\
&= (1 + \sigmazero + \rho\partial_\rho\sigmazero)^2 + (\rho\partial_z\sigmazero)^2. 
\end{aligned}
\end{equation*}
On $\BB$, we have 
\begin{equation*}
\begin{aligned}
\left|d\sigma\right|^2 &\geq (1 + \sigmazero)^2 + 2\rho(\partial_\rho\sigmazero)(1 + \sigmazero) \\
&\geq 1 - 4\delta_0. 
\end{aligned}
\end{equation*}
The latter is obtained by assuming that $\delta_0<1$ and by the control of $L^\infty$ norm of $\rho\partial_\rho\sigmazero$. Therefore, we choose $\delta_0$ so that the latter is positive. 
\item $\T{\alpha}{\beta}$ is divergence free and satisfies \eqref{symmetry:T} (See Section \ref{compu:Tab} ) with respect to $g$ given by \eqref{ansatz:metric}. 
\end{enumerate}
Therefore, we apply Theorem \ref{Reduced:system:B} to obtain that 
\begin{itemize}
\item $\alpha$ satisfies \eqref{compatible::lambda}, 
\item if $\lambda$ satisfies \eqref{lambdazero:bis}, then $\lambda$ also satisfies \eqref{eq:2:lambda}. 
\end{itemize}
\end{proof}
\noindent The remaining of this section is to prove the following result
\begin{lemma}
\label{lemma::83}
Let $\overline \delta_0>0$ and let $\alpha_0\in(0, 1)$. There exists $0<\delta_0\leq \overline\delta_0$ such that $\displaystyle \forall \lambdazero\in B_{\delta_0}(\Llambda)\,;\, \forall \delta\in[0, \delta_0[$, 
\begin{itemize}
\item $\displaystyle \mathscr L (\lambdazero, \delta)$ is well-defined and lies in $\displaystyle B_{\delta_0}(\Llambda)$, 
\item $\displaystyle \mathscr L (\lambdazero, \delta)$ solves \eqref{lambdazero:modif} and verifies
\begin{equation*}
\begin{aligned}
& \left|\left|\mathscr L (\lambdazero, \delta)\right|\right|_{\Llambda} \leq C\left(\left|\left|\lambdazero\right|\right|_{\Llambda}^2 +   \delta\right), \\ 
& \left|\left|\mathscr L (\lambdazero_1, \delta_1) - \mathscr L (\lambdazero_2, \delta_2)\right|\right|_{\Llambda} \leq C\left(\left( \left|\left|\lambdazero_ 1\right|\right|_{\Llambda} +  \left|\left|\lambdazero_2\right|\right|_{\Llambda}\right)\left|\left|\lambdazero_ 1 - \lambdazero_2\right|\right|_{\Llambda} +   |\delta_1 - \delta_2|\right),
\end{aligned}
\end{equation*}
\end{itemize}
\end{lemma}
\begin{proof}
The proof is similar to what has been done before: the estimates are proven in each region $\BAbarre\cup\BHbarre, \Bnbarre, \Bsbarre$. 
\begin{enumerate}
\item First of all, we claim that there exists $C(\alpha_0)>0$ such that 
\begin{equation*}
\left|\left|r^2(\alpha - \alpha_K)\right|\right|_{\hat C^{1, \alpha_0}(\BAbarre\cup\BHbarre)} + \left|\left|r^2d\log(1 + \Xzero)\right|\right|_{\hat C^{1, \alpha_0}(\BAbarre\cup\BHbarre)} \leq C \delta + C \left| \left|\left(\sigmazero, B, (\Xzero, \Yzero), \thetazero\right) \right|\right|_{\Lsigma\times\cdots\Ltheta}
\end{equation*}
\begin{itemize}
\item We start with the term $\alpha - \alpha_K$. We write 
\begin{equation*}
\begin{aligned}
(\alpha - \alpha_K)_\rho &= \frac{1}{4}\left(|\partial \sigma|^2 (\partial_\rho\sigma)\sigma X^{-2}(\partial_\rho X)^2  - \rho X_K^{-2}(\partial_\rho X_K)^2\right)  \\
&+  \frac{1}{4} (\partial_{\rho}\sigma)|\partial\sigma|^{-2}\sigma\frac{-(\partial_{z}X)^2+(\theta_{\rho})^2-(\theta_{z})^2}{X^2} + |\partial\sigma|^{-2}(\partial_{\rho}\sigma)(\partial^2_{\rho}\sigma - \partial^2_{z}\sigma) + |\partial\sigma|^{-2}(\partial_{z}\sigma)((\partial^2_{\rho,z}\sigma)) \\
    &+ \frac{1}{2}|\partial\sigma|^{-2}X^{-2}((\partial_{\rho}X)(\partial_{z}X) + (\theta_{\rho})(\theta_{z}))) -\frac{1}{4}\rho X_K^{-2}(-(\partial_{z}X_K)^2 + (\partial_{\rho}Y_K)^2- (\partial_{z}Y_K)^2)
\end{aligned}
\end{equation*}
Recall from Lemma \ref{asymptotics:h1} and Lemma \ref{asymptotics:h2} that $\displaystyle \partial_\rho \log X_K $ and thus $\displaystyle \partial_\rho \log X$  behave like $\displaystyle \frac{1}{\rho}$ near the axis. 
\\Furthermore, since the renormalised unknowns lie in the right space, the $\hat C^{1, \alpha_0}$ estimates is straightforward for all the terms except the following: 
\begin{equation*}
\begin{aligned}
&\frac{1}{4}\left(|\partial \sigma|^{-2} (\partial_\rho\sigma)\sigma X^{-2}(\partial_\rho X)^2  - \rho X_K^{-2}(\partial_\rho X_K)^2\right) \\
\end{aligned}
\end{equation*}
In order to estimate the latter, we write 
\begin{equation*}
\begin{aligned}
&\frac{1}{4}\left(|\partial \sigma|^{-2} (\partial_\rho\sigma)\sigma X^{-2}(\partial_\rho X)^2  - \rho X_K^{-2}(\partial_\rho X_K)^2\right) \\
&= \frac{1}{4}\left(|\partial \sigma|^{-2} (\partial_\rho\sigma)\sigma(\partial_\rho\log  X)^2  - \rho (\partial_\rho\log X_K)^2\right) \\
&= \frac{1}{4}\left(1 - \rho|\partial\sigma|^{-2}\left(\partial_\rho\sigmazero(1 + \sigmazero)  + \rho|\partial \sigmazero|^2\right) \right)\rho\left(\partial_\rho\log(1 + \Xzero)\right)^2 + \frac{1}{2} \rho\left(\partial_\rho\log X_K \right)\left(\partial_\rho\log(1 + \Xzero) \right) \\
&- \frac{1}{4}|\partial \sigma|^{-2}\left((1 + \sigmazero)\partial_\rho\sigmazero + \rho|\partial\sigmazero|^2 \right)\rho^2\left(\partial_\rho\log X_K \right)^2. 
\end{aligned}
\end{equation*}
Since $\rho\partial_\rho\log X_K$ is bounded and smooth near the axis and all the remaining terms are $C^{1, \alpha}$ controlled. We obtain the desired estimate. 
\item The term $\displaystyle r^2d\log(1 + \Xzero)$ is easily controlled by $\left|\left| \Xzero\right|\right|_{\LX}$ in $\hat C^{1, \alpha_0}(\BAbarre\cup\BHbarre)$. 
\item We apply the dominated convergence theorem and obtain that $\mathscr L(\lambdazero, \delta)$ is well defined and lies in $\hat C^{1, \alpha_0}(\BAbarre\cup\BHbarre)$. Moreover, we have the bounds 
\begin{equation*}
\left|\left|\mathscr L(\lambdazero, \delta) \right|\right|_{\hat C^{1, \alpha_0}(\BAbarre\cup\BHbarre)} \leq C \left| \left|\left(\sigmazero, B, (\Xzero, \Yzero), \thetazero\right) \right|\right|_{\Lsigma\times\cdots\Ltheta}. 
\end{equation*}
\end{itemize}
\item Now, we show that $\mathscr L(\lambdazero, \delta)$ solves \eqref{lambdazero:modif}. We compute, 
\begin{equation*}
\begin{aligned}
\partial_\rho \mathscr L(\lambdazero, \delta) &= \left(\alpha_\rho - (\alpha_K)_\rho - \frac{1}{2}\partial_\rho\log(1 + \Xzero) \right)(\rho, z) \\
\partial_z \mathscr L(\lambdazero, \delta) &= -\int_\rho^\infty \partial_z\left(\alpha_\rho - (\alpha_K)_\rho - \frac{1}{2}\partial_\rho\log(1 + \Xzero) \right)(\tau, z)\,d\tau \\ 
\end{aligned}
\end{equation*}
By Lemma \ref{lemma::82}, $\alpha$ verifies 
\begin{equation*}
\begin{aligned}
d(\alpha - \alpha_K) &= d\alpha \\
&= \beta_2\wedge\left(d \lambda - \alpha - \frac{1}{2}d\log X \right) \\
&= \beta_2\wedge\left(d \lambdazero - d\lambda_K - \alpha - \frac{1}{2}d\log X_K - \frac{1}{2}d\log (1 + \Xzero)\right). 
\end{aligned}
\end{equation*}
We recall that 
\begin{equation*}
d\lambda_K = \alpha_K - \frac{1}{2}d\log X_K. 
\end{equation*}
Hence, 
\begin{equation*}
\begin{aligned}
d(\alpha - \alpha_K) &= \beta_2\wedge\left(d \lambdazero - (\alpha - \alpha_K) - \frac{1}{2}d\log (1 + \Xzero)\right)  \\ 
\end{aligned}
\end{equation*}
Hence, 
\begin{equation*}
\begin{aligned}
\partial_z \mathscr L(\lambdazero, \delta) &= \alpha_z(\rho, z) - (\alpha_z)_K(\rho, z) - \frac{1}{2}\partial_z(1 + \Xzero)  - \int_\rho^\infty(d\alpha)_{\rho z}(\tau, z)\,d\tau \\ 
&=  \alpha_z(\rho, z) - (\alpha_z)_K(\rho, z) - \frac{1}{2}\partial_z(1 + \Xzero)   \\
&- \int_\rho^\infty\left(  \beta_2\wedge\left(d \lambdazero - (\alpha - \alpha_K) - \frac{1}{2}d\log (1 + \Xzero)\right)_{\rho z}(\tau, z)\right)\,d\tau. 
\end{aligned}
\end{equation*}
\item The estimates in the region $\Bnbarre$ follow in the same manner up to computations for the change of coordinates from $(\rho, z)$ to $(s, \chi)$. We refer to the proof of Lemma 12.2.2 in \cite{chodosh2017time} for details. 
\begin{equation*}
\left|\left|\mathscr(\lambdazero, \delta) \right|\right|_{\hat C^{1, \alpha_0}(\Bnbarre)} \leq  C \left| \left|\left(\sigmazero, B, (\Xzero, \Yzero), \thetazero\right) \right|\right|_{\Lsigma\times\cdots\Ltheta}. 
\end{equation*}
\item The remaining estimates follow in the same manner. 
\item Finally, we choose $\delta_0>0$ sufficiently small that  $\displaystyle \mathscr L(\lambdazero, \delta)\in B_{\delta_0}(\Llambda)$. 
\end{enumerate}
\end{proof}

\subsubsection{Solving the original equations \eqref{lambdazero}}
Now, let $\delta_0>0$ and $(\lambdazero(\delta))_{\delta\in[0, \delta_0[}$ be the one-parameter family of solutions to \eqref{lambdazero:modif}. In this section, we will show that the latter actually satisfies \eqref{lambdazero:bis}, that is 
\begin{equation*}
d\lambdazero = \alpha - \alpha_K - \frac{1}{2}d\log(1 + \Xzero). 
\end{equation*}
To this end, let $z\in\mathbb R$ be fixed and set
\begin{equation*}
f_z(\rho) :=  \left(\partial_z\lambdazero - (\alpha_z - (\alpha_K)_z) - \frac{1}{2}\partial_z\log(1 + \Xzero)\right)(\rho, z)
\end{equation*}
We state the following lemma
\begin{lemma}
\label{vanish:fz}
$\forall z\in\mathbb R$ we have 
\begin{equation*}
\forall \rho\geq 0 \;,\; f_z(\rho) = 0. 
\end{equation*}
\end{lemma}
\begin{proof}
\begin{itemize}
\item First of all, since $\lambdazero(\delta)$ solves \eqref{lambdazero:modif}, we have
\begin{equation*}
\begin{aligned}
f_z(\rho) &= \left(\partial_z\lambdazero - (\alpha_z - (\alpha_K)_z) - \frac{1}{2}\partial_z\log(1 + \Xzero)\right)(\rho, z) \\
&= \int_\rho^{\infty}\, \left(\beta_2\wedge\left(d\lambdazero - (\alpha - \alpha_K) - \frac{1}{2}d\log(1 + \Xzero) \right) \right)_{\rho, z}(\tau, z)\,d\tau \\ 
&= \int_\rho^{\infty}\,  \left(((\beta_2)_\rho d\rho + (\beta_2)_z dz)\wedge\left(\partial_z\lambdazero - (\alpha - \alpha_K)_z - \frac{1}{2}\partial_z\log(1 + \Xzero) \right) \right)_{\rho, z}(\tau, z)\,d\tau \\ 
&= \int_\rho^{\infty}\,(\beta_2)_\rho(\tau, z)f_z(\tau)\,d\tau. 
\end{aligned}
\end{equation*}
\item Now, since $\sigmazero\in\Lsigma$, we have 
\begin{equation*}
\left| (\beta_2)_\rho\right| \leq Cr^{-3}. 
\end{equation*}
Therefore $\forall \rho_0>0$, there exists $c_0 = c_0(\rho_0)>0$ such that 
\begin{equation*}
\forall (\rho, z)\in[\rho_0, \infty[\times\mathbb R \; ,\; \left| (\beta_2)_\rho\right| \leq C\rho^{-3}. 
\end{equation*}
Moreover, there exists $c_1>0$ such that 
\begin{equation*}
\forall (\rho, z)\in\Bbarre\;\;,\;\;  \left|f_z(\rho)\right| \leq c_1. 
\end{equation*}
Hence, 
\begin{equation*}
\begin{aligned}
\left| f_z(\rho)\right| &\leq c_0\int_\rho^\infty\, \left|f_z(\tau)\right|\tau^{-3}\,d\tau \\
&\leq \frac{c_0c_1}{2}\rho^{-2}.
\end{aligned}
\end{equation*}
\item Therefore 
\begin{equation*}
\forall n\in\mathbb N \;,\; \forall (\rho, z)\in [\rho_0, \infty[\times\mathbb R\; \left|f_z(\rho)\right| \leq c_1\left(\frac{c_0}{2\rho^2}\right)^n\frac{1}{n!}
\end{equation*}
\item Hence, when $n\to\infty$, the right hand side goes to $0$ and thus, $\forall \rho_0>0 \;,\; \forall  (\rho, z)\in [\rho_0, \infty[\times\mathbb R \;:\; f_z(\rho) = 0. $
\end{itemize}
Finally, by continuity of $f_z$, $f_z$ vanishes on $[0, \infty[$. 
\end{proof}
Therefore, by the previous Lemma, we have 
\begin{equation*}
\partial_z\lambdazero =  \alpha_z - (\alpha_K)_z- \frac{1}{2}\partial_z\log(1 + \Xzero). 
\end{equation*}

\subsubsection{Conclusion}
Finally, we combine the previous results in order to prove Proposition \ref{non:linear:lambda}:

\begin{proof}
\begin{enumerate}
\item First of all, we apply Proposition \ref{non:linear:theta} to find a one-parameter family of solutions $\displaystyle \left( \sigmazero,  \left(0, 0, B_z^{(A)}\right), \right.$ $\left. \left(\Xzero, \Yzero \right), \thetazero \right)\left(\lambdazero, \delta\right)$ to their respective equations  which depends continuously on $(\lambdazero, \delta)\in B_{\delta_0}(\Llambda)\times[0, \delta_0[$. 
\item We apply Lemma \ref{lemma::83} in order to show that $\mathscr L$ satisfies the assumptions of Theorem \ref{Fixed::Point::1}. 
\item We apply Theorem \ref{Fixed::Point::1} with: 
\begin{equation*}
\mathcal T = \mathscr L\; ,\; \mathcal L = \Llambda\;,\; \mathcal P = [0, \overline \delta_0[
\end{equation*}
Therefore, after choosing $0<\delta_0\leq \overline\delta_0$, we obtain a one-parameter family of solutions $(\lambdazero(\delta))_{\delta\in[0, \delta_0[}$ satisfying 
\begin{equation*}
\begin{aligned}
& \left|\left|\lambdazero( \delta)\right|\right|_{\Llambda} \leq C\delta, \\ 
& \left|\left|\lambdazero(\delta_1) - \lambdazero(\delta_2)\right|\right|_{\Llambda} \leq C |\delta_1 - \delta_2|,
\end{aligned}
\end{equation*}
\item Henceforth, the solution map $ \left( \sigmazero,  \left(0, 0, B_z^{(A)}\right), \left(\Xzero, \Yzero \right), \thetazero\right)$ can be seen as a one-parameter family depending on $\delta$ in the following way

\begin{equation*}
\begin{aligned}
        \left|\left|\left(\sigmazero, \left(0, 0, B_z^{(A)}\right), \left(\Xzero, \Yzero\right), \thetazero\right)\left(\delta\right)\right|\right|_{\Lsigma\times\LB\times\LX\times\LY\times\Ltheta} \le C(\alpha_0)\left( ||\lambdazero(\delta)||^2_{\Llambda} + \delta\right) \\
        &\leq C(\delta^2 + \delta) \leq C \delta. 
        \end{aligned}
  \end{equation*}
  and $\displaystyle \forall \delta_i\in [0, \delta_0[$,
        \begin{equation*}
        \begin{aligned}
        &\left|\left|\left(\sigmazero, \left(0, 0, B_z^{(A)}\right), \left(\Xzero, \Yzero \right), \thetazero\right)\left(\delta_1\right) - \left(\sigmazero, \left(0, 0, B_z^{(A)}\right), \left(\Xzero, \Yzero \right), \thetazero\right)\left(\delta_2\right)\right|\right|_{\Lsigma\times \LB\times\LX\times\LY\times\Ltheta} \\ 
        &\le C(\alpha_0)\left(\left(\left|\left|\lambdazero(\delta_1)\right|\right|_{\Llambda} + \left|\left|\lambdazero(\delta_2)\right|\right|_{\Llambda}\right)\left|\left|\lambdazero(\delta_1) - \lambdazero(\delta_2)\right|\right|_{\Llambda}  + \left|\delta_1 - \delta_2\right| \right) \\
        &\le C(\alpha_0)\delta. 
        \end{aligned}
        \end{equation*}

\end{enumerate}
\end{proof}

\section{Proof of the Main Result}
\label{Final::proof}
By the previous section, we obtain the following result
\begin{Propo}
Let $\delta_0>0$ and let $\displaystyle \delta\ni[0, \delta_0[\to \left(\sigmazero, B, \left(\Xzero, \Yzero\right), \thetazero, \lambdazero \right)(\delta)$ be the one parameter family family of solutions to the reduced EV-system obtained by Proposition \ref{non:linear:lambda}. Then, 
\begin{itemize}
\item the metric $g_\delta$ given by 
\begin{equation*}
g_\delta := -V_\delta dt^2 + 2W_\delta dtd\phi + X_\delta d\phi^2 + e^{2\lambda_\delta}\left( d\rho^2 + dz^2\right)
\end{equation*}
where 
\begin{equation}
\label{g:::delta}
\begin{aligned}
\lambda_\delta &:= \lambdazero(\delta) + \lambda_K \\
X_\delta :&= X_K(1 + \Xzero(\delta)) \\
W_\delta :&= X_\delta(\thetazero(\delta) + X_K^{-1}W_K) \\ 
V_\delta &:= \frac{\sigma^2_\delta - W^2_\delta}{X_\delta} \quad\text{where}\quad \sigma_\delta := \rho(1 + \sigmazero(\delta)), 
\end{aligned}
\end{equation}
\item and the distribution function $f^\delta$ given by 
\begin{equation*}
f^\delta(t, \phi, \rho, z, p^\rho, p^\phi, p^z) := \Phi(\ve_\delta, \ell_z)\Psi_\eta(\rho -  \rho_1((\ve_\delta, (\ell_z)_\delta), (X_\delta, W_\delta, \sigma_\delta)))
\end{equation*}
where 
\begin{equation*}
\begin{aligned}
\varepsilon_\delta &:= \frac{\sigma_\delta}{\sqrt{X_\delta}}(1 + |p|^2)^{\frac{1}{2}} - \frac{W_\delta}{X_\delta}(\ell_z)_\delta , \\
(\ell_z)_\delta &:= \sqrt{X_\delta}p^\phi,
\end{aligned}
\end{equation*}
and $\rho_1$ is the second largest solution of the equation 
\begin{equation*}
E_{(\ell_z)_\delta}(X_\delta, W_\delta, \sigma_\delta, \rho, 0) = \ve_\delta. 
\end{equation*}
\end{itemize}
solve the Einstein-Vlasov system on $\spacetime = \mathbb R\times\mathbb S^1 \times \BB$. 
\end{Propo}
It remains to prove the following
\begin{enumerate}
\item The spacetime $(\spacetime, g_\delta)$ is $C^{2, \alpha}-$extendable to  a black hole spacetime in the sense of Definition \ref{extendibility}, 
\item The extended solution $(\tilde \spacetime, \tilde g_\delta)$ is asymptotically flat. 
\end{enumerate}
\noindent First of all,  we claim that 
\begin{lemma}
there exists $\Omega_\delta\in\mathbb R$ such that 
\begin{equation*}
\left.\frac{W_\delta}{X_\delta}\right|_{\Horizon}= -\Omega. 
\end{equation*}
\end{lemma}
\begin{proof}
We recall that $\forall \delta\geq 0$, we have 
\begin{equation*}
\partial_z\left( \frac{W_\delta}{X_\delta} \right) = \frac{\sigma_\delta}{X_\delta}\left(\partial_\rho Y_\delta + (B_\rho)_\delta \right).
\end{equation*}
On the horizon, the right hand side vanishes, since $\sigma_\delta$ vanishes. Therefore, there exists $\Omega_\delta\in\mathbb R$ such that 
\begin{equation*}
\forall z \in]-\beta, \beta[\quad, \quad \frac{W_\delta}{X_\delta}(0, z) = -\Omega_\delta
\end{equation*}
\end{proof}

\noindent We recall the following results from \cite[Section 13]{chodosh2017time}
\begin{Propo}
\label{ext::ext}
Let $\delta\in[0, \delta_0[$. Then, 
\begin{enumerate}
\item $\displaystyle \left.\left(e^{2\lambda} - \rho^{-2}X\right)\right|_\Axis = 0$. 
\item there exists $\kappa(\delta)>0$ such that $\displaystyle \left.\left(e^{2\lambda}  - \kappa_\delta^{-2}\rho^{-2}\left(V -2\Omega_\delta W - \Omega_\delta^2X\right) \right)\right|_\Horizon = 0.$
\item On the set $\Bnbarre$, we have 
\begin{equation*}
\left.\left( (\chi^2 + s^2)e^{2\lambda} - s^{-2}X\right)\right|_{\left\{s = 0 \right\}\cup \left\{\chi = 0 \right\}} = 0. 
\end{equation*}
\item On the set $\Bnbarre$, we have 
\begin{equation*}
\left.\left( (\chi^2 + s^2)e^{2\lambda} - \kappa_\delta^{-2}\chi^{-2}\left(V -2\Omega_\delta W - \Omega_\delta^2X\right) \right)\right|_{\left\{s = 0 \right\}\cup \left\{\chi = 0 \right\}} = 0. 
\end{equation*}
\item On the set $\Bsbarre$, we have 
\begin{equation*}
\left.\left( ((\chi')^2 + (s')^2)e^{2\lambda} - (s')^{-2}X\right)\right|_{\left\{s' = 0 \right\}\cup \left\{\chi' = 0 \right\}} = 0. 
\end{equation*}
\item On the set $\Bsbarre$, we have 
\begin{equation*}
\left.\left( ((\chi')^2 + (s')^2)e^{2\lambda} - \kappa_\delta^{-2}(\chi')^{-2}\left(V -2\Omega_\delta W - \Omega_\delta^2X\right) \right)\right|_{\left\{s' = 0 \right\}\cup \left\{\chi' = 0 \right\}} = 0. 
\end{equation*}
\end{enumerate}
\end{Propo}

\begin{lemma}
\label{extend::0}
Assume that $f(\rho, z)$ is a smooth function on $\BAbarre$, resp. $\BHbarre$, that is smooth when considered as a function on $\mathbb R^n$ with the metric $d\rho^2 + \rho^2d\mathbb S^{n-2} + dz^2$ for $n>2$. Then, $f(\rho, z) = g(\rho^2, z)$ for some smooth function on $\BAbarre$, resp $\BHbarre$.
\\ If $f(s, \chi)$ is a smooth function on $\Bnbarre$, that is smooth when considered as a function on $\mathbb R^n$ with the metric $ds^2 + s^2d\phi_1^2 + d\chi^2 + \chi^2d\phi_2^2$ for $n>2$. Then, $f(s, \chi) = g(s^2, \chi)$ for some smooth function on $\Bnbarre$.
\end{lemma}

The above proposition and lemma, together with Proposition \ref{extendibility} yield the extendability of $(\spacetime, g_\delta)$, for all $\delta\geq 0$. More precisely, we obtain the following result 

\begin{Propo}
\label{extend::Kerr:delta}
The spacetime $(\spacetime, g_\delta)$ is $C^{2, \alpha}-$extendable to  a Lorentzian manifold with corners $(\tilde\spacetime, \tilde g_\delta)$ which is stationary and axisymmetric, and whose boundary corresponds to a bifurcate Killing event horizon. 
\end{Propo}

\begin{proof}
\begin{enumerate}
\item We will only write the details for the extendability near the axis. The other extensions are obtained in the same way. 
\item By construction, all the metric coefficients are $C^{2, \alpha}$ on $\BB$. 
\item By Proposition \ref{extendibility}, the Kerr metric is $C^{2, \alpha}$ (it is in fact $C^\infty$) extendable  in the sense of Definition \ref{extendibility}. 
\item By Lemma \ref{extend::0}, in a neighbourhood of the axis of symmetry, we have
\begin{equation*}
\begin{aligned}
\Xzero(\delta)(\rho, z) &= \Xzero_{\delta, \Axis}(\rho^2, z), \\
\thetazero(\delta)(\rho, z) &= \thetazero_{\delta, \Axis}(\rho^2, z), \\
\sigmazero(\delta)(\rho, z) &= \sigmazero_{\delta, \Axis}(\rho^2, z), \\
\end{aligned}
\end{equation*}
for $C^{2, \alpha}$ functions, $\Xzero_{\delta, \Axis}, \thetazero_{\delta, \Axis}, \sigmazero_{\delta, \Axis}$ defined on $\tilde \Axis$. 

\item In  $\tilde\Axis$, we have 
\begin{equation*}
\left.X_\delta\right|_{\tilde\Axis} = \rho^2X_\Axis(\rho^2, z)(1+\Xzero(\delta)(\rho, z)),  
\end{equation*}
where $X_\Axis(0, z)>0$. Since $\delta$ is small so that $||\Xzero(\delta)||_{\LX}<1$, then there exists a $C^{2, \alpha}$ function $X_{\delta, \Axis}:\tilde\Axis \to\mathbb R$ such that $X_{\delta, \Axis}(0, z)>0$ and 
\begin{equation*}
\left.X_\delta\right|_{\tilde\Axis}(\rho, z) = \rho^2X_{\delta, \Axis}(\rho^2, z).   
\end{equation*}
Hence, $X_\delta$ verifies the third point of Definition \ref{ext:A}. 
\item By \eqref{g:::delta}, we have 
\begin{equation*}
W_\delta = X_\delta(\thetazero(\delta) + X_K^{-1}W_K)  
\end{equation*}
and 
\begin{equation*}
V_\delta = \frac{\sigma^2_\delta - W^2_\delta}{X_\delta} \quad\text{where}\quad \sigma_\delta := \rho(1 + \sigmazero(\delta)), 
\end{equation*}
Again, we use the extendability properties of $X_K$ and $W_K$ in order to obtain 
\begin{equation*}
\begin{aligned}
\left.W_\delta\right|_{\tilde\Axis} &= \left.X_\delta\right|_{\tilde\Axis}(\thetazero(\delta)(\rho, z) + X_\Axis(\rho^2, z)^{-1}W_\Axis(\rho^2, z))  \\
&= \rho^2X_{\delta, \Axis}(\rho^2, z)(\thetazero_{\delta, \Axis}(\rho^2, z) + X_\Axis(\rho^2, z)^{-1}W_\Axis(\rho^2, z)) \\
&= \rho^2W_{\delta, \Axis}(\rho^2, z). 
\end{aligned}
\end{equation*}
Hence, $W_\delta$ verifies the second point of Definition  \ref{ext:A}. 
\item The first point follows from the definition of $V_\delta$. 
\item As for the fourth point, we apply the first point of Proposition \ref{ext::ext}: 
\begin{equation*}
\displaystyle \left.\left(e^{2\lambda} - \rho^{-2}X\right)\right|_\Axis = 0. 
\end{equation*}
\end{enumerate}
Finally, we apply Proposition \ref{extendable} to conclude. 
\end{proof}

\begin{remark}
The regularity of the metric coefficients depend on the regularity of the distribution function, which depends on the regularity of the profile $\Phi$. Moreover, the way one extends the resulting spacetimes to a larger spacetimes with event horizon depends on the regularity of the profile $\Phi$. Hence, if $\Phi$ is $C^k$, then one can $C^{k+2, \alpha}$-extend.  
\end{remark}
\noindent Finally, we state the following result 

\begin{Propo}
$\forall \delta\in[0, \delta_0[$, the spacetime obtained by Proposition \ref{extend::Kerr:delta}, $(\tilde \spacetime, \tilde g_\delta)$, is asymptotically flat in the sense of Definition \ref{asymptotic::flatness}. 
\end{Propo}

\appendix

\section{Coordinate Systems for asymptotic flatness }
\label{coord:for:AF}
In this appendix, we recall from   \cite[Appendix A]{chodosh2015stationary} the different  change of coordinates used in order to define  "spatial infinity". This allows us to define a notion of asymptotic flatness adapted to our work. 
\\ Assume that $(\spacetime, g)$ is extendable to a regular black hole spacetime with corners $(\tilde\spacetime, \tilde g)$. $\tilde\spacetime$ is given by 
\begin{equation*}
\tilde \spacetime = \mathbb R_t\times(0, 2\pi)_\phi\times\Bbarre_{\rho, z},
\end{equation*}
where $\Bbarre$ is covered by $\BHbarre, \BAbarre, \Bnbarre$ and $\Bsbarre$. Let $K\subset \subset \BAbarre$ be a compact subset of $\BAbarre$ such that 
\begin{equation*}
\BHbarre \cup \Bnbarre \cup \Bsbarre \nsubseteq \Bbarre\backslash K. 
\end{equation*}
We introduce the coordinates $(t, x, y, z)\in\mathbb R\times \mathbb R^3$ in the region $\mathbb R\times(0, 2\pi)\times \Bbarre\backslash K$ defined by 
\begin{equation}
\begin{aligned}
x:= \rho\cos\phi,  \\
y:= \rho\sin\phi. 
\end{aligned}
\end{equation}
Therefore, a stationary axisymmetric spacetime $(\spacetime, g)$ is asymptotically flat if the $(t,x, y, z)$ coordinates defined above, the metric $\tilde g$ has the following expansion in the region $\mathbb R\times(0, 2\pi)\times \Bbarre\backslash K$

\begin{equation*}
\tilde g = \left( 1 + O\left(r^{-1}\right)\right)\left(-dt^2 + dx^2 + dy^2 + dz^2 \right) + O\left(r^{-2}\right)\left( dtdx + dtdy + dxdy\right). 
\end{equation*}

%

\section{Classical inequalities and estimates}
\label{classical::inequalities}
\begin{theoreme}[\cite{bartnik1986mass}]
\label{bartnik}
If $\delta<0$, then there exists $C>0$ such that $\forall u\in W^{1,p}_{\delta}(\mathbb{R}^3)$
\begin{equation*}
    ||u||_{p,\delta} \le C||u_r||_{p,\delta-1},
\end{equation*}
where $u_r$ is defined by
\begin{equation*}
    u_r := \left<r\right>^{-1}(x,y,z)^t\cdot\partial u.
\end{equation*}
\end{theoreme}

\begin{lemma}
\label{sph::}
Let $f\;:\;\mathbb R^2\to\mathbb R$ be a $C^2$ spherically symmetric function and $(\rho, \phi)$denote polar coordinates on $\mathbb R^2$. Then, for any point $(x_0, y_0)\in\mathbb R^2$ we have 
\begin{equation*}
\left( x_0^2 + y_0^2\right)^{-1/2}\left.\partial_\rho f \right|_{(x, y) = (x_0, y_0)} = \left.\partial_{xx}f\right|_{(x, y) = \left(0, \left( x_0^2 + y_0^2\right)^{1/2}\right) }
\end{equation*}
\end{lemma}
\noindent We recall Theorem $6.6$ of \cite{gilbarg2015elliptic} concerning Schauder estimates:
\begin{theoreme}
\label{Schauder:general}
Let $\Omega$ be a $C^{2, \alpha}$ domain in $\mathbb R^n$ and let $u\in C^{2, \alpha}(\overline \Omega)$ be a solution of $\displaystyle Lu = f\; \text{in}\; \Omega$, where $f\in C^{0, \alpha}(\overline\Omega)$, $L$ is an elliptic operator and  the coefficients of $L  $ satisfy, for positive constants $\lambda$, $\Lambda$, 
\begin{equation*}
a^{ij}\xi_i\xi_j \geq \lambda|\xi|^2 \quad\forall x\in \Omega, \xi\in\mathbb R^n, 
\end{equation*}
and 
\begin{equation*}
|a^{ij}|_{0, \alpha; \Omega}, |b^{i}|_{0, \alpha; \Omega}, |c|_{0, \alpha; \Omega} \leq \Lambda. 
\end{equation*}
Let $\phi(x)\in C^{2, \alpha}(\overline\Omega)$ and suppose $u = \phi$ on $\partial\Omega$. Then 
\begin{equation*}
|u|_{2;\alpha; \Omega} \leq C(n, \alpha, \lambda, \Lambda, \Omega)\left(  |u|_{0;\alpha; \Omega} + |\phi|_{0;\alpha; \Omega} + |f|_{0;\alpha; \Omega} \right). 
\end{equation*}
\end{theoreme}
\noindent We recall Theorem $10.3$ of \cite{lieb2001graduate} concerning Schauder estimates for the Laplacian in $\mathbb R^n$
\begin{theoreme}[Newtonian estimates]
\label{Newton:estimates}
Let $f\in C_c^{k, \alpha}(\mathbb R^n)$ with $k\in\mathbb N,\; \alpha\in(0, 1)$ and let $K_f$ be given by 
\begin{equation*}
K_f(x):= \int_{\mathbb R^n}\, G(x, y)f(y)\,dy\quad\text{where }\quad G(x, y) := -\frac{1}{|x-y|^{n - 2}}
\end{equation*}
Then, 
\begin{equation*}
K_f\in C^{k+2, \alpha}(\mathbb R^n)
\end{equation*}
and there exists $C = C(k, \alpha, n)>0$
\begin{equation*}
||K_f||_{C^{k+2, \alpha}(\mathbb R^n)} \leq C ||f||_{C^{k, \alpha}(\mathbb R^n)}. 
\end{equation*}
\end{theoreme}
\begin{theoreme}[Calderon-Zygmund estimates]
\label{Cal:Zyg}
Let $f\in L^p(\mathbb R^n)$ and let $K_f$ be as defined in Theorem \ref{Newton:estimates}. Then, $K_f$ has weak second derivative in $L^p(\mathbb R^n)$ and we have 
\begin{equation*}
||K_f||_{\dot W^{2, p}(\mathbb R^n)} \leq C ||f||_{L^p(\mathbb R^n)}.
\end{equation*}
\end{theoreme}

\section{Classical Carter-Robinson theory}
In the absence of a matter field, the metric quantities $X, W, \sigma , \theta, \lambda$ satisfy the following equations on $\BB$:

\begin{itemize}
\item $\sigma = \rho$. Hence, 
\begin{equation*}
\Delta_{\mathbb R^3}\sigma = 0. 
\end{equation*}
\item $\theta$ is given by 
\begin{equation*}
    \theta = \frac{X^2}{\sigma}(\partial_z(X^{-1}W)d\rho - \partial_{\rho}(X^{-1}W)dz).
\end{equation*}
\item $W$ satisfies the equation 
\begin{equation*}
\partial_{\rho}(X^{-1}W)d\rho + \partial_{z}(X^{-1}W)dz = \frac{\rho}{X^2}((\partial_{\rho}Y)dz-(\partial_{z}Y)d\rho).    
\end{equation*}
\item $(X, Y)$ satisfies the following system on $\mathbb R^3$
    \begin{equation}\label{X::Y}
        \left\{
        \begin{aligned}
            \rho^{-1}\partial_{\rho}(\rho\partial_{\rho}X) + \rho^{-1}\partial_{z}(\rho\partial_{z}X) &= \frac{(\partial_{\rho}X)^2 + (\partial_{z}X)^2 - (\partial_{\rho}Y)^2 - (\partial_{z}Y)^2}{X}, \\
            \rho^{-1}\partial_{\rho}(\rho\partial_{\rho}Y) + \rho^{-1}\partial_{z}(\rho\partial_{z}Y) &= \frac{2(\partial_{\rho}Y)(\partial_{\rho}X) + 2(\partial_{z}Y)(\partial_{z}X)}{X}.
        \end{aligned}
        \right.
        \end{equation}
        
 \item $\lambda$ satisfies the equation 
 \begin{equation}
            \left\{
            \begin{aligned}
                \partial_{\rho}\lambda &= \frac{1}{4}\rho{X}^{-2} ((\partial_{\rho}X)^2-(\partial_{z}X)^2+(\partial_{\rho}Y)^2-(\partial_{z}Y)^2) - \frac{1}{2}\partial_{\rho}\log{X}, \\
                \partial_{\rho}\lambda &= \frac{1}{4}\rho{X}^{-2} ((\partial_{\rho}X)(\partial_{z}X)+(\partial_{\rho}Y)(\partial_{z}Y)) - \frac{1}{2}\partial_{z}\log{X}.
            \end{aligned}
            \right.
            \end{equation}
\end{itemize}
\noindent In particular $(X_K, Y_K, \sigma_K, \theta_K, \lambda_K)$ satisfies the above equations.


\newpage


\printindex[Symbols]

\newpage
\bibliographystyle{plain}
\bibliography{references}

\setcounter{secnumdepth}{4}
\setcounter{tocdepth}{4}

\newcommand{\HRule}{\rule{\linewidth}{0.5mm}}
\renewcommand{\arraystretch}{1.5}

\end{document}